\documentclass[11pt, book]{memoir}

\settrims{0pt}{0pt} 
\settypeblocksize{*}{36pc}{*} 
\setlrmargins{*}{*}{1} 
\setulmarginsandblock{1in}{1in}{*} 
\setheadfoot{\onelineskip}{2\onelineskip} 
\setheaderspaces{*}{1.5\onelineskip}{*} 
\checkandfixthelayout

\usepackage{mathtools}
\usepackage{amsthm}
\usepackage{amssymb}
\usepackage{stmaryrd}
\usepackage{bbm}
\usepackage{accents}
\usepackage{newpxtext}
\usepackage[utf8]{inputenc}
\usepackage[varg,bigdelims]{newpxmath}
\usepackage[usenames,dvipsnames]{xcolor}
\usepackage{tikz}
\usepackage{graphicx}
\usepackage{enumitem}
\usepackage{mathpartir}
\usepackage[bookmarks=true, colorlinks=true, linkcolor=blue!50!red, citecolor=orange,
pdfencoding=unicode]{hyperref}
\usepackage{subfiles}
\usepackage[capitalize]{cleveref}
\usepackage[backend=biber,style = alphabetic]{biblatex}
\usepackage{ebproof}

\usepackage{makeidx}

\makeindex
\crefname{axiom}{Axiom}{Axioms}
\usetikzlibrary{
	cd,
	math,
	decorations.markings,
	positioning,
	arrows.meta,
	shapes,
	calc,
	fit,
	quotes,
	intersections}
\hypersetup{final}
\setlist{nosep}

\tikzset{
   oriented WD/.style={
      every to/.style={out=0,in=180,draw},
      label/.style={
         font=\everymath\expandafter{\the\everymath\scriptstyle},
         inner sep=0pt,
         node distance=2pt and -2pt},
      semithick,
      node distance=1 and 1,
      decoration={markings, mark=at position .5 with {\arrow{stealth};}},
      ar/.style={postaction={decorate}},
      execute at begin picture={\tikzset{
         x=\bbx, y=\bby,
         every fit/.style={inner xsep=\bbx, inner ysep=\bby}}}
      },
   bbx/.store in=\bbx,
   bbx = 1.5cm,
   bby/.store in=\bby,
   bby = 1.75ex,
   bb port sep/.store in=\bbportsep,
   bb port sep=2,
   bb port length/.store in=\bbportlen,
   bb port length=4pt,
   bb min width/.store in=\bbminwidth,
   bb min width=1cm,
   bb rounded corners/.store in=\bbcorners,
   bb rounded corners=2pt,
   bb small/.style={bb port sep=1, bb port length=2.5pt, bbx=.4cm, bb min width=.4cm, bby=.7ex},
   bb/.code 2 args={
      \pgfmathsetlengthmacro{\bbheight}{\bbportsep * (max(#1,#2)+1) * \bby}
      \pgfkeysalso{draw,minimum height=\bbheight,minimum width=\bbminwidth,outer sep=0pt,
         rounded corners=\bbcorners,thick,
         prefix after command={\pgfextra{\let\fixname\tikzlastnode}},
         append after command={\pgfextra{\draw
            \ifnum #1=0{} \else foreach \i in {1,...,#1} {
               ($(\fixname.north west)!{\i/(#1+1)}!(\fixname.south west)$) +(-\bbportlen,0) 
coordinate (\fixname_in\i) -- +(\bbportlen,0) coordinate (\fixname_in\i')}\fi 
            \ifnum #2=0{} \else foreach \i in {1,...,#2} {
               ($(\fixname.north east)!{\i/(#2+1)}!(\fixname.south east)$) +(-\bbportlen,0) 
coordinate (\fixname_out\i') -- +(\bbportlen,0) coordinate (\fixname_out\i)}\fi;
         }}}
   },
   bb name/.style={append after command={\pgfextra{\node[anchor=north] at (\fixname.north) {#1};}}}
}

\tikzset{
	unoriented WD/.style={
		every to/.style={draw},
		shorten <=-\portlen, shorten >=-\portlen,
		label distance=-2pt,
		thick,
		node distance=\spacing,
		execute at begin picture={\tikzset{
			x=\spacing, y=\spacing}}
		},
	pack size/.store in=\psize,
	pack size = 8pt,
	spacing/.store in=\spacing,
	spacing = \psize,
	link size/.store in=\lsize,
	link size = 2pt,
	port len/.store in=\portlen,
	port len = \lsize,
	pack color/.store in=\pcolor,
	pack color=blue,
	surround sep/.store in=\ssep,
	surround sep=\psize,
	link/.style={
		circle,
		anchor=center,
		draw=black,
		fill=black,
		inner sep=0pt,
		minimum size=\lsize
	},
	pack/.style={
		circle,
		anchor=center,
		draw = \pcolor!50!black,
		fill = \pcolor!20,
		inner sep = .25*\psize,
		minimum size = \psize
	},
	outer pack/.style={
		ellipse,
		anchor=center,
		draw,
		inner sep=\ssep,
		color=\pcolor!50!black,
	},
	intermediate pack/.style={
		ellipse,
		anchor=center,
		dashed,
		draw,
		inner sep=\ssep,
		color=\pcolor!50!black,
	},
}

\newcommand{\mypic}[4]{
	\node (lend) {};
	\node [right=#1 of lend] (L) {#4};
	\node [right=#2 of L] (R) {#4};
	\node [right=#1 of R] (rend) {};
	\node at ($(L)!.5!(R)+(0,#2/2)$) (M) {};
	\draw[thick, black, #3] (L.center) -- (M.center) -- (R.center);
	\fill[fill=black!15] (L.center) -- (M.center) -- (R.center) -- cycle;
	\draw[thick, <->] (lend) -- (rend);
}

\theoremstyle{plain}
\newtheorem{theorem}{Theorem}[chapter] 
\newtheorem{proposition}[theorem]{Proposition}
\newtheorem{corollary}[theorem]{Corollary}
\newtheorem{lemma}[theorem]{Lemma}

\theoremstyle{definition}
\newtheorem{definition}[theorem]{Definition}

\newtheorem{notation}[theorem]{Notation}
\newtheorem{axiom}{Axiom}
\newtheorem*{axiom*}{Axiom}

\theoremstyle{remark}
\newtheorem{example}[theorem]{Example}
\newtheorem{remark}[theorem]{Remark}
\newtheorem{warning}[theorem]{Warning}

\renewcommand{\ss}{\subseteq}

\newcommand{\const}[1]{\mathtt{#1}}
\newcommand{\Set}[1]{\mathrm{#1}}
\newcommand{\cat}[1]{\mathcal{#1}}
\newcommand{\Cat}[1]{\mathbf{#1}}
\newcommand{\fun}[1]{\mathit{#1}}
\newcommand{\Fun}[1]{\mathsf{#1}}
\DeclarePairedDelimiter{\church}{\llbracket}{\rrbracket}
\DeclarePairedDelimiter{\Church}{\llbracket}{\rrbracket}
\DeclarePairedDelimiter{\abs}{\lvert}{\rvert}
\newcommand{\mach}[1]{\mathcal{#1}}

\DeclareMathOperator{\id}{id}
\DeclareMathOperator{\Hom}{Hom}

\DeclareMathOperator*{\colim}{colim}

\newcommand{\cocolon}{:\!}
\newcommand{\iso}{\cong}
\newcommand{\To}[1]{\xrightarrow{#1}}

\newcommand{\from}{\leftarrow}
\newcommand{\From}[1]{\xleftarrow{#1}}

\newcommand{\inj}{\rightarrowtail}
\newcommand{\wavyto}{\rightsquigarrow}

\newcommand{\tn}[1]{\textnormal{#1}}
\newcommand{\ol}[1]{\overline{#1}}

\newcommand{\wt}[1]{\widetilde{#1}}
\newcommand{\ubar}[1]{\underaccent{\bar}{#1}}

\newcommand{\hs}{\hspace{1.1pt}}
\newcommand{\delnext}[1]{}

\newcommand{\internal}[1]{\raisebox{-.03ex}{$\mathbbmtt{#1}$}}

\newcommand{\NN}{\mathbb{N}}
\newcommand{\QQ}{\mathbb{Q}}

\newcommand{\ZZ}{\mathbb{Z}}
\newcommand{\RR}{\mathbb{R}}
\newcommand{\RRub}{\RR^\infty}
\newcommand{\RRubhalf}{\RRub_{\geq0}}
\newcommand{\IR}{\mathbb{I\hs R}}
\newcommand{\II}{\bar{\ubar{\RR}}}
\newcommand{\LR}{\ubar{\RR}}
\newcommand{\UR}{\bar{\RR}}
\newcommand{\LRub}{\ubar{\RR}^\infty}
\newcommand{\URub}{\bar{\RR}^\infty}
\newcommand{\IRub}[1][\delnext]{\IR^\infty}
\newcommand{\IIub}{\II^\infty}
\newcommand{\IRpre}[1][\delnext]{\IR\hs_{#1,\tn{pre}}}
\newcommand{\IIpre}{\II_{\tn{pre}}}
\newcommand{\LRpre}{\ubar{\RR}_{\tn{pre}}}
\newcommand{\URpre}{\bar{\RR}_{\tn{pre}}}
\newcommand{\LRplus}{\ubar{\RR}^{+}}
\newcommand{\URplus}{\bar{\RR}^{+}}
\newcommand{\IIubpre}{\bar{\ubar{\RR}}^\infty_{\tn{pre}}}
\newcommand{\IRubpre}[1][\delnext]{\IR\hs^\infty_{#1,\tn{pre}}}
\newcommand{\LRubpre}{\ubar{\RR}^\infty_{\tn{pre}}}
\newcommand{\URubpre}{\bar{\RR}^\infty_{\tn{pre}}}
\newcommand{\LRpluspre}{\ubar{\RR}^{+}_{\tn{pre}}}
\newcommand{\URpluspre}{\bar{\RR}^{+}_{\tn{pre}}}
\newcommand{\tNN}{\internal{N}\hs}
\newcommand{\tQQ}{\internal{Q}\hs}
\newcommand{\tZZ}{\internal{Z}\hs}

\newcommand{\tRR}{\internal{R}\hs}
\newcommand{\tIR}{\internal{I\hs R}\hs}
\newcommand{\tII}{\bar{\ubar{\tRR}}\hs}
\newcommand{\tLR}{\ubar{\tRR}\hs}
\newcommand{\tUR}{\bar{\tRR}\hs}
\newcommand{\tRRub}{\tRR^\infty}
\newcommand{\tIRub}{\internal{I\hs R}^\infty}
\newcommand{\tLRub}{\ubar{\tRR}^{\infty}}
\newcommand{\tURub}{\bar{\tRR}^{\infty}}
\newcommand{\tIIub}{\bar{\ubar{\tRR}}^{\infty}}
\newcommand{\ctRR}{\mathrm{c}\tRR}
\newcommand{\ctIR}{\mathrm{c}\tIR}
\newcommand{\ctII}{\mathrm{c}\tII}
\newcommand{\ctLR}{\mathrm{c}\tLR}
\newcommand{\ctUR}{\mathrm{c}\tUR}
\newcommand{\ctRRub}{\mathrm{c}\tRRub}
\newcommand{\ctIRub}{\mathrm{c}\tIRub}
\newcommand{\ctIIub}{\mathrm{c}\tIIub}
\newcommand{\ctLRub}{\mathrm{c}\tLRub}
\newcommand{\ctURub}{\mathrm{c}\tURub}
\newcommand{\tIRpre}[1][\delnext]{\tIR\hs_{#1,\tn{pre}}}
\newcommand{\tIIpre}{\tII_{\tn{pre}}}
\newcommand{\tLRpre}{\ubar{\tRR}_{\tn{pre}}}
\newcommand{\tURpre}{\bar{\tRR}_{\tn{pre}}}

\newcommand{\tIIubpre}{\bar{\ubar{\tRR}}^\infty_{\tn{pre}}}
\newcommand{\tIRubpre}[1][\delnext]{\tIR\hs^\infty_{#1,\tn{pre}}}
\newcommand{\tLRubpre}[1][\delnext]{\ubar{\tRR}^\infty_{#1,\tn{pre}}}
\newcommand{\tURubpre}{\bar{\tRR}^\infty_{\tn{pre}}}

\newcommand{\IRinvInline}{\IR/\rhd}
\newcommand{\IRinv}{\IR_{/\rhd}}
\newcommand{\tConst}{\mathtt{C}}
\newcommand{\ShFun}[1]{\mathrm{Fn}(#1)}

\newcommand{\Ind}[1]{\Fun{Ind}\tn{-}#1}
\newcommand{\Psh}[1]{\Fun{Psh}(#1)}
\newcommand{\Shv}[1]{\Fun{Shv}(#1)}
\newcommand{\Cont}[1]{\Fun{Cont}(#1)}

\newcommand{\Tw}[1]{\Fun{Tw}(#1)}
\newcommand{\B}{\Fun{B}}

\newcommand{\AD}{\mathrm{AD}}

\newcommand{\Prop}{\const{Prop}}
\newcommand{\Time}{\const{Time}}
\newcommand{\unit}{\const{1}}

\newcommand{\Int}{\Cat{Int}}
\renewcommand{\Top}{\Cat{Top}}
\newcommand{\Predom}{\Cat{Predom}}
\newcommand{\unitspeed}{\mathrm{unit\_speed\ }}
\newcommand{\Bool}{2}

\newcommand{\BaseTopos}{\cat{B}}
\newcommand{\BaseSite}{S_{\IRinvInline}}

\newcommand{\op}{^\tn{op}}
\newcommand{\el}[1]{\tn{el}#1}
\newcommand{\binmeet}{\curlywedge}
\newcommand{\binjoin}{\curlyvee}
\newcommand{\join}{\mathbin{\star}}
\newcommand{\asSh}{\Fun{sh}}
\newcommand{\bang}{\,!}

\newcommand{\subint}[2]{\langle #1,#2\rangle}
\newcommand{\restrict}[2]{#1\big|\hspace{0in}_{#2}}
\newcommand{\restrictsm}[2]{#1|\hspace{0in}_{#2}}
\newcommand{\basicCover}{\rightslice}
\newcommand{\sqss}{\sqsubseteq}
\newcommand{\specupclose}{{\uparrow}}
\newcommand{\specdownclose}{{\downarrow}}
\newcommand{\upclose}{{\rotatebox[origin=c]{90}{$\twoheadrightarrow$}}}
\newcommand{\downclose}{{\rotatebox[origin=c]{90}{$\twoheadleftarrow$}}}
\newcommand{\down}{\mathord{\downarrow}}
\newcommand{\up}{\mathord{\uparrow}}
\newcommand{\specord}{\leqslant}
\newcommand{\upspecord}{\eqslantless}
\newcommand{\PoShv}[1]{\Fun{Shv}_0(#1)}
\newcommand{\Id}{\mathrm{Id}}
\newcommand{\RId}{\mathrm{RId}}
\newcommand{\cRId}{\mathrm{cRId}}
\newcommand{\RFilt}{\Omega_{\mathrm{Filt}}}
\newcommand{\Opens}{\Omega}
\renewcommand{\U}{\mathcal{U}}
\newcommand{\amap}[3]{#2(#1,#3)}

\newcommand{\apart}{\,\#\,}

\newcommand{\WB}[2]{\langle\!\langle#1,#2\rangle\!\rangle}
\newcommand{\wb}{W}
\newcommand{\ddt}{\frac{d}{dt}}

\newcommand{\imp}{\Rightarrow}
\renewcommand{\iff}{\Leftrightarrow}

\newcommand{\Pointwise}{\pi}
\newcommand{\AtSymbol}{{@}}
\newcommand{\SeeSymbol}{{\down}}  
\newcommand{\InSymbol}{{\upclose}}
\newcommand{\At}[2][]{\AtSymbol^{#1}_{#2}}
\newcommand{\See}[2][]{\SeeSymbol^{#1}_{#2}}
\newcommand{\In}[2][]{\InSymbol^{#1}_{#2}}
\newcommand{\AtInline}[1]{@{#1}}
\newcommand{\SeeInline}[1]{\SeeSymbol{#1}}

\newcommand{\Until}{\mathbin{\mathcal{U}}}
\newcommand{\Since}{\mathbin{\mathcal{S}}}
\newcommand{\Release}{\mathbin{\mathcal{R}}}

\newcommand{\vertss}{\rotatebox{270}{\!$\ss$}}

\newcommand{\hybrid}{\mathrm{Hyb}}

\newcommand{\side}[1]{$\scriptstyle #1$}
\newcommand{\fv}{\mathrm{fv}}
\newcommand{\Prem}{\;\mathit{Prem}}

\newcommand{\erase}[1]{}

\allowdisplaybreaks
\setsecnumdepth{subsection}
\settocdepth{section}
\begin{document}

\title{\vspace{.5in}\HUGE Temporal Type Theory\\\bigskip\LARGE
A topos-theoretic approach to systems and behavior\\\vspace{2.5in}\textit{by}\vspace{-.1in}}

\author{
\LARGE \;Patrick Schultz\hspace{.5in} David I. Spivak
}
\date{}

\maketitle
\thispagestyle{empty}

\clearpage
\tableofcontents*

\chapter{Introduction}\label{sec:intro}

\section{Overview}
In this book we provide a new mathematical formalism for proving properties about the behavior of
systems. A system is a collection of interacting components, each of which may have some internal
implementation that is reflected in some external behavior. This external behavior is what other
neighboring systems interact with, through a shared environment. Properties of a behavior can be
established over a given duration (sometimes called frame or window)\index{time window} of time, and
we propose a mathematical language for working with these behavioral properties.

\subsection{Behavior types}

A \emph{behavior type}\index{behavior type} $B$ is the information of a set of ``different things
that can occur'' over any length of time. For example, a movie is a behavior type: given any
duration of time, there is a set of possible snippets of the movie that have said duration, and the
snippets of the movie are what we would call its behaviors. In fact, ``all possible movies'' is another
a behavior type, because to any duration of time, we could associate the set of all 24-frame/second
sequences of photographs together with an overlay of sound. All possible music, all possible fights
or boxing matches, all behaviors that an airplane or set of airplanes are capable of, etc.---each of
these can be modeled as a behavior type.

We give a category-theoretic description of behavior types, using the language of
sheaves.\index{sheaf} That is, to every time window\index{time!window}, say of length $\ell$, a
behavior type $B$ is responsible for providing a set $B(\ell)$ of all behaviors, called
\emph{length-$\ell$ sections of $B$}, that can possibly occur over this time window. And for every
inclusion of one time window into another, the behavior type $B$ is responsible for providing a
\emph{restriction map}\index{restriction} that restricts each long behavior to the shorter time
window. We say a behavior type \emph{has composable sections}\index{behavior!composable} when, for any two
overlapping time window, behaviors over the first that match behaviors over the second can be
\emph{glued} to form a behavior over the union time window.

Although many behavior types have composable sections, not all do. For example, the sheaf of
monotonic functions to $\RR$ has composable sections, because if for every small interval
$[t,t+\epsilon]$, we have $f(t)\leq f(t+\epsilon)$ then for every interval $[t_1,t_2]$ whatsoever,
we also have $f(t_1)\leq f(t_2)$. In contrast, the following sort of ``roughly monotonic'' behaviors are not composable\index{behavior!non-composable!relaxed monotonicity}\index{behavior!non-composable}:
\begin{equation}\label{eqn:relaxed_monotonicity}
  \forall t_1\forall t_2\ldotp(t_1+5\leq t_2)\imp f(t_1)\leq f(t_2).
\end{equation}
This formula says ``if you wait at least 5 seconds between taking samples, you will find $f$ to be
increasing.'' Such a property cannot be determined except on intervals of length at least 5. Any behavior---
if tested over a very short time window---will comply with this property, but if one glues two compliant
behaviors that agree along an overlap, the result may not be compliant. This example has the property that
it satisfies a composition property for pairs of intervals whose overlap is large enough. An example of a behavior type without any sort of composability is what we might call ``functions of bounded difference''
\index{behavior!non-composable!bounded difference}, e.g. satisfying
\begin{equation}\label{eqn:bounded_difference}
  \forall t_1\forall t_2\ldotp |f(t_1)-f(t_2)|\leq 5.
\end{equation}

\bigskip

Composable sections or not, however, all the behavior types we consider will be \emph{sheaves} in an
appropriate setting, and morphisms of sheaves allow us to connect one behavior type to another. For
example, there is a morphism from the sheaf corresponding to a movie to the sheaf corresponding to its sound-track. One
might formalize the system-environment notion as a morphism $S\to E$ and say two systems $S_1$ and $S_2$ share an
environment when given morphisms $S_1\to E\from S_2$. Thus it is through sheaf morphisms that different
behaviors can interact.\index{interaction} Two singers may be hearing the same drum beat, or two
airplanes may share the same communication channel.

The behavior types and their morphisms form a category $\BaseTopos$ of sheaves, and as such, a
\emph{(Grothendieck) topos}\index{topos}\index{Grothendieck topos|see {topos}}. Toposes are
particularly nice categories,\index{category!topos as} in that they enjoy many properties that make
them in some sense similar to the category of sets. One way to make this precise is to say that toposes satisfy the Giraud axioms
\cite{Artin.Grothendieck.Verdier:1971a}. But to more clearly draw the analogy with $\Cat{Set}$, let
us suffice to say that toposes are regular categories, and that they have limits and colimits,
effective equivalence relations, exponential objects, and a subobject classifier $\Omega$.

The subobject classifier\index{subobject classifier} in particular is an important object when
thinking about properties, e.g.\ properties of behavior. For $\Cat{Set}$, the subobject classifier is
$\Omega_\Cat{Set}=\{\tn{True, False}\}$. Many properties take the form of ``yes/no'' questions; such
questions distinguish elements of a set according to a given property. That is, for any set $S$ and
property $P$, there is a corresponding yes/no question $S\to\{\tn{True, False}\}$ that classifies
the subset of $P$-satisfying elements in $S$. For example the property of an integer being odd is
classified by a function $\texttt{is\_odd}\colon\ZZ\to\{\tn{True, False}\}$.

An analogous fact holds for behavior types. There is a behavior type $\Omega_\BaseTopos$ which
classifies behavior subtypes. Rather than ``yes/no'' questions, properties of behavior types are
``compliant when'' questions. For example, the property of a continuous function $f\colon\RR\to\RR$
being greater than 0 is classified in terms of the largest open subset $\{x\in\RR\mid f(x)>0\}$ on
which it is compliant.\index{behavior!property of} Similarly, the property of $f$ satisfying
\cref{eqn:bounded_difference} is classified by the collection of all open intervals over which $f$
has upper- and lower-bounds less than 5 apart.

Every topos has an associated internal language\index{language!internal|see
{internal}}\index{internal!language} and higher-order logic.\index{logic!higher-order} The symbols
of the logic are the typical ones, $\top$, $\bot$, $\wedge$, $\vee$, $\imp$, $\iff$, $\neg$, $\forall$, and
$\exists$, and the sort of reasoning steps allowed are also typical constructive
logic.\index{logic!constructive} For example, given $P$ and $P\imp (Q\wedge Q')$, one can derive
$Q$. However, our topos $\BaseTopos$ is not boolean, so $P\vee\neg P$ does not hold in general, nor does
$(\neg\forall x\ldotp Px)\imp\exists x\ldotp\neg Px$. Again, the reasoning is \emph{constructive},
meaning that a proof of a statement gives a witness to its truth: to prove an existential, one must
provide a witness for it and similarly for a disjunction.

One need not imagine behavior types as sheaves when doing this logic---it is all purely formal---and
yet anything that can be proven within this logic reflects a specific truth about behavior types as sheaves. The separate-but-connected relationship between the logic and its semantics allows us to work
with a highly abstract and complex topos using standard constructive logic, reasoning as though with
sets and their elements. Thus for example, one could prove sound theorems about $\BaseTopos$ in an
undergraduate course on formal logic, without ever discussing categories, let alone toposes.

\subsection{Goal: to prove properties of systems}

Our goal is to understand what can possibly occur when multiple components---each with their own
type of behavior---interact in any given way. If we know something about the behavior of each
component---if each component has a \emph{behavior-contract} or guarantee that specifies something about how it will
behave among all its possibilities---then we may be able to guarantee something about how the entire system
will behave. This is relevant to industries in which different suppliers provide different parts, each with
its own behavioral guarantee. In such a setting, one still wants to draw conclusions about the system formed
by arranging these components so that they interact in some specified way.

For example, a thermostat, a furnace, and a room comprise three components, each of which may be
guaranteed to satisfy a certain behavior contract:\index{behavior!contract|see {behavior
contract}}\index{behavior contract}\index{behavior!property of} the thermostat promises to sense
temperature of the room and promises to send a signal to the furnace if the temperature is too low.
The furnace promises to heat the room air at a certain rate when given the signal, and the room
promises to lose heat at some maximum rate. We may want to prove a behavior contract for the whole
thermostat-furnace-room system, e.g.\ that the temperature will remain within certain bounds. In
this book we provide a formal system---a temporal type theory and a higher-order temporal logic\index{logic!temporal!higher-order}---in which such proofs can be carried out.

In other words, we will provide a language for proving properties of interconnected dynamical
systems,\index{dynamical systems}\index{dynamical systems!interconnection of} broadly construed.
Dynamical systems are generally considered to come in three flavors: continuous, discrete, and
hybrid\index{dynamical system!continuous}\index{dynamical system!discrete}\index{dynamical
system!hybrid}, according to how time is being modeled. In the above example, the temperature of the
room could be modeled by a continuous dynamical system, the thermostat by a discrete dynamical
system, and the whole setup is a hybrid system. However our notion of behavior type is much more
general, serving as a sort of ``big tent'' into which other conceivable notions of behavior can be
translated and subsequently compared. For example, there is no differential
equation\index{differential equation} whose solution set is the roughly-monotonic curves from
\cref{eqn:relaxed_monotonicity}, because two such trajectories $f$ can ``cross'', but these
trajectories constitute a perfectly good behavior type in $\BaseTopos$. Behavior
types also include infinite-dimensional systems, and as such may be a good language for considering
adaptive control \cite{Aastrom.Wittenmark:2013a}, though we do not pursue that here.

There are already many different temporal logics---including linear temporal, (metric,
Halpern-Shoham) interval temporal, and signal temporal
\cite{Rescher.Urquhart:2012a,Alur.Feder.Henzinger:1996a,Maler.Nickovic:2004a,Halpern.Shoham:1991a}---for
describing behavior.\index{logic!temporal} Some of these generalize others, but none is most
general. While we do not prove it here, we believe and will give evidence for the assertion that our temporal logic can serve as a ``big
tent'', in which all other such logics can embed.

Some of these temporal logics have very powerful
model checkers that can produce a proof or counterexample in finite (though often exponential or
doubly-exponential) time, and we make no such claim: our logic is wildly undecidable. However, as a
big tent, our formalism can embed these proofs from other logics and integrate them as a last step
in a process.%
\footnote{It may be objected that these temporal logics are often boolean whereas our topos $\BaseTopos$ is
not; in this case one would simply embed the statements into the boolean subtopos
$\BaseTopos_{\neg\neg}\ss\BaseTopos$.\index{subtopos!boolean}}
Our work may also be useful to those who simply want a topos-theoretic and category-theory-friendly
approach to understanding behavior.

\section{Behavior types as sheaves}\label{sec:choices}

While the introduction has been informal so far, the work in this book is fairly technical. We
assume the reader has a good understanding of category theory, including basic familiarity with
Grothendieck toposes and internal languages; see \cite{Johnstone:2002a,MacLane.Moerdijk:1992a}. Some
familiarity with domain theory \cite{Gierz.Keimel.Lawson.Mislove.Scott:2003a}\index{domain theory},
and of frames and locales \cite{Picado.Pultr:2012a}\index{frame}\index{locale} would be useful but
is not necessary.

Consider the usual topological space of real numbers, $\RR$,\index{real numbers!topological space
of} which has a basis consisting of open intervals $(r,s)$.\index{interval!open} One might be
tempted to model a behavior-type as a sheaf $S$ on $\RR$,\index{sheaf} but we did not make this
choice for several reasons that we will soon explain. Nonetheless, sheaves on $\RR$ will be important for our work, and they give a
good starting point for the discussion, so let's briefly consider the structures and properties of a
sheaf $B$ on $\RR$.

For any open interval $(r,s)\ss\RR$, there is a set $B(r,s)$ whose elements we call
behaviors\index{behavior} of type $B$ over the interval. If $r\leq r'\leq s'\leq s$, there is a
restriction map $B(r,s)\to B(r',s')$ expressing how a system behavior over the longer interval
restricts to one over the shorter interval. Given two overlapping open intervals, say $(r_1,r_3)$
and $(r_2,r_4)$ where $r_1<r_2<r_3<r_4$, and given behaviors $b'\in B(r_1,r_3)$ and $b''\in
B(r_2,r_4)$ such that $\restrict{b'}{(r_2,r_3)}=\restrict{b''}{(r_2,r_3)}$, there exists a unique
behavior $b\in M(r_1,r_4)$ extending both: $\restrict{b}{(r_1,r_3)}=b'$ and
$\restrict{b}{(r_2,r_4)}=b''$. This sort of ``composition'' gluing\index{gluing!composition}
condition can be succinctly written as a finite limit:
\begin{equation}\label{eqn:behavior_composition}
  B(r_1,r_4)\cong B(r_1,r_3)\times_{B(r_2,r_3)}B(r_2,r_4)
\end{equation}
A gluing condition---expressed as the limit of a certain diagram---holds more generally for any open
covering of an open interval by other open intervals in $\RR$. However, it suffices to add just one more kind, which we might
call the ``continuity'' gluing condition\index{gluing!continuity}; namely those for the following sort of
telescoping inclusion:
\begin{equation}\label{eqn:behavior_continuity}
  B(r,s)\cong\lim_{r<r'<s'<s}B(r',s').
\end{equation}

These two gluing conditions---composition gluing and continuity gluing---on a functor $B$ allow us to identify it with a sheaf on $\RR$. However, we do
not take sheaves on $\RR$ as our model of behavior types for two reasons: non-composability of
behaviors and translation invariance of behavior types. We explain these next.

\subsection{Non-composable behaviors}\label{sec:noncomposable_behaviors}

One usually imagines behaviors as composable\index{behavior!composable}: given two behaviors that
match on an overlapping subinterval, they can be joined to form one behavior. This was expressed in
condition \eqref{eqn:behavior_composition}, but it is not always what we want. For example, if we glue
together two roughly monotonic curves as in \cref{eqn:relaxed_monotonicity}, the result may fail to
be roughly montonic.\index{behavior!non-composable}

Another way to see this is in terms of the internal logic \index{internal!logic} and its semantics.
\index{semantics} In any sheaf topos---though we will speak in terms of topological spaces---when a
proposition is true over some open set, it must be true over every open subset. Such a property in a
temporal setting is often called a \emph{safety property}\index{safety property}: if a system is to
be compliant over an interval of time, it must be compliant over every
sub-interval.\index{behavior!compliant} The semantics of safety properties is thus that of \emph{falsifiability}: a system
satisfies a proposition over an interval if and only if there is no subinterval on which the
proposition is false. For example, a property like ``if event $a$ happens, then event $b$ will
happen 10 seconds later'' is impossible to falsify on intervals of length less than 10 seconds, so
we must consider it to be vacuously satisfied on such short intervals. Again, if an interval is too short to falsify a proposition, the proposition is deemed true on that interval.

Now we can more clearly see why composition-gluing condition creates a problem. It says that if a proposition is true on
each element of open cover, then it is true on the union. In $\RR$, we have $(0,4)\cup(2,6)=(0,6)$,
so if a proposition is unfalsifiable on $(0,4)$ and on $(2,6)$ we must call it true on $(0,6)$.
Again consider the formula \cref{eqn:relaxed_monotonicity}. It gives an example of a proposition
that is unfalsifiable on intervals whose length is strictly less than 5, and hence would be true on
$(0,4)$ and $(2,6)$. Thus it would be forced by the composition-gluing condition to be true on
$(0,6)$---and, by induction, on an interval of any length---which is not the semantics we want.

The way we solve the above problem is by enlarging our topological space. Consider the
(non-Hausdorff) topological space $\IR$, called the \emph{interval domain}\index{interval domain},
whose points are compact intervals $[d,u]\ss\RR$. As with $\RR$, the space $\IR$ has a basis of open
sets indexed by pairs of real numbers $r<s$. Namely, for any $r<s$, the corresponding basic open is
$(r,s)\coloneqq\{[d,u]\mid r<d\leq u<s\}\ss\IR$. Then a subset $U\ss\IR$ is open if and only if it
can be expressed as a union of these basic opens.

Note that $(0,4)\cup(2,6)\neq(0,6)$ in this topology, because the left-hand side does not contain the
point (compact interval) $[1,5]$. We will see that $\IR$ is a domain, called the \emph{interval
domain}, meaning it is a topological space---the topology defined in the previous paragraph is called the Scott
topology\index{Scott topology|see {topology}}\index{topology!Scott} for the domain---and its points
have a natural poset structure that completely determines its frame\index{frame} of open sets.
Likewise, $\IR$ is a sober topological space,\index{topological space!sober} so its frame of open
sets completely determines its poset of points.

We use $\IR$ rather than $\RR$ as our main space of interest because it allows us to capture
non-composable behaviors.

\subsection{Translation-invariant behavior types}\label{sec:translation_invariance}

The second reason for not using $\RR$ persists even when we pass to $\IR$, but is simpler to
explain. Namely, both $\RR$ and $\IR$ come equipped with a specific reference point in time, the
origin. An important assumption in science is that experiments run today are still valid tomorrow,
all other things being equal. The concepts being tested may depend on durations of time, but they
are independent of the ``date'', say, the number of years since the big bang or the birth of an
influential person. Some behaviors, such as holidays, are dependent on date, but we regard such time
dependence\index{behavior type!temporal} as an additional feature---something to made
explicit---rather than as the norm.

Another reason not to use $\IR$ itself is important to be able to compare behaviors that occur at different times. But the open sets $(0,1)$ and $(5,6)$ have nothing in common in $\IR$, unless one includes the obvious translation action of $\RR$ on $\IR$\index{action!$\RR$}. Without building this action into the structure, there is no
connection between the behaviors that can occur over one interval and those that can
occur over another, and hence one cannot make any comparison. For example, it is awkward to express the concept of \emph{delay}---namely that one behavior is ``the same as'' another, despite the fact that they occur over different intervals---in the type theory of $\IR$.\index{delay}

To remedy this second problem, we work in a translation-invariant setting: A behavior that can occur over one interval could also occur over any other. In this setting, the time-line itself becomes a behavior type, we call $\Time$. Over a given interval of length $\ell$, a behavior of type $\Time$ can be regarded as the behavior of a perfect clock: it starts at some time $t_0$ and ends at $t_0+\ell$. If a specific behavior $b$ is dependent on a choice of clock time, then should will be explicit in the sense that $t:\Time$ will occur in the formula for $b$.

To get a bit more technical, our base topos $\BaseTopos$\index{base topos $\BaseTopos$|see {topos, of
behavior types}} is a quotient of $\Shv{\IR}$, i.e.\ there is a geometric surjection
$p_*\colon\Shv{\IR}\to\BaseTopos$. One characterization of geometric surjections is that the inverse
image part, in this case $p^*\colon\BaseTopos\to\Shv{\IR}$, is faithful. Another is that
$\BaseTopos$ is the category of coalgebras for the left-exact comonad $p^*p_*$ on $\Shv{\IR}$.
Intuitively, objects in $\BaseTopos$ can be thought of as sheaves $X\in\Shv{\IR}$ that are
\emph{translation-invariant}, i.e.\ for which one has coherent isomorphisms $X(a,b)\cong X(a+r,b+r)$
for every open interval $(a,b)$ and real $r\in\RR$. This is the translation action of $\RR$ on
$\IR$, mentioned above,\index{translation} and we denote by $\IRinv$ the localization of
$\IR$ at the collection of translation maps. $\IRinv$ is no longer a space, i.e.\ it is a category not a poset, but it has a natural
site structure and $\BaseTopos\coloneqq\Shv{\IRinv}$ is the corresponding topos of sheaves.

In
fact, $\BaseTopos$ is an \'etendue\index{etendue@\'{e}tendue|see {topos,
\'{e}tendu}}\index{topos!\'{e}tendue}, meaning there is a specific object, namely $\Time$, such that
the slice topos $\BaseTopos/\Time\cong\Shv{\IR}$ is localic.\index{topos!localic}
\[
\begin{tikzcd}[column sep=large]
  \Shv{\IR}\;\cong\;\BaseTopos/\Time \ar[r,shift left=7pt, "p_*"] \ar[r, shift right=7pt, "p_!"'] &
  \BaseTopos\;\cong\;\Shv{\IRinv} \ar[l, "p^*" description]
\end{tikzcd}
\]
The word \'etendue means ``extent''\index{extent} and indeed objects in $\IRinv$ are extents---or
durations---of time, over which behaviors can occur.

\subsection{Four relevant toposes}\label{sec:four_relevant_toposes}

All four toposes in the square below play a role in this work:\index{geometric morphism}\index{topos}
\[
\begin{tikzcd}
  \Shv{\RR}\ar[r,tail]\ar[d,two heads]&\Shv{\IR}\ar[d,two heads, near start, "p_*"]\\
  \BaseTopos_\pi\ar[r,tail]&\BaseTopos
\end{tikzcd}
\]
The toposes in the top row comprise temporal sheaves---those over a specified time-line---whereas
those in the bottom row comprise behaviors that are translation-invariant. Sheaves in the left
column have composable behaviors, whereas sheaves in the right column can include more general,
non-composable behaviors.\index{behavior!composable}\index{behavior!non-composable}\index{behavior
type!translation-invariant}\index{behavior type!temporal} We discussed $\Shv{\RR}$, $\Shv{\IR}$, and $\BaseTopos=\Shv{\IRinv}$ above.

The horizontal maps are geometric embeddings, i.e.\ they correspond to modalities, and the vertical
maps are geometric surjections. The topos $\BaseTopos_\pi$ and the left-hand and bottom maps are uniquely determined from the top and right-hand maps, by way of the
surjection-embedding factorization systems on toposes. The $\pi$ stands for
``pointwise'',\index{modality!pointwise $\pi$} which means that properties of sheaves in
$\BaseTopos_\pi$ can all be determined in neighborhoods of points in $\RR$.

\section{Temporal type theory}

The main subject of this book is the definition of a type theory---and its associated logic---that has its semantics in the topos $\BaseTopos=\Shv{\IRinv}$ discussed above.

The logic we present is higher-order logic, plus subtypes and quotient
types.\index{logic!higher-order}\index{type!subtype}\index{type!quotient} Higher-order logic---as
well as its strong connection with topos theory---has been very well-studied
\cite{Lambek.Scott:1988a,Fourman:1977a,Boileau.Joyal:1981a,Awodey:2016,Johnstone:2002a}, which is
part of our motivation for using it. In fact, toposes support not just higher-order logic, but also
dependent type theories.\index{type!dependent} This means that one can use an automated
proof-assistant based on dependent type theory---such as Coq or Lean\index{automated proof assistant}
\cite{Coquand.Huet:1988a,deMoura.Soonho.Avigad.VanDoorn.vonRaumer:2015a}---to validate proofs.

The primary goal of the type theory and logic we present is to support defining and reasoning about
behavior types.\index{logic!temporal} As such, the logic is capable of expressing statements from standard temporal logics, such as Linear Temporal Logic (LTL). The primary ``temporal operator'' from LTL is called \emph{until}. The
meaning of until---written $\Until$ in the logic---is often presented by a formula such as
\begin{equation}\label{eqn:temporal_until}
  (\phi_1 \Until \phi_2)(t) \coloneqq \exists(r:\Time)\ldotp \bigl[(t<r) \wedge \phi_2(r) \wedge
    \forall u\ldotp (t<u<r \imp \phi_1(u))\bigr].
\end{equation}
Here the temporal propositions $\phi_1$ and $\phi_2$ are being represented as propositions that have
an explicit dependence on time, i.e.\ as functions $\Time\to\Prop$. As one can attempt to read off
from \cref{eqn:temporal_until}, $\phi_1\Until\phi_2$ constructs a new temporal proposition which is
true if both the following hold: $\phi_2$ is true sometime in the future, and $\phi_1$ is true from
now until that point.

In fact, \cref{eqn:temporal_until} is an example of a formula in our logic. The type $\Prop$ is what
makes the logic higher-order, and the type $\Time$---which is definable from the one
atomic term of our theory, and whose semantics is the behavior type $\Time$ discussed in
\cref{sec:translation_invariance}---allows for temporal statements like $\phi_1\Until\phi_2$ to be
expressed. It is from this perspective that we find it reasonable to refer to our system as
\emph{temporal type theory}. See \cref{sec:temporal_logic} for a more in-depth discussion of how our
type theory and logic relates to existing temporal logics, in particular to LTL and Metric Temporal
Logic.

Every property we can discuss is a safety property\index{safety property}, in the sense described in
\cref{sec:noncomposable_behaviors}, so all of our connectives and quantifiers take safety properties
to safety properties. An example of the expressive power of the temporal type theory: we are able to
internally define real-valued functions of time, as well as their derivatives, and prove that the
derivative satisfies the Leibniz rule.

As a first test that our formal system is strong enough to be useful in practice, we use a
simplified version of the safe separation problem\index{National Airspace System!safe separation}
for airplanes in the US National Airspace System (NAS). There one wants to avoid situations in which
airplanes get too close to one another. To achieve this, the current system consists of an
interaction between radar, a Traffic Collision Avoidance System (or TCAS), pilot decision-making,
and actuators and thrusters on the surface of the airplane. The position of each airplane roughly
follows a differential equation, with time-varying parameters such as ``climb at rate $r(t)$'',
supplied by the pilot. The pilot's decisions take into account the commands from air traffic control
and the advice of the TCAS, which alerts the pilot to urgent situations and suggests corrective
maneuvers. In turn, this information is determined by the relative position of the airplanes,
completing the loop.

The NAS thus requires continuous interaction between many different types of behavior. Some of the
components are modeled continuously\index{dynamical system!continuous}, such as the motion of the
plane, while others are modeled discretely\index{dynamical system!discrete}, such as the TCAS alerts
and suggestions. The pilot generally takes the advice given and carries it out after some delay. It
is the combination of continuous, discrete, and delay behaviors that we believe is the essence of
the safe separation problem. In \cref{sec:applications} we will prove a version of the safe
separation property---a version that is greatly simplified but that still involves the above
essential elements---in which there is only one airplane that must obtain safe separation from the
ground.

\section{Related work}\label{sec:related_work}

\cite{Spivak.Vasilakopoulou.Schultz:2016a} presents a topos similar to that described above, in \cref{sec:choices}. The authors
(ourselves plus C. Vasilakopoulou) consider sheaves and presheaves on a certain category $\Int$ of
intervals: sheaves for behavior types and presheaves for behavior contracts. Sheaves on $\Int$ can be identified
with discrete Conduch\'{e} fibrations\index{discrete Conduch\'{e} fibration}---or unique
factorization liftings---over the monoid $(\RR_{\geq0},0,+)$ of nonnegative real numbers. This may
have interesting connections with decomposition spaces \cite{Galvez.Kock.Tonks:2015a}, but more
closely related are applications to dynamical systems, pursued in
\cite{Lawvere:1986a,Bunge.Fiore:2000a,Fiore:2000a}.\index{dynamical systems!categorical approaches}

This book considers instead a subtopos $\BaseTopos\ss\Shv{\Int}$, obtained by sheafifying $\Int$-sheaves or
$\Int$-presheaves with respect to the ``continuity'' gluing condition \eqref{eqn:behavior_continuity}, and
hence building in continuity at the ground floor. This is a natural step, as the category $\Int\op$
is already a continuous category\index{category!continuous} in the sense of
Johnstone and Joyal, and $\BaseTopos$ is equivalent to the topos of continuous functors
$\Int\op\to\Cat{Set}$; see \cite{Johnstone.Joyal:1982a} or \cref{chap:IRinv_continuous}. The present work develops a type theory with a semantics in the sheaf
topos $\BaseTopos$, but which can be used independently of the sheaf semantics.

There are also connections with \cite{Joyal.Nielsen.Winskel:1996a} and following works, where dynamics is considered in terms of morphisms from a given category of paths. For example in \cite{Haghverdi.Tabuada.Pappas:2003a} an object in the path category is an interval ``modeling a clock running on [an open interval] at unit rate'', a perspective quite similar to our own.

Our general approach to dynamical systems also has something in common with that of the early
cyberneticists,\index{cybernetics} such as Ashby or Weiner; for example, our notion of behavior type
is roughly what  Ashby \cite{Ashby:2013a} calls the ``field'' of a system. However, our work is more closely aligned with the relatively recent ``behavioral approach,''\index{behavioral approach} as
advocated in \cite{Willems:2007a}. There, a dynamical system is defined to be a triple
$(\mathbb{T},\mathbb{W},\mathcal{B})$, where $\mathbb{T}$ represents time, and which we fix to be
$\mathbb{T}=\RR$, where $\mathbb{W}$ is an arbitrary set of ``signal values,'' and where
$\mathcal{B}\subseteq\mathbb{W}^{\mathbb{T}}$ is the set of possible behaviors as a subset of all
functions $\mathbb{T}\to\mathbb{W}$.

An object of the category $\BaseTopos$ is closely related to such a dynamical system. The primary
difference is that, instead of specifying only those possible ``infinitely extended'' behaviors
$\mathbb{T}\to\mathbb{W}$, an object $B\in\BaseTopos$ must specify a subset
$B(a,b)\subseteq\mathbb{W}^{(a,b)}$ for every open interval $(a,b)\subseteq\RR$, subject to the
conditions that for any $f\colon(a,b)\to\mathbb{W}$,
\begin{itemize}
  \item if $f\in B(a,b)$ then $a\leq a'<b'\leq b$ implies $f|_{(a',b')}\in B(a',b')$;
  \item if $f|_{(a',b')}\in B(a',b')$ for any $a<a'<b'<b$, then $f\in B(a,b)$;
  \item if $f\in B(a,b)$, then $[t\mapsto f(t-r)]\in B(a+r,b+r)$ for any $r\in\RR$.
\end{itemize}
We find this to be more natural from an applications perspective, since an infinitely extended
behavior is by definition unobservable.

Behavior types $B\in\BaseTopos$ are also more general, in that we can consider behavior types in
which there exist behaviors $B(0,5)$ which admit no extension to $B(0,10)$, say. As an example of
this phenomenon, we could define $B(a,b)$ to be the set of all differentiable functions
$f\colon (a,b)\to\RR$ satifying the differential equation $f'=1+f^2$. Then
$\tan(t)\colon(-\frac{\pi}{2},\frac{\pi}{2})\to\RR$ defines a behavior in
$B(-\frac{\pi}{2},\frac{\pi}{2})$ for which there is no extension to a behavior on any larger
interval.

As another example, consider the behavior type $B$ of possible trajectories of population levels of
wolves and deer in some ecosystem. For the sake of analysis, we may want to consider the subobject
$B'\hookrightarrow B$ consisting of only the desired behaviors, say where neither the wolves nor the
deer go extinct. It is certainly possible for the ecosystem to be in a state in which, according to
the model being used, there does not exist a future in which neither species goes extinct.

Finally, as mentioned earlier in this introduction, one of the goals of this work is to serve as a
``big tent'', so that many approaches to models of systems can be embedded into our category
$\BaseTopos$, and many temporal logics invented to analyze those models can be embedded into our
axiomatics. In that sense, our goal is not to compete with other bodies of work, but to help weave
them all together.

\section{Notation, conventions, and background}

Here we lay out some of the basic notation and conventions used throughout the book. We then give a lightning review of sites, sheaves, and toposes.

\subsection{Notation and conventions for categories}

\begin{notation}\label{not:basic_for_cats}\index{notation}
  We say that a set $S$ is \emph{inhabited} if there exists an element $s\in S$. This is the constructive way to say that $S$ is nonempty. Similarly, we say that a category is inhabited if it has an object.\index{set!inhabited}
  
  For any $n\in\NN$, we sometimes write $n$ to denote the linearly ordered category
  $\bullet^1\to\cdots\to\bullet^n$ with $n$ objects.

  All categories in this book are 1-categories unless otherwise stated. We denote composition of
  morphisms $A\To{f}B\To{g}C$ in the classical ``Leibniz'' order $g\circ f\colon A\to C$, and we
  denote identity morphisms as $\id$. If $\cat{C}$ is a category, we may denote the hom-set between
  objects $c,c'\in\cat{C}$ either by $\cat{C}(c,c')$ or $\Hom_{\cat{C}}(c,c')$, or simply $[c,c']$ if
  $\cat{C}$ is known from context.
\end{notation}

\subsection{Definition of sites, sheaves, and toposes}

Whenever we speak of a topos in this book, we always mean a Grothendieck topos, which is the
category of sheaves on a site. We quickly remind the reader of the definition of site and topos. Readers for whom the definition seems overly abstract may simply skim it on a first reading. The
following is taken from \cite{Johnstone:2002a}.

\begin{definition}[Coverage, site, sheaf, topos, geometric morphism]
\label{def:coverage_site_sheaf_topos}\index{coverage}\index{site}\index{sheaf}\index{topos}\index{topology!Grothendieck|see {coverage}}\index{covering family|see {family}}
  Let $\cat{C}$ be a category. For any object $U\in\cat{C}$, a \emph{family over $U$} is a subset of
  $\Hom_{\cat{C}}(-,U)$. Explicitly, a family consists of an object $U$, a set $I$, and for each
  $i\in I$ a morphism $f_i\colon U_i\to U$ for some $U_i\in\cat{C}$. Let $\Phi(\cat{C})$ denote the
  set of families in $\cat{C}$.

  A \emph{coverage} $\chi$ consists of a set $T$ and a function $F\colon T\to \Phi(\cat{C})$
  satisfying the following condition. Suppose given an object $U\in\cat{C}$, a morphism $g\colon
  V\to U$, and a $t\in T$, such that $F(t)=(f_i\colon U_i\to U\mid i\in I)$. Then there exists a
  $t'\in T$ such that $F(t')=(h_j\colon V_j\to V\mid j\in J)$ and such that for each $j\in J$ there
  exists some $i\in I$ and some $g_{j,i}\colon V_j\to U_i$ making the following square commute:
  \[
  \begin{tikzcd}
    V_j\ar[r, "h_j"]\ar[d, "g_{j,i}"']&V\ar[d, "g"]\\
    U_i\ar[r, "f_i"']&U
  \end{tikzcd}
  \]
  A \emph{site} is a category $\cat{C}$ equipped with a coverage $\chi=(T,F)$. For any $t\in T$, we
  say $F(t)=(U,I,f)$ is a \emph{$\chi$-covering family over $U$}.

  Suppose given a functor $X\colon\cat{C}\op\to\Cat{Set}$, a family $(f_i\colon U_i\to U\mid i\in
  I)$, and an element $x_i\in X(U_i)$ for each $i\in I$. We say that the $x_i$ are
  \emph{compatible}\index{family!compatible} with respect to the family if, for every pair of
  elements $i,i'\in I$, object $W\in\cat{C}$, and commutative square
  \[
  \begin{tikzcd}
    W\ar[r, "g_i"]\ar[d, "g_{i'}"']&U_i\ar[d, "f_i"]\\
    U_{i'}\ar[r, "f_{i'}"']&U
  \end{tikzcd}
  \]
  the equation $X(g_i)(x_i)=X(g_{i'})(x_{i'})$ holds in $X(W)$.

  A \emph{sheaf} on a site $S=(\cat{C},\chi)$ is a functor $B\colon\cat{C}\op\to\Cat{Set}$ such that
  for each $\chi$-covering family $(f_i\colon U_i\to U\mid i\in I)$ and compatible family $\{b_i\mid i\in
  I\}$, there exists a unique $b\in B(U)$ such that $B(f_i)(b)=b_i$ for each $i\in I$. A
  \emph{morphism of sheaves} $B\to B'$ is simply a natural transformation of functors, and this
  defines the category of sheaves on $(\cat{C},\chi)$. The category of sheaves on the site $S$ is denoted
  $\Shv{S}$.

  A \emph{topos} is any category $\cat{E}$ for which there exists a site $S$ and an equivalence of
  categories $\cat{E}\cong\Shv{S}$. If $\cat{E}'$ is another topos, a \emph{geometric morphism} $f\colon\cat{E}\to\cat{E'}$ is a functor which has a left adjoint, such that the left adjoint preserves finite limits.
\end{definition}

\section{What to expect from the book}

We hope the reader can learn useful things by reading this book. Of course, what is learned depends
on the reader: her background, interests, level of effort, etc. Before giving an outline of the
chapters, we will discuss some potential items of interest.

\subsection{What the reader can hope to learn}

Unfortunately, this book could not be written for a general audience. For example, the authors are well-aware that anyone unfamiliar with toposes probably had difficulty knowing how to think about \cref{def:coverage_site_sheaf_topos}. There is certainly material in this book that is amenable to such readers, including \cref{sec:IR_half_plane}, \cref{sec:logical_prelims}, \cref{sec:axiomatics}, and \cref{sec:Real_numbers,sec:applications}.

However, as mentioned above in
\cref{sec:choices}, the book is mainly written for readers who have seen
toposes before. We would like to think students of all levels can learn something about toposes by reading this book.
Indeed, it can be considered as an extended example of a single topos $\BaseTopos$, as well as
its slices and subtoposes. \Cref{sec:logical_prelims} is a stand-alone chapter that is meant to introduce readers to type theory and logic, especially as they relate to toposes. For example, we provide detailed type-theoretic accounts of real numbers and other numeric objects that have semantics in an arbitrary topos.

While some knowledge about Scott domains would certainly be useful at times, it is not necessary. Again, we hope this
book will provide insight into the subject of domain theory by offering an extended example. Semantically, we
will work in a quotient of the domain $\IR$, but we will also have occasion to consider
domains internal to a topos. For example, the derivative of a continuous function is an
interval-valued function, and it is defined in a domain-theoretic style (see \cref{def:derivative}).

Finally, the reader will also learn about temporal logic. We spend a bit of space to try and convey
how existing temporal logics fit into our theory. But mostly, we hope the reader will have interest
in our own particular \emph{higher order temporal logic}, as it motivates everything in the book.

\subsection{Contributions}\label{sec:contributions}

Our main contribution is to offer a new temporal type theory (the first of its kind as far as we
know), together with a novel topos-theoretic semantics. We believe it can mediate between several
existing formal systems for dealing with time, e.g.\ for describing cyber-physical systems.

Along the way to understanding Dedekind real numbers---and generalizations like proper and improper
intervals---in our topos-theoretic semantics, we were led to a number of results which we believe
might be of independent interest. When using the temporal type theory to reason about behaviors
which are represented by continuous real-valued functions, we found the need to make use of Dedekind
real number objects in various subtoposes as well as the standard Dedekind real number object. In
the type theory, these appear as objects defined using versions of the standard Dedekind axioms
which have been modified by modal operators. Studying these ``modal Dedekind real number objects''
was greatly simplified by considering generalizations which remove some of the standard Dedekind
axioms. The reason is that these generalized numeric objects form \emph{domains}, which one can
think of as a particularly nice class of topological spaces which are intimately connected to order
theory. We found that for particular kinds of subtopos---in particular closed, quasi-closed, and
dense proper subtoposes---and for particular kinds of domains which are presented by sufficiently
nice bases, we could give strong comparisons between domains in a subtopos and domains in the
enclosing topos.

These general domain-theoretic results are collected in \cref{sec:predomain}---a section we tried to make readable,
independently of the rest of the book---while the special cases about generalized Dedekind numeric
objects are collected in \cref{sec:Real_numbers}. As an example application, the theory of
differentiation that we develop in the temporal type theory in \cref{sec:derivatives} makes
essential use of generalized Dedekind numeric objects in several different subtoposes.

In addition, we contribute a few other new ideas to the literature. In \cref{sec:proof_IR_sheaves}
we prove a result about dense morphisms of posites. While this formally follows from results in
\cite{Shulman:2012a}, we have provided a new direct proof of the simpler posite case. In
\cref{sec:toposes_types_logic} we give an informal introduction to type theory and higher-order
logic, a subject for which informal accounts seem to be lacking in the literature.

In \cref{sec:applications} we generalize the usual notion of hybrid system from the control theory
literature. We also show how to integrate several different temporal logics into our own.

Finally in \cref{sec:connected_dbifib_factorization} we prove the existence and various
properties---e.g.\ regarding interaction with the $\Ind$completion---of the (connected, discrete
bifribration) orthogonal factorization system on $\Cat{Cat}$. Although this material is well-known,
it seemed difficult to locate in the literature.

\subsection{Chapter outline}

This book has eight chapters and two appendices. \Cref{sec:intro} has hopefully summarized and explained the goals of this work: to study temporal properties of a very general class of behavior types, using the language and tools of toposes. We said that each behavior type is modeled as a set of possible behaviors over each interval of time, namely as sheaves on some sort of time-line. However, we made two choices that keep this from being as simple as one might be tempted to expect. The first, and more important, is that matching behaviors on overlapping intervals need not be composable (\cref{sec:noncomposable_behaviors}); the second is that the set of possible behaviors over an interval should not depend on its position in the time-line (\cref{sec:translation_invariance}).

We formalize these two ideas in the following two chapters. In \cref{sec:the_interval_domain} we define the time-line we will use, which is called the interval domain and denoted $\IR$. The real line sits inside it as the length-0 intervals $\RR\ss\IR$. We give a continuous bijection between $\IR$ and the upper half-plane in $\RR^2$, which not only allows the reader to visualize $\IR$ but also proves quite useful for semantic purposes throughout the book. In this chapter we also review the definition of posites, $(0,1)$-sheaves, Scott domains, and give four equivalent definitions of the time-line as a posite.

In \cref{chap:translation_invariance} we deal with translation invariance, roughly by taking the quotient of $\IR$ with respect to the translation action of the group $\RR$. The result is a category $\IRinv$, which is no longer a Scott domain, but instead a continuous category, in the sense of Johnstone and Joyal. The topos of sheaves on the corresponding site is $\BaseTopos$, our main topic of study throughout the book.

We transition from the external viewpoint to the internal viewpoint in \cref{sec:logical_prelims}. This chapter is fairly independent of the rest of the book, standing as a review of some connections between toposes, type theory, and higher-order logic. In particular, we spend the first half of the chapter reviewing these notions at a high-enough level not to get bogged down in specifics, but allowing the book to be fairly self-contained in terms of its type theory and logic. We then provide a short section recalling the notion of modalities $j$---also known as Lawvere-Tierney topologies---and their relationship to subtoposes. We spend the remainder of the chapter discussing real numbers and related numeric objects in subtoposes. In particular, we explain $j$-local arithmetic and inequalities.

The technical heart of the book is in \cref{sec:axiomatics,sec:soundness}. In \cref{sec:axiomatics}, we axiomatize the higher-order logic of our temporal type theory, which includes one atomic predicate---defining $\Time$---and several axioms. We also discuss a few modalities that correspond to important subtoposes of $\BaseTopos$. In \cref{sec:soundness} we prove the soundness of our axioms in $\BaseTopos$. This requires explaining the semantics of the various numeric objects, such as the real numbers, as well as the semantics of the modalities.

In \cref{sec:Real_numbers}, we work with numeric objects relative to various modalities $j$. For example, using what we call the ``point-wise'' modality, we have access to the sheaf of real-valued functions on the usual real line, inside of $\BaseTopos$. In particular, we compare these numeric types internally for differing modalities in \cref{sec:Real_numbers_modalities} and give their $\BaseTopos$-semantics in \cref{sec:semantics_numerics_modalities}. This section relies on technical work from \cref{sec:predomain}. Perhaps most interestingly, in \cref{sec:derivatives} we internally define the derivative of such a real-valued function with respect to $t:\Time$ and prove that this definition is linear and satisfies the Leibniz rule. We also prove externally that its $\BaseTopos$-semantics is that of derivatives in the usual sense.

The main body of the book concludes with \cref{sec:applications}, where we discuss several applications of the work. For example we give an embedding of discrete, continuous, and hybrid dynamical systems into our temporal type theory. We also explain delays, and our general perspective on behavior contracts for interconnected systems. One of our main inspirations for this work was a case study involving safe separation in the National Airspace System; this is discussed in \cref{sec:case_study}. Finally in \cref{sec:temporal_logic} we discuss the relationship between our higher-order temporal logic and some of the better known temporal logics from the literature.

The book also has two appendices. In \cref{sec:predomain} we define a technical tool---which we call predomains---by which to reduce the complexity of domains and the morphisms between them. We explain how our constructions work relative to arbitrary modalities, i.e.\ within arbitrary subtoposes. Finally, in \cref{chap:IRinv_continuous} we prove that $\IRinv$ is a continuous category in the sense of Johnstone and Joyal.

\section{Acknowledgments}
{
	\label{sec:acknowledgments}	
	Many of the ideas presented here began to take form in 2015 -- 2016 while working closely with 
Christina Vasilakopoulou, a postdoc in our group at the time. We benefitted greatly from---and 
thoroughly enjoyed---this collaboration.
	
	We have learned a great deal from our collaborators Kevin Schweiker, Srivatsan Varadarajan, and 
especially Alberto
	Speranzon at Honeywell Labs, as well as from our sponsors
	at NASA, including Alwyn Goodloe, Anthony Narkawicz, and Paul Minor.	We also want to thank Aaron 
Ames, Brendan Fong, Tobias Fritz, Ben Sherman, and R\'{e}my Tuy\'{e}ras for helpful discussions. 
Thanks also to Ben Sherman and R\'{e}my Tuy\'{e}ras for comments on an early draft of this book.
	
	The work presented here was supported by NASA grant NNH13ZEA001N, as well as by AFOSR grants 
FA9550--14--1--0031 and FA9550--17--1--0058.
}

\chapter{The interval domain}\label{sec:the_interval_domain}

In this chapter we will introduce the \emph{interval domain}\index{interval domain} $\IR$, which is
a topological space that represents the line of time in our work to come. The points of this space
can be thought of as compact intervals $[a,b]$ in $\RR$. The specialization
order\index{specialization order!on points} on points gives $\IR$ a non-trivial poset structure---in
fact it is a domain---and as such it is far from Hausdorff.\index{topological space!domain as}

The topos of sheaves on the space $\IR$ will play a very important role throughout the book, so we
begin in \cref{sec:posite} with a review of sheaves on topological spaces, or more generally on
posets equipped with a coverage. In \cref{sec:rev_cont_poset} we review the theory of domains, and
we define $\IR$ in \cref{sec:interval_domain}. In \cref{sec:IR_half_plane} we show how to view $\IR$ in terms of the usual Euclidean upper half-plane. Finally in \cref{sec:proof_IR_sheaves} we discuss Grothendieck posites, which allow us to prove the equivalence between four different formulations of the topos of sheaves on $\IR$.

There are a few places in this chapter where we refer to predomains, which are fairly technical and are the
subject of \cref{sec:predomain}. However, none of that material is necessary to understand the
the present chapter. It will become more important for the technical results about
arithmetic between real number objects in various subtoposes, which we will discuss later in
\cref{sec:Dedekind_j}. For the time being, we suggest the reader look briefly at the
predomain material if the interest arises, but otherwise feel free to skim it or take it on
faith.

\section{Review of posites and $(0,1)$-sheaves}\label{sec:posite}

The material in this section, \ref{sec:posite}, is largely taken from \cite{Nlab:posite}. For any poset $(P,\leq)$ and
$p\in P$, the down-set $\down p\coloneqq\{p'\in P\mid p'\leq p\}$ will play a fundamental role
throughout this section. For a set $S$, let $P(S)\coloneqq\{V\mid V\ss S\}$ denote its power set.
Given $V\ss S$, we write $\down V$ to denote the set $\bigcup_{v\in V}\down v$.

A posite is a site---i.e.\ a category equipped with a coverage (see
\cref{def:coverage_site_sheaf_topos})---for which the category is a
poset.\index{poset}\index{coverage}\index{posite}\index{coverage!on a posite} Any topological space
has an underlying posite: the poset is that of open subsets, and covering families are collections of open sets whose union is another.\index{topological space!as
site}\index{site!from a topological space|see {posite}} The definition of a coverage can be
simplified slightly on categories that are posets\index{poset!as category}. 

\begin{definition}[Coverage on poset]\label{def:coverage}\index{poset!coverage on|see {posite}}
  A \emph{coverage} on a poset $S$ is a relation ${\basicCover}\subseteq P(S)\times S$ such that:
  \begin{itemize}
    \item If $V\basicCover u$ then $V\subseteq\specdownclose u$.
    \item If $V\basicCover u$ and $u'\leq u$, then there exists a $V'\subseteq (\specdownclose V\cap\down u')$
      such that $V'\basicCover u'$,
  \end{itemize}
  If $V\basicCover u$, we say that $V$ is a \emph{basic cover}\index{cover, basic} of $u$. A
  \emph{posite} is a poset equipped with a coverage.
\end{definition}

\begin{example}
  If $B$ is a basis for a topological space $X$, then $B$ ordered by inclusion is a posite. For any
  $b\in B$ and $V\subseteq B$, we have $V\basicCover b$ iff $\bigcup V = b$ (where the union is
  taken as subsets of $X$, which is not necessarily the join in $B$).
\end{example}

As a posite is in particular a site,\index{site} the usual definition of sheaves of sets applies,
constructing a (localic) topos from any posite; see \cref{def:coverage_site_sheaf_topos} or
\cite[Theorem C.1.4.7]{Johnstone:2002a}. However, for any posite $S$ we can also consider ``sheaves
of truth values,'' or $(0,1)$-sheaves,\index{sheaf!of truth values|see {sheaf,
(0,1)}}\index{sheaf!(0,1)} thus constructing a frame $\PoShv{S}$. If $B$ is a basis for a
topological space $X$,\index{basis!for topology} then $\PoShv{B}$ will simply be the frame of
opens\index{frame} in $X$.

\begin{definition}[$(0,1)$-sheaf]\label{def:poset_ideal}
  Let $S$ be a posite. A \emph{$(0,1)$-sheaf}%
  \footnote{$(0,1)$-sheaves are also sometimes called ideals, but this terminology clashes with two
  other notions we use in this book that are also called ideals---namely, ideals in a poset and
  rounded ideals in a predomain (see \cref{def:ideal_dcpo,def:rounded_ideal_filter})---thus we call
  them $(0,1)$-sheaves to avoid confusion.\index{ideal!on a posite|see {sheaf, (0,1)}}}
  is a subset $I\subseteq S$ such that:
  \begin{description}[labelwidth=1.3in]
    \item[\quad Down-closure:] If $u\leq v$ and $v\in I$, then $u\in I$.
    \item[\quad Sheaf-condition:] If $V\basicCover u$ and $V\subseteq I$, then $u\in I$.
  \end{description}\index{sheaf condition}
  The poset of $(0,1)$-sheaves on $S$ is denoted $\PoShv{S}$.
\end{definition}

A coverage on a category $\cat{C}$ is called \emph{subcanonical}
\index{coverage!subcanonical} if the representable functor $\cat{C}(-,c)$ is in fact a sheaf, for every $c\in\cat{C}$. The coverage is called
\emph{canonical}\index{coverage!canonical} if it is the largest subcanonical coverage.

\begin{example}\label{ex:canonical_coverage_frame}\index{frame!canonical coverage of}
  A frame is a poset with infinite joins, finite meets, and which satisfies the infinite distributive law.
  Any frame $S$ has a \emph{canonical coverage}, with $V\basicCover u$ iff $\sup V=u$. For any
  topological space $X$, $\Shv{X}\iso\Shv{\Omega(X)}$, where $\Omega(X)$ is the frame of opens of
  $X$ with its canonical coverage.
\end{example}

\section{Domains and posites}\label{sec:rev_cont_poset}\index{domain|(}

In this section we will quickly review domains in general; see \cite{Gierz.Keimel.Lawson.Mislove.Scott:2003a} for a thorough treatment. We also show how to endow an arbitrary domain with a natural posite structure.

\subsection{Review of domains (continuous posets)}

Domains come at the intersection of order theory and topology. They can be thought of as spaces for
which the specialization order on points determines the topology, or as posets with ``two extra left
adjoints'', as we will see in \cref{def:domain}. In this section and throughout
the rest of the book, we often use the notation $\sqss$ for the order relation on domains.

\begin{definition}[Directed sets, dcpos, ideals, down-closure]
\label{def:ideal_dcpo}\index{directed}\index{dcpo|see {poset, directed complete}}\index{poset!directed complete}\index{partial order|see {poset}}\index{ideal}\index{down-closed}
  Let $(S,\sqss)$ be an arbitrary poset. A subset $D\subseteq S$ is called \emph{directed} (or \emph{up-directed}) if $D$ is
  inhabited and, for any $u_1,u_2\in D$, there exists a $u'\in D$ with $u_1\sqss u'$ and $u_2\sqss
  u'$. A \emph{directed-complete poset (dcpo)} is a poset $S$ for which every directed subset $D\ss
  S$ has a supremum $\sup D$. A subset $I\subseteq S$ is called an \emph{ideal} if it is directed
  and down-closed: $i\in I$ and $j\sqss i$ implies $j\in I$. The collection $\Id(S)$ of all ideals
  in $S$ is a poset ordered by inclusion. Given any element $u\in S$, its \emph{down-closure}
  $\mathord{\downarrow}u \coloneqq \{\,v\in S\mid v\sqss u\,\}$ is an ideal, and this map defines a
  monotonic map $\down\colon S\to\Id(S)$.
\end{definition}

\begin{proposition}\label{prop:directed_complete_poset}
  Let $S$ be a poset. The down-closure $\down\colon S\to\Id(S)$ has a left adjoint if and only if $S$ is
  directed-complete. The left adjoint $\mathrm{sup}\colon\Id(S)\to S$ sends any ideal $I\subseteq S$
  to its supremum $\sup I$.\index{directed-complete}\index{supremum}
\end{proposition}

\begin{definition}[Way below]
\label{def:way_below}\index{way-below}
  Let $(S,\sqss)$ be a directed-complete poset and let $u,v\in S$ be elements. One says that $u$ is \emph{way below} $v$,
  denoted $u\ll v$, if the following is satisfied for each ideal $I\subseteq S$: if $v\sqss\sup I$, then $u\in I$.
\end{definition}

\begin{remark}\label{rem:way_below}
  If $u\ll v$, then $u\sqss v$ since $v=\sup\specdownclose v$. From the definition, it is easy to check that if
  $u'\sqss u$ and $u\ll v$ and $v\sqss v'$, then $u'\ll v'$. It follows that $\ll$ is transitive.
\end{remark}

\begin{proposition}\label{prop:cont_poset}
  Let $S$ be a directed-complete poset. The following are equivalent:
  \begin{itemize}
    \item For every $u\in S$, the set $\downclose u \coloneqq \{\,v\in S\mid v\ll u\,\}$ is an ideal, and $u
      = \sup\downclose u$.
    \item The functor $\mathrm{sup}\colon\Id(S)\to S$ has a further left adjoint.
  \end{itemize}
\end{proposition}

The left adjoint to $\mathrm{sup}\colon\Id(S)\to S$ sends an element $u\in S$ to $\downclose u$.

\begin{definition}[Domain]
\label{def:domain}\index{poset!continuous|see {domain}}\index{domain!as poset}
  A directed-complete poset $(S,\sqss)$ satisfying one of the equivalent conditions of
  \cref{prop:cont_poset} is often called a \emph{continuous poset}; we will usually call it a
  \emph{domain}.%
  \footnote{This nomenclature is the same as that in \cite[Definition
  I-1.6]{Gierz.Keimel.Lawson.Mislove.Scott:2003a}. Continuous posets, and their generalizations to
continuous categories, are also discussed in \cite[C.4.2]{Johnstone:2002a} and
\cite{Johnstone.Joyal:1982a}.}
  The order $\sqss$ is called the \emph{specialization order}.\index{specialization order!on a
  domain|see {domain, as a poset}}
\end{definition}

\begin{example}
The category $2$ (see \cref{not:basic_for_cats}) is posetal: $0\sqss 1$. Every nonempty subset of it is directed. It is directed-complete and its ideals are $\{0\}$ and $\{0,1\}$. Its way-below relation coincides with its order $a\ll b\iff a\sqss b$, and it is a domain.
\end{example}

\begin{proposition}\label{prop:Scott_opens}
  Let $S$ be a directed-complete poset and $U\ss S$ an up-closed subset.\index{up-closed} Then the
  following are equivalent:
  \begin{enumerate}
    \item[1.] for any directed set $D\subseteq S$, if $\sup D\in U$ then there exists $s\in S$
      with $s\in D\cap U$.
    \item[2.] for any ideal $I\subseteq S$, if $\sup I\in U$ then there exists $s\in S$ with
      $s\in I\cap U$.
  \end{enumerate}
  If $S$ is a domain, then the above are also equivalent to the following:
  \begin{enumerate}
    \item[3.] for all $u\in U$ there exists some $u'\in U$ with $u'\ll u$.
  \end{enumerate}
\end{proposition}
\begin{proof}
  This is \cite[Lemma II-1.2]{Gierz.Keimel.Lawson.Mislove.Scott:2003a} and \cite[Proposition
  II-1.10]{Gierz.Keimel.Lawson.Mislove.Scott:2003a}, but it can also be checked directly. For
  example, $2\imp 1$ because for directed $D$, the down-closure $I\coloneqq\down D$ is an ideal, and
  given $s\in I\cap U$ we have $s\in U$ and $s\leq d$ for some $d\in D$. Since $U$ is up-closed,
  $d\in D\cap U$.
\end{proof}

We next define Scott open sets in any directed-complete partial order, even though all the dcpos we
consider in this paper will be domains.

\begin{definition}[Scott open set]
\index{Scott-open set|see {open set, Scott}}\index{open set!Scott}
  A subset $U\ss S$ in a dcpo is called \emph{Scott-open} if it is up-closed and satisfies any of
  the equivalent conditions of \cref{prop:Scott_opens}.
\end{definition}

We prove a few properties of domains that may be useful; others can be found in \cite[Section
II-1]{Gierz.Keimel.Lawson.Mislove.Scott:2003a}.

\begin{proposition}
\label{prop:basic_facts_domains}\index{domains!basic facts about}\index{basis!way-up closures as}
  Let $(S,\sqss)$ be a domain.
  \begin{enumerate}
    \item The way-below relation is interpolative: $s_1\ll s_2$ implies $\exists s$ with $s_1\ll
      s\ll s_2$.\index{interpolative!way below as}
    \item For any element $u\in S$, the set $\upclose u$ is Scott open.
    \item A set $U$ is Scott open iff $\up u\ss U$ and $\downclose u\cap U\neq\varnothing$ for each
      $u\in U$.
    \item The Scott-opens constitute a topology on the set $S$, called the \emph{Scott
      topology}.\index{topology!Scott}
    \item The set $\{\,\upclose u\mid u\in S\,\}$ is a basis for the Scott topology.
    \item The order $\sqss$ is the specialization order on points in the Scott topology.
    \item A set $C$ is closed in the Scott topology iff $\down c\ss C$ and $\sup I\in C$ for
      every ideal $I\ss C$.
    \item The closure of a point $s\in S$ is its down closure $\downarrow s$.
  \end{enumerate}
\end{proposition}
\begin{proof}
  Suppose $s_1\ll s_2$. Then one can check that $I\coloneqq\bigcup_{s\ll s_2}\downclose s$ is
  directed, and by \cref{rem:way_below} it is down-closed. For all $s\ll s_2$ we have $\downclose
  s\ss I$, which implies $s=\sup\downclose s\sqss\sup I$. It follows that
  $s_2=\sup\downclose s_2\sqss\sup I$, and this implies $s_1\in I$. This proves 1., and 2.\
  and 3.\ follow directly. Using 3.\ the only nontrivial part of 4.\ is showing that the
  intersection of two open sets is open, but this follows from the fact that $\downclose u$ is
  directed. For 5.\ each $\upclose u$ is open by the interpolative property, and 3.\ implies these
  form a basis. For 6.\ we need to show $u\sqss u'$ iff $(u\in U\imp u'\in U)$ holds for all open
  $U$. It suffices to consider $U=\upclose u_0$, in which case the forward direction is obvious and
  the backwards direction is $u=\sup\downclose u\sqss\sup\downclose u'=u'$. For 7.\ use that
  $C$ is closed iff its complement is up-closed and satisfies \cref{prop:Scott_opens}. For 8.\
  $\down s$ satisfies the properties from 7.\ and is the smallest such set containing $s$.
\end{proof}

\begin{example}\label{ex:Bool^S}\index{domain!powerset}\index{domain!of booleans}
  For any set $X$, the power-set $S\coloneqq\Bool^X$, ordered by inclusion $u\sqss v$ iff $u\ss v$,
  is a domain. One can check that $u\ll v$ iff $u$ corresponds to a finite subset of $v$.
\end{example}

We define a function between domains to be \emph{continuous}\index{continuous!Scott} if it satisfies
any of the  equivalent conditions in \cref{prop:continuous_map}; see \cite[Proposition
II-2.1]{Gierz.Keimel.Lawson.Mislove.Scott:2003a} for proof.

\begin{proposition}\label{prop:continuous_map}
  Let $S,T$ be domains, and let $f\colon S\to T$ be a function between their underlying sets. The
  following are equivalent:
  \begin{itemize}
    \item $f$ is continuous with respect to the Scott topology,
    \item $f$ is order-preserving and $f(\sup(I))=\sup(f(I))$ for all ideals $I$,
    \item $f(s)=\sup\{f(s')\mid s'\ll s\}$ for all $s\in S$.
  \end{itemize}
\end{proposition}

\begin{remark}\label{rem:domains_topspaces}\index{domain!as topological space}
  The inclusion functor from domains to topological spaces is fully faithful, and in fact, domains
  are particularly nice topological spaces. On the one hand, the frame of open sets in a domain can
  be recovered from---in fact is defined in terms of---its poset of points under specialization. On
  the other hand, any domain is a sober space (see \cite[Proposition
  II-1.11(ii)]{Gierz.Keimel.Lawson.Mislove.Scott:2003a}), meaning that its poset of points can be
  recovered from its frame of open sets.\index{topological space!sober}
\end{remark}

\subsection{Predomains and their corresponding posites}
\label{sec:predomains_posites}\index{predomain}

We give a couple technical results which involve material on predomains from \cref{sec:predomains},
and which will be used in the next section. One may choose to skip the proofs on a first reading.

A predomain is like a basis for a domain\index{basis!predomain as}. Indeed, a predomain $B$
determines a domain denoted $\RId(B)$, called the domain of \emph{rounded
ideals},\index{ideal!rounded}\index{rounded ideal|see {ideal}} but $B$ is typically more convenient
to work with because it involves much less data. We now show that for every predomain $B$, its upper
specialization order $(B,\upspecord)$\index{specialization order!on a predomain!upper} has a natural
posite structure, and $(0,1)$-sheaves on this posite can be identified with Scott opens in
$\RId(B)$. Thus, the set $B$ can directly encode the frame of opens of the space $\RId(B)$, and the
topos of sheaves on $\RId(B)$, in addition to its points. Since every domain is naturally a
predomain (see \cref{prop:domains_as_predomains}), the following results also hold when $B$ is a
domain.

\begin{proposition}
\label{prop:predomain_posite}\index{posite!from a predomain}
  Let $(B,\prec)$ be a predomain and $\upspecord$ its upper specialization order. Then there is a
  posite $S_B$ whose underlying poset is $(B,\upspecord)\op$, and whose coverage\index{coverage} is 
  defined by $V\basicCover b$ iff $V=\upclose b$.
\end{proposition}
\begin{proof}
  We only need to prove that $\basicCover$ is a coverage, i.e.\ the two conditions of
  \cref{def:coverage}. For the first condition, if $v\in\upclose b$ then $b\upspecord v$. For the
  second, suppose $b\upspecord b'$. The condition follows from the fact that $\upclose b' \subseteq
  \upclose b$.
\end{proof}

\begin{proposition}\label{prop:predomain_posite_opens}
  Let $(B,\prec)$ be a predomain, $\upspecord$ its upper specialization order, and consider
  the posite $S_B$ as defined in \cref{prop:predomain_posite}. The maps $I\mapsto\{\, b\in B\mid
  \exists b'\in I\ldotp b'\prec b\,\}$ and $U\mapsto\{\,b\in B\mid \upclose b\subseteq U\,\}$ define
  an isomorphism of frames, $\PoShv{S_B} \iso \Opens(B)$.
\end{proposition}
\begin{proof}
  The well-definedness of the maps and the fact that they are mutually inverse can be checked
  directly from \cref{def:poset_ideal,def:predomain,def:predomains_rounded_specs}.
\end{proof}

\begin{proposition}\label{prop:down_closed_up_closed}\index{specialization order!on a predomain!lower}
  Suppose that for a predomain $(B,\prec)$, lower specialization implies upper specialization; i.e.\
  for all $b,b'\in B$, if $b\specord b'$ then $b\upspecord b'$. Then there is a posite $S'_B$ with
  underlying poset $(B,\specord)\op$, and with coverage defined by $V\basicCover b$ iff $V=\upclose
  b$.
\end{proposition}
\begin{proof}
  The proof that $\basicCover$ defines a coverage is essentially the same as in
  \cref{prop:predomain_posite}.
\end{proof}

\begin{remark}
  We show in \cref{cor:S_iso_S'} that, in the setting of \cref{prop:down_closed_up_closed}, there is an isomorphism
  \[
    \PoShv{S_B} \iso \PoShv{S'_B}.
  \]
\end{remark}

\index{domain|)}

\section{The interval domain and its associated topos}
\label{sec:interval_domain}\index{interval domain|(}

In \cref{sec:def_interval_domain} we define one of the central objects of this book: the interval
domain $\IR$. If this is the reader's first encounter with the interval domain, its order may appear
opposite to what is expected. Namely $\sqss$ refers to the specialization order\index{specialization
order!on points} on intervals rather than to the inclusion order.

Next in \cref{sec:IR_predomain} there is a brief section on the relationship between $\IR$, its
subspace $\RR$, and the notion of Dedekind cuts, which will also be a central theme throughout the
book. In \cref{thm:IR_sheaves} we state four equivalent definitions of the topos $\Shv{\IR}$.

\subsection{Definition of the interval domain}\label{sec:def_interval_domain}

\begin{definition}[The interval domain]\label{def:interval_domain}
  The \emph{interval domain} $\IR$ is the set $\{\,(d,u)\in\RR\times\RR\mid d\leq u\,\}$, an element
  of which we denote by $[d,u]$, with order given by $[d,u]\sqss [d',u']$ iff $d\leq d'\leq u'\leq
  u$.%
  \footnote{
    Throughout this book, we often use $d$ or $\delta$ to stand for ``down'' and $u$ or $\upsilon$
    (upsilon) to stand for ``up''.
  }
\end{definition}

Conceptually, the interval domain is the set of bounded closed intervals in
$\RR$,\index{interval!compact} ordered by reverse inclusion. Thus the maximal elements of $\IR$ are
the zero-length intervals, which we can identify with elements of $\RR$. For intuition about the order, regard intervals as \emph{approximations of real numbers}, so that an increase in domain order corresponds
to an increase in precision.

We have the following characterization of the way-below relation (\cref{def:way_below}) in $\IR$.

\begin{proposition}\label{prop:way_below_IR}\index{way-below!in interval domain}
  Every ideal $I\ss\IR$ has a supremum. The way-below relation in $\IR$ is given by
  \[
    [d',u']\ll[d,u]\quad\text{iff}\quad d'<d\text{ and }u<u',
  \]
  and $[d,u]=\sup\downclose[d,u]$. Thus $\IR$ is indeed a domain in the sense of
  \cref{def:domain}.
\end{proposition}
\begin{proof}
  Define $\bar{d} = \sup_{[d,u]\in I} d$ and $\bar{u} = \inf_{[d,u]\in I} u$. We will show that
  $[\bar{d},\bar{u}]$ is the supremum of $I$ in $\IR$. First we must show that $\bar{d}\leq\bar{u}$,
  or equivalently, that for any $[d_1,u_1],[d_2,u_2]\in I$ we have $d_1\leq u_2$. But $I$ is
  directed, so there exists a $[d',u']\in I$ such that $[d_1,u_1]\sqss[d',u']$ and
  $[d_2,u_2]\sqss[d',u']$, and it follows that $d_1\leq d'\leq u'\leq u_2$. Finally, it is clear
  that $[\bar{d},\bar{u}]$ is the least upper bound of $I$.

  Now suppose $d'<d$ and $u<u'$; we want to show $[d',u']\ll[d,u]$. Suppose $I$ is an ideal with
  $\sup I = [\bar{d},\bar{u}]$, such that $[d,u]\sqss[\bar{d},\bar{u}]$. Then $d'<\bar{d}$ and
  $\bar{u}<u'$, which implies there exist $[d_1,u_1],[d_2,u_2]\in I$ with $d'\leq d_1$ and $u_2\leq
  u'$. Then by directedness, there exists a $[d'',u'']\in I$ with $[d_1,u_1]\sqss[d'',u'']$ and
  $[d_2,u_2]\sqss[d'',u'']$. It follows that $d'\leq d_1\leq d''$ and $u''\leq u_2\leq u'$, so
  $[d',u']\sqss[d'',u'']\in I$ shows $[d',u']\in I$.

  Conversely, suppose $[d',u']\ll[d,u]$. Define $I\coloneqq\{[d'',u'']\mid d''<d\leq u<u''\}$. It is
  easy to see that $I$ is an ideal, and that $\sup I = [d,u]$. Hence $[d',u']\in I$, showing that
  $d'<d$ and $u<u'$.
\end{proof}

We often make use of the following notation:\index{closure!down}\index{closure!way-up}
\[
  \down[d,u]=\{\,[d',u']\mid d'\leq d\leq u\leq u'\,\}
  \qquad\text{and}\qquad
  \upclose[d,u]=\{\,[d',u']\mid d<d'\leq u'<u\,\}.
\]

A set $U$ is called \emph{filtered} if it is nonempty and down-directed; it is an open filter if it is additionally open. Here is a proposition that gives a few facts about $\IR$ that will be useful later.\index{filter}
\begin{proposition}\label{prop:facts_on_IR}
\begin{enumerate}
	\item Sets of the form $\upclose[d,u]$, for $d\leq u$ form a basis for the topology on $\IR$.
	\item A subset $U$ is open iff it is up-closed and $\downclose[d,u]\cap U\neq\varnothing$ for every $[d,u]\in U$.\index{up-closed}
	\item $\IR$ has binary meets.
	\item $\IR$ has conditional joins: if $[d_1,u_1]$ and $[d_2,u_2]$ have an upper bound, then they have a join.
	\item For any extended real numbers $r_1,r_2\in\RR\cup\{-\infty,\infty\}$ with $r_1<r_2$, the open set $U_{r_1,r_2}\coloneqq\{[d,u]\in\IR\mid r_1<d\leq u<r_2\}$ is filtered.
	\item Every filtered open set $U$ is of the form $U_{r_1,r_2}$, for some $r_1<r_2$ in $\RR\cup\{-\infty,\infty\}$.
\end{enumerate}
\end{proposition}
\begin{proof}
1 and 2 are shown in \cref{prop:basic_facts_domains}. Let $[d_1,u_1]$ and $[d_2,u_2]$ be arbitrary points. It is easy to check that their meet is $[\min(d_1,d_2),\max(u_1,u_2)$, showing 3. For 4, if there exists $[d,u]$ above each, then $d_1\leq d$ and $d_2\leq d$, so $\max(d_1,d_2)\leq d$, and similarly $\min(u_1,u_2)\geq u$, so their join is $[\max(d_1,d_2),\min(u_1,u_2)$.

For 5, clearly $U=U_{r_1,r_2}$ is nonempty and it is open by condition 2. If $[d_1,u_1]\in U$ and $[d_2,u_2]\in U$ then $[\min(d_1,d_2),\max(u_1,u_2)]\in U$, so $U$ is down-directed.

For 6, let $U$ be an open filter, and let $r_1\coloneqq\inf_{[d,u]\in U}d$ and $r_2\coloneqq\sup_{[d,u]\in U}u$. Clearly $U\ss U_{r_1,r_2}$, so take any $[d,u]$ such that $r_1\leq d\leq u<r_2$; we must show $[d,u]\in U$. We know there exists $[d_1,u_1]$ and $[d_2,u_2]$ in $U$ such that $d_1\leq d$ and $u\leq u_2$. But then their meet is in $U$, and it contains $[d_1,u_2]$ and hence $[d,u]$.
\end{proof}

\subsection{Discussion of $\IR$, $\RR$, and Dedekind cuts}\label{sec:IR_predomain}

The logic and type theory we present in this book is constructive. As the interval domain plays an
important role within the type theory, not just in defining its semantics, we make a brief
digression to discuss how to define $\IR$ constructively, and its relation to standard constructive
definitions of the real numbers.

Each element of $\IR$ can be identified with a pair of subsets $\delta,\upsilon\subseteq\QQ$
satisfying the following conditions:
\begin{enumerate}
  \item $\delta$ and $\upsilon$ are inhabited,
  \item $d\in\delta \iff \exists(d'\in\delta)\ldotp d<d'$,
  \item $u\in\upsilon\iff \exists(u'\in\upsilon)\ldotp u'<u$, and
  \item $\delta\cap\upsilon = \varnothing$.
\end{enumerate}
The element $[d,u]\in\IR$ corresponds to $(\delta,\upsilon)$ where $q\in\delta$ iff $q<d$, and
$q\in\upsilon$ iff $u<q$.\index{cut}

When the pair of subsets $(\delta,\upsilon)$ satisfies the above four conditions, we call it a \emph{disjoint pair of cuts}, in the sense of Dedekind cuts. This is part of a larger story, including the domain structure for this formulation of $\IR$, which is fleshed out in \cref{prop:intervals_as_RIds} and \cref{sec:predomains}; see in particular \cref{ex:RId_predomain}.

Classically, it is easy to show that maximal elements of $\IR$ correspond precisely to real numbers.
However, constructively, maximality is slightly too weak to capture the Dedekind reals. Indeed, seeing each element of $\IR$ as a disjoint pair of cuts, it is a simple exercise to show that 
$(\delta,\upsilon)\in\IR$ is maximal iff it is a \emph{MacNeille real}\index{real number!MacNeille}, i.e.\ if it satisfies the following axioms for any $d,u:\QQ$
\begin{align*}
  &[(d<u) \wedge \neg(u\in\upsilon)] \imp d\in\delta\\
  &[(d<u) \wedge \neg(d\in\delta)] \imp u\in\upsilon.
\end{align*}
For most purposes, the \emph{Dedekind reals} $\RR$ are constructively better behaved\index{constructive logic|see {logic, constructive}}\index{logic!constructive}. Rather than maximal, a Dedekind real is a disjoint pair of cuts $(\delta,\upsilon)$ that is also \emph{located}, i.e.\ such that\index{located}
\[
  d<u \imp (d\in\delta \vee u\in\upsilon).
\]
Every Dedekind real is a MacNeille real; the converse holds classically but not constructively. Throughout this book, we will be focused on Dedekind reals.

\subsection{Five equivalent definitions of the topos $\Shv{\IR}$}

We first define two different posites,\index{posite} though we will see in \cref{thm:IR_sheaves}
that they define the same topos.\index{topos}\index{coverage!of interval domain}

\begin{definition}[The interval posite $S_{\IR}$]
\label{def:interval_posite}\index{posite!interval}\index{posite!rational interval}
  The \emph{interval posite}, which we denote $S_\IR$, is what we called $S'_{\IR}$ in
  \cref{prop:down_closed_up_closed}. Explicitly:
  \begin{itemize}
    \item The underlying set is $\IR=\{\,[d,u]\in\RR\times\RR\mid d\leq u\,\}$ as in \cref{def:interval_domain},
    \item the order is given by inclusion, $[d,u]\leq[d',u']$ iff $d'\leq d\leq u\leq u'$, and
    \item the coverage is given by $V\basicCover [d,u]$ iff $V=\{\,[d',u']\in\IR \mid d<d'\leq
      u'<u\,\}$.\index{coverage}
  \end{itemize}
  Note that if $d=u$ then $[d,u]$ has an empty covering family.
  \smallskip

  The \emph{rational interval posite}, denoted $S^\QQ_{\IR}$, is the posite $S_{\IRpre}$ as defined
  in \cref{prop:predomain_posite}. Explicitly:
  \begin{itemize}
    \item The underlying set is $\IRpre=\{\,(d,u)\in\QQ\times\QQ \mid d<u\,\}$,
    \item the order is given by inclusion, $(d,u)\leq(d',u')$ iff $d'\leq d<u\leq u'$; and
    \item the coverage is given by $ V\basicCover(d,u)$ iff $V=\{\,(d',u')\in\IRpre \mid
      d<d'<u'<u\,\}$.
  \end{itemize}
\end{definition}

\begin{remark}
Because we have $\RR\subseteq \IR$, the set $\RR$ of Dedekind reals inherits a subspace topology. There are posites\index{posite} $S_\RR$ and $S^\QQ_\RR$ whose frame of $(0,1)$-sheaves\index{sheaf!(0,1)} is the frame of opens in $\RR$. These posites have the same underlying posets as $S_\IR$ and $S^\QQ_\IR$ (respectively), but more covering families. We explain this connection for the reader's intuition; we will not need it for our work, though one will see reincarnations of it, e.g.\ in \cref{cor:covering_pi}.

Recall from \cref{def:interval_posite} that the posite $S_{\IR}$ for $\IR$ has the coverage $V\basicCover[d,u]$ iff $V=\{\,[d',u']\mid d<d'\leq u'<u\,\}$. We informally called these ``continuity'' coverings in \cref{eqn:behavior_continuity}.
The posite $S_{\RR}$ for $\RR$ has an expanded coverage: $V\basicCover[d,u]$ iff either
\begin{itemize}
	\item $V=\{\,[d',u']\mid d<d'\leq u'<u\,\}$, or
	\item $V=\{\,[d,u'],[d',u]\,\}$ where $d\leq d'\leq u'\leq u$.
\end{itemize}
In other words it includes the ``composition'' coverings of \cref{eqn:behavior_composition}. The coverage for the posite $S^\QQ_\RR$ is similarly analogous to that of $S^\QQ_\IR$.
\end{remark}

We can now give several equivalent constructions for the same topos. \Cref{thm:IR_sheaves} will be
proven in \cref{sec:proof_IR_sheaves}
\begin{theorem}\label{thm:IR_sheaves}
  The following categories are equivalent:
  \begin{enumerate}
    \item The category of sheaves $\Shv{\Opens(\IR)}$ on the frame of opens of $\IR$, with the
      canonical coverage $\{V_i\}_{i\in I}\basicCover U$ iff $\bigcup_{i\in I} V_i=U$.
    \item The category of sheaves $\Shv{S_{\IR}}$ on the interval posite $S_{\IR}$.
		\item The category of sheaves $\Shv{S^\QQ_{\IR}}$ on the rational interval posite $S^\QQ_{\IR}$.
    \item The category of $\Cont{\IR}$ of \emph{continuous functors} $\IR\to\Cat{Set}$, i.e.\ functors taking directed
      suprema in $\IR$ to colimits in $\Cat{Set}$.
  \end{enumerate}
\end{theorem}

\section{\texorpdfstring{$\IR$}{\unichar{"1D540}\unichar{"211D}} and the upper half-plane}\label{sec:IR_half_plane}\index{half-plane|(}

Let $H\coloneqq \{(x,y)\in\RR^2\mid y\geq 0\}$ denote the upper half-plane. It is common to consider the bijection $h\colon|\IR|\To\cong H$ from the set of points in $\IR$ to $H$, sending each interval to the pair consisting of its midpoint and radius. For the sake of mnemonic, say $(m,r)\coloneqq h([d,u])$, and call $(m,r)$ the \emph{half-plane representation}\index{interval domain!half-plane representation of} of $[d,u]$. Then the bijection is given by
\[m=\frac{u+d}{2},\qquad r=\frac{u-d}{2}\quad\qquad\text{and}\quad\qquad d=m-r,\quad u=m+r\]
Clearly the subset $\RR\ss\IR$ corresponds to the horizontal axis of $H$. Given a subset $S\ss|\IR|$, we also write $h(S)$ to mean the image of $S$ under $h$.

We can transport the specialization order and the way-below relation on $\IR$ across this bijection:
\begin{equation}\label{eqn:half_plane_order}
\begin{aligned}
	[d_1,u_1]&\sqss[d_2,u_2]&\text{iff}&&d_1\leq d_2&\leq u_2\leq u_1&\text{iff}&&|m_1-m_2|&\leq r_1-r_2\\
	[d_1,u_1]&\ll[d_2,u_2]&\text{iff}&&d_1< d_2&\leq u_2< u_1&\text{iff}&&|m_1-m_2|&< r_1-r_2
\end{aligned}
\end{equation}
We denote the order on the half-plane just as we do for points in the interval domain, writing $(m_1,r_1)\sqss(m_2,r_2)$ and $(m_1,r_1)\ll(m_2,r_2)$ in the above cases. The \emph{upper cone} of any $(m,r)\in H$ is denoted $\up(m,r)\coloneqq\{(m',r')\mid (m,r)\sqss(m',r')\}$. Similarly, the \emph{way-up closure}\index{closure!way-up} of $(m,r)$ is denoted $\upclose(m,r)\coloneqq\{(m',r')\mid (m,r)\ll(m',r')\}$. Both consist of points in the cone visually below $(m,r)$, as shown here:
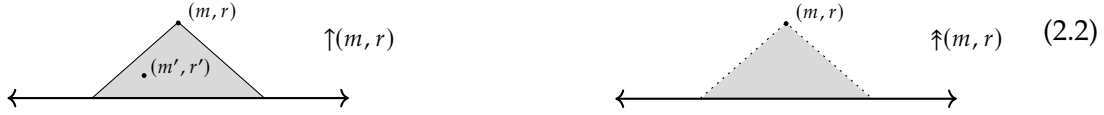
\begin{equation}\label{eqn:up_way_up_cone}
\begin{tikzpicture}[baseline=(name)]
	\mypic{1}{2}{}{}
	\node [circle, inner sep=.5pt, draw, fill=black] at (M) {};
	\node [circle, inner sep=.5pt, draw, fill=black] at ($(L)!.3!(R)+(0,.3)$) (M') {};
	\node [above right = -7pt and -4pt of M, font=\tiny] {$(m,r)$};
	\node [above right = -5pt and -2pt of M', font=\tiny] {$(m',r')$};
	\node [above=10pt of rend, font=\footnotesize] (name) {$\up(m,r)$};
\end{tikzpicture}
\hspace{1in}
\begin{tikzpicture}[baseline=(name)]
	\mypic{1}{2}{dotted}{}
	\node [circle, inner sep=.5pt, draw, fill=black] at (M) {};
	\node [above right = -7pt and -4pt of M, font=\tiny] {$(m,r)$};
	\node [above=10pt of rend, font=\footnotesize] (name) {$\upclose(m,r)$};
\end{tikzpicture}
\end{equation}

\begin{warning}\label{warn:distinct_topologies}
Despite their underlying sets being isomorphic, we consider the interval domain $\IR$ and the half-plane $H$ to have distinct topologies: $\IR$ has the Scott topology and $H$ has the Euclidean topology. For every Scott-open set $U\ss\IR$, its image $h(U)$ is open in the Euclidean topology---giving us \cref{prop:H_to_IR_cont}---but the converse is not true.\index{topology!Scott}\index{topology!Euclidean} Thus if $h([d,u])=(m,r)$ then the sets $\upclose[d,u]$ and $\upclose(m,r)$ are isomorphic as sets, but they are not isomorphic as topological spaces.
\end{warning}

\begin{proposition}\label{prop:H_to_IR_cont}
Both $h$ and $h^{-1}$ are poset isomorphisms, and $h^{-1}\colon H\to\IR$ is continuous.
\end{proposition}
\begin{proof}
Since we defined the order on $H$ in terms of that on $\IR$ and the bijection $|\IR|\cong H$, the first claim is true by definition. For the second, note that the Scott topology has a basis $\{\upclose x\}_{x\in\IR}$ by \cref{prop:facts_on_IR}. Clearly the cone $\upclose(m,r)=\{\,(m',r')\mid |m-m'|<r-r'\,\}$ is open in the upper half-plane.
\end{proof}

\subsection{Open sets in $\IR$ as Lipschitz functions}\index{Lipschitz|(}

Recall that for any two metric spaces $(A,d_A)$ and $(B,d_B)$ and number $M\geq0$, an
\emph{$M$-Lipschitz} function is a function $f\colon A\to B$ such that for all $a_1,a_2\in A$, the
relation
\[
  d_B(f(a_1),f(a_2))\leq M*d_A(a_1,a_2)
\]
holds. The same definition makes sense for extended metrics (where we allow $d(x_1,x_2)=\infty$). We
will be interested in $1$-Lipschitz functions, which we just call \emph{Lipschitz}.

It is a standard fact that $M$-Lipschitz functions are absolutely continuous and hence
differentiable almost everywhere (non-differentiable on a set with Lebesgue measure 0), and that their
derivative is bounded between $-M$ and $M$.

Consider $\RR$ as a metric space in the usual way, i.e.\ $d(x_1,x_2)\coloneqq \abs{x_1-x_2}$, and
consider the space $\RRubhalf\coloneqq\RR_{\geq0}\cup\{\infty\}$ with the metric
\[
  d(y_1,y_2)\coloneqq
  \begin{cases}
    \abs{y_1-y_2}&\tn{if }y_1\in\RR_{\geq0}\text{ and }y_2\in\RR_{\geq0}\\
    \infty   &\tn{if }y_1\in\RR_{\geq0}\text{ and }y_2=\infty\\
    \infty   &\tn{if }y_1=\infty       \text{ and }y_2\in\RR_{\geq0}\\
    0        &\tn{if }y_1=\infty       \text{ and }y_2=\infty
  \end{cases}
\]
For any space $X$, the set of 1-Lipschitz functions $X\to\RRubhalf$ has the structure of a poset:
say that $f\leq g$ if $f(x)\leq g(x)$ for all $x$.

Our next goal is to show that the Scott open sets of $\IR$ are in one-to-one correspondence with the Lipschitz functions $\RR\to\RR^\infty_{\geq0}$; see \cref{thm:opens_Lipschitz}.

\begin{lemma}\label{lemma:open_Lipschitz}
  For any Scott open set $U\in\Omega_\IR$, the function $f_U\colon\RR\to\RRubhalf$, defined by
  \[
    f_U(m) \coloneqq \sup \{\,r\in\RRubhalf \mid (m,r) \in U\,\}
  \]
  is a 1-Lipschitz function and preserves order. Moreover, the following are equivalent:
  \begin{itemize}
    \item $U=\IR$,
    \item $f_U(m)=\infty$ for all $m\in\RR$,
    \item $f_U(m_0)=\infty$ for some $m_0\in\RR$.
  \end{itemize}
  And also the following are equivalent for all $m\in\RR$ and $r\in\RR_{\geq 0}$:
  \begin{itemize}
    \item $f_U(m)=r$
    \item $\upclose(m,r)\ss U$ and $(m,r)\notin U$.
  \end{itemize}
\end{lemma}
\begin{proof}
  Clearly $U\ss V$ implies $f_U(m)\leq f_V(m)$ for all $m$, so the assignment preserves order. We
  first prove the equivalences, and then show that $f=f_U$ is a 1-Lipschitz function at the end.

  For the first three, clearly the first implies the second, and the second implies the third.
  Suppose $f(m_0)=\infty$ for some $m_0\in\RR$. Then $U=\IR$ because for any $(m,r)$ we must have
  $(m_0,r+\abs{m-m_0})\in U$, which implies $(m,r)\in U$ since $U$ is up-closed.

  For the second pair, note that $\upclose(m,r)=\{(m',r')\mid \abs{m-m'}<r-r'\}$. It is easy to see
  that $\upclose(m,r)\ss U$ iff $r\leq f(m)$. By definition, Scott open sets are rounded, which
  implies $(m,r)\in U$ iff there is some $r'$ with $r<r'$ and $(m,r')\in U$. From this it is easy to
  see that $(m,r)\in U$ iff $r<f(m)$.

  It remains to check that $f$ is a 1-Lipschitz function. If $f(m)=\infty$ for any $m$ then $f$ is
  constant, and hence a 1-Lipschitz function. So suppose $f(m)\in\RR_{\geq 0}$ for all $m$; we want
  to show that $|f(m')-f(m)|\leq|m'-m|$ for all $m,m'\in\RR$. We may assume $m<m'$; we will show
  $f(m')-f(m)\leq m'-m$ and $f(m)-f(m')\leq m'-m$. For contradiction, suppose $m'-m+f(m)<f(m')$.
  Then $(m',m'-m+f(m))\in U$, which implies $(m,f(m))\in U$, which is false as shown above. The
  other case is similar.
\end{proof}

\begin{lemma}\label{lemma:Lipschitz_open}
  For any 1-Lipschitz function $f\colon\RR\to\RRubhalf$, the set $U_f\ss\IR$ defined by
  \[
    U_f\coloneqq\{(m,r)\in\IR\mid r<f(m)\}
  \]
  is Scott-open.
\end{lemma}
\begin{proof}
  To show that $U=U_f$ is Scott open, we must show that $(m,r)\in U$ iff there exists $(m',r')\in U$
  such that $\abs{m'-m}<r'-r$. The forward direction is easy: if $r<f(m)$ then there exists $r'$
  such that $r<r'<f(m)$, and the conclusion follows by letting $m'\coloneqq m$. For the backward
  direction, if $(m',r')$ satisfies $r'<f(m')$ and $|m'-m|\leq r'-r$, then
  $f(m')-f(m)\leq\abs{f(m')-f(m)}\leq\abs{m'-m}\leq r'-r$, so $r<f(m)$ and hence $(m,r)\in U$.
\end{proof}

\begin{theorem}\label{thm:opens_Lipschitz}
The assignments $U\mapsto f_U$ from \cref{lemma:open_Lipschitz} and $f\mapsto U_f$ from \cref{lemma:Lipschitz_open} are mutually inverse, giving an isomorphism between the poset $\Omega_\IR$ of open sets in the interval domain and the poset of 1-Lipschitz functions $\RR\to\RRubhalf$. 
\end{theorem}
\begin{proof}
It is straightforward to check that these two constructions are mutually inverse.
\end{proof}

Subsets of the form $\upclose[d,u]$ form a basis of $\IR$, by \cref{prop:basic_facts_domains}, so it is useful to consider its poset $\Omega(\upclose[d,u])$ of open subsets.

\begin{definition}[Dyck path]\label{def:dyck}\index{Dyck path}
  For $d\leq u$, let $[d,u]\ss\RR$ denote the closed interval in the usual topology. A
  \emph{(continuous) Dyck path on $[d,u]$} is a function $D\colon[d,u]\to\RR_{\geq 0}$ such that
  \begin{itemize}
    \item $D(d)=0=D(u)$, and
    \item $|D(x')-D(x)|\leq |x'-x|$ for all $x,x'\in[d,u]$.
  \end{itemize}
  The second condition is equivalent to saying that $D$ is a 1-Lipschitz function. The order is
  given by $D\leq D'$ iff $D(x)\leq D'(x)$ for all $x\in[d,u]$.
\end{definition}

\begin{corollary}\label{cor:opens_Dyck}
For any $d\leq u$, there is an isomorphism of posets between $\Omega(\upclose[d,u])$ and the poset of Dyck paths on $[d,u]$.
\end{corollary}
\begin{proof}
By \cref{thm:opens_Lipschitz}, $\{U\in\Omega_\IR\mid U\ss\upclose[d,u]\}$ is isomorphic to the set
\[\{f\colon\RR\to\RRubhalf\mid f\leq f_{\upclose[d,u]}\text{ and }f\text{ is a 1-Lipschitz function}\}\]
where $f_{\upclose[d,u]}$ is the Lipschitz function associated to the open set ${\upclose[d,u]}$ as in \cref{lemma:open_Lipschitz}. It is easy to see that $f_{\upclose[d,u]}$ is supported on the closed interval $[d,u]$, and in fact that it is the largest 1-Lipschitz function supported on this interval. Thus the above set is indeed isomorphic to the set of length-$\ell$ Dyck paths.
\end{proof}

Here are pictures of two Dyck paths on the interval $[d,u]$:
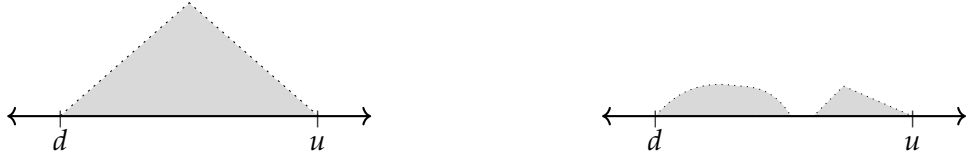
\[
\begin{tikzpicture}[x=1cm, y=1cm, text height=1.5ex, text depth=.25ex, baseline=(d)]
	\mypic{.5}{3}{dotted}{\tiny|};
	\node at ($(L)+(0,-2ex)$) (d) {$d$};
	\node at ($(R)+(0,-2ex)$) (u) {$u$};
\end{tikzpicture}
\hspace{1in}
\begin{tikzpicture}[x=1cm, y=1cm, text height=1.5ex, text depth=.25ex, baseline=(d)]
	\tikzmath{
		real \margin, \length;
		let \margin = .5;
		let \length = 3;
	}
	\node (lend) {};
	\node [right=\margin of lend] (L) {\tiny|};
	\node [right=\length of L] (R) {\tiny|};
	\node [right=\margin of R] (rend) {};
	\filldraw[fill=black!15, draw=black, dotted] 
		(L.center) to[bend left] +(1.1,.4) to[bend left] +(.7,-.4) -- +(.3,0) -- +(.7,.4) -- (R.center) --cycle;
	\draw[thick, <->] (lend) -- (rend);
	\node at ($(L)+(0,-2ex)$) (d) {$d$};
	\node at ($(R)+(0,-2ex)$) (u) {$u$};
\end{tikzpicture}
\]

\begin{remark}\index{Dyck path!discrete}
A \emph{discrete Dyck path} is a path in the quadrant $\NN\times\NN$ from $(0,0)$ to $(n,n)$, for some $n\in\NN$, such that every edge has length 1 and is oriented either upward or righward. Discrete Dyck paths are discussed in combinatorics literature because the number of discrete Dyck paths for varying $n$ is the $n$th Catalan number. Just as continuous Dyck paths represent the subobject classifier for sheaves on $\IR$, it is similarly the case that discrete Dyck paths represent the subobject classifier for presheaves on the category $\Tw{\ZZ,\leq}$ of intervals in $\ZZ$. Discrete Dyck paths have also been discussed in the context of temporal logic; see \cite{Ferrari:2016a}.
\end{remark}
\index{Lipschitz|)}

\subsection{Real-valued functions on $\IR$}\label{sec:real_valued_IR}

Recall that the underlying sets of the upper half-plane $H$ and the interval domain $\IR$ are in bijection, but we endow $H$ with the Euclidean topology and $\IR$ with the Scott topology; see \cref{warn:distinct_topologies}. In this section we use the function $h^{-1}\colon H\to\IR$ from \cref{prop:H_to_IR_cont}, which sends $(m,r)\mapsto[m-r,m+r]$, to better understand continuous functions from $\IR$ to various sorts of real-number objects.

Let $\LR$ denote the space of \emph{lower real numbers}\index{real number!lower}. Its underlying set is $\RR\cup\{\infty\}$ and a subset $U\ss\LR$ is open if either $U=\LR$ or there exists some $r\in\LR$ such that $U=\{r'\in\LR\mid r<r'\}$; in particular, $\{\infty\}$ is not open. In fact $\LR$ forms a domain; (see \cref{ex:RId_predomain}) the above topology coincides with the Scott topology, the usual $\leq$ coincides with the domain order $\sqss$, and the usual $<$ coincides with the way-below relation $\ll$.\index{inequality}\index{way-below}

For topological spaces $X,Y$, let $\Top(X,Y)$ denote the set of continuous maps $X\to Y$.

\begin{proposition}\label{prop:H_to_LR}
Let $H$ denote the upper half-plane, endowed with the $\sqss$ partial order from \cref{eqn:half_plane_order}. The function $H\to\IR$ induces an injection $\Top(\IR,\LR)\to\Top(H,\LR)$, whose image is precisely $\{\,f\in\Top(H,\LR)\mid f\text{ is monotonic}\,\}$.
\end{proposition}
\begin{proof}
Since $H\to\IR$ is continuous and bijective on points, composing with it induces the injection above. Certainly anything in the image is monotonic, because Scott-continuous functions are. So we need to show that if $f\colon H\to\LR$ is continuous and monotonic then it is Scott-continuous. By \cref{prop:continuous_map}, a function $f\colon\IR\to\LR$ is Scott-continuous iff $f(x)=\sup\{f(y)\mid y\ll x\}$ for all $x\in\IR$. 

Suppose $f$ is continuous and monotonic and take $x\in\IR$. If $f(x)<\infty$, then we want to show that for any $\epsilon>0$ there exists $y\ll x$ with $f(x)-\epsilon<y$. By monotonicity, this is equivalent to showing there exists $y\ll x$ with $|y-f(x)|<\epsilon$. But the $\epsilon$-ball around $f(x)$ is open, so its preimage $U$ is an open neighborhood of $x$ in $H$, and any open neighborhood of $x$ in $H$ intersects the open cone $\upclose x$ (see \cref{eqn:up_way_up_cone}), and we are done. If instead $f(x)=\infty$, we want to show that for any $N>0$ there exists $y\ll x$ with $N<y$; the proof is similar.
\end{proof}

Let $\UR$ denote the space of \emph{upper real numbers}\index{real number!upper}. Its underlying set is $\RR\cup\{-\infty\}$ and a set $U\ss\UR$ is open if either $U=\UR$ or there exists some $r\in\LR$ such that $U=\{r'\in\LR\mid r'<r\}$; in particular, $\{-\infty\}$ is not open. In fact $\UR$ forms a domain; (see \cref{ex:RId_predomain}) the above topology coincides with the Scott topology, the usual $\geq$ coincides with the domain order $\sqss$, and the usual $>$ coincides with the way-below relation $\ll$.

The proof of the following proposition is analogous to that of \cref{prop:H_to_LR}. Note that $f\colon H\to\UR$ is monotonic iff $x\sqss x'$ implies $f(x)>f(x')$, in keeping with the order on $\UR$.

\begin{proposition}
Let $H$ denote the upper half-plane, endowed with the above $\sqss$ partial order. The function $H\to\IR$ induces an injection $\Top(\IR,\UR)\to\Top(H,\UR)$, whose image is precisely $\{\,f\in\Top(H,\UR)\mid \forall x, x'\in H\ldotp x\leq x'\imp f(x)>f(x')\,\}$. 
\end{proposition}

For any space $X$ and function $f\colon X\to\RR\cup\{\infty\}$, one says that $f$ is \emph{lower semi-continuous}\index{function!lower semi-continuous}\index{function!upper semi-continuous} (respectively, $f\colon X\to\RR\cup\{-\infty\}$ is \emph{upper semi-continuous}) if it is continuous as a map to the lower reals $\LR\cong\RR\cup\{\infty\}$ (respectively, to the upper reals $\UR\cong\RR\cup\{-\infty\}$). A more intuitive way to think about a lower semi-continuous function $f$ is that for every point $x_0\in X$ and number $\epsilon>0$, there is a basic open neighborhood $U$ of $x_0$ such that $-\epsilon<f(x)-f(x_0)<\infty$ for all $x\in U$, i.e.\ $f$ may increase discontinuously, but it decreases continuously.

Let $\II$ denote the space of \emph{extended intervals}\index{interval!extended}, given by the product $\II=\LR\times\UR$. Then $\IR$ and $\RR$ can be identified with the following subsets of $\II$:
\[
\IR\cong\{(x,y)\in\II\mid x\leq y\}
\qquad\text{and}\qquad
\RR\cong\{(x,y)\in\II\mid x=y\}
\]

We have proved the following.

\begin{corollary}\label{cor:sheaf_IR_Dedekind}
  Choose $[d,u]\in\IR$ and let $(m,r)\in H$ be the corresponding point in $H$. The Scott-open
  $\upclose[d,u]\subseteq\IR$ corresponds to the subspace $\upclose(m,r)\subseteq H$ (which we
  consider with the Euclidean topology). Then we have isomorphisms\index{topology!Scott}\index{topology!Euclidean}
  \begin{align*}
    &\Top(\upclose[d,u],\LR) \cong \{\,f\colon \upclose(m,r)\to\RR\cup\{\infty\} \mid
      f\text{ is lower semi-continuous and order-preserving}\,\} \\
    &\Top(\upclose[d,u],\UR) \cong \{\,f\colon \upclose(m,r)\to\RR\cup\{-\infty\} \mid
      f\text{ is upper semi-continuous and order-reversing}\,\} \\
    &\Top(\upclose[d,u],\II) \cong \Top(\upclose[d,u],\LR) \times \Top(\upclose[d,u],\UR) \\
    &\Top(\upclose[d,u],\IR) \cong \{\,(\ubar{f},\bar{f})\in\Top(\upclose[d,u],\II) \mid
      \forall x\in \upclose(m,r)\ldotp \ubar{f}(x) \leq \bar{f}(x)\,\} \\
    &\Top(\upclose[d,u],\RR) \cong \{\,(\ubar{f},\bar{f})\in\Top(\upclose[d,u],\II) \mid
      \forall x\in \upclose(m,r)\ldotp \ubar{f}(x) = \bar{f}(x)\,\} \cong \RR.
  \end{align*}
\end{corollary}

In the last case, we have $\Top(\upclose[d,u],\RR)\iso\RR$ because $\ubar{f}=\bar{f}$ is both
order-preserving and order-reversing, and therefore constant since for any $x,y\in \upclose(m,r)$,
$\bar{f}(x)=\bar{f}(x\wedge y)=\bar{f}(y)$.

\index{half-plane|)}
\index{interval domain|)}

\section{Grothendieck posites}\label{sec:proof_IR_sheaves}

In this section, specifically in \cref{thm:IR_sheaves_proven}, we give the proof of
\cref{thm:IR_sheaves}. This material is not necessary for the rest of the book, so readers can feel
free to skip to \cref{chap:translation_invariance}.

\subsection{Basic theory of Grothendieck posites and dense morphisms}

\begin{definition}[Grothendieck posite]\label{def:Grothendieck_posite}
  Let $(S,\leq)$ be a poset. A \emph{Grothendieck coverage}\index{coverage!Grothendieck} on $S$ is a relation
  ${\basicCover}\subseteq P(S)\times S$ such that:
  \begin{itemize}
    \item If $V\basicCover u$ then $V\subseteq\specdownclose u$.
    \item If $\specdownclose V\basicCover u$, then $V\basicCover u$.
    \item If $V\basicCover u$ and $u'\leq u$, then $(\specdownclose V\cap\specdownclose u')
      \basicCover u'$.
    \item $\{u\}\basicCover u$ for all $u\in S$.
    \item If $W\basicCover u$ and $V\subseteq\specdownclose u$ such that for all $w\in W$,
      $(V\cap\specdownclose w)\basicCover w$, then $V\basicCover u$.
  \end{itemize}
  A \emph{Grothendieck posite}\index{posite!Grothendieck} is a poset equipped with a Grothendieck
  coverage. A Grothendieck coverage is in particular a coverage---in the sense of
  \cref{def:coverage}---so a Grothendieck posite has an underlying posite. A $(0,1)$-sheaf on a
  Grothendieck posite is simply a $(0,1)$-sheaf on the underlying posite; see \cref{def:poset_ideal}.
\end{definition}

\begin{remark}\label{rmk:Grothendieck_properties}
  It is easily shown that a Grothendieck posite additionally satisfies:
  \begin{itemize}
    \item If $V\basicCover u$ and $V\subseteq V'\subseteq \specdownclose u$, then $V'\basicCover u$.
    \item If $V\basicCover u$ and $V'\basicCover u$, then $(\specdownclose V \cap \specdownclose
      V')\basicCover u$.
  \end{itemize}
Furthermore, if $S$ is a poset, and we have a family of Grothendieck coverages $\{\basicCover_i\}_{i\in I}$
indexed by an arbitrary set $I$, then one checks directly from \cref{def:Grothendieck_posite} that the intersection $\basicCover$, defined by
$V\basicCover u$ iff $V\basicCover_i u$ for all $i\in I$, is also a Grothendieck coverage.
\end{remark}

\begin{lemma}\label{lem:poset_right_Kan}
  Let $S$ and $T$ be posets, and let $D(S)$
  denote the set of down-closed subsets of $S$, and similarly for $D(T)$. Any order-preserving map $F\colon S\to T$ induces an adjunction
  $F_* \dashv F^!$, where $F_*\colon D(T)\to D(S)$ and $F^!\colon D(S)\to D(T)$ are defined 
  \begin{equation}\label{eqn:F_upper_lower}
    F_*(J)=F^{-1}(J) \coloneqq \{ u\in S \mid Fu\in J \}
    \qquad\text{and}\qquad
    F^!(I) \coloneqq \{v\in T \mid F^{-1}(\specdownclose v) \subseteq I\}.
  \end{equation}
\end{lemma}
\begin{proof}
  This is easily checked directly. Alternatively, note that if we consider $S$ and $T$ as
  categories, then $D(S)$ and $D(T)$ are the functor categories to the free-arrow category $2$, $F_*$ is
  precomposition by $F$, and $F^!$ is right Kan extension along $F$.
\end{proof}

Recall from \cref{def:poset_ideal} the sheaf condition for down-closed subsets $I$ of a posite: If $V\basicCover u$ and $V\subseteq I$, then $u\in I$.\index{sheaf condition}

\begin{lemma}\label{lem:posite_sheafification}\index{sheafification!$(0,1)$}
  For any Grothendieck posite $S$, the inclusion $\PoShv{S}\hookrightarrow D(S)$ has a left adjoint
  $a\colon D(S)\to\PoShv{S}$, defined $a(I) \coloneqq \{u\in S \mid (I\cap\specdownclose u)\basicCover u\}$.
\end{lemma}
\begin{proof}
  First we check that $a(I)$ is a $(0,1)$-sheaf. For down-closure, suppose $u'\leq u\in a(I)$. Then $(I\cap\specdownclose
  u)\basicCover u$, so $(I\cap\specdownclose u') = (I\cap\specdownclose u\cap\specdownclose
  u')\basicCover u'$ by \cref{def:Grothendieck_posite}, and hence $u'\in a(I)$. For the sheaf condition, if $W\basicCover u$ and $W\subseteq a(I)$, then for all $w\in W$ we have $\down w\ss\down u$ and $(I\cap\specdownclose w)\basicCover w$, so $(I\cap\specdownclose u\cap\specdownclose w) =
  (I\cap\specdownclose w)\basicCover w$. Hence $(I\cap\specdownclose u)\basicCover
  u$ and $u\in a(I)$.

  Finally, it is simple to check that $I\subseteq a(I)$ for any $I\in D(S)$, and that $a(J)=J$ for
  any $J\in\PoShv{S}$, showing that $a$ is left adjoint to the inclusion.
\end{proof}

\begin{remark}\label{rem:PoShv_complete}
  It follows from \cref{lem:posite_sheafification} that the inclusion $\PoShv{S}\hookrightarrow
  D(S)$ preserves all existing meets, and indeed it is simple to check that $\PoShv{S}$ is closed
  under arbitrary intersections, so that $\PoShv{S}$ is a complete lattice. The join of a family $X\subseteq\PoShv{S}$ can be constructed $\sup X = \bigcap\{I\in\PoShv{S}\mid \forall x\in X\ldotp x\ss I\}$. It also follows that this join can be constructed $\sup X = a(\bigcup X)$. Recall from \cref{ex:canonical_coverage_frame}, that the canonical
  coverage on $\PoShv{S}$ is given by $X\basicCover J\iff\sup X = J$.
\end{remark}

\begin{definition}
  Let $(S,\basicCover)$ be a posite. We define a Grothendieck coverage $\widehat{\basicCover}$ on $S$,
  called the \emph{Grothendieck completion} of $\basicCover$, by taking $\widehat{\basicCover}$ to be
  the intersection of the set of Grothendieck coverages containing $\basicCover$.
\end{definition}

The next proposition shows that taking the Grothendieck completion of a posite doesn't change the
$(0,1)$-sheaves.

\begin{proposition}\label{prop:completion_PoShv_iso}
  Let $(S,\basicCover)$ be a posite. Let $\overline{\basicCover}$ be any coverage on $S$ containing
  $\basicCover$ and contained in the Grothendieck completion $\widehat{\basicCover}$, i.e.\ satisfying
  \begin{itemize}
    \item $V\basicCover u$ implies $V\mathop{\overline{\basicCover}} u$, and
    \item for any Grothendieck coverage $\widetilde{\basicCover}$ containing $\basicCover$,
      if $V\mathop{\overline{\basicCover}}u$ then $V\mathop{\widetilde{\basicCover}}u$.
  \end{itemize}
  Let $I\subseteq S$ be any down-closed subset. Then $I$ satisfies the sheaf condition for $\basicCover$ iff
  it satisfies the sheaf condition for $\overline{\basicCover}$.
  
  Hence if $\hat{S}$ is the posite with
  the same underlying poset as $S$, but with coverage $\widehat{\basicCover}$, then
  $\PoShv{S}\iso\PoShv{\hat{S}}$. Similarly, $\Shv{S}\iso\Shv{\hat{S}}$.
\end{proposition}
\begin{proof}
  Fix a down-closed $I\ss S$. Because $\overline{\basicCover}$ contains $\basicCover$, it is clear that if $I$ satisfies the sheaf
  condition for $\overline{\basicCover}$, then it does for $\basicCover$ as well. For the converse,
  suppose $I$ satisfies the sheaf condition for $\basicCover$. We define a coverage $\widetilde{\basicCover}$ by
  \[
    V\mathop{\widetilde{\basicCover}}u \iff (V\ss\down u)\wedge\forall(v\leq u)\ldotp (\specdownclose V\cap\specdownclose
    v)\subseteq I \imp v\in I.%
    \footnote{(We will neither need this nor prove this, but $\widetilde{\basicCover}$ is the largest coverage for which $I$ is a sheaf.)}
  \]
  Then $\widetilde{\basicCover}$ contains $\basicCover$; indeed, suppose $V\basicCover u$. Then $V\ss\down u$ and if $v\leq u$ then there exists a $V'\subseteq(\specdownclose V\cap\specdownclose v)$ such that $V'\basicCover v$ by \cref{def:coverage}. So
  if $(\specdownclose V\cap\specdownclose v)\subseteq I$, then $V'\subseteq I$ and so $v\in I$.

  Moreover, $\widetilde{\basicCover}$ is a Grothendieck coverage. We verify the last condition, leaving
  the other, simpler, checks to the reader. Suppose $W\mathop{\widetilde{\basicCover}} u$ and
  $V\subseteq\specdownclose u$, such that for all $w\in W$, $(V\cap\specdownclose
  w)\mathop{\widetilde{\basicCover}} w$; we want to show $V\mathop{\widetilde{\basicCover}} u$. Taking
  any $v\leq u$, we claim that $(\specdownclose V\cap\specdownclose v)\subseteq I$ implies
  $(\specdownclose W\cap\specdownclose v)\subseteq I$, which would then imply that $v\in I$ since
  $W\mathop{\widetilde{\basicCover}} u$. To see the claim, if $x\leq w\in W$ and $x\leq v$, then by
  assumption $(V\cap\specdownclose w)\mathop{\widetilde{\basicCover}} w$, so $(\specdownclose
  V\cap\specdownclose w\cap\specdownclose x)\subseteq(\specdownclose V\cap\specdownclose v)\subseteq
  I$ implies $x\in I$. Hence $\widetilde{\basicCover}$ is a Grothendieck coverage containing
  $\basicCover$.
  
  We have assumed $I$ satisfies the sheaf condition for $\basicCover$, and we want to show that
  it does so for $\overline{\basicCover}$, i.e. that $V\mathop{\overline{\basicCover}} u$ and $V\subseteq I$ implies $u\in I$. But
  $V\mathop{\overline{\basicCover}} u$ implies $V\mathop{\widetilde{\basicCover}}$, so if $V\subseteq I$ then
  $(\specdownclose V\cap\specdownclose u)=\specdownclose V\subseteq I$, hence $u\in I$.

  The proof of the last claim is analogous; see also \cite[Prop.\ C2.1.9]{Johnstone:2002a}.
\end{proof}

The following definition is adapted from \cite{Shulman:2012a}:

\begin{definition}\label{def:dense_posite_morphism}\index{posite!dense morphism of|(}\index{dense!posite morphism|see {posite, dense morphism of}}
  Let $F\colon S\to T$ be an order-preserving map between Grothendieck posites. Call $F$ a
  \emph{dense morphism of posites} if
  \begin{itemize}
    \item for all $u\in S$ and $V\subseteq\specdownclose u$, $V\basicCover u$ if and only if
      $FV\basicCover Fu$,
    \item for all $u\in T$, $(FS\cap\specdownclose u)\basicCover u$,
    \item for all $u,v\in S$, if $Fu\leq Fv$ then $(\specdownclose u\cap\specdownclose v)\basicCover
      u$.
  \end{itemize}
\end{definition}

\begin{proposition}\label{prop:dense_morphism_to_iso}
  A dense morphism of Grothendieck posites $F\colon S\to T$ induces an isomorphsim
  $\PoShv{S}\iso\PoShv{T}$.
\end{proposition}
\begin{proof}
  From \cref{lem:poset_right_Kan} we have a map $F_*\colon D(T)\to D(S)$ and its right adjoint
  $F^!$, and it easy to verify that $F_*$ restricts to a map $F_*\colon\PoShv{T}\to\PoShv{S}$.

  Given any sheaf $J\in\PoShv{T}$, we first show $J=F^!F_* J$. The containment $J\subseteq F^!F_*J$ is the
  unit of the adjunction. In the other direction, suppose $v\in F^!F_*J$. Unwinding \cref{eqn:F_upper_lower}, for any $u\in S$, $Fu\leq
  v$ implies $Fu\in J$; hence $(FS\cap\specdownclose
  v)\subseteq J$. Then $(FS\cap\specdownclose v)\basicCover v$ because $F$ is dense, and we obtain $v\in J$ by the sheaf condition. Thus indeed $J=F^!F_* J$.

  Similarly, for any $I\in\PoShv{S}$, we have $I=F_*F^! I$, where $F_*F^! I\subseteq I$ is simply
  the counit of the adjunction. To show the converse $I\ss F_*F^!I$, choose $u\in I$; we need to show $F^{-1}(\down Fu)\ss I$, so choose also $u'\in S$ with $Fu'\leq Fu$. Clearly $(\specdownclose
  u'\cap\specdownclose u)\subseteq I$, and by density $(\specdownclose u'\cap\specdownclose u)\basicCover u'$, hence $u'\in I$ by the sheaf condition on $I$.

All that remains is to verify that $F^!(I)$ is a $(0,1)$-sheaf for any $I\in\PoShv{S}$. We already know it is down-closed by \cref{lem:poset_right_Kan}, so suppose $V\subseteq F^!(I)$ and $V\basicCover v$. We want to show that $v\in F^!(I)$, so
  consider $u\in S$ with $Fu\leq v$; we want to show $u\in I$. Since $\basicCover$ is a Grothendieck coverage, $V\basicCover v$ and $Fu\leq v$ imply $(\specdownclose V \cap
  \specdownclose Fu)\basicCover Fu$, which implies $(FS \cap \specdownclose V \cap \specdownclose Fu) \basicCover Fu$ by density and \cref{rmk:Grothendieck_properties}. But $(FS \cap \specdownclose V \cap \specdownclose Fu) = F(F_*(\specdownclose V) \cap
  F_*(\specdownclose Fu))$, so $(F_*(\specdownclose V) \cap F_*(\specdownclose Fu)) \basicCover u$.
  Since $(F_*(\specdownclose V) \cap F_*(\specdownclose Fu)) \subseteq F_*(\specdownclose
  V) \subseteq F_*F^!(I) = I$, we have $u\in I$. Hence $v\in F^!(I)$.
\end{proof}

Recall from \cref{rem:PoShv_complete} the canonical coverage on $\PoShv{S}$ and the associated sheaf functor $a\colon D(S)\to\PoShv{S}$ from \cref{lem:posite_sheafification}.\index{sheafification!$(0,1)$}

\begin{proposition}\label{prop:posite_dense_inclusion}
  For any Grothendieck posite $S$, the map $i\colon S\to\PoShv{S}$ given by
  \[
    i(u) \coloneqq a(\specdownclose u) = \{ v\in S \mid (\specdownclose u \cap \specdownclose v)\basicCover
    v \}
  \]
  defines a dense morphism of posites, regarding $\PoShv{S}$ with its canonical coverage.
\end{proposition}
\begin{proof}
  First note that for $I\in\PoShv{S}$ and $v\in S$, we have $iv\ss I$ iff $v\in I$. To show that $i$ is a dense morphism, there are three conditions to check. First, consider $u\in S$, $V\subseteq\specdownclose u$,
  and we show that $V\basicCover u$ iff $iV\basicCover iu$, where $iV\coloneqq\{iv\mid v\in V\}$. 
  Since $V\ss \down u\ss iu$, we have $iV\ss iu$, so by \cref{rem:PoShv_complete},
  \[\textstyle
    \sup(iV) 
    = 
    \bigcap\{I\in\PoShv{S}\mid \forall v\in V\ldotp iv\subseteq I\}
    =
    \bigcap\{I\in\PoShv{S}\mid V\ss I\}
  \]
  and $iV\basicCover
  iu$ iff $\bigcap\{I\mid V\ss I\}=iu$. One containment is automatic, so $iV\basicCover iu$ iff $iu\ss\bigcap\{I\mid V\ss I\}$, and this holds iff for each $I\in\PoShv{S}$, $V\subseteq I$ implies $u\in I$. Thus by the sheaf condition,
  $V\basicCover u$ implies $iV\basicCover iu$. Conversely, suppose $V\subseteq I$ implies $u\in I$
  for any $I\in\PoShv{S}$. Then $V\subseteq a(\specdownclose V)$ implies $u\in a(\specdownclose V)$,
  hence $\specdownclose V = (\specdownclose V \cap \specdownclose u)\basicCover u$, and therefore
  $V\basicCover u$.

  For the second condition, we want to show that for any $I\in\PoShv{S}$, $(iS\cap\specdownclose
  I)\basicCover I$, i.e.\ $I = \sup \{ iu \mid u\in I \}$, which is clear.

  For the third and last condition, we want to show that for any $u,v\in S$, if $iu\subseteq iv$
  then $(\specdownclose u\cap\specdownclose v)\basicCover u$. But $u\in iu\subseteq iv$ implies
  $(\specdownclose u\cap\specdownclose v)\basicCover u$ by definition.
\end{proof}

\begin{proposition}\label{prop:sheaves_on_posite}
  For any posite $S$, consider the frame $\PoShv{S}$ with the canonical coverage. Then $\Shv{S}\iso\Shv{\PoShv{S}}$.
\end{proposition}
\begin{proof}
  Let $\hat{S}$ be the Grothendieck completion of $S$ as in \cref{prop:completion_PoShv_iso}, by which we obtain isomorphisms $\PoShv{S}\iso\PoShv{\hat{S}}$ and
  $\Shv{S}\iso\Shv{\hat{S}}$. Then \cref{prop:dense_morphism_to_iso,prop:posite_dense_inclusion}
  together give the middle isomorphism in
  \[
    \Shv{S}\iso\Shv{\hat{S}}\iso\Shv{\PoShv{\hat{S}}}\iso\Shv{\PoShv{S}}.\qedhere
  \]
\end{proof}

\index{posite!dense morphism of|)}

\subsection{Equivalence of the various toposes}

We construct two Grothendieck coverages associated to a predomain $B$, then show that they are the
Grothendieck completions of the posites $S_B$ and $S'_B$ from
\cref{prop:predomain_posite,prop:down_closed_up_closed}.

\begin{proposition}\label{prop:predomain_Grothendieck_posite}\index{posite!from a predomain}
  Let $(B,\prec)$ be a predomain and $\upspecord$ its upper specialization order. Then there is a
  Grothendieck posite $\hat{S}_B$ with underlying poset $(B,\upspecord)\op$, and with
  coverage\index{coverage} defined by $V\basicCover b$ iff $\upclose b\subseteq \specupclose
  V\subseteq\specupclose b$.
\end{proposition}
\begin{proof}
  We only need to prove that $\basicCover$ is a Grothendieck coverage, i.e.\ the five conditions of
  \cref{def:Grothendieck_posite}. The first four conditions are clear. For the final condition, suppose
  $W\basicCover b$, $V\subseteq\specupclose b$, and for all $w\in W$, $(V\cap\specupclose
  w)\basicCover w$. Thus $\upclose b\subseteq \specupclose W\subseteq\specupclose b$, and for all
  $w\in W$, $\upclose w\subseteq \specupclose V$. Then we want to show that $V\basicCover b$, i.e.\
  $\upclose b\subseteq \specupclose V$. Clearly $\upclose w\subseteq\specupclose V$ for all
  $w\in\specupclose W$, hence by the interpolative property of $\ll$ (see \cref{prop:basic_facts_domains}), we have $\upclose b\subseteq \bigcup_{w\in\specupclose W}\upclose w \subseteq
  \specupclose V$.
\end{proof}

Recall from \cref{prop:down_closed_up_closed}\index{specialization order!on a predomain!lower}
  that for a predomain $(B,\prec)$, we say lower specialization implies upper specialization if $b\specord b'$ implies $b\upspecord b'$ for all $b,b'\in B$.

\begin{proposition}
\label{prop:Grothendieck_down_closed_up_closed}\index{specialization order!on a predomain!lower}
  Suppose that for a predomain $(B,\prec)$, lower specialization implies upper specialization. Then there is a Grothendieck posite
  $\hat{S}'_B$ with underlying poset $(B,\specord)\op$, and with coverage defined by $V\basicCover
  b$ iff $\upclose b\subseteq \specupclose V\subseteq \{b'\in B\mid b\specord b'\}$, and there is an
  isomorphism
  \[
    \PoShv{\hat{S}_B} \iso \PoShv{\hat{S}'_B}
  \]
\end{proposition}
\begin{proof}
  The proof that $\basicCover$ defines a Grothendieck coverage is essentially the same as
  \cref{prop:predomain_posite}.

  To see that $\PoShv{(B,\upspecord)\op} \iso \PoShv{(B,\specord)\op}$, we show that the identity on
  elements map $(B,\specord)\op\to(B,\upspecord)\op$ is a dense morphism of sites. By assumption,
  $b\specord b'$ implies $b\upspecord b'$, so this map is order-preserving. Let us write
  ${\bar{\uparrow}}b = \{b'\in B\mid b\specord b'\}$ and $\specupclose b = \{b'\in B\mid b\upspecord
  b'\}$, so we have ${\bar{\uparrow}}b\subseteq\specupclose b$ for any $b\in B$, and we similarly
  write $\overline{\basicCover}$ for the coverage on $(B,\specord)\op$ and $\basicCover$ for the coverage
  on $(B,\upspecord)\op$. Then for any $V\subseteq{\bar{\uparrow}}b$, we clearly have
  $V\mathrel{\overline{\basicCover}} b$ iff $V\basicCover b$.

  The remaining two conditions of \cref{def:dense_posite_morphism} are similarly trivial to verify.
  The proposition then follows from \cref{prop:dense_morphism_to_iso}.
\end{proof}

\begin{proposition}\label{prop:predomain_posite_completion}
  For any predomain $B$, the Grothendieck posite $\hat{S}_B$ from \cref{prop:predomain_Grothendieck_posite} is the Grothendieck completion of the
  posite $S_B$ defined in \cref{prop:predomain_posite}. Similarly, if $b\specord b'$ implies
  $b\upspecord b'$, then $\hat{S}'_B$ is the Grothendieck completion of $S'_B$.
\end{proposition}
\begin{proof}
  Let us temporarily write $\widehat{\basicCover}$ for the coverage on $\hat{S}_B$, and $\basicCover$
  for the coverage on $S_B$. Clearly $V\basicCover b$ implies $V\mathop{\widehat{\basicCover}} b$.
  Suppose $\widetilde{\basicCover}$ is any Grothendieck coverage containing $\basicCover$, i.e.\ such
  that $\upclose b\mathop{\widetilde{\basicCover}} b$ for any $b\in B$. Then if
  $V\mathop{\widehat{\basicCover}} b$, i.e.\ if $\upclose b\subseteq \specupclose V\subseteq\specupclose
  b$, then $\specupclose V\mathop{\widetilde{\basicCover}} b$ by \cref{rmk:Grothendieck_properties},
  hence $V\mathop{\widetilde{\basicCover}}b$. Thus $\widetilde{\basicCover}$ contains $\widehat{\basicCover}$,
  so $\widehat{\basicCover}$ is the smallest Grothendieck coverage containing $\basicCover$.

  The second case is similar.
\end{proof}

\begin{corollary}\label{cor:S_iso_S'}
  Let $B$ be a predomain for which lower specialization implies upper specialization. Then
  $\PoShv{S_B} \iso \PoShv{S'_B}$.
\end{corollary}
\begin{proof}
  This follows from
  \cref{prop:Grothendieck_down_closed_up_closed,prop:predomain_posite_completion,prop:completion_PoShv_iso}.
\end{proof}

The following theorem was stated without proof in \cref{thm:IR_sheaves}; we can now prove the equivalence of 1,2, and 3, and sketch the proof that these are equivalent to 4.

\begin{theorem}\label{thm:IR_sheaves_proven}
  The following categories are equivalent:
  \begin{enumerate}
    \item The category of sheaves $\Shv{\Opens(\IR)}$ on the frame of opens of $\IR$, with the
      canonical coverage $\{V_i\}_{i\in I}\basicCover U$ iff $\bigcup_{i\in I} V_i=U$.
    \item The category of sheaves $\Shv{S'_{\IR}}$ on the interval posite $S_{\IR}$.
    \item The category of sheaves $\Shv{S^\QQ_{\IR}}$ on the rational interval posite $S^\QQ_{\IR}$.
    \item The category of $\Cont{\IR}$ of \emph{continuous functors} $\IR\to\Cat{Set}$, i.e.\
      functors taking directed suprema in $\IR$ to colimits in $\Cat{Set}$.
  \end{enumerate}
\end{theorem}
\begin{proof}
  The proof of the equivalence $\Shv{\Opens(\IR)} \iso \Cont{\IR}$ is sketched in
  \cref{rmk:sheaves_cont_functors}. For the rest, recall from \cref{def:interval_posite} that $S^\QQ_{\IR} = S_{\IRpre}$ and that $S_{\IR}=S'_{\IR}$, as defined in \cref{prop:down_closed_up_closed}. It follows from
  \cref{prop:predomain_posite_opens,cor:S_iso_S',prop:domains_as_predomains} that $\PoShv{S'_{\IR}}
  \iso \PoShv{S_{\IR}} \iso \Opens(\IR) \iso \PoShv{S_{\IRpre}}$. Hence by
  \cref{prop:sheaves_on_posite},
  \[
    \Shv{S'_{\IR}} \iso \Shv{\PoShv{S'_{\IR}}} \iso \Shv{\Opens(\IR)} \iso
    \Shv{\PoShv{S^\QQ_{\IR}}} \iso \Shv{S^\QQ_{\IR}}.\qedhere
  \]
\end{proof}

\chapter{Translation invariance}\label{chap:translation_invariance}

As discussed in \cref{sec:translation_invariance}, the topos $\Shv{S_{\IR}}$ of sheaves on the interval domain is slightly unsatisfactory as a model of
behaviors. For example, to serve as a compositional model of dynamical systems, we do not want the set of possible
behaviors in some behavior type to depend on any global time. In this chapter, we define a topos $\BaseTopos$\index{topos!of behavior types} of
``translation-invariant behavior types''\index{behavior type!translation-invariant} by defining a translation-invariant version of the interval domain and a corresponding site, denoted $\IRinv$ and $\BaseSite$ respectively, and letting $\BaseTopos\coloneqq\Shv{\BaseSite}$.\index{interval domain!translation-invariant} In an appendix, \cref{chap:IRinv_continuous}, we prove that $\IRinv$ is a continuous category in the sense of Johnstone-Joyal \cite{Johnstone.Joyal:1982a}, giving us an analogue to \cref{thm:IR_sheaves}; this is briefly discussed in \cref{sec:IRinv_continuous}. 

\section{Construction of the translation-invariant topos $\BaseTopos$}

Recall the interval domain $(\IR,\sqss)$ from \cref{def:interval_domain}. We will define $\BaseTopos$ in terms of a site $\IRinv$, obtained as the quotient of $\IR$ by a free $\RR$-action.\index{interval domain}\index{action, $\RR$}

\subsection{The translation-invariant interval domain, $\IRinv$}

We have two primary definitions of the translation-invariant interval domain $\IRinv$, and it is
easy to show their equivalence. On one hand, there is a continuous action $\rhd$\index{translation} of $\RR$ on $\IR$,%
\footnote{The $\rhd$ symbol used here (e.g.\ ``$r\rhd[d,u]$'') to denote a particular $\RR$-action on $\IR$, is different from the $\basicCover$ symbol, used in \cref{sec:the_interval_domain} (e.g.\ ``$V\basicCover u$'') to denote a basic cover in a posite.}
given by $r\rhd[d,u]\coloneqq [r+d,r+u]$, and we can consider the category of elements for this action. It is given by adjoining an isomorphism $(d,u,r):[d,u]\to[d+r,u+r]$
to $\IR$ for each $r\in\RR$, with $(d,u,r)^{-1}=(d+r,u+r,-r)$ and $(d+r,u+r,r')\circ(d,u,r) = (d,u,r+r')$.

On the
other hand, and more conveniently, we can simply take $\IRinv$ to be the category of orbits of the $\rhd$ action.\index{interval domain!translation-invariant} The
equivalence class containing an interval $[d,u]$ can be represented by its length $\ell\coloneqq u-d$, and the
equivalence class of a specialization $[d,u]\sqss[d',u']$ can be represented by the nonnegative numbers $r\coloneqq d'-d$ and $s\coloneqq u-u'$, how much to ``shave off'' from the left and right of the longer interval. This leads to the following definition.\index{length}\index{restriction}

\begin{definition}[The translation-invariant interval category $\IRinv$]\label{def:IRinv}
  Define the \emph{translation-invariant interval category}, denoted $\IRinv$, to have
  \begin{itemize}
    \item objects $\{\,\ell\in\RR\mid \ell\geq 0\,\}$,
    \item morphisms $\IRinv(\ell,\ell') = \{\,\subint{r}{s}\in\RR\times\RR\mid (r\geq 0)\wedge(s\geq
      0)\wedge(\ell=r+\ell'+s)\,\}$,
		\item identities $\subint{0}{0}$, and
    \item composition $\subint{r'}{s'}\circ\subint{r}{s} = \subint{r'+r}{s'+s}$.
  \end{itemize}
	We refer to an object $\ell\in\IRinv$ as a \emph{translation-invariant interval}.
\end{definition}
Note that one can consider $\rhd$ as an $\RR$-action on $\IR\op$ as well, and there is an isomorphism of categories $(\IR\op)_{/\rhd}\cong(\IR_{/\rhd})\op$; in the future we will elide the difference and denote them both simply by $\IRinv\op$.

\begin{remark}
The category $\IRinv$ is isomorphic to the twisted arrow category $\Tw{\B\RR_{\geq0}}$, where $\RR_{\geq 0}$ is the monoid of nonnegative real numbers\index{monoid!of nonnegative reals} and $\B\RR_{\geq 0}$ denotes the corresponding category with one object. This category was fundamental in an earlier paper \cite{Spivak.Vasilakopoulou.Schultz:2016a} and was considered by Lawvere and others \cite{Lawvere:1986a,Bunge.Fiore:2000a,Fiore:2000a} in the context of dynamical systems; see also \cite{Johnstone:1999a}.\index{category!twisted arrow}\index{dynamical systems!categorical approaches}
\end{remark}

There is an evident quotient functor $p\colon\IR\to\IRinv$, which sends $[d,u]\in\IR$ to
$p[d,u]\coloneqq u-d\in\IRinv$, and sends $[d,u]\sqss[d',u']$ to $\subint{r}{s}\colon\ell\to\ell'$, where $\ell\coloneqq u-d$, $\ell'\coloneq u'-d'$, $r\coloneqq d'-d$, and $s\coloneqq u-u'$. In \cref{lem:discrete_fibration} we will show that $p$ is a discrete bifibration,\index{functor!discrete bifibration}\index{bifibration|see {functor, discrete bifibration}} i.e.\ both a discrete fibration and a discrete opfibration. This roughly means that if $[d,u]$ is an interval with extent $\ell\coloneqq u-d$, then any extent containing $\ell$ (resp.\ contained in $\ell$) will have a unique lift to an interval containing $[d,u]$ (resp.\ contained in $[d,u]$).

\begin{lemma}\label{lem:discrete_fibration}
  The functor $p\colon\IR\to\IRinv$ is a discrete bifibration, and it is surjective on objects. The same holds for $p\op\colon\IR\op\to\IRinv\op$.
\end{lemma}
\begin{proof}
  Clearly, if $p$ is a bifibration and surjective on objects then so is $p\op$.
  
  It is obvious that $p$ is surjective on objects. Given $[d,u]\in\IR$ and $\subint{r}{s}\colon(u-d)\to\ell$, a morphism $[d,u]\sqss[d',u']$ projects to
  $\subint{r}{s}$ iff $r=d'-d$ and $s=u-u'$, hence $d'=d+r$ and $u'=u-s$ is the unique lift of $\subint{r}{s}$.
  This shows that $p$ is a discrete opfibration. The proof that $p$ is a discrete fibration is
  similar.
\end{proof}

Recall the way-below relation on $\IR$ from \cref{prop:way_below_IR}: $[d,u]\ll[d',u']$ iff $d<d'\leq u'<d$ iff $r>0$ and $s>0$. We follow Johnstone and Joyal (\cite{Johnstone.Joyal:1982a} and \cite[Definition C.4.2.12]{Johnstone:2002a}) and give the following definition.\index{way-below}

\begin{definition}[Wavy arrow]
\label{def:wavy_arrow}\index{wavy arrow}
For any $\ell,\ell'\in\IRinv$, we say that a morphism $\subint{r}{s}\colon\ell\to\ell'$ is a \emph{wavy arrow} and denote it $\subint{r}{s}\colon\ell\wavyto\ell'$, if $r>0$ and $s>0$.

We will use the same notation for morphisms in $\IRinv\op$, i.e.\ write $f\colon \ell'\wavyto\ell$, and call $f$ a wavy arrow, iff $f\op$ is a wavy arrow in $\IRinv$, as defined above.
\end{definition}

It is tempting to think of $\ell'\wavyto\ell$ as a ``strict subinterval'', but one must remember that this strictness applies to both sides, $r$ and $s$.

\subsection{The topos $\BaseTopos$ of behavior types}

We are now ready to define the site of translation-invariant intervals and the corresponding topos of sheaves.

\begin{definition}[The site $\BaseSite$ and the topos $\BaseTopos$ of behavior types]\label{def:BaseSite}\index{base site $\BaseSite$}\index{coverage!of base site}
We define the site of \emph{translation-invariant intervals}, denoted $\BaseSite$, as follows:
	\begin{itemize}
		\item The underlying category is $\IRinv\op$ as in \cref{def:IRinv}, and
		\item the coverage consists of one family, $\{\subint{r}{s}\colon\ell'\to\ell\mid r,s>0\}$, for each object $\ell$.
	\end{itemize}
That is, the collection of wavy arrows $\{\ell'\wavyto\ell\}\ss\IRinv\op(\ell',\ell)$ covers $\ell$. We refer to $\BaseTopos\coloneqq\Shv{\BaseSite}$ as the topos of \emph{behavior types}.\index{topos!of behavior types}
\end{definition}

\Cref{def:BaseSite} given above is analogous to the posite version given in \cref{def:interval_posite}. Indeed, in \cref{prop:IRinv_continous} we will show that for any $\ell'\in\IRinv$, the functor $\IRinv\op\to\Cat{Set}$ given by $\ell\mapsto\{f\colon\ell\wavyto\ell'\in\IRinv(\ell,\ell')\}$ is always flat, and hence an object of $\Ind(\IRinv)$. This assignment is functorial in $\ell'$ and is in fact left adjoint to $\colim\colon\Ind(\IRinv)\to\IRinv$. Thus it is analogous to the map $\downclose\colon\IR\to\Id(\IR)$ which is left adjoint to $\sup\colon\Id(\IR)\to\IR$ (see \cref{prop:cont_poset,prop:way_below_IR}).

\begin{remark}\label{rem:use_l>0}
The object $0\in\IRinv$ has an empty covering family in $\BaseSite$. Thus any object $X$ of the topos $\BaseTopos=\Shv{\BaseSite}$ has $X(0)\cong\{*\}$, and is hence completely determined by its values on objects $\ell>0$. It is often more convenient to work with the subsite $\BaseSite'\ss\BaseSite$ whose underlying category is the full subcategory of $\IRinv$ spanned by the objects $\ell>0$, and whose covering families match those in $\BaseSite$. In $\BaseSite'$, every covering family is filtered.
\end{remark}

\begin{proposition}\label{prop:quotient_geom_morphism}
The quotient functor $p\op\colon\IR\op\to\IRinv\op$ (see \cref{lem:discrete_fibration}) induces a geometric morphism
$p_*\colon\Shv{S_{\IR}}\to\BaseTopos$ with left exact left adjoint $p^*\colon\BaseTopos\to\Shv{S_\IR}$.\index{geometric morphism}
\end{proposition}
\begin{proof}
It suffices by \cite[C2.3.18]{Johnstone:2002a} to check that $p\op$ is cover-reflecting.\index{site!comorphism of} This is immediate once one consults \cref{def:interval_posite} and \cref{def:BaseSite}.
\end{proof}

A topos $\cat{E}$ is \emph{locally connected} if the global sections geometric morphism $\Gamma\colon\cat{E}\to\Cat{Set}$ is an essential geometric morphism, i.e.\ if the inverse image functor $\Fun{Cnst}$ has a left adjoint $\Pi_0$; see \cite[Lemma C.3.3.6]{Johnstone:2002a}. It is \emph{connected} if $\Fun{Cnst}$ is fully faithful; this is equivalent to the condition that $\Pi_0$ preserves the terminal object. To prove that $\cat{E}$ is locally connected, it suffices to show that it is the topos of sheaves on a locally connected site, meaning that for each object $U$, all covering sieves of $U$ are connected. 

\begin{proposition}\label{prop:BaseTopos_conn_locally_conn}\index{topos!connected}\index{topos!locally connected}
The topos $\BaseTopos$ is locally connected and connected.
\end{proposition}
\begin{proof}
By \cref{rem:use_l>0}, there is a defining site $\BaseSite'$ for $\BaseTopos$, in which all covering families are filtered; they are in particular connected, so $\BaseTopos$ is locally connected.

The composite $\Cat{Set}\To{\Fun{Cnst}}\Shv{\BaseSite'}\to\Psh{\BaseSite'}$ is fully faithful, because  the category underlying $\BaseSite'$ is connected. Since $\Shv{\BaseSite'}\to\Psh{\BaseSite'}$ is fully faithful, being the direct image of a geometric inclusion, it follows that $\Fun{Cnst}$ is fully faithful as well. 
\end{proof}

\begin{corollary}\label{cor:exponential_constant}
If $C,D$ are constant sheaves, then so is $C^D$.
\end{corollary}
\begin{proof}
We use the Yoneda lemma and \cref{prop:BaseTopos_conn_locally_conn}. For any $X\in\BaseTopos$, we have
\begin{align*}
	[X,\Fun{Cnst}(C)^{\Fun{Cnst}(D)}]
	&\cong[X\times\Fun{Cnst}(D),\Fun{Cnst}(C)]\\
	&\cong[\Pi_0(X\times\sqcup_{d\in D}1),C]\\
	&\cong[D\times\Pi_0(X),C]\cong[X,\Fun{Cnst}(C^D)].\qedhere
\end{align*}
\end{proof}

\section{\texorpdfstring{$\IRinv$}{\unichar{"1D540}\unichar{"211D}\unichar{"2215}\unichar{"25B9}} as
a continuous category}\label{sec:IRinv_continuous}

Continuous categories, defined by Johnstone and Joyal in \cite{Johnstone.Joyal:1982a}, are a generalization of domains (continuous posets)\index{category!continuous}\index{domain}. In fact, a poset is a domain iff it is continuous as a category. Under this generalization from posets to categories, ideals are replaced by Ind-objects,\index{Ind-object|see {object, $\Ind$-}}\index{object!$\Ind$-} directed suprema\index{supremum} are replaced by filtered colimits, the way-below relation \index{way-below} is replaced by the set of wavy-arrows (\cref{def:wavy_arrow})\index{wavy arrow}, and Scott-continuous\index{Scott continuous|see {continuous, Scott}}\index{continuous!Scott} morphisms between domains are replaced by continuous functors: those which preserve filtered colimits.

Recall that a \emph{point} of a topos\index{topos!point of}\index{point|see {topos, point of}} $\cat{E}$ is a geometric morphism\index{geometric morphism} $p\colon\Cat{Set}\to\cat{E}$, and a morphism of points $p\to q$ is a natural transformation of inverse images $p^*\to q^*$. Let $P_\cat{E}$ denote the category of points in $\cat{E}$. There is always a functor $s\colon\cat{E}\to\Cat{Set}^{P_\cat{E}}$, sending a sheaf $X$ and a point $p$ to the set $p^*(X)$. When $\cat{E}$ is the topos of sheaves on a continuous category, $s$ is fully faithful and its essential image is the category (topos) of continuous functors.

We discuss the above subject in greater detail in \cref{chap:IRinv_continuous}. For example, in \cref{prop:IRinv_continous} we prove that $\IRinv$ is a continuous category. We also prove an equivalence of categories 
\[\Shv{\BaseSite}\cong\Cont{\IRinv}\]
where $\Cont{\IRinv}$ denotes the category of continuous---i.e.\ filtered-colimit preserving---functors $\IRinv\to\Cat{Set}$. This gives two equivalent definitions of the topos $\BaseTopos$, analogous to the equivalence of \cref{thm:IR_sheaves} 3 and 4.

\section{The subobject classifier}\label{sec:subobject_classifier_BaseTopos}

The subobject classifier of the topos $\Shv{S_\IR}$ has a nice visual interpretation as the set of Lipschitz functions in the upper half-plane (see \cref{sec:IR_half_plane})\index{half-plane}. We can transport this geometric picture to $\BaseTopos$ because the semantics of the subobject classifier $\Omega$ in $\BaseTopos=\Shv{\BaseSite}$ is strongly related to those of the subobject classifier $\Omega_{\IR}$ for the topos $\Shv{S_\IR}$, as we now explain.\index{subobject classifier}

For any object $\ell\in\BaseSite$ and $a\in\RR$, there is a basic open set $\upclose[a,b]$ in $\IR$, where $b=a+\ell$; it is defined as
\[
\upclose[a,b]\coloneqq\{[x,x']\mid a<x\leq x'<b\}.
\]
Consider
the frame $\Omega_\IR(\upclose[a,b])$ of open subsets $U\ss\upclose[a,b]$. It is independent
of $a$ in the sense that there is a canonical isomorphism
$\Omega(\upclose[a,b])\cong\Omega(\upclose[0,\ell])$. Recall from \cref{prop:basic_facts_domains} that $\IR$ has a basis consisting
of such open intervals $\upclose[a,b]$.

\begin{proposition}\label{prop:Prop_Omega}
  There is a canonical bijection $\Omega(\ell)\To{\cong}\Omega_{\IR}(\upclose[0,\ell])$.
\end{proposition}
\begin{proof}
  As in any sheaf topos, the set $\Omega(\ell)$ can be identified with the set of
  closed sieves on $\ell$ in $\BaseSite$. This set depends only on the slice
  category $(\ell\downarrow\IRinv)$, and because the projection $p\colon\IR\to\IRinv$ is a discrete
  opfibration, as shown in \cref{prop:quotient_geom_morphism}, we indeed have $(\ell\downarrow\IRinv)\iso([0,\ell]\downarrow\IR)$.
\end{proof}

By the above proposition, any $\omega\in\Omega(\ell)$ corresponds to a Scott-open set
$U_\omega\ss\upclose[0,\ell]$ in $\IR$. We can use Dyck paths (\cref{cor:opens_Dyck}) to picture the subobject classifier $\Omega$: for any $\ell$, the poset $\Omega(\ell)$ is isomorphic to the set of Dyck paths on $[0,\ell]$.\index{Dyck path} Given $\subint{r}{s}\colon\ell'\to\ell$, the open set $U_{\omega'}$ corresponding to the restriction
$\omega'\coloneqq\omega|_{\subint{r}{s}}$ is
\[U_{\omega'} = U_\omega \cap\upclose[r,\ell-s].\]

\section{The behavior type $\Time$}
\label{sec:Time}\index{Time@$\Time$}\index{Time@$\Time$!semantics of}

Let $p^*\dashv p_*$ be the geometric morphism from \cref{prop:quotient_geom_morphism}. We claim that
it is \emph{\'etale}, meaning that there is an object $T\in\BaseTopos$ such that $p_*$ is equivalent
to the projection $\BaseTopos/T\to\BaseTopos$. It would follow that $p^*$ has a further left adjoint
$p_!$, and that $T=p_!(1)$.

\subsection{The geometric morphism $\Shv{S_\IR}\to\BaseTopos$ is \'{e}tale}

To prove that $p_*$ is \'etale, we show directly that its left adjoint $p^*$ has a further left adjoint, and
use this to show that $p_*$ is isomorphic to the projection $\BaseTopos/p_!(1)\to\BaseTopos$.

The left adjoint $p_!\colon\Shv{S_{\IR}}\to\BaseTopos$, if it exists, must be given by left Kan
extension along $p\op\colon\IR\to\IRinv$. The quotient functor $p$ is a discrete opfibration,
as shown in \cref{lem:discrete_fibration}. It follows that the left Kan extension of
$X\colon\IR\to\Cat{Set}$ is simply fiberwise coproduct:
\begin{equation}\label{eqn:p_!}
(p_!X)(\ell)=\coprod_{\{(d,u)\mid u-d=\ell\}}X([d,u])
\end{equation}
with restriction along $\subint{r}{s}\colon\ell'\to\ell$ given by $((d,u),x)|_{\subint{r}{s}} = ((d+r,u-s),x|_{(d+r,u-s)})$. It
only remains to show that $(p_!X)\colon\IRinv\to\Cat{Set}$ is indeed a sheaf, but this is clear given
that $X$ is a sheaf.

We define the sheaf $\Time\coloneqq p_!(1)$. Using the above formula for $p_!$, we can identify
it with the functor $\IRinv\to\Cat{Set}$ whose length-$\ell$ sections are given by
\begin{equation}\label{eqn:def_time}
  \Time(\ell)=\{\, (d,u)\in\RR^2 \mid u-d=\ell \,\},
\end{equation}
and whose restriction map for $\subint{r}{s}\colon\ell\to\ell'$ is given by $\Time\subint{r}{s}(d,u)
= (d+r,u-s)$.

\begin{proposition}\label{prop:IR_BaseTopos/Time}\index{topos!\'{e}tendue}
  There is an equivalence of categories $\Shv{S_{\IR}}\cong\BaseTopos/\Time$, such that the diagram
  \[
    \begin{tikzcd}[column sep=small]
      \Shv{S_{\IR}} \ar[rr,"\cong"] \ar[dr,"p_*"']
        && \BaseTopos/\Time \ar[dl,"P_*"] \\
      & \BaseTopos &
    \end{tikzcd}
  \]
  commutes, where $(P^*,P_*)$ is the base change geometric morphism along $!_\Time\colon\Time\to 1$
  in $\BaseTopos$.
\end{proposition}
\begin{proof}
  The equivalence sends any sheaf $X\in\Shv{S_{\IR}}$ to $p_!(!_X)$, where $!_X\colon X\to 1$ is the
  unique map to the terminal sheaf. Recall that $p_!(X)(\ell)=\coprod_{u-d=\ell}X(d,u)$; the
  function $p_!(!_X)_\ell$ sends $x\in X(d,u)$ to $(d,u)\in\Time(\ell)$. It is straightforward to
  check that this determines an equivalence.

  The base change geometric morphism $(P^*,P_*)$ is trivially \'etale, and in particular $P^*$
  has a further left adjoint $P_!\colon\BaseTopos/\Time\to\BaseTopos$, which is given by
  post-composition with $P$, i.e.\ $P_!$ is the standard projection out of the slice category.
  Then it suffices to show the diagram involving $p_!$ and $P_!$ commutes,
  but this is obvious.
\end{proof}

\begin{corollary}
$\BaseTopos$ is an \'{e}tendue.
\end{corollary}
\begin{proof}
It is easy to see by \cref{eqn:def_time} that the unique map $\Time\to 1$ is an epimorphism; so the result follows from \cref{thm:IR_sheaves_proven,prop:IR_BaseTopos/Time}.
\end{proof}

\begin{remark}\label{rem:time_clock_temporal}
  We think of the sheaf $\Time$ as the behavior type of a clock.\index{Time@$\Time$!as clock behavior} Given a behavior of length $\ell$,
  a $(d,u)\in\Time(\ell)$, we imagine this as the behavior of a clock which reads $d$ at the
  beginning of the interval and reads $u$ at the end. Clearly either one determines the other; see \cref{rem:time_semantics}. In temporal logic applications\index{logic!temporal} (see \cref{sec:temporal_logic}) it is
  sometimes helpful to think of the location where the clock reads zero as representing ``now'',
  even when that location is outside of the interval, i.e.\ when $d$ and $u$ are both positive or
  both negative.
\end{remark}

\subsection{$\BaseTopos$ as a quotient of $\Shv{S_{\IR}}$}

The topos $\BaseTopos$ is the quotient of the localic topos $\Shv{S_{\IR}}$.\index{topos!localic} One way to see this is that $p^*\colon\BaseTopos\to\Shv{S_{\IR}}$ is faithful. Another way to see it is that $\BaseTopos$ is equivalent to the category of coalgebras on the left-exact comonad $p^*p_*$ on $\BaseTopos/\Time$.\index{comonad}\index{coalgebra}

With $P$ as in \cref{prop:IR_BaseTopos/Time}, it is easy to see that for $X\in\BaseTopos$, we have $P^*X=X\times\Time$, and for $(\tau\colon Y\to\Time)\in\BaseTopos/\Time$, we have $P_!\tau=Y$ and $P_*\tau=\{f\colon\Time\to Y\mid \tau\circ f=\id_\Time\}$, i.e.\ $P_*\tau$ is the sheaf of sections of $\tau$. The sections of $p_!X$ were given in \cref{eqn:p_!}, though we repeat it here for the reader's convenience, along with those of $p^*$ and $p_*$. For $X\in\BaseTopos$, we have
\[p^*X([d,u])\cong X(u-d),\]
and for $Y\in\Shv{S_{\IR}}$, we have
\[
	p_!Y(\ell)\cong\coprod_{\{(d,u)\mid u-d=\ell\}}Y([d,u])
	\qquad\text{and}\qquad
	p_*Y(\ell)\cong\prod_{\{(d,u)\mid u-d=\ell\}}Y([d,u])
\]

For any sheaf $Y\in\Shv{S_\IR}$ and $r\in\RR$, let $Y\rhd r$ be the sheaf $(Y\rhd r)([d,u])\coloneqq Y([d+r,u+r])$. The comonad $p^*p_*$ thus sends $Y$ to $\prod_{r\in\RR}(Y\rhd r)$; the counit is projection at $r=0$, and the comultiplication is the obvious map $\prod_{r\in\RR}(Y\rhd r)\to\prod_{r_1,r_2}(Y\rhd(r_1+r_2))$. Thus an algebra for the comonad $p^*p_*$ is a sheaf $Y\in\Shv{S_\IR}$ equipped with a homomorphism $h_r\colon Y\to (Y\rhd r)$ for each $r\in\RR$, such that $h_0=\id_Y$ and $h_{r_1+r_2}=h_{r_1}\circ h_{r_2}$. Since $\RR$ is a group, this in particular implies that $h_r$ is an isomorphism for all $r\in\RR$.

\begin{proposition}
A sheaf in $\BaseTopos$ can be identified with a sheaf $Y\in\Shv{S_\IR}$ equipped with a translation isomorphism $h_r\colon Y\To{\cong}(Y\rhd r)$ for each $r\in\RR$, such that $h_0=\id_Y$ and $h_{r_1+r_2}=h_{r_1}\circ h_{r_2}$.\index{translation}
\end{proposition}

\chapter{Logical preliminaries}\label{sec:logical_prelims}\index{logic|(}\index{type theory|(}

In this chapter, we transition from an external point of view to an internal one. In
\cref{sec:the_interval_domain} we defined the interval domain $\IR$, and in
\cref{chap:translation_invariance} we defined a quotient $\BaseTopos\cong\Shv{\BaseSite}$ of its
topos of sheaves. A main goal of this book is to define a temporal type theory---including one atomic predicate and
ten axioms---that has semantics in $\BaseTopos$; we do this in \cref{sec:axiomatics}. In the present
chapter, we attempt to provide the reader with a self-contained account of the sort of type theory
and logic we will be using, as well as some important concepts definable therein.

We begin in \cref{sec:toposes_types_logic} with an informal introduction to our type theory and
logic, as well as its relation to toposes. We then discuss modalities in
\cref{sec:modalities},\index{modality}, which are the internal view of what in topos theory
literature are often called Lawvere-Tierney topologies or local operators. Logically, modalities are
internal monads on the type $\Prop$ of propositions, whereas semantically they correspond to
subtoposes. Finally in \cref{sec:Dedekind_j}, we discuss numeric types, e.g.\ the Dedekind real
numbers and related types, relative to an arbitrary modality. In this section we also discuss
inequalities between, and the arithmetic of, numeric types.

\section{Informal introduction to type theory}\label{sec:toposes_types_logic}

In this section we will informally introduce the type theory we use in the remainder of the book. We
justify the lack of formality in three ways. First, most of what we do is relatively standard
material: higher-order logic on top of the simply-typed lambda
calculus\index{lambda calculus, simply typed} with sum types\index{type!sum} and a type of natural
numbers\index{type!of natural numbers}, plus subtypes\index{type!subtype} and quotient
types\index{type!quotient}. However, we also make very limited use of dependent
types\index{type!dependent}, so providing full details of our type theory would be a major endeavor.
This is the second reason: a full account would simply take too much space to be worthwhile. And the
third reason is that we want to be users of logic---not logicians---here. Thus we will present
enough information to get readers started; for those who want a full account of categorical logic
and type theory, we recommend \cite{Jacobs:1999a}.

What will matter most is that our type theory is constructive\index{logic!constructive} and has an
object $\Prop$ of propositions---and hence has a higher-order logic.\index{logic!higher order} It is
well-known that this logic has semantics in any topos \cite{MacLane.Moerdijk:1992a}. Sometimes we
use dependent types, subtypes, and quotient types; it is less well-known but true that these also
have semantics in any topos \cite{Maietti:2005a}. While we were not able to find a reference, we
believe that the Calculus of Constructions (\cite{Coquand.Huet:1988a})---used in the automated proof
assistants Lean and
Coq\index{Coq|see {automated proof assistant}}\index{Lean|see {automated proof assistant}}---also
has sound semantics in any topos, so the system we present below could be formalized in either proof
assistant.%
\footnote{
  The dependent type theory in \cite{Maietti:2005a} is sound and complete for 1-toposes, relying
  heavily on extensivity of types\index{type!extensive}. This is fine for 1-toposes, but it makes
  computability impossible and hence is not useful for a proof assistant like Coq. On the other
  hand, the Calculus of Constructions (\cite{Coquand.Huet:1988a}) and related formalisms used in Coq
  and Lean are formulated using a hierarchy of universes. We believe this aspect of the Calculus of
  Constructions is sound for toposes, and we believe this follows from \cite{Streicher:2005a}, but
  could not find an explicit reference.
}

\subsection{Notions from the simply-typed lambda calculus}\label{sec:notions_lambda}\index{lambda calculus, simply typed|(}

In this section (\cref{sec:notions_lambda}) we give an informal account of the simply-typed lambda
calculus. The material is standard; see e.g.\ \cite{Lambek.Scott:1988a,Jacobs:1999a}. The key words
we will discuss are: \emph{types}, \emph{variables}, \emph{contexts}, \emph{terms}, \emph{substitution}, and
\emph{conversion rules}. As we get into \cref{sec:HOL} and begin to talk about logic, we will add
\emph{propositions}, \emph{premises}, and \emph{truth judgments}.\index{type}\index{variable}\index{context}\index{term}\index{substitution}\index{conversion rule}\index{proposition}\index{premise}\index{truth judgment}

It is important to understand the relationship between a type theory and its semantics in the
appropriate kind of category, say a topos $\cat{E}$, without conflating the two.\index{semantics!of a type theory} Each has value
independent of the other; for example, computers do not work on semantics, they work on syntax.
Having a syntax that ``compiles'' to our chosen semantics can be quite freeing, because proofs in
the logic are far easier to check than those in the semantics they compile to. Moreover, there is
not always a perfect correspondence between what is important in type theory and what is important
in semantics. For example, variables and contexts in a type theory have little semantic meaning.
Still, the connection between the type theory and its $\cat{E}$-semantics is important, so we
summarize the touch points now, and will continue to do so throughout the section.

Types in our type theory correspond to objects in $\cat{E}$. Variables and contexts---which are
lists of variables and types---mainly serve a bookkeeping role. They gain semantic meaning when they
are attached to terms. That is, a well-formed term always has a context and a type; it looks like
this
\begin{equation}\label{eqn:sample_term}
  a:A, b_1:B, b_2:B\vdash f(a,g(b_1,b_2),a):C
\end{equation}
Here, $A$, $B$, and $C$ are types; $a$, $b_1$, and $b_2$ are variables; $a$, $b_1$, $b_2$,
$g(b_1,b_2)$, and $f(a,g(b_1,b_2),a)$ are terms. These terms appear in the context $a:A, b_1:B, b_2:B$ and, for example the term $f(a,g(b_1,b_2),a)$
has type $C$. A term corresponds to a morphism in $\cat{E}$; its domain is the product of types in
its context, and its codomain is its type. So in the above example, $f(a,g(b_1,b_2),a)$ presents a
morphism $A\times B\times B\to C$.

There is a grammar for types and a grammar for terms, meaning that new types and terms can be
produced from old using what are called \emph{type constructors} and \emph{term constructors}.\index{constructor!type}\index{constructor!term} There
are several rules that dictate when the resulting types and terms are well-formed and when two terms
are declared equal; the latter are the conversion rules. Issues like which variables are free and
how substitution works for replacing a variable by a term (one must be careful that variables are
not accidentally bound in the process)\index{variable!free/bound}, can be carefully worked out. This is important for computer
applications, and may be interesting, but it will not concern us much here because mathematicians
have generally developed an engineer's sense of how to work with types and terms without needing to
know much about what goes on under the hood.

We now expand on the above summary of the simply-typed lambda calculus and its semantics.

\paragraph{Types.}\index{type}

We begin with types. One may declare atomic types $\const{T}_1,\const{T}_2,\ldots$ (though we will
not need to do so in this book). The type $\Prop$ of propositions, the type $\tNN$ of natural
numbers, the empty type $\const{empty}$, and the unit type $1$ are given as a base types, without
needing to declare them. From these, one may form new types by taking finite products, finite sums,
and arrow types; these are the type constructors. So suppose we declare $\const{T}$ to be our only
atomic type. Then we also have, for example, the three following types:\index{constructor!type}\index{type!atomic}\index{type!unit}\index{type!product}\index{type!empty}\index{type!of natural numbers}\index{type!of propositions}\index{type!arrow}
\[
  \const{T}\times \const{T}\times \const{T},\quad
  1, \quad
  \const{T}\times\Prop\to(\tNN\to (\const{T}+1)).
\]
To give types semantics in $\cat{E}$, one must choose an object for each atomic type
($\const{T}_i$). The rest is automatic: the type $\Prop$ corresponds to the subobject classifier in
$\Cat{E}$, the type $\tNN$ corresponds to the natural numbers object in $\cat{E}$, the types
$\const{empty}$ and $1$ correspond to the initial and final objects of $\cat{E}$, the product types
correspond to products in $\cat{E}$, and arrow types correspond to exponential objects in $\cat{E}$.

Before moving on, we make two notes. First, there is a common convention to reduce parenthetical clutter: a sequence of
arrow-types is parsed by ``right associativity''.\index{associativity!right} That is, $\tau_1\to(\tau_2\to\tau_3)$ can be
written simply as $\tau_1\to\tau_2\to\tau_3$. Second, we mentioned that one can declare some atomic types. We will later see that one can declare atomic terms and atomic predicates as well. All of these declarations make up what is called a \emph{signature}. We will present a specific signature for temporal type theory in \cref{sec:axiomatics}.\index{signature}

\paragraph{Variables and contexts.}

One assumes that an infinite set of symbols $V$ has been fixed in advance for use as variables. Then saying ``$a$
is a variable'' just means $a\in V$. We also assume that $V$ is disjoint from all other sets of
symbols we use, to avoid ambiguity.\index{variable}

On its own, a variable $a\in V$ has no type. For expressions involving $a$ to be meaningful, a type
for $a$ must first be declared. As mentioned above, a context is a list of distinct variables, each with a
type declaration\index{context}. For example, if $\tau_1$ and $\tau_2$ are types and $x$, $y$, and $z$ are
variables then $x:\tau_1,\,y:\tau_2,\,z:\tau_2$ is a context. Thus every expression will 
be considered relative to some (explicit or implicit) context, which defines the type of each variable used
in the expression. To make contexts explicit, we use the standard notation
$x:\tau_1,\,y:\tau_2\vdash e$ to state that the expression $e$ is being considered in the context
$x:\tau_1,\,y:\tau_2$. While contexts are technically lists, type theories such as ours have rules
that ensure that order does not matter.

A context, such as $x_1:\tau_1,\ldots,x_n:\tau_n$, is often denoted $\Gamma$; if $\Gamma'$ is
another context, we may write $\Gamma,\Gamma'$ to denote the union of these contexts.

\paragraph{Terms.}\index{term}

Terms are certain expressions which represent ``elements'' of types. A term $t$ may contain
variables as in \cref{eqn:sample_term}, in which case $t$ represents a parameterized family of elements. We require that
every term $t$ be explicitly paired with a context declaring the types of all free variables
occuring in $t$, and we also require that every term itself has a unique type. We write
$\Gamma\vdash t:\tau$ to record that the term $t$ has type $\tau$ in context $\Gamma$.

As part of a signature\index{signature}\index{term!atomic}, one may declare a finite number of atomic terms $c:\tau$, each having a specified type, in this case $c$ has type
$\tau$. Each declared atomic term is valid in the empty context, e.g.\ $\vdash c:\tau$ would be
judged ``valid''. A variable of a given type is also a term, but it is only valid in a context
containing it. For example, $a:\tNN, b:\const{T}\vdash a:\tNN$ is judged valid, assuming
$a$ and $b$ are variables.

For the base types $\Prop$, $\tNN$, $\const{empty}$, and $1$, as well as for each type constructor,
there are associated term constructors. We will hold off discussing $\Prop$ until \cref{sec:HOL},
because it is the most complex, in order to focus on the main idea here.\index{constructor!type}\index{constructor!term}

The type $\tNN$ has three term constructors:\label{page:Nats}
\begin{itemize}
  \item $\vdash 0:\tNN$,
  \item $a:\tNN\;\vdash s(a):\tNN$, \;and 
  \item $e:\tau,\;f:\tau\to\tau,\;g:\tNN\;\vdash\const{rec}(e,f,g):\tau$,
\end{itemize}
where the last exists for any type $\tau$. Similarly,
\begin{itemize}
  \item product types have term constructors called \emph{tupling} and \emph{projecting};
  \item sum types have term constructors called \emph{co-tupling} and \emph{co-projecting}; and
  \item arrow types have term constructors called \emph{lambda abstraction} and \emph{application}.
\end{itemize}
All of these are standard,%
\footnote{Note that what we call sum types are called ``coproduct types'' in \cite[Section 2.3]{Jacobs:1999a}.}
but we discuss the last of these to give more of the flavor.

We begin with the term constructor for evaluation, because it is straightforward:\index{evaluation}
\[
  f:\tau_1\to\tau_2,\;a:\tau_1\;\vdash \const{ap}(f,a):\tau_2.
\]
We often denote $\const{ap}(f,a)$ simply by $f(a)$ or $fa$.

\index{lambda abstraction}The lambda abstraction term constructor is a bit more subtle. The corresponding mathematical idiom
is something like this: ``Suppose $a$ is an integer. Using $a$ we can define a rational $b$ as
follows $\ldots$. This defines a function $\tZZ\to\tQQ$, given by $a\mapsto b$.'' For any context
$\Gamma$, types $\tau_1,\tau_2$, and term $\Gamma,\;a:\tau_1\vdash b:\tau_2$, where obviously $b$ is
allowed to use variables from the context, there is a term $\Gamma\vdash\lambda(a:\tau_1)\ldotp
b:(\tau_1\to\tau_2)$.
In case the notation is not clear, this says $\lambda(a:\tau_1)\ldotp b$ has arrow type
$\tau_1\to\tau_2$.\index{type!arrow}

In the expression $\lambda(a:\tau_1)\ldotp b$, the variable $a$ has been \emph{bound}\index{variable!free/bound}, meaning it is
no longer free. This has two consequences. The first is that the variable $a$ can be replaced by any
other variable without changing the term; this is called $\alpha$-conversion and we mention it
briefly when we discuss conversion rules; see \cref{eqn:example_conversions}. The second is that one
can no longer substitute an arbitrary expression for $a$; we discuss substitution next.\index{conversion!$\alpha$}

\paragraph{Substitution.}\index{substitution}
As always, a variable of type $\tau$ represents an indeterminate value of type $\tau$. Thus, if
$x:\tau\vdash e:\tau'$ is a term---so $e$ may contain the variable $x$---then we are free to replace
$x$ throughout $e$ with any concrete term of type $\tau$.

In this sense, terms can be substituted for free variables. Suppose $a:X$ is a variable of some type $X$ and
$\Gamma\vdash e:X$ is a term of the same type. Then $[a\coloneqq e]$ denotes a sort of
search-and-replace function, sending terms to terms. That is, for any term of the form
$\Gamma',a:X\vdash f:\tau$, we can search for free occurrences of $a$ in $f$ and replace them with
$e$, and the result is the term $\Gamma',\Gamma\vdash [a\coloneqq e]f:\tau$.

As the $[a\coloneqq e]$ notation can get unwieldy, we introduce a convenient shorthand. If
$\Gamma,x:\tau\vdash e:\tau'$ is a term, then we can write $\Gamma,x:\tau\vdash e(x):\tau'$ to draw
attention to the fact that the variable $x$ occurs in $e$ (even though other variables from $\Gamma$
might also occur in $e$). Then instead of $[x\coloneqq t]e$, we can more simply write $e(t)$.

\paragraph{Conversion rules.}\label{page:conversion_rules}\index{conversion}\index{conversion!universal properity as}
The term constructors are meant to express universal properties about our various type constructors,
and we want to force these to be true in the $\cat{E}$-semantics. So far, we have constructed the
terms---which correspond to morphisms in $\cat{E}$---but we have not specified anything that says
the appropriate diagrams commute. This is the job of conversion rules.

A conversion rule is a rule of the form
\[
  \Gamma\vdash e_1\equiv e_2:\tau,
\]
where $\Gamma\vdash e_1:\tau$ and $\Gamma\vdash e_2:\tau$ are terms. There are conversion rules for
each type constructor, but again they are standard as well as straightforward, so we only write out
a few, namely those for natural numbers and arrow types. We also leave out the context and the
typing to clarify the idea:
\begin{equation}\label{eqn:example_conversions}
\begin{aligned}
  \const{rec}(e,f,0)&\equiv e&\text{base case}\\
  \const{rec}(e,f,s(g))&\equiv f(\const{rec}(e,f,g))&\hspace{1in}\text{recursive step}\\
  (\lambda(a:\tau_1)\ldotp f(a))(e)&\equiv f(e)&\text{$\beta$-reduction}\\
  \lambda(a:\tau_1)\ldotp f(a) &\equiv f&\text{$\eta$-conversion}\\
  \lambda(a:\tau_1)\ldotp f(a)&\equiv\lambda(b:\tau_1)\ldotp f(b)&\text{$\alpha$-conversion}	
\end{aligned}
\end{equation}
If we want to be a bit more pedantic and add back in the context and type, the first above is
actually
\[
  e:\tau,\;f:\tau\to\tau,\;g:\tNN\vdash\const{rec}(e,f,0)\equiv e:\tau,
\]
etc. Hopefully the necessary context in each of the cases from \cref{eqn:example_conversions} is
clear enough.

\begin{example}[Arithmetic of natural numbers]\label{ex:addition_of_naturals}\index{arithmetic!of natural numbers}
  Suppose we want to construct a term $m:\tNN,n:\tNN\vdash \const{plus}(m,n):\tNN$ for addition of
  natural numbers. Define it to be
  \begin{equation}\label{eqn:addition_of_nats}
    m:\tNN,n:\tNN\vdash\const{rec}(m,\lambda(x:\tNN)\ldotp s(x),n):\tNN,
  \end{equation}
  where $s$ is the successor. We use the usual infix notation $m+n$ as shorthand for this term.
  In fact, we could have written \cref{eqn:addition_of_nats} before discussing conversion rules, but now that we have conversion rule we check that our definition
  actually does what we want: $m+0\equiv m$ and $m+s(n)\equiv s(m+n)$.
  Multiplication of natural numbers is given by
  \[
    m:\tNN,n:\tNN\vdash\const{rec}(0,\lambda(p:\tNN)\ldotp m+p, n).
  \]

  We can also define a function $+\colon(\tNN\times\tNN)\to\tNN$ by lambda abstraction:
  $+(m,n)\coloneqq\lambda((m,n):\tNN\times\tNN)\ldotp\const{plus}(m,n)$, and similarly for
  multiplication.
\end{example}

\paragraph{Propositions, premises, and truth judgments.}\index{proposition}\index{premise}\index{truth judgment}
From the perspective of specifying a type theory and logic, one has---in addition to types, terms,
etc.---things called \emph{propositions} $P$, which have a status similar to types and terms. While
a type is something that intuitively has ``elements'', and the type theory provides rules for
constructing terms of that type, a proposition is something that intuitively has a ``truth value'',
and the type theory provides rules for constructing proofs of that proposition.

We are presenting a higher-order logic, which means there is a special type $\Prop$, and
propositions are identified with terms of type $\Prop$. We will describe how to build new
propositions from old in \cref{sec:HOL}.

Before getting to that, we will quickly explain premises and truth judgments. A \emph{premise} is
like a context: it is a finite list of propositions. That is, if $P_1,\ldots,P_n$ are valid
propositions, then $\Theta=\{P_1,\ldots,P_n\}$ is a premise. Finally a \emph{truth judgment} consists of a
context $\Gamma$, a premise $\Theta$, and a proposition $Q$; it looks like this:\index{context!premise}
\begin{equation}\label{eqn:truth_judgment}
  \Gamma\mid\Theta\vdash Q.
\end{equation}
\Cref{eqn:truth_judgment} basically says that, in the context $\Gamma$ (where all the variables used in $P_1,\ldots,P_n$
and $Q$ are defined), the conjunction of $P_1$ through $P_n$ is enough to conclude $Q$. But we will
explain this more in the next section, where we discuss how one decides which truth judgments are
valid.
\index{lambda calculus, simply typed|)}

\subsection{Higher order logic}\label{sec:HOL}\index{logic!higher order}

Logic is a set of rules about propositions and predicates, by which we can decide validity. We
discussed the type $\Prop$ of propositions above; a \emph{predicate} is just a term of type $\Prop$,
say $\Gamma\vdash P:\Prop$.\index{predicate}\index{predicate!atomic} We often refer to $P$
simply as a proposition, since it is a term of type $\Prop$, rather than as a predicate. In other
words, the fact that there is a context---possibly empty---is assumed, as usual.

Just like a type theory begins with a set of atomic types and a set of atomic terms, it also begins
with a set of atomic predicates, again of the form $\Gamma\vdash P:\Prop$. Together the atomic types, terms, and predicates make up a signature\index{signature}.

\subsubsection{Obtaining new propositions from old}\index{equality}

One important way to form a proposition is by equality. Namely, given a type $\tau$ and terms
$a:\tau$ and $b:\tau$, there is a proposition $(a=b):\Prop$. There are also the logical connectives:
given propositions $\Gamma\vdash P:\Prop$ and $\Gamma\vdash Q:\Prop$, we have
\begin{gather*}
  \top:\Prop\qquad
  \bot:\Prop\qquad
  P\wedge Q:\Prop\qquad
  P\vee Q:\Prop\qquad
  P\imp Q:\Prop
\end{gather*}
all in the context $\Gamma$.%
\footnote{The propositions $\neg P$ and $P\iff Q$ are just shorthands for $P\imp\bot$ and $(P\imp Q)\wedge(Q\imp P)$.}
These are pronounced \emph{true, false, $P$ and $Q$, $P$ or $Q$}, and \emph{$P$ implies $Q$}, respectively. Given a predicate $\Gamma,x:\tau\vdash P:\Prop$, one also has the quantifiers
\[
  \Gamma\vdash\forall(x:\tau)\ldotp P:\Prop \qquad
  \Gamma\vdash\exists(x:\tau)\ldotp P:\Prop
\]
pronounced \emph{for all $x$ of type $\tau$, $P$} and \emph{there exists $x$ of type $\tau$ such that $P$}. Note that, just like lambda abstraction binds a variable in a term, the quantifiers bind a variable
in a proposition.

Before moving on, we note a common convention to reduce parenthetical clutter: a sequence of
implications is parsed by ``right associativity''. That is, $P\imp(Q\imp R)$ can be written simply
as $P\imp Q\imp R$. Once we have rules for deduction, we will be able to deduce that $P\imp Q\imp R$
is equivalent to $(P\wedge Q)\imp R$.\index{associativity!right}

\subsubsection{Valid truth judgments}\index{truth judgment|(}

To do logic, one begins with a set of axioms\index{axiom}, and proceeds to judge which propositions are true. The
technical form of this is the truth judgment \cref{eqn:truth_judgment}. In this section we give the
rules for determining which truth judgments are valid in a given premise context $\Gamma\mid\Theta$,
given a set $\cat{A}$ of axioms. Validity of a truth judgment is inductively \emph{derived} using
certain allowable steps, called \emph{proof rules}. We divide the proof rules into four main
subgroups: premises and axioms, equalities, the connectives, and the quantifiers.

\paragraph{Premises and axioms.}

An axiom is a truth judgment, i.e.\ a statement of the form $\Gamma\mid\Theta\vdash P$. Axioms are
automatically valid as truth judgments. Similarly, $\top$ is always valid. If $P:\Prop$ is in the
list of propositions defining a premise $\Theta$ then $\Gamma\mid\Theta\vdash P$ is valid. Technically, these
three rules could be written as follows:
\begin{equation}\label{eqn:premise_axiom}
  \begin{prooftree}
    \Hypo{\Gamma\vdash\Theta\Prem}
    \Infer1{\Gamma\mid\Theta\vdash\top}
  \end{prooftree}
\hspace{.5in}
  \begin{prooftree}
    \Hypo{\Gamma\vdash\Theta\Prem}
    \Infer[left label={\side{P\in\Theta}}]1{\Gamma\mid\Theta\vdash P}
  \end{prooftree}
\hspace{.5in}
    \begin{prooftree}
      \Infer[left label={\side{(\Gamma\mid\Theta\vdash P)\in\cat{A}}}]0{\Gamma\mid\Theta\vdash P}
    \end{prooftree}
\end{equation}
The first says that if $\Theta$ is a valid premise in context $\Gamma$ (i.e.\ if all the symbols
used in the propositions in $\Theta$ are defined in $\Gamma$), then $\Gamma\mid\Theta\vdash\top$ is
valid. The second is similar, but adds a side condition that $P$ is one of the propositions in
$\Theta$, in which case $\Gamma\mid\Theta\vdash P$ is valid. The third says that axioms are
automatically valid.

As expected, one can prove by induction that if one increases the premise by adding a new
proposition, the set of valid truths one can derive does not decrease. This is called
\emph{weakening}. Thus if $\Gamma\vdash Q:\Prop$ is a term and
$\Gamma\mid\Theta\vdash Q$ is a valid truth judgment, then so is $\Gamma\mid\Theta,P\vdash Q$ for
any proposition $P$ in the same context, $\Gamma\vdash P:\Prop$.

\paragraph{Equalities.}

Recall the notion of conversion rules from \cref{sec:notions_lambda}, page
\pageref{page:conversion_rules}. These are of the form $\Gamma\vdash e_1\equiv e_2:\tau$. Every such
conversion rule gives rise to a valid truth judgment $\Gamma\vdash e_1=e_2$. Written more
technically,
\[
  \begin{prooftree}
    \Hypo{\Gamma\vdash\Theta\Prem}
    \Hypo{\Gamma\vdash e_1\equiv e_2:\tau}
    \Infer2{\Gamma\mid\Theta\vdash e_1=e_2}
  \end{prooftree}
\]
New valid truth judgments of the form $\Gamma\vdash e_1=e_2$ can be generated from old by
reflexivity, symmetry, and transitivity, as well as by substitution. For example, we also have
\[
  \begin{prooftree}
    \Hypo{\Gamma\mid\Theta\vdash e_1=e_2}
    \Hypo{\Gamma\mid\Theta\vdash e_2=e_3}
    \Infer2{\Gamma\mid\Theta\vdash e_1=e_3}
  \end{prooftree}
\]

\paragraph{The connectives.}

Most of the valid truth judgments coming from connectives and quantifiers are fairly obvious, but
there are differences between classical logic and constructive logic, so we should be a bit careful.
This difference only makes an appearance in disjunction $\vee$ and existential quantification
$\exists$. Instead of starting there, we begin with a familiar case, namely conjunction. The rules
for conjunction are
\[
  \begin{prooftree}
    \Hypo{\Gamma\mid\Theta\vdash\phi_1}
    \Hypo{\Gamma\mid\Theta\vdash\phi_2}
    \Infer2{\Gamma\mid\Theta\vdash\phi_1\wedge\phi_2}
  \end{prooftree}
  \hspace{.5in}
  \begin{prooftree}
    \Hypo{\Gamma\mid\Theta\vdash\phi_1\wedge\phi_2}
    \Infer1{\Gamma\mid\Theta\vdash\phi_1}
  \end{prooftree}
  \hspace{.5in}
  \begin{prooftree}
    \Hypo{\Gamma\mid\Theta\vdash\phi_1\wedge\phi_2}
    \Infer1{\Gamma\mid\Theta\vdash\phi_2}
  \end{prooftree}
\]
These correspond to mathematical idioms such as ``We have $\phi_1$ and $\phi_2$, so in particular we
have $\phi_1$.''

The rules for implication are:
\[
  \begin{prooftree}
    \Hypo{\Gamma\mid\Theta,\phi_1\vdash\phi_2}
    \Infer1{\Gamma\mid\Theta\vdash\phi_1\imp\phi_2}
  \end{prooftree}
  \hspace{.5in}
  \begin{prooftree}
    \Hypo{\Gamma\mid\Theta\vdash\phi_1\imp\phi_2}
    \Hypo{\Gamma\mid\Theta\vdash\phi_1}
    \Infer2{\Gamma\mid\Theta\vdash\phi_2}
  \end{prooftree}
\]
The first corresponds to the mathematical idiom ``To prove $\phi_1\imp\phi_2$, we assume $\phi_1$
and attempt to show $\phi_2$,'' while the second corresponds to the idiom ``We know that $\phi_1$
implies $\phi_2$, so since $\phi_1$ holds, so does $\phi_2$.''

The constant $\bot$, representing falsehood, may seem at first sight rather useless, but it plays an
important role in constructive logic. Its only rule is the following, which says intuitively ``if we
can prove false, then we can prove anything'':
\[
  \begin{prooftree}
    \Hypo{\Gamma\vdash\phi:\Prop}
    \Hypo{\Gamma\mid\Theta\vdash\bot}
    \Infer2{\Gamma\mid\Theta\vdash\phi}
  \end{prooftree}
\]

In constructive logic, negation is not a primitive connective. Rather, $\neg P$ is simply shorthand
for $P\imp\bot$. Thus $\neg P$ means precisely ``if $P$ were true, we could prove false.''
So trivially, we have that $P\imp\bot$ implies $\neg P$, but it is emphatically \emph{not} the case
that $\neg P\imp \bot$ implies $P$. That is, in constructive logic there is no
``proof-by-contradiction'': just because assuming $P$ is false leads to a contradiction, that does
not provide a proof that $P$ is true. Instead, all we can say is that $\neg P\imp \bot$ implies
$\neg\neg P$. As an exercise, the reader might try proving that $\neg\neg\neg P$ implies $\neg P$.

There are three rules for disjunction. The first corresponds to the idiom ``We know that either
$\phi_1$ or $\phi_2$ holds. Either way we can prove $\phi_3$, so $\phi_3$ holds.''
\[
	\begin{prooftree}
    \Hypo{\Gamma\mid\Theta\vdash\phi_1\vee\phi_2}
    \Hypo{\Gamma\mid\Theta,\phi_1\vdash\phi_3}
    \Hypo{\Gamma\mid\Theta,\phi_2\vdash\phi_3}
    \Infer3{\Gamma\mid\Theta\vdash\phi_3}
  \end{prooftree}
\]
The next two are more obvious, but note that---other than via axioms---there is no way to derive
$\phi\vee\neg\phi$. The only way to derive a disjunction is to derive one of the disjuncts.
Similarly, given $\phi_1\imp\phi_2$, it does not necessarily follow that $\neg\phi_1\vee\phi_2$.
\[
  \begin{prooftree}
    \Hypo{\Gamma\vdash\phi_2:\Prop}
    \Hypo{\Gamma\mid\Theta\vdash\phi_1}
    \Infer2{\Gamma\mid\Theta\vdash\phi_1\vee\phi_2}
  \end{prooftree}
\hspace{.5in}
  \begin{prooftree}
    \Hypo{\Gamma\vdash\phi_1\Prop}
    \Hypo{\Gamma\mid\Theta\vdash\phi_2}
    \Infer2{\Gamma\mid\Theta\vdash\phi_1\vee\phi_2}
  \end{prooftree}
\]

\paragraph{The quantifiers.}

Like lambda-abstraction, the quantifiers change the context by binding variables. Here are the proof
rules for the universal quantifier:
\[
	\begin{prooftree}
    \Hypo{\Gamma,x:\tau\mid\Theta\vdash\phi}
    \Infer[left label={\side{x\notin\fv(\Theta)}}]1{\Gamma\mid\Theta\vdash\forall(x:\tau)\ldotp\phi}
  \end{prooftree}
\hspace{.5in}
  \begin{prooftree}
    \Hypo{\Gamma\mid\Theta\vdash\forall(x:\tau)\ldotp\phi(x)}
    \Hypo{\Gamma\vdash e:\tau}
    \Infer2{\Gamma\mid\Theta\vdash\phi(e)}
  \end{prooftree}
\]
The first rule corresponds to the mathematical idiom, ``To prove $\forall(x:\tau)\ldotp\phi$, it
suffices to take an arbitrary element $x$ of type $\tau$ and prove $\phi$.'' The word ``arbitrary''
signifies that $x$ is not already being used in the assumptions $\Theta$, which is formalized by
saying $x$ is not in the set $\fv(\Theta)$ of free variables occurring in $\Theta$. The second rule
corresponds to the idiom, ``Since $\phi$ holds for all $x$ of type $\tau$, it holds in particular
for $e$.''

Finally, we have the existential quantifier. The first proof rule corresponds to the idiom, ``To
prove that $\exists(x:\tau)\ldotp\phi$ holds, it suffices to find a witness $e$ such that $\phi(e)$ holds.''
\[
  \begin{prooftree}
    \Hypo{\Gamma,x:\tau\vdash\phi(x):\Prop}
    \Hypo{\Gamma\vdash e:\tau}
    \Hypo{\Gamma\mid\Theta\vdash\phi(e)}
    \Infer3{\Gamma\mid\Theta\vdash\exists(x:\tau)\ldotp\phi(x)}
  \end{prooftree}
\]
The second proof rule corresponds to the idiom, ``We know that there exists an $x$ for which
$\phi_1$ holds, and we know that for an arbitrary $x$, if $\phi_1$ holds then so does $\phi_2$, so
it follows that $\phi_2$ also holds.''
\[
  \begin{prooftree}
    \Hypo{\Gamma\mid\Theta\vdash\exists(x:\tau)\ldotp\phi_1}
    \Hypo{\Gamma,x:\tau\mid\Theta,\phi_1\vdash\phi_2}
    \Infer[left label={\side{ x\notin\fv(\Theta,\phi_2)}}]2{\Gamma\mid\Theta\vdash\phi_2}
  \end{prooftree}
\]

\begin{remark}\label{rem:HOL_redundancy}
  Higher order logic is extremely expressive, and the logic we have presented has considerable
  redundancy. In fact, with the higher order proposition type $\Prop$, only implication and
  universal quantification are needed. That is, we have the following equivalences:
  \begin{align*}
    \bot &\quad\iff\quad \forall(\alpha:\Prop)\ldotp \alpha \\
    \top &\quad\iff\quad \bot \imp \bot \\
    \phi \vee \psi &\quad\iff\quad
      \forall(\alpha:\Prop)\ldotp (\phi\imp\alpha) \imp (\psi\imp\alpha) \imp \alpha \\
    \phi \wedge \psi &\quad\iff\quad
      \forall(\alpha:\Prop)\ldotp (\phi \imp \psi \imp \alpha) \imp \alpha \\
    \exists(x:\tau)\ldotp \phi &\quad\iff\quad
      \forall(\alpha:\Prop)\ldotp (\forall(x:\tau)\ldotp \phi \imp \alpha) \imp \alpha.
  \end{align*}
  Even equality can be defined in this way. Given terms $e_1,e_2:\tau$,
  \[
    e_1 = e_2 \quad\iff\quad \forall(P:\tau\to\Prop)\ldotp P(e_1) \imp P(e_2).
  \]
  We mention this to give the reader an idea of some of the unexpected consequences of the type
  $\Prop$. However, the rules stated in \cref{sec:HOL}, before this remark, are easier to work with and---unlike the above higher-order reformulations---trivially generalize to the
  non-higher-order setting.
\end{remark}
\index{truth judgment|)}

\subsubsection{The extensionality axioms}\index{propositional extensionality|see {extensionality, propositional}}\index{extensionality!propositional}\index{function extensionality|see {extensionality, function}}\index{extensionality!function}
Propositional extensionality says that equivalent propositions are automatically equal:
$\forall(P,Q:\Prop)\ldotp(P\iff Q)\imp(P=Q)$. This is an axiom that we will assume throughout the
book. It is sound in any Grothendieck topos because propositions correspond externally to closed
sieves (see e.g.\ \cite[Prop III.3.3]{MacLane.Moerdijk:1992a}, which form a partial order, not just
a preorder). Moreover, propositional extensionality is assumed in proof assistants like Coq and Lean.

Function extensionality, like propositional extensionality, is sound in any topos and commonly assumed in proof assistants. It says that two functions $X\to Y$ are equal if and only if their values are equal on all $x:X$.

\begin{axiom}[Extensionality axioms]~
\label{ax:propositional_extensionality}\index{axiom!extensionality}
\index{extensionality!propositional}\index{extensionality!function}
\begin{description}
\item[Propositional extensionality]
  \[
    P:\Prop,Q:\Prop\mid P\iff Q\vdash P=Q
  \]
\item[Function extensionality]
 Let $X$ and $Y$ be types. Then
	\[
		f:X\to Y, g:X\to Y\mid\forall(x:X)\ldotp fx=gx\vdash f=g
	\]
\end{description}
\end{axiom}

Using \cref{eqn:premise_axiom}, the extensionality axioms can be used in derivations of valid truth judgments. The
specific type signature we give in \cref{sec:axiomatics} will contain ten more axioms, all of which
fit into the logic at the same point as this one. Again, the reason we do not include propositional
extensionality with the other ten is that it is sound in any topos, and is not specific to ours.

Often we will write axioms, propositions, definitions etc.\ more informally, e.g.\ replacing the formal syntax of the statements in \cref{ax:propositional_extensionality} with something like the following.
 \begin{description}
	\item[Propositional extensionality] Suppose that $P$ and $Q$ are propositions, and assume that $P\iff Q$ holds. Then $P=Q$.
\item[Function extensionality] Suppose that $X$ and $Y$ are types and that $f,g:X\to Y$ are functions. If $\forall(x:X)\ldotp fx=gx$ holds, then $f=g$. 
\end{description}

\subsection{Subtypes and quotient types}\label{sec:subtypes_quotients}

Here we add two new type constructors: subtypes and quotient types. We could not have covered these earlier, 
when discussing the other type constructors, because these two depend on the logical layer.

\paragraph{Subtypes.}\index{type!subtype}\index{constructor!type}

Let $\tau$ be a type, and suppose given a predicate in the one-variable context, $x:\tau\vdash
\phi:\Prop$. The corresponding subtype is denoted $\{x:\tau\mid\phi\}$.%
\footnote{It is natural to want to
extend this idea to contexts with more than one variable, but that takes us into dependent types;
see \cref{sec:dependent_types}.}

There are two term constructors, two conversion rules, and a truth judgment for subtypes. The first
term constructor says that if $\Gamma\vdash e:\tau$ is a term and
$\Gamma\mid\varnothing\vdash\phi(e)$ is a valid truth judgment, then there is a new term
$\Gamma\vdash\textsf{i}(e):\{x:\tau\mid\phi\}$. The second term constructor says that if
$\Gamma\vdash e':\{x:\tau\mid\phi\}$ then there is a new term $\Gamma\vdash \textsf{o}(e'):\tau$.
Then conversion rules say $\textsf{o}(\textsf{i}(e))\equiv e$ and $\textsf{i}(\textsf{o}(e'))\equiv
e'$. However, in practice it is more convenient to simply drop the $\textsf{i}$'s and
$\textsf{o}$'s.

The truth judgment for subtypes is that if $\Gamma,x:\tau\mid\Theta(x),\phi(x)\vdash\psi(x)$ is a
valid truth judgment, then so is
$y:\{z:\tau\mid\phi(z)\}\mid\Theta(\textsf{o}(y))\vdash\psi(\textsf{o}(y))$. With the informal
notation drapping $\textsf{i}$'s and $\textsf{o}$'s, this can be simplified to
$y:\{z:\tau\mid\phi(z)\}\mid\Theta(y)\vdash\psi(y)$.

\begin{remark}\label{rem:HOL_sum_type}
  We saw in \cref{rem:HOL_redundancy} that all of the first-order logic connectives can be defined
  in terms of just $\imp$ and $\forall$, by making clever use of the higher-order proposition type
  $\Prop$. With the addition of subtypes, this surprising expressivity of higher-order logic extends
  to the type theory. For example, the rules for sum types are redundant. The sum of two types $A$
  and $B$ can be \emph{defined} as
  \begin{multline}\label{eq:HOL_sum_type}
    A+B \quad\coloneqq\quad
      \{ (\phi,\psi) : (A\to\Prop)\times(B\to\Prop) \mid \\
        (\mathrm{is\_sing}(\phi) \wedge \mathrm{is\_empty}(\psi)) \vee
        (\mathrm{is\_empty}(\phi) \wedge \mathrm{is\_sing}(\psi)) \}
  \end{multline}
  where
  \begin{align*}
    \mathrm{is\_sing}(\phi) &\coloneqq \exists(a:A)\ldotp \phi(a) \wedge
      \forall(a':A)\ldotp \phi(a') \imp a=a' \\
    \mathrm{is\_empty}(\phi) &\coloneqq \forall(a:A)\ldotp \neg\phi(a).
  \end{align*}
  This can be made precise in two ways. First, if we have explicit sum types in the type theory, then it is
  possible to construct an isomorphism between the ``real'' sum $A+B$ and the type defined in
  \eqref{eq:HOL_sum_type}. Second, even without explicit sum types, it is possible to prove that the
  type defined in \eqref{eq:HOL_sum_type} satisfies all of the rules that define a sum type.

  Again, it is more convenient to just use the explicit sum type rules. However, there are
  properties of sum types which hold in a higher-order logic---and which can be proven from
  \eqref{eq:HOL_sum_type}---which are not provable in non-higher-order logic. We will see an example
  of such a property in \cref{prop:const_decidable_eq}, where we show that the sum of two types with decidable equality has decidable equality.
\end{remark}

\paragraph{Quotient types.}\index{type!quotient}
Let $\tau$ be a type, and suppose given a predicate in the two-variable context, also known as a
binary relation, $x:\tau,y:\tau\vdash R(x,y):\Prop$\index{relation!binary}. The corresponding quotient type is denoted
$\tau/R$.%
\footnote{As with subtypes, it is natural to want to extend the above idea to contexts with more
than two variables of the same type, but doing so takes us into dependent types; see
\cref{sec:dependent_types}.}

There are two term constructors, two conversion rules, and a truth judgment for quotient types. The
first term constructor says that if $\Gamma\vdash e:\tau$ is a term, then there is a new term
$\Gamma\vdash[e]_R:\tau/R$. We read $[e]_R$ as ``the equivalence class of $e$.'' The second says
that if $\Gamma,z:\tau\vdash e'(z):\tau'$ is a term and $\Gamma,x:\tau,y:\tau\mid R(x,y)\vdash
e'(x)=e'(y)$ is a valid truth judgment, then there is a new term $\Gamma,a:\tau/R\vdash\textsf{pick
}x\textsf{ from }a\textsf{ in }e'(x):\tau'$. The intuition is that $a$ represents an equivalence
class, so we can pick some representative $x$ from $a$ and form the term $e'(x)$, with the guarantee
that it doesn't matter which representative we choose. The first conversion rule says
$\textsf{pick }x\textsf{ from }[e]_R\textsf{ in }e'(x)\equiv e'(e)$, and the second says
$\textsf{pick }x \textsf{ from }Q\textsf{ in }e'([x]_R)\equiv e'(Q)$.

The truth judgment for quotient types is that if $\Gamma\vdash e:\tau$ and $\Gamma\vdash e':\tau$
are terms, then $\Gamma\mid R(e,e')\vdash [e]_R=[e']_R$ is a valid truth judgment.

\begin{example}\label{ex:integers_rationals}\index{type!of integers}\index{type!of rationals}
  Recall the definition of addition and multiplication of natural numbers from
  \cref{ex:addition_of_naturals}. We can also define the inequality $\leq$ for natural numbers,
  $m:\tNN,n:\tNN\vdash\const{leq}(m,n)$, to be the term
  $\const{leq}(m,n)\coloneqq\exists(p:\tNN)\ldotp n+p=m$. Write $m\leq n$ if $\const{leq}(m,n)$.

  With quotient types in hand, we can define the types $\tZZ$ and $\tQQ$. For the former, let $R$ be
  the following equivalence relation on $\tNN\times\tNN$,
  \[
    (p,m):\tNN\times\tNN, (p',m'):\tNN\times\tNN\vdash p+m'=p'+m:\Prop
  \]
  Then $\tZZ$ is defined to be the type $(\tNN\times\tNN)/R$. One can construct addition,
  subtraction, and multiplication for $\tZZ$ using the term constructors for quotient types. For
  example, addition in $\tZZ$ is given by the term
  \[
    z:\tZZ, z':\tZZ\vdash\textsf{pick }(p,m)\textsf{ from }z\textsf{ in }(\textsf{pick
    }(p',m')\textsf{ from }z'\textsf{ in }[p+p',m+m']_R:\tZZ.
  \]
  To construct this term---as mentioned above---one must check that the terms $p+p'$ and $m+m'$ are well-defined, i.e.\
  independent of the choice of $(p,m)$ and $(p',m')$.

  Similarly, for $\tQQ$, let $\tNN_+$ denote the subset type $\tNN_+\coloneqq\{n:\tNN\mid n\geq
  1\}$. Then define $S$ be the following equivalence relation on $\tZZ\times\tNN_+$,
  \[
    (n,d):\tZZ\times\tNN_+, (n',d'):\tZZ\times\tNN_+\vdash d'*n=d*n':\Prop
  \]
  and let $\tQQ=(\tZZ\times\tNN_+)/S$. Again, one can construct addition, subtraction, and
  multiplication for $\tQQ$, as well as the partial reciprocal function, using the term constructors
  for quotient types.
\end{example}

\subsection{Dependent types}\label{sec:dependent_types}\index{type!dependent}

Consider the inequality $x:\tNN, y:\tNN \vdash x\leq y$. Then it is natural to form the subtype
$\tNN_{\leq y}\coloneqq\{x:\tNN \mid x\leq y\}$. However, this type \emph{depends} on a variable
$y$, which the type theory we have sketched so far is unable to handle. Intuitively, it is best to
think of the type $\{x:\tNN \mid x\leq y\}$ as a \emph{family} of types, parameterized by $y:\tNN$.
In particular, for each concrete term such as $4:\tNN$ (where $4$ is shorthand for $ssss0$; see page 
\pageref{page:Nats}), there is a type $\tNN_{\leq 4}\coloneqq\{x:\tNN \mid x\leq 4\}$ obtained by
substitution $[y\coloneqq 4]$.\index{type!family}

While such type families are an intuitively natural and useful concept, formalizing dependent type
theory (where types are allowed to depend on terms) is significantly more subtle than the simple
type theory we have so far described. As an illustration of the extra difficulty, consider the type
families $\{x:\tNN \mid x\leq y\}$ and $\{x:\tNN \mid x\leq (y+0)\}$. We know that $y\equiv y+0$ by
the definition of addition---i.e.\ we consider $y$ and $y+0$ to be \emph{the same} term---and as
such we should consider $\tNN_{\leq y}$ and $\tNN_{\leq(y+0)}$ to be \emph{the same} type, even
though they are syntactically different. Thus in a dependent type theory, one must extend the
conversion relation to types as well as terms.

In the theory presented in the rest of this book, there are a few occasions where we need to
consider subtypes and quotient types which are technically only well-formed in a dependent type
theory, such as $\tNN_{\leq y}$. We trust that it will be intuitively clear how to work with these
types, and we refer the reader interested in the details of dependent type theory to sources such as
\cite{Jacobs:1999a,Voevodsky:2013a}.

Recall the basic idea of the semantics of type theory in a topos $\cat{E}$: each type is
assigned an object of $\cat{E}$, each term is assigned a morphism, and propositions
$a:A\vdash\phi:\Prop$ correspond to subobjects of $A$, i.e.\ monomorphisms
$\{a:A\mid\phi\}\hookrightarrow A$. The semantics of a dependent type $a:A\vdash B$ is an arbitrary morphism $B\to A$. So one can think of dependent types as
generalizing subtypes: a proposition that is dependent on a context determines a monomorphism into the context, whereas a
type that is dependent on a context determines an arbitrary morphism into the context. For example, the dependent
type $y:\tNN\vdash\tNN_{\leq y}$---which we think of as a family of types parameterized by $y$---is represented by the composite
\[
  \{x:\tNN,y:\tNN\mid x\leq y\} \hookrightarrow \tNN\times\tNN \To{\tn{pr}_2} \tNN.
\]
A concrete member of the family, such as $\tNN_{\leq 4}$, is represented by a \emph{fiber} of this
map, e.g.\ $\tNN_{\leq 4}$ is represented by the pullback
\[
  \begin{tikzcd}
    \tNN_{\leq 4} \ar[r] \ar[dd] \ar[ddr,phantom,"\lrcorner"{very near start}]
    & \{x:\tNN,y:\tNN\mid x\leq y\} \ar[d] \\
    & \tNN\times\tNN \ar[d,"\tn{pr}_2"] \\
    1 \ar[r,"4"'] & \tNN
  \end{tikzcd}
\]
In general, a helpful intuition is that non-dependent types (with no free variables) are represented
by objects of $\cat{E}$, while dependent types in the context $(x_1:A_1,\dots,x_n:A_n)$ are
represented by objects of the slice category $\cat{E}/(A_1\times\cdots\times A_n)$.

\section{Modalities}
\label{sec:modalities}\index{modality}\index{topos!modality in|see {modality}}

Now that we have informally laid out the basics of the type theory we use in this book, we proceed to discuss some constructions one can do with it. In this section we discuss how modalities allow one to consider subtoposes within the type theory of the larger topos.

A modality is a term of type $\Prop\to\Prop$ satisfying a few axioms. Topos-theoretically, a modality is precisely what \cite{MacLane.Moerdijk:1992a} calls a
\emph{Lawvere-Tierney topology} and what \cite{Johnstone:2002a} calls a \emph{local operator}; we
use the term ``modality'' to emphasize the logical viewpoint.\index{local operator|see
{modality}}\index{topology!Lawvere-Tierney|see {modality}}

We define modalities in \cref{sec:modalities_in_general}, discuss various related notions---such as closed propositions and sheaves---in \cref{sec:closed_separated_sheaves}, and finally discuss the relationship between modalities and subtoposes in \cref{sec:modalitites_subtoposes}.

\subsection{Definition of modality}\label{sec:modalities_in_general}

In any topos, a modality is an internal monad\index{monad!on $\Prop$} on $\Prop$, considered as an
internal poset. The definition can be given type-theoretically as follows:

\begin{definition}\label{def:modality}
A \emph{modality} is a map $j:\Prop\to\Prop$ satisfying the following three conditions for any $P,Q:\Prop$,
\begin{itemize}
	\item $P\imp jP$,
	\item $jjP\imp jP$, and
	\item $(P\imp Q)\imp jP\imp jQ$.
\end{itemize}
\end{definition}

The third condition can be replaced by various equivalent statements, as we now show.

\begin{lemma}\label{lemma:modality}
  Suppose $j:\Prop\to\Prop$ satisfies $P\imp jP$ and $jjP\imp jP$ for any $P:\Prop$. Then the
  following are equivalent:
  \begin{enumerate}
    \item $(P\imp Q)\imp jP\imp jQ$,
    \item $jP\imp (P\imp jQ)\imp jQ$,
    \item $j(P\imp Q)\imp jP\imp jQ$.
    \item $(jP\wedge jQ)\iff j(P\wedge Q)$.
  \end{enumerate}
\end{lemma}
\begin{proof}
  $2\imp 1$ and $3\imp 1$ are easy. $1\imp 2$ follows from $jjQ\imp jQ$. We next show $1\imp 3$. But 3
  is equivalent to $jP\imp j(P\imp Q)\imp jQ$, so first assume $jP$. Assuming 1, we have $(P\imp
  Q)\imp jQ$. So if we apply 1 again on this last implication and assume $j(P\imp Q)$ then we have $jjQ$
  hence $jQ$.

  Next we show $3\imp 4$. It is easy to show
  $j(P\wedge Q)\imp jP$ and hence $j(P\wedge Q)\imp(jP\wedge jQ)$. We need to show the converse
  $(jP\wedge jQ)\imp j(P\wedge Q)$, which is equivalent to $jP\imp jQ\imp j(P\wedge Q)$. Clearly
  $P\imp (Q\imp P\wedge Q)$, so applying $j$ and using 3 with the assumption $jP$ and again with
  $jQ$, we obtain $j(P\wedge Q)$ as desired.

  Finally for $4\imp 1$, note that $(P\imp Q)$ is equivalent to $(P\wedge Q)\iff P$ and hence to
  $(P\wedge Q)=P$ by propositional extensionality, \cref{ax:propositional_extensionality}. Assuming $(P\wedge Q)=P$ and $jP$ we have
  $j(P\wedge Q)$ and hence $jQ$ by 4.
\end{proof}

\begin{example}\label{ex:modalities_order}\index{modality!order on}
  Modalities can be ordered by reverse implication, $j\leq j'$ iff $j'P\imp jP$ for
  all $P:\Prop$. The top element in this ordering is the identity modality $P\mapsto P$. The bottom
  element is the constant modality $P\mapsto\top$. Somewhere in between is the
  double-negation\index{modality!double negation} modality $P\mapsto\neg\neg P$. We will discuss the
  relationship between modalities and subtoposes in \cref{sec:modalitites_subtoposes}.
  \index{poset!of modalities|see{modalities, order on}}
\end{example}

\subsection{$j$-closed propositions, $j$-separated types, and $j$-sheaves}\label{sec:closed_separated_sheaves}

Given a modality $j$ and a type $X$, we can internally express the semantic notion that the sheaf
associated to $X$ is $j$-separated or a $j$-sheaf.

\begin{definition}
\label{def:j_closed}\index{proposition!$j$-closed}
  A proposition $P:\Prop$ is called \emph{$j$-closed} if it satisfies $jP\imp P$. A predicate
  $P:X\to\Prop$ is called \emph{$j$-closed} if $\forall(x:X)\ldotp jPx\imp Px$.
\end{definition}

For any modality $j$, the reflection sending a proposition to a $j$-closed proposition is given by $P\mapsto jP$. 

The following remark is straightforward to verify but \emph{very useful in practice}. We hope the reader takes careful note of it, because we will use it often. Reading a proof and wondering ``where did the $j$'s go?'' is a sign that the reader should revisit this remark. The goal is to understand how logic in the $j$-subtopos is constructed in the larger topos.

\begin{remark}[$j$-logic]\label{rem:j-logic}\index{logic!relative to a modality}
  Let $j$ be a modality. If $P:\Prop$ is written as a complex expression, $jP$ can often be simplified
  recursively. For example, we have the following equivalences and implications:
  \begin{equation}\label{eqn:j_logic}
  \begin{gathered}
      j\top\Leftrightarrow\top
    \hspace{.7in}
      j(\phi\wedge\psi)\Leftrightarrow j\phi\wedge j\psi
    \hspace{.7in}
      j(\psi\imp\phi)\imp(\psi\imp j\phi)
    \\
      j\forall(x:X)\ldotp\phi(x)\imp\forall(x:X)\ldotp j\phi(x)
    \hspace{.7in}
      \exists(x:X)\ldotp j\phi(x)\imp j\exists(x:X)\ldotp\phi(x)
  \end{gathered}
  \end{equation}
  Now assume $\phi$ is $j$-closed. We can replace two of the implications by equivalences:
  \begin{gather*}
      j(\psi\imp\phi)\Leftrightarrow(\psi\imp\phi)
    \hspace{.7in}
      j\forall(x:X)\ldotp\phi(x)\Leftrightarrow\forall(x:X)\ldotp\phi(x)
  \end{gather*}
  and, quite usefully, when proving $\phi$ we can drop $j$ from the front of all hypotheses:
  \[
    [(\psi_1\wedge\cdots\wedge\psi_n)\imp\phi]\iff[(j\psi_1\wedge\cdots\wedge j\psi_n)\imp\phi].
  \]
\end{remark}

Here is an example of a proof that uses the above remark. We go through it slowly.

\begin{proposition}\label{prop:dec_eq_disjunction}
Suppose $P:\Prop$ is decidable, i.e.\ $P\vee\neg P$ holds. For any $Q:\Prop$ we have $j(P\vee Q)\iff P\vee jQ$.
\end{proposition}
\begin{proof}
For one direction we suppose $j(P\vee Q)$ and prove $P\vee jQ$. If $P$ holds we are done, so suppose $\neg P$. We will prove $jQ$, so by \cref{rem:j-logic} we can drop $j$ from the hypothesis, at which point we have $P\vee Q$ and thus $Q$.

The other direction is easier and does not depend on $P$ being decidable. In case $P$, we have $P\vee Q$ so $j(P\vee Q)$. In case $jQ$ we can drop the $j$, and $Q$ implies $P\vee Q$ and hence $j(P\vee Q$).
\end{proof}

\begin{definition}[$j$-separated, $j$-sheaf]\label{def:separated_sheaf}
\index{type!$j$-separated}\index{type!$j$-sheaf}\index{sheaf!for a modality $j$|see {type, $j$-sheaf}}
  A type $X$ is \emph{$j$-separated} iff it satisfies $\forall(x,x':X)\ldotp j(x=x')\imp(x=x')$.
  The \emph{$j$-separification of $X$}\index{separification}, denoted $\mathrm{sep}_jX$ is the
  quotient of $X$ by the internal equivalence relation $j(x=x')$. The \emph{$j$-sheafification of
  $X$},\index{sheafification} denoted $\asSh_j(X)$, is the subtype of predicates $\phi:X\to\Prop$
  satisfying
  \begin{description}
    \item[\quad $j$-closed predicate:] $\forall(x:X)\ldotp j\phi(x)\imp\phi(x)$,
    \item[\quad $j$-local singleton:] $j\exists(x:X)\ldotp\forall(x':X)\ldotp \phi(x')\iff j(x=x').$
  \end{description}\index{singleton}
  If $X$ is separated, one may drop the inner $j$ from the $j$-local singleton condition. If $\{\cdot\}:X\to(X\to\Prop)$ is the usual singleton function, then the $j$-local singleton condition is equivalent to $j\exists(x:X)\ldotp j(\phi=\{x\})$.

  There is a map $\eta_j\colon X\to \asSh_j(X)$; it sends $x:X$ to the predicate
  $\eta_j(x):X\to\Prop$ given on $x':X$ by
  \begin{equation}\label{eqn:sheafification_inclusion}
    \eta_j(x)(x')\coloneqq j(x=x').
  \end{equation}
  It is easy to check that $\eta_j(x)$ is $j$-closed and $j$-locally singleton, and $X$ is
  $j$-separated iff $\eta_j$ is an internal injection. $X$ is called a \emph{$j$-sheaf} if $\eta_j$
  is also an internal surjection, i.e.\ satisfies
  $\forall(\phi:\asSh_j(X))\ldotp\exists(x:X)\ldotp\phi=\eta_j(x)$.
\end{definition}

\begin{proposition}\label{prop:characterize_j_sheaf}
  A type $X$ is a $j$-sheaf iff it is $j$-separated and, for any $j$-closed predicate
  $\phi:X\to\Prop$ the singleton condition on $\phi$ is $j$-closed in the sense that
  \begin{equation}\label{eqn:sheaf2}
    j\big(\exists(x:X)\ldotp\forall(x':X)\ldotp\phi(x') \iff (x=x')\big)
      \imp \exists(x:X)\ldotp\forall(x':X)\ldotp\phi(x') \iff (x=x').
  \end{equation}
\end{proposition}
\begin{proof}
  By \cref{def:separated_sheaf}, we may assume $X$ is $j$-separated, and we need to show that
  $\eta_j$ is surjective iff \cref{eqn:sheaf2} holds for every $j$-closed predicate $\phi$. First
  suppose $\eta_j$ is surjective and that $\phi$ is a $j$-closed predicate that satisfies the
  hypothesis of implication \eqref{eqn:sheaf2}. Then by definition we have $\phi:\asSh_j(X)$, so
  $\exists(x:X)\ldotp\phi=\eta_j(x)$, which is the conclusion of \eqref{eqn:sheaf2}. For the other
  direction, if every $j$-closed predicate $\phi:X\to\Prop$ satisfies \eqref{eqn:sheaf2} then
  $\eta_j$ is surjective.
\end{proof}

\begin{proposition}\label{prop:prop_j-j-sheaf}
  For any modality $j$, the type $\Prop_j\coloneqq\{P:\Prop\mid jP\imp P\}$ is a
  $j$-sheaf.\index{type!$j$-sheaf}
\end{proposition}
\begin{proof}
  It is easy to check that $\Prop_j$ is $j$-separated. So choose a $j$-closed predicate
  $\phi:\Prop_j\to\Prop$ and assume the hypothesis of \eqref{eqn:sheaf2}. Let
  $P\coloneqq\forall(Q:\Prop_j)\ldotp\phi(Q)\imp Q$, which is $j$-closed by \cref{rem:j-logic}. It
  is our candidate for the required existential; i.e.\ using \cref{prop:characterize_j_sheaf} we
  will be done if we can show $\forall(P':\Prop_j)\ldotp \phi(P')\iff (P=P')$. One checks using
  \cref{rem:j-logic} that $P$ is $j$-closed. Thus we can drop the $j$ from our hypothesis, so we
  have some $x:\Prop_j$ satisfying $\forall(x':\Prop_j)\ldotp\phi(x')\iff (x=x')$.

  First we show $\phi(P)$. It is easy to show $P\imp x$, because we have $\phi(x)$ by hypothesis.
  But we also have $x\imp P$ because for any $Q:\Prop_j$ such that $\phi(Q)$, we have $Q=x$. Thus
  $x=P$, so $\phi(P)$, as desired. Choosing $P':\Prop_j$, it remains to show that
  $\phi(P')\imp(P=P')$. But given $\phi(P')$ and the already-established $\phi(P)$, we have $P=x=P'$
  by hypothesis, so we are done.
\end{proof}

\begin{proposition}\label{prop:subtype_j_sheaf}\index{sheaf!$j$}
  Given a $j$-sheaf $X$ and a $j$-closed predicate $\psi:X\to\Prop_j$, the subtype $\{x:X\mid\psi
  x\}$ is a $j$-sheaf.
\end{proposition}
\begin{proof}
  Any subtype of a separated type is separated, so it suffices to prove \eqref{eqn:sheaf2} for any
  $\phi:\{x:X\mid\psi x\}\to\Prop_j$. The hypothesis says $j\exists(x:X)\ldotp\psi
  x\wedge\forall(x':X)\ldotp\psi x'\imp(\phi x\iff (x=x'))$. It follows easily that
  $j\exists(x:X)\ldotp\forall(x':X)\ldotp(\psi x'\wedge\phi x)\iff (x=x')$, and we may drop the $j$
  because $X$ is a $j$-sheaf. The conclusion of \eqref{eqn:sheaf2} follows directly.
\end{proof}

\begin{proposition}\label{prop:product_j_sheaves}\index{sheaf!product of}
  The product of $j$-sheaves is a $j$-sheaf.
\end{proposition}
\begin{proof}
  We use \cref{prop:characterize_j_sheaf}. It is easy to see that the product of separated sheaves
  is separated. Supposing $j\exists((x_1,x_2):X_1\times X_2)\ldotp\forall((x_1',x_2'):X_1\times
  X_2)\ldotp\phi(x_1',x_2')\iff(x_1,x_2)=(x_1,'x_2')$, one proves
  \[
    j\exists(\phi_1:X_1\to\Prop)(\phi_2:X_2\to\Prop)\ldotp
    \forall(x_1':X_1)(x_2':X_2)\ldotp \phi(x_1',x_2') \iff \phi_1(x_1')\wedge\phi_2(x_2')
  \]
  by taking $(x_1,x_2)$ from the hypothesis and letting $\phi_i(x_i')\coloneqq(x_i=x_i')$ for
  $i=1,2$. We now obtain $j\exists(x_i:X_i)\forall(x_i':X_i)\ldotp\phi_1(x_i')\iff(x_i=x_i')$, and
  the result follows from \cref{prop:characterize_j_sheaf}. 
\end{proof}

\begin{proposition}\label{prop:comm_existential_unit_surj}
Let $X$ a type. Suppose that $j$ is a modality satisfying
\[\forall(P:X\to\Prop)\ldotp[j\exists(x:X)\ldotp Px]\imp[\exists(x:X)\ldotp jPx].\]
Then the map $\eta_j:X\to\asSh_j(X)$ is surjective.
\end{proposition}
\begin{proof}
Suppose $\phi:X\to\Prop_j$ satisfies $j\exists(x:X)\ldotp\forall(x':X)\ldotp\phi(x')\iff(x=x')$.
Then by hypothesis, there is some $x:X$ satisfying $j\forall(x':X)\ldotp\phi(x')\iff(x=x')$, and this implies $\forall(x':X)\ldotp j(\phi(x')\iff(x=x'))$, which in turn implies $\forall(x':X)\ldotp \phi(x')\iff j(x=x')$.
\end{proof}

\begin{proposition}\index{separated}
  Suppose $X$ has decidable equality. If $j$ is a dense modality (i.e.\ $j\bot=\bot$) then $X$ is
  $j$-separated.
\end{proposition}
\begin{proof}
  If $X$ has decidable equality then it is easy to check $j(x_1=x_2)\imp(x_1=x_2)$.
\end{proof}

\begin{proposition}\label{prop:sep_predicates}
  Suppose $X$ decidable equality. Let $j$ be a modality and let $X\to\mathrm{sep}_j(X)$ denote its
  $j$-separification. Then the natural map below is an isomorphism:
  \[
    (\mathrm{sep}_j(X)\to\Prop_j)\To{\cong}(X\to\Prop_j).
  \]
\end{proposition}
\begin{proof}
  The map $f:(X\to\Prop_j)\to(\mathrm{sep}_j(X)\to\Prop_j)$, applied to any $P:X\to\Prop_j$ is given
  as follows. Any $y:\mathrm{sep}_j(X)$ is represented by some $x:X$, so we may let $f(P)(y)=P(x)$
  and it remains to prove that this respects the equivalence relation: $j(x=x')\imp (P(x)=P(x'))$.
  $x=x'$ implies $Px\iff Px'$, so it suffices to see that $Px\iff Px'$ is $j$-closed; see
  \cref{eqn:j_logic}.
\end{proof}

\subsection{Modalities and subtoposes}\label{sec:modalitites_subtoposes}

In this section we discuss how the logical notion of modality corresponds to the semantic notion of
subtopos. Much of the following is taken from \cite[A.4.4 and A.4.5]{Johnstone:2002a}; again, what
we call modalities are called \emph{local operators} there.\index{modality!order on}
\index{topos!subtopos of}\index{subtopos!corresponding to a modality}

There are several equivalent definitions of subtoposes of $\cat{E}$. One is a subcategory $\cat{E}'\ss\cat{E}$ with a left-exact left-adjoint. Another is a modality $j$ on $\cat{E}$, by which one can define the subcategory $\cat{E}_j$ of $j$-sheaves and the left-exact left adjoint $\asSh_j$ defined in \cref{def:separated_sheaf}.

As mentioned in \cref{ex:modalities_order}, there is a partial order on modalities: $j\leq j'$ iff
$\forall(P:\Prop)\ldotp j'P\imp jP$. With this partial order, modalities on a topos $\cat{E}$ form a
lattice (in fact a co-Heyting algebra), and there is a poset isomorphism between that of modalities
and that of subtoposes of $\cat{E}$ under inclusion.

To every proposition $U:\Prop$ we can associate three modalities: the open modality
$o(U)$\index{modality!open}, the closed modality $c(U)$\index{modality!closed}, and the quasi-closed
modality $q(U)$\index{modality!quasi-closed}. A topos of sheaves for one of these is respectively called an \emph{open subtopos}, a \emph{closed subtopos}, and a \emph{quasi-closed subtopos}. If $\cat{E}=\Shv{L}$ for some locale $L$, then $U$ can be identified with an open subspace $U\ss L$, and its topos of sheaves is equivalent to $\Shv{L}_{o(U)}$. Similarly, $\Shv{L}_{c(U)}$ is equivalent to the topos of sheaves on the closed subspace complementary to $U$, and $\Shv{L}_{q(U)}$ is equivalent to the topos of sheaves on the localic intersection of all dense subspaces of the complement of $U$. Logically, these modalities are defined as follows:
\begin{itemize}
  \item $o(U)P \coloneqq U \imp P$,
  \item $c(U)P \coloneqq U \vee P$,
  \item $q(U)P \coloneqq (P \imp U) \imp U$.
\end{itemize}
We say that a modality $j$ is \emph{open} if there exists $U:\Prop$ such that $j=o(U)$, and
similarly that $j$ is \emph{closed} or \emph{quasi-closed} if there exists $U:\Prop$ such that $j=c(U)$ or $j=q(U)$.

\begin{remark}
  In a boolean topos, the closed modality $U\vee P$ is the same as the open modality $\neg U\imp P$,
  but not in an arbitrary topos.
\end{remark}

\begin{example}
  Consider the case $U=\bot$, corresponding to the empty subtopos. The associated open modality
  $o(U)$ sends $P$ to $\top$, i.e.\ it is the terminal modality. The associated closed modality
  $c(U)$ is the identity, i.e.\ it is the initial modality.

  The most interesting is the associated quasi-closed modality $q(U)$, which in this case is
  double-negation, sending $P$ to $\neg\neg P$. It corresponds to a boolean subtopos which is dense
  in $\cat{E}$. Semantically, the double-negation modality\index{modality!double negation} can be
  thought of as ``almost always''. Every quasi-closed modality $q(U)$ is a double-negation modality, relative to
  the closed modality $j=c(U)$, in the sense that $q(U)P=(P\imp j\bot)\imp j\bot$.
\end{example}

As mentioned above, $j_1$ corresponds to a subtopos of $j_2$ iff $j_2(P)\imp j_1(P)$ for all
propositions $P$. For example, it is easy to check that for any proposition $U$, we have
$\forall(P:\Prop)\ldotp c(U)(P)\imp q(U)(P)$. In the lattice of modalities, the join of $j_1$ and
$j_2$ is given by ($j_1(P)\wedge j_2(P)$) of modalities. If terms $c:\tConst$ of a constant sheaf
index modalities $j_c$, then $\forall (c:\tConst)\ldotp j_c$ is a modality, and it corresponds to
their $\tConst$-indexed join.

If $j_1$ and $j_2$ are modalities then it is not necessarily the case that $j_1j_2$ will be a modality. It is easy to prove directly from definitions that $j_1j_2$ is a modality if either $j_1$ is open or $j_2$ is closed. If $j_1j_2$ is a modality then it is the meet of $j_1$ and $j_2$. Modalities $j_1$ and $j_2$ are
called \emph{disjoint}\index{modality!disjoint} if their meet $j$ is the bottom element, i.e.\ if
$j\bot$ holds.


\section{Dedekind $j$-numeric types}\label{sec:Dedekind_j}

\index{real number!Dedekind}
\index{real number|seealso {numeric object}}\index{real numbers|seealso {type, of real numbers}}
\index{object!real numbers}\index{type!of real numbers}
\index{real numbers|seealso {object, real numbers}}

We continue to work within the type theory discussed in \cref{sec:toposes_types_logic}. Using the notion of modality discussed in \cref{sec:modalities}, we define various sorts of numeric objects that exist and are internally definable for any topos $\cat{E}$ and modality $j$.

\subsection{Background}

From our work in \cref{sec:toposes_types_logic}, and in particular \cref{ex:integers_rationals}, we
have seen that the types $\tNN$, $\tZZ$, and $\tQQ$, together with their additive and multiplicative
structures (see \cref{ex:addition_of_naturals}) exist in our type theory. It turns out that,
semantically, the sheaves corresponding to $\tNN$, $\tZZ$, and $\tQQ$ are (locally) constant in any
topos, corresponding to the images of the sets $\NN,\ZZ,\QQ$ under the inverse image part of the
unique geometric morphism $\cat{E}\leftrightarrows\Cat{Set}$.

One manifestation of this constancy is that the $=$ and $<$ relations on $\tNN$, $\tZZ$, or $\tQQ$
are decidable in any topos. Another is that one can get away with blurring the distinction between
internal rational numbers---i.e.\ terms of the type $\tQQ$ in the internal language---and external
rational numbers; similarly for $\tNN$ and $\tZZ$. 

With $\tQQ$ in hand, the standard construction of the type $\tRR$ of real numbers from the rationals by Dedekind cuts can also be
carried out in the type theory with good results. However, unlike $\tNN$, $\tZZ$, and $\tQQ$, one should not expect the type $\tRR$ of real numbers to be semantically constant in an arbitrary topos, and similarly one should expect neither the $=$ nor the $<$ relation to be decidable.%
\footnote{
  It will turn out that the type $\tRR$ of real numbers \emph{is} constant in our main topos of interest, $\BaseTopos$, but not in many of the subtoposes we consider. For example the type $\tRR_\pi$ of real numbers in $\BaseTopos_\pi$ does not have decidable equality, roughly because continuous functions that are unequal globally may become equal locally.
}
In particular one cannot generally regard a section of $\tRR$ as an ordinary real number. Perhaps the best
intuition in an arbitrary topos is that terms of the type $\tRR$ correspond semantically to
continuous real-valued functions on the topos (as a generalized topological space).

In Dedekind's construction of the real numbers \cite{Dedekind:1872a}, a real number $x:\tRR$ is
defined in terms of two subsets of rational numbers: those that are less than $x$ and those that are
greater than $x$. In our context, we replace the external notion, ``two subsets of $\QQ$'' with the
internal notion, ``two subtypes of $\tQQ$''. They are classified by predicates $\delta:\tQQ\to\Prop$
and $\upsilon:\tQQ\to\Prop$, which roughly correspond to the predicates ``is less than $x$'' and
``is greater than $x$''.%
\footnote{
  Throughout this book, we use $\delta$ (delta) and $d$ for ``down'' and $\upsilon$ (upsilon) and
  $u$ for ``up''. Thus $\delta$ classifies a ``down-set'' of $\tQQ$, and $\upsilon$ classifies an ``up-set''.
}

These satisfy several axioms: for example for any $q_1<x$, there exists $q_2$ such that $q_1<q_2$ and $q_2<x$. This is translated into the logical statement
\begin{equation}\label{eqn:roundedness_example}
\forall(q_1:\tQQ)\ldotp\delta q_1\imp\exists(q_2:\tQQ)\ldotp (q_1<q_2)\wedge\delta q_2.
\end{equation}
This is one of the usual Dedekind axioms---often called \emph{roundedness}---which forces $\delta$ to act like $<$ rather than $\leq$. As mentioned above, the Dedekind axioms make sense in any topos $\cat{E}$, giving the standard construction of the type $\tRR$ of real numbers in $\cat{E}$; see e.g.\ \cite[VI.8]{MacLane.Moerdijk:1992a}.

However, we make two slightly non-standard moves. The first is that we will be interested in several numeric types---not just the real numbers---which are defined by relaxing some of the Dedekind axioms. Doing so in the topos $\Cat{Set}$ would yield a definition of the sets $\LR$, $\UR$, $\II$, $\IR$, and $\RR$, discussed in
\cref{sec:real_valued_IR}; in this book we will want to consider the analogous objects in other toposes.

The other non-standard move is that we want to work relative to an arbitrary modality $j$, as defined in \cref{def:modality}. To do so, we take each of the standard Dedekind axioms and apply $j$
``at every position''; the result will be an axiom for the associated numeric type in the $j$-subtopos. For example, the roundedness
axiom \eqref{eqn:roundedness_example} becomes
\[
  j\forall(q_1:\tQQ)\ldotp j[j\delta q_1\imp j\exists(q_2:\tQQ)\ldotp j(q_1<q_2)\wedge j\delta q_2].
\]
But if $\delta$ is a $j$-closed proposition, then we can use $j$-logic (see \cref{rem:j-logic}) to find an equivalent statement that is much simpler:
\[
  \forall(q_1:\tQQ)\ldotp \delta q_1\imp j\exists(q_2:\tQQ)\ldotp (q_1<q_2)\wedge\delta q_2.
\]
We make similar replacements of all the axioms from \cite[VI.8]{MacLane.Moerdijk:1992a}.

The result of our two non-standard moves is that we obtain what we call the Dedekind $j$-numeric
types, which we now formally define.

\subsection{Definition of the Dedekind $j$-numeric types}\label{sec:define_Dedekind}

In the definition below, one can simply remove the $j$'s (i.e.\ use the identity modality $jP=P$) to obtain the usual notion of Dedekind real numbers, and related numeric types.

\begin{definition}[Dedekind $j$-numeric types]
\label{def:local_reals}\index{numeric type}
\index{Dedekind numeric type|see {numeric type}}
  Let $\cat{E}$ be a topos, let $\tQQ$ be the type of rational numbers\index{type!rational numbers},
  and let $j$ be a modality on $\cat{E}$.\index{modality} Consider the following conditions on
  predicates $\delta:\tQQ\to\Prop$ and $\upsilon:\tQQ\to\Prop$,
  \begin{tabbing}
    \qquad\=
      0a.\quad\= $\forall(q:\tQQ)\ldotp j(\delta q)\imp\delta q$\hspace{1.2in}\=
      0b.\quad\= $\forall(q:\tQQ)\ldotp j(\upsilon q)\imp\upsilon q$
    \\
      \>1a.\> $\forall q_1,q_2\ldotp (q_1<q_2)\imp\delta q_2\imp\delta q_1$
      \>1b.\> $\forall q_1,q_2\ldotp(q_1<q_2)\imp\upsilon q_1\imp\upsilon q_2$
    \\
      \>2a.\> $\forall q_1\ldotp\delta q_1\imp j\exists q_2\ldotp (q_1<q_2)\wedge\delta q_2$
      \>2b.\> $\forall q_2\ldotp\upsilon q_2\imp j\exists q_1\ldotp (q_1<q_2)\wedge\upsilon q_1$
    \\
      \>3a.\> $j\exists q\ldotp\delta q$
      \>3b.\> $j\exists q\ldotp\upsilon q$
    \\
      \>4.\> $\forall q\ldotp(\delta q \wedge \upsilon q)\imp j\bot$ \\
      \>5.\> $\forall q_1,q_2\ldotp (q_1<q_2) \imp j(\delta q_1\vee\upsilon q_2)$
  \end{tabbing}
  We refer to these conditions as $j$-closed (0),\index{cut!$j$-closed} down/up-closed (1),
  $j$-rounded (2),\index{cut!$j$-rounded} $j$-bounded (3)\index{cut!$j$-bounded}, $j$-disjoint
  (4)\index{cuts!$j$-disjoint}, and $j$-located (5)\index{cuts!$j$-located}. It is easy to check that
  the $j$-disjointness condition is equivalent to $\forall(q_1,q_2:\tQQ)\ldotp\delta
  q_1\imp\upsilon q_2\imp j(q_1<q_2)$.
  
  When $j=\id$ is the trivial modality, we refer to the conditions simply as \emph{rounded}, \emph{bounded}, etc. In
  particular, axioms 0a and 0b can be dropped entirely when $j=\id$.

  We define ten real-number-like types, including the usual type of Dedekind real numbers $\tRR_j$
  for the subtopos $\cat{E}_j$, by using various subsets of these axioms. We call these
  \emph{Dedekind $j$-numeric types}; see \cref{table:numeric_types}.\index{numeric type!$j$-local}
  \begin{table}[h]
  \index{type!of intervals}\index{type!of lower real numbers}\index{type!of improper intervals}
  \index{type!of upper real numbers}\index{type!of real numbers}\index{type!of unbounded intervals}
  \index{type!of unbounded lower real numbers}\index{type!of unbounded improper intervals}
  \index{type!of unbounded upper real numbers}\index{type!of unbounded real numbers}
    \[
    \renewcommand{\arraystretch}{1.45}
    \begin{array}{l|l|l}
      \text{Name of type in }\cat{E}&\text{Notation}&\text{Definition}\\\hline
      \text{$j$-local unbounded lower reals}&\tLRub_j&\{\delta\mid\text{0a, 1a, 2a}\}\\
      \text{$j$-local unbounded upper reals}&\tURub_j&\{\upsilon\mid\text{0b, 1b, 2b}\}\\
      \text{$j$-local unbounded improper intervals}&\tIIub_j&\{(\delta,\upsilon)\mid\text{0, 1, 2}\}\\
      \text{$j$-local unbounded (proper) intervals}&\tIRub_j&\{(\delta,\upsilon)\mid\text{0, 1, 2, 4}\}\\
      \text{$j$-local unbounded reals}&\tRRub_j&\{(\delta,\upsilon)\mid\text{0, 1, 2, 4, 5}\}\\
      \text{$j$-local lower reals}&\tLR_j&\{\delta\mid\text{0a, 1a, 2a, 3a}\}\\
      \text{$j$-local upper reals}&\tUR_j&\{\upsilon\mid\text{0b, 1b, 2b, 3b}\}\\
      \text{$j$-local improper intervals}&\tII_j&\{(\delta,\upsilon)\mid\text{0, 1, 2, 3}\}\\
      \text{$j$-local (proper) intervals}&\tIR_j&\{(\delta,\upsilon)\mid\text{0, 1, 2, 3, 4}\}\\
      \text{$j$-local real numbers}&\tRR_j&\{(\delta,\upsilon)\mid\text{0, 1, 2, 3, 4, 5}\}
    \end{array}
    \]
  \caption{Ten Dedekind $j$-numeric types}\label{table:numeric_types}
  \end{table}
  We refer to those having only one cut ($\tLR_j$, $\tUR_j$, $\tLRub_j$, and $\tURub_j$) as
  \emph{one-sided} and to those having both cuts ($\tRR_j$, $\tIR_j$, $\tII_j$, $\tRRub_j$,
  $\tIRub_j$, and $\tIIub_j$) as \emph{two-sided}. We refer to those without the locatedness axiom
  ($\tLR_j$, $\tUR_j$, $\tLRub_j$, $\tURub_j$, $\tIR_j$, $\tII_j$, $\tIRub_j$, and $\tIIub_j$) as
  the \emph{$j$-numeric domains}, a name that will be justified in
  \cref{prop:intervals_as_RIds}.\index{numeric type!one-sided}\index{numeric
  type!two-sided}\index{numeric type!domain}\index{numeric domain|see {numeric type, domain}}

  Given unbounded reals $r=(\delta,\upsilon):\tRRub_j$ and $r'=(\delta',\upsilon'):\tRRub_j$, we often
  use the more familiar notation for inequalities between them or involving a rational $q:\tQQ$:\index{inequality!$j$-local}
  \begin{align*}
    q<r&\coloneqq\delta q&
    r<q&\coloneqq\upsilon q&
    r<_jr'&\coloneqq j\exists(q:\tQQ)\ldotp (r<q)\wedge(q<r')\\
    r\leq_j q&\coloneqq\delta q\imp j\bot&
    q\leq_j r&\coloneqq\upsilon q\imp j\bot&
    r'\leq_j r&\coloneqq(r<_jr')\imp j\bot
  \end{align*}
The same notation makes sense when $r,r':\tRR_j$ are (bounded) reals, but one should be a bit careful with other $j$-numeric types. We will discuss this more in \cref{sec:inequalities}.
\end{definition}

\begin{remark}\label{rem:rationals_as_reals}
  We associate to any $q:\tQQ$ a pair of cuts $(\delta_q,\upsilon_q)$, defined as follows on
  $q':\tQQ$
  \begin{equation}\label{eqn:q_cuts}
    \delta_q q'\iff j(q<q')
    \qquad\text{and}\qquad
    \upsilon_q q'\iff j(q'<q).
  \end{equation}
  It is easy to check that $(\delta_q,\upsilon_q)$ is a $j$-local real number, so we have a map
  $\tQQ\to\tRR_j$. In particular, for $r=(\delta_r,\upsilon_r):\tRR_j$, we write $q=r$ to mean
  $\forall(q':\tQQ)\ldotp (\delta_q q' \iff \delta_r q') \wedge (\upsilon_q q' \iff \upsilon_r q')$.

  There are also obvious isomorphisms $\tII_j\cong\tLR_j\times\tUR_j$ and $\tIIub_j \cong \tLRub_j
  \times \tURub_j$. Here is a diagram of relationships among the Dedekind $j$-numeric types:
  \begin{equation}\label{eqn:relations_among_Dedekinds}
  \begin{tikzcd}[column sep=15pt, row sep = 5pt]
    \tQQ\ar[rr,hook]&&\tRR_j\ar[dd, hook]\ar[rr, hook]&&
    \tIR_j\ar[dd, hook]\ar[rr, hook]&&
    \tII_j\ar[dd, hook]\ar[dl, two heads]\ar[dr, two heads]
  \\
    &&&&
    &\tLR_j&&
    \tUR_j
  \\
    &&
    \tRRub_j\ar[rr, hook]&&
    \tIRub_j\ar[rr, hook]&&
    \tIIub_j\ar[dl, two heads]\ar[dr, two heads]
  \\
    &&&&
    &\tLRub_j\ar[from=uu, hook, crossing over]&&
    \tURub_j\ar[from=uu, hook, crossing over]
  \end{tikzcd}
  \end{equation}
\end{remark}

\begin{remark}\label{rem:numeric_j_sheaf}\index{separated!$j$-numeric types as}
  Each of the Dedekind $j$-numeric types, $\tRR_j$, $\tIR_j$, $\tII_j$, $\tLR_j$, $\tUR_j$,
  $\tRRub_j$, $\tIRub_j$, $\tIIub_j$, $\tLRub_j$, and $\tURub_j$ is a $j$-sheaf. Proving this requires function extensionality; see \cref{ax:propositional_extensionality}.\index{extensionality!function}
\end{remark}

\subsection{Some preliminary numerical facts}\label{sec:preliminary_numeric}

Most of the material in this section is taken from \cite{Bauer.Taylor:2009a}. The results here will mainly be used in the next
section, to show that addition and multiplication preserve locatedness; see
\cref{thm:mult_is_continuous}. Throughout the section, $j$ is an arbitrary modality.

\begin{proposition}[$\tQQ$ is Archimedean]
\label{prop:Archimedean}\index{Archimedean property}
\[\forall(q_1,q_2:\tQQ)\ldotp(q_1>0\wedge q_2>0)\imp\exists(k:\tNN)\ldotp q_2<k*q_1.\]
\end{proposition}
\begin{proof}
We may assume $q_1=\frac{a+1}{m+1}$ and $q_2=\frac{b+1}{n+1}$ for $a,b,m,n\in\tNN$. Then we can set $k\coloneqq(m+1)*(b+1)$.
\end{proof}

\begin{proposition}\label{prop:within_sequence_j}
  For any $n:\tNN$, rationals $q_0<q_1<\cdots<q_{n+2}$, and real $r:\tRR_j$, we have
  \[
    q_0<r<q_{n+2}\imp j\exists(k:\tNN)\ldotp (0\leq k\leq n) \wedge (q_k<r<q_{k+2}).
  \]
\end{proposition}
\begin{proof}
We proceed by induction on $n$. The claim is trivial for $n=0$, so suppose it is true for arbitrary $n$. Suppose further that there are rationals $q_0<q_1<\cdots<q_{n+3}$ and $q_0<r<q_{n+3}$. By locatedness, $j(q_{n+1}<r\vee r<q_{n+2})$; we can drop the $j$ by \cref{rem:j-logic},%
\footnote{Throughout this section, we will use \cref{rem:j-logic} without mentioning it.}
to obtain $q_{n+1}<r\vee r<q_{n+2}$. In the first case we have $q_{n+1}<r<q_{n+3}$ and in the second we use the inductive hypothesis.
\end{proof}

The next two propositions say that any real number can be bounded between rationals that are arbitrarily close together, in the additive and the multiplicative sense.

\begin{proposition}\label{prop:arithmetic_located_j}\index{located!arithmetically}
  Any $r:\tRR_j$ is \emph{$j$-arithmetically located} in the sense that
  \[
    \forall(p:\tQQ)\ldotp p>0\imp j\exists(d,u:\tQQ)\ldotp(d<r<u)\wedge(0<u-d<p).
  \]
\end{proposition}
\begin{proof}
By boundedness, $j\exists(d',u':\tQQ)\ldotp d'<r<u'$, and by the Archimedean property (\cref{prop:Archimedean}), $\exists(k:\tNN)\ldotp u'-d'<k*\frac{p}{3}$. For each $i:\tNN$, let $q_i\coloneqq d'+i*\frac{p}{3}$, so $d'=q_0<q_1<\cdots<q_{k+2}$ and $u'<q_k<q_{k+2}$. By \cref{prop:within_sequence_j}, there is some $i$ with $0\leq i\leq k$ with $q_i<r<q_{i+2}$. The conclusion is satisfied using $d\coloneqq q_i$ and $u\coloneqq q_{i+2}$, because $u-d=\frac{2p}{3}<p$.
\end{proof}

\begin{proposition}\label{prop:multiplicatively_located}\index{located!multiplicatively}
  Any $r:\tRR_j$ with $0<r$ is \emph{$j$-multiplicatively located} in the sense that
  \[
    \forall(p,p':\tQQ)\ldotp 0<p<p'\imp j\exists(d,u:\tQQ)\ldotp
    (0<d<r<u)\wedge\left(\frac{u}{d}<\frac{p'}{p}\right).
  \]
\end{proposition}
\begin{proof}
  By roundedness $j\exists(q:\tQQ)\ldotp 0<q<r$. Choose $q':\tQQ$ such that $0<q'<q(p'-p)$. By
  arithmetic locatedness, $j\exists(d,u:\tQQ)\ldotp(d<r<u)\wedge(0<u-d<\frac{q'}{p'})$. Then $q<u$
  by disjointness, so $p'(u-d)<q'<q(p'-p)<u(p'-p)$, hence $up<p'd$ and $\frac{u}{d}<\frac{p'}{p}$.
\end{proof}

The standard definition of addition of unbounded improper
intervals\index{arithmetic!addition}\index{function!arithmetic|see {arithmetic}}, relativized to the $j$-modality, is given by
$(\delta_1,\upsilon_1)+(\delta_2,\upsilon_2)=(\delta',\upsilon')$, where
\begin{equation}\label{eqn:addition_unbounded_improper}
\begin{aligned}
  \delta'q&\iff j\exists(q_1,q_2:\tQQ)\ldotp\delta_1q_1\wedge\delta_2q_2\wedge(q<q_1+q_2)\\
  \upsilon'q&\iff j\exists(q_1,q_2:\tQQ)\ldotp\upsilon_1 q_1\wedge\upsilon_2 q_2\wedge(q_1+q_2<q).
\end{aligned}
\end{equation}

\begin{theorem}\label{thm:monoids}\index{monoid!numeric objects as}
The type $\tIIub_j$ has the structure of a commutative monoid, with the addition operation $+$ from \cref{eqn:addition_unbounded_improper}. Each of the types in \cref{eqn:relations_among_Dedekinds} is closed under this operation, and each of the maps shown there is a monoid homomorphism.
\end{theorem}
\begin{proof}
We begin by working only with lower cuts; upper cuts are similar. It is easy to see that $+$ is commutative and associative. For example the lower cut for both sides of $r_1+(r_2+r_3)=(r_1+r_2)+r_3$ will be equivalent to
\[j\exists(q_1,q_2,q_3:\tQQ)\ldotp\delta_1q_1\wedge\delta_2q_2\wedge\delta_3q_3\wedge(q<q_1+q_2+q_3).
\]
The additive unit is $0:\tQQ$; indeed it suffices to check the following equivalence for any lower cut $\delta$ and rational $q$:
\[
\delta q\iff j\exists (q_1,q_2:\tQQ)\ldotp q_1<0\wedge \delta q_2\wedge(q<q_1+q_2).
\]
For the backwards direction, the existential is equivalent to $\exists(q_2:\tQQ)\ldotp \delta q_2\wedge(q<q_2)$, so we obtain $\delta q$ by down-closure (1a) and $j$-closure (0a). For the forwards direction, suppose $\delta q$. Then by roundedness (2a) twice, $\exists(q',q'':\tQQ)\ldotp(q<q'<q'')\wedge\delta q''$, and the conclusion is satisfied by letting $q_1\coloneq q'-q''$ and $q_2\coloneqq q''$.

Thus $\tIIub_j$ has the structure of a commutative monoid. It is straightforward to check that if $\delta_1$ and $\delta_2$ are bounded (3a) then so is their sum, so the operation restricts to one on $\tII_j$, which then also has the structure of a commutative monoid.

For $j$-disjointness (4) and $j$-locatedness (5), we consider both lower and upper cuts. If $r_1$ and $r_2$ are $j$-disjoint then we want to see that $(\delta',\upsilon')=r_1+r_2$ is too. Suppose $\delta'q\wedge\upsilon'q$; we want to show $j\bot$. From
\[
j\exists(q_1,q_2,q_1',q_2':\tQQ)\ldotp\delta_1q_1\wedge\delta_2q_2\wedge(q<q_1+q_2)\wedge\upsilon_1q_1'\wedge\upsilon_2q_2'\wedge(q_1'+q_2'<q)
\]
one obtains $q_1'+q_2'<q_1+q_2$ and by Trichotomy of rationals, $(q_1'<q_1)\vee(q_2'<q_2)$. Either way, we obtain $j\bot$ from down/up closure and disjointness.

Suppose $r_1,r_2:\tRR_j$; we need to show $r_1+r_2$ is $j$-located, so take $q<q'$ and let $p\coloneqq\frac{q'-q}{2}$. By \cref{prop:arithmetic_located_j}, we can find rationals $q_1<r_1<q_1'$ and $q_2<r_2<q_2'$ such that $(q_1'-q_1)<p$ and $(q_2'-q_2)<p$. It is easy to check that $q_1+q_2<r_1+r_2<q_1'+q_2'$. Either $q<q_1+q_2$ or $q\geq q_1+q_2$; in the first case we get $q<r_1+r_2$ and in the second case we can get $r_1+r_2<q'$, as desired.
\end{proof}

\begin{corollary}\label{cor:reals_group}\index{group, real numbers as}
The monoid structure on $\tRR_j$ is in fact that of a group.
\end{corollary}
\begin{proof}
Take any $(\delta_1,\upsilon_1):\tRR_j$ and define cuts $\delta_2,\upsilon_2$ on $q:\tQQ$ by
\[
	\delta_2q\iff\upsilon_1(-q)
	\qquad\text{and}\qquad
	\upsilon_2q\iff\delta_1(-q).
\]
It is easy to check that $(\delta_2,\upsilon_2)$ is a $j$-local real. Let $(\delta',\upsilon')\coloneqq(\delta_1,\upsilon)+(\delta_2,\upsilon_2)$ be their sum, as in \cref{eqn:addition_unbounded_improper}. We want to show $\delta'q\iff j(q<0)$ and $\upsilon'q\iff j(0<q)$. Consider the latter, the former being similar; we must check
\[j(0<q)\iff j\exists(q_1,q_2:\tQQ)\ldotp\upsilon_1q_1\wedge\delta_1(-q_2)\wedge(q_1+q_2<q).\]
The forward direction is \cref{prop:arithmetic_located_j}; the backwards direction comes down to the fact that $(\delta_1,\upsilon_1)$ is $j$-disjoint.
\end{proof}

\subsection{Numeric domains}

In \cref{def:local_reals} we referred to eight of the $j$-numeric types as \emph{$j$-numeric domains}. In this section, specifically \cref{prop:intervals_as_RIds}, we show that they actually are domains
internal to $\cat{E}$.

In fact, these eight domains correspond to the
eight predomains we defined in \cref{ex:our_j_domains}. Though the subject of predomains is a bit technical, its value to us it that predomains
provide a very useful tool for defining arithmetic operations between numeric types.\index{numeric
type!as $j$-rounded ideals}

Again, the eight predomains from \cref{ex:our_j_domains} can be defined internally to any topos.
Write $\tLRpre$, $\tURpre$, $\tIIpre$, $\tIRpre[j]$, $\tLRubpre$, $\tURubpre$, $\tIIubpre$, and
$\tIRubpre[j]$ to denote the corresponding internal constructions in any topos $\cat{E}$ and
modality $j$. For example $\tLRpre$ is just the rationals with the usual $<$ ordering: $\tLRpre=(\tQQ,<)$. Because our type theory has sum types (see \cref{sec:notions_lambda}) we can also make sense of the unbounded predomains, such as  $\tLRubpre=(\tQQ+\{\infty\},<)$.

Recall the $j$-local rounded ideal construction from \cref{def:predomain_modality}.

\begin{proposition}\label{prop:intervals_as_RIds}
  Let $\cat{E}$ be a topos, and let $j$ be a modality on $\cat{E}$. Then we have
  isomorphisms\index{numeric type!domain}
  \begin{align*}
    \tLR_j&\cong\RId_j(\tLRpre)&
    \tUR_j&\cong\RId_j(\tURpre)&
    \tII_j&\cong\RId_j(\tIIpre)&
    \tIR_j&\cong\RId_j(\tIRpre[j])\\
    \tLRub_j&\cong\RId_j(\tLRubpre)&
    \tURub_j&\cong\RId_j(\tURubpre)&
    \tIIub_j&\cong\RId_j(\tIIubpre)&
    \tIRub_j&\cong\RId_j(\tIRubpre[j])
  \end{align*}
  Thus each $j$-numeric domain in the sense of \cref{def:local_reals} is in fact an internal domain---and in
  particular an internal topological space---in the subtopos
  $\cat{E}_j$.\index{internal!domain}\index{numeric type!domain}
\end{proposition}
\begin{proof}
  Consider the four conditions in the definition of $j$-rounded ideal found in
  \cref{def:predomain_modality}; below we will refer to them 0', 1', 2', 3'. They match exactly with
  the $j$-closed, up/down-closed, $j$-roundedness, and $j$-boundedness conditions of
  \cref{def:local_reals}. The first two isomorphisms are hence direct. \Cref{prop:RId_product} is
  easily modified to the context of a modality $j$, showing that
  \begin{equation}\label{eqn:rounded_ideal_product_Q}
    \RId_j(\tIIpre) = \RId_j(\tLRpre\times\tURpre) \cong \RId_j(\tLRpre)\times\RId_j(\tURpre)
    \cong \tLR\times\tUR \cong \tII,
  \end{equation}
  giving the third isomorphism. It is not quite as automatic, but still easy to see that we have the
  fourth isomorphism $\tIR_j\cong\RId_j(\tIRpre[j])$, because a rounded ideal $I\ss\tIIpre$ is
  contained in $\tIRpre[j]$ iff $(q,q')\in I\imp j(q<q')$ iff $\delta q\imp\upsilon q'\imp j(q<q')$
  iff $I$ is $j$-disjoint.

  The bottom row is more interesting; we consider only the case for unbounded lower reals, the
  others being similar. We want to obtain the isomorphism
  \[
    \{\delta':\tQQ\sqcup\{-\infty\}\to\Prop\mid\text{0', 1', 2', 3'}\}
    \cong
    \{\delta:\tQQ\to\Prop\mid\text{0a, 1a, 2a}\}
  \]
  because the left-hand side is $\RId_j(\tLRubpre)$ and the right-hand side is $\tLRub$.

  Given $\delta'$ on the left, define $\delta q\coloneqq\delta' q$; it is easy to check that it
  satisfies 0a, 1a, 2a. Conversely, given $\delta$ on the right, define $\delta' q'$ to be $\delta
  q'$ if $q'\in\tQQ$ and $\top$ if $q'=-\infty$. Recall from \cref{ex:predomains} that we defined
  the order relation so that $-\infty<-\infty$. With that in mind, the conditions 0', 1', 2', 3' are
  easy.

  The second statement follows directly from \cref{prop:intervals_as_RIds,cor:RIdj_IsDomain}.
\end{proof}

Domains have two useful relations, $\sqss$ and $\ll$; see \cref{sec:rev_cont_poset}. Thus by \cref{prop:intervals_as_RIds}, these
symbols make sense for each of the eight $j$-numeric domains.

\begin{proposition}\label{prop:sqss_ll}\index{way-below!for numeric types}
\index{specialization order!for numeric types}
  Suppose that $\delta,\delta':\tLR_j$ or $\delta,\delta':\tLRub_j$. And suppose that
  $\upsilon,\upsilon':\tUR_j$ or $\upsilon,\upsilon':\tURub_j$. Then
  \begin{align*}
    \delta'\sqss\delta&\iff\forall(q:\tQQ)\ldotp\delta' q\imp\delta q&
    \delta'\ll\delta&\iff j\exists(q:\tQQ)\ldotp\delta q\wedge(\delta'q\imp j\bot)\\
    \upsilon'\sqss\upsilon&\iff\forall(q:\tQQ)\ldotp\upsilon' q\imp\upsilon q&
    \upsilon'\ll\upsilon&\iff j\exists(q:\tQQ)\ldotp\upsilon q\wedge(\upsilon'q'\imp j\bot)
  \end{align*}
  For $(\delta,\upsilon)$ and $(\delta',\upsilon')$ in $\tII_j$, $\tIIub_j$, $\tIR_j$, or
  $\tIRub_j$, we have
  \[
    [(\delta,\upsilon)\sqss(\delta',\upsilon')]\iff[(\delta\sqss\delta')\wedge(\upsilon\sqss\upsilon')]
    \quad\text{and}\quad
    [(\delta,\upsilon)\ll(\delta',\upsilon')]\iff[(\delta\ll\delta')\wedge(\upsilon\ll\upsilon')]
  \]
\end{proposition}
\begin{proof}
  Consider the $\delta$ case; the rest are similar. The statement for $\sqss$ is basically just the definition of these domains, and comes directly from
  \cref{cor:RIdj_IsDomain}. The statement about $\ll$ will also follow from \cref{cor:RIdj_IsDomain}.
  Indeed, it says
  \[
    \delta'\ll\delta \iff
    j\exists(q:\tQQ)\ldotp\delta q\wedge\forall(q':\tQQ)\ldotp\delta'q'\imp j(q'<q),
  \]
  and the right-hand side is equivalent to $j\exists(q:\tQQ)\ldotp\delta q\wedge(\delta'q\imp
  j\bot)$ by trichotomy for rationals, $(q<q')\vee(q=q')\vee(q>q')$.\index{trichotomy}
\end{proof}

\begin{remark}
  Rather than define Dedekind $j$-numeric types as predicates on $\tQQ$, as we did in
  \cref{def:local_reals}, one could also define them as predicates on $\asSh_j\tQQ$, the
  $j$-sheafification of rationals. However, it turns out that doing so makes no difference:
  semantically, it defines the same topos of sheaves.
  
  Indeed, it is enough by \cref{prop:sep_predicates} to check semantically that
  $j$-separification agrees with $j$-sheafification,
  $\mathrm{sep}_j\mathcal{Q}\cong\asSh_j\mathcal{Q}$, for the constant type $\mathcal{Q}$ on any set
  $Q$. So suppose given a $j$-covering family $\{V_i\to V\}_{i\in I}$ and sections
  $q_i\in\mathcal{Q}(V_i)$ such that $j(q_i=q_{i'})$ for $i,i'\in I$. For each $i$, we may assume
  each $q_i\in Q$, so it has an extension $\ol{q_i}\in\mathcal{Q}(V)$. We want to show
  $j(\ol{q_i}=\ol{q_{i'}})$. We have $(\ol{q_i}=\ol{q_{i'}})\vee\neg(\ol{q_i}=\ol{q_{i'}})$, so we
  may assume $\neg(\ol{q_i}=\ol{q_{i'}})$ from which we obtain $\neg(q_i=q_{i'})$ and hence $j\bot$,
  which proves the result.
\end{remark}

Before concluding this section, we record a lemma that will become useful in \cref{sec:differentiablility}, when we discuss differentiability.

\begin{lemma}\label{lemma:located_spec_disjoint}
  Let $x,y:\tII_j$ be $j$-improper intervals such that $x\sqss y$. If $x$ is $j$-located and $y$ is
  $j$-disjoint, then $x=y$. 
\end{lemma}
\begin{proof}
  Write $x=(\delta_x,\upsilon_x)$ and $y=(\delta_y,\upsilon_y)$. For all $q:\tQQ$ we have $\delta_x
  q\imp\delta_yq$ and $\upsilon_xq\imp\upsilon_yq$, so it suffices to show $\delta_y q\imp\delta_xq$
  for any $q:\tQQ$ (the $\upsilon$ case is similar). Since $\delta_x$ is $j$-closed, we may drop $j$
  from the front of all hypotheses by \cref{rem:j-logic}.

  Assuming $\delta_yq$, we have $\exists q'\ldotp q<q'\wedge\delta_yq'$ by $j$-roundedness. Since
  $x$ is $j$-located we have $\delta_xq\vee\upsilon_xq'$. In the first case we are done; in the
  second case we have $\upsilon_yq'$ which is a contradiction because $y$ is $j$-disjoint.
\end{proof}

\subsection{Inequalities}
\label{sec:inequalities}\index{inequality!$j$-local}

In this section we will discuss inequalities for Dedekind $j$-numeric objects, and in \cref{sec:arithmetic} we discuss their arithmetic. We learned some of this from \cite{Kaucher:1980a} and \cite{Sainz.Armengol.Calm.Herrero.Jorba.Vehi:2014a}, which in particular discuss the order-theoretic and algebraic properties of the extended interval domain $\tII$.

The proofs in the remainder of this chapter,  \cref{sec:inequalities,sec:j_constant_numerics,sec:arithmetic},
will depend heavily on the work on predomains\index{predomain} discussed
in \cref{sec:predomain}, however the results can usually be stated without it. We write $\tQQ$ to
denote the type of rational numbers, together with the usual $<$ ordering, and we write $\tQQ\op$
for the opposite ordering. Both of these are predomains, and we will often consider the predomain
$\tQQ\times\tQQ\op$.

We defined inequalities for reals in \cref{def:local_reals}. We can easily generalize the definition to $j$-local improper intervals. Given two
$j$-local unbounded improper intervals $x_1,x_2:\tIIub_j$, we define $x_1<_jx_2$ by
\begin{equation}\label{eqn:j_less_than}
  x_1<_jx_2\iff j\exists(q:\tQQ)\ldotp (x_1<q)\wedge(q<x_2),
\end{equation}
where as usual $x_1<q$ and $q<x_2$ respectively mean $\upsilon_1 q$ and $\delta_2 q$, for
$x_i=(\delta_i,\upsilon_i)$.%
\footnote{In \cref{def:local_reals} we gave the warning that one should be a bit careful with the $<$ relation on other numeric types. As an example of what can violate standard intuition, let $x=(\delta,\upsilon):\tII$ be given by $\delta q\iff q<1$ and $\upsilon q\iff -1<q$. Then $x<x$ holds.

For the one-sided numeric types, e.g.\ $\tLR$, there is a good notion of $\leq$, namely $\forall(q:\tQQ)\ldotp\delta q\imp\delta'q$. To obtain $<$ one might try $(\delta\leq\delta')\wedge(\delta\neq\delta')$, but this is semantically too strong in a general topos.}
The same definition also works for the other two-sided $j$-numeric
types $\tIRub_j$, $\tRRub_j$, $\tII_j$, $\tIR_j$, and $\tRR_j$.

\begin{proposition}\label{prop:<_open}\index{inequality!as open subset}
  For any modality $j$, the set $\{(x_1,x_2)\mid x_1<_j x_2\}\ss\tIIub_j\times\tIIub_j$ is an open
  subset. The same is true when $\tIIub_j$ is replaced by any of the other two-sided $j$-numeric
  types.
\end{proposition}
\begin{proof}
  The other two-sided $j$-numeric types are subspaces of $\tIIub_j$, so it suffices to check that
  one. Begin by defining an open of the predomain $\tIIubpre\times\tIIubpre$ by
  \[
    U_<\coloneqq\{(d_1,u_1),(d_2,u_2)\mid u_1<d_2\}\in\Opens(\tIIubpre\times\tIIubpre).
  \]
  This is easily seen to be an open in the sense of \cref{def:predomains_rounded_specs}: it is
  up-closed (decreasing $u_1$ and increasing $d_2$ stays in the set) and rounded (if $u_1<d_2$, then
  we can easily find $d'_1,u'_2,d'_2,u'_2$ with $d'_1<d_1$, $u_1<u'_1$, $d'_2<d_2$, $u_2<u'_2$, and
  still $u'_1<d'_2$).

  By \cref{thm:opens_iso_Scott_opens}, this determines a Scott open subset of
  $\RId_j(\tIIubpre\times\tIIubpre)$ given by $\U_{U_<}=\{I\mid\exists(d_1,u_1,d_2,u_2)\in I\ldotp
  u_1<d_2\}$. Under the isomorphism $\tIIub_j\times\tIIub_j\cong\RId_j(\tIIubpre\times\tIIubpre)$
  from \cref{prop:intervals_as_RIds,prop:RId_product}, this is equivalent to
  \[
    \Bigl\{\bigl((\delta_1,\upsilon_1),(\delta_2,\upsilon_2)\bigr) \Bigm|
      j\exists(d_1,u_1,d_2,u_2)\ldotp \delta_1 d_1 \wedge
      \upsilon_1 u_1 \wedge \delta_2 d_2 \wedge \upsilon_2 u_2 \wedge u_1 < d_2\Bigr\},
  \]
  or more simply $\{(\delta_1,\upsilon_1),(\delta_2,\upsilon_2)\mid j\exists q\ldotp \upsilon_1 q
  \wedge \delta_2 q\}$, for which \cref{eqn:j_less_than} is shorthand. In the language of
  \cref{def:predomain_modality}, $(x_1<_jx_2)$ iff $(x_1,x_2)\models_jU_<$.
\end{proof}

We now move on to discussing the non-strict inequality, $\leq_j$.

\begin{proposition}\label{prop:non-strict_inequality}
  The following are equivalent for all $a,b:\tRR_j$:
  \begin{enumerate}
    \item $(b <_j a)\imp j\bot$
    \item $\forall(q:\tQQ)\ldotp q<a\imp q<b$
    \item $\forall(q:\tQQ)\ldotp b<q\imp a<q$
  \end{enumerate}  
\end{proposition}
\begin{proof}
  This is proven for an arbitrary topos in \cite[Lemma D.4.7.6]{Johnstone:2002a}, so it easily
  applies to the $j$-subtopos. However it is also straightforward, so we we show the proof for
  $1\imp 2$; the other implications are similarly easy.

  Write $a=(\delta_a,\upsilon_a)$ and $b=(\delta_b,\upsilon_b)$. To prove $1\imp 2$, assume $(b <_j
  a)\imp j\bot$, and choose $q$ such that $\delta_aq$. It suffices to prove $j\delta_b q$ because
  $\delta_bq$ is $j$-closed; thus we may drop $j$ from the front of all hypotheses. By
  $j$-roundedness, we have $j\exists q'\ldotp q<q'\wedge\delta_aq'$, so (dropping $j$) we choose
  such a $q'$ and, by $j$-locatedness, we have $\upsilon_b q'\vee\delta_b q$. In one case we are
  done; in the other we have $\delta_aq'\wedge\upsilon_bq'$, meaning $b<_ja$, and we obtain $j\bot$
  by assumption.
\end{proof}

Suppose $a,b:\tRR_j$ are $j$-local reals. We write
\begin{equation}\label{eqn:def_leq}\index{inequality!non-strict}
  a\leq_jb
\end{equation}
iff any of the equivalent conditions of \cref{prop:non-strict_inequality}, e.g.\ $(b<_ja)\imp
j\bot$, is satisfied.

It is easy to check that the relation $\leq_j$ on $\tRR_j$ is reflexive, transitive, and
anti-symmetric, $(a\leq_jb)\wedge(b\leq_ja)\imp a=b$. It is also easy to check that
$(a<_jb)\imp(a\leq_jb)$, as well as the following facts:
\begin{gather*}
  (a<_jb)\wedge(b\leq_jc)\imp (a<_jc)\qquad\text{and}\qquad
  (a\leq_jb)\wedge(b<_jc)\imp (a<_jc).
\end{gather*}
As mentioned above, \cref{prop:non-strict_inequality} is about $j$-local reals, \emph{not} $j$-local
improper intervals. 

\begin{remark}
The above story also works for unbounded reals and one-sided numeric types. For $a,b:\tRRub_j$ one
can define $a\leq_j b$ to be any of the following equivalent conditions
  \begin{enumerate}
    \item $(b <_j a)\imp j\bot$
    \item $\forall(q:\tQQ\cup\{-\infty\})\ldotp q<_ja\imp q<_jb$
    \item $\forall(q:\tQQ\cup\{\infty\})\ldotp b<_jq\imp a<_jq$
  \end{enumerate}
  For $a,b:\tLR$ or $a,b:\tLRub$, only condition 2 is defined, and so we define $a\leq_j b$ as
  condition 2; this agrees with the domain order on $\tLR$ and $\tLRub$. For $a,b:\tUR$ or
  $a,b:\tURub$, only condition 3 is defined, and so we define $b\geq_ja$ as condition 3; this agrees
  with the domain order on $\tUR$ and $\tURub$.
\end{remark}

\subsection{$j$-constant numeric types}\label{sec:j_constant_numerics}\index{numeric type!$j$-constant}

The present section (\cref{sec:j_constant_numerics}) is non-standard material, regarding what we call \emph{constant numeric types}. The technical underpinnings of these ideas is given in \cref{sec:constant_predomain}.

\begin{definition}[$j$-constant numeric types]\label{def:constant_dedekind}
  Let $j$ be a modality. Say that $\phi:\tQQ\to\Prop$ is \emph{$j$-decidable} if it satisfies
  \[
    \forall(q:\tQQ)\ldotp\phi q\vee (\phi q\imp j\bot).
  \]
  Let $\ctLR_j$ denote the subtype of $\tLR_j$ consisting of those $\delta$ that are $j$-decidable,
  and similarly define $\ctUR_j$, $\ctIR_j$, $\ctII_j$, $\ctRR_j$ $\ctLRub_j$, $\ctURub_j$,
  $\ctIRub_j$, $\ctIIub_j$, and $\ctRRub_j$. We refer to these as the \emph{$j$-constant numeric
  types}.
\end{definition}

Note that $\ctII_j=\ctLR_j\times\ctUR_j$ and $\ctIIub_j=\ctLRub_j\times\ctURub_j$.

\begin{remark}\label{rem:basis_for_our_domains}\index{domain!basis of}
  For each of the eight $j$-numeric domains (see \cref{prop:sqss_ll}), the constant types form a basis in the sense of
  \cref{def:basis_domain}; indeed, this follows from
  \cref{ex:constant_predomains,prop:predom_to_RId_j}. This fact implies, for example, that for any
  $x,y:\tII_j$,
  \[
    (x\ll y)\iff\exists(c:\ctII_j)\ldotp x\ll c\ll y
    \quad\text{and}\quad
    x\sqss y\iff\forall(c:\ctII_j)\ldotp(c\ll x)\imp(c\sqss y)
  \]
  and similarly when $\tII$ is replaced with any of the other $j$-numeric domains.
\end{remark}

\begin{proposition}\label{prop:const_lower_upper_reals}
  For any $j$ there are internal bijections
  \[
    \ctLRub_j\cong\ctRRub_j\cong\ctURub_j.
  \]
\end{proposition}
\begin{proof}
  For the first claim, consider the functions $\ctRRub_j\to\ctLRub_j$ and $\ctRRub_j\to\ctURub_j$
  given by sending $(\delta,\upsilon)$ to $\delta$ and $\upsilon$ respectively. In the first case,
  the inverse sends $\delta$ to $(\delta,\upsilon)$ where for any $q:\tQQ$,
  \begin{equation}\label{eqn:upsilon_from_delta}
    \upsilon q\iff j\exists(q':\tQQ)\ldotp(q'<q)\wedge(\delta q\imp j\bot)
  \end{equation}
  It is easy to check that $\upsilon$ is $j$-closed, up-closed, and $j$-rounded, and that the pair
  $(\delta,\upsilon)$ is $j$-disjoint and $j$-located. To see that $\upsilon$ is $j$-decidable, we
  proceed as in \cref{prop:const_type_decidably_inhabited}. Take $q:\tQQ$ and consider the constant
  type $\tQQ_{<q}=\{q':\tQQ\mid q'<q\}$. We have $\forall(q':\tQQ_{<q})\ldotp(\delta q'\imp
  j\bot)\vee\delta q'$, so $\forall(q':\tQQ_{<q})\ldotp\exists(q'':\tQQ_{<q})\ldotp(\delta q''\imp
  j\bot)\vee\delta q'$, and thus $\upsilon q\vee\forall(q':\tQQ_{<q})\ldotp\delta q'$ by
  \cref{ax:distributivity}. It is easy to show that $\forall(q':\tQQ_{<q})\ldotp\delta q'$ implies
  $\upsilon q\imp j\bot$, so $\upsilon$ is indeed $j$-constant. Finally, to see that the two
  functions are mutually inverse, one takes $(\delta,\upsilon):\ctRRub_j$ and shows using
  $j$-locatedness that \cref{eqn:upsilon_from_delta} holds for any $q:\tQQ$.
\end{proof}

\begin{corollary}\label{cor:bijections_constant_intervals_j}
  There are internal bijections
  \[
    \ctIIub_j\cong\ctRRub_j\times\ctRRub_j
    \qquad\text{and}\qquad
    \ctIRub_j\cong\{(d,u)\in\ctRRub_j\times\ctRRub_j\mid d\leq_j u\}
  \]
\end{corollary}
\begin{proof}
  We obtain the left-hand isomorphism from \cref{prop:const_lower_upper_reals} and the isomorphisms
  $\ctIIub_j\cong\ctLRub_j\times\ctURub_j$ and $\ctII_j\cong\ctLR_j\times\ctUR_j$. Using the
  left-hand isomorphism, we will also have the right-hand isomorphism if we can show that for any
  $d,u:\tRRub$ we have $[\forall q\ldotp ((q<d\wedge u<q)\imp j\bot)]\iff d\leq_j u$. The left-hand
  side is constructively equivalent to $(\exists q\ldotp q<d\wedge u<q)\imp j\bot$ and the
  right-hand side is by definition $(u<_jd)\imp j\bot$, so it suffices to show $u<_jd\iff
  j\exists(q:\tQQ)\ldotp q<d\wedge u<q$, but that is the definition (\ref{def:local_reals}).
\end{proof}

\begin{notation}\label{not:constant_intervals}\index{notation}
  For $d,u:\ctRRub_j$, we may write $[d,u]:\ctIIub_j$ to denote the constant interval
  $(\delta,\upsilon):\ctIIub_j$ given by
  \[
    \delta q\iff j(q<d)
    \qquad\text{and}\qquad
    \upsilon q\iff j(u<q)
  \]
  Thus we have $[-,-]:\ctRRub_j\times\ctRRub_j\to\ctIIub_j$.

  Recall from \cref{ex:predomains} that $\tIIubpre=\tQQ\times\tQQ$ with $(q_1,q_2)\prec
  (q_1',q_2')\iff (q_1<q_1')\wedge(q_2'<q_2)$. Composing $[-,-]$ with the square of the map
  $i:\tQQ\to\tRRub_j$, one obtains $[i,i]:\tIIubpre\to\tIIub_j$ which sends $(q,q')$ to $[q,q']$; it
  is the usual inclusion $\downclose$ of the predomain into its associated domain (see
  \cref{lemma:rcj_downclosure_RId}).
\end{notation}

Let $f:X\to Y$ be a function. One might say it is $j$-injective if it satisfies
$\forall(x,x':X)\ldotp f(x)=f(x')\imp j(x=x')$, but for \cref{prop:const_j_bijective} we want a
stronger notion. So say that $f$ is a \emph{strongly $j$-injective} if it satisfies
\begin{equation}\label{eqn:strongly_j_injective}
  \forall(x,x':X)\ldotp f(x)=f(x')\imp (x=x')\vee j\bot.
\end{equation}

\begin{proposition}\label{prop:const_j_bijective}
  There are functions
  \[
    \ctRR_j\sqcup\{\infty\}\to\ctLR_j
    \qquad\text{and}\qquad
    \tRR_j\sqcup\{-\infty\}\to\ctUR_j.
  \]
  which are internally surjective and strongly $j$-injective.
\end{proposition}
\begin{proof}
  The map $\ctRR_j\sqcup\{\infty\}\to\ctLR_j$ sends $(\delta,\upsilon):\ctRR_j$ to $\delta$, and
  sends $\infty$ to the unique predicate $\delta$ satisfying $\forall q\ldotp\delta q$. We will show
  it is an internal surjection and a $j$-injection, the claim for $\ctUR_j$ being similar. Take
  $\delta:\ctLR_j$ and define $\upsilon$ as in \cref{eqn:upsilon_from_delta}, and the pair
  $(\delta,\upsilon)$ is in $\ctRR_j$ iff $\upsilon$ is $j$-bounded. We already have shown that
  $\forall q\ldotp\upsilon q\vee (\upsilon q\imp j\bot)$ and again follows from
  \cref{prop:const_type_decidably_inhabited} that $(\exists q\ldotp\upsilon q)\vee\forall
  q\ldotp(\upsilon q\imp j\bot)$. One can prove that the second case is equivalent to $\forall
  q\ldotp\delta q$, and so our function is indeed surjective.

  To see that the map is strongly $j$-injective, take $x,x':\ctRR_j\sqcup\{\infty\}$. We need only
  consider two of the four cases, namely when $x,x':\ctRR_j$, and when $x:\ctRR_j$ and $x'=\infty$.
  For the former we need to show that if both $x=(\delta,\upsilon)$ and $x'=(\delta,\upsilon')$ are
  in $\ctRR_j$ then $(\upsilon=\upsilon')\vee j\bot$. So take $q:\tQQ$; we want to show $\upsilon
  q\iff\upsilon'q$. Because $\upsilon$ and $\upsilon'$ are constant, it suffices to consider the
  case $\upsilon q\wedge(\upsilon'q\imp j\bot)$. We obtain $j\bot$ by $j$-roundedness of $\upsilon$,
  $j$-disjointness of $(\delta,\upsilon)$ and $j$-locatedness of $(\delta,\upsilon')$.

  Finally, suppose $x=(\delta,\upsilon):\ctRR_j$ and $x'=\infty$. It suffices to show that if
  $\delta q$ holds for all $q$ then $j\bot$. We get this because $\upsilon$ is $j$-bounded and
  $(\delta,\upsilon)$ is $j$-disjoint.
\end{proof}

\begin{proposition}\label{prop:const_predomains_RIDs}\index{numeric type!domain}
  For any $j$, the subtype of $j$-constant elements of any $j$-numeric domain is again a domain, namely
  the domain of constant rounded ideals on the associated predomain:
  \begin{align*}
    \ctLR_j  &\cong\cRId_j(\tLRpre)  &
    \ctUR_j  &\cong\cRId_j(\tURpre)  &
    \ctII_j  &\cong\cRId_j(\tIIpre)  &
    \ctIR_j  &\cong\cRId_j(\tIRpre[j])  \\
    \ctLRub_j&\cong\cRId_j(\tLRubpre)&
    \ctURub_j&\cong\cRId_j(\tURubpre)&
    \ctIIub_j&\cong\cRId_j(\tIIubpre)&
    \ctIRub_j&\cong\cRId_j(\tIRubpre[j])
  \end{align*}
\end{proposition}
\begin{proof}
  This is just a matter of seeing that
  \cref{def:constant_dedekind,def:predom_const_approx_decidable} are equivalent.
\end{proof}

\subsection{Arithmetic}
\label{sec:arithmetic}
\index{arithmetic|(}

The usual arithmetic operations on the usual set $\II$ of extended intervals are well-known (see
\cite{Kaucher:1980a} and \cite{Goldsztejn:2011a}, who calls them ``modal intervals''). For example
addition is straightforward, $(a_1,b_1)+(a_2,b_2)=(a_1+a_2,b_1+b_2)$, whereas multiplication
involves multiple cases depending on signs. It turns out that disjoint intervals and located
disjoint intervals---i.e.\ real numbers---are each closed under these
operations.\index{arithmetic!Kaucher}\index{arithmetic}\index{arithmetic!$j$-local}

Here we discuss the internal $j$-local arithmetic for an arbitrary modality $j$. Our main task is to
understand when 2-variable functions in $\tIIpre\coloneqq(\tQQ\times\tQQ\op)$ give rise to
continuous 2-variable functions on $\tII$. For elements $b_1=(q_1,q_1')$ and $b_2=(q_2,q_2')$ in
$\tQQ\times\tQQ\op$ we write $b_1\prec b_2$ to mean $q_1<q_2$ and $q_2'<q_1'$. We write $b_1\specord b_2$
to mean $q_1\leq q_2$ and $q_2'\leq q_1'$. This notation is aligned with that in
\cref{sec:predomains}.\index{predomain}

\Cref{prop:binary_fn_interval_domains} is a bit technical, but it paves the way for the $j$-arithmetic, from \cref{thm:addition_is_continuous} through the end of the chapter.

\begin{proposition}\label{prop:binary_fn_interval_domains}
\index{arithmetic!via approximable mapping}\index{approximable mapping}
  Let $f\colon \tIIpre\times \tIIpre\to \tIIpre$ be a function satisfying the following four
  conditions:
  \begin{equation}\label{eqn:four_conditions_approximable}
  \begin{aligned}
    (b_1'\prec b_1)\imp f(b_1',b_2)\specord f(b_1,b_2)&\quad&
    (c\prec f(b_1,b_2))\imp\exists b_1'\ldotp(b_1'\prec b_1)\wedge(c\prec f(b_1',b_2))\\
    (b_2'\prec b_2)\imp f(b_1,b_2')\specord f(b_1,b_2)&\quad&
    (c\prec f(b_1,b_2))\imp\exists b_2'\ldotp(b_2'\prec b_2)\wedge(c\prec f(b_1,b_2'))
  \end{aligned}
  \end{equation}
  Then $f$ determines an approximable mapping $f^*\colon \tIIpre\times \tIIpre\to \tIIpre$, given by
  $f^*(b_1,b_2,b')\iff f(b_1,b_2)<b'$. In turn, $f^*$ determines a continuous morphism of domains
  \[
    F_j\coloneqq\RId_j(f^*)\colon\tII_j\times\tII_j\to\tII_j
  \]
  for any modality $j$, and this morphism preserves constants.

  Exactly the same statements are true for the other $j$-numeric domains, i.e.\ when the predomain
  $\tIIpre$ is replaced by $\tLRpre$, $\tURpre$, $\tIRpre[j]$, $\tLRubpre$, $\tURubpre$,
  $\tIIubpre$, or $\tIRubpre[j]$, and the domain $\tII_j$ is replaced by the corresponding domain of
  $j$-rounded ideals, $\tLR_j$, $\tUR_j$, $\tIR_j$, $\tLRub_j$, $\tURub_j$, $\tIIub_j$, or
  $\tIRub_j$.
\end{proposition}
\begin{proof}
  Both $\tQQ$ and $\tQQ\op$ are linear and unbounded, so the predomain $\tQQ\times\tQQ\op$ is
  rounded and has binary joins given by $(\max,\min)$. The four conditions are exactly those
  required by \cref{prop:functions_approx_mappings}, so have an approximable mapping $f^*\colon
  \tIIpre\times \tIIpre\to \tIIpre$. It is decidable (in the sense of
  \cref{def:predom_const_approx_decidable}) because the $<$-relation on $\tQQ$ is decidable.

  By \cref{prop:approximable_j,prop:RId_product}, we obtain a morphism of
  domains\index{domain!morphism of}
  \[
    \RId_j(f^*)\colon\RId_j(\tIIpre)\times\RId_j(\tIIpre)\to\RId_j(\tIIpre),
  \]
  or simply $\tII_j\times\tII_j\to\tII_j$, since by definition $\RId_j(\tIIpre)=\tII_j$. This map
  preserves constants by \cref{prop:decidable_approx_pres_consts}.
\end{proof}

For any $f\colon\tIIpre\times\tIIpre\to\tIIpre$, we write $\ubar{f}$ and $\bar{f}$ for its first and
second projections via the isomorphism $\tIIpre\to\tLRpre\times\tURpre$. Then the continuous map
$F_j$ from \cref{prop:binary_fn_interval_domains} is given by the formula
$F_j\big((\delta_1,\upsilon_1),(\delta_2,\upsilon_2)\big)\coloneqq(\delta',\upsilon')$ where
\begin{equation}
\label{dia:two_var}
  \begin{aligned}
    \delta'd' &\Leftrightarrow j\exists(d_1,d_2,u_1,u_2:\tQQ)\ldotp
      \Big[ \delta_1d_1 \wedge \delta_2d_2 \wedge \upsilon_1u_1 \wedge \upsilon_2u_2
      \wedge \Big(d' < \ubar{f}\big( (d_1,u_1), (d_2,u_2) \big) \Big) \Big] \\
    \upsilon'u' &\Leftrightarrow j\exists(d_1,d_2,u_1,u_2:\tQQ)\ldotp
      \Big[ \delta_1d_1 \wedge \delta_2d_2 \wedge \upsilon_1u_1 \wedge \upsilon_2u_2
      \wedge \Big( \bar{f}\big( (d_1,u_1), (d_2,u_2) \big) < u' \Big) \Big]
  \end{aligned}
\end{equation}
Indeed, this follows directly from \cref{eqn:function_RId(H)_j}.

We include one more technical lemma here, even though it will not be used again until \cref{sec:Real_numbers}. The reader can feel free to skip to \cref{thm:addition_is_continuous}.

\begin{lemma}\label{prop:Fj1j2j3}
  Let $\phi_1,\phi_2,\phi_3:\Prop$ be such that $(\phi_1\vee\phi_2)\imp\phi_3$, and let $j_1$,
  $j_2$, and $j_3$ be the corresponding closed modalities. Let
  $f:\tIIpre\times\tIIpre\to\tIIpre$ be a function satisfying the four conditions \cref{eqn:four_conditions_approximable} and the following ``density'' condition:
  \[\forall(c:\tIIpre)\ldotp\exists(b_1,b_2:\tIIpre)\ldotp c\prec f(b_1,b_2).\]
  Then there is an induced continuous function
  $F_{j_1,j_2,j_3}:\tII_{j_1}\times\tII_{j_2}\to\tII_{j_3}$ and it preserves constants.
\end{lemma}
\begin{proof}
  This follows from \cref{prop:binary_fn_interval_domains,prop:functions_approx_mappings,prop:j1j2j3j'_decidable_const}.
\end{proof}

\subsubsection{Addition}\index{arithmetic!addition}\index{arithmetic!interval|(}
As mentioned above, the usual addition of improper intervals is defined coordinate-wise,
$(d_1,u_1)+(d_2,u_2)=(d_1+d_2,u_1+u_2)$, where $d_1,u_1,d_2,u_2:\tQQ$ are rational numbers.

\begin{theorem}\label{thm:addition_is_continuous}\index{arithmetic!as continuous}
  For any modality $j$, addition defines a continuous function
  $+_j\colon\tII_j\times\tII_j\to\tII_j$, which preserves constants. The formula for addition
  reduces to $(\delta_1,\upsilon_1)+(\delta_2,\upsilon_2)=(\delta',\upsilon')$, where
  \begin{equation}\label{eqn:addition}
    \begin{aligned}
    \delta'q &\Leftrightarrow j\exists(q_1,q_2:\tQQ)\ldotp
      \delta_1q_1 \wedge \delta_2q_2 \wedge (q<q_1+q_2) \\
    \upsilon'q &\Leftrightarrow j\exists(q_1,q_2:\tQQ)\ldotp
      \upsilon_1q_1 \wedge \upsilon_2q_2 \wedge (q_1+q_2<q)
    \end{aligned}
  \end{equation}
  A similar statement holds when $\tII_j$ is replaced with any of the other Dedekind $j$-numeric
  types.
\end{theorem}
\begin{proof}
  Working in the predomain $\tLRpre=(\tQQ,<)$, it is clear that if $q_1<q_2$ then $q_1+q'\leq
  q_2+q'$ and $q'+q_1\leq q'+q_2$ for any $q'$. It is also clear that if $q'<q_1+q_2$ then there
  exists some $q_1'<q_1$ such that $q'<q_1'+q_2$ and similarly for $q_2$. Thus the conditions of
  \cref{prop:binary_fn_interval_domains} are satisfied, and this defines an addition operation for
  eight of the ten Dedekind $j$-numeric types. It is easy to see that \cref{eqn:addition} agrees
  with \cref{dia:two_var}.

  It remains to check that if $x_1$ and $x_2$ are $j$-located, then so is their sum. So choose
  rationals $d<u$; we want to show $j((d<x_1+x_2)\vee(x_1+x_2<u))$. By the arithmetic locatedness of
  $x_1$ (\cref{prop:arithmetic_located_j}), we have
  \[
    j\exists(d_1,u_1:\tQQ)\ldotp(d_1<x_1<u_1)\wedge(0<u_1-d_1<u-d).
  \]
  Then $d-d_1<u-u_1$ so there exists $d_2, u_2$ such that $d-d_1<d_2<u_2<u-u_1$, and hence
  $d<d_1+d_2$ and $u_1+u_2<u$. By locatedness of $x_2$, we have $j(d_2<x_2\vee x_2<u_2)$, hence
  \begin{multline*}
    j[(\exists(d_1,d_2:\tQQ)\ldotp (d_1<x_1) \wedge (d_2<x_2) \wedge (d<d_1+d_2)) \vee {} \\
    (\exists(u_1,u_2:\tQQ)\ldotp (x_1<u_1) \wedge (x_2<u_2) \wedge (u_1+u_2<u))]
    \qedhere
  \end{multline*}
\end{proof}

\subsubsection{Subtraction}

In \cref{ex:predomain_multiplication} we defined approximable mappings for the difference, the max,
and the product of two improper intervals. The formula for subtraction reduces to
$(\delta_1,\upsilon_1)-_j(\delta_2,\upsilon_2)=(\delta',\upsilon')$, where
\begin{equation}\label{eqn:subtraction}\index{arithmetic!subtraction}
\begin{aligned}
  \delta'q&\iff j\exists(d,u:\tQQ)\ldotp\delta_1d\wedge\upsilon_2u\wedge(q<d-u)\\
  \upsilon'q&\iff j\exists(d,u:\tQQ)\ldotp\upsilon_1u\wedge\delta_1d\wedge(u-d<q).
\end{aligned}
\end{equation}

\begin{remark}
  Note that this notion of subtraction extends the classical one for the set $\IR$ of proper
  intervals, where one thinks of $[d,u]$ as representing the set $\{x\mid d\leq x\leq u\}$. Then the
  difference is given by
  \[
    [d_1,u_1]-[d_2,u_2]=[d_1-u_2,u_1-d_2],
  \]
  which is the set of all pairwise differences. However, note that $\IR$ \emph{does not form a
  group}, and neither does $\tIR$! Subtraction is defined but is not inverse to addition.

  Often the classical extended interval domain $\II$ is used to correct this problem, embedding
  $\IR$ into something that does form a group, where for example additive inverse of $[d,u]$ is
  $[-d,-u]$, sending proper intervals to improper intervals. However this operation is not
  constructively well-behaved; there does not seem to be an analogous operation on $\tII$ in a
  topos. Thus we use the definition given in \cref{eqn:subtraction}, as it is the most convenient
  for us, and simply supply the above warning. Note that subtraction is inverse to addition for
  real numbers: $(\tRR_j,+_j,0)$ forms an abelian group for any $j$, as shown in
  \cref{cor:reals_group}.
\end{remark}

\subsubsection{Maximum}

The formula for the max function given in \cref{ex:predomain_multiplication} reduces to
$\max\!_j((\delta_1,\upsilon_1),(\delta_2,\upsilon_2))=(\delta',\upsilon')$, where
\begin{equation}\label{eqn:maximization}\index{arithmetic!max function}
\begin{aligned}
  \delta'q &\iff j\exists(q_1,q_2:\tQQ)\ldotp
    \delta_1q_1 \wedge \delta_2q_2 \wedge (q<\max(q_1,q_2)) \\
  \upsilon'q &\iff j\exists(q_1,q_2:\tQQ)\ldotp
    \upsilon_1q_1 \wedge \upsilon_2q_2 \wedge (\max(q_1,q_2)<q).
\end{aligned}
\end{equation}
However, we can do better.

\begin{proposition}\label{prop:max_simple}
  We have the following equivalences:
  \[
    q<\max\!_j(x_1,x_2)\iff j(q<x_1\vee q<x_2)
    \qquad\text{and}\qquad
    \max\!_j(x_1,x_2)<q\iff x_1<q\wedge x_2<q
  \]
\end{proposition}
\begin{proof}
  Begin with the first statement. For the converse direction, suppose $q<x_1$, the case $q<x_2$
  being similar. Then by roundedness $j\exists q_1\ldotp q<q_1<x_1$, so for any $q_2<x_2$ we have
  $q<\max(q_1,q_2)$. For the forwards direction, suppose $q<\max\!_j(x_1,x_2)$ in the sense of
  \cref{eqn:maximization}. Then there exists $q_1,q_2$ with $q_1<x_1$ and $q_2<x_2$ and
  $q<\max(q_1,q_2)$. Then either $q<q_1$ or $q<q_2$, and we are done. The second statement is
  similarly straightforward, but uses that $j(P\wedge Q)=jP\wedge jQ$.
\end{proof}

\subsubsection{Multiplication}

The formula for multiplication is more involved, but again it comes directly from
\cref{ex:predomain_multiplication}. It is given by
$(\delta_1,\upsilon_1)*_j(\delta_2,\upsilon_2)=(\delta',\upsilon')$, where
\begin{equation}\label{eqn:multiplication}\index{arithmetic!multiplication}
  \begin{aligned}
    \begin{split}
      \delta'q \iff j\exists(d_1,d_2,u_1,u_2:\tQQ)\ldotp
        \delta_1d_1 \wedge \delta_2d_2 \wedge \upsilon_1u_1 \wedge \upsilon_2u_2 \wedge \hspace{1.2in}\\
      q<\max(d_1^+ d_2^+,u_1^- u_2^-)-\max(u_1^+ d_2^-, u_2^+ d_1^-) \\
    \end{split}\\
    \begin{split}
      \upsilon'q \iff j\exists(d_1,d_2,u_1,u_2:\tQQ)\ldotp
        \delta_1d_1 \wedge \delta_2d_2 \wedge \upsilon_1u_1 \wedge \upsilon_2u_2 \wedge \hspace{1.2in}\\
      \max(u_1^+ u_2^+,d_1^- d_2^-)-\max(d_1^+ u_2^-, u_1^-d_2^+)<q
    \end{split}
  \end{aligned}
\end{equation}
where $q^+=\max(q,0)$ and $q^-=\max(-q,0)$ for any $q:\tQQ$. It may be helpful to note that $q=q^+
-q^-$.

The following lemma---which simplifies the above multiplication formula in the case of positive
proper intervals---is straightforward to prove.

\begin{lemma}\label{lemma:product_of_positives}
  Suppose $(\delta_1,\upsilon_1)$ and $(\delta_2,\upsilon_2)$ are each $j$-disjoint. If they are also positive, i.e.\
  $\delta_10$ and $\delta_20$, then their product is $(\delta_1,\upsilon_1)*_j(\delta_2,\upsilon_2)
  = (\delta',\upsilon')$ where
  \begin{align*}
    \delta'q &\iff j\exists(d_1,d_2:\tQQ)\ldotp
      \delta_1d_1 \wedge \delta_2d_2 \wedge (q<d_1*d_2) \wedge (0<d_1)\wedge (0<d_2) \\
    \upsilon'q &\iff j\exists(u_1,u_2:\tQQ)\ldotp \upsilon_1u_1\wedge\upsilon_2u_2\wedge (u_1*u_2<q)
  \end{align*}
\end{lemma}

\begin{theorem}\label{thm:mult_is_continuous}\index{arithmetic!as continuous}
  For any modality $j$, multiplication defines a continuous function
  $*_j\colon\tII_j\times\tII_j\to\tII_j$, which preserves constants. A similar statement holds when
  $\tII_j$ is replaced with any of the other two-sided $j$-numeric types, $\tIR_j$, $\tRR_j$,
  $\tIIub_j$, $\tIRub_j$, or $\tRRub_j$.
\end{theorem}
\begin{proof}
  The first statement follows from \cref{prop:binary_fn_interval_domains} and from
  \cref{ex:predomain_multiplication}, where it is shown that all of the functions whose composite is
  \cref{eqn:multiplication} are approximable mappings between predomains, e.g.\ $(q\mapsto
  q^+):\tIIpre\to\tIIpre$. The statements for $\tIR_j$, $\tIIub_j$, and $\tIRub_j$ follow similarly,
  replacing $\tIIpre$ with $\tIRpre[j]$, $\tIIubpre$, and $\tIRubpre[j]$ respectively. It remains to
  prove the result for $\tRR_j$ and $\tRRub_j$. For this it suffices to show that if
  $x_1,x_2:\tRRub_j$ are located then so is $y\coloneqq x_1*_jx_2$. Take any rationals $d<u$.

  As a first step, we prove $j(d<y\vee y<u)$ under two preliminary assumptions: that $0<x_1$ and
  that $0<d$. By the multiplicative locatedness of $x_1$ (\cref{prop:multiplicatively_located})
  \[
    j\exists(d_1,u_1:\tQQ)\ldotp 0<d_1<x_1<u_1\wedge\left(\frac{u_1}{d_1}<\frac{u}{d}\right).
  \]
  Choose $q, q':\tQQ$ such that $u_1d<q<q'<ud_1$ and define $d_2\coloneqq\frac{q}{d_1u_1}$ and
  $u_2\coloneqq\frac{q'}{d_1u_1}$. So $d<d_1d_2$, $d_2<u_2$, and $u_1u_2<u$, and by locatedness of
  $x_2$, we have $j(d_2<x_2\vee x_2<u_2)$. Hence,
  \begin{multline*}
    j[(\exists(d_1,d_2:\tQQ)\ldotp(d_1<x_1)\wedge(d_2<x_2)\wedge(d<d_1d_2))\vee\\
    (\exists(u_1,u_2:\tQQ)\ldotp(x_1<u_1)\wedge(x_2<u_2)\wedge(u_1u_2<u))]
  \end{multline*}
  giving $j(d<y\vee y<u)$ as desired.

  As a second step we retain the assumption $0<x_1$, but we drop the second assumption and consider
  general $d<u$. Either $d<0$ or $0<u$. If $0<u$ then there exists $d'$ with $\max(0,d)<d'<u$, so
  $j(d'<y\vee y<u)$ by the first step, and hence $j(d<y\vee y<u)$. If $d<0$ then $-u<-d$ and, just
  as before, there exists $d'$ with $\max(0,-u)<d'<-d$, so $j(d'<-y\vee -y<-d)$ by the first step,
  so $j(d<y\vee y<u)$.

  Finally, dropping all preliminary assumptions, take any located $x_1$ and $x_2$; we will show
  $d<y\vee y<u$ for $y\coloneqq x_1*_jx_2$. Let $m\coloneqq\max(-d,u)$, so since $d<0\vee 0<u$ we
  have $0<m$, and hence by locatedness we have
  \[
    (-1<x_1\vee x_1<0)\wedge(0<x_1\vee x_1<1)\wedge(-m<x_2\vee x_2<0)\wedge(0<x_2\vee x_2<m).
  \]
  This implies $(x_1<0)\vee(0<x_1)\vee(x_2<0)\vee (0<x_2)\vee(|x_1|<1\wedge|x_2|<m)$. The first four
  cases are covered by the second step, above. In the last case $-1<x_1<1$ implies $-|x_2|<y<|x_2|$,
  and $|x_2|<m$ implies $|x_2|<-d\vee |x_2|<u$, hence $(d<-|x_2|<y)\vee(y<|x_2|<u)$.
\end{proof}

\subsubsection{Division}

Division is just multiplication by a reciprocal. We next define a reciprocal function $r\mapsto 1/r$ for positive $j$-local reals $r>0$. It is easily extended to a reciprocal for negative $j$-reals $r<0$ by $r\mapsto -(1/(-r))$.

Consider the
predomain $(\tQQ^{>0},<)$, its opposite $(\tQQ^{>0},>)$, and the product $(\tII^{>0})_{\tn{pre}}$.
There is a function $\tIIpre^{>0}\to\tIIpre^{>0}$ sending $(d,u)$ to $(1/u,1/d)$, and it induces an
approximable mapping by \cref{prop:func_to_approx_map}. Thus for any
$(\delta,\upsilon):\tII^{>0}_j$, define the reciprocal to be $(\delta',\upsilon'):\tII^{>0}_j$,
where for a positive rational $q>0$ we put
\begin{equation}\label{eqn:reciprocal}\index{arithmetic!reciprocal}
  \delta'q\iff j\exists(u:\tQQ)\ldotp \upsilon u\wedge(q*u<1)
  \qquad\text{and}\qquad
  \upsilon'q\iff j\exists(d:\tQQ)\ldotp\delta d\wedge(1<q*d)
\end{equation}

\begin{lemma}\label{lemma:reciprocals}
  $(0<q<1/x)\iff x<1/q$, for any $x:\tII_j$ and $q:\tQQ$.
\end{lemma}
\begin{proof}
  Since $x$ is $j$-closed, it suffices to show $j(x<1/q)$. By definition, $(0<q<1/x)\iff
  j\exists(u:\tQQ)\ldotp(x<u)\wedge (q*u<1)$, and we finish with a chain of equivalences
  \[
    j\exists(u:\tQQ)\ldotp(x<u)\wedge (q*u<1)
    \quad\iff\quad
    j\exists(u:\tQQ)\ldotp(x<u)\wedge (u<1/q)
    \quad\iff\quad
    j(x<1/q).\qedhere
  \]
\end{proof}

\begin{proposition}\label{prop:reciprocal}
  If $x:\tII^{>0}_j$ is $j$-disjoint then $1/x$ is $j$-disjoint; and if $x$ is $j$-located then
  $1/x$ is $j$-located. Moreover, for any positive real $r:\tRR_j^{>0}$ we have $r*r'=1$, where
  $r'\coloneqq(\delta',\upsilon')$ is as in \cref{eqn:reciprocal}.
\end{proposition}
\begin{proof}
  Let $(\delta',\upsilon')$ be the reciprocal of $x$. For $j$-disjointness, suppose $0<q$ and
  $\delta'q\wedge\upsilon' q$. Then $j\exists d,u\ldotp\delta d\wedge\upsilon u\wedge q*u<1<q*d$,
  and we obtain $j\bot$. For $j$-locatedness, suppose $q_1<q_2$. Then $1/q_2<1/q_1$, so
  $\delta(1/q_2)\vee\upsilon(1/q_1)$, and the result follows by \cref{lemma:reciprocals}.

  For the final claim, $1:\tRR_j$ is the $j$-local real represented by $1:\tQQ$; see
  \cref{eqn:q_cuts}. We first show that if $x$ is $j$-disjoint then $q<x*(1/x)\imp q<1$:
  \begin{align}
    q<x*(1/x)\nonumber
    &\iff j\exists q_1,q_2\ldotp(0<q_1<x)\wedge(0<q_2<1/x)\wedge(q<q_1*q_2)\\\label{eqn:equiv_recip}
    &\iff j\exists q_1,q_2\ldotp(0<q_1<x)\wedge(x<1/q_2)\wedge (q<q_1*q_2)\\\nonumber
    &\imp j\exists q_1,q_2\ldotp(0<q_1<1/q_2)\wedge q<q_1*q_2\\\nonumber
    &\iff j(q<1)
  \end{align}
  The first equivalence is \cref{lemma:product_of_positives}, the second is
  \cref{lemma:reciprocals,rem:j-logic}, the implication is by $j$-disjointness, and the fourth is an
  easy calculation. It remains to prove that if $r:\tRR_j$ is real then $q<1\imp q<x*(1/x)$; assume
  $q<1$. By \cref{prop:multiplicatively_located}, we have
  \[
    j\exists q_1,q_2\ldotp(0<q_1<r<q_2)\wedge(q<q_1/q_2),
  \]
  which is equivalent to \eqref{eqn:equiv_recip} when $x\coloneqq r$, so we obtain $q<r*(1/r)$.
\end{proof}

One must be careful to note that for numeric types other than $\tRR$, the distributive law does not
hold. For example, if $x=[0,1]$, $y=[1,1]$, and $z=[-1,-1]$, then
\begin{align*}
  x*(y+z)
  =[0,1]*([1,1] + [-1, -1])
  &= [0,0]\\
  &\neq [-1,1]
  =[0,1]+[-1,0]
  =x*y+x*z.
\end{align*}
\index{arithmetic!interval|)}
However, it is well-known that the real numbers form a commutative ring in any topos.

\begin{proposition}\label{prop:field}\index{real numbers!field of}
  The type $\tRR_j$ of $j$-local real numbers forms a commutative ring satisfying the additional
  axiom
  \begin{equation}\label{eqn:j-field}
    \forall(r:\tRR_j)\ldotp \big(j(r<0\vee r>0)\iff j\exists(r':\tRR_j)\ldotp r*_jr'=1\big).
  \end{equation}
\end{proposition}
\begin{proof}
  In any topos, including the subtopos $\cat{E}_j$, the Dedekind real numbers object $R$ forms a
  residue field, in particular a commutative ring $(R,0,+,1,*)$ satisfying $\forall(r:R)\ldotp
  (r<0\vee r>0)\iff\exists(r':R)\ldotp r*r'=1$. A constructive proof can be found in
  \cite[Proposition D.4.7.10]{Johnstone:2002a}.
\end{proof}

\begin{proposition}\label{prop:j_on_intervals}
  For any morphism of modalities $j'\imp j$, there is an induced continuous map
  \begin{align*}
  j\colon\tRR_{j'}\to\tRR_{j}
  \end{align*}
  defined by sending $(\delta,\upsilon)$ to $(j\delta,j\upsilon)$, and it preserves constants, arithmetic, and inequalities. The analogous statement holds for each of the other nine Dedekind numeric types.
\end{proposition}
\begin{proof}
The continuity of the above-defined $j$ follows from \cref{prop:predomain_modality_map}. More explicitly, one can check that each of the $j'$-axioms for $\delta$ (or $\upsilon$) immediately implies the corresponding $j$-axiom for $j\delta$ (or $j\upsilon$). The map is continuous because it is the restriction of a morphism of domains.

The fact that $j$ preserves constants follows from \cref{prop:j_preserves_constants}. The fact that it preserves arithmetic follows from \cref{prop:approximable_j_square}, because all arithmetic operations are induced by approximable mappings. For inequalities, suppose $a<_{j'}b$ in the sense of \cref{eqn:j_less_than}, i.e.\ $j'\exists q\ldotp\upsilon_aq\wedge\delta_bq$. This implies $j\exists q\ldotp j\upsilon_aq\wedge j\delta_bq$, as desired.
\end{proof}

\index{arithmetic|)}

\index{type theory|(}
\index{logic|)}

\chapter{Axiomatics}\label{sec:axiomatics}\index{logic|(}

In \cref{sec:logical_prelims} we explained a connection between toposes, type theory, and logic. We also discussed modalities and numeric types in an arbitrary topos. In the current chapter we will lay out the signature---meaning the atomic types, atomic terms, and axioms---for our specific topos, $\BaseTopos$. It turns out that our signature consists of no atomic types, one atomic term, and ten axioms.

The real numbers and other numeric objects are particularly relevant in our study of $\BaseTopos$. We introduced Dedekind $j$-numeric objects in \cref{def:local_reals}, where the $j$ signifies an arbitrary modality. We will use that generality in \cref{sec:Real_numbers}, but not in the axioms. That is, in the present chapter, none of our numeric objects will be decorated by a $j$.

\begin{notation}
Recall from \cref{def:local_reals} that by  $\tII$, $\tRR$, and $\tRRub$, we mean the Dedekind numeric types $\tII_j$, $\tRR_j$, and $\tRRub_j$, where $j=\id$ is the identity modality. These numeric types are called the \emph{improper intervals}, the \emph{real numbers}, and the \emph{unbounded real numbers}, respectively.\index{numeric type}
\end{notation}

One reason that these Dedekind numeric types are so important here is that $\Time$ is defined in terms of them. Semantically speaking, a behavior of type $\Time$ assigns to every duration $\ell$ some interval $[t_0,t_0+\ell]$ of real numbers: the values shown on a given clock during that interval. In other words, a behavior of type $\Time$ is a ``time window''.\index{time window}\index{Time@$\Time$!as clock behavior} We will introduce the type $\Time$ in \cref{sec:first_axioms_time}, but we begin in \cref{sec:constant_types} with a discussion of what we call constant types.

Note: Although we will write our proofs in typical mathematical style, it is important to make note of the fact that \emph{every proof and definition in \cref{sec:axiomatics} can be formalized in the type theory of \cref{sec:logical_prelims}}.

\section{Constant types}\label{sec:constant_types}\index{type!constant|(}

The notion of ``constant types'' plays an important role in these axiomatics. The word ``constant'' does not really refer to the type itself, but to its intended semantics.

Rather than defining a
property of constancy for types, we consider it an axiomatic property. Specifically, whenever in an
axiom we say ``for all constant types $\tConst$'', we regard the axiom as an axiom \emph{schema}.

Recall that a proposition $P:\Prop$ is called
\emph{decidable}\index{decidable}\index{proposition!decidable|see {decidable}} if $P\vee\neg P$
holds. We call $P:X\to\Prop$ a \emph{decidable predicate}\index{predicate, decidable|see {decidable}} if $P(x)$
is decidable for all $x:X$. If $P$ is a decidable predicate, we also refer to the subtype $\{\,x:X\mid P(x)\,\}$ as a \emph{decidable subtype}.\index{type!subtype!decidable}\index{type!product}

\begin{definition}[Constant types]\label{def:constant_types}
  We define a \emph{constant type} to be any member of the smallest collection of types containing
  $\unit$, $\tNN$, $\tZZ$, $\tQQ$, and $\tRRub$, and that is closed under finite sums, finite
  products, function types, and decidable subtypes.
\end{definition}

We will see in \cref{cor:decidable_reals} that $\tRR$ is also a constant type. This is just a definition; the axioms throughout this chapter will provide all the associate type-theoretic meaning.

\begin{remark}
  It may come as a bit of a surprise to experts in topos theory that the real numbers object $\tRR$
  is semantically a constant sheaf in $\BaseTopos=\Shv{\BaseSite}$. We prove this in
  \cref{cor:Dedekind_cut_semantics}, during our discussion of soundness.

  For the time being, we can only give some evidence for the claim, namely that the type of real
  numbers object in the slice topos $\Shv{S_\IR}$ is constant (see \cref{prop:IR_BaseTopos/Time}).
  Indeed, for any topological space $X$, the internal sheaf of Dedekind reals is externally the
  sheaf of continuous real-valued functions on X, and these are all constant for $X=\IR$, as we saw
  in \cref{cor:sheaf_IR_Dedekind}.
\end{remark}

\subsection{First axioms for constant types}

We now begin laying out the axioms and their more immediate consequences. Our first axiom says that
the cuts $\delta,\upsilon$ defining any real number are decidable.\index{cuts!decidable} The
soundness of \cref{ax:decidable_reals} is proven in \cref{prop:axiom1}.

\setcounter{axiom}{0}
\begin{axiom}\label{ax:decidable_reals}\index{axiom!of $\BaseTopos$!decidable reals}
  $\forall ((\delta,\upsilon):\tRRub)(q:\tQQ)\ldotp
    (\delta q \vee \neg\delta q) \wedge (\upsilon q \vee \neg\upsilon q)$.
\end{axiom}

We can now begin proving some propositions. As mentioned above, every proof in \cref{sec:axiomatics}---though written in typical mathematical style---can be formalized in a standard type theory and higher-order logic, as outlined in \cref{sec:logical_prelims}.

\begin{proposition}[Trichotomy]
\label{prop:trichotomy}\index{trichotomy}
  For all $r_1,r_2:\tRRub$, either $r_1<r_2$ or $r_1=r_2$ or $r_1>r_2$. In particular
  $(r_1=r_2)\vee\neg(r_1=r_2)$.
\end{proposition}
\begin{proof}
  Let $r=(\delta,\upsilon)\coloneqq r_2-r_1$; it suffices to show that $(\delta
  0)\vee(0=r)\vee(\upsilon 0)$. By \cref{ax:decidable_reals} it suffices to show $0=r$ assuming
  $\neg\delta 0$ and $\neg\upsilon 0$. It follows easily that for all $q:\tQQ$ we have
  $(q<0)\iff\delta q$ and $(0<q)\iff\upsilon q$, so $0=r$. The second statement is clear because
  $(r<0)\imp\neg(0=r)$.
\end{proof}

The next axiom is referred to by \cite[C.1.1.16.e]{Johnstone:2002a} as a ``dual Frobenius rule''.\index{dual Frobenius rule} It also is related to the notion of ``covertness'' for locales; one may say that any constant object is covert\index{covertness}. Kuratowski-finite objects in a topos also satisfy this condition; see \cite[Lemma D.5.4.6]{Johnstone:2002a}. The soundness of \cref{ax:distributivity} is proven in \cref{prop:axiom2}.%

\begin{axiom}\label{ax:distributivity}\index{axiom!of $\BaseTopos$!dual Frobenius}
  For any constant type $\tConst$, any $P:\Prop$, and any $P'\colon\tConst\to\Prop$,
  \begin{equation*}
    \left[ \forall(c:\tConst)\ldotp P\vee
    P'(c) \right] \imp \left[ P \vee \forall(c:\tConst)\ldotp P'(c) \right].
  \end{equation*}
\end{axiom}
Note that the converse is straightforward to prove constructively, hence \cref{ax:distributivity}
implies $\left[ \forall(c:\tConst)\ldotp P\vee P'(c) \right] \iff \left[ P \vee
\forall(c:\tConst)\ldotp P'(c) \right]$.

\begin{remark}
  \cref{ax:distributivity} is a co-distributivity property ($\bigwedge_c (P\vee
  P'(c))=P\vee\bigwedge_c P'(c)$),\index{co-distributivity} with universal quantification over a
  constant type taking the role of infinite intersections.
\end{remark}

We say that a type $X$ is \emph{inhabited} if it satisfies
$\exists(x:X)\ldotp\top$.\index{type!inhabited} 

\begin{proposition}\label{prop:const_type_decidably_inhabited}
  For any constant type $\tConst$ and decidable predicate\index{decidable} $P:\tConst\to\Prop$ (satisfying
  $\forall(c:\tConst)\ldotp Pc\vee\neg Pc$), we have 
  \begin{equation}\label{eqn:const_type_dec}
    \exists(c:\tConst)\ldotp Pc\vee\forall(c:\tConst)\ldotp\neg Pc.
  \end{equation}
  It follows immediately that both $\exists(c:\tConst)\ldotp Pc$ and $\forall(c:\tConst)\ldotp Pc$ are decidable. In particular, the proposition that $\tConst$ is inhabited, namely $\exists(c:\tConst)\ldotp\top$,
  is decidable.
\end{proposition}
\begin{proof}
  Because $P$ is decidable, we have $\forall(c:\tConst)\ldotp(\exists(c':\tConst)\ldotp Pc')\vee\neg
  Pc$, and \cref{eqn:const_type_dec} follows from \cref{ax:distributivity}. It follows immediately that $\exists(c:C)\ldotp Pc$ is decidable; since $\neg P$ is decidable, it also follows immediately that $\forall(c:C)\ldotp Pc$ is decidable. The final statement follows by taking $Pc=\bot$.
\end{proof}

Say that a type $A$ \emph{has decidable equality}\index{decidable equality} if $\forall(a,a':A)\ldotp(a=a')\vee\neg(a=a')$ holds.

\begin{proposition}\label{prop:types_decidable_equality}
Let $A$ and $B$ be types. If $A$ and $B$ have decidable equality, then so do the types $A+B$, $A\times B$, and $\{a:A\mid Pa\}$ for any decidable predicate $P:A\to\Prop$.
\end{proposition}
\begin{proof}
The case of
  decidable subtypes is easiest, and we leave this case to the reader. Consider the case of finite products. Let $A$ and $B$ be types with decidable equality, and
  let $x,y:A\times B$ be terms. Then we have $(\pi_A(x) = \pi_A(y)) \vee \neg(\pi_A(x) = \pi_A(y))$
  and simililarly for $\pi_B$. In the case $\pi_A(x)=\pi_A(y)$ and $\pi_B(x)=\pi_B(y)$, we get $x =
  (\pi_A(x),\pi_B(x)) = (\pi_A(y),\pi_B(y)) = y$. For the other three cases, suppose without loss of
  generality that $\neg(\pi_A(x)=\pi_A(y))$. Then if $x=y$, it follows that $\pi_A(x)=\pi_A(y)$
  giving $\bot$. Hence $\neg(x=y)$.

  For the case of finite sums, we make use of \cref{rem:HOL_sum_type}, including the notation $\mathrm{is\_empty}$ and $\mathrm{is\_sing}$. Let $A$ and $B$ be
  types with decidable equality, and let $x,y:A+B$ be terms. By \eqref{eq:HOL_sum_type}, we can
  consider $x$ to be a pair $(x_A,x_B):(A\to\Prop)\times(B\to\Prop)$ such that $(\mathrm{is\_sing}(x_A)
  \wedge \mathrm{is\_empty}(x_B)) \vee (\mathrm{is\_empty}(x_A) \wedge \mathrm{is\_sing}(x_B))$. We similarly
  let $y$ be the pair $(y_A,y_B)$.

  First consider the case $\mathrm{is\_sing}(x_A) \wedge \mathrm{is\_sing}(y_A)$. Then there exist elements
  $a_x,a_y:A$ such that $x_A(a_x)$, $y_A(a_y)$, $\forall(a:A)\ldotp x_A(a)\imp a=a_x$, and
  $\forall(a:A)\ldotp x_Y(a)\imp a=a_y$. By decidability of $A$, we have
  $(a_x=a_y)\vee\neg(a_x=a_y)$. It is now easy to show that $a_x=a_y \imp x=y$ (using propositional
  extensionality) and $\neg(a_x=a_y)\imp\neg(x=y)$.

  In the case $\mathrm{is\_sing}(x_A) \wedge \mathrm{is\_empty}(y_A)$, we have $x_A(a_x)$ and $\neg
  y_A(a_x)$, therefore $\neg(x=y)$. The other two cases are symmetric.
\end{proof}

\begin{corollary}[Constant types have decidable equality]\label{prop:const_decidable_eq}
  If $\tConst$ is a constant type, then it has decidable equality.
\end{corollary}
\begin{proof}
  The type $1$ obviously has decidable equality, and one can prove by induction that $\tNN$ and
  $\tZZ$ do too. The type $\tQQ$ is a quotient of $\tNN\times\tZZ$, under which equality
  $(n,z)=(n',z')$ holds iff $n*z'=n'*z$, which is decidable. And $\tRRub$ has decidable equality by
  \cref{prop:trichotomy}.
  
  By \cref{def:constant_types}, it suffices to check that finite sums, finite products, function types, and
  decidable subtypes of types with decidable equality have decidable equality. This has been done already in \cref{prop:types_decidable_equality} for all but function types. So suppose that $\tConst$ and $\tConst'$ are constant types with decidable equality. Given $f,g:\tConst\to\tConst'$, take $P:\tConst\to\Prop$ to be the decidable predicate $P\coloneqq\lambda(c:\tConst)\ldotp fc=gc$. By \cref{prop:const_type_decidably_inhabited}, $\forall(c:\tConst)\ldotp fc=gc$ is decidable, and the result follows by function extensionality, \cref{ax:propositional_extensionality}.\index{extensionality!function}
\end{proof}

Because constant types are closed under decidable subtypes (see
\cref{def:constant_types}), we have the following corollary to
\cref{ax:distributivity}.

\begin{corollary}\index{decidable}
  For any constant type $\tConst$, any $P:\Prop$, $P':\tConst\to\Prop$, and any $Q:\tConst\to\Prop$
  such that $\forall(c:\tConst)\ldotp Q(c)\vee\neg Q(c)$,
  \[
    [\forall(c:\tConst)\ldotp Q(c)\imp (P\vee P'(c))] \iff
    [P \vee \forall(c:\tConst)\ldotp Q(c)\imp P'(c)].
  \]
\end{corollary}

Recall from \cref{def:local_reals} the terminology of Dedekind numeric objects, including the conditions called ``roundedness, boundedness, disjointness, and locatedness''. We will use these names freely in this section. For example, any $u:\tRRub$ is rounded, meaning $\forall(q:\tQQ)\ldotp q<u\imp\exists(q':\tQQ)\ldotp q<q'<u$.

\begin{proposition}\label{prop:stronger_located}
  Let $d,u,r:\tRRub$ with $d<u$. Then $(d<r)\vee(r<u)$.
\end{proposition}
\begin{proof}
  By definition, $d<u\iff\exists(q:\tQQ)\ldotp d<q<u$. By roundedness there exists $q':\tQQ$ with
  $q<q'<u$. Then by locatedness $(q<r)\vee(r<q')$, which implies the result.
\end{proof}

\Cref{def:constant_types} says that $\tRRub$ is constant; now we can prove that $\tRR$ is constant
also.

\begin{corollary}\label{cor:decidable_reals}
  The type $\tRR$ of Dedekind reals is constant.\index{real numbers!as decidable}
\end{corollary}
\begin{proof}
  By \cref{def:local_reals}, $\tRR=\{(\delta,\upsilon):\tRRub\mid P(\delta,\upsilon)\}$,
  where  $P:\tRRub\to\Prop$ is the boundedness condition:
  $P(\delta,\upsilon)=(\exists(q:\tQQ)\ldotp\delta q)\wedge(\exists(q:\tQQ)\ldotp\upsilon q)$. It
  suffices by \cref{def:constant_types} to show that $P$ is decidable. This follows from
  \cref{prop:const_type_decidably_inhabited,ax:decidable_reals}.
\end{proof}

A function $f:A\to B$ is an \emph{internal surjection}\index{internal!surjection} if it satisfies
$\forall(b:B)\ldotp\exists(a:A)\ldotp f(a)=b$. It is an \emph{internal
injection}\index{internal!injection} if it satisfies $\forall(a,a':A)\ldotp (f(a)=f(a'))\imp
(a=a')$. It is an \emph{internal bijection}\index{internal!bijection} if it satisfies both. We will not need the following, but it provides an example of something we can now prove in the type theory.

\begin{proposition}\label{prop:unbounded_reals_inftys}
  There is an internal bijection $\tRR\sqcup\unit\sqcup\unit\to\tRRub$.
\end{proposition}
\begin{proof}
  Define the function to send $(\delta,\upsilon):\tRR$ to $(\delta,\upsilon):\tRRub$, it sends the
  first copy of $\unit$ to $(\bot,\top)$, or more precisely the pair
  $(\lambda(q:\tQQ)\ldotp\bot,\lambda(q:\tQQ)\ldotp\top)$, and similarly it sends the second copy of
  $\unit$ to $(\top,\bot)$. This is clearly injective since $\tQQ$ is inhabited and the cuts for
  real numbers are bounded. The function is surjective because for every $(\delta,\upsilon):\tRRub$,
  we can decide if $\delta$ or $\upsilon$ is bounded or not by
  \cref{prop:const_type_decidably_inhabited}. If one is unbounded (i.e.\ $\bot$), then the other is
  $\top$ by locatedness.
\end{proof}

\index{type!constant|)}

\subsection{Interlude: how to read the axiomatics section}

Now that the reader has seen some axioms and propositions that follow, he or she may be wondering
about how the authors chose the axioms and consequences they did. Like in any mathematical theory,
the soundness of an axiom or result is deduced from firm principles, whereas the choice of which
axioms and results to announce is decided by history and necessity. One is science, the other is
art.\index{axioms!authors choice of}

The consistency of all of our axioms is proven in \cref{sec:justification}, by showing that they are
sound in the sheaf topos $\BaseTopos=\Shv{\BaseSite}$. The soundness of the lemmas, propositions,
etc.\ in this section then follow from the soundness of higher-order logic for toposes; see e.g.\
\cite{Lambek.Scott:1988a}.

Note that there is little hope that a finite set of axioms would be complete for $\BaseTopos$. Thus
we had to make choices, and the question one might ask is how we did so. Indeed, we carried many
competing criteria in mind while choosing the ten axioms and the consequences we announce. We were
guided throughout by the National Airspace System ``safe-separation'' case study, detailed in
\cref{sec:case_study}, which involves a combination of differential equations, discrete-time
signaling, and time-delays. We chose this example because it was sufficiently complex and
``exterior'' to ourselves as mathematicians. Moreover, when tackling it, we were forced to define
derivatives with respect to our internal notion of $t:\Time$, and we were able to prove the Leibniz
property internally to the logic. Having done so, we gained confidence that our system is indeed
fairly robust.

Our criteria roughly were that the axioms should:
\begin{itemize}
  \item be powerful enough to prove the ``safe-separation property'';
  \item be powerful enough to prove other facts we thought ``should be provable'';
  \item be written at a ``low logical level'', e.g.\ commutation of $\vee$ and $\wedge$;
  \item capture ``simple geometric facts'' about the topos $\BaseTopos$; and
  \item be ``not too numerous'' in number;
\end{itemize}
For example \cref{ax:distributivity} turns out to be quite powerful, it is written at a low logical
level, and it tells us about how constant objects behave in $\BaseTopos$. However, sometimes the above
criteria conflict with each other. For example
\cref{ax:time_not_decidable_d,ax:time_not_decidable_u} are chosen for their power, even though they
are logically a bit convoluted and the geometric facts they convey are not intuitive.

A reader who has understood the rules of higher-order constructive logic (see \cref{sec:HOL}) will have no trouble with the proofs in this
section: they are all simply logical deductions. We try to give hints about the semantics when
possible, but our goal is build up enough theory to internally define derivatives and eventually
succeed in proving the safety property. We now return to these axioms and their results.

\section{Introducing $\Time$}
\label{sec:first_axioms_time}\index{Time@$\Time$}

We gave a semantic definition of $\Time$ as an object in $\BaseTopos$ in \cref{sec:Time}. Below we
will define a certain type---in our internal language---that we also call $\Time$. We will prove
that the first is a valid semantic meaning of the second in \cref{lemma:Time_sub_II}.

$\Time$ will be defined as a subtype of the extended intervals $\tII$ (see \cref{def:local_reals}),
classified by a predicate\index{unit speed}\index{type!of improper intervals}
\[
  \unitspeed: \tII\to\Prop.
\]
This is the one atomic term we need to define our temporal type theory. With it in hand, we define $\Time$ to be the following subtype (see \cref{sec:subtypes_quotients}):
\begin{equation}\label{eqn:define_time}
  \Time \coloneqq \{\,(\delta, \upsilon):\tII \mid \unitspeed(\delta, \upsilon)\,\}.
\end{equation}
Just like we reserved the letters $q$, $r$, etc.\ for variables of type rationals and reals (see
\cref{def:local_reals}), we reserve the letters $t,t',t_1,t_2$, etc.\ for variables of type $\Time$.

\begin{notation}\label{not:time}
We use the usual notation for inequalities comparing times $t=(\delta_t,\upsilon_t):\Time$ with rationals
$q:\tQQ$ and unbounded reals $r=(\delta_r,\upsilon_r):\tRRub$,\index{inequality!involving $\Time$}
\begin{equation}\label{eqn:comparing_time}
\begin{aligned}
  q < t &\coloneqq \delta_t q&\qquad&&r<t&\coloneqq\exists(q:\tQQ)\ldotp \upsilon_r q\wedge\delta_tq\\
  t < q &\coloneqq \upsilon_t q&&&t<r&\coloneqq\exists(q:\tQQ)\ldotp \upsilon_tq\wedge \delta_r q
\end{aligned}
\end{equation}
Thus using our notation, the top right-hand equation could be rewritten $r<t\iff\exists q\ldotp
r<q<t$. We also introduce the following notations for $d,u,r:\tRRub$ and $t:\Time$,
\begin{itemize}
  \item $[d,u] \ll t \coloneqq (d < t) \wedge (t < u)$
  \item $t \apart [d,u] \coloneqq (t < d) \vee (u < t)$
  \item $t \apart r \coloneqq t \apart [r,r]$
\end{itemize}
The symbol $\apart$\index{apart} is pronounced \emph{apart}.
\end{notation}

Note that $t\apart 0$ means $(t<0)\vee(0<t)$ which is strictly stronger than $\neg(t=0)$. Indeed, we will see in \cref{prop:t_not_0} that the latter is automatically true.

\Cref{ax:time_inhabited} ensures that $\Time$ is inhabited (it is equivalent to
$\exists(t:\Time)\ldotp\top$ by
\cref{eqn:define_time}).\index{type!inhabited}\index{Time@$\Time$!as
inhabited} \Cref{ax:time_complementary_cuts} implies the cuts are disjoint, but it is stronger than
disjointness\index{cuts!disjoint}; it can be read as saying ``if the time $t$ is never greater than
$q$, then it is always less than $q$,'' and dually. The soundness of
\cref{ax:time_inhabited,ax:time_complementary_cuts,ax:time_not_decidable_d,ax:time_not_decidable_u}
is shown in \cref{prop:axiom3}.

\begin{axiom}\index{axiom!of $\BaseTopos$!basics of $\Time$}
  \begin{enumerate}[label=3\alph*., ref=3\alph*]
    \item $\exists((\delta,\upsilon):\tII)\ldotp\unitspeed(\delta,\upsilon)$.
      \label[axiom]{ax:time_inhabited}
    \item $\forall(q:\tQQ)((\delta,\upsilon):\Time)\ldotp (\neg\delta q \iff \upsilon q)
      \wedge (\neg\upsilon q \iff \delta q)$.
      \label[axiom]{ax:time_complementary_cuts}
    \item $\forall(t:\Time)(P:\tQQ\to\Prop)\ldotp (\forall q\ldotp Pq\vee\neg Pq) \imp (\forall
      q\ldotp Pq\vee q<t) \imp \exists q\ldotp Pq\wedge q<t$.
      \label[axiom]{ax:time_not_decidable_d}
    \item $\forall(t:\Time)(P:\tQQ\to\Prop)\ldotp (\forall q\ldotp Pq\vee\neg Pq) \imp (\forall
      q\ldotp Pq\vee t<q) \imp \exists q\ldotp Pq\wedge t<q$.
      \label[axiom]{ax:time_not_decidable_u}
  \end{enumerate}
\end{axiom}

\Cref{ax:time_complementary_cuts} can be rewritten less pedantically using
\eqref{eqn:comparing_time}:
\[
  \forall(q:\tQQ)(t:\Time)\ldotp(\neg(q<t)\iff(t<q))\wedge(\neg(t<q)\iff(q<t)).
\]

The primary use of \cref{ax:time_not_decidable_d,ax:time_not_decidable_u} will be to ensure that
many axioms stated in terms of rationals in fact generalize to reals. For a first example, we can
generalize \cref{ax:time_complementary_cuts}.

\begin{proposition}\label{prop:time_complementary_real_cuts}
  For all times $t:\Time$ and unbounded reals $r:\tRRub$
  \[
    \neg(t<r) \iff r<t.
    \qquad\text{and}\qquad
    \neg(r<t) \iff t<r
   \]
\end{proposition}
\begin{proof}
  We prove $\neg(t<r) \iff r<t$. The left side is by definition $\neg(\exists
  q\ldotp(t<q)\wedge(q<r))$. This is constructively equivalent to $\forall
  q\ldotp\neg((t<q)\wedge(q<r))$. Using $(q<r)\vee\neg(q<r)$ from \cref{ax:decidable_reals} this is
  equivalent to $\forall q\ldotp \neg(t<q)\vee\neg(q<r)$, which is equivalent to $\forall
  q\ldotp\neg(q<r)\vee(q<t)$ by \cref{ax:time_complementary_cuts}. Because $q<r$ is decidable by
  \cref{ax:decidable_reals}, we can apply \cref{ax:time_not_decidable_d} to get $\exists
  q\ldotp\neg(q<r)\wedge(q<t)$. Taking such a $q$, there exists a $q'$ with $q<q'<t$ because $t$ is
  rounded ($\Time\ss\tII$), and because $r$ is located we have $(r<q')\vee(q<r)$. Since $\neg(q<r)$,
  we have proven $\exists q'\ldotp (r<q')\wedge(q'<t)$, i.e.\ $r<t$, as desired.
\end{proof}

In particular the axioms for $\Time$ imply that it is nowhere constant.\index{Time@$\Time$!as
nowhere constant}

\begin{proposition}\label{prop:t_not_0}
  $\forall(t:\Time)(r:\tRRub)\ldotp\neg(t=r)$.
\end{proposition}
\begin{proof}
  If $t=r$ then $\neg(t<r)$ and $\neg(r<t)$, which implies $\bot$ by
  \cref{prop:time_complementary_real_cuts}.
\end{proof}

\Cref{prop:time_apartness_complement} says that if $t$ is never between $d$ and $u$, then it is
either always less than $d$ or always greater than $u$.

\begin{proposition}\label{prop:time_apartness_complement}
  For all times $t:\Time$ and reals $d,u:\tRRub$, if $d<u$ then $\neg([d,u]\ll
  t)\iff(t\apart[d,u])$.
\end{proposition}
\begin{proof}
  Let $d<u$. The backwards direction follows from \cref{ax:time_complementary_cuts}, so consider the
  forwards direction. Assume $\neg(\exists(d',u':\tQQ)\ldotp d<d'<t<u'<u)$, we want to show
  $(t<d)\vee(u<t)$. By \cref{ax:decidable_reals,ax:time_complementary_cuts}, the assumption is
  constructively equivalent to $\forall(d',u':\tQQ)\ldotp\neg(d<d'<t)\vee\neg(t<u'<u)$. Using
  \cref{ax:distributivity}, we can split this up into $(\forall d'\ldotp\neg(d<d'<t))\vee(\forall
  u'\ldotp\neg(t<u'<u))$. By \cref{prop:time_complementary_real_cuts} this is equivalent to
  $(t<d)\vee(u<t)$.
\end{proof}

\begin{proposition}\label{prop:apartness_continuity}
  For all $t:\Time$ and $r,r':\tRRub$, there is an equivalence between $t\apart[r,r']$ and
  $\forall(q,q':\tQQ)\ldotp ((r<q) \wedge (q'<r')) \imp t\apart[q,q']$.
\end{proposition}
\begin{proof}
  The proposition is proven by the following series of equivalences:
  \begin{alignat*}{2}
    &\forall(q,q':\tQQ)\ldotp ((r<q) \wedge (q'<r')) \imp t\apart[q,q'] && \\
    &\forall(q,q':\tQQ)\ldotp \neg(r<q) \vee \neg(q'<r') \vee (t<q) \vee (q'<t)
      \qquad && \text{\cref{ax:decidable_reals}} \\
    &(\forall q\ldotp \neg(r<q)\vee(t<q)) \vee (\forall q'\ldotp\neg(q'<r')\vee(q'<t))
      \qquad && \text{\cref{ax:distributivity}} \\
    &(t<r)\vee(r'<t) \qquad && \text{\cref{ax:time_not_decidable_d,ax:time_not_decidable_u}}
      \qedhere
  \end{alignat*}
\end{proof}

Recall from \cref{thm:monoids} that, even without any axioms introduced in this chapter, $\tRR$
forms a group under $+$. The next axiom encodes the
essential connection between $\Time$ and $\tRR$, namely that $\Time$ is an
``$\tRR$-torsor.''\index{torsor!$\Time$ as} This simply means that we can add a real number to a
time to get a new translated time, that given any two times we can compute the difference between
them as a real number, and moreover that these two operations are
inverse.\index{arithmetic!involving $\Time$} The soundness of \cref{ax:torsor} is proven in
\cref{prop:axiom4}.

\begin{axiom}\label{ax:torsor}\index{axiom!of $\BaseTopos$!$\Time$ as torsor}
  $\Time$ is an $\tRR$-torsor:
  \begin{enumerate}[label=4\alph*.]
    \item \label[axiom]{ax:torsor_act} $\forall(t:\Time)(r:\tRR)\ldotp t+r\in\Time$
    \item \label[axiom]{ax:torsor_diff}$\forall(t_1,t_2:\Time)\ldotp\exists!(r:\tRR)\ldotp
      t_1+r=t_2.$
  \end{enumerate}
\end{axiom}

\begin{proposition}\label{prop:torsor_for_comparison}
  For any $t:\Time$ and $r,r':\tRR$, it is the case that $(t+r<r')\iff (t<r'-r)$.
\end{proposition}
\begin{proof}
  It follows from definitions (\cref{eqn:comparing_time,eqn:addition_unbounded_improper}) that,
  $t+r<r'$ iff $\exists(q_1,q_2:\tQQ)\ldotp(t<q_1)\wedge(r<q_2)\wedge (q_1+q_2<r')$. This holds iff
  $\exists(q_1',q_2')\ldotp(t<q_1+q_2')\wedge(q_1'<s)\wedge(r<-q_2')$, where $q_1'=q_1+q_2$ and
  $q_2'=-q_2$. The latter holds iff $t<r'+(-r)$; see \cref{cor:reals_group}.
\end{proof}

\begin{proposition}\label{prop:time_decidable_eq}
$\Time$ has decidable equality.
\end{proposition}
\begin{proof}
Combine \cref{ax:torsor}, \cref{thm:monoids}, and \cref{prop:const_decidable_eq} applied to the constant type $\tRR$.
\end{proof}

The next axiom states that for any terms $t:\Time$ and $q:\tQQ$, the propositions $t<q$ and $q<t$
are each \emph{coprime}.\index{coprime|see {axiom}} Semantically, this corresponds to the fact that
the corresponding open subsets of the interval domain are filtered (inhabited, up-closed,
down-directed)\index{filter}; see \cref{prop:facts_on_IR}. The soundness of \cref{ax:coprime} is
proven in \cref{prop:axiom5}.

\begin{axiom}\label{ax:coprime}\index{axiom!of $\BaseTopos$!coprime}
  For all $d, u:\tQQ$, all $t:\Time$, and all $P,Q:\Prop$,
  \begin{enumerate}[label=5\alph*.]
    \item $(t<u \imp (P\vee Q)) \imp [(t<u \imp P) \vee (t<u \imp Q)]$\label[axiom]{ax:coprime_u}.
    \item $(d<t \imp (P\vee Q)) \imp [(d<t \imp P) \vee (d<t \imp Q)]$\label[axiom]{ax:coprime_d}.
  \end{enumerate}
\end{axiom}

The converse of each statement is easy. We can immediately generalize \cref{ax:coprime} from $\tQQ$
to $\tRR$.

\begin{proposition}\label{prop:real_coprime}
  For all $d, u:\tRRub$, all $t:\Time$, and all $P,Q:\Prop$,
  \begin{enumerate}
    \item $(t<u \imp (P\vee Q)) \iff [(t<u \imp P) \vee (t<u \imp Q)]$.
    \item $(d<t \imp (P\vee Q)) \iff [(d<t \imp P) \vee (d<t \imp Q)]$.
  \end{enumerate}
\end{proposition}
\begin{proof}
  The two statements are similar, and the backwards direction of each is trivial, so we just
  consider the forwards direction of the first. Choosing $u=(\delta,\upsilon):\tRRub$, we can
  rewrite the hypothesis by undwinding definitions \eqref{eqn:comparing_time} and using
  \cref{ax:coprime}:
  \begin{align*}
    \left[ (\exists q\ldotp t<q \wedge \delta q) \imp (P\vee Q) \right]
    &\iff \forall q\ldotp \left[ (t<q \wedge \delta q) \imp (P\vee Q) \right] \\
    &\iff \forall q\ldotp \left[ \delta q \imp (t<q \imp (P\vee Q)) \right] \\
    &\iff \forall q\ldotp \left[ \delta q \imp ((t<q \imp P) \vee (t<q \imp Q)) \right]
  \end{align*}
  We can similarly rewrite the conclusion:
  \begin{align*}
    &\left[ (\exists q\ldotp t<q \wedge \delta q) \imp P \right] \vee
      \left[ (\exists q\ldotp t<q \wedge \delta q) \imp Q \right] \\
    &\qquad \iff \left[ \forall q\ldotp(t<q \wedge \delta q) \imp P \right] \vee
      \left[ \forall q\ldotp(t<q \wedge \delta q) \imp Q \right] \\
    &\qquad \iff \forall q_1,q_2\ldotp \left[(t<q_1 \wedge \delta q_1) \imp P\right] \vee
      \left[(t<q_2\wedge\delta q_2) \imp Q\right]
  \end{align*}
  where the backwards direction of the last iff is by two applications of \cref{ax:distributivity}.

  We will now prove the desired (reformulated) implication. Let $q_1$ and $q_2$ be arbitrary
  rationals, and assume without loss of generality that $q_1\leq q_2$. By \cref{ax:decidable_reals},
  we have $\delta q_2 \vee \neg\delta q_2$. If $\neg\delta q_2$ then $(t<q_2 \wedge \delta q_2) \imp
  Q$ vacuously. On the other hand, if $\delta q_2$, then by the hypothesis we have $(t<q_2\imp
  P)\vee(t<q_2\imp Q)$. In the second case we're clearly done. In the first case, since we have
  $(t<q_1\wedge\delta q_1)$ implies $t<q_2$ because $q_1\leq q_2$, and this implies $P$.
\end{proof}

\begin{lemma}\label{lem:interval_coprime}
  For all $d,u:\tRRub$, and all $t:\Time$, the proposition $[d, u]\ll t$ is coprime, i.e.\ for all
  $P,Q:\Prop$,
  \[
    ([d,u]\ll t \imp (P\vee Q)) \iff [([d,u]\ll t \imp P) \vee ([d,u]\ll t \imp Q)].
  \]
\end{lemma}
\begin{proof}
  For any $R:\Prop$, $[d,u]\ll t \imp R$ is equivalent to $d<t \imp (t<u \imp R)$. The lemma then
  follows from two applications of \cref{prop:real_coprime}.
\end{proof}

\begin{lemma}\label{lem:time_lems}
  For any $d,u:\tRRub$ and any $t:\Time$, if $d<u$ then
  \begin{enumerate}[label=\alph*.]
    \item $(t<u \imp d<t) \iff d<t,$
    \item $(d<t \imp t<u) \iff t<u,$
    \item $(t<u \imp t<d) \iff (t<d \vee u<t),$
    \item $(d<t \imp u<t) \iff (t<d \vee u<t).$
  \end{enumerate}
  If $d<r<u$ and $d<r'<u$, then
  \begin{enumerate}[resume, label=\alph*.]
    \item $([d,u]\ll t \imp t<r) \iff t\apart[r,u]$,
    \item $([d,u]\ll t \imp r<t) \iff t\apart[d,r]$,
    \item $([d,u]\ll t \imp t\apart[r,r']) \iff t\apart[r,r']$.
    \item $(t<u \imp t\apart[r,r']) \iff t\apart[r,r']$,
    \item $(d<t \imp t\apart[r,r']) \iff t\apart[r,r']$,
  \end{enumerate}
  For any $r$, if $d<u$, then
  \begin{enumerate}[resume, label=\alph*.]
    \item $(t\apart[r,d] \imp t<u) \iff t<u$,
    \item $(t\apart[u,r] \imp d<t) \iff d<t$.
  \end{enumerate}
  If $(r<u)\wedge(r'<u)$ or $(d<r)\wedge(d<r')$, (making no assumption about the order of $d,u$), then
  \begin{enumerate}[resume, label=\alph*.]
    \item $(t\apart[u,d] \imp t\apart[r,r']) \iff t\apart[r,r']$.
  \end{enumerate}
\end{lemma}
\begin{proof}
  The reverse directions are all trivial.

  For the forward direction of a., it suffices by \cref{prop:time_complementary_real_cuts} to show
  that $\neg(t < d)$: supposing $t < d$, then $t < u$ (since $d < u$), which by assumption implies
  $d < t$, and by \cref{prop:time_complementary_real_cuts} again this is a contradiction. The proof
  of b.\ is dual.

  For the forward direction of c., it suffices by \cref{prop:time_apartness_complement} to show that
  $\neg([d,u]\ll t)$: if $d < t < u$ then by assumption $t < d$, a contradiction (by
  \cref{prop:time_complementary_real_cuts}). Similarly for d.

  For the forward direction of e., it suffices by c.\ to show that $t<u\imp t<r$. Assuming $t<u$,
  we have $d<t\imp t<r$, which by b.\ implies $t<r$. Similarly for f.

  For the forward direction of g., suppose $[d,u]\ll t \imp t\apart[r,r']$. By
  \cref{lem:interval_coprime}, we have $([d,u]\ll t \imp t<r) \vee ([d,u]\ll t \imp r'<t)$,
  which by e.\ and f.\ implies $t\apart[r,u] \vee t\apart[d,r']$, or equivalently $t\apart[r,r']$.

  Finally, h.\ and i.\ are immediate from g., using just $([d,u]\ll t)\imp (t<u)$; j.\ follows
  easily from b., and likewise k.\ from a.; and l.\ is immediate from h.\ and i using
  \[
    (t\apart[u,d] \imp t\apart[r,r']) \iff [(t<u \imp t\apart[r,r']) \wedge (d<t \imp
    t\apart[r,r'])]. \qedhere
  \]
\end{proof}

We will also need a sort of infinitary version of \cref{ax:coprime}.\index{coprime!infinitary|see
{axiom}} The soundness of \cref{ax:infinitary_coprime} is proven in \cref{prop:axiom6}.

\begin{axiom}\label{ax:infinitary_coprime}\index{axiom!of $\BaseTopos$!infinitary coprime}
  Let $\tConst$ be an inhabited constant type. Then for all $t:\Time$ and
  $P:\tConst\to\Prop$,
  \begin{enumerate}[label=6\alph*.]
    \item $\forall(u:\tRRub)\ldotp(t<u \imp \exists(c:\tConst)\ldotp P(c)) \imp [\forall(u':\tQQ)\ldotp
      u'<u\imp\exists(c:\tConst)\ldotp t<u'\imp P(c)]$\label[axiom]{ax:infinitary_coprime_u}.
    \item $\forall(d:\tRRub)\ldotp(d<t \imp \exists(c:\tConst)\ldotp P(c)) \imp [\forall(d':\tQQ)\ldotp
      d<d'\imp\exists(c:\tConst)\ldotp d'<t\imp P(c)]$\label[axiom]{ax:infinitary_coprime_d}.
  \end{enumerate}
\end{axiom}

\begin{proposition}\label{prop:converse_ax_infinitary_coprime}
  The converse of \cref{ax:infinitary_coprime_u} and the converse of \cref{ax:infinitary_coprime_d}
  also hold.
\end{proposition}
\begin{proof}
  This follows from the definition of $t<u$, i.e.\ that $\exists (u':\tQQ)\ldotp t<u'<u$.
\end{proof}

\begin{remark}
Note that \cref{ax:infinitary_coprime} without the extra $u'<u$ clause, namely
\[(t<u \imp \exists(c:\tConst)\ldotp P(c)) \imp^? \exists(c:\tConst)\ldotp t<u\imp P(c),\]
does not hold in our semantics. Indeed, let $\tConst\coloneqq\tQQ$ and in the context of some $t,u$ define $P(q)\coloneqq t<q<u$. While it is true that $t<u$ implies $\exists q\ldotp t<q<u$, it is not the case that there exists $q$ with $(t<u)\imp t<q<u$.
\end{remark}

\begin{proposition}\label{prop:real_inf_coprime}
  Let $\tConst$ be an inhabited constant type. Then for all $d,u:\tRRub$, all $t:\Time$, and all
  $P:\tConst\to\Prop$, the following are equivalent:
  \begin{enumerate}
    \item $t<u \imp \exists(c:\tConst)\ldotp P(c)$
    \item $\forall(u':\tQQ)\ldotp u'<u\imp\exists(c:\tConst)\ldotp t<u'\imp P(c)$
    \item $\forall(u':\tRR)\ldotp u'<u\imp\exists(c:\tConst)\ldotp t<u'\imp P(c)$.
  \end{enumerate}
  Likewise, the following are equivalent:
  \begin{enumerate}
    \item $d<t \imp \exists(c:\tConst)\ldotp P(c)$
    \item $\forall(d':\tQQ)\ldotp d<d'\imp\exists(c:\tConst)\ldotp d'<t\imp P(c)$
    \item $\forall(d':\tRR)\ldotp d<d'\imp\exists(c:\tConst)\ldotp d'<t\imp P(c)$.
  \end{enumerate}
\end{proposition}
\begin{proof}
  $3\imp2$ is obvious, $2\iff 1$ is \cref{prop:converse_ax_infinitary_coprime}, and $2\imp 3$ is the
  definition of $u'<u$: there exists $u'':\tQQ$ with $u'<u''<u$.
\end{proof}

\begin{lemma}
\label{lem:inf_interval_coprime}
\index{existential quantifier, commuting with}
  Let $\tConst$ be an inhabited constant type. Then for all $d,u:\tRRub$, all $t:\Time$, and all
  $P:\tConst\to\Prop$, the following are equivalent:
  \begin{enumerate}
    \item $[d,u]\ll t\imp \exists(c:\tConst)\ldotp P(c)$
    \item $\forall(d',u':\tQQ)\ldotp (d<d'\wedge u'<u)\imp\exists(c:\tConst)\ldotp [d',u']\ll
      t\imp P(c)$
    \item $\forall(d',u':\tRR)\ldotp (d<d'\wedge u'<u)\imp\exists(c:\tConst)\ldotp [d',u']\ll
      t\imp P(c)$.
  \end{enumerate}
\end{lemma}
\begin{proof}
  $3\imp 2\imp 1$ is obvious, and $1\imp 3$ follows from two applications of
  \cref{prop:real_inf_coprime}.
\end{proof}

As a sort of dual to \cref{lem:interval_coprime} (saying $[d, u]\ll t$ is coprime), the next axiom
states that for any $d\le u$, the proposition $t\apart[u,d]$ is \emph{prime}.\index{prime} The
soundness of \cref{ax:prime} is proven in \cref{prop:axiom7}.

\begin{axiom}\label{ax:prime}\index{axiom!of $\BaseTopos$!prime}
  For all $d, u:\tQQ$ with $d\le u$ and all $t:\Time$, the proposition $t\apart[u,d]$ is prime,
  i.e.\ for all propositions $P,Q:\Prop$,
  \begin{equation}\label{eqn:primeness}
    \left[ (P\wedge Q) \imp t\apart[u,d] \right] \imp
      \left[ (P \imp t\apart[u,d]) \vee (Q \imp t\apart[u,d]) \right].
  \end{equation}
\end{axiom}

Again, the converse trivially holds, and the statement is immediately generalized from rationals to
reals.

\begin{proposition}\label{prop:prime_reals}
  For all $d,u:\tRRub$ with $d\leq u$ and all $t:\Time$, the proposition $t\apart[u,d]$ is prime,
  i.e.\ \cref{eqn:primeness} holds for all $P,Q:\Prop$.
\end{proposition}
\begin{proof}
  Assuming $(P\wedge Q)\imp t\apart[u,d]$, we want to prove $(P\imp t\apart[u,d])\vee(Q\imp
  t\apart[u,d])$. Using \cref{prop:apartness_continuity}, it follows that $P\imp t\apart[u,d]$ is
  equivalent to $\forall(d',u':\tQQ)\ldotp ((d'<d)\wedge(u<u'))\imp(P\imp t\apart[u',d'])$. This
  combined with \cref{ax:distributivity} shows that the desired conclusion is equivalent to
  \begin{multline}\label{eq:prime_reals_goal}
    \forall(d_1,u_1,d_2,u_2:\tQQ)\ldotp ((d_1<d)\wedge(u<u_1)\wedge(d_2<d)\wedge(u<u_2)) \\
      \imp [(P\imp t\apart[u_1,d_1]) \vee (Q\imp t\apart[u_2,d_2])].
  \end{multline}
  Likewise, the assumption is equivalent to
  \[
    \forall(d',u':\tQQ)\ldotp ((d'<d)\wedge(u<u'))\imp((P\wedge Q)\imp t\apart[u',d']),
  \]
  and applying \cref{ax:prime} to this gives
  \begin{equation}\label{eq:prime_reals_hypo}
    \forall(d',u':\tQQ)\ldotp ((d'<d)\wedge(u<u'))\imp[(P\imp t\apart[u',d']) \vee (Q\imp
    t\apart[u',d'])].
  \end{equation}

  To prove \cref{eq:prime_reals_goal}, suppose we are given rationals $d_1<d$, $u<u_1$, $d_2<d$, and
  $u<u_2$. Let $d'\coloneqq\max(d_1,d_2)$ and $u'\coloneqq\min(u_1,u_2)$. From
  \cref{eq:prime_reals_hypo}, we have $(P\imp t\apart[u',d'])\vee(Q\imp t\apart[u',d'])$, which
  implies the result.
\end{proof}

\section{Important modalities in temporal type theory}
\label{sec:modalities_in_BaseTopos}\index{modality!on $\BaseTopos$}

Recall from \cref{sec:modalities} the notion of modalities $j:\Prop\to\Prop$ in a type theory and their relationship with toposes. In the current section we will specifically be interested in modalities with interesting semantics in $\BaseTopos$, and in the slice topos\index{topos!slice}
$\BaseTopos/\Time$; we call the latter \emph{temporal modalities}\index{modality!temporal}. A
temporal modality\index{modality!temporal} is a map $j:\Time\to(\Prop\to\Prop)$ such that
$j(t)$---or as we will generally denote it, $j^t$---is a modality for all $t:\Time$.

\subsection{Definition of the  modalities $\See{}$, $\At{}$, $\In{}$, and $\pi$}\label{sec:define_See_At_In_pi}

In \cref{sec:modalitites_subtoposes} we discussed three sorts of modalities that arise from any given proposition $U:\Prop$, namely the open modality $o(U)$, the closed modality $c(U)$, and the quasi-closed modality $q(U)$. Below we will introduce some modalities that turn out to be useful in our type theory. Given $d,u:\tRR$ and $t:\Time$, we will use the closed modality $c(t\apart[d,u])$, the quasi-closed modality $q(t\apart[d,u])$, and the open modality $o(t\ll[d,u])$; see \cref{not:time}. These are denoted $\See[t]{[d,u]}$, $\At[t]{[d,u]}$, and $\In[t]{[d,u]}$, respectively. We will also introduce one more modality, that does not depend on $d$, $u$, or $t$.

\begin{notation}[$\See{}$, $\At{}$, $\In{}$, $\pi$]\label{not:our_modalities}\index{notation}~\\
\vspace{-.15in}

  For $t:\Time$ and $d,u:\tRRub$, recall that $t\apart[d,u]=(d<t)\vee (t<u)$ and $[d,u]\ll t=(d<t)\wedge (t<u)$. Given $P:\Prop$, we write
  \begin{itemize}[itemsep=4pt, topsep=5pt]
    \item $\See[t]{[d,u]}P \coloneqq t\apart[u,d] \vee P$
    \item $\At[t]{[d,u]}P \coloneqq (P \imp t\apart[u,d]) \imp t\apart[u,d]$
    \item $\In[t]{[d,u]}P\coloneqq ([d,u]\ll t)\imp P$
    \item $\Pointwise P \coloneqq \forall(t:\Time)\ldotp \At[t]{[0,0]}P$
  \end{itemize}\index{modality!$\At{},\See{},\In{},\pi$}



  Note that $\pi$ is the only one among the above list that is a modality on $\BaseTopos$; the rest
  are temporal modalities, i.e.\ they exist only in the context of some $t:\Time$. One can pronounce
  the notation $\pi P$ as ``Pointwise, $P$ holds'', which means $P$ holds at every $0$-length point.
  Note that we will not use the symbol $\pi$ to indicate projections in this book, so it is not
  overloaded if we use it to denote this modality. 
  
  When there is only one $t:\Time$ in the context, we often drop it from the notation, e.g.\ writing 
  $\See{[d,u]}$ rather than $\See[t]{[d,u]}$. One can pronounce the notation
  $\At{[d,u]}P$ as ``At $[d,u]$, $P$ holds''. The authors often pronounce the notation $\In{[d,u]}P$
  as ``In the interval $[d,u]$, $P$ holds'' and $\See{[d,u]}P$ as ``Seeing the interval $[d,u]$, $P$
  holds''.\index{modality!pronouncing}
  
  When $d=u$, we often simplify the notation by writing $\See{d}\coloneqq\See{[d,d]}$ and $\At{d}\coloneqq\At{[d,d]}$. For example, with this notation, we have $\pi P\coloneqq\forall(t:\Time)\ldotp\At{0}P$.
\end{notation}

\begin{remark}\label{rem:modalities_geom_interesting}
The modalities in \cref{not:our_modalities} correspond semantically to geometrically
interesting subtoposes. We will discuss this in detail in \cref{sec:semantics_of_modalities}, but we give the idea for the time being.

For any $d,u:\tRRub$ with $d\leq u$, there is an associated point $[d,u]\in\IR$; see \cref{sec:def_interval_domain}. This point defines at least three subsets of $\IR$, an open
subspace $\upclose[d,u]\ss\IR$, a closed subspace $\down[d,u]$, and the one-point subspace
$\{[d,u]\}$, which in fact is never closed; in fact, the closure of $\{[d,u]\}$ is $\down[d,u]$. As our notation hopefully makes obvious, the subspaces $\upclose[d,u]$, $\downclose[d,u]$, and $\{[d,u]\}$  correspond to the
modalities $\In{[d,u]}$, $\See{[d,u]}$ and $\At{[d,u]}$ respectively; see also
\cref{sec:modalitites_subtoposes}.

In \cref{not:our_modalities} the above modalities were defined for any $d,u:\tRRub$, regardless of order. When $u<d$ or either is $\pm\infty$, the above ``subspace'' explanation does not apply exactly, but it can be massaged a bit. That is, for any $d,u\in\RRub$, we can consider the subspace $\down[d,u]\ss\IR$ defined as follows:
\begin{equation}\label{eqn:abuse_down[d,u]}
  \down[d,u]\coloneqq\{[d',u']\in\IR\mid(d'\leq d)\wedge(u\leq u')\}\ss\IR
\end{equation}
In other words, if $d\leq u$ then $\down[d,u]$ denotes the set of intervals containing $[d,u]$,
whereas if $u<d$ then $\down[d,u]$ denotes the set of intervals which nontrivially intersect
$[u,d]$.

It remains to consider the semantics of the pointwise modality $\pi$. It defines one of the ``four relevant toposes'' discussed early on, in \cref{sec:four_relevant_toposes}. In the context of $t:\Time$, the $\pi$-modality corresponds to the subspace $\RR\ss\IR$.
\end{remark}

\subsection{Interlude on the semantics of $\See{}$, $\At{}$, $\In{}$, and $\pi$}
\index{Lipschitz|(}\index{Dyck path|(}

We briefly explain the semantics of the $\See{[d,u]}$ and $\At{[d,u]}$ modalities for $d\leq u$ in
$\RR$ using Dyck paths; see \cref{def:dyck}\index{modality!picturing with Dyck paths}. On any
length $\ell$, a proposition $P$ corresponds to an open subspace of $\upclose[0,\ell]$, and can hence be drawn as a Dyck path: $P$ is true at $[d,u]$ iff $[d,u]$ is under the Dyck path.

We first explain in a few steps how the modalities $\See{[d,u]}$ and $\At{[d,u]}$ operate. We begin
by drawing two arbitrary propositions $P,Q$ dashed, each together with the proposition
$t\apart[u,d]=(t<u)\vee (d<t)$ dotted.
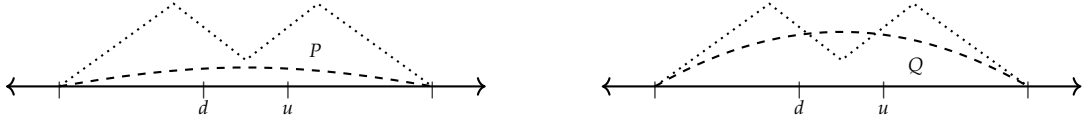
\[
\begin{tikzpicture}[font=\tiny, text height=1.5ex, text depth=.25ex]
  \tikzmath{
    real \margin, \outer, \inner, \smheight, \height, \dist;
    let \margin = .5;
    let \outer = 1.5;
    let \inner =.7;
    \height=(\inner+\outer)/2;
    \smheight=\inner/2;
    let \dist = 1;
  };
  \node (lend1) {};
  \node [right=\margin of lend1] (L1) {|};
  \node [right=\outer of L1] (d1) {|};
  \node [right=\inner of d1] (u1) {|};
  \node [right=\outer of u1] (R1) {|};
  \node [right=\margin of R1] (rend1) {};
  \node at ($(L1)!.5!(u1)+(0,\height)$) (L1u1) {};
  \node at ($(d1)!.5!(R1)+(0,\height)$) (d1R1) {};
  \node at ($(d1)!.5!(u1)+(0,\smheight)$) (little1) {};
  \draw [thick, <->] (lend1) -- (rend1);
  \draw [thick, dotted] (L1.center) -- (L1u1.center) -- (little1.center) -- (d1R1.center) -- (R1.center);
  \draw [thick, dashed] (L1.center) to[bend left=10pt] node[pos=.7, above] {$P$} (R1.center);
  \node at ($(d1)+(0,-1.5ex)$) {$d$};
  \node at ($(u1)+(0,-1.5ex)$) {$u$};
  \node [right=\dist of rend1] (lend2) {};
  \node [right=\margin of lend2] (L2) {|};
  \node [right=\outer of L2] (d2) {|};
  \node [right=\inner of d2] (u2) {|};
  \node [right=\outer of u2] (R2) {|};
  \node [right=\margin of R2] (rend2) {};
  \node at ($(L2)!.5!(u2)+(0,\height)$) (L2u2) {};
  \node at ($(d2)!.5!(R2)+(0,\height)$) (d2R2) {};
  \node at ($(d2)!.5!(u2)+(0,\smheight)$) (little2) {};
  \draw [thick, <->] (lend2) -- (rend2);
  \draw [thick, dotted] (L2.center) -- (L2u2.center) -- (little2.center) -- (d2R2.center) -- (R2.center);
  \draw [thick, dashed] (L2.center) to[bend left=30pt] node[pos=.7, below] {$Q$} (R2.center);
  \node at ($(d2)+(0,-1.5ex)$) {$d$};
  \node at ($(u2)+(0,-1.5ex)$) {$u$};
\end{tikzpicture}
\]
Next we draw the propositions $\See{[d,u]}P$ and $\See{[d,u]}Q$.
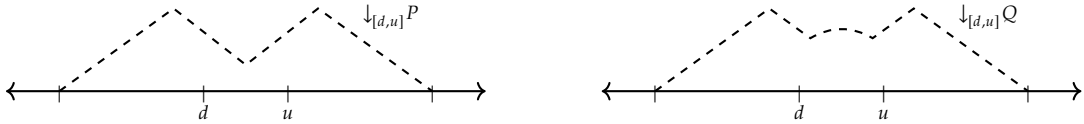
\[
\begin{tikzpicture}[font=\tiny, text height=1.5ex, text depth=.25ex]
  \tikzmath{
    real \margin, \outer, \inner, \height, \maxheight, \smheight, \dist;
    let \margin = .5;
    let \outer = 1.5;
    let \inner =.7;
    \height=(\inner+\outer)/2;
    \maxheight=(2*\outer+\inner)/2;
    \smheight=\inner/2;
    let \dist = 1;
  };
  \node (lend1) {};
  \node [right=\margin of lend1] (L1) {|};
  \node [right=\outer of L1] (d1) {|};
  \node [right=\inner of d1] (u1) {|};
  \node [right=\outer of u1] (R1) {|};
  \node [right=\margin of R1] (rend1) {};
  \node at ($(L1)!.5!(u1)+(0,\height)$) (L1u1) {};
  \node at ($(d1)!.5!(R1)+(0,\height)$) (d1R1) {};
  \node at ($(d1)!.5!(u1)+(0,\smheight)$) (little1) {};
  \draw [thick, <->] (lend1) -- (rend1);
  \draw [thick, dashed] (L1.center) -- (L1u1.center) -- (little1.center) -- (d1R1.center) -- node[pos=.3, above right] {$\See{[d,u]}P$} (R1.center);
  \node at ($(d1)+(0,-1.5ex)$) {$d$};
  \node at ($(u1)+(0,-1.5ex)$) {$u$};
  \node [right=\dist of rend1] (lend2) {};
  \node [right=\margin of lend2] (L2) {|};
  \node [right=\outer of L2] (d2) {|};
  \node [right=\inner of d2] (u2) {|};
  \node [right=\outer of u2] (R2) {|};
  \node [right=\margin of R2] (rend2) {};
  \node at ($(L2)!.5!(u2)+(0,\height)$) (L2u2) {};
  \node at ($(d2)!.5!(R2)+(0,\height)$) (d2R2) {};
  \draw [thick, <->] (lend2) -- (rend2);
  \path [name path = Lu] (L2.center) -- (L2u2.center) -- (u2.center);
  \path [name path = dR] (d2.center) -- (d2R2.center) -- (R2.center);
  \path [name path = Q] (L2.center) to[bend left=30pt] (R2.center);
  \draw [name intersections={of=Lu and Q, by={trash, x}}, name intersections={of=Q and dR, by=y}] [thick, dashed]
    (L2.center) -- (L2u2.center) -- (x) to[bend left=30pt] (y) -- (d2R2.center) -- node[pos=.3, above right] {$\See{[d,u]}Q$} (R2.center);
  \node at ($(d2)+(0,-1.5ex)$) {$d$};
  \node at ($(u2)+(0,-1.5ex)$) {$u$};
\end{tikzpicture}
\]
Next we draw the propositions $P\imp t\apart[u,d]$ and $Q\imp t\apart[u,d]$, on our way to the
$\At{}$-modality:
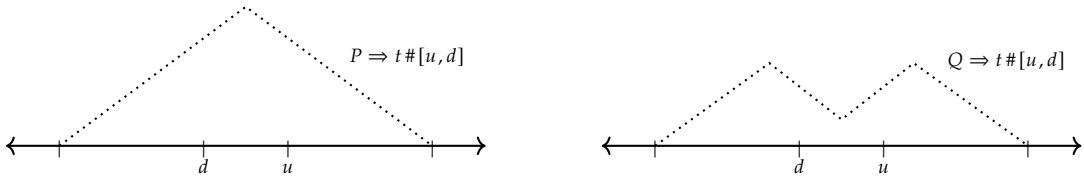
\[
\begin{tikzpicture}[font=\tiny, text height=1.5ex, text depth=.25ex]
  \tikzmath{
    real \margin, \outer, \inner, \height, \maxheight, \smheight, \dist;
    let \margin = .5;
    let \outer = 1.5;
    let \inner =.7;
    \height=(\inner+\outer)/2;
    \maxheight=(2*\outer+\inner)/2;
    \smheight=\inner/2;
    let \dist = 1;
  };
  \node (lend1) {};
  \node [right=\margin of lend1] (L1) {|};
  \node [right=\outer of L1] (d1) {|};
  \node [right=\inner of d1] (u1) {|};
  \node [right=\outer of u1] (R1) {|};
  \node [right=\margin of R1] (rend1) {};
  \node at ($(L1)!.5!(R1)+(0,\maxheight)$) (apex) {};
  \draw [thick, <->] (lend1) -- (rend1);
  \draw [thick, dotted] (L1.center) -- (apex.center) -- node[pos=.5, above right] {$P\imp t\apart[u,d]$} (R1.center);
  \node at ($(d1)+(0,-1.5ex)$) {$d$};
  \node at ($(u1)+(0,-1.5ex)$) {$u$};
  \node [right=\dist of rend1] (lend2) {};
  \node [right=\margin of lend2] (L2) {|};
  \node [right=\outer of L2] (d2) {|};
  \node [right=\inner of d2] (u2) {|};
  \node [right=\outer of u2] (R2) {|};
  \node [right=\margin of R2] (rend2) {};
  \node at ($(L2)!.5!(u2)+(0,\height)$) (L2u2) {};
  \node at ($(d2)!.5!(R2)+(0,\height)$) (d2R2) {};
  \node at ($(d2)!.5!(u2)+(0,\smheight)$) (little2) {};
  \draw [thick, <->] (lend2) -- (rend2);
  \draw [thick, dotted] (L2.center) -- (L2u2.center) -- (little2.center) -- (d2R2.center) -- node[pos=.2, above right] {$Q\imp t\apart[u,d]$}(R2.center);
  \node at ($(d2)+(0,-1.5ex)$) {$d$};
  \node at ($(u2)+(0,-1.5ex)$) {$u$};
\end{tikzpicture}
\]
We are finally ready to draw the propositions $\At{[d,u]}P$ and $\At{[d,u]}Q$:
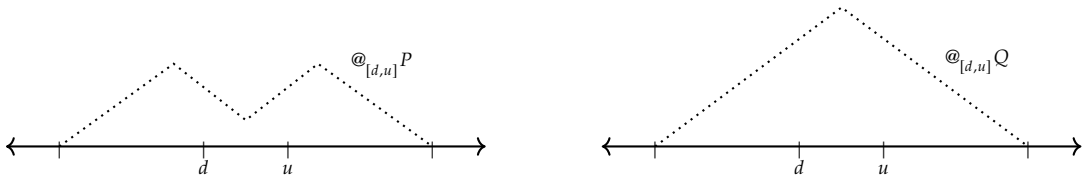
\[
\begin{tikzpicture}[font=\tiny, text height=1.5ex, text depth=.25ex]
  \tikzmath{
    real \margin, \outer, \inner, \height, \maxheight, \smheight, \dist;
    let \margin = .5;
    let \outer = 1.5;
    let \inner =.7;
    \height=(\inner+\outer)/2;
    \maxheight=(2*\outer+\inner)/2;
    \smheight=\inner/2;
    let \dist = 1;
  };
  \node (lend2) {};
  \node [right=\margin of lend2] (L2) {|};
  \node [right=\outer of L2] (d2) {|};
  \node [right=\inner of d2] (u2) {|};
  \node [right=\outer of u2] (R2) {|};
  \node [right=\margin of R2] (rend2) {};
  \node at ($(L2)!.5!(u2)+(0,\height)$) (L2u2) {};
  \node at ($(d2)!.5!(R2)+(0,\height)$) (d2R2) {};
  \node at ($(d2)!.5!(u2)+(0,\smheight)$) (little2) {};
  \draw [thick, <->] (lend2) -- (rend2);
  \draw [thick, dotted] (L2.center) -- (L2u2.center) -- (little2.center) -- (d2R2.center) -- node[pos=.2, above right] {$\At{[d,u]}P$} (R2.center);
  \node at ($(d2)+(0,-1.5ex)$) {$d$};
  \node at ($(u2)+(0,-1.5ex)$) {$u$};
  \node [right=\dist of rend2] (lend1) {};
  \node [right=\margin of lend1] (L1) {|};
  \node [right=\outer of L1] (d1) {|};
  \node [right=\inner of d1] (u1) {|};
  \node [right=\outer of u1] (R1) {|};
  \node [right=\margin of R1] (rend1) {};
  \node at ($(L1)!.5!(R1)+(0,\maxheight)$) (apex) {};
  \draw [thick, <->] (lend1) -- (rend1);
  \draw [thick, dotted] (L1.center) -- (apex.center) -- node[pos=.5, above right] {$\At{[d,u]}Q$} (R1.center);
  \node at ($(d1)+(0,-1.5ex)$) {$d$};
  \node at ($(u1)+(0,-1.5ex)$) {$u$};
\end{tikzpicture}
\]
The picture on the left is also that of $\At{[d,u]}\bot=t\apart[u,d]=\See{[d,u]}\bot$. That on the
right is $\top$. We can see that applying the $\At{[d,u]}$ to any $P:\Prop$ returns one or the other
of these, based on whether the point $[d,u]$ is strictly below the Dyck path $P$ or not.

Next, on the left we draw $d<t<u$ dotted and $P$ dashed. On the right we draw $\In{[d,u]}P$.
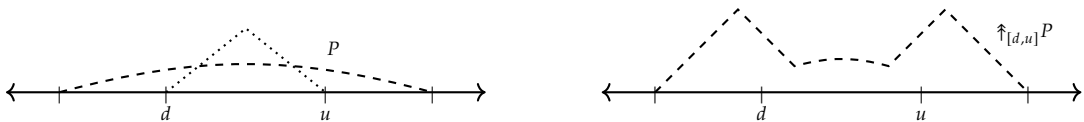
\[
\begin{tikzpicture}[font=\tiny, text height=1.5ex, text depth=.25ex]
  \tikzmath{
    real \margin, \outer, \inner, \height, \maxheight, \smheight, \dist;
    let \margin = .5;
    let \outer = 1;
    let \inner = 1.7;
    \height=(\inner+\outer)/2;
    \maxheight=(2*\outer+\inner)/2;
    \smheight=\inner/2;
    let \dist = 1;
  };
  \node (lend1) {};
  \node [right=\margin of lend1] (L1) {|};
  \node [right=\outer of L1] (d1) {|};
  \node [right=\inner of d1] (u1) {|};
  \node [right=\outer of u1] (R1) {|};
  \node [right=\margin of R1] (rend1) {};
  \node at ($(d1)!.5!(u1)+(0,\smheight)$) (middle1) {};
  \draw [thick, dotted] (d1.center) -- (middle1.center) -- (u1.center);
  \draw [thick, <->] (lend1) -- (rend1);
  \draw [thick, dashed] (L1.center) to[bend left=15pt] node [pos=.7, above right] {$P$} (R1.center);
  \node at ($(d1)+(0,-1.5ex)$) {$d$};
  \node at ($(u1)+(0,-1.5ex)$) {$u$};
  \node [right=\dist of rend1] (lend2) {};
  \node [right=\margin of lend2] (L2) {|};
  \node [right=\outer of L2] (d2) {|};
  \node [right=\inner of d2] (u2) {|};
  \node [right=\outer of u2] (R2) {|};
  \node [right=\margin of R2] (rend2) {};
  \node at ($(d2)!.5!(u2)+(0,\smheight)$) (middle2) {};
  \draw [thick, <->] (lend2) -- (rend2);
  \path [name path = In] (d2.center) -- (middle2.center) -- (u2.center);
  \path [name path = P] (L2.center) to[bend left=15pt] (R2.center);
  \path [name path = LL] (L2) -- +(\height,\height);
  \path [name path = RR] (R2) -- +(-\height,\height);
  \path [name path = Lmid, name intersections={of=In and P, by={InP1, InP2}}] (InP1) -- +(-\height,\height);
  \path [name path = Rmid, name intersections={of=In and P, by={InP1, InP2}}] (InP2) -- +(\height,\height);
  \draw [name intersections={of=LL and Lmid, by={LLmid}}, name intersections={of=RR and Rmid, by={RRmid}}, name intersections={of=In and P, by={InP1, InP2}}, thick, dashed]
  (L2.center) -- (LLmid) -- (InP1) to[bend left=15pt] (InP2) -- (RRmid) --  node [above right] {$\In{[d,u]}P$} (R2.center);
  \node at ($(d2)+(0,-1.5ex)$) {$d$};
  \node at ($(u2)+(0,-1.5ex)$) {$u$};
\end{tikzpicture}
\]
Finally, we draw a proposition $P$ on the left and $\pi P$ on the right. Thinking in terms of
curves, one can see that $\pi$ removes the subtlety of ``height'', replacing a Lipschitz curve $P$
with the maximal one that is zero at the same places as $P$:
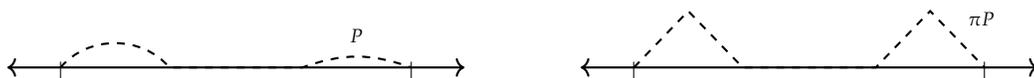
\[
\begin{tikzpicture}[font=\tiny, text height=1.5ex, text depth=.25ex]
  \tikzmath{
    real \margin, \outer, \inner, \height, \dist;
    let \margin = .5;
    let \outer = 1.1;
    let \inner = 1.5;
    \height=\inner/2;
    let \dist = 1;
  };
  \node (lend1) {};
  \node [right=\margin of lend1] (L1) {|};
  \node [right=\outer of L1] (d1) {};
  \node [right=\inner of d1] (u1) {};
  \node [right=\outer of u1] (R1) {|};
  \node [right=\margin of R1] (rend1) {};
  \draw [thick, <->] (lend1) -- (rend1);
  \draw [thick, dashed] (L1.center) to[bend left=50pt] (d1.center) -- (u1.center) to[bend left=20pt] node [above] {$P$} (R1.center);
  \node [right=\dist of rend1] (lend2) {};
  \node [right=\margin of lend2] (L2) {|};
  \node [right=\outer of L2] (d2) {};
  \node [right=\inner of d2] (u2) {};
  \node [right=\outer of u2] (R2) {|};
  \node [right=\margin of R2] (rend2) {};
  \node at ($(L2)!.5!(d2)+(0,\height)$) (midL2) {};
  \node at ($(u2)!.5!(R2)+(0,\height)$) (midR2) {};
  \draw [thick, <->] (lend2) -- (rend2);
  \draw [thick, dashed] (L2.center) -- (midL2.center) -- (d2.center) -- (u2.center) -- (midR2.center) --  node [above right] {$\pi P$} (R2.center);
\end{tikzpicture}
\]
\index{Lipschitz|)}\index{Dyck path|)}

\subsection{First properties of the modalities $\See{}$, $\At{}$, $\In{}$, and $\pi$}\label{sec:first_properties_of_our_modalities}

Before introducing any new axioms in terms of our modalities, we explore some basic consequences of their definitions. We break \cref{sec:first_properties_of_our_modalities} into subsections; the first subsection is about $\pi$. The rest are about $\See{}$, $\At{}$, and $\In{}$ in various capacities, in particular regarding:
\begin{itemize}
	\item disjunction,
	\item implication,
	\item existential quantification,
	\item a De Morgan property,
	\item decidability, and
	\item containment and disjointness.
\end{itemize}

\subsubsection{First facts about $\pi$}

\begin{proposition}\label{prop:pi_R_version}
  For any $t:\Time$ and $P:\Prop$, there is an equivalence $\pi
  P\iff\forall(r:\tRR)\ldotp\At[t]{r}P$.
\end{proposition}
\begin{proof}
  By definition $\pi P\iff\forall(t':\Time)\ldotp\At[t']{0}P$. By the torsor axiom,
  \cref{ax:torsor_diff}, for any $t':\Time$ there exists a unique $r:\tRR$ such that $t=t'+r$. Thus
  $t'\apart 0$ is equivalent to $t\apart r$ by \cref{prop:torsor_for_comparison}, and hence
  $\forall(t':\Time)\ldotp\At[t']{0}P$ is equivalent to $\forall(r:\tRR)\ldotp\At[t]{r}P$.
\end{proof}

\begin{proposition}\label{prop:bounded_pi}
  Suppose given $P:\Prop$ and $t:\Time$. Then for any $a<b:\tRRub$
  we have
  \[
    \big(\forall(r:\tRR)\ldotp a<r<b\imp \At[t]{r}P\big)\imp a<t<b\imp \pi P
  \]
\end{proposition}
\begin{proof}
  Assume the hypotheses. By definition of $a<t<b$ and by roundedness of $t$, there exists
  $a_1,a_2,b_1,b_2:\tQQ$ such that $a<a_1<a_2<t$ and $t<b_2<b_1<b$. By \cref{prop:pi_R_version}, to
  prove $\pi P$ it suffices to show $\big(\forall(r:\RR)\ldotp \At[t]{r}P\big)\imp P$.

  Choose $r:\tRR$; we want to prove $\At[t]{r}P$. Because $r$ is located by \cref{def:local_reals} (5),
  we have either $a_1<r$ or $r<a_2$ and either $b_2<r$ or $r<b_1$. Thus we have four cases, one of
  which is $a<r<b$ and the other three of which imply $t\apart r$. In all cases we obtain
  $\At[t]{r}P$ as desired.
\end{proof}

\subsubsection{Disjunction}

For any modality $j$ it is easy to check that $(jP\vee jQ)\imp j(P\vee Q)$; the converse holds for $\See{}$ and $\In{}$, and most of the time for $\At{}$.

\begin{lemma}\label{lem:At_preserves_or}
  For any $d, u:\tRRub$ the modalities $\See{[d,u]}$ and $\In{[d,u]}$ commute with disjunction:
  \begin{align*}
    \See{[d,u]}(P\vee Q)&\iff\See{[d,u]}(P)\vee\See{[d,u]}(Q)\\
    \In{[d,u]}(P\vee Q)&\iff\In{[d,u]}(P)\vee\In{[d,u]}(Q).
  \intertext{When $d\leq u$, the modality $\At{[d,u]}$ also commutes with disjunction:}
    \At{[d,u]}(P\vee Q)&\iff\At{[d,u]}(P)\vee\At{[d,u]}(Q).
  \end{align*}
\end{lemma}
\begin{proof}
  The first statement is obvious by \cref{not:our_modalities}, and the second statement is just a reformulation of \cref{lem:interval_coprime}. The third statement follows from
  \cref{ax:prime}, because $(P\vee
  Q)\imp(t\apart[u,d])\iff[(P\imp(t\apart[u,d]))\wedge(Q\imp(t\apart[u,d]))]$.
\end{proof}

One may consider $\bot$ as nullary disjunciton. We will see in \cref{cor:pi_negneg}, after adding some axioms, that $\pi\bot=\bot$. For the other modalities, we have the following lemma.

\begin{lemma}\label{lemma:See_At_In_False}
  For all $t:\Time$ and $r_1,r_2:\tRRub$ we have $\See{[r_2,r_1]}\bot\iff\At{[r_2,r_1]}\bot\iff
  t\apart [r_1,r_2]$. Assuming $r_1<r_2$, we have
  \[
    \See{[r_2,r_1]}\bot\iff\At{[r_2,r_1]}\bot\iff\In{[r_1,r_2]}\bot.
  \]
\end{lemma}
\begin{proof}
  The first statement follows directly from the definitions, \cref{not:our_modalities}. The second,
  that $t\apart[r_1,r_2]\iff\neg([r_1,r_2]\ll t)$ is \cref{prop:time_apartness_complement}.
\end{proof}

\subsubsection{Implication}

For any modality $j$ it is easy to check that $j(P\imp Q)\imp (P\imp jQ)$ holds. The reader can verify immediately from \cref{not:our_modalities} that the converse holds for $\In{[d,u]}$, for any $d,u:\tRRub$,
\begin{equation}\label{eqn:In_out_of_imp}
  \In{[d,u]}(P\imp Q)\iff (P\imp\In{[d,u]}Q).
\end{equation}

\subsubsection{Existential quantification}\label{page:existential}

For each of the modalities, $\In{}$, $\See{}$, $\At{}$, and $\pi$, we will prove some form of commutation with existential quantifiers. We will say something about $\See{}$ and $\In{}$ now, and give the statements for $\At{}$ and $\pi$ in \cref{prop:At_preserves_exists} and \cref{lem:exists_Pi_closed}.

\begin{lemma}\label{eqn:exists_see_closed}\index{existential quantifier, commuting with}
Let $X$ be an inhabited type, and suppose given $t:\Time$ and $P:X\to\Prop$.
\[
  \See[t]{[d,u]}(\exists(x:X)\ldotp P(x)) \iff \exists(x:X)\ldotp\See[t]{[d,u]}P(x)
\]
\end{lemma}
\begin{proof}
$\See[t]{[d,u]}P(x)$ is shorthand for a disjunction, $t\apart[u,d]\vee P(x)$. For any type $X$, predicate $P:X\to\Prop$, and proposition $Q:\Prop$, one shows directly that $\exists(x:X)\ldotp Q\vee P(x)$ implies $Q\vee\exists(x:X)\ldotp P(x)$. The converse is easy too but requires $X$ to be inhabited.
\end{proof}

The following is a direct consequence of \cref{prop:comm_existential_unit_surj} and \cref{eqn:exists_see_closed}.

\begin{corollary}
For any inhabited type $X$, the inclusion $\eta_{\SeeInline{[d,u]}}:X\to\asSh_{\SeeInline{[d,u]}}(X)$, of $X$ into its $\See{[d,u]}$-sheafification, is surjective.
\end{corollary}

The best we can do for $\In{}$ was already given in \cref{lem:inf_interval_coprime}. It says that for an inhabited constant type $\tConst$,
  \begin{equation}\label{eqn:In_exists_commute_ish}
    \In{[d,u]}(\exists(c:\tConst)\ldotp P(c))\iff[\forall(d',u':\tRR)\ldotp (d<d'\wedge
    u'<u)\imp\exists(c:\tConst)\ldotp\In{[d',u']}P(c)].
  \end{equation}

\begin{remark}
  \Cref{eqn:In_exists_commute_ish} cannot be strengthened to avoid the quantification over $d',u'$. 
  To see this, take as a context $t:\Time$ and recall that $t$ is rounded. Then the basic idea is that $t<0\imp\exists(q:\tQQ)\ldotp t<q<0$ holds
  but $\exists(q:\tQQ)\ldotp t<0\imp t<q<0$ does not.
  
  More precisely consider the predicate $P:\tQQ\to\Prop$ given by
  \[
    P(q)\coloneqq (t<q)\wedge(q<0).
  \]
  It is easy to prove that $\In[t]{[-\infty,0]}\exists(q:\tQQ)\ldotp P(q)$, whereas one can show that
  $\exists(q:\tQQ)\ldotp\In[t]{[-\infty,0]} P(q)$ is not sound in our semantics.
\end{remark}

\subsubsection{De Morgan}
Semantically speaking, \cref{lem:DeMorgan} says the De Morgan laws hold in the closed subtopos $\down[d,u]$; see
\cite[Section D4.6]{Johnstone:2002a}.\index{De Morgan}

\begin{lemma}\label{lem:DeMorgan}
  Choose $d,u:\tRRub$, $t:\Time$, and $P:\Prop$. If $d\leq u$ then
  \[
    \At{[d,u]}P\vee(P\imp\At{[d,u]}\bot).
  \]
\end{lemma}
\begin{proof}
  Apply \cref{prop:prime_reals} to the implication $(P\wedge(P\imp t\apart[u,d]))\imp t\apart[u,d]$.
\end{proof}

Perhaps more memorably, the proposition $\At{}P$ is either equal to $\At{}\top$ or to $\At{}\bot$:
\[(\At{[d,u]}P=\At{[d,u]}\top)\vee(\At{[d,u]}P=\At{[d,u]}\bot).\]

\subsubsection{Decidability}

Recall the definition (\ref{def:j_closed}) of a proposition $P$ being $j$-closed for a modality
$j$.\index{proposition!$j$-closed} For any modality $j$, say that a $j$-closed proposition $P$ is \emph{$j$-decidable} if $P\vee(P\imp
j\bot)$ holds.\index{decidable!$j$-}

\begin{corollary}\label{cor:AtClosed_SeeComplemented}
  Choose $d,u:\tRRub$, $t:\Time$, and $P:\Prop$. Then
  \begin{enumerate}
    \item $P$ is $\At{[d,u]}$-closed iff it is $\See{[d,u]}$-closed and $\See{[d,u]}$-decidable.
    \item If $P$ is $\At{[d,u]}$-closed then it is $\At{[d,u]}$-decidable.
  \end{enumerate}
\end{corollary}
\begin{proof}
  The second statement follows from the first because $\See{[d,u]}\bot\iff\At{[d,u]}\bot$
  (\cref{lemma:See_At_In_False}).

  Clearly, if $\At{[d,u]}P\imp P$ then $P$ is $\See{[d,u]}$-closed, and it is
  $\See{[d,u]}$-decidable by \cref{lem:DeMorgan}. Conversely, suppose $\See{[d,u]}P\imp P$ and
  $P\vee(P\imp\See{[d,u]}\bot)$ and $\At{[d,u]}P$; we want to show $P$. We may assume
  $P\imp\See{[d,u]}\bot$, so applying $\At{[d,u]}$ we get $\At{[d,u]}\See{[d,u]}\bot$, which implies
  $\See{[d,u]}\bot$ by \cref{lemma:See_At_In_False}, and $\See{[d,u]}\bot$ implies $\See{[d,u]}P$
  and hence $P$.
\end{proof}

\subsubsection{Containment and disjointness}

There are a number of containment and disjointness relations among our modalities. We begin with
\cref{prop:simple_modalities}, which is just a matter of unwinding notation
(\ref{not:our_modalities}).\index{modality!disjoint}

\begin{proposition}\label{prop:simple_modalities}
  The following hold for any $P:\Prop$, $t:\Time$, and $d_1,d_2,u_1,u_2:\tRRub$,
  \begin{enumerate}
    \item If $d_1\leq d_2$ and $u_2\leq u_1$, then $\In{[d_1,u_1]}P\imp\In{[d_2,u_2]}P$.
    \item If $d_1\leq d_2$ and $u_2\leq u_1$, then $\See{[d_2,u_2]}P\imp\See{[d_1,u_1]}P$.
    \item $\See{[d_1,u_1]}\See{[d_2,u_2]}P\iff\See{[\min(d_1,d_2),\max(u_1,u_2)]}P$.
    \item $\In{[d_1,u_1]}\In{[d_2,u_2]}P\iff\In{[\max(d_1,d_2),\min(u_1,u_2)]}P$.
    \item $\See{[d_1,u_1]}P\imp\At{[d_1,u_1]}P$.
  \end{enumerate}
\end{proposition}

\begin{proposition}\label{prop:Atud}
  Suppose $d<u:\tRRub$. Then $\At{[u,d]}P\iff\In{[d,u]}\neg\neg P$.
\end{proposition}
\begin{proof}
  Suppose $\At{[u,d]}P$, or equivalently $(P\imp\At{[u,d]}\bot)\imp\At{[u,d]}\bot$. Then $\neg
  P\imp\In{[d,u]}\bot$ by \cref{lemma:See_At_In_False}, giving $\In{[d,u]}\neg\neg P$ by
  \cref{eqn:In_out_of_imp}. Conversely, suppose $\In{[d,u]}\neg\neg P$ and $(P\imp\At{[u,d]}\bot)$.
  Again by \cref{lemma:See_At_In_False,eqn:In_out_of_imp}, $P\imp\In{[d,u]}\bot$, so $\In{[d,u]}\neg
  P$, giving $\In{[d,u]}\bot$ and hence $\At{[u,d]}\bot$ as desired.
\end{proof}

\begin{lemma}\label{lem:nested_points}
  $\At{[d_1,u_1]}P \imp \See{[d_1,u_1]}\At{[d_2,u_2]}P$ holds for any $d_1\leq d_2\leq u_2\leq u_1$
  and $P:\Prop$.
\end{lemma}
\begin{proof}
  \Cref{lem:DeMorgan} gives $(P\imp t\apart{[u_2,d_2]})\vee\At{[d_2,u_2]}P$. The result follows
  since $P\imp t\apart{[u_2,d_2]}$ implies $P\imp t\apart{[u_1,d_1]}$, which, assuming
  $\At{[d_1,u_1]}P$, implies $t\apart{[u_1,d_1]}$.
\end{proof}

\Cref{lemma:points_in_subtoposes,lemma:points_not_in_subtoposes} semantically tell us when a point
is in, or is disjoint from, various open and closed subtoposes; see
\cref{sec:modalitites_subtoposes}.\index{modality!open}\index{modality!closed}\index{modality!quasi-closed}

\begin{lemma}\label{lemma:points_in_subtoposes}
  For any $P:\Prop$ and $d_1,d_2,u_1,u_2:\tRRub$,
  \begin{enumerate}
    \item if $d_1\leq d_2\wedge u_2\leq u_1$ then $\See{[d_2,u_2]}P\imp\At{[d_1,u_1]}P$
    \item if $d_1<d_2\leq u_2<u_1$ then $\In{[d_1,u_1]}P\imp\At{[d_2,u_2]}P$.
  \end{enumerate}
\end{lemma}
\begin{proof}
  The first is two applications of \cref{prop:simple_modalities}, and the second follows easily from
  \cref{lem:time_lems}(g).
\end{proof}

\begin{lemma}\label{lemma:points_not_in_subtoposes}
  For any $d_1,d_2,u_1,u_2:\tRRub$ with $d_2\leq u_2$,
  \begin{enumerate}
    \item if $d_1<d_2$ or $u_2<u_1$ then $\At{[d_2,u_2]}\See{[d_1,u_1]}\bot$
    \item if $d_1\leq d_2$ or $u_2\leq u_1$ then $\In{[d_2,u_2]}\At{[d_1,u_1]}\bot$.
  \end{enumerate}
\end{lemma}
\begin{proof}
  For the first, we assume $d_1<d_2$ or $u_2<u_1$ and need to show $(t\apart[u_1,d_1]\imp
  t\apart[u_2,d_2])\imp t\apart[u_2,d_2]$, but this is just \cref{lem:time_lems}(l). For the second,
  assume $(d_1\leq d_2)\vee(u_1\leq u_2)$. We need to show $d_2<t<u_2\imp t\apart[u_1,d_1]$, but
  this is direct.
\end{proof}

\begin{proposition}\label{prop:points_in_subtoposes}
  Let $d_1,u_1,d_2,u_2:\tRRub$. Then for any $P:\Prop$,
  \begin{align*}
    \At{[d_1,u_1]}\See{[d_2,u_2]}P
    &\iff [((d_1\leq d_2) \wedge (u_2\leq u_1)) \imp \At{[d_1,u_1]}P]. \\
    &\iff [(d_2<d_1) \vee (u_1<u_2) \vee \At{[d_1,u_1]}P] \\                                  
  \intertext{If $d_2\leq u_2$, then}
    \In{[d_1,u_1]}\At{[d_2,u_2]}P 
    &\iff [((d_1<d_2) \wedge (u_2<u_1)) \imp \At{[d_2,u_2]}P]\\
    &\iff [(d_2\leq d_1) \vee (u_1\leq u_2) \vee \At{[d_2,u_2]}P].
  \end{align*}
\end{proposition}
\begin{proof}
  $\At{[d_1,u_1]}\See{[d_2,u_2]}P$ implies $((d_1\leq d_2) \wedge (u_2\leq u_1)) \imp
  \At{[d_1,u_1]}P$ by \cref{lemma:points_in_subtoposes}. This implies $(d_2<d_1) \vee (u_1<u_2) \vee
  \At{[d_1,u_1]}P$ by trichotomy (\ref{prop:trichotomy}). This implies
  $\At{[d_1,u_1]}\See{[d_2,u_2]}P$ by \cref{lemma:points_not_in_subtoposes}.

  For the second statement, suppose $d_2\leq u_2$. Then $\In{[d_1,u_1]}\At{[d_2,u_2]}P$ implies
  $((d_1<d_2) \wedge (u_2<u_1)) \imp \At{[d_2,u_2]}P$ by \cref{lemma:points_in_subtoposes}. This
  implies $(d_2\leq d_1) \vee (u_1\leq u_2) \vee \At{[d_2,u_2]}P$ by trichotomy. This implies
  $\In{[d_1,u_1]}\At{[d_2,u_2]}P$ by \cref{lemma:points_not_in_subtoposes}.
\end{proof}

\begin{proposition}\label{prop:At_At}
  If $d_1\leq d_2\leq u_2\leq u_1$, then for any $P:\Prop$
  \[
    \At{[d_1,u_1]}\At{[d_2,u_2]}P \iff \See{[d_1,u_1]}\At{[d_2,u_2]}P.
  \]
  If $d_1<d_2\leq u_2$ or $d_2\leq u_2<u_1$, then
  \[
    \At{[d_2,u_2]}\At{[d_1,u_1]}P \iff \top.
  \]
\end{proposition}
\begin{proof}
  The first statement follows directly from \cref{lem:nested_points}. For the second it suffices to
  show $\At{[d_2,u_2]}\At{[d_1,u_1]}\bot$, which reduces to $\At{[d_2,u_2]}t\apart[u_1,d_1]$. This
  follows directly from \cref{lem:time_lems}(l).
\end{proof}

The next proposition looks at the intersection of $\upclose[d,u]$ and $\down[r_1,r_2]$.

\begin{proposition}\label{prop:open_closed_intersection}
  Choose $d,u,r,r':\tRRub$. If $r\leq d$ or $u\leq r'$, then
  \[
    \In{[d,u]}\See{[r,r']}P \iff \top.
  \]
  If $d<r$ and $r'<u$, then
  \begin{align*}
    \In{[d,u]}\See{[r,r']}\bot
    &\iff(t<d) \vee (t<r') \vee (u<t)\vee (r<t)
    \\
    &\iff t\apart[\mathrm{max}(d,r'),\mathrm{min}(u,r)]
  \end{align*}
  In particular, if $d<r<u$ and $d<r'<u$, then
  \[
    \In{[d,u]}\See{[r,r']} P \iff \See{[r,r']}\In{[d,u]}P.
  \]
\end{proposition}
\begin{proof}
  The first claim is easy: $\In{[d,u]}\See{[r,r']}\bot$ is $d<t<u\imp (t<r'\vee r<t)$, which is
  obvious. For the second claim, \cref{lem:interval_coprime} says that $\In{[d,u]}\See{[r,r']}\bot$
  is equivalent to $(d<t<u\imp t<r')\vee(d<t<u\imp r<t)$. Assume the first case; since either $d<r'$
  or $r'\leq d$, we may assume $d<r'$ in which case we apply \cref{lem:time_lems}(e). The second
  case is similar.

  For the third claim, the left-hand side is equivalent to $\In{[d,u]}((\See{[r,r']}\bot)\vee P)$
  and the right-hand side is equivalent to $(\See{[r,r']}\bot)\vee(\In{[d,u]}P)$. The result follows
  from the second claim and \cref{lem:interval_coprime}.
\end{proof}

\Cref{prop:In_union} says that the open set $\upclose[d,u]$ is covered by those way above it.
\begin{proposition}\label{prop:In_union}
  Choose $t:\Time$, $P:\Prop$, and $d,u:\tRRub$. If $u\leq d$ then $\In[t]{[d,u]}P=\top$. If $d<u$ then
  \[
    \In[t]{[d,u]}P\iff[\forall(d',u':\tQQ)\ldotp d<d'<u'<u\imp \In[t]{[d',u']}P].
  \]
\end{proposition}
\begin{proof}
  Since $\Time\ss\tII$ by \cref{eqn:define_time}, any $t:\Time$ is a pair of disjoint, rounded cuts.
  Disjointness gives the first statement and roundedness gives the second. See also
  \cref{eqn:comparing_time}.
\end{proof}

\section{Remaining axiomatics}\label{sec:remaining_axioms}

We have now laid out some of the basic consequences that follow from the definition of our modalities $\See{}$, $\At{}$, $\In{}$, and $\pi$. Because these correspond to semantically interesting subtoposes (see \cref{rem:modalities_geom_interesting}, we can begin to record facts about these subtoposes internally to the type theory.

For example, \cref{prop:points_in_subtoposes} internally characterizes exactly which points are in the subtoposes of the
form $\down[d,u]$ and $\upclose[d,u]$. We want to prove internally that those subtoposes
are the union of those points they contain. This will be done in \cref{prop:open_subtopos_union,prop:closed_subtopos_union}, which follows from our next axiom. \Cref{ax:enough_points}
can be read as saying that the entire topos is covered by the single point subtoposes $\{\,[d,u]\,\}$.

The soundness of \cref{ax:enough_points} is proven in \cref{prop:axiom8}.

\begin{axiom}\label{ax:enough_points}\index{axiom!of $\BaseTopos$!enough points}
  For all $P\colon\Prop$ and $t:\Time$,
  \[
    (\forall(d,u:\tQQ)\ldotp d<u \imp \At{[d,u]}P) \imp P.
  \]
\end{axiom}

Of course it follows directly that $(\forall(d,u:\tQQ)\ldotp d\leq u \imp \At{[d,u]}P) \imp P.$

\begin{corollary}\label{cor:times_not_located}
  For all $t:\Time$, we have
  \[
    \neg(\forall(d,u:\tQQ)\ldotp d<u \imp t\apart[u,d])
    \qquad\text{and}\qquad
    \neg(\forall(q:\tQQ)\ldotp t\apart q)
  \]
\end{corollary}
\begin{proof}
  For the first statement, simply take $P=\bot$ in \cref{ax:enough_points}. The second statement
  follows from the first, since $d<q<u$ and $t\apart q$ implies $t\apart[u,d]$.
\end{proof}

Recall from \cref{eqn:define_time} that $\Time$ is a subtype of the type of improper intervals $\tII$. Recall from \cref{def:local_reals} that an improper interval
$t=(\delta,\upsilon)$ is said to be located if $d<u$ implies $\delta d\vee\upsilon u$ for any
$d,u:\tQQ$. Since $t\apart [u,d]$ is defined as $(d<t)\vee (t<u)=\delta t\vee \upsilon u$, \cref{cor:times_not_located} reads directly as the statement ``a time $t$ (as a
pair of cuts) is not located.'' If times were located, they would be real numbers, which we've seen
are constant in the semantics. Hence the corollary can be thought of as saying that time is not
constant.

We have shown that $t:\Time$ is decidedly \emph{not} located; however, the next proposition shows that, relative to the
$\Pointwise$ modality, $t$ \emph{is} located, i.e.\ time is ``pointwise located.''\index{Time@$\Time$!as $\pi$-located}

\begin{proposition}\label{prop:Time_pi_located}
  \[
    \forall(t:\Time)(d,u:\tRRub)\ldotp d<u \imp \Pointwise(t\apart[u,d]).
  \]
\end{proposition}
\begin{proof}
  By \cref{prop:pi_R_version}, this is equivalent to
  \[
    \forall(t:\Time)(d,u:\tRRub)\ldotp d<u \imp \forall(r:\tRR)\ldotp \At{r}(t\apart[u,d]).
  \]
  Taking arbitrary $t$, $d$, $u$, and $r$, with $d<u$, we have $(r<u)\vee(d<r)$ by
  \cref{prop:stronger_located}. The result then follows from \cref{lem:time_lems}(l).
\end{proof}

We internally express the semantic fact that the topos $\Shv{\IR}$ is the union of the closed subtoposes
corresponding to $\down[q,q]$ for $q:\tQQ$ in \cref{prop:union_closed}. Note that this proposition
is a statement about subtoposes, or equivalently, about sub-locales. Indeed it \emph{does not} imply
the (decidedly false) external statement that the topological space $\IR$ is the union of its closed
subspaces of the form $\down[q,q]$ for $q:\tQQ$.

\begin{proposition}\label{prop:union_closed}
  The following holds for any proposition $P:\Prop$ and time $t:\Time$,
  \[
    [\forall(q:\tQQ)\ldotp\See{q}P]\imp P.
  \]
\end{proposition}
\begin{proof}
  Choose $t:\Time$ and assume $\forall(q:\tQQ)\ldotp\See{q}P$. This implies
  $P\vee\forall(q:\tQQ)\ldotp t\apart q$ by \cref{ax:distributivity}, and the conclusion follows
  from \cref{cor:times_not_located}.
\end{proof}

\begin{corollary}\label{cor:union_closed}
  The following holds for any proposition $P:\Prop$,
  \[
    [\forall(t:\Time)\ldotp\See[t]{0}P]\imp P.
  \]
\end{corollary}
\begin{proof}
  This follows from \cref{prop:union_closed,ax:torsor}, because given $t:\Time$ and
  $\forall(t':\Time)\ldotp\See[t']{0}P$, we obtain $\forall(r:\tRR)\ldotp\See[t]{r}P$, which implies
  $P$.
\end{proof}

\begin{lemma}\label{lem:At_decidable_prop}\index{decidable}
  Let $P$ be a decidable proposition, i.e.\ where $P\vee\neg P$ holds. Then
  \[
    \At{[d,u]}P \iff\See{[d,u]}P\qquad\text{and}\qquad\pi P\iff P.
  \]
\end{lemma}
\begin{proof}
  The converses to both statements are obvious (see \cref{prop:simple_modalities}), so we proceed to
  show $\At{[d,u]}P \imp\See{[d,u]}P$. Assume $P\vee\neg P$. Trivially, $P\imp\See{[d,u]}P$, so
  suppose $\neg P$. Then $\At{[d,u]}P$ implies $\At{[d,u]}\bot$, which is equivalent to
  $\See{[d,u]}\bot$, which implies $\See{[d,u]}P$.

  Since now $\At{0}P\imp\See{0}P$, we have $\pi P\imp P$ by \cref{prop:union_closed}.
\end{proof}

\begin{corollary}\label{cor:pi_negneg}
  $\bot$ is decidable, so $\pi\bot=\bot$ and we have $\pi P\imp \neg\neg P$ for all $P:\Prop$.
\end{corollary}

\begin{corollary}\label{cor:dec_eq_pi_sep}
If $X$ is a type with decidable equality, then it is $\pi$-separated.
\end{corollary}

\begin{proposition}\label{prop:Time_is_pi_closed}\index{Time@$\Time$!as $\pi$-closed}
  The cuts defining any $t:\Time$ are $\pi$-closed. That is, for all $q:\tQQ$,  we have
  $\pi(t<q)\imp t<q$ and $\pi(q<t)\imp q<t$.
\end{proposition}
\begin{proof}
  Choose $t:\Time$ and $q:\tQQ$, and suppose $\pi(t<q)$. By \cref{ax:time_complementary_cuts} it
  suffices to show $\neg(q<t)$, so assume $q<t$. Then we have $\pi(\bot)$, which implies $\bot$ by
  \cref{lem:At_decidable_prop}.
\end{proof}

\begin{remark}\label{ex:Time_subtype_rvar}\index{Time@$\Time$!as variable real}
\index{Time@$\Time$!semantics of}
  In \cref{eqn:define_time}, $\Time$ was defined as a subtype of the extended interval type $\tII$,
  but in fact it can also be identified with a subtype of the variable reals $\Time\ss\tRR_\pi$.
  Indeed consider the cuts $\delta,\upsilon:\tQQ\to\Prop$ defining any $t=(\delta,\upsilon):\Time$.
  By definition, $\delta$ is down-closed and $\upsilon$ is up-closed, and both are rounded and
  bounded. It follows they are down/up-closed (1), $j$-rounded (2), and $j$-bounded (3) for any $j$.
  The fact that they are $\pi$-closed (0) was shown in \cref{prop:Time_is_pi_closed}. The fact that
  they are $\pi$-disjoint (4) was shown in \cref{ax:time_complementary_cuts}. Finally, the fact that
  they are $\pi$-located (5) was shown in \cref{prop:Time_pi_located}. See also
  \cref{rem:time_semantics}
\end{remark}

In \cref{prop:closed_subtopos_union,prop:open_subtopos_union} we prove statements that reflect facts
about the subspaces $\down[d,u]$ and $\upclose[d,u]$ as collections of points in $\IR$.

\begin{proposition}\label{prop:closed_subtopos_union}
  For any $d,u:\tRRub$ and $P:\Prop$ we have:
  \[
    [\forall(d',u':\tQQ)\ldotp (d'\leq d \wedge u\leq u') \imp \At{[d',u']}P] \iff
    \See{[d,u]}P.
  \]
\end{proposition}
\begin{proof}
  The backwards implication follows from \cref{lemma:points_in_subtoposes}. By
  \cref{prop:points_in_subtoposes}, the left-hand side is equivalent to $\forall(d',u':\tQQ)\ldotp
  \At{[d',u']}\See{[d,u]}P$, which by \cref{ax:enough_points} implies $\See{[d,u]}P$.
\end{proof}

For any reals $d<u$, we have $\upclose[d,u]=\bigcup_{d<d'\leq u'<u}[d',u']$ by
\cref{prop:way_below_IR}. This is expressed internally by \cref{prop:open_subtopos_union}.

\begin{proposition}\label{prop:open_subtopos_union}
  For any $d,u:\tRRub$ and $P:\Time\to\Prop$, we have:
  \[
    [\forall(d',u':\tQQ)\ldotp d<d'\leq u'<u \imp \At{[d',u']}P] \iff \In{[d,u]} P.
  \]
\end{proposition}
\begin{proof}
  By \cref{prop:points_in_subtoposes}, the left-hand side is equivalent to
  \[
    \forall(d',u':\tQQ)\ldotp d'\leq u' \imp \In{[d,u]}\At{[d',u']}P,
  \]
  which is equivalent to
  \[
    d<t<u \imp \forall(d',u':\tQQ)\ldotp d'\leq u' \imp \At{[d',u']}P,
  \]
  and finally, applying \cref{ax:enough_points} this is equivlant to $d<t<u\imp P$.
\end{proof}

\begin{corollary}\label{cor:expanding_ins}
  For all $t:\Time$ and $P:\Prop$ we have
  \[
    [\forall(q:\tQQ)\ldotp\In{[-q,q]} P]\imp P.
  \]
\end{corollary}
\begin{proof}
  Choose $t$ and $P$, and suppose $\forall(q:\tQQ)\ldotp\In{[-q,q]} P$. By \cref{prop:union_closed}
  it suffices to show $\forall(q':\tQQ)\ldotp\See{[q',q']}P$, so choose $q'$. By
  \cref{prop:closed_subtopos_union}, it is enough to show that for any rationals $d'\leq q'\leq u'$
  we have $\At{[d',u']}P$. Let $q=\max(-d'-1,u'+1)$, so $-q<d'$ and $u<q$. Then $\In{[-q,q]}P$
  implies $\At{[d',u']}P$ by \cref{prop:open_subtopos_union}, and we are done.
\end{proof}

\begin{corollary}\label{cor:bounded_expanding_ins}
  For all $t:\Time$, $P:\Prop$, and $r_1,r_2:\tRR$ we have
  \[
    [\forall(q_1,q_2:\tQQ)\ldotp (q_1<r_1)\imp(r_2<q_2)\imp\In{[q_1,q_2]} P]\imp P.
  \]
\end{corollary}
\begin{proof}
  Assume the hypothesis. If $q:\tQQ$ is a rational such that $\max(-r_1,r_2)<q$, then
  $\In{[-q,q]}P$. Let $q_0$ be one such rational. Then for any $q:\tQQ$ either $q\leq q_0$ or
  $q_0<q$. In either case we have $\In{[-q,q]}P$ by \cref{prop:simple_modalities}, so we are done by
  \cref{cor:expanding_ins}.
\end{proof}

\begin{proposition}\label{prop:exists_all_ins}
  Let $\tConst$ be an inhabited constant type. Then for any $P:\Prop$ and $t:\Time$,
  \[
    [\forall(q:\tQQ)\ldotp\exists(c:\tConst)\ldotp\In{[-q,q]}P]\imp\exists(c:\tConst)\ldotp Pc.
  \]
\end{proposition}
\begin{proof}
  Assume the hypothesis. By \cref{cor:expanding_ins}, it suffices to show that
  $\In{[-q,q]}\exists(c:\tConst)\ldotp P(c)$ holds for any $q:\tQQ$. By
  \cref{lem:inf_interval_coprime}, this is equivalent to $\forall(d,u:\tQQ)\ldotp(-q<d\wedge
  u<q)\imp\exists(c:\tConst)\ldotp\In{[d,u]}P(c)$. But by hypothesis we have some $c:\tConst$ with
  $\In{[-q,q]}P(c)$, and this implies $\In{[d,u]}P(c)$ by \cref{prop:simple_modalities}.
\end{proof}

\Cref{ax:point_to_open} says that if a proposition $P$ holds at a point, then it holds in some open
neighborhood of the point. The converse of \cref{ax:point_to_open} holds, as we will see in \cref{prop:point_to_open}. The soundness of \cref{ax:point_to_open} is proven in \cref{prop:axiom9}.

\begin{axiom}\label{ax:point_to_open}\index{axiom!of $\BaseTopos$!point to neighborhood}
  For all times $t:\Time$, propositions $P:\Prop$, and $d,u:\tRR$ with $d\leq u$,
  \[
    \At{[d,u]}P \imp \See{[d,u]}\exists(d',u':\tQQ)\ldotp d'<d\leq u<u' \wedge \In{[d',u']}P.
  \]
\end{axiom}

In \cref{sec:first_properties_of_our_modalities}, page \pageref{page:existential}, we said that all of our modalities commute with existential quantification in one way or another. In \cref{prop:exists_all_ins} we give the limited sense in which the $\At{}$-modality commutes with existential quantification.\index{existential quantifier, commuting with}

\begin{proposition}\label{prop:At_preserves_exists}
  For any constant type $\tConst$, any $t:\Time$ and $P:\tConst\to\Prop$, and any $d,u:\tRR$ with
  $d\leq u$, if $\exists(c:\tConst)\ldotp \top$ then
  \[
    \At[t]{[d,u]}(\exists(c:\tConst)\ldotp P(c)) \iff \exists(c:\tConst)\ldotp\At[t]{[d,u]}P(c).
  \]
\end{proposition}
\begin{proof}
  The converse is obvious (see \cref{eqn:j_logic}), so we prove the forward direction. By
  \cref{ax:point_to_open}, $\At{[d,u]}(\exists(c:\tConst)\ldotp P(c))$ is equivalent to
  \[
    \See{[d,u]}\exists(d',u':\tQQ)\ldotp (d'<d)\wedge(u<u')\wedge
    \In{[d',u']}\exists(c:\tConst)\ldotp P(c)
  \]
  which is equivalent to
  \[
    \See{[d,u]}\exists(d',u':\tQQ)\ldotp (d'<d)\wedge(u<u')\wedge
    \forall(d'',u'':\tQQ)\ldotp d'<d''<u''<u' \imp \exists(c:\tConst)\ldotp \In{[d'',u'']}P(c)
  \]
  by \cref{lem:inf_interval_coprime}. Then by choosing $d'<d''<d$ and $u<u''<u'$, and using the fact
  (see e.g.\ \cref{prop:open_subtopos_union}) that $\In{[d'',u'']}P(c)$ implies $\At{[d,u]}P(c)$,
  this implies $\See{[d,u]}\exists(c:\tConst)\ldotp\At{[d,u]}P(c)$. It is easy to see that this is
  equivalent to $\exists(c:\tConst)\ldotp\At{[d,u]}P(c)$ as desired.
\end{proof}

The following is a direct consequence of \cref{prop:comm_existential_unit_surj} and \cref{prop:At_preserves_exists}.

\begin{corollary}
For any inhabited type $X$, the inclusion $\eta_{\AtInline{[d,u]}}:X\to\asSh_{\AtInline{[d,u]}}(X)$, of $X$ into its $\At{[d,u]}$-sheafification, is surjective.
\end{corollary}

\begin{lemma}\label{lem:see_exists_open}
  For any reals $d\leq r_1\leq u$ and $d\leq r_2\leq u$, the following are equivalent:
  \begin{enumerate}
    \item $\See{[r_1,r_2]} \exists(d',u':\tQQ)\ldotp(d'<d\leq u<u')\wedge \In{[d',u']}P$
    \item $\exists(d',u':\tQQ)\ldotp(d'<d\leq u<u')\wedge \In{[d',u']}\See{[r_1,r_2]}P$
  \end{enumerate}
\end{lemma}
\begin{proof}
  It is easy to prove $\exists(d',u':\tQQ)\ldotp d'<d \leq u<u'$. Hence 1.\ is equivalent to
  \[
    \exists(d',u':\tQQ)\ldotp(d'<d \leq u<u')\wedge \See{[r_1,r_2]}\In{[d',u']}P.
  \]
  By \cref{prop:open_closed_intersection} this is equivalent to 2.
\end{proof}

\begin{proposition}\label{prop:point_to_open}
  The following are equivalent, for any $t:\Time$, $P:\Prop$, and $d,u:\tRR$ with $d\leq u$:
  \begin{enumerate}
    \item $\At{[d,u]} P$
    \item $\See{[d,u]} \exists(d',u':\tQQ)\ldotp(d'<d \leq u<u')\wedge \In{[d',u']}P$
    \item $\exists(d',u':\tQQ)\ldotp (d'<d\leq u<u') \wedge \forall(d'',u'':\tQQ)\ldotp
      ((d'<d''\leq d) \wedge (u\leq u''<u')) \imp \At{[d'',u'']}P$
  \end{enumerate}
\end{proposition}
\begin{proof}
  $1\imp 2$ is \cref{ax:point_to_open}, and the converse $2\imp 1$ follows from
  \cref{lemma:points_in_subtoposes}.

  We show $2\iff 3$ by the chain of equivalences
  \begin{align*}
    &\See{[d,u]} \exists(d',u':\tQQ)\ldotp(d'<d\leq u<u')\wedge \In{[d',u']}P \\
    &\iff\quad \exists(d',u':\tQQ)\ldotp(d'<d\leq u<u') \wedge \In{[d',u']}\See{[d,u]}P \\
    &\iff\quad \exists(d',u':\tQQ)\ldotp(d'<d\leq u<u') \\
    &\qquad\qquad \wedge \forall(d'',u'':\tQQ)\ldotp
      d'<d''\leq u''<u'\imp \At{[d'',u'']}\See{[d,u]}P \\
    &\iff\quad \exists(d',u':\tQQ)\ldotp(d'<d\leq u<u') \\
    &\qquad\qquad \wedge \forall(d'',u'':\tQQ)\ldotp d'<d''\leq u''<u' \imp d''\leq d
      \imp u\leq u'' \imp \At{[d'',u'']}P \\
    &\iff\quad \exists(d',u':\tQQ)\ldotp(d'<d\leq u<u') \\
    &\qquad\qquad \wedge \forall(d'',u'':\tQQ)\ldotp
      d'<d''\leq d\leq u\leq u''<u' \imp \At{[d'',u'']}P
  \end{align*}
  where the first three equivalences are by \cref{lem:see_exists_open},
  \cref{prop:open_subtopos_union}, and \cref{prop:points_in_subtoposes} respectively.
\end{proof}

\begin{proposition}\label{prop:clean_pi}
  For any $P:\Prop$, we have
  \[
    \pi P\iff\forall(t:\Time)\ldotp\exists(q:\tQQ)\ldotp(0<q)\wedge (-q<t<q\imp P).
  \]
\end{proposition}
\begin{proof}
  By definition, $\pi P\iff\forall(t:\Time)\ldotp\At{0}P$, which by
  \cref{ax:point_to_open,cor:union_closed} is equivalent to
  \[
    \forall(t:\Time)\ldotp\exists(d,u:\tQQ)\ldotp (d<0<u)\wedge(d<t<u\imp P)
  \]
  and the result follows by taking $q=\min(-d,u)$.
\end{proof}

\begin{proposition}\label{prop:converse_covering_pi}
  Suppose $P:\Prop$ is $\pi$-closed. Then we have
  \[
    \forall(t:\Time)(r_1,r_2:\tRR)\ldotp r_1<r_2 \imp (t\apart[r_2,r_1]\imp P)\imp P.
  \]
\end{proposition}
\begin{proof}
  Suppose $\pi P\imp P$, and choose $t$ and $r_1<r_2$ with $t\apart[r_2,r_1]\imp P$; then
  $(t<r_2)\imp P$ and $(r_1<t)\imp P$. It suffices by \cref{prop:bounded_pi} to show
  $\forall(r:\tRR)\ldotp \At[t]{r}P$. Choose $r$ and suppose $P\imp t\apart r$; we want to show
  $t\apart r$. Since $r$ is located, either $r_1<r$ or $r<r_2$, so we may apply
  \cref{lem:time_lems}(l).
\end{proof}

\begin{corollary}\label{cor:converse_covering_pi}
  Suppose $P:\Prop$ is $\pi$-closed. Then we have
  \[
    \forall(t:\Time)(r_1,r_2,r_3,r_4:\tRR)\ldotp r_1<r_2<r_3<r_4 \imp
    \In{[r_1,r_3]}P \imp \In{[r_2,r_4]}P \imp\In{[r_1,r_4]}P.
  \]
\end{corollary}
\begin{proof}
  Given $t$, reals $r_1<r_2<r_3<r_4$, and the assumptions $\In{[r_1,r_3]}P \imp \In{[r_2,r_4]}P$ and
  $r_1<t<r_4$, we directly obtain $t<r_3\imp P$ and $r_2<t\imp P$; in other words
  $t\apart[r_3,r_2]\imp P$. Thus we obtain $P$ as desired from \cref{prop:converse_covering_pi}.
\end{proof}

The following axiom immediately implies the converses to
\cref{prop:converse_covering_pi,cor:converse_covering_pi}. The two together roughly say that
$\pi$-closed sheaves are those for which behaviors compose.\index{behavior!composable} The soundness
of \cref{ax:covering_pi} is proven in \cref{prop:axiom10}.

\begin{axiom}\label{ax:covering_pi}\index{axiom!of $\BaseTopos$!covering for $\pi$ modality}
  For any proposition $P:\Prop$, if
  \[
    \forall(t:\Time)(q_1,q_2:\tQQ)\ldotp q_1<q_2\imp(t\apart[q_2,q_1]\imp P)\imp P
  \]
  then $\Pointwise P \imp P$.
\end{axiom}

\begin{corollary}\label{cor:covering_pi}
  For any proposition $P:\Prop$, if
  \[
    \forall(t:\Time)(q_1,q_2,q_3,q_4:\tQQ)\ldotp
      q_1<q_2<q_3<q_4 \imp \In{[q_1,q_3]}P \imp \In{[q_2,q_4]}P \imp \In{[q_1,q_4]}P
  \]
  then $\Pointwise P \imp P$.
\end{corollary}
\begin{proof}
  Assume the hypothesis; to show $\pi P\imp P$, it suffices to show that the hypothesis of
  \cref{ax:covering_pi} holds. So choose $t:\Time$, rationals $q_2,q_3:\tQQ$ with $q_2<q_3$, and
  suppose $t\apart[q_3,q_2]\imp P$. Then $t<q_3\imp P$ and $q_2<t\imp P$. It follows that for any
  $q_1,q_2$, if $q_1<q_2$ and $q_3<q_4$ then $\In{[q_1,q_4]}P$. Thus we obtain $P$ by
  \cref{cor:bounded_expanding_ins}.
\end{proof}

\begin{proposition}\label{prop:constant_pi_sheaf}
If $\tConst$ is a constant type then it is a $\pi$-sheaf.
\end{proposition}
\begin{proof}
$\tConst$ is $\pi$-separated by \cref{cor:dec_eq_pi_sep,prop:const_decidable_eq}. To show it is a $\pi$-sheaf, we use \cref{prop:characterize_j_sheaf}. That is, we assume that $\phi:\tConst\to\Prop_\pi$ satisfies $\pi P$, where $P\coloneqq\exists(c:\tConst)\ldotp\forall(c':\tConst)\ldotp(\phi c'\iff c=c')$ is the ``local singleton condition'', and we prove $P$. By \cref{ax:covering_pi}, it suffices to show $(t\apart[q_2,q_1]\imp P)\imp P$ for any $t:\Time$ and rationals $q_1<q_2$.

Assuming $t\apart[q_2,q_1]\imp P$, we can use \cref{prop:real_inf_coprime} to choose intermediate rationals $q_1<q_1'<q_2'<q_2$ such that the following hold
\begin{equation}\label{eqn:conjunction_we_need}
\begin{aligned}
	\exists(c_1:\tConst)\ldotp(q_1'<t)&\imp\forall(c':\tConst)\ldotp(\phi c'\iff c_1=c')\\
	\exists(c_2:\tConst)\ldotp(t<q_2')&\imp\forall(c':\tConst)\ldotp(\phi c'\iff c_2=c')
\end{aligned}
\end{equation}
By decidable equality, either $c_1=c_2$ or $\neg(c_1=c_2)$. In the first case we obtain $t\apart[q_2',q_1']\imp\forall(c':\tConst)\ldotp(\phi c'\iff c=c_1)$, so we are done by \cref{prop:converse_covering_pi}, because $\forall(c':\tConst)\ldotp(\phi c'\iff c=c_1)$ is $\pi$-closed. In the second case, using $(q_1'<t<q_2')\imp c_1=c_2$, we have $t\apart[q_1',q_2']$ by \cref{prop:time_apartness_complement}. It follows that $t<q_2'$ or $q_1'<t$, and either way we use obtain $\exists(c:\tConst)\ldotp\forall(c':\tConst)\ldotp(\phi c'\iff c=c')$ from \cref{eqn:conjunction_we_need}.
\end{proof}

We will need the following lemma later.
\begin{lemma}\label{lemma:pi_t<0_0<t}
For any $t:\Time$, $P,Q:\Prop_\pi$, the statement $(t<0\wedge P)\vee(0<t\wedge Q)$ is $\pi$-closed.
\end{lemma}
\begin{proof}
Choose $t':\Time$ and $q_1<q_2$, and assume $t'\apart[q_2,q_1]\imp [(t<0\wedge P)\vee(0<t\wedge Q)]$; by \cref{ax:covering_pi} it suffices to show $[(t<0\wedge P)\vee(0<t\wedge Q)]$. Using \cref{ax:torsor} let $t'=t+r$ for $r:\tRR$, and let $q_1'\coloneqq q_1-r$ and $q_2'\coloneqq q_2-r$. By \cref{prop:real_coprime} we have $(t\apart[q_2',q_1']\imp(t<0\wedge P))\vee(t\apart[q_2',q_1']\imp(0<t\wedge Q))$. Since $t<0$ and $0<t$ are $\pi$-closed by \cref{prop:Time_is_pi_closed}, we are done by \cref{prop:converse_covering_pi}.
\end{proof}

A common application of \cref{ax:covering_pi} is to use the following proposition,
which says that the subtopos $\BaseTopos_\pi$ is ``proper'', see e.g.\ \cite[Def.
VI-6.20]{Gierz.Keimel.Lawson.Mislove.Scott:2003a}.

\begin{proposition}\label{lem:exists_Pi_closed}
  The inclusion $\Prop_{\pi}\subseteq\Prop$ is continuous, i.e.\ preserves directed joins indexed by
  constant types.

  In detail, let $D$ be a constant type and let $(D,\leq)$ be a directed preorder, i.e.\ one
  satisfying
  \[
    \exists(d:D)\ldotp\top
    \qquad\text{and}\qquad
    \forall(d_1,d_2:D)\ldotp\exists(d':D)\ldotp (d_1\leq d')\wedge(d_2\leq d').
  \]
  Let $P\colon D\to\Prop$ be monotonic ($\forall(d_1,d_2)\ldotp d_1\leq d_2 \imp Pd_1 \imp Pd_2$)
  and $\pi$-closed ($\forall(d:D)\ldotp \pi Pd \imp Pd$). Then
  \[[\pi\exists(d:D)\ldotp Pd]\imp[\exists(d:D)\ldotp Pd].\]
\end{proposition}
\begin{proof}
  To set up an application of \cref{ax:covering_pi}, let $t:\Time$, let $q_1<q_2$ be rationals, and
  suppose $t\apart[q_2,q_1]\imp\exists(d:D)\ldotp Pd$. By \cref{prop:real_inf_coprime}, this
  is equivalent to the conjunction of \cref{eq:exists_Pi_closed_a,eq:exists_Pi_closed_b}:
  \begin{gather}
    \label{eq:exists_Pi_closed_a}
    \forall(q'_1:\tQQ)\ldotp (q_1<q'_1) \imp \exists(d:D)\ldotp q'_1<t\imp Pd\\
    \label{eq:exists_Pi_closed_b}
    \forall(q'_2:\tQQ)\ldotp (q'_2<q_2) \imp \exists(d:D)\ldotp t<q'_2\imp Pd
  \end{gather}
  Choose any rationals $q'_1$ and $q'_2$ satisfying $q_1<q'_1<q'_2<q_2$. Then applying
  \eqref{eq:exists_Pi_closed_a} gives us a $d_1$ such that $q'_1<t\imp P(d_1)$ and applying
  \eqref{eq:exists_Pi_closed_b} gives us a $d_2$ such that $t<q'_2\imp P(d_2)$. By directedness of
  $D$, there is a $d$ such that $d_1\leq d$ and $d_2\leq d$, so by monotonicity of $P$ we have
  $t\apart[q'_2,q'_1]\imp Pd$. This implies $P(d)$ by \cref{prop:converse_covering_pi}.
\end{proof}

\index{logic|)}

\chapter{Semantics and soundness}
\label{sec:soundness}\index{soundness}\index{semantics}

In this chapter, we prove that the temporal type theory, developed in \cref{sec:logical_prelims,sec:axiomatics} is sound in the topos $\BaseTopos$. To do so, we begin in \cref{sec:JK_semantics} by recalling the Kripke-Joyal
semantics\index{semantics!Kripke-Joyal} by which to interpret type-theoretic formulas in the topos $\BaseTopos$. Then in
\cref{sec:Dedekind_time} we discuss the sheaf of real numbers and $\Time$. We then proceed to our main goal: proving
that our type signature---i.e.\ the one atomic predicate symbol and the ten axioms presented in \cref{sec:axiomatics}---are sound in $\BaseTopos$. This is done in \cref{sec:justification}, which begins with a table
\cref{table:summary} summarizing the type signature.

\section{Categorical semantics}\label{sec:JK_semantics}\index{base site $\BaseSite$}\index{semantics!categorical}

When proving that a type signature is sound in a topos $\cat{E}$, one must give $\cat{E}$-interpretations to each type, term, and predicate in the type theory. To do so, one first assigns an object of $\cat{E}$ to each atomic type and a morphism of $\cat{E}$ to each atomic term. Recall that predicates in the type theory are identified with terms of type $\Prop$. One thus assigns a morphism, whose codomain is the subobject classifier $\Omega_{\cat{E}}$, to each atomic predicate. Using categorical semantics (see e.g.\ \cite{Jacobs:1999a}), the remaining types and terms are then recursively assigned interpretations, as briefly discussed in \cref{sec:toposes_types_logic}. The remaining part of the signature is the set of axioms, which are in particular predicates, so the remaining part of the work is to prove that each axiom is interpreted as \emph{true} in $\cat{E}$. More precisely, the morphism $X\to\Omega_{\cat{E}}$ assigned to each must factor through $\const{true}\colon 1\to\Omega_{\cat{E}}$.

We hope that the informal introduction in \cref{sec:toposes_types_logic} is enough to give the reader a basic understanding of how to interpret the recursively-constructed types and terms in the topos, e.g.\ product types and arrow types are sent to product objects and exponential objects, projections are sent to projections, etc. The least straightforward part, and the part we will be focusing on, is the $\cat{E}$-semantics of the logical connectives and quantifiers, so we explain this in more detail below.

The usual way to express the semantic interpretation of a type $X$ or term $a:A, b:B\vdash e:C$, is using \emph{Church brackets}, e.g.\ $\church{X}$ or $\church{e}\colon\church{A}\times\church{B}\to\church{C}$ in $\cat{E}$. At certain points in this section, it is typographically more convenient to simply write $X\in\cat{E}$ or $A\times B\to C$, i.e.\ to drop the Church brackets. We hope this will not cause too much confusion.

We now recall the Kripke-Joyal sheaf semantics as it relates to our particular topos $\BaseTopos=\Shv{\BaseSite}$; see e.g.\ \cite[Theorem
VI.7.1]{MacLane.Moerdijk:1992a} for the general case. The definition of $\BaseSite$ is given in \cref{def:BaseSite}, but see also \cref{rem:use_l>0}, which says that since the object $0\in\BaseSite$ has an empty covering family, we may assume $\ell>0$ for each object $\ell\in\BaseSite$.

Given a sheaf $X$, a length $\ell>0$, a section $\alpha\in X(\ell)$, and a predicate $\phi:X\to\Prop$,\index{length}
one can use forcing notation\index{forcing notation} and write $\ell\Vdash\phi(\alpha)$ to mean that
$\alpha$ is a section in the subobject determined by $\phi$, or equivalently that $\phi(\alpha)$
factors through $\top:1\to\Prop$.

For any morphism $\subint{r}{s}\colon\ell'\to\ell$ in $\IRinv\op$ and section $\alpha\in X(\ell)$ we
denote the restriction by $\restrict{\alpha}{\subint{r}{s}}$.\index{restriction} An object $\ell$ in
$\BaseSite$ has only one non-trivial covering family,\index{family!covering} namely the set of all
wavy arrows%
\footnote{
  Recall from \cref{def:BaseSite} that a wavy arrow $\subint{r}{s}\colon\ell'\wavyto\ell$ in $\IRinv\op$ is an interval
  inclusion that is strict on both sides, i.e.\ where $r>0$ and $s>0$.
}
$\subint{r}{s}\colon\ell'\wavyto\ell$ in $\IRinv\op$. Thus predicates $\phi$ will always satisfy the
following two rules:

\begin{description}\label{page:monotonicity}
  \item[Monotonicity:]\index{monotonicity} If $\ell\Vdash\phi(\alpha)$, then
    $\ell'\Vdash\phi(\restrict{\alpha}{\subint{r}{s}})$ for all $\subint{r}{s}\colon\ell'\to\ell$.
  \item[Local character:]\index{local character} If
    $\ell'\Vdash\phi(\restrict{\alpha}{\subint{r}{s}})$ for all wavy arrows
    $\subint{r}{s}\colon\ell'\wavyto\ell$, then $\ell\Vdash\phi(\alpha)$.
\end{description}
\medskip

The usual Kripke-Joyal sheaf semantics simplifies in our topos $\BaseTopos$, in particular for the semantics of $\phi\vee\psi$ and $\neg\phi$. We will explain the simplification in \cref{rem:usual_JK_sem}. A summary of the Kripke-Joyal semantics for $\BaseTopos$ is shown in \cref{table:Joyal_Kripke}. For typographical reasons, we do not
use Church brackets in this section.

\begin{table}[h]
\begin{enumerate}[label=(\roman*)]
\itemsep.25em
  \item $\ell\Vdash \phi_1(\alpha)\wedge\phi_2(\alpha)$ iff $\ell\Vdash \phi_1(\alpha)$ and
    $\ell\Vdash\phi_2(\alpha)$;
  \item $\ell\Vdash \phi_1(\alpha)\vee\phi_2(\alpha)$ iff $\ell\Vdash \phi_1(\alpha)$ or
    $\ell\Vdash\phi_2(\alpha)$;
  \item $\ell\Vdash \phi_1(\alpha)\imp\phi_2(\alpha)$ iff for all $\subint{r}{s}\colon\ell'\to\ell$,
    $\ell'\Vdash\phi_1(\restrict{\alpha}{\subint{r}{s}})$ implies $\ell'\Vdash\phi_2(\restrict{\alpha}{\subint{r}{s}})$;
  \item $\ell\Vdash\neg\phi(\alpha)$ iff for all $\subint{r}{s}\colon\ell'\to\ell$, it is not the case that
    $\ell'\Vdash\phi(\restrict{\alpha}{\subint{r}{s}})$;
  \item $\ell\Vdash\exists(y:Y)\ldotp\phi(\alpha,y)$ iff, for each wavy arrow
    $\subint{r}{s}\colon\ell'\wavyto\ell$, there exists $\beta\in Y(\ell')$ such that $\ell'\Vdash\phi(\restrict{\alpha}{\subint{r}{s}},\beta)$;
  \item $\ell\Vdash\forall(y:Y)\ldotp\phi(\alpha,y)$ iff, for all $\subint{r}{s}\colon\ell'\to\ell$ and $\beta\in
    Y(\ell')$, we have $\ell'\Vdash\phi(\restrict{\alpha}{\subint{r}{s}},\beta)$.
\end{enumerate}
\caption{The Kripke-Joyal semantics in the case of $\BaseTopos=\Shv{\BaseSite}$}
\label{table:Joyal_Kripke}
\end{table}

\begin{remark}\label{rem:usual_JK_sem}
The table \cref{table:Joyal_Kripke} summarizes the Kripke-Joyal semantics in the topos $\BaseTopos=\Shv{\BaseSite}$, where a couple things simplify, namely the rules for $\ell\Vdash\phi_1(\alpha)\vee\phi_2(\alpha)$ and $\ell\Vdash\neg\phi_1(\alpha)$, where $\alpha\in X(\ell)$. The usual rule for these are
\begin{description}
  \item[\qquad] $\ell\Vdash\phi_1(\alpha)\vee\phi_2(\alpha)$ iff, for each
    $\subint{r}{s}\colon\ell'\wavyto\ell$, either $\ell'\Vdash\restrict{\phi_1}{\subint{r}{s}}$ or $\ell'\Vdash\restrict{\phi_2}{\subint{r}{s}}$.
  \item[\qquad] $\ell\Vdash\neg\phi(\alpha)$ iff the empty family is a cover for $\ell$.
\end{description}
The only object of $\BaseSite$ covered by the empty family is $\ell=0$, which we are not considering (as discussed in the opening remarks of \cref{sec:JK_semantics}). Thus the above rule for negation reduces to the one shown in \cref{table:Joyal_Kripke} (iv). To prove that the our simplification (ii) of the rule for implication is also valid, we will argue non-constructively. Clearly (ii) implies the above, so we want the converse.

So suppose that $\ell\Vdash\phi_1(\alpha)\vee\phi_2(\alpha)$ in the sense shown here, and suppose for contradiction that we have neither $\ell\Vdash \phi_1(\alpha)$ nor $\ell\Vdash\phi_2(\alpha)$. Then by local character, there is some $\subint{r_1}{s_1}\colon\ell_1\wavyto\ell$ and $\subint{r_2}{s_2}\colon\ell_2\wavyto\ell$ such that neither $\ell_1\Vdash\phi_1(\restrict{\alpha}{\subint{r_1}{s_1}})$ nor $\ell_2\Vdash\phi_2(\restrict{\alpha}{\subint{r_2}{s_2}})$. Let $r'\coloneqq\max(r_1,r_2)$ and $s'\coloneqq\min(s_1,s_2)$, giving $\subint{r'}{s'}\colon\ell'\wavyto\ell$. Then by assumption either 
$\ell'\Vdash\restrict{\phi_1}{\subint{r'}{s'}}$ or $\ell'\Vdash\restrict{\phi_2}{\subint{r'}{s'}}$, which is a contradiction by monotonicity.
\end{remark}

The following lemma is easy to verify, but we prove it here because will use it quite often, and without mentioning it again.

\begin{lemma}\label{lemma:peel_off_package}
  Let $\Phi$ be a formula of the form $\forall(a_1:A_1)\cdots(a_n:A_n)\ldotp
  \phi(a_1,\ldots,a_n)\imp\psi(a_1,\ldots,a_n)$. Then $\Phi$ is sound iff, for every
  $\ell\in\IRinv$, every $\alpha_1\in\church{A_1}(\ell)$ and so on through
  $\alpha_n\in\church{A_n}(\ell)$, if $\ell\Vdash \phi(\alpha_1,\ldots,\alpha_n)$ then $\ell\Vdash
  \psi(\alpha_1,\ldots,\alpha_n)$.
\end{lemma}
\begin{proof}
  By definition, $\Phi$ is sound in the Kripke-Joyal semantics if $\ell\Vdash\Phi$ holds for all
  $\ell\in\IRinv$. This holds iff the following long formula does: For all $\ell\in\IRinv$,
  \begin{alignat*}{2}
    &\text{for all }\subint{r_1}{s_1}\colon\ell_1\to\ell
      &&\text{ and }a_1\in\church{A_1}(\ell_1),\\
    &\text{for all }\subint{r_2}{s_2}\colon\ell_2\to\ell_1
      &&\text{ and }a_2\in\church{A_2}(\ell_2),\\
    &\quad\vdots&&\quad\vdots\\
    &\text{for all }\subint{r_n}{s_n}\colon\ell_n\to\ell_{n-1}
      &\quad&\text{ and }a_n\in\church{A_n}(\ell_n),\text{ and }\\
    &\text{for all }\subint{r_{n+1}}{s_{n+1}}\colon\ell_{n+1}\to\ell_n,
      &&\text{ if }\ell_{n+1}\Vdash\phi(\restrict{a_1}{\subint{R_1}{S_1}},\ldots,\restrict{a_n}{\subint{R_n}{S_n}})\\
    &&&\quad\text{ then }\ell_{n+1}\Vdash\psi(\restrict{a_1}{\subint{R_1}{S_1}},\ldots,\restrict{a_n}{\subint{R_n}{S_n}})
  \end{alignat*}
  where $R_k=r_{k+1}+\cdots+r_{n+1}$ and $S_k=s_{k+1}+\cdots+s_{n+1}$.

  If the long formula above holds, then taking $r_i=0=s_i$ for each $i$, we have the desired
  conclusion: if $\ell\Vdash \phi(\alpha_1,\ldots,\alpha_n)$ then $\ell\Vdash
  \psi(\alpha_1,\ldots,\alpha_n)$. On the other hand, suppose it is the case that for every
  $\ell\in\IRinv$, if $\ell\Vdash \phi(\alpha_1,\ldots,\alpha_n)$ then $\ell\Vdash
  \psi(\alpha_1,\ldots,\alpha_n)$. Then in particular this holds for $\ell_{n+1}$ and
  $\restrict{\alpha_i}{\subint{R_i}{S_i}}$, so the long formula above also holds.
\end{proof}

\section{Constant objects and decidable predicates}\label{sec:const_decidable}\index{sheaf!constant}\index{decidable}

In order to prove that our axioms are sound, we need to interpret all of the types, terms, and predicates that appear in them. The most non-standard of these is the notion of constant types, so we record some of their semantic properties in this section.

\subsection{Constant objects in $\BaseTopos$}\label{sec:const_objects_basetopos}

For any sheaf topos $\cat{E}$, there is a unique geometric morphism
$\Gamma\colon\cat{E}\leftrightarrows\Cat{Set}\cocolon\Fun{Cnst}$ to the category of sets, and an
object $X\in\cat{E}$ is called \emph{constant} if it is in the image of the left adjoint,
$\Fun{Cnst}$. Because $\BaseTopos=\Shv{\BaseSite}$ is locally connected (\cref{prop:BaseTopos_conn_locally_conn}), constant presheaves $(\ell\mapsto C)\colon\IRinv\to\Cat{Set}$ are in fact sheaves. Thus for any constant type $\tConst$ and $\ell\in\IRinv$, we have $\church{\tConst}(\ell)=C$ and all the restriction maps are identity.

Morally, the converse should also hold: a type $\tConst$ in the type theory should be called constant if the corresponding sheaf
$\church{\tConst}$ is constant. However, in order to present a self-contained axiomatics, we gave a recursive formulation of the collection of constant types in \cref{def:constant_types}, namely as the smallest collection containing
  $\unit$, $\tNN$, $\tZZ$, $\tQQ$, and $\tRRub$, and that is closed under finite sums, finite
  products, exponentials, and decidable subtypes. In any sheaf topos, the objects $\tNN$, $\tZZ$, and $\tQQ$ \index{object!natural numbers}\index{object!integers}\index{object!rational numbers} are constant sheaves. We prove that $\tRRub$ is constant in \cref{cor:Dedekind_cut_semantics}. Sums and products of
constant sheaves are constant, as are complemented subobjects of constant sheaves. Exponentials of constant sheaves are constant by \cref{cor:exponential_constant}.

While we only consider this particular collection, the ideas in this section work in general. Later in \cref{sec:walks} we will allow ourselves to add new constant types to the signature, and considering them as such in the axiomatics.

\paragraph{Exponentiating constant types}

For any sheaf $X$, the exponential $X^\tConst$ has a simple description:\index{exponential
object} the set of sections $(X^\tConst)(\ell)$ is simply the set of functions $C\to X(\ell)$.
Indeed, $(X^\tConst)(\ell) \iso \BaseTopos(y(\ell),X^\tConst) \iso \BaseTopos(\tConst,X^{y(\ell)})
\iso \Cat{Set}(C,\Gamma(X^{y(\ell)})) \iso \Cat{Set}(C,X(\ell))$, where $\Gamma$ is the global
sections functor, right adjoint to the constant sheaf functor $\Cat{Set}\to\BaseTopos$. Given a
function $f\colon C\to X(\ell)$ and a morphism $\subint{r}{s}\colon\ell'\to\ell$, the restriction is
given by $(\restrict{f}{\subint{r}{s}})(c) = \restrict{(f(c))}{\subint{r}{s}}$.

In particular, for the type of predicates $\Prop^\tConst$, a section
$\phi\in\church{\Prop^\tConst}(\ell)$ is just a function $\phi\colon
C\to\church{\Prop}(\ell)=\Omega(\upclose[0,\ell])$, with restrictions given by
$(\restrict{\phi}{\subint{r}{s}})(c) = \phi(c) \cap\upclose[r,\ell-s]$.

\subsection{Pointwise semantics of predicates on constant types}

We just discussed the semantics of predicates on constant types, but there is another, dual, description of the set $\church{\Prop^\tConst}(\ell)$ which will be useful.
Using the characterization of opens in a domain\index{domain}\index{domain!as topological space} as
continuous maps to the domain $\Bool$,\index{domain!of booleans} we have
\[
  \church{\Prop^\tConst}(\ell) \iso \Cat{Set}(C,\Omega(\upclose[0,\ell])) \iso
  \Cat{Set}(C,\Top(\upclose[0,\ell],\Bool)) \iso \Top(\upclose[0,\ell], \Bool^C),
\]
where $\Bool^C$ is the powerset domain of $C$ (see \cref{ex:Bool^S}). Concretely, if $\phi\colon
C\to\Omega(\upclose[0,\ell])$ is a function, then we write the corresponding continuous map as follows for any $0< d\leq u<\ell$:
\begin{equation}\label{eqn:pointwise_semantics_notation}
  [d,u]\mapsto(\phi_{[d,u]}\subseteq C),\qquad\text{ given by }c\in\phi_{[d,u]}\text{ iff }[d,u]\in\phi(c).
\end{equation}
By \cref{prop:continuous_map}, continuity is simply the condition that $\phi_{[d,u]} =\!\!
\bigcup_{[d',u']\ll[d,u]}\!\phi_{[d',u']}.$ Equivalently, this time from the perspective of Kripke-Joyal semantics, we have
$c\in\phi_{[d,u]}$ iff there exists an open neighborhood of $[d,u]$---a subinterval
$\subint{r}{s}\colon\ell'\to\ell$ such that $r<d$ and $u<\ell-s$---such that
$\ell'\Vdash(\restrict{\phi}{\subint{r}{s}})(c)$.

The reason this dual description of predicates on constant types is useful is that their semantics
can often be understood ``pointwise''.\index{type!constant, pointwise semantics of}

\begin{proposition}\label{prop:coherent_pointwise}
  With notation as in \cref{eqn:pointwise_semantics_notation}, suppose given
  $\phi,\phi'\in\church{\Prop^\tConst}(\ell)$ and $\psi\in\church{\Prop^{\tConst\times\tConst'}}(\ell)\cong((\Prop^\tConst)^{\tConst'})(\ell)$. Then for any $[d,u]\in\upclose{[0,\ell]}$,
  \begin{itemize}[itemsep=3pt]
    \item $\top_{[d,u]} = C$
    \item $(\phi\wedge\phi')_{[d,u]} = \phi_{[d,u]} \cap \phi'_{[d,u]}$
    \item $\bot_{[d,u]} = \varnothing$
    \item $(\phi\vee\phi')_{[d,u]} = \phi_{[d,u]} \cup \phi'_{[d,u]}$
    \item $c\in(\exists(c':\tConst)\ldotp\psi(c'))_{[d,u]}$ iff there exists $ c'\in C$ such that
      $(c,c')\in\psi_{[d,u]}$.
  \end{itemize}
  There is less control over implication and universal quantification. For any $[d,u]\in\upclose{[0,\ell]}$,
  \begin{itemize}[itemsep=3pt]
    \item $(\phi\imp\phi') = \top$ iff $\phi_{[d,u]}\subseteq\phi'_{[d,u]}$ for all
      $[d,u]\in\upclose{[0,\ell]}$
    \item $(\forall(c:\tConst)\ldotp\phi(c)) = \top$ iff $\phi_{[d,u]}=C$ for all
      $[d,u]\in\upclose{[0,\ell]}$.
  \end{itemize}
\end{proposition}
\begin{proof}
  We prove the existential case---which is the most difficult case---to give the idea. The remaining
  cases are proven similarly.

  Consider a $\psi\in\church{\Prop^{\tConst\times\tConst'}}(\ell)$ and a point
  $[d,u]\in\upclose{[0,\ell]}$, and let $\psi'\coloneqq\exists(c':\tConst)\ldotp\psi(c')$, i.e.\
  $\psi'(c)=\exists(c':\tConst)\ldotp \psi(c,c')$. Unwinding the definitions, we find that
  $c\in\psi'_{[d,u]}$ if and only if there exists a $\subint{r}{s}\colon\ell'\wavyto\ell$ such that
  $r<d\leq u<\ell-s$, and such that for all $\subint{r'}{s'}\colon\ell''\wavyto\ell'$, there
  exists a $c'\in C$ such that $\ell''\Vdash(\restrict{\psi}{\subint{r+r'}{s+s'}})(c,c')$. But note
  that we can choose this $\ell''$ to itself be an open neighborhood of $[d,u]$. More precisely, we can
  choose $r',s'$ such that $r<r+r'<d\leq u\leq\ell-(s+s')<\ell-s$.
  
  Hence, the above simplifies to the following
  condition: $c\in\psi'_{[d,u]}$ if and only if there exists a $\subint{r}{s}\colon\ell'\wavyto\ell$
  such that $r<d$ and $u<\ell-s$, and such that there exists a $c'\in C$ such that
  $\ell'\Vdash(\restrict{\psi}{\subint{r}{s}})(c,c')$. In other words, $c\in\psi'_{[d,u]}$ if and only if
  there exists a $c'\in C$ such that $(c,c')\in\psi_{[d,u]}$.
\end{proof}

The class of first-order formulas generated by $\top$, $\wedge$, $\bot$, $\vee$, and $\exists$---in particular, not including $\imp$ and $\forall$---are
typically referred to as \emph{coherent formulas}.\index{formula, coherent}\index{coherent formula|see {formula, coherent}} If $\phi,\psi$ are
coherent formulas, then for any sequence of variables $x_1,\ldots,x_n$, a statement of the form
$\forall x_1,\ldots x_n\ldotp \phi\imp \psi$ is called a \emph{coherent axiom}.\index{axiom!coherent}
\Cref{prop:coherent_pointwise} says that the fragment of the type theory consisting of constant
types and coherent axioms has a particularly simple pointwise semantics.\index{semantics!pointwise}

\begin{remark}\label{rem:decidable}\index{decidable}
  Suppose $\phi:\tConst\to\Prop$ satisfies the additional condition $\forall(c:\tConst)\ldotp\phi
  c\vee\neg(\phi c)$, i.e.\ $\phi$ is a decidable predicate. It is easy to check that this is
  equivalent to the requirement that $\phi_{[d,u]}=\phi_{[d',u']}$, as subsets of $C$, for all
  $[d,u],[d',u']\in\upclose{[0,\ell]}$. In other words, we can identify $\phi$ with a subset of $C$.
\end{remark}

\section{Semantics of Dedekind numeric objects and $\Time$}\label{sec:Dedekind_time}

At this point, there is nothing stopping us from proving that each axiom from \cref{sec:axiomatics} is sound in $\BaseTopos$. However, these axioms make repeated use of certain condensed definitions---namely, the Dedekind numeric types, the subtype $\Time$, and certain modalities $j:\Prop\to\Prop$---the semantics of which it will be useful to unpack. We discuss the semantics of numeric types and $\Time$ in \cref{sec:sem_ded_numeric} and \cref{sec:sem_of_time}, and we discuss the semantics of the modalities in \cref{sec:semantics_of_modalities}.

\subsection{Semantics of Dedekind numeric objects}\label{sec:sem_ded_numeric}

In \cref{sec:const_decidable} we discussed the semantics of predicates on constant sheaves. A particularly important case is that of Dedekind cuts, which are
predicates $\tQQ\to\Prop$. In this section we discuss the semantics of the types $\tII$ and $\tRR$; see \cref{def:local_reals}. In particular we will show that the sheaf $\church{\tRR}$ of Dedekind reals in
$\BaseTopos=\Shv{\BaseSite}$ is constant.

\begin{definition}[Sheaf of continuous functions]
\label{def:ShFun}\index{sheaf!of continuous functions}\index{topological space!continuous functions to|see {sheaf}}
  Let $X$ be any topological space. Then we can define the \emph{sheaf of continuous $X$-valued
  maps} $\ShFun{X}\in\BaseTopos$ as follows: sections are given by
  \[
    \ShFun{X}(\ell) \coloneqq \Top(\upclose{[0,\ell]},X)
  \]
  where $\upclose{[0,\ell]}$ is given the Scott topology, and for any
  $f\in\Top(\upclose{[0,\ell]},X)$ and any $\subint{r}{s}\colon\ell\to\ell'$, the restriction
  $\restrict{f}{\subint{r}{s}}$ is given by the composition
  \begin{equation*}
    \begin{tikzcd}
      \upclose{[0,\ell']} \iso \upclose{[r,\ell-s]} \ar[r,hook]
      & \upclose{[0,\ell]} \ar[r,"f"] & X.
    \end{tikzcd}
  \end{equation*}

  In particular, if $P$ is any poset with directed joins, we define $\ShFun{P}$ by equipping $P$
  with the Scott topology.\index{topology!Scott} Hence $\ShFun{P}(\ell)$ is the set of
  Scott-continuous functions $\upclose{[0,\ell]}\to P$, i.e.\ monotonic maps which preserve directed
  joins.
\end{definition}

We consider the type of pairs $(\delta,\upsilon):(\tQQ\to\Prop)\times(\tQQ\to\Prop)$ satisfying the
axioms of \cref{def:local_reals} for the trivial modality $j=\id$. Note that each of those axioms is (equivalent to) a coherent axiom.
Hence \cref{prop:coherent_pointwise} allows us to use an alternative pointwise semantics to
understand Dedekind cuts. We explain this below and summarize in \cref{prop:Dedekind_cut_semantics}.\index{object!numeric!pointwise semantics of}\index{formula, coherent}

Recall the Dedekind numeric objects $\tLR$, $\tUR$, $\tII$, $\tIR$, $\tRR$, $\tLRub$, $\tURub$,
$\tIIub$, $\tIRub$, and $\tRRub$ from \cref{def:local_reals}, and let $\LR$, $\UR$, $\II$, $\IR$,
$\RR$, $\LRub$, $\URub$, $\IIub$, $\IRub$, and $\RRub$ be the corresponding external posets. E.g.\
$\LR$ is the partially ordered \emph{set} of subsets $D\subseteq\QQ$ which are non-empty,
down-closed, and rounded. With the strength of classical logic, we can identify $\LR$ with
$(\RR\cup\{\infty\},\leq)$, by sending a subset $D\ss\QQ$ to its supremum in $\RR$. Likewise, we can
identify $\UR$ with $(\RR\cup\{-\infty\},\geq)$, $\II$ with $\LR\times\UR$, $\IR$ with
$\{(\ubar{x},\bar{x})\in\LR\times\UR\mid\ubar{x}\leq\bar{x}\}$, and $\RR$ with the standard set of
real numbers (having the discrete order). For the unbounded versions, we identify $\LRub$ with
$(\RR\cup\{\infty,-\infty\},\leq)$, $\URub$ with $(\RR\cup\{\infty,-\infty\},\geq)$, $\IIub$ with
$\LRub\times\URub$, $\IR$ with $\{(\ubar{x},\bar{x})\in\LRub\times\URub\mid\ubar{x}\leq\bar{x}\}$,
and $\RRub$ with the $\RR\cup\{\infty,-\infty\}$ with the discrete order.

Thus, applying \cref{prop:coherent_pointwise}, we immediately derive the following very useful
proposition.

\begin{proposition}\label{prop:Dedekind_cut_semantics}
  There are isomorphisms of sheaves
  \begin{align*}
    \Church{\tLR} &\iso \ShFun{\LR} &
    \Church{\tUR} &\iso \ShFun{\UR} &
    \Church{\tII} &\iso \ShFun{\II} &
    \Church{\tIR} &\iso \ShFun{\IR} &
    \Church{\tRR} &\iso \ShFun{\RR}\\
    \Church{\tLRub} &\iso \ShFun{\LRub} &
    \Church{\tURub} &\iso \ShFun{\URub} &
    \Church{\tIIub} &\iso \ShFun{\IIub} &
    \Church{\tIRub} &\iso \ShFun{\IRub} &
    \Church{\tRRub} &\iso \ShFun{\RRub}
  \end{align*}
  where the sheaf of continuous maps $\ShFun{P}$ to a poset $P$ is defined in \cref{def:ShFun}.
\end{proposition}

From our work in \cref{sec:IR_half_plane}, we have the following.

\begin{corollary}\label{cor:Dedekind_cut_semantics}\index{half-plane}
  For any $\ell\in\IRinv$, let $(m,r) = (\frac{\ell}{2},\frac{\ell}{2})\in H$ be the midpoint-radius
  coordinates of $[0,\ell]$. Then we have the following:
  \begin{align*}
    \church{\tLR}(\ell) &\cong \{\,f\colon \upclose(m,r)\to\RR\cup\{\infty\} \mid
      f\text{ is lower semi-continuous and order-preserving}\,\} \\
    \church{\tUR}(\ell) &\cong \{\,f\colon \upclose(m,r)\to\RR\cup\{-\infty\} \mid
      f\text{ is upper semi-continuous and order-reversing}\,\} \\
    \church{\tII}(\ell) &\cong \church{\tLR}(\ell) \times \church{\tUR}(\ell) \\
    \church{\tIR}(\ell) &\cong \{\,(\ubar{f},\bar{f})\in\church{\tII}(\ell) \mid
      \forall x\in \upclose(m,r)\ldotp\ubar{f}(x)\leq\bar{f}(x)\,\} \\
    \church{\tRR}(\ell) &\cong \{\,(\ubar{f},\bar{f})\in\church{\tII}(\ell) \mid
      \forall x\in \upclose(m,r)\ldotp\ubar{f}(x)=\bar{f}(x)\,\} \cong \RR\\
    \church{\tLRub}(\ell) &\cong \{\,f\colon \upclose(m,r)\to\RR\cup\{-\infty,\infty\} \mid
      f\text{ is lower semi-continuous and order-preserving}\,\} \\
    \church{\tURub}(\ell) &\cong \{\,f\colon \upclose(m,r)\to\RR\cup\{-\infty,\infty\} \mid
      f\text{ is upper semi-continuous and order-reversing}\,\} \\
    \church{\tIIub}(\ell) &\cong \church{\tLR}(\ell) \times \church{\tUR}(\ell) \\
    \church{\tIRub}(\ell) &\cong \{\,(\ubar{f},\bar{f})\in\church{\tII}(\ell) \mid
      \forall x\in \upclose(m,r)\ldotp\ubar{f}(x)\leq\bar{f}(x)\,\} \\
    \church{\tRRub}(\ell) &\cong \{\,(\ubar{f},\bar{f})\in\church{\tII}(\ell) \mid
      \forall x\in \upclose(m,r)\ldotp\ubar{f}(x)=\bar{f}(x)\,\} \cong \RRub
\end{align*}
\end{corollary}
\begin{proof}
  This follows directly from \cref{prop:Dedekind_cut_semantics,cor:sheaf_IR_Dedekind}.
\end{proof}

The following gives an alternate proof of \cref{prop:Dedekind_cut_semantics}. It is more technical, but it has the advantage that it proves
that the semantics of our internal addition, multiplication, etc., is just pointwise addition,
multiplication, etc.

\begin{lemma}\index{predomain!semantics}
  Let $B$ be a constant predomain, such that $\church{B}$ is the constant sheaf on a set
  $\tilde{B}$. If the predomain order $\prec\colon B\times B\to\Prop$ is identified with a subset
  $\prec$ of $\tilde{B}\times \tilde{B}$ as in \cref{rem:decidable}, then $(\tilde{B},\prec)$ is an
  external predomain, $\church{\RId(B)} \iso \ShFun{\RId(\tilde{B})}$, and $\church{\Opens(B)} \iso
  \ShFun{\Opens(\tilde{B})}$.

  If $H\colon B\to B'$ is a decidable approximable mapping between constant predomains, then
  $H\colon B\times B'\to\Prop$ can similarly be identified with a subset $\tilde{H}\subseteq
  \tilde{B}\times \tilde{B}'$. Then $\tilde{H}$ defines an approximable mapping between the external
  predomains $\tilde{B}$ and $\tilde{B}'$, such that the sheaf homomorphism
  $\church{\RId(H)}\in\church{\RId(B)\to\RId(B')}$ makes the diagram commute:
  \[\begin{tikzcd}[column sep=tiny]
    \church{\RId(B)} \ar[d,"\church{\RId(H)}"'] \ar[r,phantom,"\iso"] &
    \ShFun{\RId(\tilde{B})} \ar[d,"\ShFun{\RId(\tilde{H})}"] \\
    \church{\RId(B')} \ar[r,phantom,"\iso"] &
    \ShFun{\RId(\tilde{B}')}
  \end{tikzcd}\]
\end{lemma}
\begin{proof}
  This is a simple repeated application of \cref{prop:coherent_pointwise,rem:decidable}.

  $(\tilde{B},\prec)$ is an external predomain because the single axiom in \cref{def:predomain} is
  equivalent to a pair of coherent axioms\index{axiom!coherent}, while $\church{\RId(B)} \iso \ShFun{\RId(\tilde{B})}$
  because the condition for $I\colon B\to\Prop$ to be a rounded ideal\index{ideal!rounded} is
  equivalent to the collection of coherent axioms:
  \begin{gather*}
    \exists(b:B)\ldotp Ib \\
    \forall(b,b':B)\ldotp Ib \wedge b'\prec b \imp Ib' \\
    \forall(b_1,b_2:B)\ldotp Ib_1 \wedge Ib_2 \imp \exists(b':B)\ldotp b_1\prec b' \wedge b_2\prec
    b'
  \end{gather*}
  hence $I\in\church{B\to\Prop}(\ell)$ is in $\church{\RId(B)}(\ell)$ if and only if for all
  $[d,u]\in\upclose[0,\ell]$, $I_{[d,u]}\in\RId(\tilde{B})$. Finally, the isomorphism
  $\church{B\to\Prop}(\ell)\iso\Top(\upclose[0,\ell],\Bool^{\tilde{B}})$ shows that this assignment
  is continuous, giving $\church{\RId(B)}\iso\ShFun{\RId(\tilde{B})}$. Likewise for
  $\church{\Opens(B)}\iso\ShFun{\Opens(\tilde{B})}$.

  From \cref{rem:decidable}, if $H$ is decidable then it determines a subset $\tilde{H}\subset
  \tilde{B}\times \tilde{B}'$. The axioms in \cref{def:approx_map} are easily seen to be equivalent
  to a collection of coherent axioms, so $H$ is an approximable mapping if and only if $\tilde{H}$
  is an external approximable mapping. The commuting diagram follows from the fact that
  $\exists(b:B)\ldotp Ib\wedge H(b,b')$ is a coherent formula\index{formula, coherent}, so for any
  $I\in\church{\RId(B)}(\ell)$ and any $[d,u]\in\upclose[0,\ell]$,
  \[
    \RId(H)(I)_{[d,u]} = \{\,b'\in \tilde{B}' \mid \exists(b\in \tilde{B})\ldotp b\in I_{[d,u]}
    \wedge (b,b')\in H_{[d,u]}\,\} = \RId(\tilde{H})(I_{[d,u]}). \qedhere
  \]
\end{proof}

\begin{corollary}\label{cor:arithmetic_semantics}\index{semantics!of arithmetic}
  The semantics of arithmetic and order on $\tII$ agrees with pointwise arithmetic and order on $\ShFun{\II}$.
\end{corollary}

In \cref{sec:Dedekind_j}, we discussed the type of real numbers in various
subtoposes, namely those corresponding to modalities $\In{}$, $\At{}$, $\See{}$, and $\pi$. This
level of understanding is not necessary for the purposes of verifying the soundness of our axioms,
so we do not discuss it further now. However, we do need to better understand the semantics of these
modalities, so we turn to that in \cref{sec:semantics_of_modalities}. But first we discuss the semantics of $\Time$.

\subsection{Semantics of $\Time$}\label{sec:sem_of_time}

Recall that $\Time$\index{Time@$\Time$!semantics of} was defined as a sheaf in \cref{eqn:def_time},
by
\begin{equation}\label{eqn:def_time_again}
  \Time(\ell)=\{\, (d,u)\in\RR^2 \mid u-d=\ell \,\}
\end{equation}
where for any subinterval $\subint{r}{s}\colon\ell'\to\ell$, the restriction is given by
$\restrict{(d,u)}{\subint{r}{s}} = (d+r,u-s)$. This is an external definition; its internal counterpart is
also denoted $\Time$ and was defined in \cref{eqn:define_time} to be a subtype of $\tII$, the type of bounded
rounded pairs of cuts $(\delta,\upsilon):(\tQQ\to\Prop)^2$. Thus we want to show that the sheaf $\Time$ is a
candidate semantics for the type $\Time$; this is the content of the following lemma.

\begin{lemma}\label{lemma:Time_sub_II}
  The sheaf $\Time$ is isomorphic to a subsheaf of $\Church{\tII}$.
\end{lemma}
\begin{proof}
  It suffices to construct a monomorphism $c\colon\Time\to\Church{\tII}$. For any
  $\ell\in\IRinv\op$, we have from \cref{prop:Dedekind_cut_semantics} that $\Church{\tII}(\ell) \iso
  \ShFun{\II} \iso \ShFun{\LR}\times\ShFun{\UR}$. Given any section $(d,u)\in\Time(\ell)$, define
  $c_\ell(d,u)\coloneqq(\ubar{d},\bar{u})$, where $\ubar{d}\colon\upclose{[0,\ell]}\to\LR$ is the
  continuous map $\ubar{d}([d',u'])=d+d'$, and $\bar{u}\colon\upclose{[0,\ell]}\to\UR$ is the
  continuous map $\bar{u}([d',u'])=u-(\ell-u')$.

  The maps $\ubar{d}$ and $\bar{u}$ are obviously monotonic (recall that $\bar{\RR}=(\RR,\geq)$) and
  continuous. It is also clear that the function $c_\ell\colon\Time(\ell)\to\Church{\II}(\ell)$ is
  injective for all $\ell$, and that these functions commute with restrictions, thus defining a
  monomorphism $c\colon\Time\to\Church{\tII}$.
\end{proof}

For any $\ell$ and $t=(d_t,u_t):\Time(\ell)$ and $r\in\RR=\church{\tRR}(\ell)$, we have
\begin{equation}\label{eqn:semantics_time}
  (\ell\Vdash r<t)\;\text{ iff }\;r\leq d_t
  \quad\qquad\text{and}\qquad\quad
  (\ell\Vdash t<r)\;\text{ iff }\;u_t\leq r.
\end{equation}
Indeed, by \cref{cor:arithmetic_semantics} and the above proof, $\ell\Vdash r<t$ holds iff for all $d'\in\upclose{[0,\ell]}$ we have $r<d_t+d'$, which is equivalent to $r\leq d_t$.

\begin{remark}
We saw in \cref{ex:Time_subtype_rvar} that, internally, there is a containment of types $\Time\ss\tRR_\pi$, and by definition there is a containment of types $\Time=\{t:\tII\mid\unitspeed t\}\ss\tII$.%
\footnote{We will see in \cref{prop:pi_inclusion_arith} that there is also containment of types $\tRR_\pi\ss\tII$.}
In \cref{lemma:Time_sub_II} we showed that the sheaf $\Time$ is a valid semantics for this type, but in fact it is not the only one.

As the name of the atomic term ``$\unitspeed$\!'' suggests, our axioms only constrain the semantics of the type $\Time=\{t:\tII\mid\unitspeed t\}$ to be a sheaf of ``unit-speed'' real-valued functions. More precisely, for any nonzero real number $0\neq r\in\RR$, there is a sheaf $\Time_r$ given by $\Time_r(\ell)\coloneqq\{\,(d,u)\in\RR^2\mid u-d=\ell*r\}$.

This is significant when thinking about the semantics of $\Time$. Our convention in \cref{eqn:def_time_again} and throughout this book is $\Time=\Time_1$. From this point of view one should interpret each section of time as a \emph{clock behavior}, counting upward like a stopwatch. While this is straightforward, it does have counter-intuitive aspects if we accidentally think of $t:\Time$ as indicating a point in time, e.g.\ the point where the clock shows $0$. This is often the natural interpretation of the ``rough monotonicity'' formula \eqref{eqn:relaxed_monotonicity} or simply the monotonicity formula
\begin{equation}\label{eqn:monotonicity}
  \forall(t_1,t_2:\Time)\ldotp t_1\leq t_2\imp f(t_1)\leq f(t_2).
\end{equation}
Namely it looks like $t_1\leq t_2$ means $t_1$ comes before $t_2$. But in fact $t_1\leq t_2$ means clock $t_1$ shows a value that is less than clock $t_2$, so its $0$-point $t_1$ comes \emph{after} that of $t_2$. Thus \cref{eqn:monotonicity} actually says that $f$ is \emph{decreasing}!

Consider instead the sheaf $\Time_{-1}$. Here it is appropriate to interpret $t:\Time$ as indicating a \emph{moment in time}, with sections acting as a count-down. Using $\Time_{-1}$ as semantics, \cref{eqn:monotonicity} indeed says that $f$ is increasing.

One way to get the best of both worlds is to use $\Time_1$, as we do, but to restrict oneself to terms in which at most one $\Time$ variable occurs, using the torsor axiom \cref{ax:torsor} to convert all other occurrances to real numbers. For example, the above formula becomes
\[\forall(t:\Time)(r:\tRR)\ldotp(0\leq r)\imp f(t)\leq f(t+r),\]
and it has the expected semantics, that $f$ is increasing.
\index{Time@$\Time$!direction of}
\end{remark}

\section{Semantics of the modalities
  \texorpdfstring{$\In{}$, $\See{}$, $\At{}$, and $\pi$}%
  {\unichar{"1D704}, @, \unichar{"1D709}, and \unichar{"1D70B}}}
  \label{sec:semantics_of_modalities}\index{modality!semantics of}

Recall that for any $d,u:\tQQ$ the temporal modalities $\In{[d,u]}$, $\See{[d,u]}$, and $\At{[d,u]}$
are defined by
\begin{gather*}
  \In{[d,u]}P\coloneqq (d<t<u)\imp P
  \qquad\qquad
  \See{[d,u]}P\coloneqq(t\apart[u,d])\vee P
  \\
  \At{[d,u]}P\coloneqq(P\imp t\apart[u,d])\imp t\apart[u,d]
\end{gather*}
where $t\apart[u,d]$ means $(t<u)\vee(d<t)$. The modality $\pi:\Prop\to\Prop$ is defined by
\[
  \pi P\coloneqq\forall(t:\Time)\ldotp\At{[0,0]}P.
\]

\begin{proposition}\label{prop:semantics_In_At_See}
  Suppose given $P\in\church{\Prop}(\ell)$ and $t=(d_t,u_t)\in\church{\Time}(\ell)$, for some
  $\ell$, as well as $d,u\in\RR=\church{\tRR}$. Consider the open set
  $P\in\Omega(\upclose[d_t,u_t])$.
  \begin{enumerate}\itemsep=2pt
    \item If $u\leq d$ then $\ell\Vdash\In[t]{[d,u]}P$.
    \item If $d<u$ then $\ell\Vdash\In[t]{[d,u]}P$ iff $\upclose[d,u]\cap\upclose[d_t,u_t]\ss P$.
    \item $\ell\Vdash\See[t]{[d,u]}P$ iff it is the case that $[d,u]\in\upclose[d_t,u_t]$ implies
      $P=\upclose[d_t,u_t]$.
    \item Suppose $d\leq u$. Then $\ell\Vdash\At[t]{[d,u]}P$ iff $[d,u]\in\upclose[d_t,u_t]$ implies
      $[d,u]\in P$.
    \item Suppose $u<d$. Then $\ell\Vdash\At[t]{[d,u]}P$ iff $\upclose[u,d]\cap
      \upclose[d_t,u_t]\neq\varnothing$ implies $\upclose[u,d]\cap P\neq \varnothing$.
  \end{enumerate}
\end{proposition}
\begin{proof}
  The meaning of $\ell\Vdash\In{[d,u]}P$ is the following statement: for any
  $\subint{r}{s}\colon\ell'\to\ell$, if $d\leq d_t+r\leq u_t-s\leq u$ then $\ell'\Vdash
  \restrict{P}{\subint{r}{s}}$. Since $\ell'>0$, this statement is vacuously true if $u\leq d$, in
  which case $\ell\Vdash\In{[d,u]}P$; this proves claim 1. If $d<u$, the statement is true for all
  $r,s$ iff it is true when $r\coloneqq\max(0,d-d_t)$ and $s\coloneqq\max(0,u_t-u)$, i.e.\ when
  $d_t+r=\max(d_t,d)$ and $u_t-s=\min(u_t,u)$. This is the case if
  $\upclose[d,u]\cap\upclose[d_t,u_t]\ss P$, proving claim 2.

  By \cref{eqn:semantics_time} we have $\ell\Vdash (d<t)\vee(t<u)\vee P$ iff either $d\leq d_t$ or
  $u_t\leq u$ or $\ell\Vdash P$. Claim 3 follows because $\ell\Vdash P$ iff $P=\upclose[d_t,u_t]$.

  For claim 4, suppose $d\leq u$. Then $\ell\Vdash(P\imp t\apart[u,d])\imp t\apart [u,d]$ means, by
  contrapositive, that if $d_t<d\leq u<u_t$ then there exists some $\subint{r}{s}\colon\ell'\to\ell$
  such that $d_t+r<d\leq u<u_t-s$ and $\upclose[d_t+r,u_t-s]\ss P$. The hypothesis is equivalent to
  $[d,u]\in\upclose[d_t,u_t]$ and the conclusion is equivalent to $[d,u]\in P$.

  For claim 5, suppose $u<d$. Then one can prove by cases that $(d_t<d)\wedge(u<u_t)$ is equivalent
  to $\upclose[u,d]\cap\upclose[d_t,u_t]\neq\varnothing$. Again by contrapositive,
  $\ell\Vdash\At{[d,u]}P$ means that if $\upclose[u,d]\cap\upclose[d_t,u_t]\neq\varnothing$ then
  there exists $\subint{r}{s}\colon\ell'\to\ell$ such that
  $\upclose[u,d]\cap\upclose[d_t+r,u_t-s]\neq\varnothing$ and $\upclose[d_t+r,u_t-s]\ss P$. The
  conclusion is equivalent to saying there exists $d'<u'$ such that
  $\upclose[d',u']\ss\upclose[u,d]\cap P$. This clearly implies that $\upclose[u,d]\cap P\neq
  \varnothing$, and the converse follows from the fact that $\IR$ has a basis consisting of
  intervals $\upclose[d',u']$ with $d'<u'$.
\end{proof}

\begin{corollary}\label{cor:semantics_In_At_See}
  Suppose given $P\in\church{\Prop}(\ell)$ and $t=(d_t,u_t)\in\church{\Time}(\ell)$, for some
  $\ell$, as well as $d,u\in\RR=\church{\tRR}$ with $d\leq u$.
  \begin{enumerate}\itemsep=2pt
    \item $\ell\Vdash\In[t]{[d,u]}P$ iff either $d=u$ or $\ell'\Vdash \restrict{P}{\subint{r}{s}}$,
      where $\subint{r}{s}\colon\ell'\to\ell$ is given by $r\coloneqq\max(0,d-d_t)$ and $s\coloneqq
      \max(0,u_t-u)$.
    \item $\ell\Vdash\See[t]{[d,u]}P$ iff it is the case that $d_t<d\leq u<u_t$ implies $\ell\Vdash
      P$.
    \item $\ell\Vdash\At[t]{[d,u]}P$ iff it is the case that $d_t<d\leq u<u_t$ implies: there exists
      $\subint{r}{s}\colon \ell'\to\ell$ with $d_t+r<d\leq u<u_t-s$, such that $\ell'\Vdash
      \restrict{P}{\subint{r}{s}}$.
  \end{enumerate}
\end{corollary}
\begin{proof}
  Follows directly from \cref{prop:semantics_In_At_See}.
\end{proof}

\begin{proposition}\label{prop:semantics_pi}
  Suppose given $P\in\church{\Prop}(\ell)$ for some $\ell\in\IR$. Then $\ell\Vdash\pi P$ iff, for
  all $0<a<\ell$ there exists $\subint{r}{s}\colon\ell'\to\ell$ such that $r<a<\ell-s$ and
  $\ell'\Vdash\restrict{P}{\subint{r}{s}}$.
\end{proposition}
\begin{proof}
  $\ell\Vdash\pi P$ implies that for all $t\in\Time(\ell)$ we have $\At[t]{[0,0]}P$. By
  \cref{cor:semantics_In_At_See}, this implies that for all $d<0<d+\ell$ there exists
  $\subint{r}{s}\colon\ell'\to\ell$ with $0\leq r<-d$ and $0\leq s<d+\ell$, such that $\ell'\Vdash
  \restrict{P}{\subint{r}{s}}$. We have the result by letting $a\coloneqq -d$.
\end{proof}

Suppose $\tConst$ is the constant sheaf on the set $C$. For any $\phi\in\church{\Prop^\tConst}$,
recall the pointwise-semantics notation $\phi_{[d,u]}\ss C$ from
\cref{eqn:pointwise_semantics_notation}.
	
\begin{proposition}\label{prop:modality_semantics_pointwise}
  Choose $P\in\church{\Prop^\tConst}(\ell)$ and $t=(d_t,u_t)\in\church{\Time}(\ell)$, and
  take $[d',u']\in\upclose{[d_t,u_t]}$. For any $c\in C$ and $d,u\in\RR$,
  \begin{itemize}
    \item $c\in(\See{[d,u]}P)_{[d',u']}$ iff $[d',u']\sqss[d,u] \imp c\in P_{[d',u']}$.
    \item $c\in(\At{[d,u]}P)_{[d',u']}$ iff $[d',u']\sqss[d,u] \imp c\in P_{[d,u]}$.
    \item $c\in(\pi P)_{[d',u']}$ iff $c\in P_{[q,q]}$ for all $q$ with $d'\leq q\leq u'$.
  \end{itemize}
  Moreover,
  \begin{itemize}
    \item $P(c)$ is $\See{[d,u]}$-closed for all $c\in C$ iff, for all $[d',u']\in\upclose{[d_t,u_t]}$, either $[d',u']\sqss[d,u]$ or $P_{[d',u']}=C$.
    \item $P(c)$ is $\At{[d,u]}$-closed for all $c\in C$ iff, for all
      $[d',u']\in\upclose{[d_t,u_t]}$, if $[d',u']\sqss[d,u]$ then $P_{[d',u']}=P_{[d,u]}$,
      and if $[d',u']\not\sqss[d,u]$, then $P_{[d',u']}=C$.
    \item $P(c)$ is $\pi$-closed for all $c\in C$ iff for all $[d',u']\in\upclose{[0,\ell]}$,
      $P_{[d',u']}=\bigcap_{[d',u']\sqss[q,q]} P_{[q,q]}$.
  \end{itemize}
\end{proposition}
\begin{proof}
  The first three are straightforward using
  \cref{prop:coherent_pointwise,cor:semantics_In_At_See,prop:semantics_pi}. The second three
  immediately follow.
\end{proof}

\section{Proof that each axiom is sound}
\label{sec:justification}\index{soundness}\index{base topos $\BaseTopos$}

We are now set up to prove---with little effort---that each axiom is sound in $\BaseTopos$. See \cref{table:summary} for a summary of the axioms. In every proof in this
section, we will use \cref{lemma:peel_off_package}---often without mentioning it---to peel off the
top-level universal quantifiers and the top-level implication (if any such are present).

\begin{table}\index{axioms, summary table of}\index{notation}
\paragraph{Atomic symbols and other notation:}
\begin{itemize}
  \item $\tQQ$ is the usual type of rational numbers and $\Prop$ is the type of truth values.
  \item $\tIIub$ is the type of pairs $(\delta,\upsilon):(\tQQ\to\Prop)\times(\tQQ\to\Prop)$ that satisfy the following:
    \begin{itemize}
      \item $\forall(q,q':\tQQ)\ldotp(q<q')\imp[(\delta q'\imp \delta q)\wedge(\upsilon q\imp\upsilon q')]$
      \item $\forall(q:\tQQ)\ldotp[\delta q\imp\exists(q':\tQQ)\ldotp(q<q')\wedge\delta q']\wedge[\upsilon q\imp\exists(q':\tQQ)\ldotp(q'<q)\wedge\upsilon q']$
    \end{itemize}
		$\tII\ss\tIIub$ is the type additionally satisfying
		\begin{itemize}
      \item $\exists(q,q':\tQQ)\ldotp\delta q\wedge\upsilon q'$
		\end{itemize}
    $\tRR\ss\tII$ and $\tRRub\ss\tIIub$ are the subtypes additionally satisfying:
    \begin{itemize}
      \item $\forall(q,q':\tQQ)\ldotp[(\delta q\wedge\upsilon q')\imp(q<q')]\wedge[(q<q')\imp(\delta q'\vee\upsilon q)]$
    \end{itemize}
  For any $r=(\delta,\upsilon):\tRR$, we write $q<r$ to denote $\delta q$ and $r<q$ to denote $\upsilon q$. For any $(\delta',\upsilon'):\tII$, we write $x+r$ to denote $(\delta'',\upsilon'')$ with $\delta''q''\iff\exists(q,q')\ldotp\delta q\wedge\delta'q'\wedge q''<q+q'$ and $\upsilon''q''\iff\exists(q,q')\ldotp\upsilon q\wedge\upsilon'q'\wedge q+q'<q''$
  \item $\tConst$ represents any constant type; see \cref{def:constant_types}.
  \item $\unitspeed\!\!:\tII\to\Prop$ is an atomic predicate symbol, and $\Time$ is the subtype $\Time=\{t:\tII\mid\unitspeed t\}$. For any $t=(\delta,\upsilon):\Time$, we write $q<t$ to denote $\delta q$ and $r<q$ to denote $\upsilon q$.
  \item For any $d,u:\tRR$, let $t\apart[d,u]$ be shorthand for $(d<t)\vee(t<u)$. We define four modalities $\In[t]{[d,u]}$,$\See[t]{[d,u]},$ $\At[t]{[d,u]}$, and $\pi$ (see \cref{lemma:modality}), given as follows:
\[
  \begin{array}{lll}
    \In[t]{[d,u]}P\iff((d<t<u)\imp P)&&
    \See[t]{[d,u]}P\iff(t\apart[d,u]\vee P)\\
    \At[t]{[d,u]}P\iff((P\imp t\apart[d,u])\imp t\apart[d,u])&&
    \pi P\iff\forall(t:\Time)\ldotp\At[t]{[0,0]}P
  \end{array}
\]
\end{itemize}
\paragraph{Axioms}~\\
\begin{tabular}{ll}
  \cref{ax:decidable_reals}
    &$\forall ((\delta,\upsilon):\tRRub)(q:\tQQ)\ldotp (\delta q \vee \neg\delta q) \wedge (\upsilon q \vee\neg\upsilon q)$
\\
  \cref{ax:distributivity}
  &$\forall(P:\Prop)(P':\tConst\to\Prop)\ldotp\left[ \forall(c:\tConst)\ldotp P\vee
    P'(c) \right] \imp \left[ P \vee \forall(c:\tConst)\ldotp P'(c) \right]$
\\
  \cref{ax:time_inhabited}
    &$\exists(t:\Time)\ldotp\top$
\\
  \cref{ax:time_complementary_cuts}
    &$\forall(q:\tQQ)(t:\Time)\ldotp(\neg(q<t)\iff(t<q))\wedge(\neg(t<q)\iff(q<t))$
\\
  \cref{ax:time_not_decidable_d}
    &$\forall(t:\Time)(P:\tQQ\to\Prop)\ldotp (\forall q\ldotp Pq\vee\neg Pq) \imp (\forall q\ldotp Pq\vee q<t) \imp \exists q\ldotp Pq\wedge q<t$
\\
  \cref{ax:time_not_decidable_u}
    &$\forall(t:\Time)(P:\tQQ\to\Prop)\ldotp (\forall q\ldotp Pq\vee\neg Pq) \imp (\forall q\ldotp Pq\vee t<q) \imp \exists q\ldotp Pq\wedge t<q$
\\
  \cref{ax:torsor_act}
  	&$\forall(t:\Time)(r:\tRR)\ldotp t+r\in\Time$
\\
  \cref{ax:torsor_diff}
  	&$\forall(t_1,t_2:\Time)\ldotp\exists!(r:\tRR)\ldotp t_1+r=t_2$
\\
  \cref{ax:coprime_u}
    &$\forall(t:\Time)(P,Q:\Prop)(q:\tQQ)\ldotp(t<q \imp (P\vee Q)) \imp [(t<q \imp P) \vee (t<q \imp Q)]$
\\
  \cref{ax:coprime_d}
    &$\forall(t:\Time)(P,Q:\Prop)(q:\tQQ)\ldotp(q<t \imp (P\vee Q)) \imp [(q<t \imp P) \vee (q<t \imp Q)]$
\\
  \cref{ax:infinitary_coprime_u}
    &$\exists(c:\tConst)\ldotp\top\imp\forall(t:\Time)(u:\tRRub)(P:C\to\Prop)\ldotp(t<u \imp$
    \\&\qquad$ \exists(c:\tConst)\ldotp P(c)) \imp[\forall(u':\tQQ)\ldotp
      u'<u\imp\exists(c:\tConst)\ldotp t<u'\imp P(c)]$
\\
  \cref{ax:infinitary_coprime_d}
    &$\exists(c:\tConst)\ldotp\top\imp\forall(t:\Time)(d:\tRRub)(P:C\to\Prop)\ldotp(d<t \imp $
    \\&\qquad$\exists(c:\tConst)\ldotp P(c)) \imp[\forall(d':\tQQ)\ldotp
      d<d'\imp\exists(c:\tConst)\ldotp d'<t\imp P(c)]$
\\
  \cref{ax:prime}
    &$\forall(t:\Time)(P,Q:\Prop)(d,u:\tQQ)\ldotp(d\leq u)\imp$
    \\&\qquad$\left[ (P\wedge Q) \imp t\apart[u,d] \right] \imp\left[ (P \imp t\apart[u,d]) \vee (Q \imp t\apart[u,d]) \right]$
\\
  \cref{ax:enough_points}
    &$\forall(t:\Time)(P\colon\Prop)\ldotp[\forall(d,u:\tQQ)\ldotp d<u \imp \At[t]{[d,u]}P] \imp P$
\\
  \cref{ax:point_to_open}
    &$\forall(t:\Time)(P\colon\Prop)(d,u:\tRR)\ldotp d\leq u \imp$ \\
    &\qquad$\At[t]{[d,u]}P\imp \See[t]{[d,u]}\exists(d',u':\tQQ)\ldotp d'<d\leq u<u' \wedge \In[t]{[d',u']}P$
\\
  \cref{ax:covering_pi}
    &$\forall(P:\Prop)\ldotp[\forall(t:\Time)(q_1,q_2:\tQQ)\ldotp q_1<q_2 \imp (t\apart[q_2,q_1]\imp P)\imp P]\imp\pi P\imp P$
\end{tabular}
\caption{Summary of notation and axioms from \cref{sec:axiomatics}}
\label{table:summary}
\end{table}

\begin{proposition}\label{prop:axiom1}\index{real numbers!as constant}\index{axiom!of $\BaseTopos$!decidable reals}
  \Cref{ax:decidable_reals}, denoted $\phi$ below, is sound in $\BaseTopos$:
  \[
    \phi\coloneqq\forall ((\delta,\upsilon):\tRRub)(q:\tQQ)\ldotp (\delta q \vee \neg\delta q)
    \wedge (\upsilon q \vee\neg\upsilon q).
  \]
\end{proposition}
\begin{proof}
  We must show that for every $\ell\in\BaseSite$, every pair of cuts
  $(\delta,\upsilon)\in\tRR(\ell)$ and every rational $q\in\QQ=\tQQ(\ell)$, we have
  \[
    \ell\Vdash (\delta q \vee \neg\delta q) \wedge (\upsilon q \vee\neg\upsilon q).
  \]
  Note that $\tRRub$ is the constant sheaf on the set $\RRub$ of (set-theoretic) extended real
  numbers by \cref{cor:Dedekind_cut_semantics}; hence we can identify $(\delta,\upsilon)$ with an
  extended real number $r\in\RR\cup\{-\infty,\infty\}$. Thus the above holds iff we have
  $\ell\Vdash(q<r\vee r\leq q)\wedge(r<q\vee q\leq r)$, which is true because
  $\RR\cup\{-\infty,\infty\}$ is a linear order.
\end{proof}

\begin{proposition}\label{prop:axiom2}\index{axiom!of $\BaseTopos$!dual Frobenius}
  Let $\tConst$ be the constant sheaf on some set $C$. Then \cref{ax:distributivity}, denoted $\phi$
  below, is sound in $\BaseTopos$:
  \[
    \phi\coloneqq\forall(P:\Prop)(P':\tConst\to\Prop)\ldotp\left[ \forall(c:\tConst)\ldotp P\vee
    P'(c) \right] \imp \left[ P \vee \forall(c:\tConst)\ldotp P'(c) \right].
  \]
\end{proposition}
\begin{proof}
	Choose $\ell\in\BaseSite$, and sections $P\in\Prop(\ell)$ and $P'\in\Prop^\tConst(\ell)$. Assume
	that $\ell\Vdash\forall(c:\tConst)\ldotp P\vee P'(c)$ holds; we want to show $\ell\Vdash P\vee\forall(c:\tConst)\ldotp P'(c)$. Arguing classically, suppose
	$\ell\not\Vdash P$. Then for each $c\in\tConst(\ell)$, we have $\ell\Vdash P'(c)$.
  
  We want to prove $\ell\vdash\forall(c:C)\ldotp P'(c)$, so choose $f\colon\ell'\to\ell$ and
  $c'\in \tConst(\ell')$. Since $\tConst$ is constant, $\tConst(\ell)=C=\tConst(\ell')$ for every
  $\ell'$. Thus $\ell\Vdash P'(c')$, and it follows that $\ell'\Vdash P'(c')$ by monotonicity (page \pageref{page:monotonicity}).
\end{proof}

\begin{proposition}\label{prop:axiom3}\index{axiom!of $\BaseTopos$!basics of $\Time$}
  \Cref{ax:time_inhabited,ax:time_complementary_cuts,ax:time_not_decidable_d,ax:time_not_decidable_u},
  denoted $\phi_1$, $\phi_2$, $\phi_3,$ and $\phi_4$ below, are sound in $\BaseTopos$:
  \begin{align*}
    \phi_1&\coloneqq\exists(t:\Time)\ldotp\top\\
    \phi_2&\coloneqq\forall(q:\tQQ)(t:\Time)\ldotp(\neg(q<t)\iff(t<q))\wedge(\neg(t<q)\iff(q<t))\\
    \phi_3&\coloneqq\forall(t:\Time)(P:\tQQ\to\Prop)\ldotp (\forall q\ldotp Pq\vee\neg Pq) \imp
      (\forall q\ldotp Pq\vee q<t) \imp \exists q\ldotp Pq\wedge q<t\\
    \phi_4&\coloneqq\forall(t:\Time)(P:\tQQ\to\Prop)\ldotp (\forall q\ldotp Pq\vee\neg Pq) \imp
      (\forall q\ldotp Pq\vee t<q) \imp \exists q\ldotp Pq\wedge t<q.
  \end{align*}
\end{proposition}
\begin{proof}

  Refer to \cref{sec:sem_of_time} for the semantics of $\Time$. To prove $\ell\Vdash\phi_1$ it 
	suffices to recognize that the set $\Time(\ell)=\{\, (d,u)\in\RR^2  \mid u-d=\ell \,\}$ is nonempty 
	for any $\ell$. To prove $\ell\Vdash\phi_2$, take $q\in\QQ$ and $t=(d,u)\in\Time(\ell)$; we want to 
	show  $\ell\Vdash\neg(q<t)\iff(t<q)$, the other conjunct being similar.
  
  Recall from  \cref{lemma:Time_sub_II} that the semantics of $q<t$ and $t<q$ are given by 
  \cref{eqn:semantics_time}. Thus the implication $(t<q)\imp(q<t)\imp\bot$ is easy: since  $u-d=\ell\geq 0$, we have $d\leq u$ so it cannot be that both $q\leq d$ and $u\leq q$. For the converse, suppose that $\ell'\Vdash\neg(q<t)$, i.e.\ for all $\subint{r}{s}\colon\ell'\to\ell$, if $\ell'>0$ then it is not the case that $q\leq d+r$. We want to show $u\leq q$. Recall that $\ell=u-d$, by definition of the category $\IRinv$, so for any $r+\ell'+s=\ell$ with $r,s\geq0$ and $\ell'>0$ we have $d+r<q$. Consider the case $s=0$, where $d+r=u-\ell'$. For any $\ell'>0$ we get $u-\ell'<q$; this implies $u\leq q$ as desired.
  
  The proofs of $\ell\Vdash\phi_3$ and $\ell\Vdash\phi_4$ are similar, so we only prove the former.
  Choose $t=(d,u)$ and a decidable predicate $P$; by \cref{rem:decidable} we can identify $P$ with a
  subset $P\ss\QQ$. Assume $\ell\Vdash(\forall q\ldotp Pq\vee q<t)$. Then, in particular, for all
  $q\in\QQ$ either $q\in P$ or $q\leq d$. We must show $\ell\Vdash\exists(q:\QQ)\ldotp Pq\vee q<d$,
  i.e.\ that for any wavy arrow $\subint{r}{s}\colon\ell'\wavyto\ell$, there exists $q$ such that
  $q\in P$ and $q\leq d+r$. Since by definition $0<r$, simply take $q$ to satisfy $d<q<d+r$.
\end{proof}

\begin{proposition}\label{prop:axiom4}\index{torsor!$\Time$ as}\index{axiom!of $\BaseTopos$!$\Time$ as torsor}
  \Cref{ax:torsor_act,ax:torsor_diff}, denoted $\phi_1$ and $\phi_2$ below, are sound in
  $\BaseTopos$:
  \begin{align*}
    \phi_1&\coloneqq\forall(t:\Time)(r:\tRR)\ldotp t+r\in\Time\\
    \phi_2&\coloneqq\forall(t_1,t_2:\Time)\ldotp\exists(r:\tRR)\ldotp
    (t_1+r=t_2)\wedge\forall(r':\tRR)\ldotp(t_1+r'=t_2)\imp r=r'.
  \end{align*}
\end{proposition}
\begin{proof}
  By \cref{cor:Dedekind_cut_semantics,lemma:Time_sub_II}, we can regard any $t=(d,u)\in\Time(\ell)$
  and any $r\in\RR=\church{\tRR}$ as elements of $\church{\tII}$. That is, each is a pair of
  functions $U\to\RR$, say $(\ubar{d},\bar{u})$ and $(\ubar{r},\bar{r})$ respectively, where
  $U=\{(a,b)\in H\mid (\ell/2,\ell/2)\ll (a,b)\}$. Here $\ubar{d}(a,b)=d+a$, $\bar{u}(a,b)=u+b$,
  $\ubar{r}(a,b)=r$, and $\bar{r}(a,b)=r$.

  To prove $\ell\Vdash\phi_1$, take $t$ and $r$ as above. By \cref{cor:arithmetic_semantics}, the
  sum $t+r$ is the pointwise sum $(a,b)\mapsto (d+a+r,u+b+r)$. But this is the function associated
  to the time $(d+r,u+r)\in\Time(\ell)$. To prove $\ell\Vdash\phi_2$, choose $t_1=(d_1,u_1)$ and
  $t_2=(d_1,u_1)$ in $\Time(\ell)$. Their pointwise difference $f\coloneqq t_1-t_2$ is the map
  $f(a,b)=(d_1-d_2,u_1-u_2)$, which is independent of $a$ and $b$ and hence a constant section of
  $\church{\tII}(\ell)$. And since $u_1-d_1=\ell=u_2-d_2$, we have $u_1-u_2=d_1-d_2$, so $f$ can be
  identified with an element $r\in\church{\tRR}(\ell)$, again by \cref{cor:Dedekind_cut_semantics}.
\end{proof}

\begin{proposition}\label{prop:axiom5}\index{axiom!of $\BaseTopos$!coprime}
  \Cref{ax:coprime_u}, denoted $\phi$ below, is sound in $\BaseTopos$:
  \[
    \phi\coloneqq\forall(t:\Time)(P,Q:\Prop)(q:\tQQ)\ldotp(t<q \imp (P\vee Q)) \imp [(t<q \imp P)
    \vee (t<q \imp Q)].
  \]
  \Cref{ax:coprime_d}, which is similar, is also sound in $\BaseTopos$.
\end{proposition}
\begin{proof}
  The proofs of these two statements are similar so we prove the first, that $\ell\Vdash\phi$.
  Choose $t=(d,u)\in\Time(\ell)$, $q\in\QQ$, and $P,Q\in\Prop(\ell)$, and assume the hypothesis, $\ell\Vdash t<q \imp (P\vee Q)$. This means that for all
  $\subint{r}{s}\colon\ell'\to\ell$, if $u-s\leq q$ then either $\ell'\Vdash\restrict{P}{\subint{r}{s}}$ or
  $\ell'\Vdash\restrict{Q}{\subint{r}{s}}$.
  
  If $q\leq d$ then $\ell\Vdash(t<q)\iff\bot$ and the conclusion is vacuously true. If $u\leq q$ then $\ell\Vdash t<q$ so $\ell\Vdash P\vee Q$, and again the conclusion follows easily. So suppose $d<q<u$. We want to show $\ell\Vdash t<q\imp P$ or $\ell\Vdash t<q\imp Q$.

  Arguing classically, suppose that $\ell\Vdash t<q\imp P$ does not hold. Then there is some $\subint{r}{s}\colon\ell_0\to\ell$ for which $\ell_0\Vdash u-s\leq q$ but $\ell_0\not\Vdash\restrict{P}{\subint{r}{s}}$. Let $\ell'\coloneqq\ell-s$; by monotonicity, we also have $\ell'\not\Vdash\restrict{P}{\subint{0}{s}}$. Thus by hypothesis we have $\ell'\Vdash\restrict{Q}{\subint{0}{s}}$. It follows that $\ell\Vdash t<q\imp Q$, as
  desired.
\end{proof}

\begin{proposition}\label{prop:axiom6}\index{axiom!of $\BaseTopos$!infinitary coprime}
  Let $\tConst$ be the constant sheaf on some set $C$. Then \cref{ax:infinitary_coprime_u}, denoted
  $\phi$ below, is sound in $\BaseTopos$.
  \begin{multline*}
    \phi\coloneqq\exists(c:\tConst)\ldotp\top\imp\forall(t:\Time)(u:\tRRub)(P:C\to\Prop)\ldotp\\
    (t<u \imp \exists(c:\tConst)\ldotp P(c)) \imp[\forall(u':\tQQ)\ldotp
        u'<u\imp\exists(c:\tConst)\ldotp t<u'\imp P(c)]
  \end{multline*}
  \Cref{ax:infinitary_coprime_d}, which is similar, is also sound in $\BaseTopos$.
\end{proposition}
\begin{proof}
  The proofs of these two statements are similar so we prove the first, that $\ell\Vdash\phi$.  Suppose $C$ is nonempty, choose $t=(d_t,u_t)\in\Time(\ell)$, $P\in\Prop^\tConst(\ell)$, and  $u\in\RR\cup\{\infty,-\infty\}$. If $u=\infty$ or $u=-\infty$, the claim is vacuously true, so  assume $u\in\RR$. Assume the hypothesis, $\ell\Vdash t<u \imp \exists(c:\tConst)\ldotp P(c)$. If $u_t\leq u$ then $\ell\Vdash t<u$ by \cref{eqn:semantics_time} and the hypothesis directly implies the conclusion. If $u\leq d_t$ then for any $u'$, any choice of $c\in C$ will vacuously satisfy the conclusion. Thus we may assume $d_t<u<u_t$; we next unpack the hypothesis in this case.
  
Let $r_0\coloneqq0$ and $s_0\coloneqq u_t-u$; then on the restriction $\subint{r_0}{s_0}\colon\ell_0\to\ell$, we have $\ell_0 \Vdash (\restrict{t}{\subint{r_0}{s_0}}<u)$, so we also have $\ell_0 \Vdash \exists(c:\tConst)\ldotp \restrict{P}{\subint{r_0}{s_0}}(c)$. Thus for every $\subint{r_1}{s_1}\colon\ell_1\wavyto\ell_0$ there exists a $c\in C$ with $\ell_1\Vdash\restrict{P}{\subint{r_1}{s_1}}(c)$.

For the conclusion, take any $\subint{r'}{s'}\colon\ell'\to\ell$ and $u'<u$; we need to show $\ell'\Vdash\exists(c:\tConst)\ldotp(\restrict{t}{\subint{r'}{s'}}<u')\imp\restrict{P}{\subint{r'}{s'}}(c)$. For this it suffices to take an arbitrary $\subint{r''}{s''}\colon\ell''\wavyto\ell'$ and show that there exists $c\in C$ such that $\ell''\Vdash (\restrict{t}{\subint{r''}{s''}}<u')\imp \restrict{P}{\subint{r''}{s''}}(c)$. Let $r_1\coloneqq r'+r''$ and let $s_1\coloneqq u_t-u'$; then $r_1>r_0$ and $s_1>s_0$, so by the paragraph above, on the subinterval $\subint{r_1}{s_1}\colon\ell_1\wavyto\ell_0$, there exists $c\in C$ such that $\ell_1\Vdash\restrict{P}{\subint{r_1}{s_1}}(c)$. This is our candidate $c$.

To show $\ell''\Vdash (\restrict{t}{\subint{r''}{s''}}<u')\imp \restrict{P}{\subint{r''}{s''}}(c)$, take any $\subint{r'''}{s'''}\colon\ell'''\to\ell''$ and assume the hypothesis, $u_t-s'-s''-s'''<u'=u_t-s_1$. Then $\ell'''$ is a subinterval of $\ell_1$ and we are done by monotonicity: $\ell_1\Vdash\restrict{P}{\subint{r_1}{s_1}}(c)$ implies $\ell'''\Vdash\restrict{P}{\subint{r'''}{s'''}}(c)$.
\end{proof}

\begin{proposition}\label{prop:axiom7}\index{axiom!of $\BaseTopos$!prime}
  \Cref{ax:prime}, denoted $\phi$ below, is sound in $\BaseTopos$:
  \begin{multline*}
    \phi\coloneqq\forall(t:\Time)(P,Q:\Prop)(d,u:\tQQ)\ldotp(d\leq u)\imp\\
    \left[ (P\wedge Q) \imp t\apart[u,d] \right] \imp\left[ (P \imp t\apart[u,d]) \vee (Q \imp
    t\apart[u,d]) \right]
  \end{multline*}
\end{proposition}
\begin{proof}
  To show $\ell\Vdash\phi$, choose $t=(d_t,u_t)\in\Time(\ell)$, $P,Q\in\Prop(\ell)$, and
  $d,u\in\QQ$. Assume the hypothesis, that $d\leq u$ and that $\ell\Vdash(P\wedge Q) \imp
  t\apart[u,d]$. This means that for all $\subint{r}{s}\colon\ell'\to\ell$, if
  $\ell'\Vdash\restrict{P}{\subint{r}{s}}$ and $\ell'\Vdash\restrict{Q}{\subint{r}{s}}$ then either $d\leq d_t+r$ or
  $u_t-s\leq u$.

  Suppose $\ell\not\Vdash P \imp t\apart[u,d]$. Then for some $\subint{r'}{s'}\colon\ell'\to\ell$ we have
  $\ell'\Vdash\restrict{P}{\subint{r'}{s'}}$ and $d_t+r'<d$ and $u<u_t-s'$. We want to show $\ell\Vdash Q \imp
  t\apart[u,d]$, so suppose that for some $\subint{r''}{s''}\colon\ell''\to\ell$ we have
  $\ell''\Vdash\restrict{Q}{\subint{r''}{s''}}$ and for contradiction suppose that $d_t+r''<d$ and
  $u<u_t-s''$.

  Let $r_0\coloneqq\max(r',r'')$ and $s_0\coloneqq\max(s',s'')$, and let $\ell_0\coloneqq \ell-(r+s)$. Then $d_t+r_0<d\leq u<u_t-s_0$ and $\ell=u_t-d_t$, so $0<\ell_0$. We get a contradiction because $\ell_0\Vdash\restrict{P}{\subint{r_0}{s_0}}$ and $\ell_0\Vdash\restrict{Q}{\subint{r_0}{s_0}}$, which by hypothesis gives $d\leq d_t+r_0$ or $u_t-s_0\leq u$.
\end{proof}

\begin{proposition}\label{prop:axiom8}\index{axiom!of $\BaseTopos$!enough points}
  \Cref{ax:enough_points}, denoted $\phi$ below, is sound in $\BaseTopos$:
  \[
    \phi\coloneqq\forall(t:\Time)(P:\Prop)\ldotp
      [\forall(d,u:\tQQ)\ldotp d<u \imp \At{[d,u]}P] \imp P
  \]
\end{proposition}
\begin{proof}
  To show $\ell\Vdash\phi$, choose $t=(d_t,u_t)\in\Time(\ell)$ and $P\in\Prop(\ell)$, and assume the
  hypothesis. That is, for every $\subint{r}{s}\colon\ell'\to\ell$ and $d,u\in\QQ$ with $d<u$, we
  have $\ell'\Vdash\At[\restrictsm{t}{\subint{r}{s}}]{[d,u]}\restrict{P}{\subint{r}{s}}$. It follows
  by \cref{cor:semantics_In_At_See} (3), that $\ell'\Vdash\restrict{P}{\subint{r}{s}}$ for all
  $r,s>0$, i.e. for all $\subint{r}{s}\colon\ell'\wavyto\ell$. By local character, this implies
  $\ell\Vdash P$.
\end{proof}

\begin{proposition}\label{prop:axiom9}\index{axiom!of $\BaseTopos$!point to open}
  \Cref{ax:point_to_open}, denoted $\phi$ below, is sound in $\BaseTopos$.
  \begin{multline*}
    \forall(t:\Time)(P:\Prop)(d,u:\tRR)\ldotp d\leq u \imp\\
    \At{[d,u]}P \imp \See{[d,u]}\exists(d',u':\tQQ)\ldotp d'<d\leq u<u' \wedge \In{[d',u']}P
  \end{multline*}
\end{proposition}
\begin{proof}
  To show $\ell\Vdash\phi$, choose $t=(d_t,u_t)\in\Time(\ell)$, $P\in\Prop(\ell)$, and $d\leq u$,
  and suppose $\ell\Vdash\At{[d,u]}P$ holds. The result follows from the three parts of
  \cref{cor:semantics_In_At_See}. Indeed, by (2), the conclusion is vacuously true unless $d_t<d\leq
  u<u_t$, so we may assume it. But then by (3) there exists $\subint{r}{s}\colon\ell'\to\ell$ with
  $d_t+r<d\leq u<u_t-s$ such that $\ell'\Vdash \restrict{P}{\subint{r}{s}}$. Letting $d'\coloneqq d_t+r$ and $u'\coloneqq u_t-s$, the result
  follows from (1).
\end{proof}

\begin{proposition}\label{prop:axiom10}\index{axiom!of $\BaseTopos$!covering for $\pi$ modality}
  \Cref{ax:covering_pi}, denoted $\phi$ below, is sound in $\BaseTopos$.
  \[
    \phi\coloneqq\forall(P:\Prop)\ldotp[\forall(t:\Time)(q_1,q_2:\tQQ)\ldotp q_1<q_2 \imp
    (t\apart[q_2,q_1]\imp P)\imp P]\imp\pi P\imp P.
  \]
\end{proposition}
\begin{proof}
  To show $\ell\Vdash\phi$, choose $P:\Prop(\ell)$ and suppose that $\ell\Vdash\pi P$ and the
  ``covering hypothesis'', that $\ell\Vdash\forall(t:\Time)(q_1,q_2:\tQQ)\ldotp q_1<q_2 \imp
  (t\apart[q_2,q_1]\imp P)\imp P$. We use local character to show $\ell\Vdash P$. That is, suppose
  given an arbitrary $\subint{r'}{s'}\colon\ell'\wavyto\ell$ with $r',s'>0$; we will show
  $\ell'\Vdash\restrict{P}{\subint{r'}{s'}}$ using a compactness argument.

	\cref{prop:semantics_pi} says that for every $0<a<\ell$ there exists an open neighborhood of $a$ in $\RR$ on which $P$ holds; i.e.\ there exists $\subint{r_a}{s_a}\colon\ell_a\to\ell$ with $r_a<a<\ell-s_a$ and $\ell_a\Vdash\restrict{P}{\subint{r_a}{s_a}}$. The set of open intervals $\{(r_a,\ell-s_a)\ss\RR\mid 0\leq a\leq \ell\}$ covers the closed interval $[0,\ell]\ss\RR$. Since this interval is compact, we can choose a finite subcover, say $\subint{r_i}{s_i}\colon\ell_i\to\ell$ with $\ell_i\Vdash\restrict{P}{\subint{r_i}{s_i}}$ for $0\leq i\leq n$ and $n\in\NN$. We may assume that none is strictly contained in any other and thus order them by left endpoint. Then we have $r_0<r'<s'<s_n$ and $r_i<r_{i+1}<s_i<s_{i+1}$ for each $0\leq i\leq n-1$.

  For each $0\leq i\leq n$, consider the restriction $\subint{r_0}{s_i}\colon\ell'_i\to\ell$. We
  will prove by induction on $i$ that $\ell'_i\Vdash\restrict{P}{\subint{r_0}{s_i}}$. This is enough
  because if $\ell'_n\Vdash \restrict{P}{\subint{r_0}{s_n}}$ then
  $\ell'\Vdash\restrict{P}{\subint{r'}{s'}}$.

  We have that $\ell_i\Vdash\restrict{P}{\subint{r_i}{s_i}}$ holds for each $0\leq i\leq n$, so in particular the base case holds because $\ell'_0=\ell_0$. Assume the result is true for some $0\leq i<n$, i.e.\ that
  $\ell'_{i}\Vdash\restrict{P}{\subint{r_0}{s_i}}$; we want to show $\ell'_{i+1}\Vdash\restrict{P}{\subint{r_0}{s_{i+1}}}$. Since $r_{i+1}<s_i$, we may choose
  rationals $q_1,q_2:\QQ$ with $r_{i+1}<q_1<q_2<s_i$. By the covering hypothesis with $t=(r_0,s_{i+1})\in\Time(\ell'_{i+1})$, it suffices to
  show $\ell'_{i+1} \Vdash (t\apart[q_2,q_1]) \imp
  \restrict{P}{\subint{r_0}{s_{i+1}}}$.

  So choose any $\subint{r''}{s''}\colon\ell''\to\ell'_{i+1}$ and assume
  $\ell''\Vdash\restrict{t}{\subint{r''}{s''}}\apart[q_2,q_1]$, i.e.\ either
  $s_{i+1}-s''<q_2$ or $q_1<r_0+r''$; we want to show
  $\ell''\Vdash\restrict{P}{\subint{r_0+r''}{s_{i+1}+s''}}$. Since $q_2<s_i$ and $r_{i+1}<q_1$, we have in the first case that $\ell''$ is contained in $\ell'_i$ and in the second case that $\ell''$ is contained in $\ell_{i+1}$. Either way, the result follows from monotonicity.
\end{proof}

\index{base topos $\BaseTopos$|)}

\chapter{Local numeric types and derivatives}\label{sec:Real_numbers}\index{numeric type}\index{numeric type|seealso {type}}\index{type!numeric|seealso {numeric type}}

In \cref{sec:Dedekind_j} we introduced ten Dedekind $j$-numeric types for an arbitrary modality $j$. Namely, we have the $j$-local lower and upper reals, improper and proper intervals, and real numbers, as well as their unbounded counterparts. They are denoted
\[\tLR_j,\quad\tUR_j,\quad\tII_j,\quad\tIR_j,\quad\tRR_j,\quad\tLRub_j,\quad\tURub_j,\quad\tIIub_j,\quad\tIRub_j,\quad\tRRub_j.\]
However, our axiomatics (\cref{sec:axiomatics}) were given using only three of these: $\tII$, $\tRRub$, and $\tRR$. Not only did we only use three of ten, we also did not make use of the extra generality afforded by the modality $j$, i.e.\ using only $j=\id$. In this section, we explore the more general $j$-numeric types, both type-theoretically and semantically in our topos $\BaseTopos$. 

Recall from \cref{sec:modalitites_subtoposes} that for any topos $\cat{E}$, there is a one-to-one
correspondence between subtoposes of $\cat{E}$ and modalities $j:\Prop\to\Prop$. In our setting, the modalities
\[\See[t]{[d,u]},\quad\At[t]{[d,u]},\quad\pi,\]
defined in \cref{not:our_modalities} thus correspond to particular subtoposes of $\BaseTopos/\Time$ and $\BaseTopos$. The Dedekind $j$-numeric types thus correspond to numeric objects in those subtoposes.

For example, the type $\tRR_\pi$ is semantically a sheaf on the translation-invariant interval site $\BaseSite$ (see \cref{def:BaseSite}), but the ``pointwise'' modality $\pi$ in some sense eliminates all of the nonzero-length points in $\IR$. Thus the sections of $\tRR_\pi$ can be
identified with continuous real-valued functions on the ordinary real line $\RR\ss\IR$. While perhaps it would be most consistent to call $\tRR_\pi$ the type of ``pointwise reals'', we generally refer to it as the type of \emph{variable reals}. This is meant to evoke the idea that the sections of $\tRR_\pi$ are real numbers that vary in time, as opposed to the sections of $\tRR$ which are real numbers that remain constant.\index{variable reals|see {real number,
variable}}\index{{real number!variable}}

In \cref{sec:Real_numbers_modalities} we compare Dedekind $j$-numeric types for the modalities discussed above: $\See{}$, $\At{}$, and $\pi$. In \cref{sec:axiomatics} we also introduced and used the modality $\In{}$, but we will not discuss the Dedekind $\In{}$-numeric objects, except in \cref{rem:why_not_In} where we explain our reasons for leaving it out. In \cref{sec:semantics_numerics_modalities} we discuss the semantics of the $\See{}$-, $\At{}$-, and $\pi$-numeric types; it is a short section because most of the work has already been done.

In \cref{sec:derivatives} we define an internal notion of derivative and differentiability for variable reals $x:\tRR_\pi$. For example, if $x$ is differentiable, then its derivative $\ddt(x):\tRR_\pi$ is also a variable real. We prove internally that differentiation is linear, satisfies $\ddt(t)=1$ for $t:\Time$, and satisfies the Leibniz rule
\[\ddt(x*y)=x*\ddt(y)+\ddt(x)*y\]
for differentiable $x,y:\tRR_\pi$. Finally, we prove that
the semantics of the $\ddt$ operation really is that of the Newtonian derivative.

\section{Relationships between various Dedekind $j$-numeric types}\label{sec:Real_numbers_modalities}\index{numeric type!$j$-local!comparisons of}

All proofs in \cref{sec:Real_numbers_modalities} take place in the temporal type theory
given in \cref{sec:logical_prelims,sec:axiomatics}.\index{logic!constructive}

Let $j$ be a modality. Recall from \cref{def:local_reals} that every $j$-numeric type consists of either one or two predicates $\tQQ\to\Prop_{j}$, where $\Prop_j$ denotes the type of $j$-closed predicates. For example, a lower real is a predicate $\delta:\tQQ\to\Prop$ satisfying $\forall(q:\tQQ)\ldotp j\delta q\imp q$---i.e.\ $j$-closure in the sense of \cref{def:j_closed}---plus two other conditions.

There are many relationships between Dedekind $j$-numeric types for varying $j$, and some make sense in any topos. For example, let $R_j$ denote any of the ten Dedekind $j$-numeric types from \cref{def:local_reals}. \Cref{prop:j_on_intervals} says that a map of modalities $j'\imp j$ induces a map
\begin{equation}\label{eqn:j_map}
j:\tLR_{j'}\to\tLR_j
\end{equation}
which preserves arithmetic, inequalities, and constants. For example, in the case $R=\tLR$, the above map sends $\delta:\tQQ\to\Prop_{j'}$ to the predicate $j\delta:\tQQ\to\Prop_j$. In this section, we focus our attention on just a few such modalities $j$, namely $\See{}$, $\At{}$, and $\pi$.

\begin{remark}\label{rem:why_not_In}
We will not discuss the $\In{}$-numeric types, e.g. $\tII_{\InSymbol[d,u]}$, anywhere in the book. We omit them for a couple of reasons. First, none of them has seemed to come up in practice, at least to this point in the authors' investigations. Second, our work in this section and the next is to understand how the $j$-numeric types embed---both internally and externally---into their undecorated ($j=\id$) counterparts. Our proof technique for this, developed in \cref{sec:predomains_subtopos}, works only in certain cases. These cases happen to include the $\See{}$, $\At{}$, and $\pi$ modalities but not the $\In{}$-modality.
\end{remark}

\begin{notation}
Again let $R_j$ denote any of the ten Dedekind $j$-numeric types from \cref{def:local_reals}. When $j=\At[t]{[r,r]}$ for some real number $r:\tRR$ and $t:\Time$, and when $j'$ is any modality with $j'\imp j$, we have a special notation for
the map $j\colon R_{j'}\to R_j$ from \cref{prop:j_on_intervals}.

To make it explicit, suppose $R_j$ is a two-sided numeric type and suppose $j\imp\At{r}$, e.g.\ $j=\id$, $j=\See{r}$, or $j=\pi$. Then for 
$x=(\delta,\upsilon):R$, we write $x^{@t}(r)$ to denote $\At[t]{[r,r]}(x):R_j$; it is the pair of
cuts $(\At[t]{r}\delta,\At[t]{r}\upsilon)$. If $t$ is clear from context we can drop it as usual:
\begin{equation}\label{eqn:x^@}
  x^@(r)\coloneqq (\At{r}\delta,\;\At{r}\upsilon).
\end{equation}
We read $x^{@t}(r)$ as ``the value of $x$ at time $t=r$'', or just ``$x$ at
$r$''. Note that this map $x\mapsto x^@(r)$ preserves inequalities, arithmetic, and constants by \cref{prop:j_on_intervals}.

Of course the same idea works for the one-sided numeric types.
\end{notation}

\begin{example}\label{ex:t@r}\index{real number!variable}
  Considering $t:\Time$ as a variable real by \cref{ex:Time_subtype_rvar}, then for any $r:\tRR$ we
  have
  \[
    t^@(r)=\See{r}(r)
  \]
  or explicitly $\At[t]{r}(q<t)\iff\See[t]{r}(q<r)$ and $\At[t]{r}(t<q)\iff\See[t]{r}(r<q)$ for all
  $q:\tQQ$. Indeed, for the first equivalence---the second being similar---the forward direction
  follows by unraveling the definitions (\ref{not:our_modalities}) and by trichotomy $(q<r)\vee
  (q=r)\vee(r<q)$ (\cref{prop:trichotomy}), and the reverse direction follows from the fact that
  $(q<r)\imp\At[t]{r}(q<t)$, which is \cref{lem:time_lems}i.
\end{example}

\begin{remark}
  Strictly speaking $\tRR_{\AtInline{[d,u]}}$ and later $\tRR_{\SeeInline{[d,u]}}$ are
  \emph{dependent types}: they are dependent on $t:\Time$ and $d,u:\tRR$. Each is the subtype of a
  non-dependent type, $(\tQQ\to\Prop)\times(\tQQ\to\Prop)$, cut out by propositions that depend on
  $t$, $d$, and $u$. See \cref{sec:dependent_types}.
\end{remark}

There are other relationships between the various Dedekind $j$-numeric types, beyond
the one shown in \cref{eqn:j_map}. These relationships arise because several of the $j$-local Dedekind
axioms do not actually depend on $j$. The results from the remainder of this section are all to this effect; they are summarized in \cref{sec:section_summary}. Readers who feel the urge to do so may skip ahead to that summary.

\begin{proposition}\label{prop:D_id_pi_See_At}\index{cuts!disjoint}\index{cuts!located}
  Let $\phi:(\tQQ\to\Prop)\times(\tQQ\to\Prop)$. Then:
  \begin{enumerate}
    \item $\phi$ is disjoint iff it is $\pi$-disjoint.
    \item $\phi$ is $\See{[d,u]}$-disjoint iff it is $\At{[d,u]}$-disjoint.
    \item If $\phi$ is $\See{[d,u]}$-closed, then it is $\See{[d,u]}$-located iff it is located.
    \item If $\phi$ is $\At{[d,u]}$-closed, then it is $\At{[d,u]}$-located iff it is located.
  \end{enumerate}
\end{proposition}
\begin{proof}
  Statements 1 and 2 follow directly from the facts $\pi\bot\iff\bot$ and
  $\See{[d,u]}\bot\iff\At{[d,u]}\bot$, which are shown in \cref{cor:pi_negneg} and
  \cref{lemma:See_At_In_False}. Statements 3 and 4 follow directly from the fact that $\See{[d,u]}$
  and $\At{[d,u]}$ preserve disjunction, \cref{lem:At_preserves_or}.
\end{proof}

\begin{remark}\index{cuts!$\pi$-located}
  Note that $\Pointwise$-located does not imply located. Semantically, this is obvious: sections of
  $\tRR_\Pointwise$ are arbitrary continuous functions whereas sections of $\tRR$ are constants; see
  \cref{thm:omnibus_semantics,cor:pi_real_semantics}.
\end{remark}

Before proving similar results relating numeric types for varying modalities, we pause to recall the notion of $j$-constant numeric types from \cref{def:constant_dedekind}. Namely if $R_j$ is a $j$-numeric type, then $\const{c}R_j$ denotes the $j$-numeric subtype whose cuts are $j$-decidable. As a sanity check, we have said many times that the type $\tRR$ of real numbers is constant in $\BaseTopos$ (as is $\tRRub$), and we prove in \cref{prop:const_reals_again} that the terminology matches.

\begin{proposition}\label{prop:const_reals_again}\index{internal!bijection}
There are internal bijections $\ctRR\cong\tRR$ and $\ctRRub\cong\tRRub\cong\tRR\sqcup\{-\infty,\infty\}$.
\end{proposition}
\begin{proof}
The first two isomorphisms are \cref{ax:decidable_reals,cor:decidable_reals}; the isomorphism $\tRRub\cong\tRR\sqcup\{-\infty,\infty\}$ is shorthand for the isomorphism $\tRRub\cong\tRR\sqcup\unit\sqcup\unit$ provided in \cref{prop:unbounded_reals_inftys}.
\end{proof}

Note that the
analogous statement does not hold when $\tRR$ is replaced by other Dedekind numeric types. For example, there is no isomorphism $\ctLR\cong^?\tLR$, and there should not be one because $\ctLR$ is
semantically the constant sheaf on $\RR\cup\{\infty\}$, whereas $\tLR$ is semantically the sheaf of
lower-semicontinuous real-valued functions on $\IR$. Similarly, $\tRR_\pi$ is semantically the sheaf
of continuous real-valued functions on $\RR$, and $\ctRR_\pi$ is the proper subsheaf of constant
ones.

\begin{proposition}\label{prop:At_numeric_See}\index{numeric type!$\At{}$ as constant}
  For any $d,u$ there is an inclusion $\tII_{\AtInline{[d,u]}}\ss\tII_{\SeeInline{[d,u]}}$, and we
  can identify its image with the subtype of constants, in the sense that the following diagram commutes
  \[
  \begin{tikzcd}[column sep=2pt]
    \tII_{\AtInline{[d,u]}}\ar[rr, equal]\ar[dr, hook]&&\ctII_{\SeeInline{[d,u]}}\ar[dl, hook]\\
    &\tII_{\SeeInline{[d,u]}}
  \end{tikzcd}
  \]
  The inclusion preserves domain order and arithmetic. The same holds when $\tII$ is replaced by any
  of the other nine Dedekind numeric types.
\end{proposition}
\begin{proof}
  We first prove that there is an inclusion; consider the case of $\delta:\tLR$. Since $\delta$ is
  $\At{[d,u]}$-closed, it is also $\See{[d,u]}$-closed. To show $\delta$ is $\See{[d,u]}$-rounded,
  we have the chain of implications
  \begin{align*}
    \delta q
    &\imp\At{[d,u]}\exists q'\ldotp(q<q')\wedge\delta q'\\
    &\imp\exists q'\ldotp\At{[d,u]}(q<q')\wedge\delta q'\\
    &\imp\exists q'\ldotp\See{[d,u]}(q<q')\wedge\delta q'\imp\See{[d,u]}\exists q'\ldotp(q<q')\wedge\delta q',
  \end{align*}
  where the second implication is \cref{prop:At_preserves_exists} and the third is
  \cref{lem:At_decidable_prop}. One proves similarly that $\At{[d,u]}\exists q\ldotp\delta
  q\imp\See{[d,u]}\exists q\ldotp\delta q$, so $\At{}$-bounded implies $\See{}$-bounded.

  A similar argument works for $\tUR$, $\tII$, $\tLRub$, $\tURub$, and $\tIIub$. For $\tIR$ and
  $\tIRub$ we need that disjointness is preserved, and for $\tRR$ and $\tRRub$ we further need that
  locatedness is preserved; this is shown in \cref{prop:D_id_pi_See_At}.

  To see that $\tLR_{\AtInline{[d,u]}}=\ctLR_{\SeeInline{[d,u]}}$, we simply apply
  \cref{cor:AtClosed_SeeComplemented}: a predicate $\delta:\tQQ\to\Prop$ is $\At{[d,u]}$-closed iff
  it is $\See{[d,u]}$-closed and $\See{[d,u]}$-constant. The same argument works for all Dedekind
  numeric types. Finally, the fact that the inclusion preserves arithmetic follows from
  \cref{prop:decidable_approx_pres_consts}.
\end{proof}

We will see that the inclusion $\tII_{\AtInline{[d,u]}}\ss\tII_{\SeeInline{[d,u]}}$ preserves
inequalities in \cref{prop:At_inclusion_comparisons}

\begin{proposition}\label{prop:See_At_Reals_iso}
  For Dedekind reals, the inclusion $\tRR_{\AtInline{[d,u]}}\ss \tRR_{\SeeInline{[d,u]}}$ from
  \cref{prop:At_numeric_See} is an isomorphism,
  \[
    \tRR_{\AtInline{[d,u]}} \iso \tRR_{\SeeInline{[d,u]}}.
  \]
\end{proposition}
\begin{proof}
  Let $(\delta,\upsilon)\in\tRR_{\SeeInline{[d,u]}}$. We have a retraction
  $\At{[d,u]}:\tRR_{\SeeInline{[d,u]}}\to\tRR_{\AtInline{[d,u]}}$, so it suffices to show that
  $\delta$ and $\upsilon$ are already $\At{[d,u]}$-closed. The argument for $\delta$ and $\upsilon$
  are similar, so we consider $\delta$.

  It suffices to show that for any $q:\tQQ$, we have $\At{[d,u]}(\delta q) \imp \See{[d,u]}(\delta
  q)$. Using $\delta q \imp \See{[d,u]}(\exists q'\ldotp (q<q')\wedge\delta q')$, we get
  $\At{[d,u]}(\exists q'\ldotp(q<q')\wedge\delta q')$. Then using \cref{prop:At_preserves_exists} we
  get $\exists q'\ldotp\At{[d,u]}(q<q')\wedge\At{[d,u]}(\delta q')$. Take such a $q'$.

  By locatedness, we have $\See{[d,u]}(\delta q \vee \upsilon q')$, i.e.\ $t\apart[u,d]\vee\delta
  q\vee\upsilon q'$. The first two cases trivially imply $\See{[d,u]}(\delta q)$. In the last case,
  suppose $\upsilon q'$. We also have $\At{[d,u]}(\delta q')$, which by disjointness gives
  $\At{[d,u]}\bot$, which is equivalent to $t\apart[u,d]$ and hence implies $\See{[d,u]}(\delta q)$.
\end{proof}

\begin{proposition}\label{prop:closed_inclusion_arith}\index{modality!closed}
  Suppose $j$ is a closed modality (e.g.\ $j=\See{[d,u]}$). Then there is an inclusion
  $\tII_j\ss\tII$, which is a section of $j:\tII\to\tII_j$, and which preserves domain order,
  constants, arithmetic, and inequalities. An analogous statement holds when $\tII$ is replaced with
  $\tLR$, $\tUR$, $\tIIub$, $\tLRub$, or $\tURub$.
\end{proposition}
\begin{proof}
  The existence of the inclusion (which obviously preserves domain order) follows directly from
  \cref{prop:RId_inclusion_closed_modality}, and it is easy to check it is a section of $j$. The
  preservation of constants and arithmetic follows from \cref{prop:j1j2j3j'_decidable_const} where
  $j=\id$, and $j'_1=j'_2=j'_3=j$.

  To prove that the inclusion preserves $<$, recall the open $U_<:\Omega(\tIIpre\times\tIIpre)$
  defined by $U_<(d,u,d',u')\coloneqq(u<d')$ from \cref{prop:<_open}, where we showed that for any
  $(x,y)\in\tII_j\times\tII_j\cong\RId_j(\tIIpre\times\tIIpre)$, we have $(x,y)\models_j U_<$ iff
  $x<_jy$ by \cref{prop:<_open}. Now by \cref{prop:closed_j_continuous}, the inclusion
  $\tII_j\ss\tII$ is continuous, meaning $(I\models_j jU_<)\iff(I\models U_<)$. Thus
  $(x<_jy)\iff(x<y)$.
\end{proof}

\begin{remark}[Mixed arithmetic]\label{rem:enlarging_closed}\index{arithmetic!mixed modality}
  In fact, the proof of \cref{prop:closed_inclusion_arith} implies more. Let
  $\phi_1,\phi_2,\phi_3:\Prop$ be such that $(\phi_1\vee\phi_2)\imp\phi_3$, let $j'_1$, $j'_2$, and
  $j'_3$ be the corresponding closed modalities. Let $j$ be any modality such that
  $jP\imp(j'_1P\wedge j'_2P\wedge j'_3P)$ for all $P$. Then there is a sort of mixed arithmetic,
  preserving constants: for any arithmetic operation $f$ (including $+$, $\max$, $-$, $*$, etc.), we
  have the left-hand diagram
  \[
  \begin{tikzcd}[column sep=large]
    \ctII_{j'_1}\times\ctII_{j'_2}\ar[r]\ar[d, hook]&\ctII_{j'_3}\ar[d, hook]\\
    \tII_{j'_1}\times\tII_{j'_2}\ar[r,"F_{j'_1,j'_2,j'_3}"]\ar[d, hook]&\tII_{j'_3}\ar[d, hook]\\
    \tII_{j}\times\tII_{j}\ar[r,"F_{j}"']&\tII_{j}
  \end{tikzcd}
  \hspace{1in}
  \begin{tikzcd}[column sep=large]
    \tII_{\AtInline{[d_1,u_1]}}\times\ctII_{\AtInline{[d_2,u_2]}}\ar[r]\ar[d, hook]&\tII_{\AtInline{[d_3,u_3]}}\ar[d, hook]\\
    \tII_{\SeeInline{[d_1,u_1]}}\times\tII_{\SeeInline{[d_2,u_2]}}\ar[r]\ar[d, hook]&\tII_{\SeeInline{[d_3,u_3]}}\ar[d, hook]\\
    \tII_{j}\times\tII_{j}\ar[r]&\tII_{j}
  \end{tikzcd}
  \]
  where $F_{j}$ and $F_{j'_1,j'_2,j'_3}$ are as in
  \cref{prop:binary_fn_interval_domains,prop:Fj1j2j3}. Again, this is a direct application of
  \cref{prop:j1j2j3j'_decidable_const}. The right-hand diagram is the special case
  $j'_i=\See{[d_i,u_i]}$, and its commutativity follows from that and \cref{prop:At_numeric_See}.
\end{remark}

The statement of \cref{prop:closed_inclusion_arith} holds when the closed modality $j$ is replaced
by $\pi$, though for different reasons.

\begin{proposition}\label{prop:pi_inclusion_arith}
  There is an inclusion $\tII_\pi\ss\tII$, which preserves domain order and constants and is a
  section of $\pi:\tII\to\tII_\pi$. Thus we can make the identification
  \[
    \tII_{\pi} = \big\{(\delta,\upsilon):\tII \mid
      \text{$\delta$ and $\upsilon$ are $\pi$-closed}\big\}.
  \]
  Analogous statements holds when $\tII$ is replaced with any of the other numeric domains.
\end{proposition}
\begin{proof}
  We first prove there is an inclusion $\tLRub_\pi\ss\tLRub$. Suppose given $\delta\in\tLRub_\pi$
  that is $\pi$-rounded, i.e.\ $\forall(q_1:\tQQ)\ldotp\delta q_1\imp\pi\exists (q_2:\tQQ)\ldotp
  q_1<q_2\wedge\delta q_2$; we need to show it is rounded. Choose $q_1$ with $\delta q_1$, and let
  $\tConst\coloneqq\{q:\tQQ\mid q_1<q\}$; it is a constant type (see \cref{def:constant_types}).%
  \footnote{
    In fact, the type $\tConst$ is dependent on $q_1$; see \cref{sec:dependent_types}.
  }
  Thus we have $\pi\exists(q_2:\tConst)\ldotp\delta q_2$ and we want to remove the $\pi$. Because
  $\delta$ is down-closed and $(\tQQ,<)$ has binary meets, the following condition is satisfied:
  \[
    \forall(q_1,q_2:\tConst)\ldotp \exists(q':\tConst)\ldotp (\phi q_1\imp \phi q') \wedge (\phi
    q_2\imp \phi q')
  \]
  Thus by \cref{lem:exists_Pi_closed} and the fact that $\delta$ is $\pi$-closed, we obtain
  $\exists(q_2:\tConst)\ldotp\delta q_2$. This proves that $\delta$ is rounded and hence gives the
  inclusion $\tLRub_\pi\ss\tLRub$.

  The proof that $\pi$-bounded implies bounded is similar, so we obtain the inclusion
  $\tLR_\pi\ss\tLR$; the analogous inclusions for $\tURub$ and $\tUR$, and those for $\tIIub$ and
  $\tII$ follow directly. Finally, for $\tIRub$ and $\tIR$, we need that $\pi$-disjoint implies
  disjoint, but this is shown in \cref{prop:D_id_pi_See_At}.

  The inclusions obviously preserve domain order. They preserve constants because $\pi\bot\iff\bot$.
\end{proof}

Of course the analogue of \cref{prop:pi_inclusion_arith} should not hold for the other two numeric
objects, $\tRR_\pi\not\ss\tRR$ and $\tRRub_\pi\not\ss\tRRub$, since it does not hold semantically;
see \cref{thm:omnibus_semantics}.

\begin{remark}\label{rem:pi_inclusion_no_arith}
  Note that the inclusion $\tII_\pi\ss\tII$ from \cref{prop:pi_inclusion_arith} preserves \emph{neither arithmetic nor inequalities}. For
  example, recall that $\Time\ss\tRR_\pi\ss\tII_\pi$, so take $t:\Time$ and consider $t-_\pi t$ vs.\
  $t-t$. Since $\tRR_\pi$ is a group, we have $t-_\pi t=0$. However, $\tII$ is not a group, so we do
  not expect $t-t=^?0$ to hold. In fact given $\ell\in S_{\IR}$, the cuts for
  $(\delta,\upsilon)\coloneqq t-t$ are $\church{\delta q}\iff(q<-\ell)$ and $\church{\upsilon
  q}\iff(\ell<q)$, which are decidedly non-zero since $\ell>0$.

  Similarly, consider $t$ and $t+1$, the latter of which equals $t+_\pi1$ by
  \cref{prop:add_mult_constants}. Then it is not the case that $t<t+1$, whereas we can prove $t<_\pi
  t+1$ using \cref{prop:clean_pi,ax:torsor}. Indeed, it suffices to show
  \[
    \forall(r:\tRR)\ldotp\exists(q:\tQQ)\ldotp 0<q \wedge
      [r-q<t<r+q \imp \exists(q':\tQQ)\ldotp t<q'<t+1]
  \]
  and for any $r$ we can take $q\coloneqq\frac{1}{2}$. Then assuming $r-q<t<r+q$, we can find $q'$
  with $t<q'<r+q$ by roundedness, and it follows that $q'<t+1$.

  To see that semantically $t<^?t+1$ does not hold, consider a section of length $\ell\geq 1$. Then
  there is some $d\in\RR$ such that $\church{t}=(d,d+\ell)\in\Time(\ell)$, and $t+1=(d+1,d+1+\ell)$.
  If $\church{\exists(q:\tQQ)\ldotp t<q<t+1}$ then $d+\ell<q<d+1$, which is false.
\end{remark}

Despite \cref{rem:pi_inclusion_no_arith}, we can use the fact that $\pi$ is a proper modality (see
\cref{ex:pi_proper}) and the general results from \cref{ssec:proper_modalities} to establish some
useful special cases in which arithmetic and inequalities \emph{are} preserved.

\begin{proposition}\label{prop:add_mult_constants}\index{numeric type!constant}
  The inclusion $\tII_\pi\ss\tII$ commutes with adding constants and inequalities with constants; in
  other words the diagrams below commute:
  \begin{gather*}
  \begin{tikzcd}[ampersand replacement=\&]
    \ctII_\pi\times\tII_\pi \ar[d, equals] \ar[r, hook] \&[-10pt]
    \tII_\pi\times\tII_\pi \ar[r, "+_\pi"] \&
    \tII_\pi\ar[d, hook] \\
    \ctII\times\tII \ar[r, hook] \&
    \tII\times\tII\ar[r, "+"'] \&
    \tII
  \end{tikzcd}
  \hspace{.6in}
  \begin{tikzcd}[ampersand replacement=\&]
    \ctII_\pi\times\tII_\pi \ar[d, equals] \ar[r, hook] \&[-10pt]
    \tII_\pi\times\tII_\pi \ar[r, "<_\pi"] \&
    \Prop_\pi\ar[d, hook] \\
    \ctII\times\tII \ar[r, hook] \&
    \tII\times\tII \ar[r, "<"'] \&
    \Prop
  \end{tikzcd}
  \end{gather*}
  and similarly when $\ctII_\pi\times\tII_\pi$ is replaced by $\tII_\pi\times\ctII_\pi$.

  This doesn't hold for multiplication by constants in general, but it does on the subtype
  $\tIR_\pi\ss\tIR$; in other words, the following diagram commutes:
  \[
    \begin{tikzcd}[ampersand replacement=\&]
      \ctIR_\pi\times\tIR_\pi \ar[d, equals] \ar[r, hook] \&[-10pt]
      \tIR_\pi\times\tIR_\pi \ar[r, "*_\pi"] \&
      \tIR_\pi\ar[d, hook] \\
      \ctIR\times\tIR \ar[r, hook] \&
      \tIR\times\tIR\ar[r, "*"'] \&
      \tIR
    \end{tikzcd}
  \]
\end{proposition}
\begin{proof}
  To show that the inclusion commutes with adding constants, it suffices by
  \cref{prop:functions_approx_mappings,prop:proper_constant_Lawson} to show that
  $+\colon\tIIpre\times\tIIpre\to\tIIpre$ preserves meets in the second variable (and by in the
  first variable by symmetry). This requires showing that
  \[
    (d,u)+((d_1,u_1)\sqcap(d_2,u_2)) = ((d,u)+(d_1,u_1))\sqcap((d,u)+(d_2,u_2)),
  \]
  which comes down to the fact that $d+\min(d_1,d_2) = \min(d+d_1,d+d_2)$, and likewise for $\max$.

  The statement about inequalities follows similarly from \cref{prop:proper_constant_Lawson_opens}.
  Recall the open $U_<\coloneqq\{(d_1,u_1),(d_2,u_2)\mid u_1<d_2\}\in\Opens(\tIIpre\times\tIIpre)$,
  from \cref{prop:<_open}, where it was shown that $x_1<_jx_2$ iff $(x_1,x_2)\vDash_j U_<$, for
  $x_1,x_2:\tII_j$. Clearly $U_<$ is decidable, and it is filtered (i.e.\ closed under meets) in
  each variable, as witnessed by
  \begin{gather*}
    (u_1<d_2 \wedge u_1'<d_2) \imp \max(u_1,u_1')<d_2 \\
    (u_1<d_2 \wedge u_1<d_2') \imp u_1<\min(d_2,d_2').
  \end{gather*}

  Finally, multiplication $*\colon\tIIpre\times\tIIpre\to\tIIpre$ is not filtered. A counter example
  is $((-1,-1)\sqcap(1,1))*(1,-1) = (-1,1)*(1,-1) = (0,0)$, while
  $((-1,-1)*(1,-1))\sqcap((1,1)*(1,-1)) = (1,-1)\sqcap(1,-1) = (1,-1)$; see
  \cref{eqn:Kaucher_mult_predomain}. However, restricted to $\tIRpre[j]$ multiplication \emph{is}
  filtered. In this case, one can check by cases on the signs of $d_1$ and $d_2$ that the formula
  for multiplication simplifies to $(d_1,u_1)*(d_2,u_2)=(d',u')$, where
  \begin{align*}
    d' &\coloneqq \min(d_1d_2,d_1u_2,u_1d_2,u_1u_2) \\
    u' &\coloneqq \max(d_1d_2,d_1u_2,u_1d_2,u_1u_2).
  \end{align*}
  Then it is a simple exercise, e.g.\ by cases on the signs of $d',u'$, to show that
  \[
    ((d_1,u_1)\sqcap(d_2,u_2))*(d',u') = ((d_1,u_1)*(d',u'))\sqcap((d_2,u_2)*(d',u')).
    \qedhere
  \]
\end{proof}

We need a technical lemma before proving \cref{prop:At_inclusion_comparisons}.

\begin{lemma}\label{lemma:see_at_continuous}
  Suppose that $B$ is a constant predomain (see \cref{def:predom_const_approx_decidable}), and
  $I\in\RId(B)$ and $U\in\Omega(B)$ are such that $\At{[d,u]}I\imp I$ and
  $\At{[d,u]}U\imp\See{[d,u]}U$ . Then
  \[
    (I\models_{\AtInline{[d,u]}}U)\iff(I\models_{\SeeInline{[d,u]}}U)\iff(I\models U).
  \]
\end{lemma}
\begin{proof}
  There is the following chain of equivalences:
  \[
    \At{}\exists b\ldotp Ib\wedge Ub
    \;\iff\;\exists b\ldotp\At{}Ib\wedge\At{}Ub
    \;\iff\;\exists b\ldotp Ib\wedge\See{}Ub
    \;\iff\;\See{}\exists b\ldotp Ib\wedge Ub
    \;\iff\;\exists b\ldotp Ib\wedge Ub,
  \]
  where the first is \cref{prop:At_preserves_exists}, the second is by assumption (and
  \cref{prop:simple_modalities}), the third is \cref{eqn:exists_see_closed}, and the fourth is
  \cref{lemma:j_inclusion_continuous}. By \cref{def:predomain_modality}, the left-hand side is
  $I\models_{\AtInline{[d,u]}}U$, the right-hand side is $I\models U$, and in the fourth position is
  $I\models_{\SeeInline{[d,u]}}U$.
\end{proof}

\begin{proposition}\label{prop:At_inclusion_comparisons}
  The inclusions $\tII_{\AtInline{[d,u]}}\ss\tII_{\SeeInline{[d,u]}}$ and
  $\tII_{\AtInline{[d,u]}}\ss\tII$ preserve inequalities. The same holds for any of the two-sided
  Dedekind numeric types.
\end{proposition}
\begin{proof}
  Define $U_<:\Omega(\tIIpre\times\tIIpre)$ by $U_<(d,u,d',u')\coloneqq(u<d')$; it is open as seen
  in  \cref{prop:closed_inclusion_arith}, and it is decidable, hence $\At{[d,u]}U\imp\See{[d,u]}U$
  by \cref{lem:At_decidable_prop}. Thus by \cref{lemma:see_at_continuous}, for any
  $(x,y)\in(\tII_{\AtInline{[d,u]}})^2\cong\RId_{\AtInline{[d,u]}}(\tIIpre\times\tIIpre)$, we have
  \[
    (x,y) \models_{\AtInline{[d,u]}} U_< \quad\iff\quad
    (x,y) \models_{\SeeInline{[d,u]}} U_< \quad\iff\quad
    (x,y) \models U_<
  \]
  i.e. $(x<_{\AtInline{}}y) \iff (x<_{\SeeInline{}}y) \iff (x<y).$
\end{proof}

\begin{proposition}\label{prop:At_pi_inclusions}
  We have the following inclusions, all of which preserve domain order, constants, arithmetic, and
  inequalities:
  \begin{align*}
    \tLR_{\AtInline{0}}&\ss\tLR_\pi,&
    \tUR_{\AtInline{0}}&\ss\tUR_\pi,&
    \tII_{\AtInline{0}}&\ss\tII_\pi,\\
    \tLRub_{\AtInline{0}}&\ss\tLRub_\pi,&
    \tURub_{\AtInline{0}}&\ss\tURub_\pi,&
    \tIIub_{\AtInline{0}}&\ss\tIIub_\pi.
  \end{align*}
\end{proposition}
\begin{proof}
  We prove the first, the rest being similar. If $\delta:\tQQ\to\Prop$ is $\At{0}$-closed, it is
  $\pi$-closed. The result then follows from
  \cref{prop:pi_inclusion_arith,prop:At_inclusion_comparisons}.
\end{proof}

\subsection{Section summary}\label{sec:section_summary}

We now attempt to summarize the results of this section for future reference. The $j$-numeric types
were defined for a generic modality in \cref{def:local_reals}, and summarized in
\cref{table:numeric_types}. For ease of reference, we repeat the list of axioms:
\begin{tabbing}
  \qquad\=
    0a$_j$.\quad\= $\forall(q:\tQQ)\ldotp j(\delta q)\imp\delta q$\hspace{1.2in}\=
    0b$_j$.\quad\= $\forall(q:\tQQ)\ldotp j(\upsilon q)\imp\upsilon q$
  \\
    \>1a$_j$.\> $\forall q_1\,q_2\ldotp (q_1<q_2)\imp\delta q_2\imp\delta q_1$
    \>1b$_j$.\> $\forall q_1\,q_2\ldotp(q_1<q_2)\imp\upsilon q_1\imp\upsilon q_2$
  \\
    \>2a$_j$.\> $\forall q_1\ldotp\delta q_1\imp j\exists q_2\ldotp (q_1<q_2)\wedge\delta q_2$
    \>2b$_j$.\> $\forall q_2\ldotp\upsilon q_2\imp j\exists q_1\ldotp (q_1<q_2)\wedge\upsilon q_1$
  \\
    \>3a$_j$.\> $j\exists q\ldotp\delta q$
    \>3b$_j$.\> $j\exists q\ldotp\upsilon q$
  \\
    \>4$_j$.\> $\forall q\ldotp(\delta q \wedge \upsilon q)\imp j\bot$ \\
    \>5$_j$.\> $\forall q_1\,q_2\ldotp (q_1<q_2) \imp j(\delta q_1\vee\upsilon q_2)$
\end{tabbing}
where we now annotate each with the modality by which it is modified. As usual, un-annotated axioms, such as 1a or
4, refer to the identity modality $j=\id$.

\Cref{table:local_numeric_types} updates \cref{table:numeric_types} with the results of this
section. The characterizations given follow from
\cref{prop:D_id_pi_See_At,prop:At_numeric_See,prop:closed_inclusion_arith,prop:pi_inclusion_arith}
(recall in particular from \cref{prop:D_id_pi_See_At} that 4$_\SeeSymbol$ and 4$_\AtSymbol$ are
equivalent). The unbounded versions simply remove axiom 3; for example, $\tIRub_\AtSymbol =
\{(\delta,\upsilon) \mid \text{0$_\AtSymbol$, 1, 2, 4$_\SeeSymbol$}\}$. This table clearly shows all
subtype relationships. For instance, $\tIR_\pi \subseteq \tIR$, and in fact $\tIR_\pi = \tIR \cap
\tII_\pi$.

\begin{table}
\index{type!of intervals}\index{type!of lower real numbers}\index{type!of improper intervals}
\index{type!of upper real numbers}\index{type!of real numbers}\index{type!of unbounded intervals}
\index{type!of unbounded lower real numbers}\index{type!of unbounded improper intervals}
\index{type!of unbounded upper real numbers}\index{type!of unbounded real numbers}
  \[
  \renewcommand{\arraystretch}{1.45}
  \begin{array}{l|l|l}
    \text{Name of type in }\cat{E} & \text{Notation} & \text{Characterization} \\\hline
    \text{lower reals} & \tLR & \{\delta \mid \text{1a, 2a, 3a}\} \\
    \text{upper reals} & \tUR & \{\upsilon \mid \text{1b, 2b, 3b}\} \\
    \text{improper intervals} & \tII & \{(\delta,\upsilon) \mid \text{1, 2, 3}\} \\
    \text{(proper) intervals} & \tIR & \{(\delta,\upsilon) \mid \text{1, 2, 3, 4}\} \\
    \text{(constant) real numbers} & \tRR & \{(\delta,\upsilon) \mid \text{1, 2, 3, 4, 5}\} \\
    \text{$\pi$-local lower reals} & \tLR_\pi & \{\delta \mid \text{0a$_\pi$, 1a, 2a, 3a}\} \\
    \text{$\pi$-local upper reals} & \tUR_\pi & \{\upsilon \mid \text{0b$_\pi$, 1b, 2b, 3b}\} \\
    \text{$\pi$-local improper intervals} & \tII_\pi &
      \{(\delta,\upsilon) \mid \text{0$_\pi$, 1, 2, 3}\} \\
    \text{$\pi$-local (proper) intervals} & \tIR_\pi &
      \{(\delta,\upsilon) \mid \text{0$_\pi$, 1, 2, 3, 4}\} \\
    \text{$\pi$-local real numbers} & \tRR_\pi &
      \{(\delta,\upsilon) \mid \text{0$_\pi$, 1, 2, 3, 4, 5$_\pi$} \} \\
    \text{$\SeeSymbol$-local lower reals} & \tLR_\SeeSymbol &
      \{\delta \mid \text{0a$_\SeeSymbol$, 1a, 2a, 3a}\} \\
    \text{$\SeeSymbol$-local upper reals} & \tUR_\SeeSymbol &
      \{\upsilon \mid \text{0b$_\SeeSymbol$, 1b, 2b, 3b}\} \\
    \text{$\SeeSymbol$-local improper intervals} & \tII_\SeeSymbol &
      \{(\delta,\upsilon) \mid \text{0$_\SeeSymbol$, 1, 2, 3}\} \\
    \text{$\SeeSymbol$-local (proper) intervals} & \tIR_\SeeSymbol &
      \{(\delta,\upsilon) \mid \text{0$_\SeeSymbol$, 1, 2, 3, 4$_\SeeSymbol$}\} \\
    \text{$\SeeSymbol$-local real numbers} & \tRR_\SeeSymbol &
      \{(\delta,\upsilon) \mid \text{0$_\SeeSymbol$, 1, 2, 3, 4$_\SeeSymbol$, 5}\} \\
    \text{$\AtSymbol$-local lower reals} & \tLR_\AtSymbol &
      \{\delta \mid \text{0a$_\AtSymbol$, 1a, 2a, 3a}\} \\
    \text{$\AtSymbol$-local upper reals} & \tUR_\AtSymbol &
      \{\upsilon \mid \text{0b$_\AtSymbol$, 1b, 2b, 3b}\} \\
    \text{$\AtSymbol$-local improper intervals} & \tII_\AtSymbol &
      \{(\delta,\upsilon) \mid \text{0$_\AtSymbol$, 1, 2, 3}\} \\
    \text{$\AtSymbol$-local (proper) intervals} & \tIR_\AtSymbol &
      \{(\delta,\upsilon) \mid \text{0$_\AtSymbol$, 1, 2, 3, 4$_\SeeSymbol$}\} \\
    \text{$\AtSymbol$-local real numbers} & \tRR_\AtSymbol &
      \{(\delta,\upsilon) \mid \text{0$_\AtSymbol$, 1, 2, 3, 4$_\SeeSymbol$, 5}\} \\
  \end{array}
  \]
\caption{The local Dedekind numeric types}\label{table:local_numeric_types}
\end{table}

We also saw in this section that the various inclusions preserve much of the structure present, as
summarized in the following corollary.

\begin{corollary}\label{cor:everything_goes}
  Let $j$ be one of the following modalities: $\At{}$, $\See{}$, or $\pi$. Then we have the
  following subtype relationships, all of which preserve domain order and constants:
  \[
  \begin{tikzcd}[column sep=5pt, row sep=2pt]
      \tLR_j&\ss&\tLR
    &\qquad&
        \tUR_j&\ss&\tUR
    &\qquad&
          \tII_j&\ss&\tII
    \\
      \vertss&&\vertss
    &&
        \vertss&&\vertss
    &&
          \vertss&&\vertss
    \\
      \tLRub_j&\ss&\tLRub
    &&
        \tURub_j&\ss&\tURub
    &&
          \tIIub_j&\ss&\tIIub
  \end{tikzcd}
  \]
  Moreover if $j$ is $\At{}$ or $\See{}$ then all of the above inclusions also preserve arithmetic
  and inequalities.
\end{corollary}
\begin{proof}
  In all three diagrams, the inclusion of top into bottom is obvious: in each case the subtype is
  cut out by some axiom. In the case of $\pi$, the horizontal inclusions were shown in
  \cref{prop:pi_inclusion_arith}. For $\See{}$, \cref{prop:closed_inclusion_arith} shows that the
  horizontal inclusions exist and preserve arithmetic and inequalities. Finally, for $\At{}$,
  \cref{prop:At_numeric_See} shows that the horizontal inclusions exist and preserve arithmetic (by
  composing with the inclusions for $\See{}$), and \cref{prop:At_inclusion_comparisons} shows that
  these inclusions also preserve inequalities.
\end{proof}

\section{Semantics of numeric types in various modalities}
\label{sec:semantics_numerics_modalities}\index{numeric type!$j$-local!semantics of}

Recall from \cref{prop:H_to_IR_cont} that there is a poset isomorphism $h\colon\IR\cong H$ between
the interval domain and the upper half-plane $H=\{(x,y)\in\RR^2\mid y\geq 0\}$\index{half-plane}. It
is given by $h([d,u])\coloneqq (m,r)$, where $m\coloneqq(u+d)/2$ and $r\coloneqq(u-d)/2$ are the
midpoint and radius. We call $h([d,u])$ the \emph{half-plane representation} of $[d,u]$. Recall from
\cref{eqn:up_way_up_cone} that the sets $\up(m,r)$ and $\upclose(m,r)$ can be visualized as the
cones:
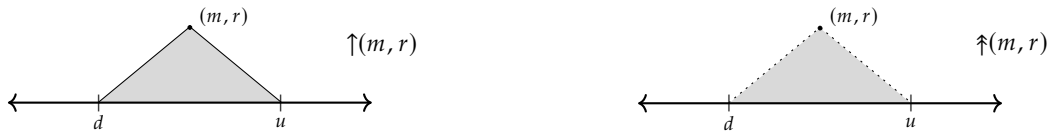
\[
  \begin{tikzpicture}[baseline=(name), font=\tiny]
    \mypic{1}{2}{}{|}
    \node [circle, inner sep=.5pt, draw, fill=black] at (M) {};
    \node [above right = -7pt and -4pt of M, font=\tiny] {$(m,r)$};
    \node [above=10pt of rend, font=\footnotesize] (name) {$\up(m,r)$};
    \node at ($(L)+(0,-1.5ex)$) {$d$};
    \node at ($(R)+(0,-1.5ex)$) {$u$};
  \end{tikzpicture}
\hspace{1in}
  \begin{tikzpicture}[baseline=(name), font=\tiny]
    \mypic{1}{2}{dotted}{|}
    \node [circle, inner sep=.5pt, draw, fill=black] at (M) {};
    \node [above right = -7pt and -4pt of M, font=\tiny] {$(m,r)$};
    \node [above=10pt of rend, font=\footnotesize] (name) {$\upclose(m,r)$};
    \node at ($(L)+(0,-1.5ex)$) {$d$};
    \node at ($(R)+(0,-1.5ex)$) {$u$};
  \end{tikzpicture}
\]
The horizontal axis in $H$ represents the real numbers $\RR\ss\IR$, so we call it the \emph{real
axis}. In particular the line segment from $d$ to $u$ is $\up(m,r)\cap h(\RR)$, but we denote it
$\up(m,r)\cap\RR$ for typographical simplicity. We will see that the $\pi$-numeric types are
determined by their values on the real axis.

For visualizing the $\See{[d,u]}$- and $\At{[d,u]}$-numeric types, it will also be useful to have a
picture of $\down(m,r)$:
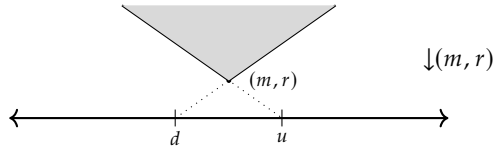
\[
  \begin{tikzpicture}[font=\tiny]
    \node (lend) {};
    \node [right=2 of lend] (L) {|};
    \node [right=1 of L] (R) {|};
    \node [right=2 of R] (rend) {};
    \node at ($(L)!.5!(R)+(0,1/2)$) (M) {};
    \draw[thin, black, dotted] (L.center) -- (M.center) -- (R.center);
    \draw[thick, <->] (lend) -- (rend);
    \node [circle, inner sep=.5pt, draw, fill=black] at (M) {};
    \node at ($(M)!-2!(L)$) (tleft) {};
    \node at ($(M)!-2!(R)$) (tright) {};
    \draw[thick, black] (M.center) -- (tleft.center);
    \draw[thick, black] (M.center) -- (tright.center);
    \fill[fill=black!15] (M.center) -- (tleft.center) -- (tright.center) -- cycle;
    \node [right = 0 of M, font=\tiny] {$(m,r)$};
    \node at ($(L)+(0,-1.5ex)$) {$d$};
    \node at ($(R)+(0,-1.5ex)$) {$u$};
    \node [above=10pt of rend, font=\footnotesize] (name) {$\down(m,r)$};
  \end{tikzpicture}
\]
As discussed in \cref{sec:IR_half_plane}, we consider $H$ as having the usual Euclidean topology.

The following omnibus theorem collects characterizations of the semantics of all the numeric types
we consider. The new characterizations follow easily from
\cref{cor:Dedekind_cut_semantics,prop:modality_semantics_pointwise}, together with the results of
\cref{sec:Real_numbers_modalities} as summarized in \cref{table:local_numeric_types}.
\index{semantics}\index{numeric type!semantics of}

\begin{theorem}
\label{thm:omnibus_semantics}\index{semantics!omnibus theorem}
  Suppose we are in a context with a variable $t:\Time$, and we fix an element
  $(d_t,u_t)\in\church{\Time}(\ell)$. Let $(m_t,r_t)\coloneqq h([d_t,u_t])$ be its half-plane
  representation.
  \begin{enumerate}
    \item $\church{\tLR}(\ell) \iso \{\,f\colon \upclose(m_t,r_t)\to\RR\cup\{\infty\} \mid
      f\text{ is lower semi-continuous and order-preserving}\,\}$
    \item $\church{\tUR}(\ell) \iso \{\,f\colon \upclose(m_t,r_t)\to\RR\cup\{-\infty\} \mid
      f\text{ is upper semi-continuous and order-reversing}\,\}$
    \item $\church{\tII}(\ell) \iso \church{\tLR}(\ell) \times \church{\tUR}(\ell)$
    \item $\church{\tIR}(\ell) \iso \{\,(\ubar{f},\bar{f})\in\church{\tII}(\ell) \mid
      \forall(p\in \upclose(m_t,r_t))\ldotp \ubar{f}(p)\leq\bar{f}(p)\,\}$
    \item $\church{\tRR}(\ell) \cong \{\,(\ubar{f},\bar{f})\in\church{\tII}(\ell) \mid
      \forall(p\in \upclose(m_t,r_t))\ldotp \ubar{f}(p)=\bar{f}(p)\,\} \cong \RR$
  \end{enumerate}
  In the $\pi$-modality, the $f$'s are determined by their values on the real axis:
  \begin{enumerate}\setcounter{enumi}{5}
    \item $\church{\tLR_\pi}(\ell) \iso \{\,f\in\church{\tLR}(\ell) \mid
      \forall(p\in\upclose(m_t,r_t))\ldotp f(p) = \inf f(\specupclose p \cap \RR)\,\}$
    \item $\church{\tUR_\pi}(\ell) \iso \{\,f\in\church{\tUR}(\ell) \mid
      \forall(p\in\upclose(m_t,r_t))\ldotp f(p) = \sup f(\specupclose p \cap \RR)\,\}$
    \item $\church{\tII_{\pi}}(\ell) \iso
      \church{\tLR_{\pi}}(\ell) \times \church{\tUR_{\pi}}(\ell)$
    \item $\church{\tIR_\pi}(\ell) \iso \{\,(\ubar{f},\bar{f})\in\church{\tII_\pi}(\ell) \mid
      \forall(p\in \upclose(m_t,r_t))\ldotp \ubar{f}(p)\leq\bar{f}(p)\,\}$
    \item $\church{\tRR_{\pi}}(\ell) \cong \{\,(\ubar{f},\bar{f})\in\church{\tII_{\pi}}(\ell) \mid
      \forall(p\in\upclose(m_t,r_t)\cap\RR)\ldotp \ubar{f}(p)=\bar{f}(p)\,\}$
  \end{enumerate}
  Let $d,u:\tRR$, and let $(m,r)=h([d,u])$ be its half-plane representation. In the $\See{}$
  modality, the $f$'s are determined by their values in $\down(m,r)$:
  \begin{enumerate}\setcounter{enumi}{10}
    \item $\church{\tLR_{\SeeInline{[d,u]}}}(\ell) \iso \{\,f\in\church{\tLR}(\ell) \mid
      \forall(p\in\upclose(m_t,r_t))\ldotp p\in\down(m,r) \vee f(p)=\infty\,\}$
    \item $\church{\tUR_{\SeeInline{[d,u]}}}(\ell) \iso \{\,f\in\church{\tUR}(\ell) \mid
      \forall(p\in\upclose(m_t,r_t))\ldotp p\in\down(m,r) \vee f(p)=-\infty\,\}$
    \item $\church{\tII_{\SeeInline{[d,u]}}}(\ell) \iso
      \church{\tLR_{\SeeInline{[d,u]}}}(\ell) \times \church{\tUR_{\SeeInline{[d,u]}}}(\ell)$
    \item $\church{\tIR_{\SeeInline{[d,u]}}}(\ell) \iso
      \{\,(\ubar{f},\bar{f})\in\church{\tII_{\SeeInline{[d,u]}}}(\ell) \mid
      \forall(p\in \upclose(m_t,r_t))\ldotp p\in\down(m,r) \imp \ubar{f}(p)\leq\bar{f}(p)\,\}$
    \item $\church{\tRR_{\SeeInline{[d,u]}}}(\ell) \iso
      \{\,(\ubar{f},\bar{f})\in\church{\tII_{\See{[d,u]}}}(\ell) \mid
      \forall p\in \upclose(m_t,r_t)\ldotp p\in\down(m,r) \imp \ubar{f}(p)=\bar{f}(p)\,\}$
  \end{enumerate}
  Let $(m,r)$ be as above. In the $\At{}$ modality, the $f$'s are constant in $\down(m,r)$:
  \begin{enumerate}\setcounter{enumi}{15}
    \item $\church{\tLR_{\AtInline{[d,u]}}}(\ell)
      \iso \{\,f\in\church{\tLR_{\SeeInline{[d,u]}}}(\ell) \mid
      \forall(p\in\upclose(m_t,r_t))\ldotp p\in\down(m,r) \imp f(p)=f((m,r))\,\}$
    \item $\church{\tUR_{\AtInline{[d,u]}}}(\ell) \iso
      \{\,f\in\church{\tUR_{\SeeInline{[d,u]}}}(\ell) \mid
      \forall(p\in\upclose(m_t,r_t))\ldotp p\in\down(m,r) \imp f(p)=f((m,r))\,\}$
    \item $\church{\tII_{\AtInline{[d,u]}}}(\ell) \iso
      \church{\tLR_{\AtInline{[d,u]}}}(\ell) \times \church{\tUR_{\AtInline{[d,u]}}}(\ell)$
    \item $\church{\tIR_{\AtInline{[d,u]}}}(\ell) \iso
      \{\,(\ubar{f},\bar{f})\in\church{\tII_{\AtInline{[d,u]}}}(\ell) \mid
      (m,r)\in\upclose(m_t,r_t) \imp \ubar{f}((m,r))\leq\bar{f}((m,r))\,\}$
    \item $\church{\tRR_{\AtInline{[d,u]}}}(\ell) \iso
      \{\,(\ubar{f},\bar{f})\in\church{\tII_{\At{[d,u]}}}(\ell) \mid
      (m,r)\in\upclose(m_t,r_t) \imp \ubar{f}((m,r))=\bar{f}((m,r))\,\}$
  \end{enumerate}
\end{theorem}

\begin{remark}\label{rem:independence_in_omnibus}
  Up to isomorphism, the sets given in numbers 1 -- 10 are independent of the choice of $t:\Time$
  and $(d_t,u_t)$, and hence $(m_t,r_t)$. The reason they are written in those terms is in order to
  facilitate the comparison between 1 -- 10 and 11 -- 20.
\end{remark}

\begin{example}\index{numeric type!$\At{}$ as constant}

The internal fact that $\tRR_{\AtInline{[d,u]}}=\tRR_{\SeeInline{[d,u]}}$, proven in   
\cref{prop:See_At_Reals_iso} is  reflected semantically in the fact that the sets 
described in 15 and 20 are the same. One might imagine an element of
$\church{\tRR_{\AtInline{[d,u]}}}(\ell)$ like this:
  \[
  \begin{tikzpicture}[font=\tiny]
    \node (lend) {};
    \node [right=2 of lend] (L) {|};
    \node [right=1 of L] (R) {|};
    \node [right=2 of R] (rend) {};
    \node at ($(L)!.5!(R)+(0,1/2)$) (M) {};
    \draw[thin, black, dotted] (L.center) -- (M.center) -- (R.center);
    \draw[thick, <->] (lend) -- (rend);
    \node [circle, inner sep=.5pt, draw, fill=black] at (M) {};
    \node at ($(M)!-2!(L)$) (tleft) {};
    \node at ($(M)!-2!(R)$) (tright) {};
    \draw[thick, black] (M.center) -- (tleft.center);
    \draw[thick, black] (M.center) -- (tright.center);
    \fill[fill=black!15] (M.center) -- (tleft.center) -- (tright.center) -- cycle;
    \node [right = 0 of M, font=\tiny] {$(m,r)$};
    \node at ($(L)+(0,-1.5ex)$) {$d$};
    \node at ($(R)+(0,-1.5ex)$) {$u$};
    \node [above=1ex of M] {\normalsize $5$};
  \end{tikzpicture}
  \]
  where $(m,r)=h([d,u])$.
\end{example}

\begin{corollary}
\label{cor:pi_real_semantics}\index{real number!variable}\index{sheaf!of real-valued functions}
  Let $(0,\ell)\ss\RR$ denote the open interval of length $\ell$. There is an isomorphism
  \[
    \church{\tRR_\pi}(\ell)\cong\{\,f\colon(0,\ell)\to\RR\mid f\text{ is continuous}\,\}.
  \]
\end{corollary}
\begin{proof}
  For $f\colon(0,\ell)\to\RR$, define $\ubar{f}\colon\upclose[0,\ell]\to\RR$ and
  $\bar{f}\colon\upclose[0,\ell]\to\RR$ by
  \[
    \ubar{f}(p)\coloneqq\inf f(\up p\cap\RR)
    \qquad\text{and}\qquad
    \bar{f}(p)\coloneqq\sup f(\up p\cap\RR)
  \]
  It is easy to see that $\ubar{f}$ is automatically order-preserving and $\bar{f}$ is automatically
  order-reversing. Thus by \cref{thm:omnibus_semantics,rem:independence_in_omnibus} we have an
  isomorphism:
  \[
    \church{\tRR_\pi} \cong \{f\colon(0,\ell)\to\RR\mid\ubar{f}
    \text{ is lower semi-continuous and }\bar{f}\text{ is upper semicontinuous}\}.
  \]
  If $f$ is continuous then $\ubar{f}$ is lower semi-continuous and $\bar{f}$ is upper
  semi-continuous; it remains to show the converse. For any $0<x<\ell$ and $\epsilon>0$, the
  neighborhood $U\coloneqq(f(x)-\epsilon,f(x)+\epsilon)$ is the intersection of
  $U_1\coloneqq\{y\in\RR\mid f(x)-\epsilon<y\}$ and $U_2\coloneqq\{y\in\RR\mid y<f(x)+\epsilon)$. By
    definition of lower- and upper-semicontinuity, the pullbacks $\ubar{f}^{-1}(U_1)$ and
    $\bar{f}^{-1}(U_2)$ are open neighborhoods of $x\in\upclose[0,\ell]$, so
    $f^{-1}=\ubar{f}^{-1}(U_1)\cap\bar{f}^{-1}(U_2)\cap\RR$ is open as well.
\end{proof}

\begin{remark}\label{rem:time_semantics}\index{Time@$\Time$!as clock behavior}
\index{Time@$\Time$!semantics of}
  Our original \emph{external} definition of $\Time$ was \cref{eqn:def_time}, $\Time(\ell)=\{\,(d,u)\in\RR^2\mid u-d=\ell\,\}$. Our original \emph{internal} definition of $\Time$ was as a subtype of $\tII$; see
  \cref{eqn:define_time}. In \cref{lemma:Time_sub_II} and throughout \cref{sec:justification} we showed that the former is a sound semantics for the latter. With \cref{thm:omnibus_semantics,cor:pi_real_semantics} in front of us, it may be worthwhile to summarize how $\Time$ fits in with $\tII$, $\tIR$, and $\tRR_\pi$.
  
  The elements of $\church{\tII}$ are pairs of lower/upper semi-continuous order preserving/reversing functions. As a subsheaf, $\church{\Time}\ss\church{\tII}$
  is the set of pairs $(\ubar{t},\bar{t})\colon\upclose[0,\ell]\to\RR$, for which there exists
  $r\in\RR$ with $\ubar{t}([d,u])\coloneqq d+r$ and $\bar{t}([d,u])\coloneqq u+r$. Clearly $\ubar{t}\leq\bar{t}$, so we also have $\church{\Time}\ss\church{\IR}$. 

  In \cref{ex:Time_subtype_rvar} we identified $\Time$ internally as a subtype of $\tRR_\pi$. Under this
  identification, we could write
  \[
    \church{\Time}(\ell)\cong\big\{\,(x\mapsto r+x)\colon(0,\ell)\to\RR\mid r\in\RR\,\big\}
  \]
  i.e.\ the set of real-valued functions on the open interval $(0,\ell)\ss\RR$ that have unit slope.
  In \cref{rem:time_clock_temporal} we called these ``clock behaviors''.
\end{remark}

\begin{remark}\index{continuity}
  Semantically, $\tRR_\pi$ denotes the sheaf of continuous real-valued functions. Something of the
  sort can also be proven internally. To do so, we need to define yet another numeric object
  $\tRR^?$, which is very similar to $\tRR$ but with a slightly weakened disjointness axiom:
  \begin{description}\index{cuts!weakly-disjoint}
    \item[\quad weakly-disjoint:]
      $(\exists(q:\tQQ)\ldotp \delta q \wedge \upsilon q)
        \imp \forall(q:\tQQ)\ldotp \delta q\wedge\upsilon q.$
  \end{description}
  So define $\tRR^?=\{(\delta,\upsilon):\tII\mid\text{ located and weakly-disjoint}\}$. If
  $x=(\delta,\upsilon):\tRR^?$ satisfies $\forall(q:\tQQ)\ldotp \delta q\wedge\upsilon q$, we write
  $x=\const{undef}$.

  We can now internally define a function $\tRR_\pi\to\Time\to\tRR^?$, sending
  $x=(\delta,\upsilon):\tRR_\pi$ to the function $f_x:\Time\to\tRR^?$ given by
  $f_x(t)\coloneqq\At[t]{0}x$, or more precisely $f_x(t)=(\delta_{x,t},\upsilon_{x,t})$ where
  $\delta_{x,t}q\iff\At[t]{0}\delta q$ and $\upsilon_{x,t}q\iff\At[t]{0}\upsilon q$.

  Soon we will prove that $f_x$ is continuous, but first we want to check that $f_x(t)\in\RR^?$ for
  any given $x,t$. It is easy to see that $f_x(t)$ is down/up-closed and bounded. It is rounded by
  \cref{prop:At_preserves_exists}. We check that it is located and weakly-disjoint. Given $q_1<q_2$
  we have $\pi(\delta q_1\vee\upsilon q_2)$, which by definition is
  $\forall(t:\Time)\ldotp\At[t]{0}(\delta q_1\vee\upsilon q_2)$, and we obtain locatedness by
  \cref{lem:At_preserves_or}. If $\At[t]{0}(\delta q\wedge\upsilon q)$ then we have
  $\At[t]{0}(\pi\bot)$ by $\pi$-disjointness, and hence $\At[t]{0}\bot$ by \cref{cor:pi_negneg}.
  Thus we obtain $\forall(q:\tQQ)\ldotp\At[t]{0}(\delta q\wedge\upsilon q)$, proving $f_x(t)$ is
  weakly-disjoint.

  Finally, we want to show that $f_x:\Time\to\RR^?$ is continuous in the sense that the following
  holds:
  \begin{multline*}
    \forall(\epsilon>0)(t:\Time)\ldotp (f_x(t)=\const{undef})\vee\\
    \exists(\delta>0)\ldotp\forall(r:\tRR)\ldotp (-\delta<t+r<\delta)\imp |f_x(t)-f_x(t+r)|<\epsilon.
  \end{multline*}
  By arithmetic locatedness (see \cref{prop:arithmetic_located_j}, which also holds for $\tRR^?$) and
  \cref{lemma:points_in_subtoposes}, it suffices to show
  \begin{multline*}
    \forall(q_1,q_2:\tQQ)(t:\Time)\ldotp(q_1<f_x(t)<q_2)\imp(f_x(t)=\const{undef})\vee\\
    \exists(d,u:\tQQ)\ldotp(d<u)\wedge\In[t]{[d,u]}(q_1<f_x<q_2).
  \end{multline*}
  and this follows from \cref{ax:point_to_open}.

	However, we have not seen any use for the above perspective, because it is far easier to work with the Dedekind axioms defining $\tRR_\pi$ than with the above $\epsilon,\delta$ definition, which easily gets unwieldy.
\end{remark}

\section{Derivatives of interval-valued functions}\label{sec:derivatives}\index{derivative}

In this section we will define what it means for one interval-valued function to be the derivative
of another. We use Newton's notation $\dot{x}$ for the derivative because its semantics will be
differentiation with respect to $t:\Time$. If $x$ has a derivative $\dot{x}$ and both
$x,\dot{x}:\tRR_\pi$ are variable reals, we say that $x$ is differentiable. For example, we will see
that $x(t)=t^2$ is differentiable and its derivative is $2*t$. We will show that differentiation is
linear and satisfies the Leibniz rule in \cref{th:leibniz}.

The derivative can be defined for any $x:\tIIub_\pi$. Because $\tIIub_\pi$ is a domain (see
\cref{prop:intervals_as_RIds,prop:sqss_ll}), we can define $\dot{x}$ by giving a directed family of
$y\in\tIIub_\pi$ which approximate $\dot{x}$, and take $\dot{x}$ to be the directed join of the
approximating family.

\begin{definition}[Approximates the derivative]
\label{def:derivative}\index{derivative!approximate}
  Let $x,y:\tIIub_\pi$ be $\pi$-local improper intervals. We say that $y$ \emph{approximates the
  derivative of $x$}, if for any $t:\Time$ it satisfies the following:
  \begin{equation}\label{eqn:deriv_approx}
    \forall(r_1,r_2:\tRR)\ldotp r_1<r_2\imp y\sqss\frac{x^@(r_2)-x^@(r_1)}{r_2-r_1}.
  \end{equation}
  For now we write $\AD(y,x)$ iff $y$ satisfies \cref{eqn:deriv_approx}.
\end{definition}

By \cref{ax:torsor}, $\AD(y,x)$ is independent of $t:\Time$; i.e.\ it is true for one choice iff it
is true for any another. Thus in the results below, a choice of $t:\Time$ is implied, but is left
out of the statements.

\begin{lemma}\label{lem:sqss_slope_At_closed}
  For all $x,y:\tIIub_\pi$ and $r_1<r_2:\tRR$, the fraction $\frac{x^@(r_2)-x^@(r_1)}{r_2-r_1}$ is
  in $\tII_{\AtInline{[r_1,r_2]}}$, and the proposition $y\sqss\frac{x^@(r_2)-x^@(r_1)}{r_2-r_1}$ is
  $\At{[r_1,r_2]}$-closed.
\end{lemma}
\begin{proof}
  By \cref{prop:At_numeric_See,rem:enlarging_closed}, the difference $x^@(r_2)-x^@(r_1)$ is a
  constant in $\tIIub_{\SeeInline{[r_1,r_2]}}$, and thus it is in $\tIIub_{\AtInline{[r_1,r_2]}}$ by
  \cref{prop:At_numeric_See}. The reciprocal $1/(r_2-r_1)$ is a (constant) real by
  \cref{prop:reciprocal}, so the product
  $(\delta,\upsilon)\coloneqq\frac{x^@(r_2)-x^@(r_1)}{r_2-r_1}$ is in
  $\tIIub_{\AtInline{[r_1,r_2]}}$. By \cref{cor:everything_goes}, we can consider both sides of the
  expression $y\sqss (\delta,\upsilon)$ as terms in $\tII$, and $\sqss$ is simply implication:
  $q<y\imp \delta q$ and $y<q\imp \upsilon q$. Since $\delta$ and $\upsilon$ are
  $\At{[r_1,r_2]}$-closed, so are these implications.
\end{proof}

\begin{lemma}\label{lem:In_AD}
  Choose $x,y:\tIIub_\pi$ and $d,u:\tRRub$. Then $\In{[d,u]}\AD(y,x)$ is
  equivalent to
  \[
    \forall(r_1,r_2:\tRR)\ldotp d<r_1<r_2<u \imp y\sqss\frac{x^@(r_2)-x^@(r_1)}{r_2-r_1}.
  \]
\end{lemma}
\begin{proof}
  By \cref{eqn:In_out_of_imp}, $\In{[d,u]}\AD(y,x)$ is equivalent to
  \[
    \forall(r_1,r_2:\tRR)\ldotp r_1<r_2 \imp
    \In{[d,u]}\left(y\sqss\frac{x^@(r_2)-x^@(r_1)}{r_2-r_1}\right).
  \]
  By \cref{lem:sqss_slope_At_closed,prop:points_in_subtoposes}, this is equivalent to
  \[
    \forall(r_1,r_2:\tRR)\ldotp r_1<r_2 \imp
    ((d<r_1)\wedge(r_2<u))\imp\left(y\sqss\frac{x^@(r_2)-x^@(r_1)}{r_2-r_1}\right),
  \]
  which is clearly equivalent to the desired conclusion.
\end{proof}

\begin{proposition}\label{prop:AD_Pi-closed}
  For any $x,y:\tIIub_\pi$, the proposition $\AD(y,x)$ is $\pi$-closed.
\end{proposition}
\begin{proof}
  We prove this using \cref{ax:covering_pi}. Let $q_1<q_2$ be rationals, and suppose
  $t\apart[q_2,q_1]\imp\AD(y,x)$. This is equivalent to
  $\In{[-\infty,q_2]}\AD(y,x)\wedge\In{[q_1,\infty]}\AD(y,x)$. We want to show $\AD(y,x)$. By
  \cref{lem:In_AD}, we have
  \begin{gather}
    \label{eq:AD_Pi-closed_A}
    \forall(r_1,r_2:\tRR)\ldotp r_1<r_2<q_2 \imp y\sqss\frac{x^@(r_2)-x^@(r_1)}{r_2-r_1} \\
    \label{eq:AD_Pi-closed_B}
    \forall(r_1,r_2:\tRR)\ldotp q_1<r_1<r_2 \imp y\sqss\frac{x^@(r_2)-x^@(r_1)}{r_2-r_1},
  \end{gather}
  and we want to show
  \[
    \forall(r_1,r_2:\tRR)\ldotp r_1<r_2 \imp y\sqss\frac{x^@(r_2)-x^@(r_1)}{r_2-r_1}.
  \]
  Take any $r_1<r_2:\tRR$. By trichotomy (\ref{prop:trichotomy}) $(q_1<r_1)\vee(r_1\leq q_1)$ and
  $(r_2<q_2)\vee(q_2\leq r_2)$. If $q_1<r_1$ then we are done by \eqref{eq:AD_Pi-closed_B}, and if
  $r_2<q_2$ we are done by \eqref{eq:AD_Pi-closed_A}.

  Now assume $r_1\leq q_1$ and $q_2\leq r_2$, and choose any $r_m$ with $q_1<r_m<q_2$. Then
  \eqref{eq:AD_Pi-closed_A} and \eqref{eq:AD_Pi-closed_B} give us
  \[
    y\sqss\frac{x^@(r_m)-x^@(r_1)}{r_m-r_1}
    \qquad\text{and}\qquad
    y\sqss\frac{x^@(r_2)-x^@(r_m)}{r_2-r_m}.
  \]
  By continuity of multiplication (\cref{thm:mult_is_continuous}), we get $yr_m-yr_1\sqss
  x^@(r_m)-x^@(r_1)$ and $yr_2-yr_m\sqss x^@(r_2)-x^@(r_m)$, and by continuity of addition
  (\cref{thm:addition_is_continuous}), we can add these together to get
  \[
    y(r_2-r_1) = yr_2 - yr_1 \sqss x^@(r_2)-x^@(r_1)
  \]
  as desired.
\end{proof}

Recall from \cref{not:constant_intervals} the notation for constant intervals $[d,u]:\ctIIub_j$,
where $d,u:\ctRRub_j$, or in particular where $d,u:\tQQ$. Recall the $\ll$ relation on $\tIIub$ from
\cref{prop:sqss_ll}.

\begin{proposition}\label{prop:AD_directed_family}
  For a fixed $x:\tIIub_\pi$, the family $\{\,y:\tIIub_\pi \mid \AD(y,x)\,\}$ is directed and
  down-closed.
\end{proposition}
\begin{proof}
  Clearly $\AD([-\infty,\infty],x)$ since $[-\infty,\infty]\sqss z$ for any $z:\tIIub$. If
  $\AD(y_1,x)$ and $\AD(y_2,x)$, it is equally clear that $\AD(y_1\sqcup y_2,x)$. Down-closedness is
  obvious.
\end{proof}

\begin{proposition}\label{prop:derivative_cuts}\index{derivative!cuts of}
  Let $x:\tIIub_\pi$, and let $(\dot{\delta},\dot{\upsilon}):\tIIub_\pi$ denote the join of the
  directed family $\{\,y:\tIIub_\pi\mid\AD(y,x)\,\}$. Then for any $q:\tQQ$,
  \begin{equation}\label{eqn:cuts_derivative}
  \begin{aligned}
    \dot\delta d&\coloneqq \exists(d':\tQQ)\ldotp(d<d')\wedge\forall(r_1,r_2:\tRR)\ldotp r_1<r_2
      \imp d'<\frac{x^@(r_2)-x^@(r_1)}{r_2-r_1} \\
    \dot\upsilon u&\coloneqq \exists(u':\tQQ)\ldotp (u>u')\wedge\forall(r_1,r_2:\tRR)\ldotp r_1<r_2
      \imp u'>\frac{x^@(r_2)-x^@(r_1)}{r_2-r_1}
  \end{aligned}
  \end{equation}
\end{proposition}
\begin{proof}
  Let $D:\tIIub_\pi\to\Prop_\pi$ be given by $D(y)\coloneqq \AD(y,x)$. This defines a subobject of
  $\tIIub$, i.e.\ a family of elements of $\tIIub$. Clearly the join of this family should be given
  by the supremum of all lower-bounds together with the infimum of all upper-bounds, taken in the
  $\pi$-modality. In other words, the join $(\dot\delta,\dot\upsilon)$ of $D$
  satisfies the following formulas for any $(d,u):\tIIubpre$,
  \begin{equation*}
    \dot\delta d \coloneqq \pi\exists(y:\tIIub)\ldotp \AD(y,x) \wedge d<y,
    \qquad
    \dot\upsilon u \coloneqq \pi\exists(y:\tIIub)\ldotp \AD(y,x) \wedge y<u.
  \end{equation*}

  It is simple to show that these are equivalent to
  \begin{align*}
    \dot\delta d &\iff \pi\exists(d':\tLRubpre)\ldotp \AD([d',+\infty],x) \wedge (d<d'), \\
    \dot\upsilon u &\iff \pi\exists(u':\tLRubpre)\ldotp \AD([-\infty,u'],x) \wedge (u'<u).
  \end{align*}
  Indeed, that the latter implies the former is trivial. Conversely, suppose $\AD(y,x)$ and $d<y$. Then there
  exists a $d'$ with $d<d'<y$, and $[d',+\infty]\sqss y$ implies $\AD([d',+\infty],x)$ by
  \cref{prop:AD_directed_family}.

  We next want to apply \cref{lem:exists_Pi_closed}. The predomain $\tIIubpre$ is constant, and the
  corresponding poset $(\tIIubpre,\specord)\op$ is directed (in fact, it is linear); see
  \cref{ex:predomains_with_binarys,ex:constant_predomains,prop:intervals_as_RIds}. It follows that
  the subobject $\{d':\tIIubpre \mid d<d'\}$ is also constant and directed (linear), and it follows
  from \cref{prop:AD_directed_family} that $d''\leq d'$ implies $\AD([d',+\infty],x) \imp
  \AD([d'',+\infty])$. Hence the hypothesis of \cref{lem:exists_Pi_closed} are satisfied, which shows
  the first line of
  \begin{align*}
    \dot\delta d &\iff \exists(d':\tLRubpre)\ldotp \AD([d',+\infty],x) \wedge (d<d'), \\
    \dot\upsilon u &\iff \exists(u':\tLRubpre)\ldotp \AD([-\infty,u'],r) \wedge (u'<u).
  \end{align*}
  The second line follows similarly.

  Finally, by restricting to the case where $d\in\tQQ\subseteq\tLRubpre$ and
  $u\in\tQQ\subseteq\tURubpre$, and unfolding the definition of $\AD$, we have shown
  \eqref{eqn:cuts_derivative}.
\end{proof}

\begin{definition}[Derivative]\label{def:internal_derivative}
  Let $x:\tIIub_\pi$. We define the \emph{derivative of $x$}, denoted $\dot{x}:\tIIub_\pi$, to be
  the join of the directed family $\{\,y:\tIIub_\pi\mid\AD(y,x)\,\}$. Its cuts
  $\dot{x}=(\dot\delta,\dot\upsilon)$ are given by \cref{eqn:cuts_derivative}. We sometimes denote
  the derivative $\dot{x}$ by $\frac{d}{dt}(x)$.
\end{definition}

\begin{lemma}\label{lemma:deriv_approximates_deriv}
  For any $x:\tIIub_\pi$, $\dot{x}$ approximates the derivative of $x$, i.e.\ $\AD(\dot{x},x)$.
\end{lemma}
\begin{proof}
  We need to show that for any reals $r_1<r_2$, $\dot{x}\sqss\frac{x^@(r_2)-x^@(r_2)}{r_2-r_1}$.
  This follows directly from the fact that $\dot{x}$ is the join of the family
  $\{y:\tIIub_\pi\mid\AD(y,x)\}$, and that by definition $y\sqss\frac{x^@(r_2)-x^@(r_2)}{r_2-r_1}$
  for any such $y$.
\end{proof}

\Cref{lemma:deriv_approximates_deriv} immediately implies:

\begin{proposition}\label{prop:AD_iff_sqss_dot}
  For any $x,y:\tIIub_\pi$, $\AD(y,x) \iff y\sqss\dot{x}$.
\end{proposition}

At this point we have internally defined the derivative of any $\pi$-local improper interval and justified the terminology ``$y$ approximates the derivative of $x$'' from \cref{def:derivative}. In the next section we will show that this internal definition has the correct external semantics. Note that it should also be possible to show that variable reals $x:\tRR_\pi$ have an antiderivative as well. We leave this as an open question for interested readers.\index{open question}

\begin{lemma}\label{lemma:disjoint_AD_disjoint}\index{derivative!disjointness}
  Let $x\in\tIIub_\pi$ be $\pi$-disjoint; i.e.\ $x\in\tIRub_\pi$. Then for any $y\in\tIIub_\pi$, if
  $\AD(y,x)$, then $y$ is also disjoint.
\end{lemma}
\begin{proof}
  Suppose $[q_1,q_2]\ll y$ for rationals $q_1,q_2:\tQQ$. For any reals $r_1<r_2$,
  $y\sqss\frac{x^@(r_2)-x^@(r_1)}{r_2-r_1}$. But $\frac{x^@(r_2)-x^@(r_1)}{r_2-r_1}$ is in
  $\tIRub_{\AtInline{[r_1,r_2]}}$ by \cref{rem:enlarging_closed}, hence is
  $\At{[r_1,r_2]}$-disjoint, which implies $\At{[r_1,r_2]}(q_1<q_2)$. So
  $\forall(r_1,r_2:\tRR)\ldotp r_1<r_2\imp\At{[r_1,r_2]}(q_1<q_2)$, which by
  \cref{ax:enough_points} implies $q_1<q_2$. Therefore $y$ is disjoint.
\end{proof}

\begin{corollary}\label{cor:derivatives_disjoint}
  Taking derivatives preserves disjointness: if $x\in\tIRub_\pi$, then $\dot{x}\in\tIRub_\pi$.
\end{corollary}

\begin{remark}
  Given a bounded improper interval $x:\tII_\pi$, it may not be the case that $\dot{x}$ is
  $\pi$-bounded. For example, for any $t:\Time$ the function $x=t^{\frac{1}{3}}$, defined by $\delta
  q\iff q^3<t$ and $\upsilon q\iff t<q^3$, has unbounded derivative at $0$.
\end{remark}

\begin{proposition}\label{prop:deriv_open_closed}
  For any variable real $x:\tRR_\pi$ and real numbers $r,r_1,r_2:\tRR$ with $r_1<r_2$, we have the
  following:
  \[
  \begin{aligned}
    \forall(t:\Time)\ldotp \In{[r_1,r_2]}(r\leq \dot{x}) &\imp r \leq \frac{x^@(r_2)-x^@(r_1)}{r_2-r_1}, \\
    \forall(t:\Time)\ldotp \In{[r_1,r_2]}(\dot{x}\leq r) &\imp \frac{x^@(r_2)-x^@(r_1)}{r_2-r_1} \leq r.
  \end{aligned}
  \]
\end{proposition}
\begin{proof}
  The two are similar, so we prove the first. Take any $r,r_1,r_2:\tRR$ with $r_1<r_2$ and
  $t:\Time$, and let $y\coloneqq\frac{x^@(r_2)-x^@(r_1)}{r_2-r_1}$. It is easy to prove that
  \[
    (y < r) \imp [(r\leq \dot{x}) \imp \bot]
  \]
  From this follows:
  \begin{gather*}
    (y < r ) \imp
      [\In{[r_1,r_2]}(r\leq \dot{x}) \imp \In{[r_1,r_2]}\bot] \\
    ( y < r ) \imp
      [\In{[r_1,r_2]}(r\leq \dot{x}) \imp \At{[r_1,r_2]}\bot] \\
    \In{[r_1,r_2]}(r\leq \dot{x}) \imp
      [ ( y < r ) \imp \At{[r_1,r_2]}\bot ] \\
    \In{[r_1,r_2]}(r\leq \dot{x}) \imp ( r \leq y )
  \end{gather*}
  where in the first line we used the functoriality of the modality $\In{[r_1,r_2]}$, in the second
  we used $\In{[r_1,r_2]}\bot = t\apart[r_2,r_1] = \At{[r_1,r_2]}\bot$, and in the last line we used
  \cref{prop:non-strict_inequality} for the modality $\At{[r_1,r_2]}$, together with the fact that
  $y$ is $\At{[r_1,r_2]}$-closed so $y<r$ is equivalent to $y <
  r^@([r_1,r_2])$.
\end{proof}

\subsection{Semantics of derivatives}
\label{sec:deriv_semantics}\index{derivative!semantics of}
\index{semantics!of derivative|see {derivative, semantics of}}

Recall from \cref{thm:omnibus_semantics} that $\church{\tIIub_{\pi}}(\ell)$ is isomorphic to the set
of pairs $(\ubar{f},\bar{f})$ of functions $\upclose[0,\ell]\to\RRub$, such that $\ubar{f}$ is lower
semi-continuous, $\bar{f}$ is upper semi-continuous, and for all $x\in\upclose(0,\ell)$,
$\ubar{f}(x)=\inf \ubar{f}(\specupclose x \cap\RR)$ and $\bar{f}(x)=\sup\bar{f}(\specupclose
x\cap\RR)$. Equivalently, $\church{\tIIub_{\pi}}(\ell)$ is isomorphic to the set of pairs
$(\ubar{f},\bar{f})$ of functions $(0,\ell)\to\RRub$, where $(0,\ell)=\{r\in\RR\mid 0<r<\ell\}$ is
the usual open interval, and where $\ubar{f}$ is lower semi-continuous and $\bar{f}$ is upper
semi-continuous.

\begin{definition}\label{def:lower_upper_derivatives}\index{derivative!lower/upper}
  Let $V\ss\RR$ be an open subset of the real line, and let $(\ubar{f},\bar{f})$ be a pair of
  functions $V\to\RRub$ such that $\ubar{f}$ is lower semi-continuous and $\bar{f}$ is upper
  semi-continuous. We define the \emph{lower derivative} $\ubar{f}'$ and the \emph{upper derivative}
  $\bar{f}'$ to be the functions
  \begin{align*}
    \ubar{f}'(x) &\coloneqq \sup_{x\in U\subseteq V} \inf_{d<u\in U} \frac{\ubar{f}(u)-\bar{f}(d)}{u-d}\\
    \bar{f}'(x) &\coloneqq \inf_{x\in U\subseteq V} \sup_{d<u\in U} \frac{\bar{f}(u)-\ubar{f}(d)}{u-d}.
  \end{align*}
\end{definition}

\begin{remark}\label{rem:external_disjointness}
  The $\sup_{x\in U\ss V}$ and $\inf_{x\in U\ss V}$ in \cref{def:lower_upper_derivatives} can be
  replaced by $\lim$, because for example the infimum $\inf_{d<u\in U}\cdots$ increases as $U$
  shrinks around $x$. Thus for example, it is easy to see that if $\ubar{f}\leq\bar{f}$ then
  $\ubar{f'}\leq\bar{f}$. This is the external version of \cref{cor:derivatives_disjoint}.
\end{remark}

The following proposition shows that the above notions of lower derivative and upper derivative are
reasonable.

\begin{proposition}\label{prop:lower_upper_derivatives_equal}
  Suppose $V\ss\RR$ is open, $x\in V$ is a point, $f\colon V\to\RR$ is continuous, let
  $\ubar{f}\coloneqq f$ and $\bar{f}\coloneqq f$, and let $\ubar{f}'$ and $\bar{f}'$ be the lower
  and upper derivatives. They agree at $x$, i.e.\ $\ubar{f}'(x)=\bar{f'}(x)$, iff $f$ is
  continuously differentiable at $x$, and in this case all three coincide
  $\ubar{f}'(x)=f'(x)=\bar{f'}(x)$.
\end{proposition}
\begin{proof}
  The condition $\ubar{f}'(x)=\bar{f}'(x)$ is easily shown equivalent to the statement that there
  exists $y\in\RR$ such that for all $\epsilon>0$ there exists $\delta>0$ such that for all $d<u$ in
  the ball $(x-\delta,x+\delta)$,
  \begin{equation}\label{eqn:deriv_1}
    \frac{f(u)-f(d)}{u-d}\in(y-\epsilon,y+\epsilon)
  \end{equation}
  This clearly implies that for all $\epsilon>0$ there exists $\delta>0$ such that for all
  $0<h<\delta$,
  \begin{equation}\label{eqn:deriv_2}
    \frac{f(x+h)-f(x-h)}{2h}\in(y-\epsilon,y+\epsilon)
  \end{equation}
  which is the symmetric definition of $f'(x)=y$.

  For the other direction, choose $\epsilon>0$; since $f'$ is continuous, there is some $\delta_1>0$
  such that $f'(x)-f'(x_1)\in(y-\epsilon/2,y+\epsilon/2)$ for all $x_1\in(x-\delta_1,x+\delta_1)$.
  Again using $\epsilon/2$, we obtain $\delta_2>0$ satisfying \cref{eqn:deriv_2} for all
  $0<h<\delta_2$. Let $\delta\coloneqq\min(\delta_1,\delta_2)$ and for any $d<u$ in
  $(x-\delta,x+\delta)$, let $x_0\coloneqq(u+d)/2$ be the midpoint and let $h_0\coloneqq(u-d)/2$ the
  radius, and we obtain \cref{eqn:deriv_1}.
\end{proof}

\begin{lemma}\label{lem:AD_semantics}\index{derivative!approximate}
  Let $f=(\ubar{f},\bar{f})$ and $g=(\ubar{g},\bar{g})$ be elements of
  $\church{\tIIub_{\pi}}(\ell)$. Then $\ell\Vdash\AD(f,g)$ if and only if for all
  $[d,u]\in\upclose[0,\ell]$ and all $r_1<r_2$ with $d\leq r_1<r_2\leq u$,
  \[
    \ubar{f}([d,u]) \leq \frac{\ubar{g}([r_2,r_2]) - \bar{g}([r_1,r_1])}{r_2-r_1}
    \qquad\text{and}\qquad
    \bar{f}([d,u]) \geq \frac{\bar{g}([r_2,r_2]) - \ubar{g}([r_1,r_1])}{r_2-r_1}.
  \]
\end{lemma}
\begin{proof}
  Choose $(d_t,u_t)\in\church{\Time}(\ell)$ and for $i=1,2$ let $(\ubar{g}_i,\bar{g}_i)\coloneqq
  g^@(r_i)$.%
  \footnote{
    Technically, the semantics of the term-in-context $x:\tIIub_\pi,y:\tRR\vdash
    x^@(y):\tIIub_{\AtInline{[r,r]}}$ is a sheaf homomorphism
    $\church{\tIIub_\pi\times\tRR}\to\church{\tIIub_{\AtInline{[r,r]}}}$, and we are denoting the
    image of $(g,r)$ under the $\ell$-component of that map.
  }
  Since it is $\At{[r,r]}$-local, \cref{thm:omnibus_semantics} (18) tells us that for any point
  $[d,u]\in\upclose[d_t,u_t]$, we have
  \[
    \ubar{g}_i([d,u])=
    \begin{cases}
      \ubar{g}_i([r_i,r_i])&\tn{ if }d\leq r_i\leq u\\
      -\infty&\tn{ otherwise}
    \end{cases}
    \quad\text{and}\quad
    \bar{g}_i([d,u])=
    \begin{cases}
      \bar{g}_i([r_i,r_i])&\tn{ if }d\leq r_i\leq u\\
      \infty&\tn{ otherwise}
    \end{cases}
  \]
  Letting $(\ubar{y},\bar{y})\coloneqq\frac{g^@(r_2)-g^@(r_1)}{r_2-r_1}$ and tracing through the
  Joyal-Kripke semantics for $\ell\Vdash\AD(f,g)$ (see \cref{def:derivative}), we find that it holds
  iff for all real numbers $r_1,r_2:\RR$ with $r_1<r_2$ and $q\in\QQ$, if $\ell\Vdash\ubar{f}q$ then
  $\ell\Vdash\ubar{y}q$ and if $\ell\Vdash\bar{f}q$ then $\ell\Vdash\bar{y}q$. This holds iff
  $\ubar{f}\leq\ubar{y}$ and $\bar{f}\geq\bar{y}$, as desired.
\end{proof}

\begin{proposition}\label{prop:AD_semantics}
  For any $f,g\in\church{\tIIub_{\pi}}(\ell)$, if $\AD(f,g)$ then for all $r$, if $0<r<\ell$ then
  $\ubar{f}([r,r])\leq\ubar{g}'(r)$ and $\bar{f}([r,r])\geq\bar{g}'(r)$.
\end{proposition}
\begin{proof}
  We show that $\ubar{f}([r,r])\leq\ubar{g}'(r)$; the second part is analogous. We want to show that
  for any $\epsilon>0$, there exists an open neighborhood $r\in U\subseteq(0,\ell)$ such that for
  all $d<u\in U$, $\ubar{f}([r,r])-\epsilon < \frac{\ubar{g}(u) - \bar{g}(d)}{u-d}$. Using that
  $\ubar{f}([r,r]) = \sup_{p\in\downclose[r,r]} \ubar{f}(p)$, there exists a point $p\ll[r,r]$
  such that $\ubar{f}([r,r])-\epsilon < \ubar{f}(p)$. By the assumption that
  $\AD(f,g)$ and \cref{lem:AD_semantics}, letting $U=\upclose p$, we have that for any $d<u\in U$,
  $\ubar{f}([r,r])-\epsilon < \ubar{f}(p) \leq \frac{\ubar{g}([u,u])-\bar{g}([d,d])}{u-d}$, as
  desired.
\end{proof}

The next lemma is a sort of mean-value theorem for upper/lower derivatives.

\begin{lemma}\label{lem:mean_value_thm}\index{Mean value theorem}
  Let $f\in\church{\tIRub_{\pi}}(\ell)$, i.e.\ $(\ubar{f},\bar{f})\in\church{\tIIub_{\pi}}(\ell)$
  with $\ubar{f}\leq\bar{f}$. Then for any $a,b$ with $0<a<b<\ell$, we have
  \[
    \inf_{r\in[a,b]}\ubar{f}'(r)\leq \frac{\ubar{f}(b)-\bar{f}(a)}{b-a}
    \qquad\text{and}\qquad
    \sup_{r\in[a,b]}\bar{f}'(r) \geq \frac{\bar{f}(b)-\ubar{f}(a)}{b-a}.
  \]
\end{lemma}
\begin{proof}
  Let $s=\inf_{r\in[a,b]}\ubar{f}'(r)$. Then for any $r\in[a,b]$, $s\leq\ubar{f}'(r)$. Fix an
  $\epsilon>0$. Then for any $r\in[a,b]$, there exists a neighborhood $r\in U\subseteq(0,\ell)$ such
  that for all $d<u\in U$, $s-\epsilon < \frac{\ubar{f}(u)-\bar{f}(d)}{u-d}$. These opens
  $\{U_j\}$ thus cover $[a,b]$, so by compactness there must exist a sequence
  $a=r_0<r_1<\dots<r_{n+1}=b$ such that for all $i$, $[r_i,r_{i+1}]\subseteq U_j$ for some $j$, and
  in particular $s-\epsilon < \frac{\ubar{f}(r_{i+1})-\bar{f}(r_i)}{r_{i+1}-r_{i}}$. Then we have
  that
  \[
    \frac{\ubar{f}(b)-\bar{f}(a)}{b-a}
    \geq \sum_{i=0}^n \frac{r_{i+1}-r_i}{b-a} \frac{\ubar{f}(r_{i+1})-\bar{f}(r_i)}{r_{i+1}-r_i}
      > \sum_{i=0}^n \frac{r_{i+1}-r_i}{b-a}(s-\epsilon) = s-\epsilon,
  \]
  where the first inequality makes use of $\ubar{f}(r_i)\leq \bar{f}(r_i)$ for $1\leq i\leq n$. As
  $\epsilon$ was arbitrary, we have $s\leq\frac{\ubar{f}(b)-\bar{f}(a)}{b-a}$.

  The proof of the second part is similar.
\end{proof}

\begin{proposition}\label{prop:AD_semantics_iff}\index{derivative!semantics of}
  For any $f,g\in\church{\tIRub_{\pi}}(\ell)$, $\AD(f,g)$ if and only if $\ubar{f}([r,r])\leq
  \ubar{g}'(r)$ and $\bar{f}([r,r])\geq\bar{g}'(r)$ for all $r\in(0,\ell)$.
\end{proposition}
\begin{proof}
  The forwards direction is \cref{prop:AD_semantics_iff}. For the converse, assume that for all
  $r\in(0,\ell)$, $\ubar{f}([r,r])\leq\ubar{g}'(r)$ and $\bar{f}([r,r])\geq\bar{g}'(r)$. By
  \cref{lem:AD_semantics}, it suffices to show that for any $0<d\leq r_1<r_2\leq u<\ell$,
  $\ubar{f}([d,u])\leq\frac{\ubar{g}([r_2,r_2])-\bar{g}([r_1,r_1])}{r_2-r_1}$ and
  $\bar{f}([d,u])\geq\frac{\bar{g}([r_2,r_2])-\ubar{g}([r_1,r_1])}{r_2-r_1}$.

  For the first, we have
  \[
    \ubar{f}([d,u]) = \inf_{r\in[d,u]}\ubar{f}([r,r]) \leq \inf_{r\in[d,u]} \ubar{g}'(r) \leq
    \frac{\ubar{f}(u)-\bar{f}(d)}{u-d},
  \]
  where we have used that $f$ is $\pi$-closed for the first equality, the first inequality is by
  assumption, and the second is \cref{lem:mean_value_thm}.
\end{proof}

In the following corollary, we use \cref{cor:pi_real_semantics} to identify
$\church{\tRR_\pi}(\ell)$ with the set of continuous functions $(0,\ell)\to\RR$.\index{real
number!variable}

\begin{corollary}\label{cor:usual_semantics}\index{differentiable}\index{derivative!semantics of}
  Suppose that $f\in\church{\tRR_\pi}(\ell)$ is internally continuously differentiable in the sense
  that its internal derivative (\cref{def:internal_derivative}) is also continuous
  $\dot{f}:\church{\tRR_\pi}(\ell)$. Then $\dot{f}$ is the externally the derivative of $f$.
\end{corollary}
\begin{proof}
  Let $g=\dot{f}$ as in \cref{def:internal_derivative}. Since $f,g\in\church{\tRR_\pi}(\ell)$, we
  have $\ubar{f}=\bar{f}$ and $\ubar{g}=\bar{g}$. By
  \cref{rem:external_disjointness,prop:AD_semantics_iff}, we have $\ubar{g}'=f=\bar{g}'$. The result
  then follows from \cref{prop:lower_upper_derivatives_equal}.
\end{proof}

\subsection{Differentiability}\label{sec:differentiablility}

We know by \cref{cor:usual_semantics} that the semantics of our notion of derivative matches the
usual definition. In this section we define what it means for a variable real to be differentiable
and we give internal proofs of certain well-known formulas: that time has unit derivative,
$\dot{t}=\frac{dt}{dt}=1$, that the Leibniz rule---often called the product rule---holds, and that
differentiation is linear.

\begin{definition}[Differentiable]\index{differentiable}
  Suppose that $x:\tRR_\pi$ is a variable real, so its derivative $\dot{x}$ is a $\pi$-local
  extended interval; see \cref{def:derivative}. We say $x$ is \emph{differentiable} if it is in fact
  a variable real $\dot{x}\in\tRR_\pi$, i.e.\ if $\dot{x}$ is $\pi$-bounded and $\pi$-located.
\end{definition}

\begin{proposition}
\label{prop:time_deriv_1}\index{derivative!of $t:\Time$}
  Any $t:\Time$ is differentiable, and its derivative is $\dot{t}=1$.
\end{proposition}
\begin{proof}
  By \cref{ax:time_complementary_cuts}, \cref{cor:derivatives_disjoint}, and\cref{lemma:located_spec_disjoint}, if
  $1\sqss\dot{t}$ then $1=\dot{t}$. Thus, by \cref{lem:sqss_slope_At_closed}, it suffices to show
  \[
    q<1\imp\See{[r_1,r_2]}\left(q<\frac{t^@(r_2)-t^@(r_1)}{r_2-r_1}\right)
  \]
  the upper cut being similar. Assuming $q<1$, it suffices to show
  \[
    \See{[r_1,r_2]}\exists(q_1,q_2:\tQQ)\ldotp
      (q*(r_2-r_1)<q_2-q_1) \wedge (q_2<t^@(r_2)) \wedge (t^@(r_1)<q_1).
  \]
  It is easy to check that $q*(r_2-r_1)<r_2-r_1$, e.g.\ using \cref{prop:arithmetic_located_j}. Thus
  by definition we have $\exists(q_1,q_2)\ldotp (q*(r_2-r_1)<q_2-q_1)\wedge(q_2<r_2)\wedge(r_1<q_1)$, and the
  result follows by \cref{ex:t@r}.
\end{proof}

\begin{proposition}\label{prop:diff_shortform}
  Suppose $x:\tRR_\pi$ is differentiable. Then for any $y:\tRR_\pi$, we have $y=\dot{x}$ iff
  $\AD(y,x)$, i.e.\
  \[
    \forall(r_1,r_2:\tRR)\ldotp r_1<r_2\imp y\sqss\frac{x^@(r_2)-x^@(r_1)}{r_2-r_1}
  \]
\end{proposition}
\begin{proof}
  By \cref{prop:AD_iff_sqss_dot}, we have $\AD(y,x)$ iff $y\sqss\dot{x}$. Because $y$ is
  $\pi$-located and $\dot{x}$ is $\pi$-disjoint, we conclude by \cref{lemma:located_spec_disjoint}.
\end{proof}

In the following theorem, all arithmetic operations take place in the $\pi$ modality, i.e.\ $x+y$ and $x*y$ mean $x+_\pi y$ and $x*_\pi y$, in the sense of \cref{sec:arithmetic}.

\begin{theorem}\label{th:leibniz}\index{Leibniz rule}\index{derivative!as satisfying Leibniz rule}
\index{product rule|see {Leibniz rule}}\index{derivative!as linear operation}
  Let $x,y:\tRR_\pi$. If $x$ and $y$ are differentiable then so is $x * y$, and the Leibniz rule
  holds:
  \[
    \ddt(x * y) = x * \dot{y} + \dot{x} * y.
  \]
  Moreover, the derivative is $\RR$-linear, $\ddt(r * x + y) = r * \dot{x} + \dot{y}$
  for any $r:\tRR$.
\end{theorem}
\begin{proof}
	As mentioned above the theorem, the arithmetic operations are taking place in the $\pi$-modality.
	So assume $x$ and $y$ are differentiable; we really need to prove that $x*_\pi y$ is differentiable, that $\ddt(x *_\pi y) = x *_\pi \dot{y} +_\pi \dot{x} *_\pi y$, and that $\ddt(r *_\pi x +_\pi y) = r *_\pi \dot{x} +_\pi \dot{y}$. We continue to leave the $\pi$-subscript off of arithmetic operations, for readability purposes.
	
  Since $\dot{x}$ and $\dot{y}$ are $\pi$-bounded and $\pi$-located, so are $x* \dot{y}+
  \dot{x}* y$ and $r* \dot{x}+ \dot{y}$. We begin by proving that the Leibniz rule holds.
  Since $\ddt(x* y)$ is $\pi$-disjoint by \cref{cor:derivatives_disjoint}, it suffices to show
  that $(x*\dot{y})+ (\dot{x}* y)\sqss \ddt(x* y)$, by
  \cref{lemma:located_spec_disjoint}.

  Because $\ctII_\pi\subseteq\tII_\pi$ is a basis (see \cref{rem:basis_for_our_domains}), it
  suffices to show that for any $c:\ctII$, if $c\ll (x* \dot{y}) + (\dot{x}* y)$ then
  $c\sqss \ddt(x* y)$. Let $c:\ctII$, and suppose $c\ll (x*\dot{y}) + (\dot{x}* y)$.
  Then
  \begin{equation}\label{eq:leibniz}
    \pi\exists(c_1,c_2,c_3,c_4:\ctII)\ldotp (c_1\sqss x)\wedge(c_2\sqss\dot{y})
      \wedge(c_3\sqss\dot{x})\wedge(c_4\sqss y)\wedge(c\ll c_1*c_2+c_3*c_4).
  \end{equation}

  Ignoring the preceding $\pi$ for now, suppose we have $c_1,c_2,c_3,c_4:\ctII$ such that $c_1\sqss x$,
  $c_2\sqss \dot{y}$, $c_3\sqss\dot{x}$, $c_4\sqss y$, and $c\ll c_1*c_2+c_3*c_4$. We want to show
  that $c\sqss\ddt(x* y)$, or equivalently $\AD(c,x* y)$. So letting $r_1<r_2$ be arbitrary
  reals, we want to show that $c\sqss\frac{x^@(r_2)*y^@(r_2)-x^@(r_1)*y^@(r_1)}{r_2-r_1}$, where we
  have used that the maps $\At{r_1}:\tII_\pi\to\tII_{\AtInline{r_1}}$ and $\At{r_2}:\tII_\pi\to\tII_{\AtInline{r_2}}$
  preserve multiplication; see \cref{prop:j_on_intervals}.

  From $c_2\sqss\dot{y}$, we have $c_2\sqss\frac{y^@(r_2)-y^@(r_1)}{r_2-r_1}$ hence $c_1*c_2\sqss
  c_1*\frac{y^@(r_2)-y^@(r_1)}{r_2-r_1}\sqss x^@(r_1)\frac{y^@(r_2)-y^@(r_1)}{r_2-r_1}$. Similarly,
  $c_3*c_4\sqss\frac{x^@(r_2)-x^@(r_1)}{r_2-r_1}y^@(r_2)$. Adding these together and simplifying
  yields $c\ll c_1*c_2+c_3*c_4\sqss \frac{x^@(r_2)*y^@(r_2)-x^@(r_1)*y^@(r_1)}{r_2-r_1}$, as
  desired.

  Thus we have shown, assuming the existence of $c_1,\dots,c_4$, that $\AD(c,x* y)$ holds. Hence
  \eqref{eq:leibniz} implies $\pi\AD(c,x* y)$, which by \cref{prop:AD_Pi-closed} implies
  $\AD(c,x* y)$. We therefore have shown that $c\ll x* \dot{y} + \dot{x}* y$ implies
  $\AD(c,x* y)$, and hence $c\sqss\ddt(x* y)$, as desired.

  Showing linearity is similar: $\ddt(r*x+y)$ is again $\pi$-disjoint, so it suffices to show that
  $r*\dot{x}+\dot{y}\sqss\ddt(r*x+y)$. This follows as above, using the fact that localizing is
  linear, i.e.\ $(r*x)^@(r')=r*x^@(r')$; see \cref{prop:j_on_intervals}.
\end{proof}

\chapter{Applications}\label{sec:applications}\index{applications|(}

As mentioned in the introduction, we believe that our temporal type theory can serve as a ``big tent'' into which to embed many disparate formalisms for proving property of behaviors. In this chapter we discuss a few of these, including hybrid
dynamical systems in \cref{sec:hybrid_sheaves}\index{dynamical system!hybrid}, delays in \cref{sec:delay}\index{delay}, differential equations in \cref{sec:ODEs}\index{differential equation}, and labeled transition systems in \cref{sec:LTS}\index{labeled
transition system}. This last occurs in \cref{sec:systems_and_machines} where we briefly discuss a general framework on machines, systems, and behavior contracts. Next in \cref{sec:case_study} we give a toy example---an extreme simplification of the National Airspace System---and prove a
safety property. The idea is to show that we really can mix continuous, discrete, and delay properties without ever leaving the temporal type theory. We conclude this chapter by discussing how some temporal logics, e.g.\ metric temporal logic, embeds into the temporal type theory.

\section{Hybrid sheaves}\label{sec:hybrid_sheaves}

The NAS is a hybrid system: it includes both continuous aspects, like the differential equations
governing airplane motion, and discrete aspects, like the signals and messages that alert pilots to
possible danger.\index{National Airspace System} We begin by explaining how one might integrate
continuous and discrete behavior into a single sheaf, which we call a ``hybrid
sheaf''\index{sheaf!hybrid} or ``hybrid type''\index{type!hybrid}, depending if we are working
externally or internally.\index{hybrid|seealso {sheaf, hybrid}}\index{hybrid|seealso {type, hybrid}}

Our method of specifying (and constructing) a hybrid type takes an arbitrary behavior type---a sheaf
$C$---which we think of as a type of
\emph{continuous} behaviors, and produces a new behavior type which extends $C$ with certain allowed
discontinuities or discrete transitions, coming from a type $D$.

\subsection{Constructing hybrid sheaves}

Recall the pointwise modality $\pi:\Prop\to\Prop$ from \cref{not:our_modalities} and the
$\pi$-sheafification functor $\asSh_\pi$ from \cref{def:separated_sheaf}. Recall also the notion of
sum types and quotient types from
\cref{sec:toposes_types_logic}.\index{type!quotient}\index{sheafification}\index{type!sum}

\begin{definition}[Hybrid type]\label{def:hybrid_sheaf}\index{type!hybrid}\index{hybrid datum}
  A \emph{hybrid datum} is a tuple $(C,D,\fun{src},\fun{tgt},\tau)$, where $C$ and $D$ are arbitrary
  types, and where $\fun{src},\fun{tgt}:D\to C$ and $\tau:D\to\Time$ are functions. Consider the
  quotient type $\hybrid_1(C,D)\coloneqq (C+D)/\sim$, where for $d:D$ we put
  \begin{equation}\label{eq:hybrid_sheaf}
    d\sim\fun{src}(d) \text{ if $\tau(d)<0$}, \qquad\text{and}\qquad
    d\sim\fun{tgt}(d) \text{ if $\tau(d)>0$}.
  \end{equation}
  The \emph{hybrid behavior type} on the hybrid datum $(C,D,\fun{src},\fun{tgt},\tau)$ is defined to
  be the $\pi$-sheafification of this quotient, which we denote $\hybrid(C,D)\coloneqq\asSh_\pi\hybrid_1(C,D)$.
  There are natural maps $C\to\hybrid(C,D)$ and $D\to\hybrid(C,D)$; we refer to the image of $C$ as
  the \emph{continuous part}.\index{type!hybrid!continuous part of}
\end{definition}

Intuitively, a section $d$ of the sheaf $D$ represents a behavior which has
a discontinuity at the point $\fun{\tau}(d)=0$. We want to adjoin it to $C$, by connecting the continuous behavior $\fun{src}(d)$ on
the left of the discontinuity to the continuous behavior $\fun{tgt}(d)$ on the right. When we
sheafify, we allow continuous behaviors that have been interrupted by finitely many of these
discontinuity events.

The following is straightforward.

\begin{lemma}\label{lemma:equiv_on_hybrid}
  In the notation of \cref{def:hybrid_sheaf}, the equivalence relation on $C+D$ generated by
  \eqref{eq:hybrid_sheaf} is as follows for $c, c':C$ and $d,d':D$,
  \begin{align*}
    c\sim c'\iff\ &c=c'\\
    c\sim d\iff\ &(\tau(d)<0\wedge c=\fun{src}(d))\vee(\tau(d)>0\wedge c=\fun{tgt}(d))\\
    d\sim c\iff\ &(\tau(d)<0\wedge c=\fun{src}(d))\vee(\tau(d)>0\wedge c=\fun{tgt}(d))\\
    d\sim d'\iff\ &\big(
      d=d' \\
      &\ \vee(\fun{src}(d)=\fun{src}(d')\wedge (\tau(d)<0)\wedge(\tau(d')<0))\\
      &\ \vee(\fun{tgt}(d)=\fun{tgt}(d')\wedge (\tau(d)>0)\wedge(\tau(d')>0))\\
      &\ \vee(\fun{tgt}(d)=\fun{src}(d')\wedge \tau(d)<0<\tau(d'))\\
      &\ \vee(\fun{src}(d)=\fun{tgt}(d')\wedge \tau(d')<0<\tau(d))
    \big)
  \end{align*}
\end{lemma}

\begin{proposition}\label{prop:hybrid1_pi_sep}
Let $(C,D,\fun{src},\fun{tgt},\tau)$ be a hybrid sheaf datum, and suppose that $C$ and $D$ have decidable equality. Then the type $\hybrid_1(C,D)$ from \cref{def:hybrid_sheaf} is $\pi$-separated.
\end{proposition}
\begin{proof}
Let $x,x':C+D$; we need to show that $\pi(x\sim x')$ implies $x\sim x'$. We will work by cases on the components of $x$ and $x'$, using \cref{lemma:equiv_on_hybrid}. If $x,x'\in C$ then $\pi(x=x')$, so $x=x'$ by \cref{cor:dec_eq_pi_sep}. For the remaining cases we will apply \cref{lemma:pi_t<0_0<t}. For example, it directly implies that if $x\in C$ and $x'\in D$ (or vice-versa), then $x\sim x'$ is $\pi$-closed. Finally, in the case $x,x'\in D$, one uses \cref{prop:dec_eq_disjunction} and then repeats the above argument twice more:
\small
\begin{align*}
	&\pi(x\sim x')\\
	\imp&\pi[(x=x')\vee(P\wedge t<0\wedge t'<0)\vee(Q\wedge 0<t\wedge 0<t'\})\vee(R\wedge t<0<t')\vee(S\wedge t'<0<t)]\\
	\imp&(x=x')\vee\pi[(t<0\wedge\pi((P\wedge t'<0)\vee(R\wedge 0<t')))\vee(0<t\wedge\pi((Q\wedge0<t')\vee(S\wedge t'<0)))]\\
	\imp&(x=x')\vee(t<0\wedge\pi((P\wedge t'<0)\vee(R\wedge 0<t')))\vee(0<t\wedge\pi((Q\wedge0<t')\vee(S\wedge t'<0)))\\
	\imp&(x=x')\vee(t<0\wedge(P\wedge t'<0)\vee(R\wedge 0<t'))\vee(0<t\wedge(Q\wedge0<t')\vee(S\wedge t'<0))\\
	\iff&(x\sim x'),
\end{align*}\normalsize
where $t\coloneqq\tau(x)$, $t'\coloneqq\tau(x')$, and where $P\coloneqq(\fun{src}(x)=\fun{src}(x'))$, $Q\coloneqq(\fun{tgt}(x)=\fun{tgt}(x'))$, $R\coloneqq(\fun{tgt}(x)=\fun{src}(x'))$, and $S\coloneqq(\fun{src}(x)=\fun{tgt}(x'))$, all of which are $\pi$-closed.
\end{proof}

\begin{remark}\label{rem:info_in_hybrid}
  The only information in a section $d$ of $D$ which is relevant to this construction is the
  restriction of $d$ to arbitrarily small neighborhoods of the point $\tau(d)=0$. Externally, we can
  make this precise as follows: consider the stalk functor\index{stalk}
  $\Fun{stk_0}\colon\BaseTopos\to\Cat{Set}$ sending a sheaf to the set of zero-length germs. It can
  be obtained by composing the geometric morphism
  $\Fun{stk_{@0}}\colon\Cat{Set}\to\Shv{S_{\IR}}\iso\BaseTopos/\Time$, corresponding to the point
  $[0,0]\in\IR$, with the geometric morphism $\BaseTopos/\Time\to\BaseTopos$; see also
  \cref{prop:quotient_geom_morphism}. Given a sheaf over $\Time$, $\tau\colon D\to\Time$, the
  following square is a pullback
  \[
    \begin{tikzcd}
      \Fun{stk_{@0}}(D) \ar[r] \ar[d] \ar[dr,phantom,"\lrcorner" very near start] & 1 \ar[d,"0"] \\
      \Fun{stk_0}(D) \ar[r,"\Fun{stk_0}(\tau)"'] & \Fun{stk_0}(\Time)
    \end{tikzcd}
  \]
  Then any hybrid datum $(C,D,\fun{src},\fun{tgt},\tau)$ determines a pair of functions
  $\Fun{stk_{@0}}(D)\rightrightarrows\Fun{stk_0}(C)$. It also determines a function
  $\Fun{stk_{@0}}(D)\to\Fun{stk_0}D\to\Fun{stk_0}\hybrid(C,D)$, and the image of this function might
  be called the set of \emph{discontinuities} in the hybrid type.\index{sheaf!hybrid!discontinuities
  of}

  The hybrid behavior type construction really only depends on the sheaf $C$, the set
  $S=\Fun{stk_{@0}}(D)$, and the source and target functions $S\rightrightarrows\Fun{stk_0}(C)$.
  That is, if $(C,D',\fun{src'},\fun{tgt'},\tau')$ is a different hybrid datum on the same
  continuous part $C$, then to every bijection $\Fun{stk_{@0}}(D)\iso\Fun{stk_{@0}}(D')$ commuting
  with the source and target functions, there is a canonical isomorphism
  $\hybrid(C,D)\iso\hybrid(C,D')$. Conversely, given any set $S$ and pair of functions
  $S\to\Fun{stk_0}(C)$, there is a hybrid datum $(C,D,\fun{src},\fun{tgt},\tau)$ realizing it.

  In particular, if $f\colon D\to D'$ is a morphism commuting with $\fun{src}$, $\fun{src'}$,
  $\fun{tgt}$, $\fun{tgt'}$, $\tau$, and $\tau'$, then $f$ induces an isomorphism
  $\hybrid(C,D)\iso\hybrid(C,D')$ if and only if $\Fun{stk_{@0}}(f)$ is a bijection.

  The case when $f$ is a monomorphism is relatively easy to capture internally. Let $\phi\colon
  D\to\Prop$ be the predicate characterizing $f\colon D'\inj D$. Then $f$ induces a bijection on
  stalks if and only if the proposition
  \[
    \forall(d:D)\ldotp\At[\tau(d)]{0}\phi(d)
  \]
  holds internally. This simply says that for all sections $d\in D(\ell)$, there is a neighborhood
  of $\tau(d)=0$ on which the restriction of $d$ is in $D'$, i.e.\ the stalk of $d$ at $\tau(d)=0$
  is in the image of $\Fun{stk_0}(f)$. Then it is easy to see that the induced map
  $(C+D')/\sim\;\to(C+D)/\sim$ will again be mono, characterized by the predicate
  $\phi'\colon(C+D)/\sim\;\to\Prop$ which is $\top$ on $C$ and $\See[\tau(d)]{0}\phi(d)$ on $d\in
  D$. This predicate $\phi'$ is $\pi$-\emph{dense}, meaning $\pi\circ\phi'=\top$, since
  $\pi\top\iff\top$ and, by \cref{prop:points_in_subtoposes}, $\pi(\See[\tau(d)]{0}\phi(d)) \iff
  \At[\tau(d)]{0}\phi(d) \iff \top$ for all $d\in D$. Therefore the mono
  $(C+D')/\sim\;\inj(C+D)/\sim$ characterized by $\phi'$ induces an isomorphism between the
  sheafifications $\hybrid(C,D')\iso\hybrid(C,D)$.
\end{remark}

Later, we will need to know that a section of a hybrid sheaf\index{sheaf!hybrid} is almost always in
the continuous part, $C$. Let $\const{is\_cts}:\hybrid(C,D)\to\Prop$ classify the image of the map
$C\to\hybrid(C,D)$.

\begin{proposition}\label{prop:lts_aa_vertex}\index{modality!double negation}
  Let $\hybrid(C,D)$ be the hybrid type on the data $(C,D,\fun{src},\fun{tgt},\tau)$. A section of
  $\hybrid(C,D)$ is almost-always in the continuous part, i.e.
  \[
    \forall(h:\hybrid(C,D))\ldotp\neg\neg\const{is\_cts }(h).
  \]
\end{proposition}
\begin{proof}
  Choose $h:\hybrid(C,D)$; we first spell out its definition and that of $\const{is\_cts}(h)$. Let
  $S=C+D/\sim$. By \cref{def:separated_sheaf}, $h:S\to\Prop$ is a predicate satisfying
  \[
    \forall(s:S)\ldotp \pi h(s)\imp h(s)
    \qquad\text{and}\qquad
    \pi\exists(s_0:S)\ldotp\forall(s:S)\ldotp h(s)\iff \pi(s_0=s).
  \]
  Unwinding definitions and using \cref{eqn:sheafification_inclusion}, we see that
  $\const{is\_cts}(h)$ is equivalent to the proposition
  \[\exists(c:C)\ldotp\forall(s:S)\ldotp h(s)\iff\pi(s=c).\]

  Assume $\neg\const{is\_cts}(h)$; we will show $\pi\bot$, which is enough by \cref{cor:pi_negneg}.
  By \cref{rem:j-logic} we may choose $s_0:S$ such that $\forall(s:S)\ldotp h(s)\iff \pi(s_0=s)$. If
  $s_0\sim c$ for some $c:C$, we obtain the desired contradiction, so we may assume $s_0:D$. By
  \cref{lemma:equiv_on_hybrid}, we obtain a contradiction if $\tau(d)<0$ or $\tau(d)>0$, so we have
  $\neg(\tau(d)<0)\wedge\neg(\tau(d)>0)$. But this is a contradiction by
  \cref{ax:time_complementary_cuts}.
\end{proof}

\begin{definition}\label{def:morphism_of_hybrid_data}
Let $(C,D,\fun{src},\fun{tgt},\tau)$ and $(C',D',\fun{src}',\fun{tgt}',\tau')$ be hybrid data as in \cref{def:hybrid_sheaf}. A \emph{morphism} between them consists of a function $f\colon C\to C'$ and a function $g\colon D\to D'$, commuting with the rest of the structure, i.e.\ $\fun{src}'\circ g=f\circ \fun{src}$ and $\fun{tgt}'\circ g=f\circ \fun{tgt}$ and $\tau'\circ g=\tau$. This defines a category of hybrid data.
\end{definition}

\begin{proposition}\label{prop:Hyb_functorial}
A morphism of hybrid data induces a morphism of behavior types $\hybrid(C,D)\to\hybrid(C',D')$. In fact, $\hybrid$ is a functor.
\end{proposition}
\begin{proof}
Suppose given $f\colon C\to C'$ and $g\colon D\to D'$, such that $\fun{src}\circ g=f\circ \fun{src}'$ and $\fun{tgt}\circ g=f\circ \fun{tgt}'$ and $\tau'\circ g=\tau$. There is an induced map $C+D\to C'+D'$, and this map respects the equivalence relation in \cref{eq:hybrid_sheaf} because if $\tau(d)<0$ then $\tau'(g(d))<0$, in which case $d\sim\fun{src}(d)$ implies $g(d)\sim\fun{src}'(g(d))$. Thus we obtain a map between their quotients $S\to S'$, and hence between their $\pi$-sheafifications $\asSh_\pi S\to\asSh_\pi S'$. All of the induced maps along the way have been functorial.
\end{proof}

\subsection{Special case: walks through a graph}\label{sec:walks}\index{graph!walks through}

Suppose given a graph $G$ (in $\Cat{Set}$)
\begin{equation}\label{eqn:graph}
\begin{tikzcd}
  E\ar[r,shift left, "d"]\ar[r, shift right, "c"']&V
\end{tikzcd}
\end{equation}
where $d$ stands for ``domain'' and $c$ stands for ``codomain''. Define a \emph{walk of length
$\ell$ through $G$} to be a finite path in $G$, together with a positive duration of time assigned
to each vertex in the path, such that the total of these durations is $\ell$. In other words, we
imagine that each edge is traversed instantaneously.

Considering $E$ and $V$ as the corresponding constant types (see \cref{sec:const_objects_basetopos}), we can build the sheaf of walks $W_G$
in $G$ as a hybrid type using the following as a hybrid sheaf datum (in the sense of \cref{def:hybrid_sheaf}):
\begin{equation}\label{eqn:HSD_from_graph}
\begin{tikzcd}[column sep=large]
  \Time
  & E\times\Time \ar[l,"p_\Time"'] \ar[r,shift left, "d\;\circ\; p_E"]
    \ar[r, shift right, "c\;\circ\; p_E"']
  & V.
\end{tikzcd}
\end{equation}
Here $\Time\From{p_\Time}E\times\Time\To{p_E}E$ are the projections. Unravelling \cref{def:hybrid_sheaf} in this case, we find
that the hybrid behavior type $\hybrid(V,E\times\Time)$ is the $\pi$-sheafification of
the quotient $W'_G\coloneqq\hybrid_1(V,E\times\Time)=(V+(E\times\Time))/\sim$, where
\begin{equation}\label{eqn:graph_walk}
  (e,t) \sim d(e)\; \text{ if }t<0
  \qquad\text{and}\qquad
  (e,t) \sim c(e)\; \text{ if }t>0
\end{equation}
Semantically, the sections of $W'_G$ are walks with at most one transition.

\begin{corollary}\label{cor:W_G'_pi_sep}
The type $W_G'$ is $\pi$-separated.
\end{corollary}
\begin{proof}
This follows from \cref{prop:hybrid1_pi_sep} because constant types and $\Time$---and hence $V$ and $E\times\Time$---have decidable equality; see \cref{prop:time_decidable_eq,prop:const_decidable_eq}.
\end{proof}

We can now further unravel the definition of hybrid sheaf in this case, using \cref{def:separated_sheaf,cor:W_G'_pi_sep}. The $\pi$-sheafification $W_G\coloneqq\asSh_\pi(W'_G)=\hybrid(V,E\times\Time)$ is the type of functions $\phi:(V+E\times\Time)\to\Prop$ satisfying
\begin{enumerate}
	\item $\forall(t:\Time)(e:E)\ldotp t<0\imp(\phi(e,t)\iff\phi(d(e)))$,
	\item $\forall(t:\Time)(e:E)\ldotp 0<t\imp(\phi(e,t)\iff\phi(c(e)))$,
	\item $\forall(g:V+E\times\Time)\ldotp\pi(\phi g)\imp\phi g$, and
	\item $\pi\exists(g:V+E\times\Time)\ldotp\phi=\{g\}$
\end{enumerate}
where, as usual, $\{g\}:(V+E\times\Time)\to\Prop$ is the predicate sending $g'$ to $(g=g'):\Prop$. Semantically, $W_G$ is the sheaf of walks through $G$ with finitely-many transitions over any (finite-length) interval $\ell$.

\begin{example}
It is often useful to know which edge transitions occur in a walk through a graph $G=(E\rightrightarrows V)$.%
\footnote{The construction of $\const{trav}$ described here extends to hybrid sheaves in general; it is only a bit more complicated in the general case because we cannot assume $\hybrid_1(C,D)$ is $\pi$-separated.}
Given a section $\phi:W_G$ and a time $t:\Time$, consider the predicate
\[\const{trav}_\phi:(\tRR\times E)\to\Prop,\qquad \const{trav}_\phi(r,e)\coloneqq\At[t]{[r,r]}\big(\phi=\{(e,t-r)\}\big).\]
Semantically, this proposition is true on an interval $[d,u]$ if either $r$ is outside the interval, $(r<d)\vee(u<r)$, or if at time $r$ the walk $\phi$ is traversing edge $e$.
\end{example}

\begin{example}\index{behavior!pulsing}
Consider the terminal graph $T\coloneqq\fbox{\begin{tikzcd}\bullet\ar[loop above, looseness=6, start anchor={[xshift=-1pt, yshift=-2pt]}, end anchor={[xshift=1pt, yshift=-2pt]}]\end{tikzcd}}$. The corresponding hybrid datum $\Time=\Time\rightrightarrows 1$ is also the terminal object in the category of hybrid data (see \cref{def:morphism_of_hybrid_data}).

Walks through the terminal graph $W_T$ gives a model for simple pulsing. The only information in a section of $W_T$ is the finite set of instants at which pulses occur, i.e.\ at which the unique edge is traversed. This could be used to model musical or neuron-spiking behavior.
\end{example}

\begin{proposition}\label{prop:RW_functorial}
The above construction, taking a graph $G$ and returning a behavior type $W_G$, is functorial.
\end{proposition}
\begin{proof}
We need to show that given a morphism of graphs, one obtains a morphism of behavior types, respecting identities and composition. We know by \cref{prop:Hyb_functorial} that $\hybrid$ is functorial, so it suffices to see that our construction of the hybrid sheaf datum in \cref{eqn:HSD_from_graph} was functorial. But this is obvious: given $f\colon E\to E'$ and $g\colon V\to V'$, it is straightforward to check that the following diagram commutes:
\[
\begin{tikzcd}[sep=large]
 \Time\ar[d,equal]
  & E\times\Time \ar[l,"p_\Time"'] \ar[r,shift left, "d\;\circ\; p_E"]\ar[d, "f\times\Time"']
    \ar[r, shift right, "c\;\circ\; p_E"']
  & V\ar[d, "g"]\\
  \Time
  & E'\times\Time \ar[l,"p_\Time"] \ar[r,shift left, "d\;\circ\; p_E"]
    \ar[r, shift right, "c\;\circ\; p_E"']
  & V'
\end{tikzcd}\qedhere
\]
\end{proof}

\subsubsection{Timed random walks}\label{sec:timed_random_walks}\index{graph!timed walks through}

Fix $0\leq a\leq b$; suppose we want random walks through $G$ whose time between transitions is
required to be between $a$ and $b$. We can also construct such a sheaf internally using the hybrid
type construction. Let $G$ be as in \cref{eqn:graph}, and use the following hybrid datum:
\[
\begin{tikzcd}
  E\times\{(r,t):\tRR\times\Time\mid(a<r<b)\wedge(-r<t<b-r)\}
  \ar[r, shift left,"\fun{src}"]	\ar[r, shift right,"\fun{tgt}"']\ar[d, "\tau"']&
  V\times\{t:\Time\mid 0<t<b\}\\
  \Time
\end{tikzcd}
\]
where $\tau(e,(r,t))\coloneqq t$, $\fun{src}(e,(r,t))\coloneqq(\fun{src}(e),t+r)$, and
$\fun{tgt}(e,(r,t))=(\fun{tgt}(e),t)$. The associated hybrid type will have the desired behavior.

\section{Delays}\label{sec:delay}\index{delay}

Intuitively, a machine acts as a delay of type $A$ if it has type $A\times A$ and has the property
that its second projection produces the same section as its first projection, except translated in
time by $d$ seconds, for some real $0<d:\tRR$. This notion of ``same'' belies an implicit assumption
that sections $a_1,a_2:A$ can be compared for equality at different intervals of time. However, this
is not the semantics of equality: only sections over the same interval can be compared for equality.

Instead of comparing for equality, one could start with a morphism $\phi\colon\tConst\to\Prop^A$,
where $\tConst$ is a constant sheaf. Then for every $c:\tConst$, the predicate $\phi(c)$ is a
translation-invariant predicate on $A$, so we may apply it at different times and compare the
results. We can thus speak of $\phi$-theoretic delays. Namely, $a_2$ is the $\phi$-theoretic $d$
delay of $a_1$ if, for all $c:\tConst$ and $r:\tRR$, the section $a_1$ satisfies $\phi(c)$ on the
open set $\upclose[0,r]$ iff $a_2$ satisfies $\phi(c)$ on the translated open set $\upclose[d,d+r]$.

\begin{definition}[Delay]\label{def:delay}
  Let $A$ be a type, and suppose given a constant sheaf $\tConst$, a predicate $\phi:(C\times
  A)\to\Prop$, and a number $D:\tRR$ with $D>0$. Fix $t:\Time$. We define $(a,a'):A\times A$ to be
  \emph{$\phi$-theoretically $D$-delayed} if the following holds:
  \begin{equation}\label{eqn:delay}
    \forall(c:\tConst)(d,u:\tRR)\ldotp (d<u) \imp
      \left( \See[t]{[d,u+D]}\In[t]{[d,u]}\phi(c,a) \iff
        \See[t]{[d,u+D]}\In[t]{[d+D,u+D]}\phi(c,a') \right).
  \end{equation}
  Say that \emph{$\tConst$ separates $A$ (via $\phi$)}\index{separation!by constants}, if the
  following formula holds:
  \[
    \forall(a_1,a_2:A)\ldotp(\forall(c:\tConst)\ldotp\phi(c,a_1)\iff\phi(c,a_2))\imp(a_1=a_2).
  \]
  If $\tConst$ separates $A$, we may drop mention of $\phi$ and say that $(a,a')$ is
  \emph{$D$-delayed} iff it satisfies \cref{eqn:delay}.
\end{definition}

By the torsor axiom \ref{ax:torsor}, being $D$-delayed is not dependent on the choice of $t:\Time$.
One particularly intuitive case of delays is for numeric types\index{numeric type}.

\begin{lemma}\label{lemma:numeric_delay}
  Let $\tRR_j$ be the type of $j$-local real numbers, for a modality $j$. Then $\tQQ\times\tQQ$
  separates $R$, via the predicate $\phi(q_1,q_2,r)\coloneqq(q_1<r<q_2)$. A similar statement holds
  when $\tRR_j$ is replaced by any of the other $j$-numeric types.
\end{lemma}
\begin{proof}
  For any pair of lower reals, $\delta_1,\delta_2:\tLR_j$, it is obvious by propositional
  extensionality\index{propositional extensionality} (see \cref{ax:propositional_extensionality}) that
  $\forall(q:\tQQ)\ldotp(\delta_1q\iff\delta_2q))\imp(\delta_1=\delta_2)$. The same holds for the
  other numeric types.
\end{proof}

\begin{example}\label{ex:delay_numeric}\index{numeric type!delay of}
  Let's work out an example of the semantics of delays for the case $\tRR_\pi$, because its
  semantics are most familiar; the other numeric types are similar. Choose
  $D\in\RR=\church{\tRR}(\ell)$ with $D>0$ and $t=(d_t,u_t)\in\church{\Time}(\ell)$. Recall from
  \cref{cor:pi_real_semantics} (or more generally \cref{thm:omnibus_semantics}) that elements
  $f,f'\in\church{\tRR_\pi}(\ell)$ can be identified with continuous real-valued functions
  $f,f'\colon(d_t,u_t)\to\RR$.

  We want to understand the semantics of the predicate $\phi$ defined in \cref{lemma:numeric_delay}
  saying that $(f,f')$ are $D$-delayed:
  \begin{multline}\label{eqn:delay_variable_real}
    \ell\Vdash\forall(q_1,q_2:\tQQ)(d,u:\tRR)\ldotp(d<u)\imp\\
    \left(\See{[d,u+D]}\In{[d,u]}\left(q_1<f<q_2\right)\iff\See{[d,u+D]}\In{[d+D,u+D]}\left(q_1<f'<q_2\right)\right).
  \end{multline}
  Our \emph{goal} is to show that \cref{eqn:delay_variable_real} holds iff $f(x)=f'(x+D)$ for all
  $d_t<x<u_t-D$.

  \Cref{eqn:delay_variable_real} holds iff for every $q_1,q_2\in\QQ$ and $d<u$ in $\RR$, we have
  \[
    \ell\Vdash \left(\See{[d,u+D]}\In{[d,u]}\left(q_1<f<q_2\right)\right)
      \iff\left( \See{[d,u+D]}\In{[d+D,u+D]}\left(q_1<f'<q_2\right)\right).
  \]
  One can show that this is the case iff we have
  \[
    \ell\Vdash\See[(d_t,u_t)]{[d,u+D]}\In[(d_t,u_t)]{[d,u]}\left(q_1<\restrict{f}{(d_t,u_t)}<q_2\right)
    \quad\text{iff}\quad
    \ell\Vdash\See[(d_t,u_t)]{[d,u+D]}\In[(d_t,u_t)]{[d+D,u+D]}\left(q_1<\restrict{f'}{(d_t,u_t)}<q_2\right).
  \]
  By the semantics of $\See{}$ (\cref{cor:semantics_In_At_See}), we see that if $d<d_t$ or $u_t<u+D$
  then the above equivalence is vacuously satisfied. In other words, if $u_t-D\leq d_t$ then we already have achieved our
  goal. Thus we may suppose that $d_t\leq d< u\leq u_t-D$, and \cref{eqn:delay_variable_real} has
  become
  \[
    \ell\Vdash\In[(d_t,u_t)]{[d,u]}\left(q_1<\restrict{f}{(d_t,u_t)}<q_2\right)
    \qquad\text{iff}\qquad
    \ell\Vdash\In[(d_t,u_t)]{[d+D,u+D]}\left(q_1<\restrict{f'}{(d_t,u_t)}<q_2\right).
  \]
  The left-hand side is equivalent to the following: for all $\subint{r}{s}\colon\ell'\to\ell$, if
  $d<d_t+r<u_t-s<u$ then $q_1<\restrict{f}{(d_t+r,u_t-s)}<q_2$. Similarly for the right, except $d$ and $u$ are replaced by $d+D$ and $u+D$. Thus \cref{eqn:delay_variable_real}
  has become
  \[
    \forall(d,u\in\RR)\ldotp (d_t\leq d<u\leq u_t-D) \imp 
    \left[
      q_1<\restrict{f}{(d,u)}<q_2
      \quad\text{iff}\quad
      q_1<\restrict{f'}{(d+D,u+D)}<q_2
    \right].
  \]
  We are trying to show that this holds iff $f(x)=f'(x+D)$ for all $d_t<x<u_t-D$, and that follows
  from the continuity of $f$ and $f'$.
\end{example}

\section{Ordinary differential equations, relations, and inclusions}\label{sec:ODEs}\index{differential equation}

Consider a system of ordinary differential equations
\begin{equation}\label{eqn:system_ODE}
\begin{aligned}
	\dot{x}_1&=f_1(x_1,\ldots,x_n,a_1,\ldots,a_m)\\
	\dot{x}_2&=f_2(x_1,\ldots,x_n,a_1,\ldots,a_m)\\
	         &\;\;\vdots\\
	\dot{x}_n&=f_n(x_1,\ldots,x_n,a_1,\ldots,a_m)
\end{aligned}
\end{equation}
In this setting, one generally considers each variable as a continuous function of time and the
$x_i$ as differentiable. Thus we may consider each variable as a $\pi$-local real, $x_i:\tRR_\pi$,
$a_j:\tRR_\pi$, the semantics of which are continuous curves in $\RR$; see
\cref{cor:pi_real_semantics}. Our internal notion of differentiability is defined in
\cref{def:internal_derivative} and its semantics is indeed differentiability in the usual sense; see
\cref{cor:usual_semantics}.

Thus we may interpret \cref{eqn:system_ODE} entirely within the internal language of the topos
$\BaseTopos$. In fact, such systems can be generalized a great deal. For one thing, we can replace
the system with a differential relation or a differential inclusion. Indeed, a differential
relation\index{differential relation} is a predicate $P:(\tRR_\pi)^{2n+m}\to\Prop$, whose solution
set is
\[
  \{(x,a):(\tRR_\pi)^n\times(\tRR_\pi)^m\mid P(x,\dot{x},a)\}.
\]
Requiring derivatives to remain within certain bounds is a common example of this notion, and---in
case it was not clear---so is \cref{eqn:system_ODE}, when $P(x,x',a)$ is taken to be the proposition
$x'=f(x,a):\Prop$.

In fact, the functions $x_i$, $a_i$, and $f$ can even represent interval-valued functions. Because
our internal definition of derivative is quite broadly defined, \cref{eqn:system_ODE} will continue
to make sense. It could also be generalized to something like $f(x,a)\sqss\dot{x}$, meaning the
derivative is bounded by the interval $f(x,a)$; see \cref{prop:AD_semantics_iff}. This may be the
most natural form of differential inclusions\index{differential inclusion}, though differential
relations as above cover the rest of them.

\section{Systems, components, and behavior contracts}\label{sec:systems_and_machines}\index{behavioral approach}

\subsection{Machines and interfaces}\label{sec:machines}

We imagine machines\index{machine}---or systems\index{system}---as inhabiting
interfaces\index{interface}, which mark the boundary between system and environment. The interface
is a collection of ports, each labeled with a behavior type\index{behavior type}. For example, here
are pictures of some interfaces with varying numbers of ports:\index{interface!port of}
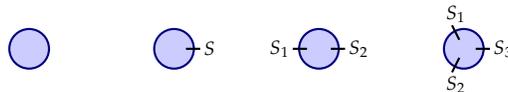
\[
\begin{tikzpicture}
[unoriented WD, spacing=13pt, pack size=15pt, surround sep=3pt, port len=0pt, font=\tiny]
  \node[pack] (f) {};
  \node[pack,right=3 of f] (g) {};
  \node[pack, right=3 of g] (h) {};
  \node[pack, right=3 of h] (i) {};
  \draw (g) to["$S$" right] +(.5,0);
  \draw (h) to["$S_2$" right] +(.5,0);
  \draw (h) to["$S_1$" left] +(-.5,0);
  \draw (i) to["$S_1$" above] +(-.2,.4);
  \draw (i) to["$S_2$" below] +(-.2,-.4);
  \draw (i) to["$S_3$" right] +(.5,0);
\end{tikzpicture}
\]
A machine inhabits an interface, and it has its own behavior type, which we call the \emph{total behavior type}, as well as a map to the each
port's behavior type. For example, a machine inhabiting the interface on the right might have total
behavior type $X$ and maps $p_i\colon X\to S_i$ for $i=1,2,3$, which we call \emph{port maps}. The tuple $\mach{X}=(X,p_1,p_2,p_3)$
defines the machine.

\begin{example}
Below in \cref{sec:LTS}, we will show how one can interpret a labeled transition system as a machines. Here we show how to interpret a system of ODEs, e.g.\ \cref{eqn:system_ODE}, as a machine. For that system, there would be $m$ ports, corresponding to the time-varying parameters $a_1,\ldots,a_m$. Each would have type $\tRR_\pi$, assuming we are expecting continuous signals. The machine would have total type
\[X=\{(x,a):\tRR_\pi^{m+n}\mid\dot{x}_1=f_1(x,a)\wedge\cdots\wedge\dot{x}_n=f_n(x,a)\}.\]
The port maps $p_i\colon X\to\tRR_\pi$ for $1\leq i\leq m$ are the projections.
\end{example}

\subsection{Systems and behavior contracts}\label{sec:behavior_contracts}

A \emph{behavior contract}\index{behavior contract}\index{behavior} on an interface $S_1,\ldots,S_n$
is a predicate on $S_1\times\cdots\times S_n$. We prefer to think of it as a proposition in
context\index{context} $\Gamma=(s_1:S_1,\ldots,s_n:S_n)$, 
\[
  s_1:S_1,\ldots,s_n:S_n\vdash\phi(s_1,\ldots,s_n):\Prop.%
  \footnote{
    A term in context, such as $s_1:S_1,\ldots,s_n:S_n\vdash\phi(s_1,\ldots,s_n):\Prop$ is roughly
    the same as a formula $\phi:(S_1\times\cdots\times S_n)\to\Prop$. The main difference is that,
    in the former case, the variables have been named.
  }
\]
If the proposition $\forall(x:X)\ldotp\phi\big(p_1(x),\ldots,p_n(x)\big)$ holds, we say that machine
$\mach{X}=(X,p_1,\ldots,p_n)$ \emph{validates} $\phi$. However, for our work of proving properties
of systems, the machine itself is often irrelevant; all that matters is its behavior contract it satisfies, namely 
$\Gamma\vdash\phi$.

A \emph{system} consists of\index{system}\index{functionality}
\begin{itemize}
  \item a finite number $n\in\NN$ of \emph{components},
  \item for each $1\leq i\leq n$ a context $\Gamma_i$, called the \emph{interface} of the component,
  \item for each $1\leq i\leq n$ a valid truth judgment $\Gamma_i\vdash\phi_i$, called the \emph{behavior contract},
  \item an \emph{outer context} $\Gamma'$,
  \item a \emph{total context} $\Gamma$, and
  \item mappings (possibly with variable re-namings) $\rho_i\colon\Gamma_i\to\Gamma$, for each
    $1\leq i\leq n$, and $\rho'\colon\Gamma'\to\Gamma$.
\end{itemize}
In other words, every variable $s:S$ in context $\Gamma_i$ or $\Gamma'$ is sent to some variable
$t:S$ of the same type ($S$) in $\Gamma$. Let $S_1,\ldots,S_k$ be the set of types in $\Gamma$
but not in the image of $\rho'$; we call these the \emph{latent variables} of the system. Conversely, types in the image of $\rho'$ are called \emph{exposed variables}.

The variables in $\phi_i$ can be renamed via $\rho_i$, so we have $\Gamma\vdash\phi_i$ for each $i$.
We say that the system satisfies outer contract $\Gamma'\vdash\phi'$ iff we have
\begin{equation}\label{eqn:outer_contract}
  \Gamma'\vdash\forall(s_1:S_1)\cdots(s_k:S_k)\ldotp[(\phi_1\wedge\cdots\wedge\phi_n)\imp\phi'].
\end{equation}
We will use this notion for our case study in \cref{sec:case_study}. The strongest contract $\Phi'$ satisfied by the system is $\Phi'\coloneqq\exists(s_1:S_1)\cdots(s_k:S_k)\ldotp\phi_1\wedge\cdots\wedge\phi_n$. Indeed, we have $\Gamma'\vdash\forall(s_1:S_1)\cdots(s_k:S_k)\ldotp[(\phi_1\wedge\cdots\wedge\phi_n)\imp\Phi']$ and if $\phi'$ also satisfies \cref{eqn:outer_contract} then $\Gamma'\vdash\Phi'\imp\phi'$.

\begin{example}
Below is a picture of a system with four components:\index{wiring diagram}
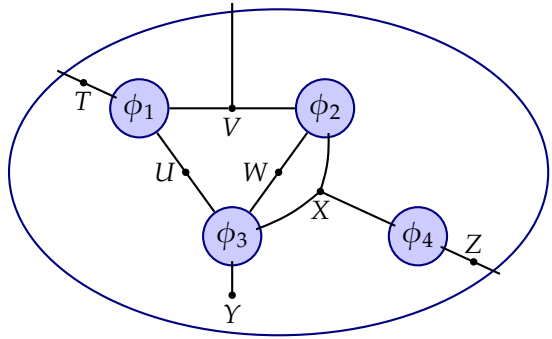
\[
\begin{tikzpicture}[unoriented WD, spacing=19pt, every label/.style={font=\small}]
	\node[pack] (f1) {$\phi_1$};
	\node[pack, below right=1.7 and 1 of f1] (f3) {$\phi_3$};
	\node[pack, above right=1.7 and 1 of f3] (f2) {$\phi_2$};
	\node[pack, below right=1.7 and 1 of f2] (f4) {$\phi_4$};
	\node[outer pack, fit=(f1) (f4)] (outer) {};
	\node[link, label=below:$T$] at ($(f1)!-.4!(outer)$) (t) {};
	\node[link, label=left:$U$] at ($(f1)!.5!(f3)$) (u) {};
	\node[link, label=below:$V$] at (f3 |-f1) (v){};
	\node[link, label=left:$W$] at ($(f3)!.5!(f2)$) (w) {};
	\node[link, label=below:$X$] at ($(w)!.3!(f4)$) (x) {};
	\node[link, below=.5 of f3, label=below:$Y$] (y) {};
	\node[link, label=above:$Z$] at ($(f4)!-.4!(outer)$) (z) {};
	\draw[shorten >=-2pt, shorten >=-2pt] (f1) -- (outer);
	\draw (v) -- (outer.north-|v);
	\draw[shorten >=-2pt, shorten >=-2pt] (f4) -- (outer);
	\draw (f1) -- (v) -- (f2) -- (f3) -- (y);
	\draw (f1) -- (f3) to[bend right=10] (x);
	\draw (x) to [bend right=10] (f2);
	\draw (x) -- (f4);
\end{tikzpicture}
\]
The four components have interfaces $\Gamma_1=(t_1:T, u_1:U, v_1:V)$; $\Gamma_2=(v_2:V, w_2:W, x_2:X)$; $\Gamma_3=(u_3:U,w_3:W, x_3:X, y_3:Y)$, and $\Gamma_4=(x_4:X,z_4:Z)$. The behavior contracts are $\Gamma_i\vdash\phi_i$ for each $1\leq i\leq 4$. The outer context is $\Gamma'=(t':T, v':V, z':Z)$. The total context is $\Gamma=(t:T, u:U, v:V, w:W, x:X, y:Y, z:Z)$. The mappings $\rho_i$ and $\rho'$ are the obvious ``letter-preserving'' functions; for example $\rho_4(x_4)=x$ and $\rho_4(z_4)=z$. The strongest contract satisfied by the system is
\begin{multline*}
  t':T, v':V, z':Z\vdash\exists(u:U)(v:V)(w:W)(x:X)(y:Y)\ldotp\\
  \phi_1(t',u,v)\wedge\phi_2(v',w,x)\wedge\phi_3(u,w,x,y)\wedge\phi_4(x,z').
\end{multline*}
\end{example}

\subsection{Control-theoretic perspective}\index{control theory}

In a control-theoretic setting, some of the ports of an interface are designated as inputs and
others are designated as outputs. In this setting, we might draw interfaces as follows:
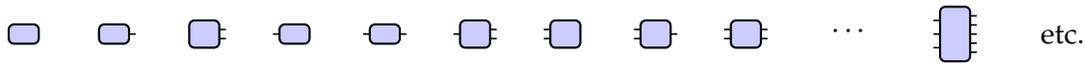
\[
\begin{tikzpicture}[oriented WD, bb small]
  \node[bb={0}{0}, fill=blue!20]                (00) {};
  \node[bb={0}{1}, fill=blue!20, right=2 of 00] (01) {};
  \node[bb={0}{2}, fill=blue!20, right=2 of 01] (02) {};
  \node[bb={1}{0}, fill=blue!20, right=2 of 02] (10) {};
  \node[bb={1}{1}, fill=blue!20, right=2 of 10] (11) {};
  \node[bb={1}{2}, fill=blue!20, right=2 of 11] (12) {};
  \node[bb={2}{0}, fill=blue!20, right=2 of 12] (20) {};
  \node[bb={2}{1}, fill=blue!20, right=2 of 20] (21) {};	
  \node[bb={2}{2}, fill=blue!20, right=2 of 21] (22) {};
  \node[                         right=2 of 22] (dt) {$\cdots$};	
  \node[bb={3}{5}, fill=blue!20, right=2 of dt] (35) {};
  \node[                         right=2 of 35] (et) {etc.};	
\end{tikzpicture}
\]
where ports on the left of a box are inputs, and those on the right are outputs.

The idea is that input trajectories deterministically drive the total system. Slightly more
precisely, let $i:X\to S$ be the input map. The condition is that for every tuple consisting of an
initial state $s$ of the system and a trajectory on each input port, each of which is compatible
with $s$ at the outset, there exists a unique trajectory of the system that extends the initial
state. We might say that a machine satisfying the existence property is \emph{total} and one
satisfying the uniqueness property is \emph{deterministic}.

Each of these can be stated in the internal language of $\BaseTopos$. It is enough to choose
$\epsilon>0$ and ask that the total/deterministic conditions apply for extensions of length
$\epsilon$. Here is an internal notion of
totalness:\index{machine!total}\index{machine!deterministic}
\begin{multline*}
  \forall(t:\Time)(P:X\to\Prop)(s:S)(r:\tRR)\ldotp r>0\imp\\
  \left(\In{[0,r]}\exists(x:X)\ldotp (ix=s)\wedge Px\right)\imp\In{[0,r+\epsilon]}\exists(x:X)\ldotp (ix=s)\wedge\In{[0,r]}Px.
\end{multline*}
And here is an internal notion of determinacy:
\[
  \forall(t:\Time)(x,x':X)(r:\tRR)\ldotp r>0\imp\left(\In{[0,r]}(x=x')\right)\imp(ix=ix')\imp\In{[0,r+\epsilon]}(x=x')
\]

In \cite{Spivak.Vasilakopoulou.Schultz:2016a} it is shown that when machines that are total,
deterministic, and what the authors call ``inertial''%
\footnote{
  The inertiality condition for an output map $p:X\to S'$ says that every internal trajectory
  produces a guaranteed output trajectory of a slightly duration. This property can be stated
  internally as follows:
  \begin{multline*}
    \forall(t:\Time)(P:X\to\Prop)(r:\tRR)\ldotp r>0\imp\left(\In{[0,r]}\exists(x:X)\ldotp Px\right)\imp\\
    \In{[0,r+\epsilon]}\exists(s':S')\ldotp\left[\left(\In{[0,r]}\exists(x:X)\ldotp Px\wedge (px=s')\right)\wedge\forall(x:X)\ldotp\left(\In{[0,r]}Px\wedge (px=s')\right)\imp (px=s')\right].
  \end{multline*}
}
are composed into systems by feeding output ports into input ports in arbitrary ways, the result is
again total, deterministic, and inertial. This proof can be carried out within the internal language
of $\BaseTopos$, although we do not do so here.

\subsection{Labeled transition systems}\label{sec:LTS}\index{labeled transition system|(}

In this section, we apply the framework with the hybrid sheaf formalism of \cref{sec:hybrid_sheaves} to understand how labeled transition systems can be interpreted in the temporal type theory. Here is the definition of labeled transition system we use.

\begin{definition}\label{def:LTS}
A \emph{labeled transition system} consists of several components:
\begin{itemize}
	\item a set $\Lambda$, elements of which are called \emph{input tokens};
	\item a set $V$, elements of which are called \emph{states};
	\item a set $E$, elements of which are called \emph{transitions};
	\item two function $s,t\colon E\to V$, called the \emph{source} and \emph{target} functions; and
	\item a function $\lambda\colon E\to\Lambda$, called the \emph{edge label function}.
\end{itemize}
\end{definition}

We want to interpret a labeled transition system as a machine, in the sense of \cref{sec:machines}. Its interface will look like this
$\begin{tikzpicture}
[unoriented WD, spacing=13pt, pack size=10pt, surround sep=3pt, port len=0pt, font=\tiny, baseline=(h.south)]
  \node[pack] (h) {};
  \draw (h) to["$I$" left] +(-.5,0);
  \draw (h) to["$O$" right] +(.5,0);
\end{tikzpicture}$
or this
$\begin{tikzpicture}[oriented WD, bby=.5ex, bb min width=3ex, font=\tiny, baseline=(a.south)]
   \node[bb={1}{1}, fill=blue!20, right=2 of 10] (a) {};
   \draw (a_in1) to["$I$" left] +(-.5pt,0);
   \draw (a_out1) to["$O$" right] +(.5pt,0);
 \end{tikzpicture}
$. In order to construct a behavior type for the ports $I$ and $O$, as well as the total behavior type of the machine, we must decide on what that behavior will be. In other words, we need to choose a way to incorporate the components of \cref{def:LTS} into a temporal story of the type of behavior it can exhibit. Consider the following story:
\begin{quotation}
Signals---in the form of input tokens---are appearing on the ports $I$ of the machine. For example, something is barking ``climb!'' or ``stay level!'', and the machine is receiving these signals at arbitrary moments in time. The machine is almost always in some state $v\in V$; the only exception is that each time an input token $l\in\Lambda$ appears, the machine goes into an instantaneous transition $e$, namely one with label $\lambda(e)=l$ and source $s(e)=v$. Immediately afterwards, it is in the target state $t(e)$. Whatever state the machine is in, it outputs its state continuously.
\end{quotation}

From the above story, we should take the input behavior type $I$ to be that of walks through the graph $E\rightrightarrows \{*\}$; see \cref{sec:walks}. Semantically, a walk of length $\ell$ is a finite path in this graph, which is just a list of elements of $E$, together with a positive duration of time between each. For example, if $E=\{\const{climb!}, \const{level!}\}$ then perhaps we could write
\[ *(0.5)\quad \lightning_{\const{climb!}}\quad *(3.3)\quad \lightning_{\const{level!}}\quad *(0.1)\quad \lightning_{\const{level!}}\quad *(5.0)\quad \lightning_{\const{climb!}}\quad *(1.1)\]
to denote a certain walk of length $10$. It rests for 0.5 seconds, instantly transitions along the edge $\const{climb!}$, rests for 3.3 seconds, instantly transitions along $\const{level!}$, etc. Each of the waiting periods occurs on the unique vertex, *, representing silence. 

We should take the output behavior $O$ to be that of walks through the graph $V\times V\rightrightarrows V$, i.e.\ the complete graph on $V$. A walk of length $\ell$ is a finite path in this graph, which is just a list of elements of $V$, together with a positive duration of time assigned to each. For example, if $V=\{a,b,c\}$ then perhaps we could write
\[a(2.4)\quad \lightning_{(a,b)}\quad b(1.1) \quad \lightning_{(b,b)}\quad b(2.0)\quad \lightning_{(b,a)}\quad a(3.3)\quad \lightning_{(a,c)}\quad c(0.8)\quad\lightning_{(c,b)}\quad b(0.4)\]
to denote a certain walk of length 10. It rests at $a$ for 2.4 seconds, transitions along the unique edge to $b$, where it rests for $1.1$ seconds, etc.

We take the total behavior $X$ of the machine to be that of walks through the graph $s,t\colon E\rightrightarrows V$ given by the source and target functions. The last step needed to show that this machine inhabits the interface (see \cref{sec:machines}) is to give maps $X\to I$ and $X\to O$. By \cref{prop:RW_functorial}, we know that the random-walks construction is functorial. Hence, it suffices to give graph homomorphisms between them, and we do so below:
\[
\begin{tikzcd}[sep=large]
E'\ar[d, shift right, "!"']\ar[d, shift left, "!"]&
E\ar[l, equal]\ar[r, "{(s,t)}"]\ar[d, shift right, "s"']\ar[d, shift left, "t"]&
V\times V\ar[d, shift right, "\tn{pr}_1"']\ar[d, shift left, "\tn{pr}_2"]
\\
\{*\}&
V\ar[l, "!"]\ar[r, equal]&
V
\end{tikzcd}
\]
Thus we have constructed a machine whose semantics matches the story above.

\index{labeled transition system|)}

\section{Case study: the National Airspace System}\label{sec:case_study}\index{National Airspace System|(}

In this section we will give a toy example, in order to see our temporal type theory---and its
higher-order temporal logic\index{logic!higher-order temporal}---in action. Our goal is to describe
very different sorts of systems---an ordinary differential equation, a time delay, and a labeled
transition system---displaying behaviors that range from continuous to discrete, instantaneous to
delayed. Each will be summarized by a behavior contract (a proposition in context), and we will
combine them to prove a property of the system formed by their interaction.
\index{interaction}\index{differential equation}\index{delay}\index{labeled transition system}

The particular example we choose to model is that of "safe separation in the National Airspace
System (NAS)" which we learned of by teaming up on a NASA grant with researchers at Honeywell Labs
(see \cref{sec:acknowledgments}). Airplanes, their pilots and
onboard equipment, as well as radars and many other factors interact to form the NAS. One property
of this system that needs to hold constantly is that of \emph{safe separation}: that airplanes
maintain a safe distance from each other. This is the system property we will prove from properties
of components. Our proofs---as well as our descriptions of component properties---will take place in
the higher-order logic of our sheaf topos $\BaseTopos\cong\Shv{\BaseSite}$, described in earlier
chapters.

\subsection{Background and setup}

We do a case study involving several disparate behavior contracts.\index{behavior contract} In the
National Airspace System, various systems like the Traffic Collision Avoidance System (TCAS)
interact to keep airplanes safely separated. Here are wiring diagrams that represent some aspects of
this system, both more globally and in a single airplane:\index{wiring diagram}
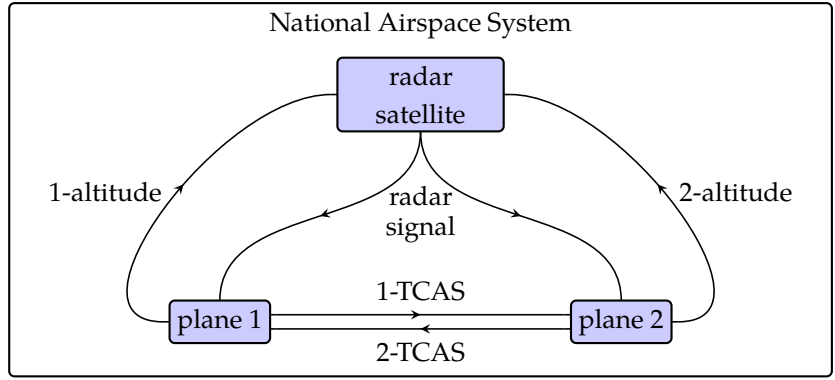
\[
\begin{tikzpicture}[oriented WD, inner sep=3pt, bb small, baseline=(RadSig), font=\small]
  \node[bb={1}{2}, fill=blue!20] (plane1) {plane 1};
  \node[bb={2}{1}, fill=blue!20, right=10 of plane1] (plane2) {plane 2\ };
  \node[bb={1}{1}, fill=blue!20, text width=2cm, align=center] at ($(plane1)!.5!(plane2)+(0,25)$) (radar) {radar satellite};
  \node[bb={0}{0}, fit={($(plane1)+(-6,-5)$) ($(plane2)+(6,5)$) ($(radar.north)+(0,5)$)}, bb name=National Airspace System] {};
  \draw[ar] (plane1_out1) to node [above=2pt] {1-TCAS} (plane2_in1);
  \draw[ar] (plane2_in2) to[out=180, in =0]  node [below=2pt] {2-TCAS} (plane1_out2);
  \draw[ar] (plane2_out1) to[in=0, out=0, looseness=1] node [above right=-.2 and -.2] {2-altitude} (radar_out1);
  \draw[ar] (plane1_in1) to[in=180, out=180, looseness=1] node [above left=-.2 and -.2] {1-altitude} (radar_in1);
  \draw[ar] (radar.270) to[out=270, in=90] node[below=5, pos=0, text width=.6in, align=center] (RadSig) {radar\\[-3pt]signal} (plane1.north);
  \draw[ar] (radar.270) to[out=270, in=90] (plane2.north);
\end{tikzpicture}
\]

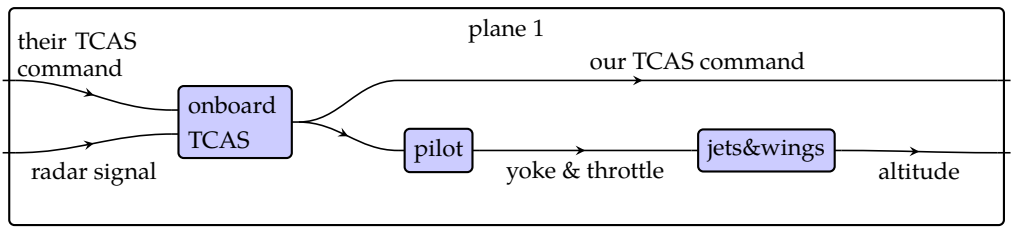
\[
\begin{tikzpicture}[oriented WD, bb port sep=1, bb port length=2.5pt, bbx=1.5cm, bb min width=.4cm, bby=1.5ex, font=\footnotesize]
  \node[bb={2}{1}, fill=blue!20, text width=1.25cm] (TCAS) {onboard TCAS};
  \node[bb={1}{1}, fill=blue!20, below right=-1.5 and 1 of TCAS] (pilot) {pilot};
  \node[bb={1}{1}, fill=blue!20, right=-.5 and 2 of pilot] (surface) {jets\&wings};
  \node[bb={2}{2}, fit={($(TCAS.north west)+(-.5,3)$) ($(surface.south east)+(.5,-1.75)$)}, bb name=plane 1] (plane) {};
  \draw[ar] (plane_in1') to node[above=3pt, text width=.8in, align=left] {their TCAS\\[-3pt] command} (TCAS_in1);
  \draw[ar] (plane_in2') to node[below=3pt] {radar signal} (TCAS_in2);
  \draw[ar] (TCAS_out1) to (pilot_in1);
  \draw[ar] let \p1=(pilot.north west), \p2=(plane_out1), \n1=\bbportlen in
    (TCAS_out1) to (\x1-\n1,\y2) -- node[above=0pt] {our TCAS command} (plane_out1');
  \draw[ar] (pilot_out1) to node[below=0pt, align=left] {yoke \& throttle} (surface_in1);
  \draw[ar] (surface_out1) to node[below=0pt] {altitude} (plane_out2');
\end{tikzpicture}
\]
Both the radar and the other airplane's TCAS send a signal to our airplane, which is picked up by
our TCAS. If there is any danger that safe separation will be violated, the TCAS alerts the pilot,
who then uses the yoke (steering wheel) and throttle (gas pedal) to send a command to the plane. The
jets and wings of the plane adjust accordingly to change the altitude.

Although this is already quite simplified, we will simplify even more. We imagine there is only one
airplane, and it is trying to attain a safe altitude. To go from the scenario described above to
this one, one could take the difference of the two airplane's altitudes. We pretend that the TCAS
single-handledly measures the plane's altitude, tells the pilot if it is too low, and the pilot
sends the message on to the plane's motors. In fact, we assume the pilot acts as a mere delay.

Here is the wiring diagram we will use:\index{wiring diagram}
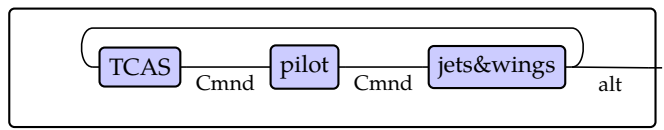
\[
\begin{tikzpicture}[oriented WD, bb port sep=1, bb port length=2.5pt, bbx=1.2cm, bb min width=.4cm, bby=1.5ex, font=\footnotesize]
  \node[bb={1}{1}, fill=blue!20] (TCAS) {TCAS};
  \node[bb={1}{1}, fill=blue!20, right=1 of TCAS] (pilot) {pilot};
  \node[bb={1}{1}, fill=blue!20, right=1 of pilot] (surface) {jets\&wings};
  \node[bb={0}{1}, fit={($(TCAS.north west)+(0,1)$) ($(surface.south east)+(0,-1)$)}] (plane) {};
  \draw (TCAS_out1) -- node[below=-1pt] {\scriptsize Cmnd} (pilot_in1);
  \draw (pilot_out1) -- node[below=-1pt] {\scriptsize Cmnd} (surface_in1);
  \draw (surface_out1) -- node[below=-1pt] {\scriptsize alt} (plane_out1');
  \draw let \p1=(TCAS.north west), \p2 = (surface.north east), \n1={\y1+\bby}, \n2=\bbportlen in
    (surface_out1) to[in=0] (\x2+\n2,\n1) -- (\x1-\n2,\n1) to[out=180] (TCAS_in1);
\end{tikzpicture}
\]
To each wire we will assign a sheaf,\index{sheaf} and two each box we will assign a behavior
contract, or proposition, on its ports. We will then combine these contracts and prove that the
whole system satisfies safe separation.

In our toy model, the altitude as a function of time is the only factor we will consider, and we
denote it by $a$. The TCAS is tasked with controlling the plane so that it will be at a safe
altitude, an altitude we denote by $\const{safe}$. We imagine that the plane can only do one of two
things: fly level or climb. The pilot commands the plane to either fly level or climb, and the plane
responds instantly by assigning its velocity $\dot{a}$ to either 0 or a number called
$\const{rate}$. This command by the pilot is in fact initiated by the TCAS, and the pilot delivers
the command to the airplane after a delay of some number of seconds, a number we denote
$\const{delay}$. The TCAS chooses its command according to the rule that if the altitude of the
plane is less than some threshold, it sends the command to $\const{climb}$; if it is greater than
this value it sends the command to stay $\const{level}$. The threshold is
$\const{safe}+\const{margin}$, where $\const{margin}>0$.

\subsection{Toy model and proof of safety}

All of the proofs in this section take place in the temporal type theory laid out in \cref{sec:logical_prelims,sec:axiomatics}.

Take $\Set{Cmnd}\coloneqq W_G$ to be the sheaf of walks (as in \cref{sec:walks}) through the graph
$G$, drawn below, which has vertex set $\Set{Cmnd}_0=\{\const{level}, \const{climb}\}$ and edge set
$E=\{\const{level!},\const{climb!}\}$:\index{graph!walks through}
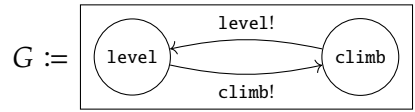
\[
\begin{tikzpicture}[font=\tiny]
  \node[circle, draw] (level) {$\const{level}$};
  \node[circle, draw, right = 2 of level] (climb) {$\const{climb}$};
  \draw[->] (level) to[bend right=12pt] node [below] {$\const{climb!}$} (climb);
  \draw[->] (climb) to[bend right=12pt] node [above] {$\const{level!}$} (level);
  \node[rectangle, draw, inner sep=5pt, fit=(level) (climb)] (surround) {};
  \node[left=0 of surround, font=\normalsize] {$G\coloneqq$};
\end{tikzpicture}
\]
So $\Set{Cmnd}$ is the $\pi$-sheafification\index{sheafification} of a quotient of
$\Set{Cmnd}_0+(E\times\Time)$, where $\Set{Cmnd}_0$ is considered as the constant sheaf on two
elements. In particular there is a map of sheaves $\Set{Cmnd}_0\to\Set{Cmnd}$. The pilot will act as
a delay that takes TCAS signals and returns commands to the yoke of the plane. The sheaf of
altitudes $\tRR_\pi$.

Suppose we are given four positive rational numbers:
$\const{safe},\const{margin},\const{del},\const{rate}:\tQQ_{>0}$, corresponding to the safe
altitude, the extra margin of safety provided by the TCAS, the pilot delay, and the ascent rate of
the plane. Take $\Gamma$ to be the type context
\[
  \Gamma\coloneqq t:\Time,T:\Set{Cmnd},P:\Set{Cmnd},a:\tRR_\pi, \const{safe}:\tQQ,\const{margin}:\tQQ,\const{del}:\tQQ,\const{rate}:\tQQ.
\]

\subsubsection{Behavior contracts for the toy model}\index{behavior contract}

We will add four axioms to the context $\Gamma$; these are the contracts satisfied by the various
components. The first contract says that $\const{margin}$ is a positive rational number and altitude
$a:\tRR_\pi$ is non-negative:
\[
  \theta_1\coloneqq(\const{margin}>0)\wedge(a\geq 0).
\]
Second, we have the TCAS contract:
\[
  \theta_2\coloneqq (a>\const{safe}+\const{margin}\imp T=\const{level})\wedge(a<\const{safe}+\const{margin}\imp T=\const{climb}).
\]
Recall that in \cref{def:derivative} we defined the derivative for $a:\tRR_\pi$, denoted
$\dot{a}:\tIIub_\pi$. The airplane contract says that the instantaneous change in altitude is
determined by the yoke:
\[
  \theta_3\coloneqq (P=\const{level}\imp \dot{a}=0)\wedge(P=\const{climb}\imp \dot{a}=\const{rate}).
\]

The delay contract is as in \cref{def:delay} for $\phi:(\Set{Cmnd}_0\times\Set{Cmnd})\to\Prop$ given
by $\phi(v,c)\coloneqq (v=c)$:
\begin{multline*}
  \theta_4\coloneqq\forall(d,u:\tRR)(v:\Set{Cmnd}_0)\ldotp (d<u)\imp\\
  \big(\See{[d,\const{delay}+u]}\In{[d,u]}(T=v)\iff
  \See{[d,\const{delay}+u]}\In{[d+\const{delay},u+\const{delay}]}(P=v)\big).
\end{multline*}

\subsubsection{Proof of safety}

Let $\Theta\coloneqq\theta_1,\theta_2,\theta_3,\theta_4$ be as above. We will prove that these four
axioms are enough to guarantee safe separation. The rest of this section takes place in the logical
context $\Gamma\mid\Theta$. In other words, we can freely refer to $T$,$P$,$a$,$\const{safe}$,
$\const{margin}$, $\const{del}$, or $\const{rate}$ as well as to axioms $\theta_1$, $\theta_2$,
$\theta_3$, and $\theta_4$.

For any $a:\tRR_\pi$ and $r:\tRR$, recall the notation $a^@(r)$ from \cref{eqn:x^@}.

\begin{lemma}\label{lemma:alt_non_decreasing}
  Given the contracts above, the plane's altitude $a:\tRR_\pi$ never decreases:
  \[
    \Gamma\mid\Theta\vdash\forall(t:\Time)(r_1,r_2:\tRR)\ldotp r_1< r_2\imp a^@(r_1)\leq a^@(r_2).
  \]
  It follows that for any $t:\Time$ and $r:\tRR$ we have the following
  \[
    q<a^@(r)\imp r<t\imp q<a
    \qquad\text{and}\qquad
    a^@(r)<q\imp t<r\imp a<q.
  \]
\end{lemma}
\begin{proof}
  By \cref{prop:lts_aa_vertex}, the command $P:\Set{Cmnd}$ is almost always a vertex,
  \[
    \neg\neg(P=\const{climb}\vee P=\const{level}).
  \]
  Since $(P=\const{climb}\vee P=\const{level})\imp \dot{a}\geq 0$, by $\theta_3$, and since
  $(\dot{a}\geq 0)=\neg(\dot{a}<0)$ is a closed proposition for the $\neg\neg$ modality, we have
  shown $\dot{a}\geq 0$. By \cref{prop:deriv_open_closed} we obtain $0\leq a^@(r_2)-a^@(r_1)$. The
  result follows from arithmetic in $\tRR_{[r_1,r_2]}$; see \cref{sec:arithmetic}.

  The second statement follows from the first by \cref{ax:torsor,prop:bounded_pi}.
\end{proof}

\begin{proposition}
  Let $M=\const{delay}+\frac{\const{safe}}{\const{rate}}$. If time $0$ is witnessed then at all
  times $t>M$ the altitude of the airplane will be safe:
  \[
    \Gamma\mid\Theta\vdash\forall(t:\Time)\ldotp\See{0}(t>M\imp a\geq\const{safe}).
  \]
\end{proposition}
\begin{proof}
  Let $N=\frac{\const{safe}}{\const{rate}}$, so $M=N+\const{delay}$. As a variable real,
  $a:\tRR_\pi$ is $\pi$-located (see \cref{def:local_reals}). Thus it is $\At{N}$-located, so since
  $\const{margin}>0$ (by $\theta_1$), we have
  \[
    \At{N}(\const{safe}<a\vee a<\const{safe}+\const{margin})
  \]
  The modality $\At{N}$ commutes with disjunction by \cref{lem:At_preserves_or}, so we break into
  cases. If $\At{N}(\const{safe}<a)$ then by \cref{lemma:alt_non_decreasing} we get
  $N<t\imp\const{safe}<a$ and hence $\See{0}(M<t\imp\const{safe}\leq a)$.

  For the second case, we assume $\At{N}(a<\const{safe}+\const{margin})$. Again by
  \cref{lemma:alt_non_decreasing} we obtain $t<N\imp(a<\const{safe}+\const{margin})$. By $\theta_2$
  we have $t<N\imp T=\const{climb}$, which implies $\In{[0,N]}(T=\const{climb})$. By $\theta_4$,
  since $0<N$, this implies $\See{[0,M]}\In{[\const{delay},M]}(P=\const{climb})$, which implies
  $\See{[0,M]}\In{[\const{delay},M]}(\dot{a}=\const{rate})$ by $\theta_3$. By
  \cref{prop:deriv_open_closed}, this implies
  \[
    \See{[0,M]}\big(\const{safe}\leq a^@(M)-a^@(\const{delay})\big)
  \]
  because $\const{safe}=\const{rate}*(M-\const{delay})$. Since $a\geq 0$ by $\theta_1$, we have
  $a^@(\const{delay})\geq0$ by \cref{prop:j_on_intervals}, and hence
  $\See{[0,M]}\At{M}(a\geq\const{safe})$. By \cref{prop:simple_modalities}, this is equivalent to
  $\See{[0,0]}\See{[M,M]}\At{M}(a\geq\const{safe})$, which implies
  $\See{0}\At{M}(a\geq\const{safe})$. One then uses \cref{lemma:alt_non_decreasing} to conclude.
\end{proof}

\index{National Airspace System|)}
\section{Relation to other temporal logics}\label{sec:temporal_logic}\index{logic!temporal|(}

Consider the temporal logic with the ``until'' and ``since'' operators,
$\mathrm{TL}(\mathrm{Until},\mathrm{Since})$. This is the standard Boolean propoitional logic
augmented with two new connectives. Given a set $\Sigma$ of atomic propositions, the set of
propositional formulas is defined by the grammar
\[
  F \coloneqq \top \mid \bot \mid P \mid \neg F \mid F_1 \wedge F_2 \mid F_1 \vee F_2 \mid F_1
  \Until F_2 \mid F_1 \Since F_2,
\]
where $P\in\Sigma$.

Intuitively, $P \Until Q$ means that $Q$ will be true at some time in the future, and that $P$ will
be true until then. Likewise $P \Since Q$ means that $Q$ was true at some time in the past, and that
$P$ has been true since then. To make descriptions like these precise, the meaning of temporal logic
operators are often defined in terms of another---first-order---logic. For example, one could
express the meaning of $P\Until Q$ as follows:
\begin{equation}\label{eq:until_semantics}
  (P\Until Q)(t_0) \coloneqq \exists t\ldotp (t_0\leq t) \wedge Q(t) \wedge
  \forall t'\ldotp (t_0\leq t'\leq t) \imp P(t').
\end{equation}
The first order logic being used here is called the ``First-Order Monadic Logic of Order,''%
\footnote{``Monadic'' here refers to the restriction that all predicates must be unary, and has no
  connection to monads in the sense of category theory. The only time we use ``monadic'' in this
  sense is when discussing other temporal logics.
}
also denoted $\mathrm{FO}(<)$. This is the standard first-order logic with a single binary relation
$<$, with the restriction that all predicates are unary (``monadic''). Then
\cref{eq:until_semantics}, together with the evident analogue for $\Since$, gives an embedding of
$\mathrm{TL}(\mathrm{Until},\mathrm{Since})$ into $\mathrm{FO}(<)$. In fact, Kamp's theorem---one of
the earliest major results in the study of temporal logic (then called ``tense logic'')---shows that
$\mathrm{TL}(\mathrm{Until},\mathrm{Since})$ is ``expressively complete'' for $\mathrm{FO}(<)$ in
any Dedekind-complete linear-time semantics.

There is a---more or less obvious---embedding of $\mathrm{FO}(<)$ into the temporal type theory
presented in this book. Each (unary) predicate $P$ of $\mathrm{FO}(<)$ is represented by a term
$P\colon\Time\to\Prop$. As an example of the embedding, \cref{eq:until_semantics} would be
represented in our type theory as
\begin{equation}\label{eq:until_embedding_constructive}
  (P\Until Q)(t_0) \coloneqq \exists(t:\Time)\ldotp (t_0\leq t) \wedge Q(t) \wedge
  \forall(t':\Time)\ldotp (t_0\leq t'\leq t) \imp P(t').
\end{equation}
However, this embedding would not be sound for the \emph{classical}
$\mathrm{TL}(\mathrm{Until},\mathrm{Since})$, since the predicate given in
\eqref{eq:until_embedding_constructive} is not decidable. There are two options to fix this. First,
one could be content with an embedding of a constructive version of
$\mathrm{TL}(\mathrm{Until},\mathrm{Since})$. But as constructive logic is strictly more expressive
that classical logic, there is no impediment to giving a sound embedding of classical
$\mathrm{TL}(\mathrm{Until},\mathrm{Since})$---our second option. To do this, we first require that
all atomic propositions are represented as decidable predicates, e.g.\ we assume
$\forall(t:\Time)\ldotp(P(t)\vee\neg P(t))\wedge(Q(t)\vee\neg Q(t))$. Then we represent $\Until$ as
\begin{align*}
  (P\Until Q)(t_0) &\coloneqq \neg\neg\exists(t:\Time)\ldotp (t_0\leq t) \wedge Q(t) \wedge
    \forall(t':\Time)\ldotp (t_0\leq t'\leq t) \imp P(t') \\
  &\iff \neg\forall(t:\Time)\ldotp
    ((t_0\leq t) \wedge \forall(u:\Time)\ldotp (t_0\leq u\leq t) \imp P(u))
    \imp \neg Q(t),
\end{align*}
where the second line is constructively equivalent to the first (hence the two are provably
equivalent in our type theory).

The other standard temporal operators are definable in terms of $\Until$, but we give explicit
representations for some of them for concreteness:
\begin{align*}
  (\bigcirc P)(t_0) &\coloneqq P(t_0+1) \\
  (\Box P)(t_0) &\coloneqq \forall(t:\Time)\ldotp (t_0\leq t) \imp P(t) \\
  (\Diamond P)(t_0) &\coloneqq \neg\neg\exists(t:\Time)\ldotp (t_0\leq t) \wedge P(t) \\
                    &\iff \neg\forall(t:\Time)\ldotp (t_0\leq t) \imp \neg P(t) \\
  (P\Release Q)(t_0) &\coloneqq \forall(t:\Time)\ldotp
    ((t_0\leq t) \wedge \forall(u:\Time)\ldotp (t_0\leq u\leq t) \imp \neg P(u))
    \imp Q(t).
\end{align*}
The standard equivalences between these operators
\[
  \Diamond P \iff \neg(\Box \neg P) \qquad
  P\Until Q \iff \neg(\neg P \Release \neg Q).
\]
are provable in the type theory.

In practice, one often wants to reason quantitatively about time, for example to say that some
proposition will be satisfied between 3 and 5 minutes from now. One well-studied quantitative
temporal logic is Metric Temporal Logic (MTL). In \cite{Hunter.Ouaknine.Worrel:2013a}, Kamp's
theorem was extended to this quantitative setting by introducing the First-Order Monadic Logic of
Order and Metric ($\mathrm{FO}(<,+\QQ)$), and showing that MTL is complete for
$\mathrm{FO}(<,+\QQ)$. This augmented first-order logic simply adds functions $+q$ to
$\mathrm{FO}(<)$ for all $q\in\QQ$. For example, the metric ``until'' connective $U_I$ of MTL, where
$I=(d,u)\in\QQ\times\QQ$ is an interval, is encoded in $\mathrm{FO}(<,+\QQ)$ by
\[
  (P\; U_I\; Q)(t_0) \coloneqq \exists t\ldotp
    \bigl[(t_0+d < t < t_0+u) \wedge Q(t) \wedge
    \forall u\ldotp (t_0<u<t \imp P(u))\bigr].
\]

The type $\Time$ in our system is a torsor over $\tRR$, and in particular there is an addition map
$+\colon \Time\times\tQQ\to\Time$. Hence the embedding of $\mathrm{FO}(<)$ is easily extended to
$\mathrm{FO}(<,+\QQ)$. For example, the metric until connective $U_I$ above is represented (in the
classical embedding) by
\[
  (P\; U_I\; Q)(t_0) \coloneqq \neg\neg\exists(t:\Time)\ldotp
    \bigl[(t_0+d < t < t_0+u) \wedge Q(t) \wedge
    \forall(u:\Time)\ldotp (t_0<u<t \imp P(u))\bigr].
\]

\index{logic!temporal|)}

\index{applications|)}

\appendix

\chapter{Predomains and approximable mappings}\label{sec:predomain}

In \cref{sec:rev_cont_poset} we introduced domains, which show up in many different disciplines,
from order theory and topology, to theoretical computer science. In this chapter, we discuss a
notion of ``basis'' for domains, which we call \emph{predomains}. They are more general than the
well-known notion of abstract basis for domains \cite[Definition
III-4.15]{Gierz.Keimel.Lawson.Mislove.Scott:2003a}, and follow the work of Steve Vickers on what he
calls ``information systems'' \cite{Vickers:1993a}. Everything in this chapter is fully
constructive,\index{logic!constructive} so it can be interpreted in the internal language of any
topos $\cat{E}$.\index{information system|see {predomain}}

In \cref{sec:predomains} we define predomains and explain how they generate domains, as well as give
examples of numeric predomains that show up throughout the book. In \cref{sec:approximable} we give
a notion of morphisms between predomains, called \emph{approximable mappings}. We also prove
there is an equivalence between the category of predomains and approximable mappings and that of
domains and Scott-continuous functions. The benefit of considering domains is that they have nice semantics, constituting a full
subcategory of topological spaces, whereas the benefit of considering predomains is that they
consist of much less data in general. Finally in \cref{sec:predomains_subtopos} we discuss how the
above work relates with subtoposes and modalities. By doing so, we were able to treat the various
Dedekind numeric objects in various modalities (see \cref{sec:Dedekind_j}) in a unified way.

While it may sometimes be difficult for a domain-theory novice to find intuition for the notions in this chapter, the proofs
are fairly elementary.

\section{Predomains and their associated domains}\label{sec:predomains}\index{predomain|(}

In this section we define predomains and then discuss how to obtain a domain from a predomain. In
\cref{sec:approximable} we will show that this construction is part of an equivalence of categories.

\subsection{Introduction to predomains}

We begin with the definition of predomain, which one will recognize is self-dual.

\begin{definition}[Predomain]\label{def:predomain}
  A \emph{predomain} is an inhabited set $B$ together with a binary relation $\prec$ satisfying
  \[
    b_1\prec b_2 \; \Leftrightarrow \; \exists(b\in B)\ldotp b_1\prec b\prec b_2.
  \]
  We sometimes denote the predomain simply by $B$ if the order is clear from context.

  To avoid uses of the axiom of choice when working internally to a topos, we assume a predomain
  comes with a specified function sending any $b_1, b_2$ such that $b_1\prec b_2$ to a $\{b_1|b_2\}$
  such that $b_1\prec \{b_1|b_2\}\prec b_2$. The notation is inspired by Conway's surreal numbers,
  and can be thought of as the ``simplest'' element between $b_1$ and $b_2$. All our work below is
  completely independent of this choice.

  Say that $(B,\prec)$ is \emph{rounded}\index{predomain!rounded} if for any $b$ there exists
  $b_1,b_2$ with $b_1\prec b\prec b_2$. Again we may assume $B$ comes with specified choices, which
  we may denote $\{\varnothing|b\}\prec b\prec\{b|\varnothing\}$.

  Any predomain $B$ has an opposite, $B\op$, with the same underlying set and opposite order. Given
  any two predomains $B_1$ and $B_2$, there is a ``product'' predomain $B_1\times B_2$ with
  underlying set the cartesian product, and with the component-wise
  order.\index{predomain!opposite}\index{predomain!products of}
\end{definition}

\begin{remark}
  When working in the internal language of a topos, replace the word ``set'' with ``type''. From
  that point of view, a predomain is a type $B$ together with a predicate $\prec:B\times B\to\Prop$,
  satisfying $\forall(b_1,b_2,b_3:B)\ldotp(b_1\prec b_2)\wedge(b_2\prec b_3)\imp(b_1\prec b_3)$, and
  a function $\{-|-\}:\{(b_1,b_2):B\times B\mid b_1\prec b_2\}\to B$, satisfying
  $\forall(b_1,b_3:B)\ldotp b_1\prec b_3\imp (b_1\prec\{b_1|b_3\})\wedge(\{b_1|b_3\}\prec b_3)$.

  In this chapter, we generally use the set-theoretic rather than topos/type-theoretic language,
  with hopes that readers will be able to use this chapter independently of the rest of the book.
\end{remark}

\begin{example}\label{ex:predomains}\index{predomain!primary examples of}\index{numeric type!domain}
  The pair $(\QQ,<)$, where $<$ is the usual order on rational numbers, is a predomain. While
  perhaps cryptic now, we denote this predomain by $\LRpre\coloneqq(\QQ,<)$; the reason will become
  clear in \cref{ex:RId_predomain}. Its opposite is $\URpre\coloneqq(\QQ,>)$.

  We refer to the predomain $\LRpre\times\URpre$ as that of \emph{improper intervals} and denote it
  $\IIpre$. Its order is obviously given by $(q_1,q_2)\prec(q_1',q_2')$ iff
  $q_1<q_1'$ and $q_2'<q_2$. The predomain of \emph{proper intervals}, denoted $\IRpre$, is given by
  $\{(q_1,q_2)\in\IIpre\mid q_1<q_2\}$ with the induced order.

  It will also be useful to define $\LRubpre$ to be the predomain $\QQ\sqcup\{-\infty\}$ equipped
  with the relation $<$ given by
  \[
    a<b\iff (a,b\in\QQ\wedge a<b)\vee(a=-\infty),
  \]
  in particular $-\infty<-\infty$. Define $\URubpre\coloneqq(\QQ\sqcup\{\infty\},>)$ to be its
  opposite. Again, $\IIubpre\coloneqq\LRubpre\times\URubpre$ and
  $\IRubpre\coloneqq\{(q_1,q_2)\in\IIubpre\mid q_1<q_2\}$, where $-\infty<\infty$.
\end{example}

We sometimes use the following shorthand notation. For any finite set $F\ss B$ and element $b\in B$,
write $F\prec b$ to denote $\forall(b'\in F)\ldotp b'\prec b$.

\begin{definition}[Up/down closure, open subsets, specialization order, joins, meets]
\label{def:predomains_rounded_specs}\index{predomain!rounded lower set in}
\index{predomain!rounded upper set in}
\index{predomain!specialization order|see {specialization order}}\index{specialization order}
\index{predomain!meets}\index{predomain!meets!conditional}\index{predomain!meets!binary}
\index{predomain!joins}\index{predomain!joins!conditional}\index{predomain!joins!binary}
\index{predomain!opens in}\index{open set!in predomain|see {predomain, opens in}}
\index{conditional joins|see {predomain, joins}}\index{conditional meets|see {predomains, meets}}
  Let $(B,\prec)$ be a predomain. For a subset $X\ss B$, define the subsets
  \[
    \upclose X\coloneqq\{b\in B\mid\exists(x\in X)\ldotp x\prec b\}
    \qquad\text{and}\qquad
    \downclose X\coloneqq\{b\in B\mid\exists(x\in X)\ldotp b\prec x\}
  \]
  called the \emph{up-closure} and the \emph{down-closure} of $X$, respectively. Despite the name
  ``closure'', one may have $X\not\ss\upclose X$ or $X\not\ss\downclose X$; however, both operations
  $\upclose$ and $\downclose$ are idempotent. A subset $X$ is called \emph{rounded upper} if
  $X=\upclose X$ and \emph{rounded lower} if $X=\downclose X$. We may refer to rounded upper subsets
  as \emph{open subsets} and denote the set of such by $\Opens(B)$.

  A predomain can be equipped with two canonical preorders, which we call the \emph{lower
  specialization order} and \emph{upper specialization order}. The lower and upper specialization
  orders are defined
  \[
    b\specord b' \coloneqq \downclose b \subseteq \downclose b'
    \qquad\text{and}\qquad
    b\upspecord b' \coloneqq \upclose b'\subseteq\upclose b.
  \]
  Clearly $b\prec b'$ implies both $b\specord b'$ and $b\upspecord b'$. Note also that any open set
  $X=\upclose X$ is an upper set in $(B,\specord)$, i.e.\ $x\in X$ and $x\specord x'$ implies $x'\in
  X$.

  Given two elements $b_1,b_2\in B$, we say that they have a \emph{join} iff there exists an element
  $b_1\binjoin b_2\in B$ that is a join in both specialization orders $(B,\specord)$ and
  $(B,\upspecord)$; similarly with meets $b_1\binmeet b_2$. We say that $B$ has \emph{conditional
  joins} if, whenever $b_1$ and $b_2$ are such that there exists a $b'$ with $\{b_1,b_2\}\specord
  b'$, then $b_1$ and $b_2$ have a join $b_1\binjoin b_2$. Dually, say that $B$ has
  \emph{conditional meets} if $b'\upspecord\{b_1,b_2\}$ implies they have a meet $b_1\binmeet b_2$.
  If any two elements have a meet (resp.\ join), we say $B$ has \emph{binary meets} (resp.\ binary
  joins).
\end{definition}

\begin{example}\label{ex:predomains_with_binarys}
  Let $B$ be any of the following predomains from \cref{ex:predomains}: $\LRpre$, $\URpre$,
  $\IIpre$, $\LRubpre$, $\URubpre$, or $\IIubpre$. Then the two specializations orders for $B$
  coincide; e.g.\ in $\URpre$ we have $b_1\specord b_2$ iff $b_1\leq b_2$ iff $b_1\upspecord b_2$.
  Moreover, $B$ is rounded and has binary meets and binary joins.

  The predomains $\IRpre$ and $\IRubpre$ are rounded and have binary meets and conditional joins,
  but not binary joins. For example, the join of $(q_1,q_2)$ and $(q_1',q_2')$ in $\IIpre$ is
  $(\max(q_1,q_1'),\min(q_2,q_2'))$, and this is not always a join in $\IRpre$ because $q_1<q_2$ and
  $q_1'<q_2'$ does not imply $\max(q_1,q_1')<^?\min(q_2,q_2')$.
\end{example}

The way we obtain domains from predomains is through rounded ideals, so the following---and its
equivalent formulation in \cref{lemma:rounded_ideals}---will be fundamental.

\begin{definition}[Rounded ideals, rounded filters]
\label{def:rounded_ideal_filter}\index{predomain!rounded ideal in}\index{ideal!rounded}
\index{predomain!rounded filter in}
  Let $(B,\prec)$ be a predomain. A subset $I\ss B$ is called a \emph{rounded ideal} if, for any
  finite set $F\ss B$, we have
  \[
    F\ss I \;\Leftrightarrow\;\exists(b\in I)\ldotp F\prec b.
  \]
  Write $\RId(B)$ for the set of rounded ideals in $B$.

  A \emph{rounded filter}\index{filter!rounded} in $B$ is a rounded ideal in $B\op$. Denote the set
  of rounded filters in $B$ by $\RFilt(B)$. Note that any rounded filter is open,
  $\RFilt(B)\ss\Omega(B)$.
\end{definition}

We spell out \cref{def:rounded_ideal_filter} in \cref{lemma:rounded_ideals}.

\begin{lemma}\label{lemma:rounded_ideals}\index{directed}\index{rounded}\index{down-closed}
  A subset $I\ss B$ is a rounded ideal iff the following three conditions hold:
  \begin{tabbing}
    \qquad\=\tn{(nonempty)}\hspace{.4in}\= $\exists b\in I$,\\
    \>\tn{(down-closed)} \>$b\in I$ and $b'\prec b$ implies $b'\in I$,\\
    \>\tn{(up-directed)}\> $\{b_1,b_2\}\ss I$ implies $\exists b'\in I, \{b_1,b_2\}\prec b'$.\\
    If $B$ has conditional joins, the up-directed condition may be replaced by the following pair of conditions:\\
    \>\tn{(rounded)} \> $b\in I$ implies $\exists b'\in I, b\prec b'$,\\
    \>\tn{(up-directed')}\> $\{b_1,b_2\}\ss I$ implies $b_1\binjoin b_2$ exists and is in $I$.
  \end{tabbing}
  Dually, a subset $U\ss B$ is a rounded filter if it is non-empty, up-closed, and down-directed.
\end{lemma}
\begin{proof}
  This is all straightforward, except the statement about replacing up-directed with rounded and
  up-directed' in the case that $B$ has conditional joins. Clearly up-directed always implies
  rounded. Supposing $B$ has conditional joins, up-directed and down-closed also implies
  up-directed'.

  For the converse, take $b_1,b_2\in I$ such that their join $b\coloneqq b_1\binjoin b_2$ exists
  with $b\in I$. We \emph{do not} necessarily have $\{b_1,b_2\}\prec^? b$; however, we do have
  $b_1\upspecord b$ and $b_2\upspecord b$ by definition, and by roundedness there exists $b'\in I$
  with $b\prec b'$, so $I$ is up-directed as desired.
\end{proof}

As mentioned in the proof of \cref{lemma:rounded_ideals}, roundedness follows from up-directedness,
and from the structure of the lemma, it may appear to take a subsidiary role. However, as the name
``rounded ideal'' suggests, roundedness is an essential aspect and---as we will see in the following
proof---a very useful technical condition.

\begin{proposition}\label{prop:RId_product}
  Let $B$ and $B'$ be predomains and $B\times B'$ their product. We have isomorphisms
  \[
    \RId(B)\times\RId(B')\To{\cong}\RId(B\times B')
    \qquad\text{and}\qquad
    \RFilt(B)\times\RFilt(B')\To{\cong}\RFilt(B\times B')
  \]
\end{proposition}
\begin{proof}
  The two statements are dual, so it suffices to prove the first. We begin by giving the maps in
  either direction. We send a pair $I\ss B$ and $I'\ss B'$ to $I\times I'\ss B\times B'$, and this
  is clearly a rounded ideal if $I$ and $I'$ are. Given some $J\ss\RId(B\times B')$, define
  \[
    I\coloneqq\{b\in B\mid\exists(b'\in B')\ldotp (b,b')\in J\}
    \qquad\text{and}\qquad
    I'\coloneqq\{b'\in B'\mid\exists(b\in B)\ldotp (b,b')\in J\}
  \]
  It is easy to see that $I$ is nonempty, rounded, and up-directed. To see it is down-closed, choose
  $b\in I$ and $b_0\prec b$. We have some $b'$ with $(b,b')\in J$ so by roundedness we can choose
  some $(b_1,b_1')$ with $(b,b')\prec(b_1,b_1')$. Then we have $(b_0,b')\prec(b_1,b_1')$, so $b_0\in
  I$. We have shown that $I\ss B$, and similarly $I'\ss B'$, are rounded ideals. Thus we have
  defined mappings in either direction, and they are evidently mutually inverse.
\end{proof}

We are ready to motivate the ``cryptic'' notation for the predomains in \cref{ex:predomains}.

\begin{example}\label{ex:RId_predomain}
  Let $\LRpre=(\QQ,<)$ be as in \cref{ex:predomains}. In the category of sets, $\RId(\LRpre)$ can be
  identified with the set of \emph{lower reals}, $\RR\sqcup\{+\infty\}$. Indeed if $r$ is an lower
  real, then one sees that $\{q\in\QQ\mid q<r\}$ is a rounded ideal by simply checking the
  conditions of \cref{lemma:rounded_ideals}. Conversely, if $I\ss\QQ$ is a rounded ideal then one
  can take its supremum $\sup(I)\in\RR\sqcup\{\infty\}$; it is easy to see these functions are
  mutually inverse. The order on rounded ideals corresponds to the usual $\leq$ order on
  $\RR\sqcup\{\infty\}$.

  Recall that $\IIpre\cong\URpre\times\LRpre$, where $\LRpre=(\URpre)\op$. By
  \cref{prop:RId_product}, a rounded ideal $I\in\RId(\IIpre)$ can be identified with a pair
  $(D_I,U_I)$ of ``cuts'', i.e.\ rounded ideals $D_I\ss\URpre$ and $U_I\ss\LRpre$.

  We will see in \cref{thm:RIdIsDomain} that $\LR\coloneqq\RId(\LRpre)$, $\UR\coloneqq\RId(\URpre)$,
  and $\II\coloneqq\RId(\IIpre)$ naturally have the structure of a domain. The last of these is the
  domain of improper intervals, e.g.\ from \cite{Kaucher:1980a}. The domain
  $\IR\coloneqq\RId(\IRpre)$ consisting of rounded ideals in the sub-predomain $\IRpre$ is the
  standard interval domain, and is important throughout the main body of this book. An element
  $I\in\IR$ corresponds to a pair of disjoint cuts, $D_I\cap U_I=\emptyset$.

  We similarly define domains corresponding to the other predomains from \cref{ex:predomains}:
  \[
    \LRub\coloneq\RId(\LRubpre),
    \quad
    \URub\coloneq\RId(\URubpre),
    \quad
    \IIub\coloneqq\RId(\IIubpre).
    \quad
    \IRub\coloneqq\RId(\IRubpre),
  \]
\end{example}

\begin{lemma}\label{lemma:rcj_downclosure_RId}
  If $B$ is a rounded predomain and has conditional joins, then for any $b\in B$, its down-closure
  $\downclose b$ is a rounded ideal.
\end{lemma}
\begin{proof}
  $\downclose b$ is obviously down-closed, and it is nonempty since $B$ is rounded
  (\cref{def:predomain}). The (rounded) and (up-directed') conditions of \cref{lemma:rounded_ideals}
  are obvious too.
\end{proof}

If $(B,\prec)$ is a predomain, so is $(\RId(B),\ss)$.

\begin{proposition}\label{prop:RId_cond_joins}
  If $B$ has conditional joins, then so does $\RId(B)$. If $B$ has binary meets, then so does
  $\RId(B)$.
\end{proposition}
\begin{proof}
  For the first, suppose $I_1,I_2,I\in\RId(B)$ with $\{I_1,I_2\}\subseteq I$. Then for any $b_1\in
  I_1$ and $b_2\in I_2$, the join $b_1\binjoin b_2$ exists and is in $I$. One checks that the set
  \[
    I_1\binjoin I_2 = \downclose \{\, b_1\binjoin b_2 \mid b_1\in I_1, b_2\in I_2 \,\}
  \]
  is a rounded ideal and that it is the join of $I_1$ and $I_2$ in $\RId(B)$. For example, to see
  that $I_1\binjoin I_2$ is inhabited, take $b_1\in I_1$ and $b_2'\prec b_2\in I_2$, note
  $b_2\specord b_1\binjoin b_2$, and find $b_2'\in I_1\binjoin I_2$.

  Now suppose $B$ has binary meets, and let $I_1,I_2\in\RId(B)$ be rounded ideals. The intersection
  $I_1\cap I_2$ is clearly down-closed. It is non-empty because there exist $b_1\in I_1$ and $b_2\in
  I_2$, hence $b_1\binmeet b_2\in I_1\cap I_2$. For roundedness, if $b\in I_1\cap I_2$, then there
  exist $b_1\in I_1$ and $b_2\in I_2$ with $b\prec\{b_1,b_2\}$, hence $b\prec b_1\binmeet b_2\in
  I_1\cap I_2$. Directedness is similar.
\end{proof}

\subsection{Domains from predomains}

Recall the notion of domains, directed sets, and the way-below relation $\ll$ from
\cref{sec:rev_cont_poset}. In this section, we use the traditional notation $\sqss$ for the order
relation on a domain.

\begin{proposition}\label{prop:domains_as_predomains}\index{domain!as predomain}
  If $(D,\sqsubseteq)$ is a domain, with way-below relation $\ll$, then $(D,\ll)$ is a predomain.
  Moreover,
  \begin{enumerate}
    \item\label{item:dom_spec} The domain order $\sqsubseteq$ and the predomain's lower
      specialization order $\specord$ coincide.
    \item\label{item:dom_ideals} For every $x\in D$, the set $\downclose x$ is a rounded ideal in
      the predomain, and the down-closure function $\downclose\colon D \to \RId(D,\ll)$ is an order
      isomorphism $(D,\sqss)\cong(\RId(D),\ss)$.
    \item\label{item:dom_opens} A subset $U\ss D$ is open in the predomain $(D,\ll)$ iff it is Scott
      open in the domain $(D,\sqss)$.
   \end{enumerate}
\end{proposition}
\begin{proof}
  We know that $\ll$ is transitive and
  interpolative\index{interpolative}\index{interpolative!(1,1)|see {interpolative}} from
  \cref{rem:way_below,prop:basic_facts_domains}, so $(D,\ll)$ is a predomain. \Cref{item:dom_spec}
  is straightforward (e.g.\ if $x\ll a\Rightarrow x\ll b$ then $a=\bigvee\downclose
  a\leq\bigvee\downclose b=b$), and so is \cref{item:dom_ideals} (the inverse is
  $\bigvee\colon\RId(D)\to D$).

  For \cref{item:dom_opens}, suppose $U\ss D$ is open in the predomain, i.e.\ $U=\upclose U$, and
  that $\bigvee X=u\in U$. There exists $u'\ll u$ with $u'\in U$, and there exists $x\in X$ such
  that $u'\leq x$. Then use any $u''\ll u'\leq x$ to see $x\in U$. Conversely, if $U$ is open in the
  domain, it is easy to check that $U\ss\upclose U$, and the fact that $\upclose U\ss U$ follows
  from \cref{item:dom_spec}.
\end{proof}

\begin{example}\label{ex:spec_not_coincide}\index{specialization order}
  The two specialization orders for the predomain underlying a domain $(D,\sqss)$ need not coincide.
  The lower specialization order $\specord$ always agrees with the domain order $\sqss$, but the
  upper specialization order can be quite different.

  Indeed, one can see this difference in the domain $\IR=\RId(\IRpre)$ from \cref{ex:RId_predomain}
  (see also \cref{prop:RIdWayBelow,thm:RIdIsDomain}), where the $\specord$ relation is better
  behaved than the $\upspecord$ relation. We denote the elements of $\IR$ by $[d,u]$, where
  $d,u\in\RR$ and $d\leq u$. Then $[d_1,u_1]\specord[d_2,u_2]$ iff $d_1\leq d_2\leq u_2\leq u_1$,
  whereas $[d_1,u_1]\upspecord[d_2,u_2]$ iff either $d_1\leq d_2\leq u_2\leq u_1$ or $d_2=u_2$. For
  example, if we identify the set $\RR$ of real numbers with intervals of the form $[r,r]$, then
  $[r_1,r_1]\specord [r_2,r_2]$ iff $r_1=r_2$ whereas $[r_1,r_1]\upspecord [r_2,r_2]$ for all $r_1,
  r_2\in\RR$.
\end{example}

One can think of a predomain as a sort of basis for a domain, and \cref{prop:domains_as_predomains}
is analogous to the statement ``the domain is a basis for itself''. We next explain how predomains
present domains. This culminates in \cref{thm:RIdIsDomain}, where we show that for any predomain
$B$, the set $\RId(B)$ of rounded ideals ordered by inclusion is a domain. Hence $\RId(B)$ is a
topological space, and we will show in \cref{thm:opens_iso_Scott_opens} that its frame of opens is
isomorphic to $\Opens(B)$. It is clear that the opens of the form $\upclose b$ form a basis for the
topology. The following lemma shows that in fact $\RFilt(B)$ is also a basis. Of course, if $B$ is
rounded and has conditional meets then $\upclose b$ is always a rounded filter, in which case this
is trivial.\index{domain!basis of}

\begin{lemma}\label{lem:PredomLemA}
  Let $(B,\prec)$ be a predomain. For any $b,b'\in B$,
  \[
    b\prec b' \Rightarrow \exists(U\in\RFilt(B))\ldotp b'\in U\subseteq \upclose b,
  \]
  and dually,
  \[
    b'\prec b \Rightarrow \exists(I\in\RId(B))\ldotp b'\in I\subseteq \downclose b.
  \]
\end{lemma}
\begin{proof}
  If $B$ has conditional meets, we can simply take $U\coloneqq\upclose\{b|b'\}$.

  Otherwise, inductively construct a descending sequence $b\prec\cdots\prec b_n\prec
  b_{n-1}\prec\cdots\prec b_1=b'$, with $b_n\coloneqq \{b|b_{n-1}\}$. Let $U=\bigcup_n\upclose b_n$.
  Clearly $U\subseteq\upclose b$, and $b'\in\upclose b_2\subseteq U$. It remains to check that $U$
  is a rounded filter. If $x\in U$ and $x\prec y$, then $y\in U$, since $b_n\prec x$ for some $n$
  hence $b_n\prec y$. And if $F\subset U$ is a finite subset, then there is some $n$ for which
  $b_n\prec F$, and $b_{n+1}\prec b_n$ implies $b_n\in U$.
\end{proof}

The next two lemmas lead up to \cref{prop:RIdWayBelow}, which characterizes the way below relation
in the poset $(\RId(B),\ss)$.

\begin{lemma}\label{lem:PredomLemB}
  For any $b\in I\in\RId(B)$, there exist $\bar{b}\in I$ and $J\in\RId(B)$ such that $b\in
  J\subseteq \downclose\bar{b}\subseteq I$.
\end{lemma}
\begin{proof}
  Since $I$ is rounded, there is some $\bar{b}\in I$ with $b\prec\bar{b}$. Then by
  \cref{lem:PredomLemA} there is a $J\in\RId(B)$ with $b\in J\subseteq\downclose\bar{b}$. Since $I$
  is down-closed, $\bar{b}\in I$ implies $\downclose\bar{b}\subseteq I$.
\end{proof}

\begin{lemma}\label{lem:PredomLemC}
  For any $I\in\RId(B)$, the set $X=\{\, J\in\RId(B) \mid \exists(j\in I)\ldotp J\subseteq\downclose
  j\,\}$ is directed, and $I=\bigcup X$.
\end{lemma}
\begin{proof}
  That $I=\bigcup X$ is clear from \cref{lem:PredomLemB}. To show directedness, suppose $J_1,J_2\in
  X$, so $J_1\subseteq\downclose j_1$ and $J_2\subseteq\downclose j_2$ for some $j_1,j_2\in I$. Then
  because $I$ is an ideal, there exists a $j\in I$ with $\{j_1,j_2\}\prec j$. By
  \cref{lem:PredomLemB}, there exists a $\bar{j}\in I$ and $J\in\RId(B)$ with $j\in
  J\subseteq\downclose\bar{j}$. Hence $J\in X$, and $J_1\subseteq\downclose j_1\subseteq\downclose
  j\subseteq J$ and similarly $J_2\subseteq J$.
\end{proof}

\index{filter!rounded|(}
\begin{proposition}\label{prop:RIdWayBelow}\index{way-below!for rounded ideals}
    For any $I,I'\in\RId(B)$, $I'\ll I$ if and only if $I'\subseteq\downclose i$ for some $i\in I$.
    Dually, for any $U,U'\in\RFilt(B)$, $U'\ll U$ iff $U'\subseteq\upclose u$ for some $u\in U$.
\end{proposition}
\begin{proof}
  Suppose $I'\ll I$. By \cref{lem:PredomLemC}, $I$ is the directed union of the set
  $\{\,J\in\RId(B)\mid\exists(j\in I)\ldotp J\subseteq\downclose j\,\}$, hence there exist
  $J\in\RId(B)$ and $j\in I$ such that $I'\subseteq J\subseteq\downclose j$.

  Conversely, suppose $i\in I$ and $I'\subseteq\downclose i$, and suppose $I\subseteq\bigcup X$ for
  some directed subset $X\subseteq\RId(B)$. Then $i\in J$ for some $J\in X$, and
  $I'\subseteq\downclose i\subseteq J$, hence $I'\ll I$.
\end{proof}

\begin{theorem}\label{thm:RIdIsDomain}\index{domain!from predomain|see {domain, as rounded ideals}}\index{domain!as rounded ideals}
  If $B$ is a predomain, then the poset $(\RId(B),\ss)$ is a domain.
\end{theorem}
\begin{proof}
  It is easy to see that rounded ideals are closed under directed unions, so $\RId(B)$ is a dcpo.
  Combining \cref{lem:PredomLemC,prop:RIdWayBelow}, every $I\in\RId(B)$ is the union of the directed
  set $\{\,J\in\RId(B)\mid J\ll I\,\}$.
\end{proof}

\begin{proposition}\label{prop:basis_domain}
  Let $(D,\sqss)$ be a domain and $A\ss D$ a subset. Then the following are equivalent:
  \begin{enumerate}
    \item For all $d\in D$, the set $\downclose d\cap A$ is directed and $d=\sup(\downclose d\cap
      A)$.
    \item For all $d_1,d_2\in D$, if $d_1\ll d_2$ then there exists $a\in A$ such that $d_1\ll a\ll
      d_2$.
    \item For all $d_1,d_2\in D$, if $d_1\ll d_2$ then there exists $a\in A$ such that $d_1\sqss
      a\ll d_2$.
  \end{enumerate}
\end{proposition}
\begin{proof}
  The equivalence of the first three, as well as other conditions, is shown in
  \cite[III-4.2]{Gierz.Keimel.Lawson.Mislove.Scott:2003a}.
\end{proof}

\begin{definition}\label{def:basis_domain}
  Let $D$ be a domain. A \emph{basis} for $D$ is a subset $A\ss D$ satisfying any of the equivalent
  conditions of \cref{prop:basis_domain}
\end{definition}

Any domain $D$ is clearly a basis for itself. It is easy to see that if $A$ is a basis and $A\ss
A'$, then $A'$ is a basis. It is also easy to check that if $A$ is a basis then $d_1\sqss d_2$ iff
$(a\ll d_1)\imp(a\ll d_2)$ for all $a\in A$.

\begin{proposition}\label{prop:predomain_as_basis}\index{domain!basis of}\index{basis|see {domain, basis of}}
  Suppose $B$ is rounded and has conditional joins. Then there is an embedding $\downclose\colon
  B\to\RId(B)$, where $b\prec b'$ iff $\downclose b\ll\downclose b'$. In particular, $B$ serves as a
  basis for $\RId(B)$.
\end{proposition}
\begin{proof}
  By \cref{lemma:rcj_downclosure_RId}, $\downclose b$ is a rounded ideal. The fact that $b\prec b'$
  iff $\downclose b\ll\downclose b'$ follows from \cref{prop:RIdWayBelow}. To see that the set
  $\{\downclose b\}_{b\in B}$ serves as a basis, one checks using \cref{prop:RIdWayBelow} that
  $\downclose I\cap\{\downclose b\}_{b\in B}=\downclose I$ for any $I\in\RId(B)$.
\end{proof}

\begin{remark}\label{rem:bases}
  In \cref{prop:predomain_as_basis}, sufficient conditions for $B$ to embed as a basis of $\RId(B)$
  were given, but in fact a necessary and sufficient condition is that $B$ have the \emph{finite
  interpolation property}\index{interpolative!(finite,1)}, $F\prec b$ iff $\exists(b'\in B)\ldotp
  F\prec b'\prec b$. Predomains with this property are well-known in the domain literature: they are
  called \emph{abstract bases}; see \cite[Definition
  III-4.15]{Gierz.Keimel.Lawson.Mislove.Scott:2003a}.\index{way-below!for predomains}\index{abstract
  basis}
\end{remark}

Next we aim to prove that $\Opens(B)$ is a frame, isomorphic to the frame of Scott-open subsets of
$\RId(B)$. We introduce the following notation for $I\in\RId(B)$ and $U\in\Opens(B)$:
\begin{equation}\label{eqn:notation_models}
  I\models^B U\coloneqq \exists(b:B)\ldotp (b\in I)\wedge (b\in U).
\end{equation}
We will see in \cref{lem:open_to_Scott_open} that $I\models^B U$ iff $I$ is contained in the
Scott-open subset corresponding to $U$. By a \emph{rounded filter} in a domain $D$ (e.g.\
$D=\RId(B)$), we mean in the predomain $(D,\ll)$; see \cref{prop:domains_as_predomains}.

\begin{lemma}\label{lem:open_to_Scott_open}
  Let $(B,\prec)$ be a predomain. For any $U\in\Opens(B)$, the set $\U_U=\{\,I\in\RId(B)\mid
  I\models^B U\,\}$ is a Scott-open subset $\U_U\in\Opens(\RId(B))$. If $U\in\RFilt(B)$ then
  $\U_U\in\RFilt(\RId(B))$.
\end{lemma}
\begin{proof}
  It is easy to see that $\U_U$ is Scott-open: it is clearly an upper set, and if
  $X\subseteq\RId(B)$ is a directed set with $\bigcup X\in\U_U$, then there is a $b\in U\cap\bigcup
  X$ so $b\in J$ for some $J\in X$, hence $J\models^B U$ and $J\in\U_U$.

  Suppose $U$ is filtered; $\U_U$ is obviously up-closed and nonempty. Let $I_1,I_2\in\U_U$ and let
  $i_1\in U\cap I_1$ and $i_2\in U\cap I_2$. Then there are $i,i'\in U$ with $i\prec i'\prec
  \{i_1,i_2\}$. By \cref{lem:PredomLemA}, there is a $J\in\RId(B)$ with $i\in J\subseteq\downclose
  i'$. Hence $i\in U\cap J$, so $J\in\U_U$, and $J\subseteq\downclose i'\subseteq I_1\cap I_2$
  implies $J\ll\{I_1,I_2\}$ by \cref{prop:RIdWayBelow}, showing that $\U_U$ is directed, hence
  filtered.
\end{proof}

\begin{lemma}\label{lem:Scott_open_to_open}
  For any Scott open $\U\in\Opens(\RId(B))$, the set $U_{\U}=\{\,b\in B \mid \exists(J\in\U)\ldotp
  J\subseteq\downclose b\,\}$ is an open of $B$. For any $I\in\U$, we have $I\models^B U_{\U}$.
  Moreover, if $\U\in\RFilt(\RId(B))$ then $U_{\U}\in\RFilt(B)$.
\end{lemma}
\begin{proof}
  Choose $I\in\U$. Since $\U$ is rounded (\cref{prop:domains_as_predomains,lemma:rounded_ideals})
  there is a $J\in\U$ with $J\ll I$. By \cref{prop:RIdWayBelow} there is a $j\in I$ with
  $J\subseteq\downclose j$. But then $j\in U_{\U}\cap I$, so $I\models^B U_{\U}$.

  We now show that $\U_{\U}$ is open. Suppose $b'\in U_{\U}$, so $J\subseteq\downclose b'$ for some
  $J\in\U$. If $b'\prec b$ then $J\subseteq\downclose b'\subseteq\downclose b$, hence $b\in U_{\U}$.
  Conversely, since $J\models^B U_\U$ there exists some $b''\in U_{\U}\cap J$, hence $b''\in
  U_{\U}$, and $b''\in J\subseteq\downclose b'$ shows $b''\prec b'$.

  Finally, suppose $\U$ is a rounded filter. If $b_1,b_2\in U_{\U}$, then $I_1\subseteq\downclose
  b_1$ and $I_2\subseteq\downclose b_2$ for some $I_1,I_2\in\U$. Because $\U$ is a rounded filter,
  there exist $I,J\in\U$ such that $J\ll I\ll\{I_1, I_2\}$, hence by \cref{prop:RIdWayBelow} there
  is some $j\in I$ with $J\subseteq\downclose j$. So $j\in U_{\U}$, and $j\in I\subseteq I_1\cap
  I_2\subseteq\downclose b_1\cap\downclose b_2$ implies $j\prec\{b_1,b_2\}$, completing the proof
  that $U_{\U}$ is a rounded filter.
\end{proof}
\index{filter!rounded|)}

\begin{theorem}\label{thm:opens_iso_Scott_opens}
  The constructions of \cref{lem:open_to_Scott_open,lem:Scott_open_to_open} determine order
  isomorphisms
  \[\Opens(\RId(B))\iso\Opens(B)
  \qquad\text{and}\qquad
  \RFilt(\RId(B))\iso\RFilt(B).\]
\end{theorem}
\begin{proof}
  Both constructions are clearly monotonic, so we show they are mutually inverse.

  If $\U\in\Opens(\RId(B))$, then $\U_{U_{\U}}=\U$ if $I\models^B U_{\U} \Leftrightarrow
  I\in\U$. The $\Leftarrow$ implication holds by \cref{lem:Scott_open_to_open}. In the other
  direction, suppose $b\in I\cap U_{\U}$. Then by definition of $U_{\U}$, there is a $J\in\U$ such
  that $J\subseteq\downclose b$. But then by \cref{prop:RIdWayBelow}, $J\ll I$, so $I\in\U$.

  If $U\in\Opens(B)$, then to show $U_{\U_U}=U$ we need to prove $b\in U$ if and only if there exists
  an $J\in\RId(B)$ such that $J\models^B U$ and $J\subseteq\downclose b$. For the forwards
  direction, if $b\in U$, then there is a $b'\in U$ with $b'\prec b$, and so by
  \cref{lem:PredomLemA} there is an $J\in\RId(B)$ with $b'\in J\subseteq\downclose b$. In the other
  direction, if $J\subseteq\downclose b$ and $b'\in J\cap U$, then $b'\prec b$ so $b\in U$.
\end{proof}

\begin{corollary}\label{cor:double_dual}
  By duality, \cref{thm:opens_iso_Scott_opens} implies $\RFilt(\RFilt(B))\iso\RId(B)$.
\end{corollary}

\begin{corollary}\label{cor:RId_domain}\index{domain!as rounded ideals}
  For any domain $D$, there is an isomorphism of domains $D\cong\RId(D,\ll)$.
\end{corollary}
\begin{proof}
  This follows from \cref{thm:opens_iso_Scott_opens}---$\Opens(\RId(D,\ll))\iso\Opens(D)$---and the
  fact that domains are sober and thus determined by their frame of opens.\index{frame}
\end{proof}

\section{Approximable mappings}
\label{sec:approximable}\index{approximable mapping|(}

Domains are, in particular, topological spaces (see \cref{rem:domains_topspaces}) and predomains
give a sort of basis for them. In this section we define a sort of mapping $B\to B'$ between
predomains, called \emph{approximable mappings}, whose purpose is to induce a continuous function
between the corresponding domains $\RId(B)\to\RId(B')$. In fact, all morphisms between domains arise
in this way---$\RId$ is an equivalence between the category of predomains and the category of
domains---as we show in \cref{cor:equiv_predoms_doms}. Again, the work in this section was adapted
from \cite{Vickers:1993a}.

\subsection{Morphisms of predomains}

Soon we will define a category $\Predom$ whose objects are predomains (see \cref{def:predomain}) and whose morphisms are approximable mappings, defined as follows.

\begin{definition}[Approximable mapping]\label{def:approx_map}
  Given two predomains $B$ and $B'$, an \emph{approximable mapping} $H\colon B\to B'$ is a relation,
  written $\amap{b}{H}{b'}$ for $b\in B$ and $b'\in B'$, satisfying
  \begin{enumerate}
    \item $b_1\prec b_2$ and $\amap{b_1}{H}{b'}$ implies $\amap{b_2}{H}{b'}$
    \item $\amap{b_2}{H}{b'}$ implies there exists a $b_1\prec b_2$ with $\amap{b_1}{H}{b'}$
    \item $b'_1\prec b'_2$ and $\amap{b}{H}{b'_2}$ implies $\amap{b}{H}{b'_1}$
    \item $\amap{b_1}{H}{b'_i}$ for all $1\leq i\leq n$ and $b_1\prec b_2$ implies there exists a $b'\in B'$
      such that $\amap{b_2}{H}{b'}$ and $b'_i\prec b'$ for all $1\leq i\leq n$, for any $n\geq 0$.
  \end{enumerate}
  The last (4) may be replaced by the following two:
  \begin{enumerate}[label=4'\alph*.]
    \item if $b_1\prec b_2$ then there exists a $b'\in B'$ with $\amap{b_2}{H}{b'}$
    \item if $b_1\prec b_2$, $\amap{b_1}{H}{b'_1}$, and $\amap{b_1}{H}{b'_2}$, then there
      exists a $b'\in B'$ with $\amap{b_2}{H}{b'}$ and $\{b'_1,b'_2\}\prec b'$.
  \end{enumerate}
  If $B$ is rounded, and $B$ and $B'$ have conditional joins, then (4) be instead be replaced by
  \begin{enumerate}[label=4''\alph*.]
    \item $\forall(b\in B)\ldotp \exists(b'\in B')\ldotp \amap{b}{H}{b'}$
    \item $\amap{b}{H}{b'_1}$ implies there exists a $b'_2\in B'$ with $b'_1\prec b'_2$ and
      $\amap{b}{H}{b'_2}$
    \item if $\amap{b}{H}{b'_1}$ and $\amap{b}{H}{b'_2}$, then $b'_1\binjoin b'_2$ exists and $\amap{b}{H}{b'_1\binjoin b'_2}$.
  \end{enumerate}

  The composite of two approximable mappings $H\colon B\to B'$ and $H'\colon B'\to B''$ is simply
  the relational composite:
  \[
    \amap{b}{(H;H')}{b''} \,\Leftrightarrow\, \exists(b'\in B')\ldotp \amap{b}{H}{b'} \wedge
    \amap{b'}{H'}{b''},
  \]
  and the identity on $B$ is $\amap{b}{\id_B}{b'}\Leftrightarrow (b'\prec b)$.
\end{definition}

There is certainly redundancy in the definition of approximable mapping. For example, given
conditions 2 and 4, condition 1 holds iff condition 3 holds.

\begin{lemma}\label{lem:approx_map}
  If $H\colon B\to B'$ is an approximable mapping, then---with no additional constraints on $B$ and $B'$---condition 4''b holds. In particular,
  $\amap{b}{H}{b'}$ iff $b\in\Opens(H)(\upclose b')$.
\end{lemma}
\begin{proof}
  Suppose $\amap{b}{H}{b'_1}$. Then by 2.\ there is some $b_1\prec b$ with $\amap{b_1}{H}{b'_1}$,
  and by 4.\ with $n=1$, there exists a $b'_2$ with $\amap{b}{H}{b'_2}$ with $b'_1\prec b'_2$.
\end{proof}

\begin{proposition}\label{prop:predom_cat}\index{predomains!category of}
  Predomains and approximable mappings form a category $\Predom$.
\end{proposition}
\begin{proof}
  Unitality of identities follows from condition 2 and its dual, condition 4'' (see
  \cref{lem:approx_map}). Suppose given approximable mappings $H\colon A\to B$ and $H'\colon B\to
  C$. The composite clearly satisfies conditions 1, 2, and 3, so we check 4.

  Let $a_1\prec a_2$ and for all $1\leq i\leq n$, $\amap{a_1}{H}{b_i}$ and $\amap{b_i}{H'}{c_i}$.
  Then there is an $a'\in A$ with $a_1\prec a'\prec a_2$, a $b\in B$ with $\amap{a'}{H}{b}$ and
  $b_i\prec b$ for all $i$, and a $b'\in B$ with $\amap{a_2}{H}{b'}$ and $b\prec b'$. Then
  $\amap{b}{H'}{c_i}$ for all $i$, so there is a $c$ with $\amap{b'}{H'}{c}$ and $c_i\prec c$ for
  all $i$. Hence $\amap{a_2}{(H;H')}{c}$.
\end{proof}

\subsection{Approximable mappings to Scott-continuous functions}

In \cref{prop:RId_ff} we will prove that an approximable mapping between predomains induces a
continuous morphism between the corresponding domains. We are already able to say what it does on
open sets.

\begin{lemma}\label{lemma:approx_open_sets}
  Any approximable mapping $H\colon B\to B'$ defines a function
  $\Omega(H)\colon\Omega(B')\to\Omega(B)$, defined on $U'\in\Omega(B')$ by
  \[
    \Omega(H)(U')\coloneqq \{u\in B\mid\exists u'\in U'\text{ such that }H(u,u')\}.
  \]
\end{lemma}
\begin{proof}
  Suppose $U'\ss B'$ is any subset and let $U\coloneqq\Omega(H)(U')$ be defined as above. If
  $b\in\upclose U$ then there exists $u\in U$ with $u\prec b$, and $b\in U$ follows from condition
  1.

  For the converse, we need the assumption $U'\in\Omega(B')$, i.e.\ $U'=\upclose U'$. For any $b\in
  U$, there exists $u\in U$ such that $u\prec b$ by condition 2.\ so $b\in \upclose U$, as desired.
\end{proof}

\begin{lemma}
Any approximable mapping $H\colon B\to B'$ defines a function $\RId(H)\colon\RId(B)\to\RId(B')$, defined on an ideal $I\in\RId(B)$ by
\begin{equation}\label{eqn:RId(H)}
  \RId(H)(I)\coloneqq \{\, b'\in B' \mid \exists(b\in I)\ldotp\amap{b}{H}{b'} \,\}.
\end{equation}
\end{lemma}
\begin{proof}
If $I\in\RId(B)$, then it is easy to check that $\RId(H)(I)$ is a rounded ideal (this uses 1, 3, 4,
and the roundedness of $I$). 
\end{proof}

We want to show that $\RId(H)$ is continuous. To do so, we will
begin by proving the following general fact, reminiscent of the theory of Chu spaces.\index{Chu space}

\begin{lemma}\label{lemma:cont_maps_chu}
  Suppose that $\cat{X}=(X,\Omega(X),\in_X)$ and $\cat{Y}=(Y,\Omega(Y),\in_Y)$ are topological
  spaces. Then there is a natural bijection between the set $\Top(\cat{X},\cat{Y})$ of continuous
  maps $\cat{X}\to\cat{Y}$ and the set
  \[
    \{(f\colon X\to Y, f^*\colon\Omega(Y)\to\Omega(X)) \mid
      \forall(x:X)(U:\Omega(Y))\ldotp f(x)\in_Y U \iff x\in_Xf^*(U) \}.
  \]
\end{lemma}

Recall from \cref{eqn:notation_models} the relation $\models^B$ on $\RId(B)\times\Omega(B)$. Using the isomorphism $\Omega(B)\cong\Omega(\RId(B))$ from \cref{thm:opens_iso_Scott_opens}, we consider $(\RId(B),\Omega(B),\models^B)$ as a topological space. Then \cref{lemma:cont_maps_chu} says that $\RId(H)$ is
continuous iff $(\RId(H)(I)\models^{B'}U')\iff(I\models^B\Omega(H)(U'))$ for any $I\in\RId(B)$ and
$U'\in\Omega(B')$.

\begin{proposition}\label{prop:RId_ff}
  For any approximable mapping $H\colon B\to B'$, the map $\RId(H)\colon\RId(B)\to\RId(B')$ from
  \cref{eqn:RId(H)} is continuous. Its inverse-image map for open sets is given by
  $\Omega(H)\colon\Omega(B')\to\Omega(B)$. Explicitly, for any $I\in\RId(B)$ and $U'\in\Opens(B')$,
  \begin{equation}\label{eq:amap_inv_image}
    \RId(H)(I)\models^{B'} U' \;\;\iff\;\; I\models^{B}\Opens(H)(U').
  \end{equation}

  Conversely, for any continuous map $f\colon\RId(B)\to\RId(B')$, represented by the induced frame
  homomorphism $f^{-1}\colon\Opens(B')\to\Opens(B)$, there is a unique approximable mapping $H\colon
  B\to B'$ such that $f=\RId(H)$, given by defining $\amap{b}{H}{b'}\coloneqq(b\in f^{-1}(\upclose
  b'))$.
\end{proposition}
\begin{proof}
  The fact that $\RId(H)$ is continuous then follows from
  \cref{lemma:cont_maps_chu,eq:amap_inv_image}, the latter of which is easily checked:
  \[
    \RId(H)(I)\models^{B'} U' \;\Leftrightarrow\; \exists(b\in I)(b'\in U')\ldotp \amap{b}{H}{b'}
      \;\Leftrightarrow\; I\models^B\Opens(H)(U').
  \]

  If $f\colon\RId(B)\to\RId(B')$ is any continuous map, then
  $f^{-1}\colon\Opens(\RId(B'))\to\Opens(\RId(B))$ is a frame homomorphism, i.e.\ an order
  preserving map which preserves finite meets and arbitrary joins. By
  \cref{thm:opens_iso_Scott_opens} we may regard this as a frame homomorphism
  $f^{-1}\colon\Opens(B')\to\Opens(B)$. Then we must check that $\amap{b}{H}{b'}\coloneqq b\in
  f^{-1}(\upclose b')$ defines an approximable mapping. Openness of $f^{-1}(\upclose b')$ implies 1
  and 2, and 3 follows because $f^{-1}$ is order preserving; we next show 4.

  Suppose $b_1\prec b_2$. In any frame, let $\top$ denote the top element and $\binmeet$ denote the
  meet. To see 4'a, we have
  \[
    b_2 \in \upclose b_1 \subseteq \top = f^{-1}(\top)
    = f^{-1}\Bigl(\bigcup_{b'\in B'}\upclose b'\Bigr)
    = \bigcup_{b'\in B'} f^{-1}(\upclose b').
  \]
  For 4'b, if $b_1\in f^{-1}(\upclose b'_1)$ and $b_1\in f^{-1}(\upclose b'_2)$, then since open
  sets are up-closed, we have $\upclose b_1\subseteq f^{-1}(\upclose b'_1)$ and $\upclose
  b_1\subseteq f^{-1}(\upclose b'_2)$, and hence,
  \[
    b_2 \in \upclose b_1 \subseteq f^{-1}(\upclose b'_1) \binmeet f^{-1}(\upclose b'_2) =
    f^{-1}(\upclose b'_1 \binmeet \upclose b'_2).
  \]
  This again implies there exists a basic open $\upclose b'\subseteq\upclose b'_1\binmeet\upclose
  b'_2$ such that $b_2\in f^{-1}(\upclose b')$.

  To see that $f=\RId(H)$, it suffices to check that $f^{-1}=\Opens(H)$ because domains are sober
  spaces (see \cref{rem:domains_topspaces}). We have
  \[
    f^{-1}(U') = f^{-1}\Bigl( \bigcup_{b'\in U'}\upclose b' \Bigr)
    = \bigcup_{b'\in U'}f^{-1}(\upclose b'),
  \]
  so $b\in f^{-1}(U')$ if and only if $\exists(b'\in U')\ldotp b\in f^{-1}(\upclose b')$, i.e.\
  $b\in\Omega(H)(U')$.

  Finally, if $H'$ is any other approximable mapping with $\RId(H')=f$, then $H=H'$ follows from
  \cref{lem:approx_map}.
\end{proof}

\begin{corollary}\label{cor:equiv_predoms_doms}\index{predomains!equivalence of categories with domains}\index{approximable mapping!as inducing continuous function}\index{topological space!generated by a predomain}
  Let $\Predom$ be as in \cref{prop:predom_cat}, and let $\Cat{Top}$ be the category of topological
  spaces. Then $\RId$ defines a fully-faithful functor $\Predom\to\Cat{Top}$. Its essential image is
  the full subcategory of continuous dcpos, i.e.\ domains, with their Scott-open topology.
\end{corollary}
\begin{proof}
  The fact that $\RId$ is bijective on hom-sets was shown in \cref{prop:RId_ff}, and the fact that
  its essential image is the domains follows from \cref{cor:RId_domain}. All that remains to be
  shown is that $\RId$ preserves composition and identity, but this is a simple check.
\end{proof}

We record now a few special kinds of approximable mappings which will be needed later: dense,
filtered, and Lawson.

\begin{definition}[Dense approximable mapping]
\label{def:dense_approx_map}\index{approximable mapping!dense}
  Let $B$ and $B'$ be rounded predomains. Say an approximable mapping $H\colon B\to B'$ is
  \emph{dense} if it satisfies the formula $\forall(b':B')\ldotp\exists(b:B)\ldotp H(b,b')$.
\end{definition}

Say that a map $h\colon R\to R'$ between topological spaces is \emph{dense} if, for every inhabited
open $U\ss R'$, the preimage $h^{-1}(U)$ is inhabited. This is the case iff the image subset
$h(R)\ss R'$ is dense in the usual sense.

\begin{proposition}
  Let $B,B'$ be predomains such that $B'$ is rounded, and $H\colon B\to B'$ an approximable mapping.
  Then $H$ is dense iff the morphism of topological spaces $\RId(H)\colon\RId(B)\to\RId(B')$ is
  dense.
\end{proposition}
\begin{proof}
  For the forward direction, suppose that $H$ is dense and that $U'\in\Omega(B')$ is inhabited. By
  \cref{prop:RId_ff} it suffices to show that $U\coloneqq\Omega(H)(U')$ is inhabited. We have $b'\in
  U'$, so there exists $b\in B$ such that $H(b,b')$, which precisely says $b\in U$.

  For the backward direction, suppose $\Omega(H)$ sends inhabited sets to inhabited sets, and take
  any $b_1'\in B'$. Since $B'$ is rounded, there exists $b_2'\in B'$ with $b_1'\prec b_2'$. Let
  $U'\coloneqq\upclose b_1'$ and note $b_2'\in U'$. Then  $U\coloneqq\Omega(H)(U')$ is inhabited, so
  there exists $b\in U$, i.e.\ there exists $b'\in U'$ such that $H(b,b')$. Hence $b_1'\prec b'$,
  and by condition 3.\ $H(b,b_1')$ as desired.
\end{proof}

\begin{definition}[Lawson, filtered approximable mapping]
\label{def:Lawson}\index{approximable mapping!Lawson}\index{approximable mapping!filtered}
  Say an approximable mapping $H\colon B\to B'$ is a \emph{Lawson approximable mapping} (for short,
  ``$H$ is Lawson'') if $H$ satisfies the strengthening of condition 2 of
  \cref{def:approx_map}, which is dual to 4: for any $n\geq 0$, if $H(b_i,b'_1)$ for all
  $1\leq i\leq n$ and $b'_2\prec b'_1$, then there exists a $b\in B$ such that $H(b,b'_2)$ and
  $b\prec b_i$ for all $1\leq i\leq n$.

  If $B$ is rounded, and both $B$ and $B'$ have conditional meets, this is equivalent to the duals
  of 4''a, 4''b, 4''c. Thus in this case, $H$ is Lawson iff $H$ is dense
  (\cref{def:dense_approx_map}) and satisfies
  \begin{equation}\label{eqn:def_Lawson}
    \big(H(b_1,b') \wedge H(b_2,b')\big)\imp H(b_1\binmeet b_2,b').
  \end{equation}
  If $H$ is any (not necessarily dense) approximable mapping and $B,B'$ are any predomains, we say
  $H$ is \emph{filtered} if it satisfies \eqref{eqn:def_Lawson}.

  More generally, say that an approximable mapping $H\colon B_1\times B_2\to B_3$ is \emph{filtered
  in the first variable} if
  \[
    \Big(H\big((b_1,b_2),b_3\big) \wedge H\big((b'_1,b_2),b_3\big)\Big) \imp H\big((b_1\binmeet b'_1,b_2),b_3\big).
  \]
\end{definition}

\subsection{Approximable mappings from functions}\index{approximable mapping!from function}
Most of the continuous maps we will need to construct will in fact be induced by \emph{functions} of
predomains---meaning functions between underlying sets---in the following way.

\begin{proposition}\label{prop:func_to_approx_map}
  Let $f\colon B\to B'$ be a function between predomains, and assume $B'$ is rounded and has
  conditional joins. Then the relation
  \[
    \amap{b}{f^*}{b'} \coloneqq b' \prec fb
  \]
  is an approximable mapping if and only if
  \begin{enumerate}
    \item $b_1\prec b_2$ implies $fb_1\specord fb_2$,
    \item $b'\prec fb_2$ implies there exists some $b_1\prec b_2$ with $b'\prec fb_1$.
  \end{enumerate}
  If these conditions hold, it follows that $b_1\specord b_2$ implies $fb_1\specord fb_2$.

  Moreover, if $B$ is also rounded, then $f^*$ is dense if and only
  if $\forall(b':B')\ldotp\exists(b:B)\ldotp b'\prec fb$. If $B$ and $B'$ also have conditional
  meets, then $f^*$ is filtered if and only if $f$ preserves meets: $fb_1 \wedge fb_2 = f(b_1\wedge
  b_2)$.
\end{proposition}
\begin{proof}
  Suppose $\amap{b}{f^*}{b'}$ defines an approximable mapping. If $b_1\prec b_2$, then
  \[
    b'\prec fb_1 \Leftrightarrow \amap{b_1}{f^*}{b'} \Rightarrow \amap{b_2}{f^*}{b'} \Leftrightarrow
    b'\prec fb_2,
  \]
  hence $fb_1\specord fb_2$. And the second condition is just part 2.\ of \cref{def:approx_map}.

  Conversely, suppose $f$ satisfies the conditions of the proposition. By
  \cref{lemma:rcj_downclosure_RId}, $\downclose b'\in\RId(B')$ for all $b'\in B'$. Conditions 3.\
  and 4.\ of \cref{def:approx_map} follow, and conditions 1.\ and 2.\ follow easily from 1.\ and 2.\
  from the proposition.

  To see that $f$ preserves the specialization order, let $b_1\specord b_2$ and $b'\prec fb_1$. Then
  by 2.\ there is a $b_0$ with $b'\prec fb_0$ and $b_0\prec b_1$. But then $b_0\prec b_1$ and
  $b_1\specord b_2$ implies $b_0\prec b_2$, hence $fb_0\specord fb_2$ by 1., which together with
  $b'\prec fb_0$ implies $b'\prec fb_2$.

  The characterizations of when $f^*$ is dense or filtered follow directly from
  \cref{def:dense_approx_map,def:Lawson}.
\end{proof}

\begin{corollary}\label{cor:predomain_inclusion}
  Any inclusion $B\ss B'$ of predomains defines an approximable mapping.
\end{corollary}

\begin{proposition}\label{prop:functions_approx_mappings}
  Let $B_1,\ldots,B_n$ and $C$ be predomains, and let $f\colon B_1\times\cdots\times B_n\to C$ be a
  function on underlying sets. Assume $C$ is rounded and has conditional joins. Then the relation
  \[
    \amap{(b_1,\ldots,b_n)}{f^*}{c} \coloneqq c\prec f(b_1,\ldots,b_n)
  \]
  is an approximable mapping if and only if it satisfies the conditions of
  \cref{prop:func_to_approx_map} in each variable separately, i.e.\ for each $1\leq i\leq n$,
  \begin{enumerate}
    \item $b_i\prec b_i'$ implies $f(b_1,\ldots,b_i,\ldots,b_n)\specord f(b_1,\ldots,b_i',\ldots,b_n)$, and
    \item $c\prec f(b_1,\ldots,b_n)$ implies there exists a $b_i'\prec b_i$ with $c\prec f(b_1,\ldots,b_i',\ldots,b_n)$.
  \end{enumerate}

  Moreover, if $B$ is also rounded, then $f^*$ is dense if and only if
  \[
    \forall(b':B')\ldotp\exists(b_1,\dots,b_n:B)\ldotp b'\prec f(b_1,\dots,b_n).
  \]
  For any $1\leq i\leq n$, if $B_i$ and $C$ also have conditional meets, then $f^*$ is filtered in
  variable $i$ if and only
  if
  \[
    f(b_1,\dots,b_i,\dots,b_n) \binmeet f(b_1,\dots,b'_i,\dots,b_n) = f(b_1,\dots,b_i\binmeet
    b'_i,\dots,b_n).
  \]
\end{proposition}
\begin{proof}
  Assume $\amap{(b_1,\ldots,b_n)}{f^*}{c}$ defines an approximable mapping. Then condition 1.\
  follows exactly as in the proof of \cref{prop:func_to_approx_map}. For the condition 2., let
  $b=(b_1,\ldots,b_n)$ and suppose $c\prec f(b)$. By \cref{def:approx_map} condition 2., there
  exists $b'=(b_1',\ldots,b_n')$ such that $b_i'\prec b_i$ for each $i$ and $c\prec f(b')$. We
  obtain $c\prec f(b_1,\ldots,b_i'\ldots,b_n)$ by $n-1$ applications of condition 1, one for each
  $j\neq i$.

  The converse direction is just as in the proof of \cref{prop:func_to_approx_map}:
  \cref{def:approx_map} conditions 3 and 4 follow from \cref{lemma:rcj_downclosure_RId}, and
  conditions 1 and 2 are $n$-applications of conditions 1 and 2 from the proposition.

  The characterizations of when $f^*$ is dense or filtered follow directly from
  \cref{def:dense_approx_map,def:Lawson}.
\end{proof}

When $f^*$ is the approximable mapping associated to a function $f\colon B_1\times\cdots\times
B_n\to C$ by \cref{prop:functions_approx_mappings}, the induced continuous map
$\RId(f^*)\colon\RId(B_1)\times\cdots\times\RId(B_n)\to\RId(C)$ has a simple form:
\begin{equation}\label{eqn:function_RId(H)}
  \RId(f^*)(I_1,\ldots,I_n)\coloneqq
    \{\, c\in C \mid \exists(b_1\in I_1)\cdots(b_n\in I_n)\ldotp c\prec f(b_1,\ldots,b_n) \,\}.
\end{equation}
Here we are using the isomorphism $\RId(B_1\times\cdots\times
B_n)\cong\RId(B_1)\times\cdots\times\RId(B_n)$ from \cref{prop:RId_product}.

The following is easy to prove and has an obvious analogue for multi-variable functions.
\begin{proposition}
  Suppose $f_1\colon B_1\to B_2$ and $f_2\colon B_2\to B_3$ are functions satisfying the conditions
  of \cref{prop:func_to_approx_map}. Then so does $f_2\circ f_1$, and we have $(f_2\circ
  f_1)^*=f_2^*\circ f_1^*$.
\end{proposition}

Our main use of predomains is to define continuous functions between various sorts of real numbers
objects in our topos, and various subtoposes. We deal with subtoposes in the next section
\ref{sec:predomains_subtopos}, but before doing so we give an example that will be useful for
defining multiplication.\index{numeric object!arithmetic from predomains}

\begin{example}\label{ex:predomain_multiplication}
  Let $\LRpluspre$ denote the predomain whose carrier set is $\QQ_{\geq0}=\{q\in\QQ\mid q\geq 0\}$
  and where $q_1\prec q_2$ iff $q_1<q_2$ or $q_1=q_2=0$. Let $\URpluspre$ denote the predomain has
  the same carrier set, but with $q_1\prec q_2$ iff $q_1<q_2$. Defining $\LRplus=\{r\in\LR\mid 0\leq
  r\}$ and $\URplus\coloneqq\{r\in\UR\mid 0\leq r\}$, then classically there are isomorphisms
  \[
    \LRplus\cong\RId(\LRpluspre)
    \qquad\text{and}\qquad
    \URplus\cong\RId(\URpluspre)
  \]

  We want to show that various operations on real numbers are continuous.\index{arithmetic!as continuous}
  For example, the subtraction operation $(q_1,q_2)\mapsto q_1-q_2$ reverses the order
  in the second variable, so it defines a function $-\colon\LRpre\times\URpre\to\LRpre$ and a
  function $-\colon\URpre\times\LRpre\to\URpre$. One checks easily that it satisfies the conditions
  of \cref{prop:functions_approx_mappings}, and hence these are approximable mappings.

  Similarly, the function $(q_1,q_2)\mapsto\max(q_1,q_2)$\index{arithmetic!max function} satisfies
  the \cref{prop:functions_approx_mappings} and hence defines approximable mappings
  $\max\colon\LRpre\times\LRpre\to\LRpre$ and $\max\colon\URpre\times\URpre\to\URpre$. We also need
  two auxiliary notations: $q^+\coloneqq\max(q,0)$ and $q^-\coloneqq\max(0-q,0)$.

  Finally, the product of nonnegative lower (resp.\ upper) reals is unproblematic. That is, the
  function $(q_1,q_2)\mapsto q_1\times q_2$ satisfies the \cref{prop:functions_approx_mappings} and
  hence defines approximable mappings $\times\colon\LRpluspre\times\LRpluspre\to\LRpluspre$ and
  $\times\colon\URpluspre\times\URpluspre\to\URpluspre$. Summarizing, we have defined approximable
  mappings
  \begin{align*}
    - &\colon \LRpre \times \URpre \to \LRpre &
    - &\colon \URpre \times \LRpre \to \URpre \\
    \max &\colon \LRpre \times \LRpre \to \LRpre &
    \max &\colon \URpre \times \URpre \to \URpre \\
    \cdot^+ &\colon \LRpre \to \LRpluspre &
    \cdot^+ &\colon \URpre \to \URpluspre \\
    \cdot^- &\colon \LRpre \to \URpluspre &
    \cdot^- &\colon \URpre \to \LRpluspre
  \end{align*}
  Taking products, we also obtain approximable mappings $-\colon\IIpre\times\IIpre\to\IIpre$ and
  $\max\colon\IIpre\times\IIpre\to\IIpre$. Since approximable mappings induce continuous functions
  between domains, one may of course remove the $_{\tn{pre}}$ subscripts in all of the above cases.

  Multiplication $*\colon\IIpre\times\IIpre\to\IIpre$ is then defined by
  $(d_1,u_1)*(d_2,u_2)=(d',u')$ where
  \begin{equation}\label{eqn:Kaucher_mult_predomain}
  \begin{aligned}
    d'&\coloneqq\max(d_1^+ d_2^+,u_1^- u_2^-)-\max(u_1^+ d_2^-, u_2^+ d_1^-)\\
    u'&\coloneqq\max(u_1^+ u_2^+,d_1^- d_2^-)-\max(d_1^+ u_2^-, d_2^+ u_1^-)
  \end{aligned}
  \end{equation}
  This is a composite of approximable mappings, and hence is approximable, so it defines a
  continuous map of domains
  \[
    *\colon\II\times\II\to\II
  \]
  which satisfies the usual properties of Kaucher multiplication; see
  \cite{Kaucher:1980a}.\index{arithmetic!multiplication}
\end{example}
\index{approximable mapping|)}

\section{Predomains in subtoposes}\label{sec:predomains_subtopos}

In this section we interpret predomains and rounded ideals inside a subtopos that is given by some
modality $j:\Prop\to\Prop$ on a topos $\cat{E}$. We also consider what happens when we change the
modality.\index{predomain!in a subtopos}\index{modality!and predomains|see {predomain, in a
subtopos}}

We continue to think of a predomain as an object $B$ in $\cat{E}$ together with a map $\prec\colon
B\times B\to\Prop$, eliding the extra data of the map $(b_1,b_2)\mapsto \{b_1|b_2\}$; see
\cref{def:predomain}.

\begin{definition}[$j$-rounded ideal, $j$-open subset]
\label{def:predomain_modality}\index{ideal!$j$-rounded}\index{open set!$j$-}
  Let $(B,\prec)$ be an internal predomain in $\cat{E}$ and $j$ a modality. Define the type of
  \emph{$j$-rounded ideals}, denoted $\RId_j(B)$, to be the subtype of $B\to\Prop$ given by those
  $I$ satisfying
  \begin{tabbing}
  \qquad
    \=($j$-closed)
  \hspace{.4in}
      \=$\forall(b:B)\ldotp j(Ib)\imp Ib$\\
    \>(down-closed)
      \>$\forall(b_1,b_2:B)\ldotp(b_1\prec b_2)\imp Ib_2\imp Ib_1$\\
    \>(up-directed)
      \>$\forall(b_1,b_2:B)\ldotp Ib_1\imp Ib_2\imp j\big(\exists b_3\ldotp (b_1\prec b_3)\wedge(b_2\prec b_3)\wedge Ib_3\big)$.\\
    \>(nonempty)
      \>$j(\exists(b:B)\ldotp Ib)$
  \end{tabbing}
  Similarly, define the type of \emph{$j$-open subsets}, denoted $\Omega_j(B)$, by
  \[
    \Omega_j(B)\coloneqq\{U:B\to\Prop\mid\forall(b:B)\ldotp Ub\iff j\exists(b':B)\ldotp Ub'\wedge
    (b'\prec b)\}
  \]
  and the relation $\models^B_j$ on $\RId_j(B)\times\Omega_j(B)$ given by
  \[
    I\models^B_jU\iff j\exists(b:B)\ldotp Ib\wedge Ub.
  \]
\end{definition}

\begin{example}\label{ex:our_j_domains}\index{predomain!primary examples of}
  For each of the predomains in \cref{ex:predomains} we can define the corresponding domain of
  $j$-rounded ideals, as we did in \cref{ex:RId_predomain}. Except for two exceptions, each of the
  resulting domains models a kind of Dedekind numeric objects, e.g.\ the lower reals, the upper
  reals, the improper intervals, etc. The exceptions are the $j$-rounded ideals in
  $\IRpre=\{(q_1,q_2)\in\IIpre\mid q_1<q_2\}$ and similarly $\IRubpre$, which do not appropriately
  handle disjointness. To correct this, we define
  \[
    \IRpre[j]\coloneqq\{(q_1,q_2)\in\IIpre\mid j(q_1<q_2)\}
    \quad\text{and}\quad
    \IRubpre[j]\coloneqq\{(q_1,q_2)\in\IIubpre\mid j(q_1<q_2)\}.
  \]
  With these definitions in hand, we can define the following domains:
  \begin{align*}
    \LR_j&\coloneqq\RId_j(\LRpre)&
    \UR_j&\coloneqq\RId_j(\URpre)&
    \II_j&\coloneqq\RId_j(\IIpre)&
    \IR_j&\coloneqq\RId_j(\IRpre[j])\\
    \LRub_j&\coloneqq\RId_j(\LRubpre)&
    \URub_j&\coloneqq\RId_j(\URubpre)&
    \IIub_j&\coloneqq\RId_j(\IIubpre)&
    \IRub_j&\coloneqq\RId_j(\IRubpre[j])
  \end{align*}
  In \cref{prop:intervals_as_RIds} it is shown that these definitions agree with our definition
  (\ref{def:local_reals}) of Dedekind $j$-numeric objects in any topos.
\end{example}

\begin{corollary}\label{cor:RIdj_IsDomain}
  For any topos $\cat{E}$, if $(B,\prec)$ is an internal predomain then $\RId_j(B)$ is an internal
  domain in $\cat{E}_j$. Given $I,I':B\to\Prop$, the domain order is given by $I'\sqss
  I\iff\forall(b:B)\ldotp I'b\imp Ib$, and the way-below relation is given by $I'\ll I\iff
  j\exists(b:B)\ldotp Ib\wedge\forall(b':B)\ldotp I'b'\imp j(b'\prec b)$.
\end{corollary}
\begin{proof}
  The first part is \cref{thm:RIdIsDomain} and the second part is \cref{prop:RIdWayBelow}, applied
  in the topos $\cat{E}_j$.
\end{proof}

In the following propositions, we explain how \cref{def:predomain_modality} is natural with respect
to changing either the modality $j$ or the predomain $B$.

\begin{proposition}\label{prop:predomain_modality_map}
  Fix a predomain $(B,\prec)$. Suppose $j'$ and $j$ are modalities, such that $j'P\imp jP$ for any
  $P:\Prop$. Then applying $j$ pointwise (i.e.\ by composing $X:B\to\Prop$ with $j:\Prop\to\Prop$)
  defines a monotonic map on rounded ideals and on open sets, which we denote
  \[
    j\colon\RId_{j'}(B)\to\RId_{j}(B)
    \qquad\text{and}\qquad
    j\colon\Omega_{j'}(B)\to\Omega_{j}(B)
  \]
\end{proposition}
\begin{proof}
  Suppose $I:B\to\Prop$ is in $\RId_{j'}(B)$; we will show $jI$ satisfies the four conditions
  necessary to be in $\RId_{j}(B)$. The first follows because $j$ is a modality, and likewise the
  third follows beacuse $Ib_2\imp Ib_1$ implies $jIb_2\imp jIb_1$. For the second, assuming
  $j'(\exists b\ldotp Ib)$ we get $j(\exists b\ldotp Ib)$ by assumption, and thus $j(\exists b\ldotp
  jIb)$. For the fourth, assuming $jIb_1$ and $jIb_2$ we get $j'j(\exists b_3\ldotp (b_1\prec
  b_3)\wedge(b_2\prec b_3)\wedge Ib_3)$, hence $j(\exists b_3\ldotp (b_1\prec b_3)\wedge(b_2\prec
  b_3)\wedge jIb_3)$ using $j'\imp j$ to remove the $j'$, and $Ib_3\imp jIb_3$ to add the $j$.  This
  is clearly monotonic, because for any $b$, $Ib\imp I'b$ implies $jIb\imp jI'b$.

  The proof that $U\in\Omega_j'(B)$ implies $jU\in\Omega_j(B)$ is similar but easier.
\end{proof}

\begin{remark}\label{rem:j_not_continuous}
  In general the map $j\colon\RId_{j'}(B)\to\RId_{j}(B)$ from \cref{prop:predomain_modality_map} is
  open (as shown), but \emph{not} continuous. Indeed, for $I\in\RId_{j'}(B)$ and
  $U\in\Omega_j(B)\ss\Omega_{j'}(B)$, the condition \cref{lemma:cont_maps_chu} that
  $(jI\models_jU)\iff^?(I\models_{j'}U)$ translates to
  \[
    (j\exists b\ldotp Ub\wedge jIb)\iff^?(j'\exists b\ldotp Ub\wedge Ib)
  \]
  and there is no reason to expect that to hold.
\end{remark}

\begin{proposition}\label{prop:approximable_j}
  Suppose given an approximable mapping $H\colon B\to B'$, and let $j$ be a modality. For any
  $I\in\RId_j(B)$ and $U'\in\Omega_j(B')$, define $\RId_j(H)(I)\coloneqq I'$ and
  $\Omega_j(H)(U')\coloneqq U$ where
  \[
    I'b'\iff j\exists(b:B)\ldotp Ib\wedge\amap{b}{H}{b'}
    \qquad\text{and}\qquad
    Ub\iff j\exists(b':B')\ldotp U'b'\wedge\amap{b}{H}{b'}.
  \]
  This defines a morphism of domains $\RId_j(H)\colon\RId_j(B)\to\RId_j(B')$.
\end{proposition}
\begin{proof}
  Take $I:\RId_j(B)$ and let $I'$ be as above; we must first show $I'$ satisfies the conditions of
  \cref{def:predomain_modality}. The first is obvious, so consider the second. Since $I$ is rounded,
  condition 4 of \cref{def:approx_map} implies $\exists(b:B)(b':B')\ldotp Ib\wedge\amap{b}{H}{b'}$,
  which directly implies $j(\exists b'\ldotp I'b')$. For the third, we are given $b_1'\prec b_2'$
  with $I'b_2'$ and need to show $I'b_1'$, but this is just condition 3. For the fourth, we assume
  $I'b_1'$ and $I'b_2'$, and hence obtain
  \[
    j\big(\exists(b_1,b_2:B)\ldotp Ib_1\wedge Ib_2\wedge\amap{b_1}{H}{b_1'}\wedge \amap{b_2}{H}{b_2'}\big)
  \]
  Because $I$ is directed we have $j\big(\exists b\ldotp
  Ib\wedge\amap{b}{H}{b_1'}\wedge\amap{b}{H}{b_2'}\big)$, and because it is rounded condition 4
  implies
  \[
    j\big(\exists (b:B)\ldotp Ib\wedge\exists(b':B')\ldotp (b_1'\prec b')\wedge (b_2'\prec b')\wedge\amap{b}{H}{b'}\big).
  \]
  Rearranging, we obtain  $j\big(\exists(b':B')\ldotp(b_1\prec b')\wedge(b_2\prec b')\wedge
  \big(\exists(b:B)\ldotp Ib\wedge \amap{b}{H}{b'}\big)\big)$, which implies $j(\exists b'\ldotp
  b_1\prec b'\wedge b_2\prec b'\wedge I'b')$, as desired.

  It remains to show that $\RId_j(H)$ is continuous, and by \cref{lemma:cont_maps_chu}, it suffices
  to take $I\in\RId_j(B)$ and $U'\in\Omega_j(B')$, let $I'\coloneqq\RId_j(H)(I)$ and
  $U\coloneqq\Omega_j(H)(U)$, and show $(I\models^B_jU)\iff(I'\models^{B'}_jU')$. Using the fact
  that $B$ and $B'$ are inhabited, it is easy to check the following equivalences, completing the
  proof:
  \begin{align*}
    (I\models^B_jU)
    &\iff j\exists b\ldotp Ib\wedge j\exists b'\ldotp U'b'\wedge H(b,b')\\
    &\iff j\exists b\exists b'\ldotp Ib\wedge U'b'\wedge H(b,b')\\
    &\iff j\exists b'\ldotp U'b'\wedge j\exists b\ldotp Ib\wedge H(b,b')\iff(I'\models^{B'}_jU').
    \qedhere
  \end{align*}
\end{proof}

As in \cref{eqn:function_RId(H)}, when $f^*$ is the approximable mapping associated to a function
$f\colon B_1\times\cdots\times B_n\to C$ by \cref{prop:functions_approx_mappings}, the induced
continuous map $\RId_j(f^*)\colon\RId_j(B_1)\times\cdots\times\RId_j(B_n)\to\RId_j(C)$ has a simple
form:
\begin{equation}\label{eqn:function_RId(H)_j}
  \RId_j(f^*)(I_1,\ldots,I_n)\coloneqq
  \{\, c\in C \mid j\exists(b_1\in I_1)\cdots(b_n\in I_n)\ldotp c\prec f(b_1,\ldots,b_n) \,\}.
\end{equation}

\begin{proposition}\label{prop:approximable_j_square}
  Suppose that $H\colon B\to B'$ is an approximable mapping and that $j_1,j_2$ are modalities such
  that $j_1P\imp j_2P$ for all $P$. Then the following squares, where the sides are as defined in
  \cref{prop:predomain_modality_map,prop:approximable_j}, commute:
  \[
  \begin{tikzcd}[column sep=45pt, row sep=20pt]
    \RId_{j_1}(B)\ar[r,"\RId_{j_1}(H)"]\ar[d,"j_2"']&\RId_{j_1}(B')\ar[d,"j_2"]\\
    \RId_{j_2}(B)\ar[r,"\RId_{j_2}(H)"']&\RId_{j_2}(B')
  \end{tikzcd}
  \hspace{.8in}
  \begin{tikzcd}[column sep=45pt, row sep=20pt]
    \Omega_{j_1}(B)\ar[d,"j_2"']&\Omega_{j_1}(B')\ar[l,"\Omega_{j_1}(H)"']\ar[d,"j_2"]\\
    \Omega_{j_2}(B)&\Omega_{j_2}(B')\ar[l,"\Omega_{j_2}(H)"]
  \end{tikzcd}
  \]
\end{proposition}
\begin{proof}
  For the first square, choose $I$ in $\RId_{j_1}B$ and $b'\in B'$. It is easy to check that indeed
  \[
    j_2j_1\big(\exists(b:B)\ldotp Ib\wedge\amap{b}{H}{b'}\big)\Leftrightarrow j_2\big(\exists(b:B)\ldotp j_2Ib\wedge\amap{b}{H}{b'}\big).
  \]
  The proof that the second square commutes is similar.
\end{proof}

\subsection{Closed modalities}\index{predomain!and closed modalities}

The map $j\colon\RId_{j'}\to\RId_j$ from \cref{prop:predomain_modality_map} has a section if $j$ is
a closed modality. Recall that $j$ is closed if there is some proposition $\phi:\Prop$ such that
$jP=\phi\vee P$ for all $P:\Prop$.

\begin{proposition}\label{prop:RId_inclusion_closed_modality}
  Let $B$ be a predomain, let $j$ be a closed modality, and let $j'$ be a modality such that $j'\imp
  j$. Then $\Omega_j(B)\ss\Omega_{j'}(B)$, and if $B$ has binary joins then
  $\RId_j(B)\subseteq\RId_{j'}(B)$.
\end{proposition}
\begin{proof}
  Suppose $jP=\phi\vee P$ and consider the second claim by taking a $j$-rounded ideal
  $I\in\RId_j(B)$. We may assume $j'=\id$, so we need to show that $I$ is nonempty and up-directed.
  We can easily prove $\phi\imp\exists(b:B)\ldotp Ib$, using that $B$ is inhabited and that $I$ is
  $j$-closed (so that $\phi\imp Ib$ for any $b$). Hence $I$ being $j$-nonempty implies that $I$ is
  nonempty.

  To show that $I$ is $j'$-up-directed, let $b_1,b_2:B$ be elements such that $Ib_1$ and $Ib_2$.
  Using again that $I$ is $j$-closed, and taking $b_3=b_1\vee b_2$, we can prove
  $\phi\imp\exists(b_3:B)\ldotp(b_1\prec b_3)\wedge(b_2\prec b_3)\wedge Ib_3$. Hence $I$ being
  $j$-up-directed implies that $I$ is up-directed.

  The proof of the first claim is similar, but easier.
\end{proof}

\begin{proposition}\label{prop:approx_inclusion_closed_modality}
  Let $B$ and $B'$ be predomains, let $j$ be a closed modality, and let $j'$ be a modality such that
  $j'\imp j$. Let $H\colon B\to B'$ be an approximable mapping. If $B$ and $B'$ have binary joins
  and $H$ is dense then the left-hand diagram commutes; if $B$ is rounded then the right-hand
  diagram commutes:
  \[
  \begin{tikzcd}[column sep=large]
    \RId_j(B)\ar[d, hook]\ar[r, "\RId_j(H)"]&\RId_j(B')\ar[d, hook]\\
    \RId_{j'}(B)\ar[r, "\RId_{j'}(H)"']&\RId_{j'}(B')
  \end{tikzcd}
  \hspace{.7in}
  \begin{tikzcd}[column sep=large]
    \Omega_j(B)\ar[d, hook]\ar[from=r, "\Omega_j(H)"']&\Omega_j(B')\ar[d, hook]\\
    \Omega_{j'}(B)\ar[from=r, "\Omega_{j'}(H)"]&\Omega_{j'}(B')
  \end{tikzcd}
  \]
\end{proposition}
\begin{proof}
  Suppose $jP=\phi\vee P$; it suffices to prove the result for $j'=\id$. For the first diagram,
  given $I\in\RId_j(B)$, we want to show that given any $b'\in B'$ the following holds:
  \[
    j\exists(b:B)\ldotp Ib\wedge H(b,b')\iff \exists(b:B)\ldotp Ib\wedge H(b,b')
  \]
  It suffices to show $\phi\imp \exists(b:B)\ldotp Ib\wedge H(b,b')$. Since $H$ is dense, we have
  $\exists b\ldotp H(b,b')$, and since $I$ is $j$-closed, we have $\forall b\ldotp\phi\imp Ib$, and
  the result follows.

  For the second diagram take $U'\in\Omega_j(B')$ and $b\in B$. We want to show
  \[
    j\exists(b':B')\ldotp U'b'\wedge H(b,b')\iff \exists(b':B')\ldotp U'b'\wedge H(b,b')
  \]
  Again it suffices to show $\phi\imp\exists(b':B')\ldotp U'b'\wedge H(b,b')$, and again since $U'$
  is $j$-closed, we have $\forall b'\ldotp\phi\imp U'b'$. Since $B$ is rounded, $\exists b'\ldotp
  H(b,b')$, and the result follows.
\end{proof}

One proves the following proposition similarly.\index{approximable mapping!dense}

\begin{proposition}\label{prop:closed_modality_denseness_lift}\index{arithmetic!mixed modality}
  Let $B_1$, $B_2$, and $B_3$ be predomains with binary joins. Let $\phi_1,\phi_2,\phi_3:\Prop$ be
  such that $(\phi_1\vee\phi_2)\imp\phi_3$, and let $j_1$, $j_2$, and $j_3$ be the corresponding
  closed modalities. Let $j'$ be any modality such that $j'P\imp(j_1P\wedge j_2P\wedge j_3P)$ for
  all $P$. Let $H\colon B_1\times B_2\to B_3$ be a dense approximable mapping. Then the dotted lift
  exists in the diagram:
  \[
  \begin{tikzcd}[column sep=large]
    \RId_{j_1}(B_1)\times\RId_{j_2}(B_2)\ar[r, dashed]\ar[d, hook]&\RId_{j_3}(B_3)\ar[d, hook]\\
    \RId_{j'}(B_1)\times\RId_{j'}(B_2)\ar[r, "\RId_{j'}(H)"']&\RId_{j'}(B_3)
  \end{tikzcd}
  \]
\end{proposition}

\begin{lemma}\label{lemma:j_inclusion_continuous}
  Suppose $B$ is any predomain, $j$ is a closed modality with $jP\coloneqq\phi\vee P$, and take any
  $I\in\RId(B)$ and $U\in\Omega_j(B)$. Then $(j\exists b\ldotp Ib\wedge Ub)\iff (\exists b\ldotp
  Ib\wedge Ub)$.
\end{lemma}
\begin{proof}
  It suffices to show $\phi\imp\exists b\ldotp Ib\wedge Ub$. Since ideals are inhabited, we have
  $\exists b\ldotp Ib$, and since $U$ is $j$-closed, $\forall b\ldotp\phi\imp Ub$. The result
  follows directly from there.
\end{proof}

\begin{proposition}\label{prop:closed_j_continuous}
  Let $B$ be a predomain, let $j$ be a closed modality, and let $j'$ be a modality such that $j'\imp
  j$. Then the inclusion $\RId_j(B)\ss\RId_{j'}$ from \cref{prop:RId_inclusion_closed_modality} is
  continuous.
\end{proposition}
\begin{proof}
  To see that $\RId_j(B)\ss\RId_{j'}$ is continuous, it suffices by \cref{lemma:cont_maps_chu} to
  show that $(I\models_j jU) \iff (I\models_{j'}U)$ for any $I\in\RId_{j}(B)$ and
  $U\in\Omega_{j'}(B)$. This translates to showing $(j\exists b\ldotp Ib \wedge jUb) \iff^?
  (j'\exists b\ldotp Ib \wedge Ub)$, and it suffices to show $(j\exists b\ldotp Ib\wedge Ub) \imp^?
  \exists b\ldotp Ib\wedge Ub$, which is \cref{lemma:j_inclusion_continuous}.
\end{proof}

\subsection{Constant predomains and ideals}\label{sec:constant_predomain}\index{predomain!constant}

In the main body of this book, we consider a sheaf topos $\BaseTopos$ for which constant sheaves $\tConst$ have a number of special properties. One of the most useful is \cref{eqn:constant_fact} shown below, which Johnstone calls the ``dual Frobenius rule''; see \cite[Example C.1.1.16e]{Johnstone:2002a}.\index{dual Frobenius rule}\index{type!constant}%
\footnote{The dual Frobenius rule shows up in the main body of this book as \cref{ax:distributivity}.}
  \begin{equation}\label{eqn:constant_fact}
    \forall(P:\Prop)(P':\tConst\to\Prop)\ldotp\left[ \forall(c:\tConst)\ldotp P\vee
    P'(c) \right] \imp \left[ P \vee \forall(c:\tConst)\ldotp P'(c) \right].
  \end{equation}
This property does not hold for constant sheaves in an arbitrary topos $\cat{E}$. However it does hold if either
\begin{itemize}
	\item $\tConst$ is finite (i.e.\ a finite coproduct $\tConst=1+1+\cdots+1$ of copies of the terminal object) or
	\item there is a defining site $(\cat{S},\chi)$ for $\cat{E}$ such that every covering family in $\chi$ is filtered.
\end{itemize}
For example, the dual Frobenius property holds for constant objects in any presheaf topos. It also holds in $\BaseTopos=\Shv{\BaseSite}$ because every covering family in $\BaseSite$ is filtered; see \cref{def:BaseSite}. Since our main application of the material in this section is to the specific topos $\BaseTopos$, we use the name ``constant predomains'' to refer to predomains having this property.

\begin{definition}[Constant predomains, ideals, decidable mappings]
\label{def:predom_const_approx_decidable}
  In a topos $\cat{E}$, we say that a predomain $(B,\prec)$ is
  \emph{constant}\index{predomain!constant} if $B$ satisfies \cref{eqn:constant_fact} and $\prec:B\times
  B\to\Prop$ is a decidable predicate, i.e.\ $\forall(b,b':B)\ldotp(b'\prec b)\vee\neg(b'\prec b)$.

  If $B$ and $B'$ are constant predomains, say that an approximable mapping $H\colon B\to B'$ is
  \emph{decidable} if $H(b,b')\vee\neg H(b,b')$ holds for all $b:B$ and $b':B'$.

  Finally, suppose $B$ is a constant predomain and $j$ is a modality. Say that a $j$-rounded ideal
  $I\in\RId_j(B)$ is \emph{$j$-constant} if it satisfies $\forall(b:B)\ldotp Ib\vee(Ib\imp j\bot)$.
  We denote the type of $j$-constant rounded ideals in $B$ by $\cRId_j(B)$.

\end{definition}\index{approximable mapping!decidable}

\begin{example}\label{ex:constant_predomains}
  The relation $<$ on $\tQQ$ is decidable in any topos $\cat{E}$. Thus if $\tQQ$ satisfies \cref{eqn:constant_fact} then the predomain $(\QQ,<)$, as well as all its variants from \cref{ex:predomains}),
  are constant in $\cat{E}$. For example, this is the case in $\BaseTopos$. Each of the variants is also rounded and has conditional joins in
  the sense of \cref{def:predomains_rounded_specs}. In \cref{prop:predom_to_RId_j}, we show that
  elements of the predomain are constant as elements of the associated domain.
\end{example}

\begin{proposition}
\label{prop:predom_to_RId_j}\index{predomain!rounded}\index{predomain!conditional joins}
  Suppose that $(B,\prec)$ is a constant predomain that is rounded and has conditional joins, and
  let $j$ be a modality. For any $b:B$, the predicate $j(b'\prec b)$ defines a map $B\to\RId_j(B)$,
  which we denote $\downclose_j$, and every rounded ideal of the form $\downclose_jb$ is
  $j$-constant:
  \[
  \begin{tikzcd}
    &\cRId_j(B)\ar[d, hook]\\
    B\ar[r, "\downclose_j"']\ar[ru, dashed]&\RId_j(B)\\
  \end{tikzcd}
  \]
  In particular, the constant elements form a basis of the domain.
\end{proposition}
\begin{proof}
  The first claim is obvious; see \cref{lemma:rcj_downclosure_RId}. To check the second claim, take
  $b,b':B$. From the fact that $\prec$ is decidable, we get  $(b'\prec b)\vee((b'\prec b)\imp\bot)$,
  hence $j(b'\prec b)\vee((b'\prec b)\imp j\bot)$ as desired. For the third claim, we use the fact
  that any superset of a basis is a basis; see \cref{prop:basis_domain}.
\end{proof}

\begin{proposition}\label{prop:j_preserves_constants}
  Let $B$ be a constant predomain and let $j'$ and $j$ be modalities such that $j'\imp j$. Then the
  dotted lift exists in the diagram:
  \[
  \begin{tikzcd}
    \cRId_{j'}(B)\ar[d, hook]\ar[r, dashed]&\cRId_{j}(B)\ar[d, hook]\\
    \RId_{j'}(B)\ar[r, "j"']&\RId_{j}(B)
  \end{tikzcd}
  \]
\end{proposition}
\begin{proof}
  It is straightforward to check that $Ib\vee (Ib\imp j'\bot)$ implies $jIb\vee (jIB\imp j\bot)$.
\end{proof}

\begin{proposition}\label{prop:decidable_approx_pres_consts}
  Let $B$ and $B'$ be constant predomains, let $j$ be a modality, and let $H:B\to B'$ be a
  decidable approximable mapping. Then the dotted lift exists in the diagram:
  \[
  \begin{tikzcd}[column sep=large]
    \cRId_j(B)\ar[d, hook]\ar[r, dashed]&\cRId_j(B')\ar[d, hook]\\
    \RId_{j}(B)\ar[r, "\RId_{j}(H)"']&\RId_{j}(B')
  \end{tikzcd}
  \]
\end{proposition}
\begin{proof}
  Choose $I\in\cRId_j(B)$ and let $I'\coloneqq\RId_j(H)(I)$, so $I'b'\iff j\exists(b:B)\ldotp
  Ib\wedge H(b,b')$. We want to show that $I'b'\vee(I'b'\imp j\bot)$ for any $b':B'$, so fix $b'$. For any $b:B$,
  since $Ib\vee (Ib\imp j\bot)$ and $H(b,b')\vee\neg H(b,b')$, it follows that $(Ib\wedge
  H(b,b'))\vee((Ib\wedge H(b,b'))\imp j\bot)$, and thus
  \[
    \forall(b:B)\ldotp I'b'\vee((Ib\wedge H(b,b'))\imp j\bot).
  \]
  By \cref{eqn:constant_fact}, we get $I'b'\vee\forall(b:B)\ldotp(Ib\wedge H(b,b'))\imp j\bot$, and
  this implies the result.
\end{proof}

\begin{proposition}\label{prop:j1j2j3j'_decidable_const}
  Let $B_1$, $B_2$, and $B_3$ be constant predomains with binary joins. Let
  $\phi_1,\phi_2,\phi_3:\Prop$ be such that $(\phi_1\vee\phi_2)\imp\phi_3$, and let $j'_1$, $j'_2$,
  and $j'_3$ be the corresponding closed modalities. Let $j$ be any modality such that
  $jP\imp(j'_1P\wedge j'_2P\wedge j'_3P)$ for all $P$. Let $H\colon B_1\times B_2\to B_3$ be a
  decidable dense approximable mapping. Then the dotted lifts exist in the diagram:
  \[
  \begin{tikzcd}[column sep=large]
    \cRId_{j'_1}(B_1)\times\cRId_{j'_2}(B_2)\ar[r, dashed]\ar[d, hook]&\cRId_{j'_3}(B_3)\ar[d, hook]\\
    \RId_{j'_1}(B_1)\times\RId_{j'_2}(B_2)\ar[r, dashed]\ar[d, hook]&\RId_{j'_3}(B_3)\ar[d, hook]\\
    \RId_{j}(B_1)\times\RId_{j}(B_2)\ar[r, "\RId_{j}(H)"']&\RId_{j}(B_3)
  \end{tikzcd}
  \]
\end{proposition}
\begin{proof}
  The bottom lift exists by \cref{prop:closed_modality_denseness_lift}, so we check the top one.
  Choose $I_1\in\cRId_{j'_1}(B_1)$ and $I_2\in\cRId_{j'_2}(B_2)$, and let
  $I_3\coloneqq\RId_{j}(H)(I_1,I_2)$. Then by using the bottom lift we have
  \[
    I_3b_3\iff \phi_3\vee\exists(b_1:B_1)(b_2:B_2)\ldotp I_1b_1\wedge I_2b_2\wedge H(b_1,b_2,b_3)
  \]
  for any $b_3:B_3$. We want to show $I_3b_3\vee(I_3b_3\imp\phi_3)$.

  For any $b_1:B_1$ and $b_2:B_2$, we have $I_1b_1\vee(I_1b_1\imp\phi_1)$ and
  $I_2b_2\vee(I_2b_2\imp\phi_2)$, as well as $H(b_1,b_2,b_3)\vee\neg H(b_1,b_2,b_3)$, and it follows
  that
  \[
    \forall(b_1,b_2)\ldotp
      (I_1b_1\wedge I_2b_2\wedge H(b_1,b_2,b_3)) \vee
      ((I_1b_1\wedge I_2b_2\wedge H(b_1,b_2,b_3)) \imp \phi_3)
  \]
  Clearly $(I_1b_1\wedge I_2b_2\wedge H(b_1,b_2,b_3))\imp I_3b_3$. By
  \cref{eqn:constant_fact} we have $I_3b_3\vee\forall(b_1,b_2)\ldotp(I_1b_1\wedge I_2b_2\wedge
  H(b_1,b_2,b_3)\imp\phi_3)$, which immediately implies the result.
\end{proof}

\subsection{Dense modalities}

\begin{definition}[Dense modality]\label{def:dense_modality}\index{modality!dense}\index{dense!modality|see {modality, dense}}
  A modality $j$ is called \emph{dense} if $j\bot \imp \bot$.
\end{definition}

\begin{proposition}\label{prop:dense_decidable_closed}
  If $j$ is a dense modality, then any decidable proposition $P$ is $j$-closed.
\end{proposition}
\begin{proof}
  Suppose $P\vee\neg P$. If $jP$ and $\neg P$, then $j\bot$, hence $\bot$. Thus we have $(jP\wedge
  P)\imp P$ and $(jP\wedge\neg P)\imp P$, giving $jP\imp P$.
\end{proof}

\subsection{Proper modalities}\label{ssec:proper_modalities}\index{modality!proper}\index{proper|see {modality, proper}}

\begin{definition}[Directed preorder, directed join, continuous inclusion]
\index{preorder!directed}
  A \emph{directed preorder} is a pair $(D,\leq)$ satisfying
  \begin{itemize}
    \item $\exists(d:D)\ldotp\top$,
    \item $\forall(d:D)\ldotp d\leq d$,
    \item $\forall(d_1,d_2,d_3:D)\ldotp(d_1\leq d_2)\imp(d_2\leq d_3)\imp(d_1\leq d_3)$, and
    \item $\forall(d_1,d_2:D)\ldotp \exists(d':D)\ldotp (d_1\leq d')\wedge(d_2\leq d')$.
  \end{itemize}
  Recall that a predicate $F:D\to\Prop$ is monotonic if $(d_1\leq d_2)\imp Fd_1\imp Fd_2$. Given a
  monotonic predicate $F$, we refer to the proposition $\exists(d:D)\ldotp Fd$ as the \emph{(directed)
  join of $F$, indexed by $D$}.

  Let $j$ be a modality. We will say the inclusion $\Prop_j\ss\Prop$ is \emph{continuous} if it
  preserves directed joins indexed by constant types. More precisely, the inclusion is continuous
  if, for every directed preorder $(D,\leq)$ such that $D$ is constant in the sense of
  \cref{sec:constant_predomain} and every monotonic $j$-closed predicate $F: D\to\Prop_j$, the
  following implication holds: $(j\exists(d:D)\ldotp Fd) \imp (\exists(d:D)\ldotp Fd)$.
\end{definition}

\begin{proposition}\label{prop:pi_equivalent_properties}
  Let $j$ be a modality. The following are equivalent:
  \begin{enumerate}
    \item The inclusion $\Prop_j\subseteq\Prop$ is continuous in the above sense.
    \item If $C$ is any inhabited constant object and $P: C\to\Prop$ is a subset such that
      \begin{equation}\label{eq:pi_equivalent_properties}
        \forall(c_1,c_2:C)\ldotp\exists(c':C)\ldotp (P(c_1)\imp P(c'))\wedge(P(c_2)\imp P(c')).
      \end{equation}
      then $(j\exists(c:C)\ldotp P(c)) \imp \exists(c:C)\ldotp jP(c)$.
  \end{enumerate}
\end{proposition}
\begin{proof}
  $1\imp 2$) Define an order on $C$ by $(c_1\leq c_2) \iff (P(c_1)\imp P(c_2))$. Then clearly $C$ is
  directed and $P\colon C\to \Prop$ is monotonic, hence $jP$ is monotonic and $j$-closed, therefore
  $j\exists(c:C)\ldotp P(c)$ implies $j\exists(c:C)\ldotp jP(c)$, which implies $(\exists(c:C)\ldotp
  jP(c))$.

  $2\imp 1$) Let $D$ be a constant, directed preorder, and $F\colon D\to\Prop_j$ a monotonic map.
  Then $D$ clearly satisfies \eqref{eq:pi_equivalent_properties}, hence $j\exists(d:D)\ldotp Fd$
  implies $\exists(d:D)\ldotp jFd$, which implies $\exists(d:D)\ldotp Fd$.
\end{proof}

\begin{definition}[Proper]\index{modality!proper}
  A modality $j$ satisfying the equivalent conditions of \cref{prop:pi_equivalent_properties} will
  be called \emph{proper}.
\end{definition}

\begin{proposition}\label{prop:proper_domain_embedding}
  For any proper modality $j$, and any constant predomain $(B,\prec)$ with binary meets, there is an
  inclusion $\RId_j(B)\subseteq\RId(B)$.
\end{proposition}
\begin{proof}
  Let $I\in\RId_j(B)$; we want to show that $I\in\RId(B)$. By \cref{lemma:rounded_ideals} we need to
  show that $\exists(b:B)\ldotp Ib$, and that if $Ib_1$ and $Ib_2$, then $\exists(b':B)\ldotp
  (b_1\prec b') \wedge (b_2\prec b') \wedge Ib'$. Both easily follow from
  \eqref{eq:pi_equivalent_properties} by letting $C\coloneqq B$ in the first case and
  $C\coloneqq\{b:B\mid (b_1\prec b) \wedge (b_2\prec b)\}$ in the second, and in both cases using
  $P\coloneqq I$ and letting $c'\coloneqq c_1\binmeet c_2$ be the meet.
\end{proof}

\begin{example}\label{ex:pi_proper}
  We saw in \cref{lem:exists_Pi_closed} that $\pi$ is a proper modality. By \cref{cor:pi_negneg} it
  is also dense in the sense of \cref{def:dense_modality}, so \cref{prop:proper_dense_constants}
  applies.
\end{example}

\begin{proposition}\label{prop:proper_dense_constants}
  For any modality $j$ which is proper and dense, and any constant predomain $B$ with binary meets,
  the inclusion $\cRId(B)\subseteq\RId(B)$ factors through $\RId_j(B)\subseteq\RId(B)$, and in fact
  $\cRId(B)=\cRId_j(B)$.
\end{proposition}
\begin{proof}
  The first claim follows immediately from
  \cref{prop:dense_decidable_closed,prop:proper_domain_embedding}, and the second claim is then
  obvious, since $j\bot \iff \bot$.
\end{proof}

We've seen in \cref{prop:RId_inclusion_closed_modality} statements about when inclusions of the form
$\RId_j(B)\subseteq\RId_{j'}(B)$ commute with approximable mappings, where our primary interest is
in the case $j'=\id$. The analagous statements about proper modalities turn out to be somewhat more
subtle.

\begin{proposition}\label{prop:proper_constant}
  Let $j$ be a proper modality, and let $H\colon B\to B'$ be any approximable mapping, where $B$ and
  $B'$ are constant predomains with binary meets. Then the square in the diagram below need not
  commute, but it does upon restriction to $\cRId(B)$:
  \[
  \begin{tikzcd}
    \cRId(B) \ar[r,hook] &
    \RId_j(B) \ar[r,"\RId_j(H)"] \ar[d,hook] &[12pt]
    \RId_j(B') \ar[d,hook] \\
    & \RId(B) \ar[r,"\RId(H)"'] &
    \RId(B')
  \end{tikzcd}
  \]
\end{proposition}
\begin{proof}
  Let $I\in\cRId(B)$. We want to show that for any $b'\in B'$, the proposition $\exists(b:B)\ldotp
  Ib\wedge H(b,b')$ is $j$-closed. The type $\{b:B\mid Ib\}$ is constant and, with the
  lower-specialization order $\specord$ on $B$, it is also directed. Moreover, it follows directly
  from \cref{def:approx_map} that if $b_1\specord b_2$ then $H(b_1,b')\imp H(b_2,b')$. Therefore,
  because $j$ is proper, $\exists(b:B\mid Ib)\ldotp H(b,b')$ is $j$-closed.
\end{proof}

We say that $H\colon B\to B'$ is \emph{decidable}\index{approximable mapping!decidable} if it is
decidable as a relation, i.e.\ $\forall b,b'\ldotp H(b,b')\vee\neg H(b,b')$. Recall the notion of
filtered approximable mappings from \cref{def:Lawson}.

\begin{proposition}\label{prop:proper_Lawson}
  Let $j$ be a proper modality, let $B$ and $B'$ be constant predomains with binary meets, and let
  $H\colon B\to B'$ be a decidable filtered approximable mapping. Then the diagram commutes:
  \[
    \begin{tikzcd}
      \RId_j(B) \ar[r,"\RId_j(H)"] \ar[d,hook] &
      \RId_j(B') \ar[d,hook] \\
      \RId(B) \ar[r,"\RId(H)"'] &
      \RId(B')
    \end{tikzcd}
  \]
\end{proposition}
\begin{proof}
  Let $I\in\RId_j(B)$. We want to show that for any $b'\in B'$, the proposition $\exists(b:B)\ldotp
  Ib\wedge H(b,b')$ is $j$-closed. Because $H$ is decidable, the type $\{b:B\mid H(b,b')\}$ is
  constant, and because $H$ is filtered, $H(b_1,b')$ and $H(b_2,b')$ implies $H(b_1\binmeet
  b_2,b')$. We also have $I(b_1)\imp I(b_1\binmeet b_2)$ and $I(b_2)\imp I(b_1\binmeet b_2)$, so
  because $j$ is proper, $\exists(b:B\mid H(b,b'))\ldotp I(b)$ is $j$-closed.
\end{proof}

\Cref{prop:proper_constant,prop:proper_Lawson} can be usefully combined, as in the following
proposition, whose proof is similar to the previous two and is left to the reader.

\begin{proposition}\label{prop:proper_constant_Lawson}
  Let $j$ be a proper modality, let $B_1$, $B_2$, and $B_3$ be constant predomains with binary
  meets, and let $H\colon B_1\times B_2 \to B_3$ be a decidable approximable mapping which is
  filtered in the second variable, i.e.\ $H((b_1,b_2),b_3)$ and $H((b_1,b'_2),b_3)$ implies
  $H((b_1,b_2\binmeet b'_2),b_3)$. Then the square in the diagram below need not commute, but it
  does upon restriction to $\cRId(B_1)\times\RId_j(B_2)$:
  \[
    \begin{tikzcd}
      \cRId(B_1)\times\RId_j(B_2) \ar[r,hook] &
      \RId_j(B_1)\times\RId_j(B_2) \ar[r,"\RId_j(H)"] \ar[d,hook] &
      \RId_j(B_3) \ar[d,hook] \\
      & \RId(B_1)\times\RId(B_2) \ar[r,"\RId(H)"'] &
      \RId(B_3)
    \end{tikzcd}
  \]
\end{proposition}

We can also establish analogues of the previous propositions for opens, rather than approximable
mappings, either by using essentially the same proofs, or by making use of the observation that
opens in a predomain $B$ are ``the same as'' approximable mappings $B\to 2$ to the two element
linear poset. We state the analogue of \cref{prop:proper_constant_Lawson} for the record.

\begin{proposition}\label{prop:proper_constant_Lawson_opens}
  Let $j$ be a proper modality, let $B_1$ and $B_2$ be constant predomains with binary
  meets, and let $U\in\Opens(B_1\times B_2)$ be a decidable open which is filtered in the second
  variable, i.e.\ $U(b_1,b_2)$ and $U(b_1,b'_2)$ implies $U(b_1,b_2\binmeet
  b'_2)$. Then the square in the diagram below need not commute, but it does upon restriction
  to $\cRId(B_1)\times\RId_j(B_2)$:
  \[
    \begin{tikzcd}
      \cRId(B_1)\times\RId_j(B_2) \ar[r,hook] &
      \RId_j(B_1)\times\RId_j(B_2) \ar[r,"{}\vDash_j U"] \ar[d,hook] &
      \Prop_j \ar[d,hook] \\
      & \RId(B_1)\times\RId(B_2) \ar[r,"{}\vDash U"'] &
      \Prop
    \end{tikzcd}
  \]
\end{proposition}

\chapter{\texorpdfstring{$\IRinv$}{\unichar{"1D540}\unichar{"211D}\unichar{"2215}\unichar{"25B9}} as a continuous category}
\label{chap:IRinv_continuous}\index{category!continuous}

Recall the category $\IRinv$ from \cref{def:IRinv}, which is a the translation-invariant version of the interval domain
$\IR$. Whereas $\IR$ is a continuous poset\index{domain}\index{category!continuous}, which we and
other authors call a domain, $\IRinv$ is a continuous category. We begin this chapter with a review
of continuous categories.\index{intervals!translation-invariant}

\section{Review of continuous categories}\label{sec:review_cont_cats}

Throughout the book, we have made significant use of the fact that $\IR$ is a continuous poset---or
domain---in understanding the semantics of the topos $\BaseTopos$. There is a strong relationship
between $\IR$ and the its quotient $\IRinv$ on which $\BaseTopos$ is the topos of sheaves. Unlike $\IR$, the category
$\IRinv$ is not a poset, let alone a domain, but we will be able to generalize much of the story
about $\IR$ to $\IRinv$, making use of the theory of \emph{continuous categories}. Continuous
categories are a generalization---first described in \cite{Johnstone.Joyal:1982a}---of continuous
posets. This section (\ref{sec:review_cont_cats}) is largely a review of
\cite{Johnstone.Joyal:1982a}, or of \cite[Ch.\ C4]{Johnstone:2002a}, so we only sketch the proofs in
this section.

\begin{definition}[Filtered categories]
  A category $I$ is \emph{filtered} if every finite diagram in $I$ has a cocone in $I$.
  \index{category!filtered}

  Let $C$ be a category. The category $\Ind{C}$ has as objects pairs $(I,F)$, where $I$ is a filtered
  category and $F\colon I\to C$ is a functor. The morphisms are given by
  \[
    \Ind{C}\big((I,F),(J,G)\big)\coloneqq\lim_{i\in I\op}\colim_{j\in J}C(Fi,Gj).
  \]
\end{definition}

Although the above definition of morphisms in $\Ind(C)$ may look convoluted, it is justified by seeing $\Ind$-objects as presheaves, as we now explain.

For any object $c\in C$, the over-category $C/c$ is filtered, so the forgetful functor $C/c\to C$ is
an object of $\Ind(C)$.\index{object!Ind} Its opposite can be identified with the category of
elements of the representable functor, $(C/c)\cong\el(yc)\op$, where $yc=C(-,c)$. Noting that any
object $(I,F)\in\Ind(C)$ can be regarded as a presheaf by $\colim(I\To{F}C\To{y}\Psh{C})$ leads to
the following.

\begin{proposition}\label{prop:ind_characterizations}
  For any category $C$, the following categories are equivalent:
  \begin{itemize}
    \item the category $\Ind{C}$,
    \item the full subcategory of $\Psh{C}$ spanned by filtered colimits of representables,
    \item the full subcategory of $\Psh{C}$ spanned by functors $F\colon C\op\to\Cat{Set}$ for which
      $\el(F)\op$ is filtered.
    \item the full subcategory of $\Psh{C}$ spanned by flat functors $F\colon C\op\to\Cat{Set}$.\index{functor!flat}
  \end{itemize}
\end{proposition}

The third and fourth are exactly the same; the former is the definition of flat in this case. Recall
from \cref{def:ideal_dcpo} the notion of ideals in a poset, namely nonempty, down-closed,
up-directed subsets. The following is \cite[Lemma 1.1]{Johnstone.Joyal:1982a}.

\begin{lemma}\label{lem:Ind_generalizes_Id}
  If $P$ is a poset, then $\Ind{P}$ is equivalent to the category of ideals in $P$.\index{ideal}
\end{lemma}

Recall from \cref{prop:directed_complete_poset,prop:cont_poset} that a poset $P$ is
directed-complete iff the canonical map $P\to\Id(P)$ has a left adjoint $\sup$, and that in this
case $P$ is continuous iff $\sup$ has a further left adjoint. In light of
\cref{lem:Ind_generalizes_Id}, \cref{prop:filtered_colims_adjoint,def:continuous_cat} generalize the
definition of continuous posets to categories.

\begin{proposition}\label{prop:filtered_colims_adjoint}
  A category $C$ has filtered colimits iff the Yoneda embedding $y\colon C\to\Ind{C}$ has a left
  adjoint, which we denote $\colim\colon\Ind{C}\to C$.
\end{proposition}

\begin{definition}[Continuous categories]\label{def:continuous_cat}\index{category!continuous}
  A category $C$ is \emph{continuous} if it has filtered colimits and the functor
  $\colim\colon\Ind{C}\to C$ has another left adjoint
  \[
    \begin{tikzcd}[column sep=50]
      C \ar[r,shift left=3, "y"]
        \ar[r, shift right=3,"\wb"']
      &\Ind{C} \ar[l,"\colim" description]
    \end{tikzcd}
  \]
  From the embedding $\Ind{C}\to\Psh{C}$, we can think of $\wb\colon C\to\Ind{C}$ as defining a
  profunctor, which we denote by
  \begin{equation}\label{eqn:abstract_waybelow}
    \WB{-}{-}\colon C\op\times C\to\Cat{Set},
  \end{equation}
  i.e.\ $\WB{c}{d}\coloneqq \wb(d)(c)$.
\end{definition}

It is easy to see that $c\cong\colim(yc)$ for any $c\in C$, so we obtain a map $i\colon\wb(c)\to
yc$. Thus for any $c,c'\in C$ and $f\in\WB{c}{c'}$, following \cite{Johnstone.Joyal:1982a}, we write
$f\colon c\wavyto c'$ and call it a \emph{wavy arrow}\index{wavy arrow}. The map $i$ is not a
monomorphism in general, so a wavy arrow is additional structure on---not just a property of---an
arrow. We will see later in \cref{prop:IRinv_continous} that in $\IRinv$, this notion of wavy arrow
coincides with the one we defined earlier in \cref{def:wavy_arrow}.

\begin{example}
  Any accessible category $C$ is continuous. The extra left adjoint $\wb\colon C\to\Ind{C}$ sends
  any object to the diagram of compact objects over
  it.\index{object!compact}\index{category!accessible}
\end{example}

\begin{remark}\label{rmk:Ind_KZ}\index{monad!KZ}
  The 2-functor $\Fun{Ind}\colon\Cat{Cat}\to\Cat{Cat}$ is the free cocompletion under directed
  colimits, and like all free cocompletions it has the structure of a \emph{lax idempotent 2-monad}
  (also known as a \emph{KZ-monad} or \emph{KZ-doctrine})
  \cite{Kock:1995a}.\index{monad!KZ}\index{object!$\Ind$} In particular this means that the
  structure map of an algebra $\colim_C\colon\Ind{C}\to C$ is always left adjoint to the unit of the
  monad $y\colon C\to\Ind{C}$, as in \cref{prop:filtered_colims_adjoint}. The category $\Ind{C}$ has
  filtered colimits, and they give the components of the monad multiplication,
  $\colim_{\Ind{C}}\colon\Ind{\Ind{C}}\to\Ind{C}$.

  When the structure map $\colim$ of an algebra has a further left adjoint $\wb$, it follows that $\wb\colon C\to\Ind{C}$
  is an algebra homomorphism (in this case, preserves directed colimits), and a section of $\colim$
  \cite[B.1.1.15]{Johnstone:2002a}.
  That is,
  \begin{gather}
    \label{eq:W_alg_hom}
    \textstyle \colim_{\Ind{C}}\circ\Ind{\wb} \iso \wb\circ\colim_C \\\nonumber
    \textstyle \colim_C\circ \wb \iso \id_C.
  \end{gather}
\end{remark}

In a continuous poset, the way-below relation is interpolative.\index{interpolative} In
\cite{Johnstone.Joyal:1982a} Joyal and Johnstone prove the analogous fact for any continuous
category $C$, namely that the following natural map is an isomorphism
\begin{equation}\label{JJ:idempotence}
  \int^{\ell'}\WB{\ell}{\ell'}\times\WB{\ell'}{\ell''}\To{\cong}\WB{\ell}{\ell''}.
\end{equation}
Thus we have an idempotent comonad on $C$ in the category of profunctors.\index{interpolative!wavy
arrows as}

\section{The (connected, discrete bifibration) factorization system}
\label{sec:connected_dbifib_factorization}
\index{orthogonal factorization system}\index{functor!discrete bifibration}\index{functor!connected}

In order to prove that $\IRinv$ is a continuous category, it will be useful to use the (connected,
discrete bifibration) orthogonal factorization system on $\Cat{Cat}$, which we learned about from
Andr\'{e} Joyal's (semi-defunct) ``CatLab''.%
\footnote{
  See \url{https://ncatlab.org/joyalscatlab/published/Factorisation+systems\#in_cat}.
}
Since we can find neither a reference for this material nor proofs in the literature, we describe it
below, assuming the reader is aware of the orthogonal factorization systems (final, discrete
fibration) and (initial, discrete opfibration) on $\Cat{Cat}$ due to Street and Walters
\cite{Street.Walters:1973a}.

Recall that the fully faithful functor $U\colon\Cat{Grpd}\to\Cat{Cat}$ has a left adjoint, given on
a category $C$ by adding a formal inverse to each morphism in $C$. An explicit description is given
in \cref{lemma:Gpd_zigzag}. We denote both this groupoidification functor\index{groupoidification}
and the corresponding idempotent monad\index{monad!idempotent} by
$\Gamma\colon\Cat{Cat}\to\Cat{Grpd}\ss\Cat{Cat}$, and we denote the unit of the adjunction by
$\eta\colon\id_{\Cat{Cat}}\to\Gamma$.\index{zigzag}

\begin{lemma}\label{lemma:Gpd_zigzag}
  For any category $C$, the unit $\eta_C\colon C\to \Gamma C$ is identity on objects. A morphism
  $c\to d$ in $\Gamma C$ can be identified with a zigzag $Z_n\to C$, up to the congruence defined
  below. A zigzag is a diagram in $C$ of the following form for some $n\in\NN$:
  \[
    c=z_0\To{f_1} z_1'\From{f_1'}\cdots\To{f_n} z_n'\From{f_n'} z_n=d
  \]
  which we denote by the list $(f_n',f_n,\ldots,f_1',f_1)$. The identity zigzag is $()$, i.e.\ the
  one of length $n=0$, and composition is given by list concatenation. The congruence is generated
  by the relations
  \begin{gather*}
    (f_2',f_2,\id,f_1)\sim(f_2',f_2\circ f_1)\hspace{.5in}
    (f_2',\id,f_1',f_1)\sim(f_2'\circ f_1', f_1)\\
    (f,f)\sim()\hspace{.5in}
    (f_2',f,f,f_1)\sim(f_2',f_1)
  \end{gather*}
  The inverse of a zigzag is given by reversing it.

  The unit $\eta_C$ sends $f\colon c\to d$ to $(\id_{d},f)$. If $G$ is a groupoid then $\eta_G$ is
  an isomorphism.
\end{lemma}

\begin{definition}[Connected]\label{def:connected}
\index{category!connected}\index{groupoid!connected}\index{functor!connected}
  A functor between groupoids is called \emph{connected} if it is essentially surjective and full. A
  functor $F$ between categories is called \emph{connected} if $\Gamma F$ is connected as a functor
  of groupoids. A groupoid or category is \emph{connected} if the unique functor to the terminal
  category  is connected.
\end{definition}

\begin{remark}
  Note that a groupoid $G$ is connected if and only if it has at least one object, and for any two
  objects $a,b\in G$ there exists a morphism $f\colon a\to b$. Thus a category $C$ is connected if
  and only if it has at least one object, and for any two objects $a,b\in C$ there exists a zig-zag
  from $a$ to $b$.
\end{remark}

Recall that a groupoid is called contractible if it is nonempty and each of its hom-sets has exactly
one element. The following is straightforward.

\begin{proposition}\label{prop:contractible groupoid}\index{groupoid!contractible}
  Let $G$ be a groupoid. The following are equivalent
  \begin{itemize}
    \item $G$ is contractible.
    \item There exists a connected functor $1\to G$.
    \item $G$ is nonempty and every functor $1\to G$ is connected.
  \end{itemize}
\end{proposition}

\begin{proposition}\label{prop:filtered_contractible_groupoid}
  If $I$ is a filtered category then there exists a connected functor $1\to I$.
\end{proposition}
\begin{proof}
  By \cref{prop:contractible groupoid} it suffices to show that $\Gamma I$ is contractible. Since
  $I$ is nonempty, so is $\Gamma I$, so take $i\in \Gamma I$. It suffices to show that $i$ has only
  one endomorphism, i.e., that any zigzag $(f_n',f_n,\ldots,f_1',f_1)\colon i\to i$ is equivalent to
  $()$. Note that for any commutative diagram of the following form,
  \[
  \begin{tikzcd}[sep=small]
    z_0\ar[dr,"f_1"']&&z_1\ar[dl,"f_1'"]\ar[dr,"f_2"']&&z_2\ar[dl,"f_2'"]\\
    &z_1'\ar[dr,"g_1"']&&z_2'\ar[dl,"g_2'"]\\
    &&z_1''
  \end{tikzcd}
  \]
  we have $(f_2',f_2,f_1',f_1)=(g_2'\circ f_2',g_1\circ f_1)$. The result then follows by induction
  on $n$, using the fact that $I$ is filtered.
\end{proof}

\begin{proposition}\label{prop:groupoid_factorization}
  Let $F\colon G\to G'$ be a functor between groupoids.
  \begin{itemize}
    \item $F$ is a discrete bifibration iff it is a discrete fibration iff it is a discrete
      opfibration.
    \item $F$ is connected iff it is initial iff it is final.
  \end{itemize}
  It follows that there is a (connected, discrete bifibration) orthogonal factorization system on
  $\Cat{Grpd}$.
\end{proposition}
\begin{proof}
  To prove the first claim, it suffices by duality to show that if $F$ is a discrete fibration then
  it is a discrete opfibration, and this is straightforward. The third claim follows from the first
  and second by the more general result for categories (\cite{Street.Walters:1973a}). For the second
  claim, $F$ is initial iff final by duality, so it remains to show it is final iff connected.

  By definition, $F$ is final iff for all $g'\in G'$ the comma category $(g'\down F)$ is nonempty
  and every two objects in it have a morphism between them. It is easy to check that $F$ is
  essentially surjective iff, for all $g'\in G'$ the comma category $(g'\down F)$ is nonempty. It is
  also easy to check that $F$ is full iff every two objects in $(g'\down F)$ have a morphism between
  them.
\end{proof}

Our first goal is to extend the (connected, discrete bifibration) orthogonal factorization system, given in \cref{prop:groupoid_factorization}, 
from $\Cat{Grpd}$ to $\Cat{Cat}$; see \cref{thm:OFS_connected_bifib}.

\begin{lemma}\label{lemma:pullback_bifibration_iso}
  A bifibration over a groupoidification is a groupoidification.
  
  More precisely, let $D$ be a
  category and suppose given a bifibration of groupoids $p\colon G\to\Gamma D$. Let $p'\colon C\to
  D$ be its pullback along $\eta_D$ as shown,
  \[
  \begin{tikzcd}
    &\Gamma C\ar[dr,"i"]\\[-10pt]
    C\ar[ur,"\eta_C"]\ar[d,"p'"']\ar[rr,"e"']\ar[dr,very near start, phantom,"\lrcorner"]&&G\ar[d,"p"]\\
    D\ar[rr,"\eta_D"']&~&\Gamma D
  \end{tikzcd}
  \]
  and let $i$ be the universal arrow as shown. Then $i$ is an isomorphism.
\end{lemma}
\begin{proof}
  We will find a functor $j\colon G\to \Gamma C$ inverse to $i$. The functors $\eta_D$ and $\eta_C$
  are identity on objects, so $e$ and $i$ are too. Hence we need only define $j$ on morphisms and
  prove that $i$ and $j$ are mutually inverse.

  Given a morphism $f\colon g\to g'$ in $G$, we may identify $p(f)$ with a zigzag $z\colon Z_n\to
  D$. Its image $\eta_D\circ z$ can be lifted uniquely along $p$ to a zigzag $z'\colon Z_n\to G$,
  because $p$ is a discrete bifibration, and it is easy to see that $z'\mapsto f$ under the
  isomorphism $\Gamma G\to G$. Since $p(f)=\eta_D\circ z$ and $C$ is a pullback, there is a unique
  functor $f'\colon Z_n\to C$ making the diagram commute: $e(f')=f$ and $p'(f')=z$. We define
  $j(f)\coloneqq\eta_C\circ f'$, and note that $(i\circ j)(f)=e(f')=f$. It is also easy to show that
  $j\circ i$ is the identity, completing the proof.
\end{proof}

\begin{proposition}\label{prop:pullback_bifibration}\index{discrete bifibration}
  Let $p\colon C\to D$ be a functor and $\Gamma p\colon\Gamma C\to\Gamma D$ its groupoidification.
  Then $p$ is a discrete bifibration iff $\Gamma p$ is a discrete bifibration and the
  $\eta$-naturality square below is a pullback:
  \[
  \begin{tikzcd}
    C\ar[r,"\eta_C"]\ar[d,"p"']&\Gamma C\ar[d,"\Gamma p"]\\
    D\ar[r,"\eta_D"']&\Gamma D\ar[from=ul, phantom, very near start, "\lrcorner"]
  \end{tikzcd}
  \]
\end{proposition}
\begin{proof}
  The pullback of a discrete bifibration is a discrete bifibration, so one direction is obvious.

  Suppose $p$ is a discrete bifibration, and consider the induced functor $P\colon D\to\Cat{Set}$
  whose category of elements is $p$. Note that for every morphism $f$ in $D$, the induced function
  $P(f)$ is a bijection. Hence $P$ factors as $D\To{\eta_D}\Gamma D\To{P'}\Cat{Set}$ for a unique
  functor $P'$, whose category of elements a discrete bifibration $p'\colon G\to\Gamma D$. The
  pullback of $p'$ along $\eta_D$ is $p$, and the result follows by
  \cref{lemma:pullback_bifibration_iso}.
\end{proof}

\begin{theorem}\label{thm:OFS_connected_bifib}
  There is an orthogonal factorization system $(\cat{L},\cat{R})$ on $\Cat{Cat}$, where $\cat{L}$
  consists of the connected functors and $\cat{R}$ consists of the discrete bifibrations.
\end{theorem}
\begin{proof}
  For any functor $f\colon C\to D$, form the following diagram from $\eta$'s naturality square
  \[
  \begin{tikzcd}[sep=15pt]
    C\ar[rr,"\eta_C"]\ar[dd,"f"']\ar[dr,"c"]&&\Gamma (C)\ar[dr,"c'"]\\
      &\wt{f}\ar[rr, near start, "e"']\ar[dl,"b"]&&\wt{\Gamma (f)}\ar[dl,"b'"]\\
    D\ar[rr,"\eta_D"']&&\Gamma (D)\ar[from=uu, crossing over, near start,"\Gamma (f)"']
  \end{tikzcd}
  \]
  by first factoring $\Gamma (f)=b'\circ c'$, where $c'$ is connected and $b'$ is a bifibration (see
  \cref{prop:groupoid_factorization}), then forming the pullback of $b'$ along $\eta_D$ to give $b$
  and $e$, and finally adding the universal morphism $c$ making the diagram commute. Clearly $b$ is
  a bifibration because it is the pullback of one.

  By \cref{lemma:pullback_bifibration_iso} the functor $e$ factors through an isomorphism
  $\Gamma(\wt{f})\cong\wt{\Gamma(f)}$. Hence we can identify $c'$ with $\Gamma(c)$, and by
  \cref{def:connected} $c$ is connected. This defines the (connected, discrete bifibration)
  factorization: $f=b\circ c$.

  It remains to show the orthogonality, so consider a solid-arrow square as below:
  \[
  \begin{tikzcd}
    X\ar[r,"f"]\ar[d,"c"']&C\ar[d,"b"]\\
    Y\ar[ur, dotted]\ar[r,"g"']&D
  \end{tikzcd}
  \]
  where $c$ is connected and $b$ is a discrete bifibration; we want to show there is a unique
  dotted-arrow $Y\to C$ making the diagram commute. By \cref{prop:pullback_bifibration}, the functor
  $\Gamma b$ is a bifibration and the front face in the commutative diagram below is a pullback:
  \[
  \begin{tikzcd}[sep=small]
    X\ar[rr]\ar[rd]\ar[dd, "c"']&&\Gamma X\ar[dd]\ar[dr]\\
    &C\ar[rr, crossing over]&&\Gamma C\ar[dd, "\Gamma b"]\\
    Y\ar[rr]\ar[rd]&&\Gamma Y\ar[dr]\ar[ur, dotted]\\
    &D\ar[from=uu, crossing over]\ar[rr]&&\Gamma D
  \end{tikzcd}
  \]
  where the back and right-hand faces are given by universality of $\Gamma$. The functor $\Gamma c$
  is connected by definition, so there is a unique dotted-arrow lift as shown. Since the front face
  is a pullback, we obtain a map $Y\to C$ making the diagram commute, as required.
\end{proof}

For any two categories $C,D$, let $C\join D$ denote their join. Recall the notation of linear order
categories---in particular $0$, $1$, and $2$---from \cref{not:basic_for_cats}. It follows easily
from \cref{def:connected} that if $C$ is a connected category, then so are the inclusions $j\colon
C\to(C\join 1)$ and $j'\colon C\to(1\join C)$.

\begin{proposition}\label{prop:bifib_preserves_reflects}\index{discrete bifibration!connected limits and}\index{discrete bifibration!connected colimits and}
  Suppose that $p\colon C\to D$ is a discrete bifibration. Then it preserves and reflects connected
  limits and colimits.
\end{proposition}
\begin{proof}
  We consider the case for colimits, the case for limits being dual. Let $2=1\join 1$ be the arrow
  category, with the two maps $1\to 2$ denoted $d=\id\join\bang\colon 1\join 0\to 1\join 1$ and
  $c=\bang\join\id\colon 0\join 1\to 1\join 1$. Also let $j=\id\join \bang\colon I\join 0\to I\join
  1$. We can represent the colimit cocone of a diagram $F\colon I\to C$ as a functor $\colim F\colon
  I\join 1\to C$ such that $(\colim F)\circ j = F$, and such that for any $\phi\colon I\join 1\to C$
  satisfying $\phi\circ j=F$ there exists a unique $\phi'\colon I\join 2\to C$ such that the two
  right triangles of the following diagram commute:
  \begin{equation}\label{eq:colim_property}
  \begin{tikzcd}
    &&I\join 1\ar[dl,"\colim F"']\ar[d, "\id\join d"]\\
    I\ar[rru, bend left, "j"]\ar[rrd, bend right, "j"']\ar[r, "F"]&C&I\join 2\ar[l, dotted, near start, "\phi'"']\\
    &&I\join 1\ar[u, "\id\join c"']\ar[ul, "\phi"]
  \end{tikzcd}
  \end{equation}

  Let $I$ be a small connected category and suppose that $F\colon I\to C$ has a colimit, $\colim F$.
  Then $p\circ(\colim F)$ is a candidate colimit of $p\circ F$, so to see that it has the universal
  property, suppose we have a solid-arrow diagram
  \begin{equation}\label{eqn:lifting_colimits}
  \begin{tikzcd}
    I\ar[r,"F"]\ar[d,"j"']&C\ar[d,"p"]\\
    I\join 1\ar[r,"\psi"']\ar[ur,dotted,"\phi"]&D
  \end{tikzcd}
  \end{equation}
  Since $j$ is connected, we obtain a unique lift $\phi$ as shown. Then letting $\phi'$ be as in
  \cref{eq:colim_property}, we claim $p\circ\phi'$ is the unique functor such that
  $(p\circ\phi')\circ(\id\join d)=p\circ(\colim F)$ and $(p\circ\phi')\circ(\id\join c)=\psi$. To
  prove uniqueness, suppose $\psi'\colon I\join 2\to D$ is another functor such that
  $\psi'\circ(\id\join d)=p\circ(\colim F)$ and $\psi'\circ(\id\join c)=\psi$. Again we find a
  unique lift $\hat{\psi'}$\,,
  \begin{equation}
  \begin{tikzcd}[column sep=40pt]
    I\join 1 \ar[r,"\colim F"] \ar[d,"\id\join d"']
      & C \ar[d,"p"] \\
    I\join 2 \ar[ur,dotted, pos=.4,"\hat{\psi'}"] \ar[r,"\psi'"']
      & D.
  \end{tikzcd}
  \end{equation}
  But then $\hat{\psi'}\circ(\id\join c)$ is also a lift for the square
  \eqref{eqn:lifting_colimits}, so $\hat{\psi'}\circ(\id\join c)=\phi$. Since also
  $\hat{\psi'}\circ(\id\join d)=\colim F$, we have $\hat{\psi'}=\phi'$ by the uniqueness of $\phi'$
  in \cref{eq:colim_property}. Hence $\psi'=p\circ\hat{\psi'}=p\circ\phi'$ as desired.

  The proof that $p$ reflects connected colimits is easier: given a diagram
  \eqref{eqn:lifting_colimits} such that $\psi$ is a colimiting cocone in $D$, one checks that the
  unique lift $\phi$ is a colimit cocone in $C$.
\end{proof}

\begin{corollary}\label{cor:bifib_Ind_surj}
  If $p\colon C\to D$ is a discrete bifibration and $p$ is surjective on objects, then $\Ind(p)$ is
  surjective on objects. In particular, if $C$ has all filtered colimits (resp.\ filtered limits)
  then so does $D$.
\end{corollary}
\begin{proof}
  Let $F\colon I\to D$ be a filtered diagram. By \cref{prop:filtered_contractible_groupoid}, there
  exists a connected functor $i\colon 1\to I$, and because $p$ is surjective on objects there is a
  map $c\colon 1\to C$ such that $p\circ c=F\circ i$. Thus we obtain a lift $G\colon I\to C$ such
  that $p\circ G=F$. If it has a colimit in $C$ then by \cref{prop:bifib_preserves_reflects} this
  colimit is preserved by $p$.
\end{proof}

\begin{lemma}\label{lemma:d_fib_exponentiating}
  If $(\cat{L},\cat{R})$ is an orthogonal factorization system on $\Cat{Cat}$ and $p\colon C\to D$
  is in the right class then so is $p^K\colon C^K\to D^K$ for any small category $K$.
\end{lemma}
\begin{proof}
  This follows from the well-known fact that $\cat{L}$ is closed under limits in the arrow category,
  because if $c$ is in the left class then so is $c\times\id_K$.
\end{proof}

\begin{lemma}\label{lemma:Ind_disc_fib}
  If $p\colon C\to D$ is a discrete fibration (resp.\ discrete opfibration), then so is the functor
  $\Ind(p)\colon\Ind(C)\to\Ind(D)$.
\end{lemma}
\begin{proof}
  We consider only the discrete fibration (d-fib) case, the other being similar. Note that $p\colon
  C\to D$ is a d-fib iff $C^K\to D^K$ is a d-fib for all small categories $K$ by
  \cref{lemma:d_fib_exponentiating}. We need to show that $\Ind(p)$ has the right lifting property
  with respect to the inclusion of the codomain, $c\colon 1\to 2$, as well as the codiagonal
  $c\sqcup_1c\to 2$, which we denote $\delta\colon 2\sqcup_12\to 2$.

  For the first, suppose given filtered diagrams $F\colon I\to C$ and $G\colon J\to D$ and a map
  $f\colon \colim(p\circ F)\to\colim(G)$ in $\Ind(D)$. By \cite[Corollary
  6.1.14]{Kashiwara.Schapira:2006a} there is a filtered category $K$, final functors $i\colon K\to
  I$ and $j\colon K\to J$, and a natural transformation $\varphi\colon p\circ F\circ i\to G\circ j$
  such that $\colim\varphi$ is isomorphic to $f$; see the left-hand diagram below.
  \[
  \begin{tikzcd}[row sep=5pt]
    &[-7pt]I\ar[r,"F"]&C\ar[dd, "p"]\\
    K\ar[ur, "i"]\ar[dr, "j"']\ar[rr, phantom,"\varphi\!\Downarrow", pos=.6]&&~\\
    &J\ar[r, "G"']&D
  \end{tikzcd}
  \hspace{1in}
  \begin{tikzcd}[column sep=40pt]
    K\ar[r, "Fi"]\ar[d, equal]&C\ar[d, "p"]\\
    K\ar[r, bend left, "pFI"]\ar[r, phantom, "\varphi\!\Downarrow"]\ar[r, bend right, "Gj"']&D
  \end{tikzcd}
  \hspace{1in}
  \begin{tikzcd}[row sep=25pt]
    1\ar[r, "Fi"]\ar[d, "c"']&C^K\ar[d, "p^K"]\\
    2\ar[r, "\varphi"']&D^K
  \end{tikzcd}
  \]
  Since pre-composition with final functors preserves colimits, the original lifting problem can be
  replaced by the one shown in the middle, or equivalently the diagram on the right. It has a lift
  because $p$, and hence $p^K$, is a discrete fibration

  To lift diagrams of the shape $\delta$, one reasons similarly. Namely, using \cite[Corollary
  6.1.14]{Kashiwara.Schapira:2006a}, one finds a single $K$ such that the lifting problem can be
  rewritten as follows:
  \[
  \begin{tikzcd}[column sep=40pt]
    K\ar[r, bend left=40pt, "\varphi_1\Downarrow"']\ar[r]\ar[r, bend right=40pt, "\varphi_2\Uparrow"]\ar[d, equal]&C\ar[d, "p"]\\
    K\ar[r, bend left]\ar[r, phantom, "\varphi\!\Downarrow"]\ar[r, bend right]&D
  \end{tikzcd}
  \hspace{1in}
  \begin{tikzcd}[column sep=40pt]
  2\sqcup_12\ar[d, "\delta"']\ar[r, "{(\varphi_1,\varphi_2)}"]&C^K\ar[d, "p"]\\
  2\ar[r, "\varphi"'] &D^K
  \end{tikzcd}
  \]
  This again has a lift because $p$ is a discrete fibration. We leave the details to the reader.
\end{proof}

\section{Proof that \texorpdfstring{$\IRinv$}{\unichar{"1D540}\unichar{"211D}\unichar{"2215}\unichar{"25B9}} is continuous}\index{intervals!translation-invariant!continuous category of}

Recall the definition of $\IRinv$, the category of the translation-invariant intervals, from
\cref{def:IRinv}. It has objects $\ell\in\RR_{\geq 0}$ and morphisms
$\subint{r}{s}\colon\ell\to\ell'$ for $r,s\geq 0$, with composition given by component-wise sum. In
\cref{def:wavy_arrow} we defined such a morphism to be a \emph{wavy arrow} if $r>0$ and $s>0$, in
which case we write $\subint{r}{s}\colon\ell\wavyto\ell'$. We can thus define a profunctor
$\WB{-}{-}\colon\IRinv\op\times\IRinv\to\Cat{Set}$ sending $\ell,\ell'\in\IRinv$ to the set
\begin{equation}\label{eqn:WB}
  \WB{\ell}{\ell'}\coloneqq\left\{\subint{r}{s}\in\IRinv(\ell,\ell')\mid r>0, s>0\right\}.
\end{equation}
Clearly this is a subprofunctor of the hom-profunctor $\IRinv(-,-)$. In order for it to agree with
our definition of wavy arrows from \cref{def:continuous_cat}, we need to show that induces a left
adjoint to $\colim$. Indeed, for any $\ell'\in\IRinv$, define a presheaf
\begin{equation}\label{eqn:def_wb}
  \wb(\ell')\colon\IRinv\op\to\Cat{Set},\qquad\ell\mapsto\WB{\ell}{\ell'}.
\end{equation}
$\wb$ is functorial in $\ell'$, lands in $\Ind(\IRinv)$, and is left adjoint to $\colim$ as we now
show.

\begin{theorem}\label{prop:IRinv_continous}
  The category $\IRinv$ is a continuous category. In particular, the functor $W$ defined in
  \cref{eqn:def_wb} is left adjoint to $\colim$.
\end{theorem}
\begin{proof}
  Consider the following diagram of categories:
  \begin{equation}\label{eqn:cont_cat_comm_sq}
  \begin{tikzcd}[column sep=45pt]
    \IR\ar[r, shift left, "\downclose"]\ar[d,"p"']&\Ind(\IR)\ar[l, shift left, "\bigvee"]\ar[d,"\Ind p"]\\
    \IRinv\ar[r, shift left, dotted, "\wb"]&\Ind(\IRinv)\ar[l, shift left, dotted, "\colim"]
  \end{tikzcd}
  \end{equation}
  On the top, we have that $\downclose$ is left adjoint to $\bigvee$ and that
  $\bigvee\circ\downclose=\id_\IR$ by \cref{prop:cont_poset}. Our goal is to fill out the bottom
  half. That is, we want to show that $\IRinv$ has filtered colimits and to construct a left adjoint
  $\wb$ to $\colim\colon\Ind(\IRinv)\to\IRinv$, such that $\colim\circ \wb$ is the identity and such
  that both the leftward and the rightward squares in \eqref{eqn:cont_cat_comm_sq} commute.

  We showed in \cref{lem:discrete_fibration} that $p\colon\IR\to\IRinv$ is a discrete bifibration
  and that it is surjective on objects. Thus $\IRinv$ has filtered colimits and $p$ preserves them
  by \cref{cor:bifib_Ind_surj,prop:bifib_preserves_reflects}; hence the leftward diagram commutes.

  For $\ell,\ell'\in\IRinv$, let $\WB{\ell}{\ell'}$ be as in \cref{eqn:WB}, and let $\wb(\ell')
  \colon\IRinv\op\to\Cat{Set}$ be as in \cref{eqn:def_wb}. The latter is a flat functor because any
  pair of objects $\subint{r_1}{s_1}\in\WB{\ell_1}{\ell'}$ and
  $\subint{r_2}{s_2}\in\WB{\ell_2}{\ell'}$ have a coproduct $\subint{\max(r_1,r_2)}{\max(s_1,s_2)}$
  in $(\IRinv\op/\ell')\op$. It is easy to check that $\wb\circ p=\Ind(p)\circ\downclose$, so the
  rightward diagram commutes.

  Note that every isomorphism in $\IRinv$ is an identity. We next see that $\colim(W(\ell))=\ell$
  for any $\ell\in\IRinv$. This follows from the analogous fact for domains $D$, that
  $\bigvee\downclose d=d$ for any $d\in D$. Indeed choose $\tilde{\ell}$ with $p(\tilde{\ell})=\ell$
  and we have
  \[
    \colim W(\ell) = \colim W(p\tilde{\ell}) = \colim\Ind(p)(\downclose\tilde{\ell})
      = p\textstyle{\bigvee}\downclose\tilde{\ell} = p\tilde{\ell} = \ell.
  \]

  We show that $W$ is left adjoint to $\colim$ by constructing a unit $\eta\colon\id\to\colim\circ
  W$ and counit $\epsilon\colon W\circ\colim\to\id$ and checking the triangle identities
  $(\colim\epsilon)\cdot(\eta\colim)=\id_{\colim}$ and $(\epsilon W)\cdot(W\eta)=\id_W$. The unit
  $\eta$ is the identity, so the triangle identities simplify to $\colim\epsilon = \id_{\colim}$ and
  $\epsilon W=\id_W$.

  We want to construct a counit $\epsilon_F\colon W(\colim F)\to F$ for any $F\in\Ind{\IRinv}$.
  Since $p$ is a bifibration and surjective on objects, $\Ind(p)$ is surjective on objects by
  \cref{cor:bifib_Ind_surj}. Choose $\wt{F}$ with $\Ind(p)(\wt{F})=F$, and define
  $\epsilon_F\coloneqq\Ind(p)\big(\downclose\bigvee\wt{F}\To{\leq} \wt{F}\big)$:
  \[
    \wb\colim(F)=\Ind(p)\big(\downclose\textstyle{\bigvee}\wt F\big)\To{\Ind(p)\big(\downclose\bigvee\wt F\To{\leq}\wt
    F\big)}\Ind(p)(\wt F)=F.
  \]
  To show that this definition of $\epsilon_F$ doesn't depend on the choice of lift $\wt{F}$, let
  $\wt{F}'$ be another lift with $\Ind(p)(\wt{F}')=F$. Then
  \[
    \Ind(p)\big(\downclose\textstyle{\bigvee}\wt{F}\big) = Wp\textstyle{\bigvee}\wt{F} = W\colim\Ind(p)(\wt{F}) = W\colim F
  \]
  and similarly $\Ind(p)(\downclose\bigvee\wt{F}') = W\colim F$, hence
  \[
    \Ind(p)\big(\downclose\textstyle{\bigvee}\wt{F}'\To{\leq}\wt{F}\big) = \Ind(p)\big(\downclose\textstyle{\bigvee}\wt{F}\To{\leq}
     \wt{F}\big) = \epsilon_F.
  \]
  It easily follows from this that $\epsilon_F$ is natural in $F$. Finally, we see that
  \[
    \colim(\epsilon_F) = \colim\Ind(p)\big(\downclose\textstyle{\bigvee}\wt{F}\To{\leq}\wt{F}\big) =
    p\big(\textstyle{\bigvee}\downclose\textstyle{\bigvee}\wt{F}\To{\id}\textstyle{\bigvee}\wt{F}\big) = p(\id_{\textstyle{\bigvee}\wt{F}})=\id_{\colim F}
  \]
  and to see that $\epsilon_{W\ell} = \id_{W\ell}$, choose a lift $p\tilde{\ell}=\ell$, so that
  $\Ind(p)(\downclose\tilde{\ell}) = Wp\tilde{\ell} = W\ell$, and
  \[
    \epsilon_{W\ell} =
    \Ind(p)\big(\downclose\textstyle{\bigvee}\downclose\tilde{\ell}\To{\id}\downclose\tilde{\ell}\big) =
      \id_{W\ell}.\qedhere
  \]
\end{proof}

\section{Two constructions of the topos $\BaseTopos$}
\label{sec:constructing_BaseTopos}\index{behavior types!topos of}

We now sketch two equivalent constructions of the topos $\BaseTopos$. This will give a sort of
analogue to \cref{thm:IR_sheaves} for continuous categories (see \cref{prop:IRinv_sheaves}). We
mainly follow \cite{Johnstone.Joyal:1982a}; see there for details.

\subsection{Adjoint endofunctors on $[\IRinv,\Cat{Set}]$}

First we construct a pair of adjoint endofunctors $t^*,t_*$ on the functor category
$[\IRinv,\Cat{Set}]$. To do so, consider $\WB{-}{-}$, from \cref{eqn:WB}, as a functor
$\IRinv\op\to[\IRinv,\Cat{Set}]$. By left Kan extension along the yoneda embedding, we get a map
$t^*\colon[\IRinv,\Cat{Set}]\to[\IRinv,\Cat{Set}]$ as shown in the following diagram:
\[
  \begin{tikzcd}[column sep=45pt, row sep=30pt]
    \IRinv\op \ar[d,"y"'] \ar[r,"\ell\mapsto\WB{\ell}{-}"]
    &{[\IRinv,\Cat{Set}]} \ar[dl,shift left,"t_*"] \\
    {[\IRinv,\Cat{Set}]} \ar[ur, shift left,"t^*"]
  \end{tikzcd}
\]
It can be computed for any $X\colon\IRinv\to\Cat{Set}$ by the usual coend formula:
\begin{equation}\label{eqn:coend}
  (t^*X)(\ell')=\int^{\ell}\WB{\ell}{\ell'}\times X(\ell)=\WB{-}{\ell'}\otimes X=\wb(\ell')\otimes X.
\end{equation}
Equivalently, the composite along the bottom of the following diagram is an expression of the usual
conical colimit formula for the pointwise left Kan extension, which in this case happens to factor
through the Ind-subcategories as in the top line:
\begin{equation}\label{eqn:t_alt}
  \begin{tikzcd}
    \IRinv \ar[r,"W"] \ar[dr,"W"'] &
    \Ind\IRinv \ar[r,"\Ind{X}"] \ar[d,hook] &
    \Ind\Cat{Set} \ar[r,"\colim"] \ar[d,hook] &
    \Cat{Set}. \\
    & \Psh{\IRinv} \ar[r,"\Psh{X}"'] &
    \Psh{\Cat{Set}} \ar[ur,"\colim"']
  \end{tikzcd}
\end{equation}
Note that $t^*$ preserves finite limits because the functor $\ell\mapsto\WB{\ell}{-}$ is flat; see
\cite[6.3.8]{Borceux:1994a}.

The functor $t^*$ has a right adjoint $t_*\colon[\IRinv,\Cat{Set}]\to[\IRinv,\Cat{Set}]$; see
\cite{Nlab:nerve-and-realization}. It can be computed by the following explicit end formula
\begin{equation}\label{eqn:t_end}
  (t_*X)(\ell)=[\IRinv,\Cat{Set}](\WB{\ell}{-},X)=\int_{\ell'}\Cat{Set}\big(\WB{\ell}{\ell'},X(\ell')\big).
\end{equation}

\begin{proposition}\label{prop:idempotent}
  The functors $t^*$ and $t_*$ are idempotent.
\end{proposition}
\begin{proof}
  Since $w\mapsto w\times X(\ell)$ is cocontinous in $w$, the idempotence of $t^*$ follows from
  \eqref{JJ:idempotence} by the Fubini theorem for coends. The idempotence of $t_*$ follows by
  adjointness.
\end{proof}

\begin{proposition}\label{prop:cont_fixed_points}
  A functor $X\colon\IRinv\to\Cat{Set}$ is a fixed-point of $t^*$ (equivalently, $X$ is in the image
  of $t^*$) iff $X$ preserves filtered colimits.
\end{proposition}
\begin{proof}
  Using \cref{eqn:t_alt}, $X$ is a fixed-point of $t^*$ iff
  $X\iso\colim_{\Cat{Set}}\circ\Ind{(X)}\circ W$. If so, then
  \begin{alignat*}{2}
    \textstyle
    \colim_{\Cat{Set}} \circ \Ind{(X)}
    &\textstyle{}\iso
      \colim_{\Cat{Set}} \circ  \Ind{(\colim_{\Cat{Set}})} \circ \Ind{(\Ind{X})} \circ \Ind{(W)} &&\\
    &\textstyle{}\iso
      \colim_{\Cat{Set}} \circ \colim_{\Ind{\Cat{Set}}} \circ \Ind{(\Ind{X})} \circ \Ind{(W)}
      \qquad&&\text{$\Cat{Set}$ is an $\Fun{Ind}$-algebra} \\
    &\textstyle{}\iso
      \colim_{\Cat{Set}} \circ \Ind{(X)} \circ \colim_{\Ind{(\IRinvInline)}} \circ \Ind{(W)}
      \qquad&&\text{Naturality, Rmk.~\ref{rmk:Ind_KZ}} \\
    &\textstyle{}\iso
      \colim_{\Cat{Set}} \circ \Ind{X} \circ W \circ \colim_{\IRinvInline}
      \qquad&&\text{\cref{eq:W_alg_hom}} \\
    &\textstyle{}\iso
      X \circ \colim_{\IRinvInline}. &&
  \end{alignat*}
  Thus $X$ preserves filtered colimits. Conversely, if $X$ preserves filtered colimits, then
  \[
    \textstyle
    \colim_{\Cat{Set}}\circ\Ind{X}\circ W \iso X\circ\colim_{\IRinv}\circ W \iso X. \qedhere
  \]
\end{proof}

\subsection{$\Shv{\BaseSite}\cong\Cont{\IRinv}$}

In \cref{def:BaseSite}, we defined $\BaseSite$ analogously to our definition of $S_{\IR}$ in
\cref{def:interval_posite}. Our goal for the rest of this section is to prove that the category
$\BaseTopos=\Shv{\BaseSite}$ is not just a category of sheaves (limit-preserving functors of a given
sort), but is also (isomorphic to) a category of continuous functors (colimit-preserving functors of
a given sort). This is the content of \cref{prop:IRinv_sheaves}.

Before doing so, we can already characterize $\BaseTopos$ as a category of fixed points. Indeed, the
endofunctors $t^*$ and $t_*$ of $[\IRinv,\Cat{Set}]$ are idempotent by \cref{prop:idempotent}, hence
each determines a subcategory, say $\Fun{Fix}(t^*)$ and $\Fun{Fix}(t_*)$.

\begin{proposition}
  The following are the same as subcategories of $[\IRinv,\Cat{Set}]$\,:
  \[
    \Fun{Fix}(t_*)=\Shv{\BaseSite}.
  \]
\end{proposition}
\begin{proof}
  By \cref{def:BaseSite}, $\BaseSite$ is the site whose underlying category is $\IRinv\op$ and whose
  coverage consists of one covering family $\wb(\ell)\ss\Hom(\ell,-)$ for each object
  $\ell\in\IRinv$. That is, the collection of way-below maps $\WB{\ell}{\ell'}$ covers $\ell$. By
  \cref{eqn:t_end}, we have $(t_*X)(\ell)=[\IRinv,\Cat{Set}](\WB{\ell}{-},X)$. Thus the category of
  fixed-points (equivalently, the image) of $t_*$ consists precisely of those functors
  $\IRinv\to\Cat{Set}$ satisfying the sheaf condition for $\BaseSite$. Hence this subcategory is
  equivalent to $\Shv{\BaseSite}$.
\end{proof}

On the other hand, we showed in \cref{prop:cont_fixed_points} that the fixed-points of $t^*$
consists precisely of those functors $\IRinv\to\Cat{Set}$ which preserve filtered colimits. We
denote this category by $\Cont{\IRinv}$. Thus we have proven the following.

\begin{proposition}
  The following are the same as subcategories of $[\IRinv,\Cat{Set}]$\,:
  \[
    \Fun{Fix}(t^*)\cong\Cont{\IRinv}.
  \]
\end{proposition}

The restrictions of $t_*$ and $t^*$ to these subcategories $\Cont{\IRinv}$ and $\Shv{\BaseSite}$
then provide an equivalence of categories, given by $(\asSh'\circ U)\colon
\Shv{\BaseSite}\leftrightarrows\Cont{\IRinv}\cocolon(\asSh\circ U')$ in the following diagram
\[
\begin{tikzcd}
  &\Shv{\BaseSite}\ar[dl, tail, "U"']\\
  {[\IRinv,\Cat{Set}]}\ar[rr,shift right,"t^*"']\ar[rd,two heads,"\asSh'"']&&
  {[\IRinv,\Cat{Set}]}\ar[ll, shift right,"t_*"']\ar[lu,two heads,"\asSh"']\\
  &\Cont{\IRinv}\ar[ur, tail, "U'"']
\end{tikzcd}
\]
where $\asSh$ is the associated sheaf functor, which is also the fixed-point functor of $t_*$, and
$\asSh'$ is the fixed point of $t^*$. The equivalence itself is proven in \cref{prop:IRinv_sheaves}.

\begin{remark}\label{rmk:sheaves_cont_functors}
  While we have been using $\IRinv$ as our primary example throughout this section, the proof of
  \cref{prop:IRinv_sheaves} is easily generalized to any continuous category, including any domain.
  See also \cite[Prop.\ 3.3, Thm.\ 3.6, Prop.\ 3.8]{Johnstone.Joyal:1982a}.

  For any domain $D$, this means that there is an equivalence of categories
  \[
    \Shv{D\op} \iso \Cont{D},
  \]
  where $D\op$ is the posite defined in \cref{prop:down_closed_up_closed} (regarding $(D,\ll)$ as a
  predomain, which always has the property that $d\specord d'$ implies $d\upspecord d'$).
\end{remark}

\begin{theorem}\label{prop:IRinv_sheaves}\index{category!continuous}
  There is an equivalence of categories between those functors $\IRinv\to\Cat{Set}$ that satisfy the
  sheaf condition and those that preserve filtered colimits:
  \begin{equation}\label{eqn:eq_Shv_Cont}
    \Shv{\BaseSite}\cong\Cont{\IRinv}
  \end{equation}
\end{theorem}
\begin{proof}
  First, note that the result will follow if we can show $t^*t_*\iso t^*$ and $t_*t^*\iso t_*$.
  Indeed, this implies that $t^*$ restricted to the image of $t_*$ and $t_*$ restricted to the image
  of $t^*$ are inverses. We sketch a proof that there is an isomorphism $t^*t_*\iso t^*$; the
  isomorphism $t_*t^*\iso t_*$ is likewise straightforward to verify directly.

  Fix any object $\ell'\in\IRinv$. By \cref{eqn:coend,eqn:t_end} an element of $(t^*t_*X)(\ell')$ is
  represented by an object $\ell$ and a wavy arrow $\ell\wavyto\ell'$ in $\IRinv$, together with a
  natural transformation $\phi\colon\WB{\ell}{-}\to X$, i.e.\ a compatible family of sections of
  $X$, one for each interval way above $\ell$. Given such data, we can choose an intermediate
  interval $\ell\wavyto\ell''\wavyto\ell'$ by the isomorphism \eqref{eq:W_alg_hom}, and an element
  $\phi(\ell\wavyto\ell'')\in X(\ell'')$. The object $\ell''\in\IRinv$ together with the element
  $\phi(\ell''\wavyto\ell')$ represents an element of $(t^*X)(\ell')$. It is not difficult to
  verify, again using \cref{eqn:coend,eqn:t_end,eq:W_alg_hom} that this does not depend on the
  choice of representative, and that it is injective and surjective.
\end{proof}

\nocite{Martin.Panangaden:2006a}
\printbibliography
\printindex

\end{document}